\documentclass[a4paper]{amsart}
\usepackage{amssymb,amsthm}
\usepackage[abbrev,nobysame]{amsrefs}
\usepackage{appendix}

\theoremstyle{plain}
\newtheorem{theorem}{Theorem}[section]
\newtheorem{lemma}[theorem]{Lemma}
\newtheorem{corollary}[theorem]{Corollary}

\theoremstyle{definition}
\newtheorem{definition}[theorem]{Definition}
\newtheorem{assumption}[theorem]{Assumption}

\theoremstyle{remark}
\newtheorem{remark}[theorem]{Remark}

\numberwithin{equation}{section}

\begin{document}

\title[Navier--Stokes equations in a curved thin domain]{Navier--Stokes equations in a curved thin domain}

\author[T.-H. Miura]{Tatsu-Hiko Miura}
\address{Graduate School of Mathematical Sciences, The University of Tokyo, 3-8-1 Komaba, Meguro, Tokyo, 153-8914 Japan}
\email{thmiura@ms.u-tokyo.ac.jp}

\subjclass[2010]{Primary: 35B25, 35Q30, 76D03, 76D05; Secondary: 35R01, 76A20}

\keywords{Navier--Stokes equations, thin domain, global existence of a strong solution, singular limit problem, surface fluids}

\begin{abstract}
  We consider the three-dimensional incompressible Navier--Stokes equations in a curved thin domain with Navier's slip boundary conditions.
  The curved thin domain is defined as a region between two closed surfaces which are very close to each other and degenerates into a given closed surface as its width tends to zero.
  We establish the global-in-time existence and uniform estimates of a strong solution for large data when the width of the thin domain is very small.
  Moreover, we study a singular limit problem as the thickness of the thin domain tends to zero and rigorously derive limit equations on the limit surface, which are the damped and weighted Navier--Stokes equations on a surface with viscous term involving the Gaussian curvature of the surface.
  We prove the weak convergence of the average in the thin direction of a strong solution to the bulk Navier--Stokes equations and characterize the weak limit as a weak solution to the limit equations as well as provide estimates for the difference between solutions to the bulk and limit equations.
  To deal with the weighted surface divergence-free condition of the limit equations we also derive the weighted Helmholtz--Leray decomposition of a tangential vector field on a closed surface.
\end{abstract}

\maketitle

\setcounter{tocdepth}{1}
\tableofcontents

\section{Introduction} \label{S:Intro}
We consider the three-dimensional Navier--Stokes equations
\begin{align} \label{E:NS_Eq}
  \partial_tu^\varepsilon+(u^\varepsilon\cdot\nabla)u^\varepsilon-\nu\Delta u^\varepsilon+\nabla p^\varepsilon = f^\varepsilon, \quad \mathrm{div}\,u^\varepsilon = 0 \quad\text{in}\quad \Omega_\varepsilon\times(0,\infty)
\end{align}
imposed with Navier's slip boundary conditions
\begin{align} \label{E:NS_Bo}
  u^\varepsilon\cdot n_\varepsilon = 0, \quad [\sigma(u^\varepsilon,p^\varepsilon)n_\varepsilon]_{\mathrm{tan}}+\gamma_\varepsilon u^\varepsilon = 0 \quad\text{on}\quad \Gamma_\varepsilon\times(0,\infty)
\end{align}
and initial condition
\begin{align} \label{E:NS_In}
  u^\varepsilon|_{t=0} = u_0^\varepsilon \quad\text{in}\quad \Omega_\varepsilon.
\end{align}
Here $\Omega_\varepsilon$ is a curved thin domain in $\mathbb{R}^3$ of the form
\begin{align} \label{E:Intro_CTD}
  \Omega_\varepsilon := \{y+rn(y) \mid y\in\Gamma,\,r\in(\varepsilon g_0(y),\varepsilon g_1(y))\}, \quad \varepsilon\in(0,1),
\end{align}
where $\Gamma$ denotes a two-dimensional closed (i.e. compact and without boundary), connected, and oriented surface in $\mathbb{R}^3$ with unit outward normal vector field $n$, and $g_0$ and $g_1$ are functions on $\Gamma$ such that $g:=g_1-g_0$ is bounded from below by a positive constant.
We denote by $\Gamma_\varepsilon$ and $n_\varepsilon$ the boundary of $\Omega_\varepsilon$ and its unit outward normal vector field.
The boundary $\Gamma_\varepsilon$ is the union of the inner and outer boundaries $\Gamma_\varepsilon^0$ and $\Gamma_\varepsilon^1$ given by
\begin{align*}
  \Gamma_\varepsilon^i := \{y+\varepsilon g_i(y)n(y) \mid y\in\Gamma\}, \quad i=0,1.
\end{align*}
Also, $\nu>0$ is the viscosity coefficient independent of $\varepsilon$ and $\gamma_\varepsilon\geq0$ is the friction coefficient which takes different values on the inner and outer boundaries, i.e.
\begin{align*}
  \gamma_\varepsilon := \gamma_\varepsilon^i \quad\text{on}\quad \Gamma_\varepsilon^i,\,i=0,1,
\end{align*}
where $\gamma_\varepsilon^0$ and $\gamma_\varepsilon^1$ are nonnegative constants.
We further write
\begin{align*}
  \sigma(u^\varepsilon,p^\varepsilon) := 2\nu D(u^\varepsilon)-p^\varepsilon I_3, \quad [\sigma(u^\varepsilon,p^\varepsilon)n_\varepsilon]_{\mathrm{tan}} := P_\varepsilon[\sigma(u^\varepsilon,p^\varepsilon)n_\varepsilon]
\end{align*}
for the stress tensor and the tangential component of the stress vector on $\Gamma_\varepsilon$, where $I_3$ is the $3\times3$ identity matrix, $n_\varepsilon\otimes n_\varepsilon$ is the tensor product of $n_\varepsilon$ with itself, and
\begin{align*}
  D(u^\varepsilon) := \frac{\nabla u^\varepsilon+(\nabla u^\varepsilon)^T}{2}, \quad P_\varepsilon := I_3-n_\varepsilon\otimes n_\varepsilon
\end{align*}
are the strain rate tensor and the orthogonal projection onto the tangent plane of $\Gamma_\varepsilon$, respectively.
Note that $[\sigma(u^\varepsilon,p^\varepsilon)n_\varepsilon]_{\mathrm{tan}}=2\nu P_\varepsilon D(u^\varepsilon)n_\varepsilon$ does not depend on $p^\varepsilon$ and the slip boundary conditions \eqref{E:NS_Bo} can be expressed as
\begin{align*}
  u^\varepsilon\cdot n_\varepsilon = 0, \quad 2\nu P_\varepsilon D(u^\varepsilon)n_\varepsilon+\gamma_\varepsilon u^\varepsilon = 0 \quad\text{on}\quad \Gamma_\varepsilon\times(0,\infty).
\end{align*}
In what follows, we mainly refer to these conditions as the slip boundary conditions.

Partial differential equations in thin domains appear in many applications in solid mechanics (thin elastic bodies), fluid mechanics (lubrication, meteorology, ocean dynamics), etc.
They have been studied for a long time since the pioneering works~\cite{HaRa92a,HaRa92b} by Hale and Raugel on damped wave and reaction-diffusion equations.
In the study of the three-dimensional Navier--Stokes equations in thin domains we expect to get the global-in-time existence of a strong solution for large data according to the smallness of the width of thin domains, since a three-dimensional thin domain with very small width can be considered almost two-dimensional.
Raugel and Sell~\cite{RaSe93} first established the global existence of a strong solution to the Navier--Stokes equations with purely periodic or mixed Dirichlet-periodic boundary conditions in the thin product domain $\Omega_\varepsilon=Q_2\times(0,\varepsilon)$ with a rectangle $Q_2$ and sufficiently small $\varepsilon>0$.
Later, Temam and Ziane~\cite{TeZi96} generalized the results in~\cite{RaSe93} in the case of the thin product domain $\Omega_\varepsilon=\omega\times(0,\varepsilon)$ with a bounded domain $\omega$ in $\mathbb{R}^2$ and other boundary conditions which are combinations of the Dirichlet, free, and periodic boundary conditions.
We refer to~\cite{Hu07,If99,IfRa01,MoTeZi97,Mo99} and the references cited therein for further generalization and improvement on the results in~\cite{RaSe93,TeZi96}.

The above cited papers deal with thin product domains whose boundaries and limit sets are both flat, but there are various kinds of nonflat thin domains in physical problems (see~\cite{Ra95} for examples of nonflat thin domains).
A nonflat thin domain was first considered by Temem and Ziane~\cite{TeZi97}, who studied the Navier--Stokes equations with free boundary conditions in a thin spherical shell
\begin{align*}
  \Omega_\varepsilon = \{x\in\mathbb{R}^3 \mid a<|x|<a+a\varepsilon\}, \quad a>0
\end{align*}
to give a mathematical justification of derivation of the primitive equations for the atmosphere and ocean dynamics (see~\cite{LiTeWa92a,LiTeWa92b,LiTeWa95}).
Another generalization of the shape of a thin domain was made by Iftimie, Raugel, and Sell~\cite{IfRaSe07}, who studied the Navier--Stokes equations in a flat thin domain with a nonflat top boundary
\begin{align*}
  \Omega_\varepsilon = \{x=(x',x_3)\in\mathbb{R}^3 \mid x'\in(0,1)^2,\,0<x_3<\varepsilon g(x')\}, \quad g\colon(0,1)^2\to\mathbb{R}
\end{align*}
with periodic boundary conditions on the lateral boundaries and Navier's slip boundary conditions on the top and bottom boundaries.
Their result was extended by Hoang~\cite{Ho10} and Hoang and Sell~\cite{HoSe10} to a flat thin domain both of whose top and bottom boundaries are not flat (see also~\cite{Ho13} for the study of two-phase flows in a flat thin domain with nonflat top and bottom boundaries).

The slip boundary conditions were proposed by Navier~\cite{Na1823}, which state that the fluid slips on the boundary with velocity proportional to the tangential component of the stress vector.
They arise in the study of the atmosphere and ocean dynamics~\cite{LiTeWa92a,LiTeWa92b,LiTeWa95} and the homogenization of the no-slip boundary condition on a rough boundary~\cite{Hi16,JaMi01}.
In Remark~\ref{R:Limit_Eq} we see that the slip boundary conditions give a ``proper'' viscous term in the sense of~\cite{EbMa70,Ta92} in surface fluid equations derived as the thin width limit of the bulk Navier--Stokes equations.

In this paper we establish the global existence of a strong solution to the Navier--Stokes equations \eqref{E:NS_Eq}--\eqref{E:NS_In} in the curved thin domain $\Omega_\varepsilon$ given by \eqref{E:Intro_CTD} for large data of order $\varepsilon^{-1/2}$.
Our result can be seen as a generalization of the results in the case of flat thin domains~\cite{Ho10,HoSe10,IfRaSe07} (see Remark~\ref{R:Assump_12}).

Another subject of this paper is a singular limit problem for \eqref{E:NS_Eq}--\eqref{E:NS_In} when $\Omega_\varepsilon$ degenerates into the closed surface $\Gamma$ as $\varepsilon\to0$.
We are concerned with derivation of limit equations on $\Gamma$ and comparison of solutions to the bulk and limit equations.
There are several works on the asymptotic behavior of eigenvalues of the Laplacian on a curved thin domain around a hypersurface (see e.g.~\cite{JiKu16,Kr14,Sch96}), but a singular limit problem for evolution equations in curved thin domains has not been studied well.
Temam and Ziane~\cite{TeZi97} first considered this problem for the Navier--Stokes equations in the thin spherical shell and proved the convergence of the average in the thin direction of a solution towards a unique solution to limit equations on a sphere in $\mathbb{R}^3$.
Prizzi, Rinaldi, and Rybakowski~\cite{PrRiRy02} studied a reaction-diffusion equation in a curved thin domain around a lower dimensional manifold and compared the dynamics of the original and limit equations (see also~\cite{PrRy03}).
Later, the present author considered the heat equation in a moving thin domain and derived a limit diffusion equation on its limit evolving surface~\cite{Mi17}.
In the recent work~\cite{Mi18} he also formally derived limit equations of the Navier--Stokes equations in a thin tubular neighborhood of an evolving closed surface.
The purpose of this paper is to give a mathematical justification (and generalization) of the result in~\cite{Mi18} in the case of the stationary curved thin domain of the form \eqref{E:Intro_CTD}.

To state our main results let us fix notations and make main assumptions.
Let $\mathcal{R}$ be the space of all infinitesimal rigid displacements of $\mathbb{R}^3$ whose restrictions on $\Gamma$ are tangential, i.e.
\begin{align} \label{E:Def_R}
  \mathcal{R} := \{u(x)=a\times x+b,\,x\in\mathbb{R}^3 \mid \text{$a,b\in\mathbb{R}^3$, $u|_\Gamma\cdot n=0$ on $\Gamma$}\}.
\end{align}
Clearly $\mathcal{R}$ is of finite dimension.
Also, $\mathcal{R}\neq\{0\}$ if and only if $\Gamma$ is axially symmetric, i.e. invariant under a rotation by any angle around some axis (see Lemma~\ref{L:IR_Surf}).
We define subspaces of $\mathcal{R}$ by
\begin{align} \label{E:Def_Rg}
  \begin{aligned}
    \mathcal{R}_i &:= \{u\in\mathcal{R} \mid \text{$u|_\Gamma\cdot\nabla_\Gamma g_i=0$ on $\Gamma$}\}, \quad i=0,1, \\
    \mathcal{R}_g &:= \{u\in\mathcal{R} \mid \text{$u|_\Gamma\cdot\nabla_\Gamma g=0$ on $\Gamma$}\} \quad (g=g_1-g_0).
  \end{aligned}
\end{align}
Here $P:=I_3-n\otimes n$ is the orthogonal projection onto the tangent plane of $\Gamma$ and $\nabla_\Gamma:=P\nabla$ is the tangential gradient on $\Gamma$ (see Section~\ref{SS:Pre_Surf}).
Note that $\mathcal{R}_0\cap\mathcal{R}_1\subset\mathcal{R}_g$.
It turns out (see Lemmas~\ref{L:CTD_AS} and~\ref{L:CTD_Rg}) that $\Omega_\varepsilon$ is axially symmetric around the same axis for all $\varepsilon\in(0,1)$ if $\mathcal{R}_0\cap\mathcal{R}_1\neq\{0\}$, while $\Omega_\varepsilon$ is not axially symmetric around any axis for sufficiently small $\varepsilon>0$ if $\mathcal{R}_g=\{0\}$.
Next we define the surface strain rate tensor $D_\Gamma(v)$ by
\begin{align} \label{E:Intro_SSR}
  D_\Gamma(v) := P(\nabla_\Gamma v)_SP, \quad (\nabla_\Gamma v)_S := \frac{\nabla_\Gamma v+(\nabla_\Gamma v)^T}{2}
\end{align}
for a (not necessarily tangential) vector field $v$ on $\Gamma$ and set
\begin{align} \label{E:Def_Kil}
  \begin{aligned}
    \mathcal{K}(\Gamma) &:= \{v \in H^1(\Gamma)^3 \mid \text{$v\cdot n=0$, $D_\Gamma(v)=0$ on $\Gamma$}\}, \\
    \mathcal{K}_g(\Gamma) &:= \{v\in\mathcal{K}(\Gamma) \mid \text{$v\cdot\nabla_\Gamma g=0$ on $\Gamma$}\}.
  \end{aligned}
\end{align}
If $\Gamma$ is of class $C^4$, then $v\in\mathcal{K}(\Gamma)$ is in fact of class $C^1$ (see Lemma~\ref{L:Kil_Reg}) and satisfies $\overline{\nabla}_Xv\cdot Y+X\cdot\overline{\nabla}_Yv=0$ on $\Gamma$ for all (continuous) tangential vector fields $X$ and $Y$ on $\Gamma$, where $\overline{\nabla}_Xv:=P(X\cdot\nabla_\Gamma)v$ is the covariant derivative of $v$ along $X$.
Such a vector field generates a one-parameter group of isometries of $\Gamma$ and is called a Killing vector field on $\Gamma$.
It is known that $\mathcal{K}(\Gamma)$ is a Lie algebra of dimension at most three.
For details of Killing vector fields we refer to~\cite{Jo11,Pe06,Ta11_I}.

\begin{remark} \label{R:Killing}
  For $w(x)=a\times x+b$ with $a,b\in\mathbb{R}^3$ direct calculations show that $D_\Gamma(w)=0$ on $\Gamma$.
  Hence if $w$ is tangential on $\Gamma$ then it is Killing on $\Gamma$, i.e. $\mathcal{R}|_\Gamma:=\{w|_\Gamma\mid w\in\mathcal{R}\}\subset\mathcal{K}(\Gamma)$.
  The set $\mathcal{R}|_\Gamma$ represents the extrinsic infinitesimal symmetry of the embedded surface $\Gamma$, while $\mathcal{K}(\Gamma)$ describes the intrinsic one.
  It is known that $\mathcal{R}|_\Gamma=\mathcal{K}(\Gamma)$ if $\Gamma$ is a surface of revolution (see also Lemma~\ref{L:IR_Kil}).
  The same relation holds if $\Gamma$ is closed and convex since any isometry between two closed and convex surfaces in $\mathbb{R}^3$ is a motion in $\mathbb{R}^3$ (a rotation and a translation) or a motion and a reflection by the Cohn-Vossen theorem (see~\cite{Sp79}).
  However, it is not known whether $\mathcal{R}|_\Gamma=\mathcal{K}(\Gamma)$ for a general (nonconvex and not axially symmetric) closed surface.
\end{remark}

In Sections~\ref{S:St}, \ref{S:Tri}, \ref{S:GE}, and~\ref{S:SL} (except for Sections~\ref{SS:St_HL} and~\ref{SS:St_Apr}) we make the following assumptions on the friction coefficients and the limit surface (see also Remarks~\ref{R:Assump_12} and~\ref{R:Recov_NS} for the conditions of Assumption~\ref{Assump_2}).

\begin{assumption} \label{Assump_1}
  There exists a constant $c>0$ such that
  \begin{align} \label{E:Fric_Upper}
    \gamma_\varepsilon^0 \leq c\varepsilon, \quad \gamma_\varepsilon^1 \leq c\varepsilon
  \end{align}
  for all $\varepsilon\in(0,1)$.
\end{assumption}

\begin{assumption} \label{Assump_2}
  Either of the following conditions is satisfied:
  \begin{itemize}
    \item[(A1)] There exists a constant $c>0$ such that
    \begin{align*}
      \gamma_\varepsilon^0 \geq c\varepsilon \quad\text{for all}\quad \varepsilon\in(0,1) \quad\text{or}\quad \gamma_\varepsilon^1 \geq c\varepsilon \quad\text{for all}\quad \varepsilon\in(0,1).
    \end{align*}
    \item[(A2)] The space $\mathcal{K}_g(\Gamma)$ contains only a trivial vector field, i.e. $\mathcal{K}_g(\Gamma)=\{0\}$.
    \item[(A3)] The relations $\mathcal{R}_g=\mathcal{R}_0\cap\mathcal{R}_1$ and $\mathcal{R}_g|_\Gamma:=\{w|_\Gamma\mid w\in\mathcal{R}_g\}=\mathcal{K}_g(\Gamma)$ hold and $\gamma_\varepsilon^0=\gamma_\varepsilon^1=0$ for all $\varepsilon\in(0,1)$.
  \end{itemize}
\end{assumption}

These assumptions are imposed to show that the bilinear form corresponding to the Stokes problem in $\Omega_\varepsilon$ with slip boundary conditions is uniformly bounded and coercive on an appropriate function space.
More precisely, let
\begin{align*}
  L_\sigma^2(\Omega_\varepsilon) = \{u\in L^2(\Omega_\varepsilon)^3 \mid \text{$\mathrm{div}\,u=0$ in $\Omega_\varepsilon$, $u\cdot n_\varepsilon=0$ on $\Gamma_\varepsilon$}\}
\end{align*}
be the standard $L^2$-solenoidal space on $\Omega_\varepsilon$.
We define
\begin{align} \label{E:Def_Heps}
  \begin{aligned}
    \mathcal{H}_\varepsilon &:=
    \begin{cases}
      L_\sigma^2(\Omega_\varepsilon) &\text{if the condition (A1) or (A2) is satisfied}, \\
      L_\sigma^2(\Omega_\varepsilon)\cap\mathcal{R}_g^\perp &\text{if the condition (A3) is satisfied},
  \end{cases} \\
  \mathcal{V}_\varepsilon &:= H^1(\Omega_\varepsilon)^3\cap \mathcal{H}_\varepsilon,
    \end{aligned}
\end{align}
where $\mathcal{R}_g^\perp$ is the orthogonal complement of $\mathcal{R}_g$ in $L^2(\Omega_\varepsilon)^3$.
Note that $\mathcal{R}_g\subset L_\sigma^2(\Omega_\varepsilon)$ under the assumption $\mathcal{R}_g=\mathcal{R}_0\cap\mathcal{R}_1$ by Lemma~\ref{L:IR_Sole}.
In Section~\ref{SS:St_Def} we prove the uniform boundedness and coerciveness in $\varepsilon$ of the bilinear form on $\mathcal{V}_\varepsilon$, which are fundamental for our arguments.

Let $\mathbb{P}_\varepsilon$ be the orthogonal projection from $L^2(\Omega_\varepsilon)^3$ onto $\mathcal{H}_\varepsilon$.
We denote by $A_\varepsilon$ the Stokes operator on $\mathcal{H}_\varepsilon$ associated with slip boundary conditions and by $D(A_\varepsilon)$ its domain (see Section~\ref{S:St}).
With these notations the problem \eqref{E:NS_Eq}--\eqref{E:NS_In} is formulated as an abstract evolution equation in $\mathcal{H_\varepsilon}$:
\begin{align*}
  \partial_tu^\varepsilon+A_\varepsilon u^\varepsilon+\mathbb{P}_\varepsilon(u^\varepsilon\cdot\nabla)u^\varepsilon = \mathbb{P}_\varepsilon f^\varepsilon \quad\text{on}\quad (0,\infty), \quad u^\varepsilon|_{t=0} = u_0^\varepsilon.
\end{align*}
We refer to~\cite{BoFa13,CoFo88,So01,Te79} and the references cited therein for the study of this equation.
For a function $\varphi$ on $\Omega_\varepsilon$ we define its average in the thin direction by
\begin{align*}
  M\varphi(y) := \frac{1}{\varepsilon g(y)}\int_{\varepsilon g_0(y)}^{\varepsilon g_1(y)}\varphi(y+rn(y))\,dr, \quad y\in\Gamma.
\end{align*}
Also, by $M_\tau u:=PMu$ we denote the averaged tangential component of a vector field $u$ on $\Omega_\varepsilon$ (see Section~\ref{SS:Ave_Def}).

Now let us give the main results of this paper.
The first result is the global-in-time existence of a strong solution for large data.

\begin{theorem} \label{T:GE}
  Let $\Omega_\varepsilon$ be the curved thin domain given by \eqref{E:Intro_CTD}.
  Suppose that
  \begin{itemize}
    \item the closed surface $\Gamma$ is of class $C^5$,
    \item $g_0,g_1\in C^4(\Gamma)$ satisfy $g=g_1-g_0\geq c$ on $\Gamma$ with some constant $c>0$, and
    \item Assumptions~\ref{Assump_1} and~\ref{Assump_2} are satisfied.
  \end{itemize}
  Then there exist constants $\varepsilon_0,c_0\in(0,1)$ such that the following statement holds: for each $\varepsilon\in(0,\varepsilon_0)$ suppose that $u_0^\varepsilon\in \mathcal{V}_\varepsilon$ and $f^\varepsilon\in L^\infty(0,\infty;L^2(\Omega_\varepsilon)^3)$ satisfy
  \begin{multline} \label{E:GE_Data}
    \|A_\varepsilon^{1/2}u_0^\varepsilon\|_{L^2(\Omega_\varepsilon)}^2+\|\mathbb{P}_\varepsilon f^\varepsilon\|_{L^\infty(0,\infty;L^2(\Omega_\varepsilon))}^2 \\
    +\|M_\tau u_0^\varepsilon\|_{L^2(\Gamma)}^2+\|M_\tau\mathbb{P}_\varepsilon f^\varepsilon\|_{L^\infty(0,T;H^{-1}(\Gamma,T\Gamma))}^2 \leq c_0\varepsilon^{-1}.
  \end{multline}
  When the condition (A3) of Assumption~\ref{Assump_2} is satisfied, suppose further that $f^\varepsilon(t)\in\mathcal{R}_g^\perp$ for a.a. $t\in(0,\infty)$.
  Then there exists a global-in-time strong solution
  \begin{align*}
    u^\varepsilon \in C([0,\infty);\mathcal{V}_\varepsilon)\cap L_{loc}^2([0,\infty);D(A_\varepsilon))\cap H_{loc}^1([0,\infty);\mathcal{H}_\varepsilon)
  \end{align*}
  to the Navier--Stokes equations \eqref{E:NS_Eq}--\eqref{E:NS_In}.
\end{theorem}

In \eqref{E:GE_Data} we write $H^{-1}(\Gamma,T\Gamma)$ for the dual space of
\begin{align*}
  H^1(\Gamma,T\Gamma) := \{v\in H^1(\Gamma)^3 \mid \text{$v\cdot n=0$ on $\Gamma$}\}.
\end{align*}
Note that $\mathcal{V}_\varepsilon=D(A_\varepsilon^{1/2})$ and the $L^2(\Omega_\varepsilon)$-norm of $A_\varepsilon^{1/2}u$ for $u\in \mathcal{V}_\varepsilon$ is uniformly equivalent in $\varepsilon$ to the canonical $H^1(\Omega_\varepsilon)$-norm of $u$ (see Lemma~\ref{L:Stokes_H1}).
We also establish several estimates for a strong solution in terms of $\varepsilon$.

\begin{theorem} \label{T:UE}
  Let $c_1$, $c_2$, $\alpha$, and $\beta$ be positive constants.
  Under the assumptions of Theorem~\ref{T:GE}, there exists $\varepsilon_1\in(0,1)$ such that the following statement holds: for $\varepsilon\in(0,\varepsilon_1)$ suppose that $u_0^\varepsilon\in \mathcal{V}_\varepsilon$ and $f^\varepsilon\in L^2(0,\infty;L^2(\Omega_\varepsilon)^3)$ satisfy
  \begin{align} \label{E:UE_Data}
    \begin{aligned}
      \|A_\varepsilon^{1/2}u_0^\varepsilon\|_{L^2(\Omega_\varepsilon)}^2+\|\mathbb{P}_\varepsilon f^\varepsilon\|_{L^\infty(0,\infty;L^2(\Omega_\varepsilon))}^2 &\leq c_1\varepsilon^{-1+\alpha}, \\
      \|M_\tau u_0^\varepsilon\|_{L^2(\Gamma)}^2+\|M_\tau\mathbb{P}_\varepsilon f^\varepsilon\|_{L^\infty(0,\infty;H^{-1}(\Gamma,T\Gamma))}^2 &\leq c_2\varepsilon^{-1+\beta}.
    \end{aligned}
  \end{align}
  When the condition (A3) of Assumption~\ref{Assump_2} is satisfied, suppose further that $f^\varepsilon(t)\in\mathcal{R}_g^\perp$ for a.a. $t\in(0,\infty)$.
  Then there exists a global strong solution $u^\varepsilon$ to \eqref{E:NS_Eq}--\eqref{E:NS_In}.
  Moreover, $u^\varepsilon$ satisfies
  \begin{align} \label{E:UE_L2}
    \begin{aligned}
      \|u^\varepsilon(t)\|_{L^2(\Omega_\varepsilon)}^2 &\leq c(\varepsilon^{1+\alpha}+\varepsilon^\beta), \\
      \int_0^t\|u^\varepsilon(s)\|_{H^1(\Omega_\varepsilon)}^2\,ds &\leq c(\varepsilon^{1+\alpha}+\varepsilon^\beta)(1+t)
    \end{aligned}
  \end{align}
  and
  \begin{align} \label{E:UE_H1}
    \begin{aligned}
      \|u^\varepsilon(t)\|_{H^1(\Omega_\varepsilon)}^2 &\leq c(\varepsilon^{-1+\alpha}+\varepsilon^{-1+\beta}), \\
      \int_0^t\|u^\varepsilon(s)\|_{H^2(\Omega_\varepsilon)}^2\,ds &\leq c(\varepsilon^{-1+\alpha}+\varepsilon^{-1+\beta})(1+t)
    \end{aligned}
  \end{align}
  for all $t\geq 0$, where $c>0$ is a constant independent of $\varepsilon$, $u^\varepsilon$, and $t$.
\end{theorem}

The proofs of Theorems~\ref{T:GE} and~\ref{T:UE} are given in Section~\ref{S:GE}.

Next we give results on a singular limit problem for the Navier--Stokes equations \eqref{E:NS_Eq}--\eqref{E:NS_In} as the curved thin domain $\Omega_\varepsilon$ degenerates into the closed surface $\Gamma$.
We define function spaces of tangential vector fields on $\Gamma$ by
\begin{align*}
  L^2(\Gamma,T\Gamma) &:= \{v\in L^2(\Gamma)^3 \mid \text{$v\cdot n=0$ on $\Gamma$}\}, \\
  L_{g\sigma}^2(\Gamma,T\Gamma) &:= \{v\in L^2(\Gamma,T\Gamma) \mid \text{$\mathrm{div}_\Gamma(gv)=0$ on $\Gamma$}\}, \\
  \mathcal{V}_g &:= L_{g\sigma}^2(\Gamma,T\Gamma)\cap H^1(\Gamma,T\Gamma),
\end{align*}
where $\mathrm{div}_\Gamma$ is the surface divergence operator on $\Gamma$ (see Sections~\ref{S:Pre} and~\ref{S:WSol}).

\begin{theorem} \label{T:SL_Weak}
  Under the assumptions of Theorem~\ref{T:GE}, let $u_0^\varepsilon\in \mathcal{V}_\varepsilon$ and $f^\varepsilon\in L^2(0,\infty;L^2(\Omega_\varepsilon)^3)$ for $\varepsilon\in(0,1)$.
  Suppose that the following conditions hold:
  \begin{itemize}
    \item[(a)] There exist $c>0$, $\varepsilon_2\in(0,1)$, and $\alpha\in(0,1)$ such that
    \begin{align*}
      \|A_\varepsilon^{1/2}u_0^\varepsilon\|_{L^2(\Omega_\varepsilon)}^2+\|\mathbb{P}_\varepsilon f^\varepsilon\|_{L^\infty(0,\infty;L^2(\Omega_\varepsilon))}^2 \leq c\varepsilon^{-1+\alpha}
    \end{align*}
    for all $\varepsilon\in(0,\varepsilon_2)$.
    \item[(b)] There exist $v_0\in L^2(\Gamma,T\Gamma)$ and $f\in L^\infty(0,\infty;H^{-1}(\Gamma,T\Gamma))$ such that
    \begin{alignat*}{3}
      \lim_{\varepsilon\to0}M_\tau u_0^\varepsilon &= v_0 &\quad &\text{weakly in} &\quad &L^2(\Gamma,T\Gamma), \\
      \lim_{\varepsilon\to0}M_\tau\mathbb{P}_\varepsilon f^\varepsilon &= f &\quad &\text{weakly-$\star$ in} &\quad &L^\infty(0,\infty;H^{-1}(\Gamma,T\Gamma)).
    \end{alignat*}
    \item[(c)] For $i=0,1$ there exists $\gamma^i\geq0$ such that $\lim_{\varepsilon\to0}\varepsilon^{-1}\gamma_\varepsilon^i=\gamma^i$.
  \end{itemize}
  When the condition (A3) of Assumption~\ref{Assump_2} is satisfied, suppose further that $f^\varepsilon(t)\in\mathcal{R}_g^\perp$ for all $\varepsilon\in(0,\varepsilon_2)$ and a.a. $t\in(0,\infty)$.
  Then there exists $\varepsilon_3\in(0,1)$ such that the problem \eqref{E:NS_Eq}--\eqref{E:NS_In} admits a global strong solution $u^\varepsilon$ for each $\varepsilon\in(0,\varepsilon_3)$ and
  \begin{align*}
    \lim_{\varepsilon\to0}Mu^\varepsilon\cdot n = 0 \quad\text{strongly in}\quad C([0,\infty);L^2(\Gamma)).
  \end{align*}
  Moreover, there exists a vector field
  \begin{align*}
    v \in C([0,\infty);L_{g\sigma}^2(\Gamma,T\Gamma))\cap L_{loc}^2([0,\infty);\mathcal{V}_g)\cap H_{loc}^1([0,\infty);H^{-1}(\Gamma,T\Gamma))
  \end{align*}
  such that
  \begin{alignat*}{3}
    \lim_{\varepsilon\to0}M_\tau u^\varepsilon &= v &\quad &\text{weakly in} &\quad L^2(0,T;H^1(\Gamma,T\Gamma)), \\
    \lim_{\varepsilon\to0}\partial_tM_\tau u^\varepsilon &= \partial_tv &\quad &\text{weakly in} &\quad L^2(0,T;H^{-1}(\Gamma,T\Gamma))
  \end{alignat*}
  for each $T>0$ and $v$ is a unique weak solution to
  \begin{multline} \label{E:Limit_Eq}
    g\Bigl(\partial_tv+\overline{\nabla}_vv\Bigr)-2\nu\left\{P\mathrm{div}_\Gamma[gD_\Gamma(v)]-\frac{1}{g}(\nabla_\Gamma g\otimes \nabla_\Gamma g)v\right\} \\
    +(\gamma^0+\gamma^1)v+g\nabla_\Gamma q = gf \quad\text{on}\quad \Gamma\times(0,\infty)
  \end{multline}
  and
  \begin{align} \label{E:Limit_Div}
    \mathrm{div}_\Gamma(gv) = 0 \quad\text{on}\quad \Gamma\times(0,\infty), \quad v|_{t=0} = v_0 \quad\text{on}\quad \Gamma
  \end{align}
  with an associated pressure $q$.
\end{theorem}

Here $\overline{\nabla}_vv=P(v\cdot\nabla_\Gamma)v$ is the covariant derivative of $v$ along itself and $D_\Gamma(v)$ is the surface strain rate tensor given by \eqref{E:Intro_SSR}.
We give the definition of a weak solution to \eqref{E:Limit_Eq}--\eqref{E:Limit_Div} and prove Theorem~\ref{T:SL_Weak} in Section~\ref{SS:SL_WeCh} (see also Lemma~\ref{L:LW_Pres} for construction of an associated pressure).
Note that the weak limit $v_0$ of $M_\tau u_0^\varepsilon$ actually belongs to $L_{g\sigma}^2(\Gamma,T\Gamma)$, while $M_\tau u_0^\varepsilon$ does not so in general (see Lemma~\ref{L:WC_Sole}).
Also, we do not divide the equations \eqref{E:Limit_Eq} by $g$ since they correspond to the weighted Helmholtz--Leray decomposition of tangential vector fields
\begin{align*}
  v = v_g+g\nabla_\Gamma q \quad\text{in}\quad L^2(\Gamma,T\Gamma), \quad v_g\in L_{g\sigma}^2(\Gamma,T\Gamma),\, g\nabla_\Gamma q\in L_{g\sigma}^2(\Gamma,T\Gamma)^\perp
\end{align*}
established in Section~\ref{SS:WS_HL}.
We further present in Section~\ref{SS:WS_NTS} the Helmholtz--Leray decomposition of vector fields on $\Gamma$ with normal components
\begin{align*}
  v = v_\sigma+\nabla_\Gamma q+qHn \quad\text{in}\quad L^2(\Gamma)^3, \quad v_\sigma \in L_\sigma^2(\Gamma),\,\nabla_\Gamma q+qHn \in L_\sigma^2(\Gamma)^\perp,
\end{align*}
where $L_\sigma^2(\Gamma):=\{v\in L^2(\Gamma)^3\mid\text{$\mathrm{div}_\Gamma v=0$ on $\Gamma$}\}$ and $H$ is (twice) the mean curvature of $\Gamma$, although it is not used in the proof of Theorem~\ref{T:SL_Weak}.

If the weak and weak-$\star$ convergence of $M_\tau u_0^\varepsilon$ and $M_\tau\mathbb{P}_\varepsilon f^\varepsilon$ are replaced by the strong convergence, then we get the strong convergence of $M_\tau u^\varepsilon$.

\begin{theorem} \label{T:SL_Strong}
  For $\varepsilon\in(0,1)$ let $u_0^\varepsilon\in \mathcal{V}_\varepsilon$ and $f^\varepsilon\in L^\infty(0,\infty;L^2(\Omega_\varepsilon)^3)$.
  Suppose that the assumptions of Theorem~\ref{T:SL_Weak} are satisfied with the condition (b) replaced by the following condition:
  \begin{itemize}
    \item [(b')] There exist $v_0\in L^2(\Gamma,T\Gamma)$ and $f\in L^\infty(0,\infty;H^{-1}(\Gamma,T\Gamma))$ such that
    \begin{alignat*}{3}
      \lim_{\varepsilon\to0}M_\tau u_0^\varepsilon &= v_0 &\quad &\text{strongly in} &\quad &L^2(\Gamma,T\Gamma), \\
      \lim_{\varepsilon\to0}M_\tau\mathbb{P}_\varepsilon f^\varepsilon &= f &\quad &\text{strongly in} &\quad &L^\infty(0,\infty;H^{-1}(\Gamma,T\Gamma)).
  \end{alignat*}
  \end{itemize}
  Then the statements of Theorem~\ref{T:SL_Weak} hold.
  Moreover, for each $T>0$ we have
  \begin{align*}
    \lim_{\varepsilon\to0}M_\tau u^\varepsilon = v \quad\text{strongly in}\quad C([0,T];L^2(\Gamma,T\Gamma))\cap L^2(0,T;H^1(\Gamma,T\Gamma)).
  \end{align*}
\end{theorem}

We give an estimate for the difference between $M_\tau u^\varepsilon$ and $v$ (see Theorem~\ref{T:Diff_Mu_V}) and prove Theorem~\ref{T:SL_Strong} in Section~\ref{SS:SL_ErST}.
Moreover, we compare $u^\varepsilon$ and the constant extension of $v$ in the normal direction of $\Gamma$ in $\Omega_\varepsilon$ (see Theorem~\ref{T:Diff_Ue_Cv}).
It is worth noting that, if we define the normal derivative of $u^\varepsilon$ with respect to $\Gamma$ as
\begin{align*}
  \partial_nu^\varepsilon(x,t) := \frac{d}{d\tilde{r}}\bigl(u^\varepsilon(y+\tilde{r}n(y),t)\bigr)\Big|_{\tilde{r}=r}, \quad x = y+rn(y) \in \Omega_\varepsilon,
\end{align*}
then $\partial_nu^\varepsilon$ is close in an appropriate sense to a surface vector field of the form
\begin{align*}
  -W(y)v(y,t)+\frac{1}{g(y)}\{v(y,t)\cdot\nabla_\Gamma g(y)\}n(y), \quad (y,t)\in\Gamma\times(0,\infty),
\end{align*}
see Theorem~\ref{T:Diff_DnUe_Cv} for details.
Here $W:=-\nabla_\Gamma n$ is the Weingarten map (or shape operator) of $\Gamma$.
Therefore, the normal derivative $\partial_nu^\varepsilon$ is not necessarily small even though the curved thin domain $\Omega_\varepsilon$ and the limit surface $\Gamma$ are stationary.

Let us explain the idea of the proofs of our main results.
In the proof of the global existence of a strong solution (Theorems~\ref{T:GE}) we follow the arguments in~\cite{Ho10,HoSe10} to show that the $H^1(\Omega_\varepsilon)$-norm of a strong solution is uniformly bounded in time by a standard energy method, which also applies to the proof of Theorem~\ref{T:UE}.
The main tools for the proof are an extension of a surface vector field to $\Omega_\varepsilon$ that satisfies the impermeable boundary condition $u\cdot n_\varepsilon=0$ on $\Gamma_\varepsilon$ given in Section~\ref{SS:Tool_TE} and average operators in the thin direction defined and investigated in Section~\ref{S:Ave}.
Using them, the slip boundary conditions, and Sobolev type inequalities on $\Omega_\varepsilon$ and $\Gamma$, we derive a good estimate for the trilinear term
\begin{align*}
  \bigl((u\cdot\nabla)u,A_\varepsilon u\bigr)_{L^2(\Omega_\varepsilon)}, \quad u\in D(A_\varepsilon)
\end{align*}
in Section~\ref{S:Tri}.
A key idea for derivation is to decompose a vector field on $\Omega_\varepsilon$ into the average part, which is almost two-dimensional, and the residual part, which satisfies the impermeable boundary condition (see Section~\ref{SS:Ave_Ex}).
Such decomposition enables us to apply a product estimate for a function on $\Omega_\varepsilon$ and that on $\Gamma$ to the average part (see Lemma~\ref{L:Prod_Ua}) and a good $L^\infty$-estimate following from the Agmon and Poincar\'{e} inequalities to the residual part (see Lemma~\ref{L:Linf_Ur}).

For the proof of the global existence and uniform estimates of a strong solution, we also require the uniform equivalence of the norms
\begin{align*}
  c^{-1}\|u\|_{H^1(\Omega_\varepsilon)} \leq \|A_\varepsilon^{1/2}u\|_{L^2(\Omega_\varepsilon)} \leq c\|u\|_{H^1(\Omega_\varepsilon)}, \quad u\in\mathcal{V}_\varepsilon=D(A_\varepsilon^{1/2})
\end{align*}
with a constant $c>0$ independent of $\varepsilon$ (see Lemma~\ref{L:Stokes_H1}).
This is a consequence of the uniform boundedness and coerciveness of the bilinear form corresponding to the Stokes problem in $\Omega_\varepsilon$ (see Lemma~\ref{L:Bili_Core}), for which the uniform Korn inequalities on $\Omega_\varepsilon$ established in Section~\ref{SS:Korn_Dom} and Assumptions~\ref{Assump_1} and~\ref{Assump_2} are essential.
It is more difficult to show the uniform norm equivalence (see Lemma~\ref{L:Stokes_H2})
\begin{align*}
  c^{-1}\|u\|_{H^2(\Omega_\varepsilon)} \leq \|A_\varepsilon u\|_{L^2(\Omega_\varepsilon)} \leq c\|u\|_{H^2(\Omega_\varepsilon)}, \quad u\in D(A_\varepsilon).
\end{align*}
The right-hand inequality follows from a uniform $L^2$-estimate for the difference between the Stokes and Laplace operators (see Lemma~\ref{L:Comp_Sto_Lap}).
To prove the left-hand inequality we derive a uniform a priori estimate for the vector Laplacian with slip boundary conditions, which involves calculations of the covariant derivatives on the boundary $\Gamma_\varepsilon$ (see Section~\ref{SS:St_Apr}).

To prove Theorems~\ref{T:SL_Weak} and~\ref{T:SL_Strong} on a singular limit problem as $\varepsilon\to0$, we proceed as in~\cite{Mi17} to transform a weak formulation for \eqref{E:NS_Eq}--\eqref{E:NS_In} into that for the averaged tangential component of a strong solution to \eqref{E:NS_Eq}--\eqref{E:NS_In} with residual terms (see Section~\ref{SS:SL_Ave}).
For this purpose, we approximate bilinear and trilinear forms in the weak formulation for \eqref{E:NS_Eq}--\eqref{E:NS_In} by those in a weak formulation for \eqref{E:Limit_Eq}--\eqref{E:Limit_Div} by using the slip boundary conditions and the average operators (see Section~\ref{SS:Ave_BT}).
Moreover, we need to construct an appropriate test function for \eqref{E:NS_Eq}--\eqref{E:NS_In} from a test function for \eqref{E:Limit_Eq}--\eqref{E:Limit_Div}, which is a weighted solenoidal vector field on $\Gamma$.
To this end, we use the impermeable extension of a surface vector field to $\Omega_\varepsilon$ and a uniform $H^1$-estimate for the gradient part of the Helmholtz--Leray decomposition on $\Omega_\varepsilon$ given in Sections~\ref{SS:Tool_TE} and~\ref{SS:St_HL}, respectively.

After transformation of the weak formulation, we derive the energy estimate for the averaged tangential component of a strong solution to \eqref{E:NS_Eq}--\eqref{E:NS_In} by using its weak formulation in Section~\ref{SS:SL_Ener}.
In derivation of the energy estimate we cannot substitute the averaged tangential component itself for its weak formulation since it is not weighted solenoidal on $\Gamma$ in general.
To overcome this difficulty we employ the weighted Helmholtz--Leray projection on $\Gamma$ given in Section~\ref{SS:WS_HL} to replace the averaged tangential component with its weighted solenoidal part.
Then we derive the energy estimate for the weighted solenoidal part by substituting it for its weak formulation and apply estimates for the gradient part of the weighted Helmholtz--Leray decomposition on $\Gamma$ to obtain the energy estimate for the original averaged tangential component, which enables us to show that it converges weakly as $\varepsilon\to0$ and that the limit is a weak solution to \eqref{E:Limit_Eq}--\eqref{E:Limit_Div} (see Section~\ref{SS:SL_WeCh}).
We also derive an estimate for the difference between the averaged tangential component of a strong solution to \eqref{E:NS_Eq}--\eqref{E:NS_In} and a weak solution to \eqref{E:Limit_Eq}--\eqref{E:Limit_Div} by using their weak formulations in Section~\ref{SS:SL_ErST}.
Here we again use the weighted Helmholtz--Leray projection on $\Gamma$ to take the difference of the solutions as a test function in the weak formulations.

Now let us give remarks on Assumption~\ref{Assump_2} and comparison of the limit equations \eqref{E:Limit_Eq}--\eqref{E:Limit_Div} with the Navier--Stokes equations on a Riemannian manifold and limit equations derived in~\cite{TeZi97}.

\begin{remark} \label{R:Assump_12}
  The conditions of Assumption~\ref{Assump_2} are valid in the following cases:
  \begin{itemize}
    \item[(A1)] If $\gamma_\varepsilon^0$ or $\gamma_\varepsilon^1$ is bounded from below by $\varepsilon$, then we may consider any closed surface $\Gamma$.
    In this case, however, at least one of the friction coefficients $\gamma^0$ or $\gamma^1$ in the limit momentum equations \eqref{E:Limit_Eq} must be positive.
    \item[(A2)] It is known (see e.g.~\cite[Proposition~2.2]{Sh_18pre}) that there exists no nontrivial Killing vector field on $\Gamma$, i.e. $\mathcal{K}(\Gamma)=\{0\}$ if the genus of $\Gamma$ is greater than one.
    In this case $\mathcal{K}_g(\Gamma)=\{0\}$ for any $g=g_1-g_0$ and we may take arbitrary nonnegative $\gamma_\varepsilon^0$ and $\gamma_\varepsilon^1$ (bounded above by $\varepsilon$).
    Note that if $\mathcal{K}_g(\Gamma)=\{0\}$ then $\mathcal{R}_g=\{0\}$ and $\Omega_\varepsilon$ is not axially symmetric around any axis for sufficiently small $\varepsilon>0$ (see Lemma~\ref{L:CTD_Rg}).
    \item[(A3)] As mentioned in Remark~\ref{R:Killing}, if $\Gamma$ is a surface of revolution or it is closed and convex then $\mathcal{R}|_\Gamma=\mathcal{K}(\Gamma)$ and thus $\mathcal{R}_g|_\Gamma=\mathcal{K}_g(\Gamma)$ for any $g=g_1-g_0$.
    Also, $\mathcal{R}_0\cap\mathcal{R}_1=\mathcal{R}_g$ is valid if, for example, $g_0$ or $g_1$ is constant.
    In this case we only consider the perfect slip boundary conditions
    \begin{align} \label{E:Bo_Per}
      u\cdot n_\varepsilon = 0, \quad P_\varepsilon D(u)n_\varepsilon = 0 \quad\text{on}\quad \Gamma_\varepsilon.
    \end{align}
    A typical example of this case is a thin spherical shell
    \begin{align*}
      \Omega_\varepsilon = \{x\in\mathbb{R}^3 \mid 1<|x|<1+\varepsilon\} \quad (\Gamma = S^2,\,g_0=0,\,g_1=1)
    \end{align*}
    around the unit sphere $S^2$ in $\mathbb{R}^3$ studied by Temam and Ziane~\cite{TeZi97} under different boundary conditions (see Remark~\ref{R:Limit_TZ}).
  \end{itemize}
  Note that in the case of a flat thin domain around the flat torus $\Gamma=\mathbb{T}^2$ we have
  \begin{align*}
    \mathcal{R}_i &= \{(a_1,a_2,0)\in\mathbb{R}^2\times\{0\} \mid \text{$a_1\partial_1g_i+a_2\partial_2g_i=0$ on $\mathbb{T}^2$}\}, \quad i=0,1, \\
    \mathcal{R}_g &= \mathcal{K}_g(\Gamma) = \{(a_1,a_2,0)\in\mathbb{R}^2\times\{0\} \mid \text{$a_1\partial_1g+a_2\partial_2g=0$ on $\mathbb{T}^2$}\}
  \end{align*}
  and the conditions (A2) and (A3) are imposed in~\cite{Ho10} and~\cite{HoSe10,IfRaSe07}, respectively.
  Hence our result on the global existence of a strong solution (Theorem~\ref{T:GE}) can be seen as a generalization of the results in~\cite{Ho10,HoSe10,IfRaSe07} for flat thin domains.
\end{remark}

\begin{remark} \label{R:Recov_NS}
  When the condition (A3) of Assumption~\ref{Assump_2} is satisfied, we should be careful about recovery of the original equations \eqref{E:NS_Eq}--\eqref{E:NS_In} from their abstract form.
  In this case Theorem~\ref{T:GE} a priori provides a global solution
  \begin{align*}
    u^\varepsilon \in C([0,\infty);\mathcal{V}_\varepsilon)\cap L_{loc}^2([0,\infty);D(A_\varepsilon))\cap H_{loc}^1([0,\infty);\mathcal{H}_\varepsilon)
  \end{align*}
  to the abstract evolution equation in $\mathcal{H}_\varepsilon=L_\sigma^2(\Omega_\varepsilon)\cap\mathcal{R}_g^\perp$:
  \begin{align} \label{E:NS_Abst_H}
    \partial_tu^\varepsilon(t)+A_\varepsilon u(t)+[\mathbb{P}_\varepsilon(u^\varepsilon\cdot\nabla)u^\varepsilon](t) = \mathbb{P}_\varepsilon f^\varepsilon(t), \quad t\in(0,\infty).
  \end{align}
  Under the assumption $\mathcal{R}_g=\mathcal{R}_0\cap\mathcal{R}_1$ the set $\mathcal{R}_g$ is a closed subspace of $L_\sigma^2(\Omega_\varepsilon)$ (see Lemma~\ref{L:IR_Sole}).
  Hence we have the orthogonal decompositions
  \begin{align*}
    L_\sigma^2(\Omega_\varepsilon) = \mathcal{H}_\varepsilon\oplus\mathcal{R}_g, \quad L^2(\Omega_\varepsilon)^3 = \mathcal{H}_\varepsilon\oplus\mathcal{R}_g\oplus G^2(\Omega_\varepsilon),
  \end{align*}
  where $G^2(\Omega_\varepsilon)=\{\nabla q \mid q\in H^1(\Omega_\varepsilon)\}$.
  By these decompositions we observe that the equations in $L^2(\Omega_\varepsilon)^3$ recovered from \eqref{E:NS_Abst_H} are of the form
  \begin{align} \label{E:NS_Abst_L2}
    \partial_tu^\varepsilon(t)-\nu\Delta u^\varepsilon(t)+[(u^\varepsilon\cdot\nabla)u^\varepsilon](t)+\nabla p^\varepsilon(t)+v(t) = f^\varepsilon(t), \quad t\in(0,\infty)
  \end{align}
  with $\nabla p^\varepsilon(t)\in G^2(\Omega_\varepsilon)$ and $v(x,t)=a(t)\times x+b(t)\in\mathcal{R}_g$.
  At this point the original momentum equations of \eqref{E:NS_Eq} are not recovered properly because of the additional vector field $v$, but we can show $v(t)=0$ if we assume $f^\varepsilon(t)\in\mathcal{R}_g^\perp$ for a.a. $t\in(0,\infty)$.
  To see this, we take the $L^2$-inner product on $\Omega_\varepsilon$ of both sides of \eqref{E:NS_Abst_L2} with $v$ (here and hereafter we fix and suppress $t$).
  Then since $\partial_tu^\varepsilon\in\mathcal{H}_\varepsilon$, $\nabla p^\varepsilon\in G^2(\Omega_\varepsilon)$, and $f^\varepsilon$ are orthogonal to $v\in\mathcal{R}_g$,
  \begin{align*}
    \|v\|_{L^2(\Omega_\varepsilon)}^2 = \nu(\Delta u^\varepsilon,v)_{L^2(\Omega_\varepsilon)}-\bigl((u^\varepsilon\cdot\nabla)u^\varepsilon,v\bigr)_{L^2(\Omega_\varepsilon)}.
  \end{align*}
  By integration by parts (see \eqref{E:IbP_St} in Section~\ref{SS:St_Def}), $\mathrm{div}\,u^\varepsilon=0$ in $\Omega_\varepsilon$, the perfect slip boundary conditions \eqref{E:Bo_Per} on $u^\varepsilon$, and $v\cdot n_\varepsilon=0$ on $\Gamma_\varepsilon$ we have
  \begin{align*}
    (\Delta u^\varepsilon,v)_{L^2(\Omega_\varepsilon)} = -2\bigl(D(u^\varepsilon),D(v)\bigl)_{L^2(\Omega_\varepsilon)}.
  \end{align*}
  (Note that $v\in\mathcal{R}_g$ satisfies $v\cdot n_\varepsilon=0$ on $\Gamma_\varepsilon$ by $\mathcal{R}_g=\mathcal{R}_0\cap\mathcal{R}_1$ and Lemma~\ref{L:IR_CTD}.)
  Moreover, since $v(x,t)=a(t)\times x+b(t)$ and $a(t)$ and $b(t)$ are independent of $x$, it follows that $D(v)=0$ in $\mathbb{R}^3$ and thus $(\Delta u^\varepsilon,v)_{L^2(\Omega_\varepsilon)}=0$.
  By integration by parts, $\mathrm{div}\,u^\varepsilon=0$ in $\Omega_\varepsilon$, and $u^\varepsilon\cdot n_\varepsilon=0$ on $\Gamma_\varepsilon$ we also have
  \begin{align*}
    \bigl((u^\varepsilon\cdot\nabla)u^\varepsilon,v\bigr)_{L^2(\Omega_\varepsilon)} = -(u^\varepsilon,(u^\varepsilon\cdot\nabla)v)_{L^2(\Omega_\varepsilon)} = -(u^\varepsilon,a\times u^\varepsilon)_{L^2(\Omega_\varepsilon)} = 0.
  \end{align*}
  From these equalities we deduce that $\|v\|_{L^2(\Omega_\varepsilon)}^2=0$, i.e. $v=0$ in \eqref{E:NS_Abst_L2}.
  Therefore, the momentum equations of \eqref{E:NS_Eq} are properly recovered from the abstract evolution equation \eqref{E:NS_Abst_H}.
  Note that our arguments in this paper may apply to the proof of the global existence of a solution to the abstract evolution equation \eqref{E:NS_Abst_H} in
  \begin{align*}
    \mathcal{H}_\varepsilon' := \{u\in L_\sigma^2(\Omega_\varepsilon) \mid \text{$(u,\bar{v})_{L^2(\Omega_\varepsilon)}=0$ for all $v\in\mathcal{K}_g(\Gamma)$}\}
  \end{align*}
  even if the condition $\mathcal{R}_g|_\Gamma=\mathcal{K}_g(\Gamma)$ is not valid.
  Here $\bar{v}$ is the constant extension of $v\in\mathcal{K}_g(\Gamma)$ in the normal direction of $\Gamma$ and the orthogonality condition with $\mathcal{K}_g(\Gamma)$ is required to verify the uniform coerciveness on $H^1(\Omega_\varepsilon)\cap\mathcal{H}_\varepsilon'$ of the bilinear form corresponding to the Stokes problem in $\Omega_\varepsilon$ (see Lemma~\ref{L:Korn_H1} and Section~\ref{SS:St_Def}).
  In this case, however, the equations \eqref{E:NS_Abst_L2} in $L^2(\Omega_\varepsilon)^3$ recovered from \eqref{E:NS_Abst_H} in $\mathcal{H}_\varepsilon'$ contain the constant extension $\bar{v}(t)$ of some $v(t)\in\mathcal{K}_g(\Gamma)$ and we cannot show $\bar{v}(t)=0$ as above.
  Indeed, for the constant extension of a vector field $v$ on $\Gamma$ we have
  \begin{multline*}
    D(\bar{v})(x) = \frac{1}{2}\Bigl[\{I_3-rW(y)\}^{-1}\nabla_\Gamma v(y)+\{\nabla_\Gamma v(y)\}^T\{I_3-rW(y)\}^{-1}\Bigr], \\
    x = y+rn(y)\in\Omega_\varepsilon,\,y\in\Gamma,\,r\in(\varepsilon g_0(y),\varepsilon g_1(y))
  \end{multline*}
  by \eqref{E:ConDer_Dom}, where $W$ is the Weingarten map of $\Gamma$, and $D(\bar{v})$ does not vanish in general just by $D_\Gamma(v)=0$ on $\Gamma$ (even if $\Gamma=S^2$ and $\bar{v}(x)=e_3\times(x/|x|)$ is the constant extension of $v(y)=e_3\times y\in\mathcal{K}(S^2)$ with $e_3=(0,0,1)$).
  Hence we cannot recover the momentum equations of \eqref{E:NS_Eq} from \eqref{E:NS_Abst_H} in $\mathcal{H}_\varepsilon'$ (similarly we fail to show $v(t)=0$ in \eqref{E:NS_Abst_L2} if $\mathcal{R}_g\neq\mathcal{R}_0\cap\mathcal{R}_1$ or $\gamma_\varepsilon^0>0$ or $\gamma_\varepsilon^1>0$).
  To avoid this problem, we impose the condition (A3) and consider \eqref{E:NS_Abst_H} in $\mathcal{H}_\varepsilon=L_\sigma^2(\Omega_\varepsilon)\cap\mathcal{R}_g^\perp$.
\end{remark}

\begin{remark} \label{R:Limit_Eq}
  If $g\equiv1$ and $\gamma^0=\gamma^1=0$ in \eqref{E:Limit_Eq}--\eqref{E:Limit_Div} then we have
  \begin{align*}
    \partial_tv+\overline{\nabla}_vv-2\nu P\mathrm{div}_\Gamma[D_\Gamma(v)]+\nabla_\Gamma q = f, \quad \mathrm{div}_\Gamma v = 0 \quad\text{on}\quad \Gamma\times(0,\infty).
  \end{align*}
  In~\cite[Lemma~2.5]{Mi18} it is shown that
  \begin{align*}
    2P\mathrm{div}_\Gamma[D_\Gamma(v)] = \Delta_Bv+Kv \quad\text{on}\quad \Gamma
  \end{align*}
  for a tangential vector field $v$ on $\Gamma$ satisfying $\mathrm{div}_\Gamma v=0$ on $\Gamma$, where $\Delta_B=-\overline{\nabla}^\ast\overline{\nabla}$ is the Bochner Laplacian (see Appendix~\ref{S:Ap_RC}) and $K$ is the Gaussian curvature of $\Gamma$ (see Section~\ref{SS:Pre_Surf}).
  Hence the above equations become
  \begin{align} \label{E:Limit_Reduce}
    \partial_tv+\overline{\nabla}_vv-\nu(\Delta_Bv+Kv)+\nabla_\Gamma q = f, \quad \mathrm{div}_\Gamma v = 0 \quad\text{on}\quad \Gamma\times(0,\infty).
  \end{align}
  These equations are called the ``proper'' Navier--Stokes equations on a Riemannian manifold in~\cite{EbMa70,Ta92} and were studied by Mitrea and Taylor~\cite{MitTa01}, Nagasawa~\cite{Nag99}, and Taylor~\cite{Ta92}.
  Note that $Kv=\mathrm{Ric}(v)$ for a tangential vector field $v$ on the embedded surface $\Gamma$ in $\mathbb{R}^3$, where $\mathrm{Ric}$ is the Ricci curvature.
  Also, the Bochner Laplacian is related to the Hodge Laplacian $\Delta_D=-(d_\Gamma d_\Gamma^\ast+d_\Gamma^\ast d_\Gamma)$ with the exterior derivative $d_\Gamma$ by the Weitzenb\"{o}ck formula $\Delta_D=\Delta_B-\mathrm{Ric}$ (see~\cite{Jo11,Pe06}).
\end{remark}

\begin{remark} \label{R:Limit_TZ}
  Temam and Ziane~\cite{TeZi97} studied the Navier--Stokes equations in the thin spherical shell $\Omega_\varepsilon = \{x\in\mathbb{R}^3 \mid 1<|x|<1+\varepsilon\}$ whose limit surface is the unit sphere $S^2$ in $\mathbb{R}^3$.
  They imposed the Hodge (or de Rham) boundary conditions
  \begin{align*}
    u\cdot n_\varepsilon = 0, \quad \mathrm{curl}\,u\times n_\varepsilon = 0 \quad\text{on}\quad \Gamma_\varepsilon,
  \end{align*}
  which is called the free boundary conditions in~\cite{TeZi97}, and mentioned that the Hodge boundary conditions are equivalent to the perfect slip boundary conditions \eqref{E:Bo_Per}.
  However, it is not true for the thin spherical shell.
  Indeed, for $u(x)=e_3\times x$ we easily see that $D(u)=0$ and $\mathrm{curl}\,u=2e_3$ in $\mathbb{R}^3$.
  Hence we have
  \begin{align*}
    u\cdot n_\varepsilon = \pm(e_3\times x)\cdot \frac{x}{|x|} = 0, \quad P_\varepsilon D(u)n_\varepsilon = 0, \quad \mathrm{curl}\,u\times n_\varepsilon = \pm2e_3\times \frac{x}{|x|}
  \end{align*}
  for $x\in \Gamma_\varepsilon$ (note that $n_\varepsilon(x)=\pm x/|x|$) and the last vector does not vanish in general.
  More generally, the Hodge boundary conditions are equivalent to the perfect slip boundary conditions only when a boundary is a part of a plane.
  For example, on the plane $\{(x',0)\mid x'\in\mathbb{R}^2\}$ these boundary conditions reduce to
  \begin{align*}
    u_3(x',0) = 0, \quad \partial_3u_1(x',0) = \partial_3u_2(x',0) = 0, \quad x'\in\mathbb{R}^2.
  \end{align*}
  The difference between the Hodge and perfect slip boundary conditions is due to the curvature of a boundary.
  Moreover, it makes limit equations in~\cite{TeZi97} different from our limit equations.
  In~\cite[Section~6]{TeZi97} the authors derived the limit equations
  \begin{align*}
    \partial_tv+\overline{\nabla}_vv-\nu\Delta_2 v+\nabla_\Gamma q = f, \quad \mathrm{div}_\Gamma v = 0 \quad\text{on}\quad S^2\times(0,\infty)
  \end{align*}
  from the Navier--Stokes equations in the thin spherical shell with Hodge boundary conditions.
  Here $\Delta_2v=P\Delta\bar{v}$ is the tangential vector Laplacian of a tangential vector field $v$ on $S^2$ with its constant extension $\bar{v}$ in the normal direction of $S^2$ (see \cite[Appendix]{TeZi97}).
  In terms of our notations given in Section~\ref{SS:Pre_Surf} it is expressed as
  \begin{align*}
    \Delta_2v = P\Delta \bar{v} = P\Delta_\Gamma v = \Delta_Bv-W^2v \quad\text{on}\quad S^2
  \end{align*}
  by Lemmas~\ref{L:Pi_Der} and~\ref{L:Boch_Lap}.
  Moreover, since the Weingarten map of $S^2$ is $W=-P$ and $v$ is tangential on $S^2$, we have $W^2v=v$.
  Thus the limit equations in~\cite{TeZi97} read
  \begin{align} \label{E:Limit_TZ}
    \partial_tv+\overline{\nabla}_vv-\nu(\Delta_Bv-v)+\nabla_\Gamma q = f, \quad \mathrm{div}_\Gamma v = 0 \quad\text{on}\quad S^2\times(0,\infty).
  \end{align}
  On the other hand, in the case of the thin spherical shell with perfect slip boundary conditions we have $g\equiv1$ and $\gamma^0=\gamma^1=0$ (the condition (A3) of Assumption~\ref{Assump_2} is satisfied).
  Hence our limit equations are of the form \eqref{E:Limit_Reduce} and, since $K=1$ for the unit sphere $S^2$, they further become
  \begin{align*}
    \partial_tv+\overline{\nabla}_vv-\nu(\Delta_Bv+v)+\nabla_\Gamma q = f, \quad \mathrm{div}_\Gamma v = 0 \quad\text{on}\quad S^2\times(0,\infty).
  \end{align*}
  The sign of the zeroth order term $v$ in the viscous term of this system is opposite to that of \eqref{E:Limit_TZ}, which produces different structures of the limit equations such as the stability of a solution.
\end{remark}

This paper is organized as follows.
In Section~\ref{S:Pre} we introduce notations for surface quantities and function spaces on a closed surface and give the definition and basic properties of a curved thin domain.
We present in Section~\ref{S:Tool} fundamental formulas and inequalities for functions on the surface and the thin domain.
In Section~\ref{S:Korn} we show the uniform Korn inequalities on the thin domain and the Korn inequalities on the surface.
The Stokes operator on the thin domain with slip boundary conditions is investigated in Section~\ref{S:St}.
We also present a uniform estimate for the gradient part of the Helmholtz--Leray decomposition on the thin domain used in the study of a singular limit problem.
In Section~\ref{S:Ave} we define average operators in the thin direction and derive several estimates for the average of functions on the thin domain.
The main purpose of Section~\ref{S:Ave} is to provide decomposition of a vector field on the thin domain into the average and residual parts and to establish useful estimates for them.
We also use the average operators to approximate bilinear and trilinear forms for functions on the thin domain by those for functions on the surface.
In Section~\ref{S:Tri} we derive a good estimate for the trilinear term, i.e. the $L^2$-inner product of the inertial and viscous terms.
The main ingredients for derivation are the estimates for the Stokes and average operators given in Sections~\ref{S:St} and~\ref{S:Ave}.
Based on the estimate for the trilinear term we establish in Section~\ref{S:GE} the global existence and uniform estimates of a strong solution to the Navier--Stokes equations in the curved thin domain (Theorems~\ref{T:GE} and~\ref{T:UE}).
The last two sections are devoted to the study of a singular limit problem when the curved thin domain degenerates into the closed surface.
In Section~\ref{S:WSol} we investigate weighted solenoidal spaces on a closed surface.
We give characterization of the annihilator of a weighted solenoidal space and prove the weighted Helmholtz--Leray decomposition on the surface with several estimates for the gradient part.
In Section~\ref{S:SL} we study the behavior of the average in the thin direction of a strong solution to the Navier--Stokes equations.
Our goal is to establish the convergence of the average towards a weak solution to the limit equations (Theorems~\ref{T:SL_Weak} and~\ref{T:SL_Strong}).
We fix notations on vectors and matrices in Appendix~\ref{S:Ap_VM} and prove a few lemmas of elementary vector calculus.
In Appendix~\ref{S:Ap_DG} we provide the proofs of lemmas in Section~\ref{S:Pre} involving calculations of surface quantities of the surface and the boundary of the thin domain.
We introduce the Riemannian connection on a surface in Appendix~\ref{S:Ap_RC} as well as present formulas on the covariant derivatives on the surface.
In Appendix~\ref{S:Ap_IR} we give several results on infinitesimal rigid displacements of $\mathbb{R}^3$ which have tangential restrictions on a closed surface.

\section{Preliminaries} \label{S:Pre}
In this section we give notations and formulas on several quantities for a two-dimensional closed surface and a thin domain in $\mathbb{R}^3$.
Some lemmas in this section are proved just by direct calculations involving differential geometry of surfaces.
We give their proofs in Appendix~\ref{S:Ap_DG} to avoid making this section too long.

Throughout this paper we denote by $c$ a general positive constant independent of the parameter $\varepsilon$.
We fix a coordinate system of $\mathbb{R}^3$ and write $x_i$, $i=1,2,3$ for the $i$-th component of a point $x\in\mathbb{R}^3$ under the fixed coordinate system.
For a vector $a\in\mathbb{R}^3$ we denote by $a_i$, $i=1,2,3$ the $i$-th component of $a$.
Sometimes we write $a^i$ or $[a]_i$ instead of $a_i$.
We also denote by $A_{ij}$ or $[A]_{ij}$, $i,j=1,2,3$ the $(i,j)$-entry of a matrix $A\in\mathbb{R}^{3\times3}$ and by $I_3$ the identity matrix of size three.
Other notations and basic formulas on vectors and matrices are given in Appendix~\ref{S:Ap_VM}.

\subsection{Closed surface} \label{SS:Pre_Surf}
Let $\Gamma$ be a two-dimensional closed (i.e. compact and without boundary), connected, and oriented surface in $\mathbb{R}^3$.
We assume that $\Gamma$ is of class $C^\ell$ with $\ell\geq 2$.
By $n$ and $d$ we denote the unit outward normal vector field of $\Gamma$ and the signed distance function from $\Gamma$ increasing in the direction of $n$.
Also, let $\kappa_1$ and $\kappa_2$ be the principal curvatures of $\Gamma$.
By the regularity of $\Gamma$ we have $n\in C^{\ell-1}(\Gamma)^3$ and $\kappa_1,\kappa_2\in C^{\ell-2}(\Gamma)$.
In particular, $\kappa_1$ and $\kappa_2$ are bounded on the compact set $\Gamma$.
Hence we can take a tubular neighborhood $N:=\{x\in\mathbb{R}^3\mid \mathrm{dist}(x,\Gamma)<\delta\}$, $\delta>0$ of $\Gamma$ such that for each $x\in N$ there exists a unique point $\pi(x)\in\Gamma$ satisfying
\begin{align} \label{E:Nor_Coord}
  x = \pi(x)+d(x)n(\pi(x)), \quad \nabla d(x) = n(\pi(x)).
\end{align}
Moreover, $d$ and $\pi$ are of class $C^\ell$ and $C^{\ell-1}$ on $\overline{N}$ (see~\cite[Section~14.6]{GiTr01} for details).
By the boundedness of $\kappa_1$ and $\kappa_2$ we also have
\begin{align} \label{E:Curv_Bound}
  c^{-1} \leq 1-r\kappa_i(y) \leq c \quad\text{for all}\quad y\in\Gamma,\,r\in(-\delta,\delta),\,i=1,2
\end{align}
if we take $\delta>0$ sufficiently small.

Let us define differential operators on $\Gamma$.
For $y\in\Gamma$ we set
\begin{align*}
  P(y) := I_3-n(y)\otimes n(y), \quad Q(y) := n(y)\otimes n(y).
\end{align*}
The matrices $P$ and $Q$ are the orthogonal projections onto the tangent plane and the normal direction of $\Gamma$ and satisfy $|P|=2$, $|Q|=1$, and
\begin{gather*}
  I_3 = P+Q, \quad PQ = QP = 0, \quad P^T = P^2 = P, \quad Q^T = Q^2 = Q, \\
  |a|^2 = |Pa|^2+|Qa|^2, \quad |Pa| \leq |a|, \quad Pa\cdot n = 0, \quad a\in\mathbb{R}^3
\end{gather*}
on $\Gamma$.
Also, $P,Q\in C^{\ell-1}(\Gamma)^{3\times3}$ by the $C^\ell$-regularity of $\Gamma$.
For $\eta\in C^1(\Gamma)$ we define the tangential gradient and the tangential derivatives of $\eta$ as
\begin{align*}
  \nabla_\Gamma\eta(y) := P(y)\nabla\tilde{\eta}(y), \quad \underline{D}_i\eta(y) := \sum_{j=1}^3P_{ij}(y)\partial_j\tilde{\eta}(y), \quad y\in\Gamma,\,i=1,2,3
\end{align*}
so that $\nabla_\Gamma\eta=(\underline{D}_1\eta,\underline{D}_2\eta,\underline{D}_3\eta)$.
Here $\tilde{\eta}$ is a $C^1$-extension of $\eta$ to $N$ with $\tilde{\eta}|_\Gamma=\eta$.
Since $P^2=P$ and $n\cdot Pa=0$ on $\Gamma$ for $a\in\mathbb{R}^3$ we have
\begin{align} \label{E:P_TGr}
  P\nabla_\Gamma\eta = \nabla_\Gamma\eta, \quad n\cdot\nabla_\Gamma\eta = 0 \quad\text{on}\quad \Gamma.
\end{align}
Note that the values of $\nabla_\Gamma\eta$ and $\underline{D}_i\eta$ are independent of the choice of an extension $\tilde{\eta}$ (see e.g.~\cite[Lemma~2.4]{DzEl13}).
In particular, the constant extension $\bar{\eta}:=\eta\circ\pi$ of $\eta$ in the normal direction of $\Gamma$ satisfies
\begin{align} \label{E:ConDer_Surf}
  \nabla\bar{\eta}(y) = \nabla_\Gamma\eta(y), \quad \partial_i\bar{\eta}(y) = \underline{D}_i\eta(y), \quad y\in\Gamma,\,i=1,2,3
\end{align}
since $\nabla\pi(y)=P(y)$ for $y\in\Gamma$ by \eqref{E:Nor_Coord} and $d(y)=0$.
In what follows, a function $\bar{\eta}$ with an overline always stands for the constant extension of a function $\eta$ on $\Gamma$ in the normal direction of $\Gamma$.
The tangential Hessian matrix of $\eta\in C^2(\Gamma)$ and the Laplace--Beltrami operator are given by
\begin{align*}
  \nabla_\Gamma^2\eta := (\underline{D}_i\underline{D}_j\eta)_{i,j}, \quad \Delta_\Gamma\eta := \mathrm{tr}[\nabla_\Gamma^2\eta] = \sum_{i=1}^3\underline{D}_i^2\eta \quad\text{on}\quad \Gamma.
\end{align*}
For a (not necessarily tangential) vector field $v=(v_1,v_2,v_3)\in C^1(\Gamma)^3$ we define the tangential gradient matrix and the surface divergence of $v$ by
\begin{align*}
  \nabla_\Gamma v :=
  \begin{pmatrix}
    \underline{D}_1v_1 & \underline{D}_1v_2 & \underline{D}_1v_3 \\
    \underline{D}_2v_1 & \underline{D}_2v_2 & \underline{D}_2v_3 \\
    \underline{D}_3v_1 & \underline{D}_3v_2 & \underline{D}_3v_3
  \end{pmatrix}, \quad
  \mathrm{div}_\Gamma v := \mathrm{tr}[\nabla_\Gamma v] = \sum_{i=1}^3\underline{D}_iv_i \quad\text{on}\quad \Gamma.
\end{align*}
Also, for $v\in C^1(\Gamma)^3$ and $\eta\in C(\Gamma)^3$ we set
\begin{align*}
  (\eta\cdot\nabla_\Gamma)v := (\eta\cdot\nabla_\Gamma v_1,\eta\cdot\nabla_\Gamma v_2,\eta\cdot\nabla_\Gamma v_3) = (\nabla_\Gamma v)^T\eta \quad\text{on}\quad \Gamma.
\end{align*}
Note that $\nabla_\Gamma v=P\nabla\tilde{v}$ and $(\eta\cdot\nabla_\Gamma)v=[(P\eta)\cdot\nabla]\tilde{v}$ on $\Gamma$ for any $C^1$-extension $\tilde{v}$ of $v$ to $N$ with $\tilde{v}|_\Gamma=v$.
When $v\in C^2(\Gamma)^3$ we write
\begin{align*}
  |\nabla_\Gamma^2v|^2 := \sum_{i,j,k=1}^3|\underline{D}_i\underline{D}_jv_k|^2, \quad \Delta_\Gamma v := (\Delta_\Gamma v_1,\Delta_\Gamma v_2,\Delta_\Gamma v_3) \quad\text{on}\quad \Gamma.
\end{align*}
For a matrix-valued function $A\in C^1(\Gamma)^{3\times 3}$ of the form
\begin{align*}
  A = (A_{ij})_{i,j} =
  \begin{pmatrix}
    A_{11} & A_{12} & A_{13} \\
    A_{21} & A_{22} & A_{23} \\
    A_{31} & A_{32} & A_{33}
  \end{pmatrix}
\end{align*}
we define the surface divergence of $A$ as a vector field on $\Gamma$ with $j$-th component
\begin{align*}
  [\mathrm{div}_\Gamma A]_j := \sum_{i=1}^3 \underline{D}_i A_{ij} \quad\text{on}\quad \Gamma,\, j=1,2,3.
\end{align*}
Next we give surface quantities on $\Gamma$.
We define the Weingarten map $W$, (twice) the mean curvature $H$, and the Gaussian curvature $K$ of $\Gamma$ by
\begin{align} \label{E:Def_WHK}
  W := -\nabla_\Gamma n, \quad H := \mathrm{tr}[W] = -\mathrm{div}_\Gamma n, \quad K := \kappa_1\kappa_2 \quad\text{on}\quad \Gamma.
\end{align}
Note that $W$, $H$, and $K$ are of class $C^{\ell-2}$ and thus bounded on $\Gamma$.

\begin{lemma} \label{L:Form_W}
  The Weingarten map $W$ is symmetric and
  \begin{align} \label{E:Form_W}
    Wn = 0, \quad PW = WP = W, \quad \mathrm{div}_\Gamma P = Hn \quad\text{on}\quad \Gamma.
  \end{align}
  Also, if $v\in C^1(\Gamma)^3$ is tangential, i.e. $v\cdot n=0$ on $\Gamma$, then
  \begin{align} \label{E:Grad_W}
    (\nabla_\Gamma v)n = Wv, \quad \nabla_\Gamma v = P(\nabla_\Gamma v)P+(Wv)\otimes n \quad\text{on}\quad \Gamma.
  \end{align}
\end{lemma}

\begin{proof}
  By \eqref{E:Nor_Coord}, \eqref{E:ConDer_Surf}, and $|n|^2=1$ on $\Gamma$ we have
  \begin{align*}
    W = -\nabla\bar{n} = -\nabla^2d, \quad Wn = -(\nabla_\Gamma n)n = -\frac{1}{2}\nabla_\Gamma(|n|^2) = 0 \quad\text{on}\quad \Gamma.
  \end{align*}
  Hence $W$ is symmetric and the first equality of \eqref{E:Form_W} holds.
  Also, the second equality follows from $W^Tn=Wn=0$ on $\Gamma$.
  We observe by \eqref{E:Def_WHK} and $W^Tn=0$ on $\Gamma$ that
  \begin{align*}
    [\mathrm{div}_\Gamma P]_j = \sum_{i=1}^3\underline{D}_i(\delta_{ij}-n_in_j) = \sum_{i=1}^3(W_{ii}n_j+W_{ij}n_i) = Hn_j \quad\text{on}\quad \Gamma,\,j=1,2,3,
  \end{align*}
  where $\delta_{ij}$ is the Kronecker delta.
  This shows the third equality of \eqref{E:Form_W}.

  Let $v\in C^1(\Gamma)^3$ satisfy $v\cdot n=0$ on $\Gamma$.
  Then
  \begin{align*}
    (\nabla_\Gamma v)n = \nabla_\Gamma(v\cdot n)-(\nabla_\Gamma n)v = Wv \quad\text{on}\quad \Gamma.
  \end{align*}
  Thus the first equality of \eqref{E:Grad_W} holds.
  Also, by $I_3=P+Q$ on $\Gamma$ and \eqref{E:P_TGr} we have
  \begin{align*}
    \nabla_\Gamma v = (\nabla_\Gamma v)P+(\nabla_\Gamma v)Q = P(\nabla_\Gamma v)P+\{(\nabla_\Gamma v)n\}\otimes n \quad\text{on}\quad \Gamma.
  \end{align*}
  Hence the second equality of \eqref{E:Grad_W} follows from the first one.
\end{proof}

By \eqref{E:Form_W}, $W$ has the eigenvalue zero associated with the eigenvector $n$.
The other eigenvalues of $W$ are the principal curvatures $\kappa_1$ and $\kappa_2$ (see e.g.~\cite[Section~14.6]{GiTr01} and~\cite[Section~VII.5]{KoNo96}) and thus $H=\kappa_1+\kappa_2$ on $\Gamma$.

When we calculate the derivatives of the constant extension of a function on $\Gamma$, the inverse matrix of $I_3-rW(y)$ for $y\in\Gamma$ and $r\in(-\delta,\delta)$ appears.

\begin{lemma} \label{L:Wein}
  The matrix $I_3-rW(y)$ is invertible and
  \begin{align} \label{E:WReso_P}
    \{I_3-rW(y)\}^{-1}P(y) = P(y)\{I_3-rW(y)\}^{-1}
  \end{align}
  for all $y\in\Gamma$ and $r\in(-\delta,\delta)$.
  Moreover, there exists a constant $c>0$ such that
  \begin{gather}
    c^{-1}|a| \leq \bigl|\{I_3-rW(y)\}^ka\bigr| \leq c|a|, \quad k=\pm1, \label{E:Wein_Bound} \\
    \bigl|I_3-\{I_3-rW(y)\}^{-1}\bigr| \leq c|r| \label{E:Wein_Diff}
  \end{gather}
  for all $y\in\Gamma$, $r\in(-\delta,\delta)$, and $a\in\mathbb{R}^3$.
\end{lemma}

\begin{lemma} \label{L:Pi_Der}
  For all $x\in N$ we have
  \begin{align} \label{E:Pi_Der}
    \nabla\pi(x) &= \left\{I_3-d(x)\overline{W}(x)\right\}^{-1}\overline{P}(x).
  \end{align}
  Therefore, the constant extension $\bar{\eta}=\eta\circ\pi$ of $\eta\in C^1(\Gamma)$ satisfies
  \begin{align} \label{E:ConDer_Dom}
    \nabla\bar{\eta}(x) = \left\{I_3-d(x)\overline{W}(x)\right\}^{-1}\overline{\nabla_\Gamma\eta}(x),\quad x\in N
  \end{align}
  and there exists a constant $c>0$ independent of $\eta$ such that
  \begin{gather}
    c^{-1}\left|\overline{\nabla_\Gamma\eta}(x)\right| \leq |\nabla\bar{\eta}(x)| \leq c\left|\overline{\nabla_\Gamma\eta}(x)\right|, \label{E:ConDer_Bound} \\
    \left|\nabla\bar{\eta}(x)-\overline{\nabla_\Gamma\eta}(x)\right| \leq c\left|d(x)\overline{\nabla_\Gamma\eta}(x)\right| \label{E:ConDer_Diff}
  \end{gather}
  for all $x\in N$.
  If $\Gamma$ is of class $C^3$ and $\eta\in C^2(\Gamma)$, then we have
  \begin{align} \label{E:Con_Hess}
    |\nabla^2\eta(x)| \leq c\left(\left|\overline{\nabla_\Gamma\eta}(x)\right|+\left|\overline{\nabla_\Gamma^2\eta}(x)\right|\right), \quad x\in N.
  \end{align}
  Moreover, $\Delta\bar{\eta}=\Delta_\Gamma\eta$ on $\Gamma$.
\end{lemma}

We give the proofs of Lemmas~\ref{L:Wein} and~\ref{L:Pi_Der} in Appendix~\ref{S:Ap_DG}.

Since $\bar{n}=\nabla d$ in $N$ by \eqref{E:Nor_Coord}, we use \eqref{E:ConDer_Dom} and $W=-\nabla_\Gamma n$ on $\Gamma$ to obtain
\begin{align} \label{E:Nor_Grad}
  \nabla\bar{n}(x) = \nabla^2d(x) = -\left\{I_3-d(x)\overline{W}(x)\right\}^{-1}\overline{W}(x), \quad x\in N.
\end{align}
If $\Gamma$ is of class $C^3$, then $n\in C^2(\Gamma)^3$ and thus
\begin{align} \label{E:NorG_Bound}
  |\nabla\bar{n}(x)| \leq c, \quad |\nabla^2\bar{n}(x)| \leq c, \quad x\in N
\end{align}
with some constant $c>0$ by \eqref{E:ConDer_Dom} and \eqref{E:Con_Hess}.

Let us define the weak tangential derivatives of a function on $\Gamma$ and the Sobolev spaces on $\Gamma$.
For a vector field $v\in C^1(\Gamma)^3$ we have
\begin{align*}
  \int_\Gamma\mathrm{div}_\Gamma v\,d\mathcal{H}^2 = \int_\Gamma\mathrm{div}_\Gamma(Pv)\,d\mathcal{H}^2+\int_\Gamma\mathrm{div}_\Gamma[(v\cdot n)n]\,d\mathcal{H}^2
\end{align*}
by the decomposition $v=Pv+(v\cdot n)n$, where $\mathcal{H}^2$ is the two-dimensional Hausdorff measure.
The first integral on the right-hand side vanishes by the Stokes theorem since $Pv$ is tangential on the closed surface $\Gamma$.
Moreover, to the second integral we apply $\mathrm{div}_\Gamma(\xi n)=\nabla_\Gamma\xi\cdot n+\xi\,\mathrm{div}_\Gamma n=-\xi H$ for $\xi\in C^1(\Gamma)$.
Then we get
\begin{align*}
  \int_\Gamma\mathrm{div}_\Gamma v\,d\mathcal{H}^2 = -\int_\Gamma(v\cdot n)H\,d\mathcal{H}^2
\end{align*}
for all $v\in C^1(\Gamma)^3$.
In particular, for $\eta,\xi\in C^1(\Gamma)$ we set $v=\eta\xi e_i$ in the above formula, where $\{e_1,e_2,e_3\}$ is the standard basis of $\mathbb{R}^3$, to obtain
\begin{align} \label{E:IbP_TD}
  \int_\Gamma(\eta\underline{D}_i\xi+\xi\underline{D}_i\eta)\,d\mathcal{H}^2 = -\int_\Gamma \eta\xi Hn_i\,d\mathcal{H}^2, \quad i=1,2,3.
\end{align}
Based on this identity, for $p\in[1,\infty]$ and $i=1,2,3$ we say that $\eta\in L^p(\Gamma)$ has the $i$-th weak tangential derivative if there exists $\eta_i\in L^p(\Gamma)$ such that
\begin{align} \label{E:Def_WTD}
  \int_\Gamma \eta_i\xi\,d\mathcal{H}^2 = -\int_\Gamma \eta(\underline{D}_i\xi+\xi Hn_i)\,d\mathcal{H}^2
\end{align}
for all $\xi\in C^1(\Gamma)$.
In this case we write $\underline{D}_i\eta=\eta_i$ and define the Sobolev space
\begin{align*}
  W^{1,p}(\Gamma) &:= \{\eta \in L^p(\Gamma) \mid \text{$\underline{D}_i\eta\in L^p(\Gamma)$ for all $i=1,2,3$}\}, \\
  \|\eta\|_{W^{1,p}(\Gamma)} &:=
  \begin{cases}
    \left(\|\eta\|_{L^p(\Gamma)}^p+\|\nabla_\Gamma\eta\|_{L^p(\Gamma)}^p\right)^{1/p} &\text{if}\quad p\in[1,\infty), \\
    \|\eta\|_{L^\infty(\Gamma)}+\|\nabla_\Gamma\eta\|_{L^\infty(\Gamma)} &\text{if}\quad p=\infty.
  \end{cases}
\end{align*}
Note that $W^{1,p}(\Gamma)$ is a Banach space.
In particular, $H^1(\Gamma)=W^{1,2}(\Gamma)$ is a Hilbert space equipped with inner product $(\eta,\xi)_{H^1(\Gamma)}:=(\eta,\xi)_{L^2(\Gamma)}+(\nabla_\Gamma\eta,\nabla_\Gamma\xi)_{L^2(\Gamma)}$.
We also define the second order Sobolev space
\begin{align*}
  W^{2,p}(\Gamma) &:= \{\eta \in W^{1,p}(\Gamma) \mid \text{$\underline{D}_i\underline{D}_j\eta\in L^p(\Gamma)$ for all $i,j=1,2,3$}\}, \\
  \|\eta\|_{W^{2,p}(\Gamma)} &:=
  \begin{cases}
    \left(\|\eta\|_{W^{1,p}(\Gamma)}^p+\|\nabla_\Gamma^2\eta\|_{L^p(\Gamma)}^p\right)^{1/p} &\text{if}\quad p\in[1,\infty), \\
    \|\eta\|_{W^{1,\infty}(\Gamma)}+\|\nabla_\Gamma^2\eta\|_{L^\infty(\Gamma)} &\text{if}\quad p=\infty.
  \end{cases}
\end{align*}
Then $W^{2,p}(\Gamma)$ is again a Banach space and $H^2(\Gamma)=W^{2,2}(\Gamma)$ is a Hilbert space.
In what follows, we write $W^{0,p}(\Gamma)=L^p(\Gamma)$ for $p\in[1,\infty]$.

\begin{lemma} \label{L:Wmp_Appr}
  For $p\in[1,\infty)$ and $m=0,1,\dots,\ell$, $C^\ell(\Gamma)$ is dense in $W^{m,p}(\Gamma)$.
\end{lemma}

We prove Lemma~\ref{L:Wmp_Appr} in Appendix~\ref{S:Ap_DG} by standard localization and mollification arguments.
As in the case of a flat domain, Poincar\'{e}'s inequality holds on $\Gamma$.

\begin{lemma} \label{L:Poin_Surf_Lp}
  Let $p\in[1,\infty)$.
  There exists a constant $c>0$ such that
  \begin{align} \label{E:Poin_Surf_Lp}
    \|\eta\|_{L^p(\Gamma)} \leq c\|\nabla_\Gamma\eta\|_{L^p(\Gamma)}
  \end{align}
  for all $\eta\in W^{1,p}(\Gamma)$ satisfying $\int_\Gamma\eta\,d\mathcal{H}^2=0$.
\end{lemma}

We refer to~\cite[Theorem~2.12]{DzEl13} for the proof of Lemma~\ref{L:Poin_Surf_Lp}.
Note that the proof given there applies to a closed, connected, and oriented hypersurface of class $C^2$.

Let $\mathcal{X}(\Gamma)$ be a function space on $\Gamma$ like $C^m(\Gamma)$, $L^p(\Gamma)$, and $W^{m,p}(\Gamma)$.
We define the space of all tangential vector fields on $\Gamma$ whose components in $\mathcal{X}(\Gamma)$ by
\begin{align*}
  \mathcal{X}(\Gamma,T\Gamma) := \{v\in\mathcal{X}(\Gamma)^3 \mid \text{$v\cdot n=0$ on $\Gamma$}\}.
\end{align*}
Note that $W^{m,p}(\Gamma,T\Gamma)$ is a closed subspace of $W^{m,p}(\Gamma)^3$, and thus a Banach space.
Moreover, an element of $W^{m,p}(\Gamma,T\Gamma)$ with $p\neq\infty$ can be approximated by smooth tangential vector fields on $\Gamma$.

\begin{lemma} \label{L:Wmp_Tan_Appr}
  Let $p\in[1,\infty)$ and $m=0,1,\dots,\ell-1$.
  Then $C^{\ell-1}(\Gamma,T\Gamma)$ is dense in $W^{m,p}(\Gamma,T\Gamma)$ with respect to the norm $\|\cdot\|_{W^{m,p}(\Gamma)}$.
\end{lemma}

\begin{proof}
  Let $v\in W^{m,p}(\Gamma,T\Gamma)\subset W^{m,p}(\Gamma)^3$.
  By Lemma~\ref{L:Wmp_Appr} we can take a sequence $\{\tilde{v}_k\}_{k=1}^\infty$ in $C^\ell(\Gamma)^3$ that converges to $v$ strongly in $W^{m,p}(\Gamma)^3$.
  For each $k\in\mathbb{N}$ let $v_k:=P\tilde{v}_k$ on $\Gamma$.
  Then $v_k\in C^{\ell-1}(\Gamma,T\Gamma)$ since $P$ is of class $C^{\ell-1}$ on $\Gamma$.
  Moreover, since $v$ is tangential on $\Gamma$, we have $v-v_k=P(v-\tilde{v}_k)$ on $\Gamma$ and thus
  \begin{align*}
    \|v-v_k\|_{W^{m,p}(\Gamma)} \leq c\|v-\tilde{v}_k\|_{W^{m,p}(\Gamma)} \to 0 \quad\text{as}\quad k\to\infty
  \end{align*}
  by the $C^{\ell-1}$-regularity of $P$ on $\Gamma$ and the strong convergence of $\{\tilde{v}_k\}_{k=1}^\infty$ to $v$ in $W^{m,p}(\Gamma)^3$.
  Hence the claim is valid.
\end{proof}

Let $H^{-1}(\Gamma)$ be the dual of $H^1(\Gamma)$ and $\langle\cdot,\cdot\rangle_\Gamma$ the duality product between $H^{-1}(\Gamma)$ and $H^1(\Gamma)$.
We consider $\eta\in L^2(\Gamma)$ as an element of $H^{-1}(\Gamma)$ by setting
\begin{align} \label{E:L2_Hin}
  \langle \eta,\xi\rangle_\Gamma:=(\eta,\xi)_{L^2(\Gamma)}, \quad \xi\in H^1(\Gamma).
\end{align}
Then the compact embeddings $H^1(\Gamma) \hookrightarrow L^2(\Gamma) \hookrightarrow H^{-1}(\Gamma)$ hold as in the case of a flat bounded domain (note that $\Gamma$ has no boundary).
Since
\begin{align*}
  |\langle \xi,\eta\varphi\rangle_\Gamma| &\leq \|\xi\|_{H^{-1}(\Gamma)}\|\eta\varphi\|_{H^1(\Gamma)} \leq c\|\eta\|_{W^{1,\infty}(\Gamma)}\|\xi\|_{H^{-1}(\Gamma)}\|\varphi\|_{H^1(\Gamma)}
\end{align*}
for $\eta\in W^{1,\infty}(\Gamma)$, $\xi\in H^{-1}(\Gamma)$, and $\varphi\in H^1(\Gamma)$, we can define $\eta\xi\in H^{-1}(\Gamma)$ by
\begin{align} \label{E:Def_Mul_Hin}
  \langle\eta \xi, \varphi\rangle_\Gamma := \langle \xi,\eta\varphi\rangle_\Gamma, \quad \varphi\in H^1(\Gamma).
\end{align}
Let $\eta\in L^2(\Gamma)$.
Based on \eqref{E:Def_WTD} we define $\underline{D}_i\eta\in H^{-1}(\Gamma)$, $i=1,2,3$ by
\begin{align} \label{E:Def_TD_Hin}
  \langle\underline{D}_i\eta,\xi\rangle_\Gamma := -(\eta,\underline{D}_i\xi+\xi Hn_i)_{L^2(\Gamma)}, \quad \xi \in H^1(\Gamma).
\end{align}
Note that this definition makes sense since $n$ and $H$ are bounded on $\Gamma$.
We consider the weak tangential gradient $\nabla_\Gamma\eta$ as an element of $H^{-1}(\Gamma)^3$ satisfying
\begin{align} \label{E:TGr_Hin}
  \langle\nabla_\Gamma\eta,v\rangle_\Gamma = -(\eta,\mathrm{div}_\Gamma v+(v\cdot n)H)_{L^2(\Gamma)}, \quad v\in H^1(\Gamma)^3.
\end{align}
Also, the surface divergence of $v\in L^2(\Gamma)^3$ is given by
\begin{align} \label{E:Sdiv_Hin}
  \langle\mathrm{div}_\Gamma v,\eta\rangle_\Gamma = -(v,\nabla_\Gamma\eta+\eta Hn)_{L^2(\Gamma)}, \quad \eta\in H^1(\Gamma).
\end{align}
Let $H^{-1}(\Gamma,T\Gamma)$ be the dual of $H^1(\Gamma,T\Gamma)$ and $[\cdot,\cdot]_{T\Gamma}$ the duality product between $H^{-1}(\Gamma,T\Gamma)$ and $H^1(\Gamma,T\Gamma)$.
It is homeomorphic to a quotient space of $H^{-1}(\Gamma)^3$.

\begin{lemma} \label{L:HinT_Homeo}
  For $f\in H^{-1}(\Gamma)^3$ we define an equivalence class
  \begin{align*}
    [f] := \{\tilde{f} \in H^{-1}(\Gamma)^3 \mid \text{$Pf=P\tilde{f}$ in $H^{-1}(\Gamma)^3$}\}.
  \end{align*}
  Then the quotient space $\mathcal{Q}:=\{[f] \mid f\in H^{-1}(\Gamma)^3\}$ is homeomorphic to $H^{-1}(\Gamma,T\Gamma)$.
\end{lemma}

Note that $\mathcal{Q}$ is a Banach space equipped with norm $\|[f]\|_{\mathcal{Q}}:=\inf_{\tilde{f}\in[f]}\|\tilde{f}\|_{H^{-1}(\Gamma)}$ (see e.g.~\cite{Ru91} for details).

\begin{proof}
  Let $f_1,f_2\in H^{-1}(\Gamma)^3$.
  If $Pf_1=Pf_2$ in $H^{-1}(\Gamma)^3$, then $\langle f_1,v\rangle_\Gamma=\langle f_2,v\rangle_\Gamma$ for all $v\in H^1(\Gamma,T\Gamma)$ by \eqref{E:Def_Mul_Hin} and $Pv=v$ on $\Gamma$.
  Hence we can define a linear operator $L$ from $\mathcal{Q}$ to $H^{-1}(\Gamma,T\Gamma)$ by $[L[f],v]_{T\Gamma}:=\langle\tilde{f},v\rangle_\Gamma$ for $[f]\in\mathcal{Q}$ and $v\in H^1(\Gamma,T\Gamma)$, where $\tilde{f}$ is any element of $[f]$.
  By this definition we also see that
  \begin{align*}
    \|L[f]\|_{H^{-1}(\Gamma,T\Gamma)} \leq \inf_{\tilde{f}\in[f]}\|\tilde{f}\|_{H^{-1}(\Gamma)} = \|[f]\|_{\mathcal{Q}}.
  \end{align*}
  Hence $L$ is bounded.
  Moreover, if $L[f_1]=L[f_2]$ in $H^{-1}(\Gamma,T\Gamma)$, then
  \begin{align*}
    \langle Pf_1,v\rangle_\Gamma = \langle f_1,Pv\rangle_\Gamma = [L[f_1],Pv]_{T\Gamma} = [L[f_2],Pv]_{T\Gamma} = \langle f_2,Pv\rangle_\Gamma = \langle Pf_2,v\rangle_\Gamma
  \end{align*}
  for all $v\in H^1(\Gamma)^3$ and thus $Pf_1=Pf_2$ in $H^{-1}(\Gamma)^3$, which means that $[f_1]=[f_2]$ and $L$ is injective.
  To show its surjectivity, let $F\in H^{-1}(\Gamma,T\Gamma)$.
  Since $H^1(\Gamma,T\Gamma)$ is a Hilbert space equipped with inner product of $H^1(\Gamma)^3$, there exists $v_F\in H^1(\Gamma,T\Gamma)$ such that $[F,v]_{T\Gamma}=(v_F,v)_{H^1(\Gamma)}$ for all $v\in H^1(\Gamma,T\Gamma)$ by the Riesz representation theorem.
  Then setting $f_F:=v_F-\sum_{i=1}^3\underline{D}_i^2v_F\in H^{-1}(\Gamma)^3$ we have
  \begin{align*}
    [F,v]_{T\Gamma} &= \sum_{i,j=1}^3\{(v_F^j,v^j)_{L^2(\Gamma)}+(\underline{D}_iv_F^j,\underline{D}_iv^j)_{L^2(\Gamma)}\} \\
    &= \sum_{i,j=1}^3\langle v_F^j-\underline{D}_i^2v_F^j,v^j\rangle_\Gamma = \langle f_F,v\rangle_\Gamma = [L[f_F],v]_{T\Gamma}
  \end{align*}
  for all $v\in H^1(\Gamma,T\Gamma)$ by \eqref{E:L2_Hin}, \eqref{E:Def_TD_Hin}, and $\sum_{i=1}^3n_i\underline{D}_iv_F^j=n\cdot\nabla_\Gamma v_F^j=0$ on $\Gamma$ for $j=1,2,3$.
  Hence $F=L[f_F]$ in $H^{-1}(\Gamma,T\Gamma)$ and $L\colon\mathcal{Q}\to H^{-1}(\Gamma,T\Gamma)$ is a bounded, injective, and surjective linear operator.
  Since its inverse is also bounded by the open mapping theorem, $\mathcal{Q}$ is homeomorphic to $H^{-1}(\Gamma,T\Gamma)$.
\end{proof}

Hereafter we identify $L[f]\in H^{-1}(\Gamma,T\Gamma)$ in the proof of Lemma~\ref{L:HinT_Homeo} with the equivalence class $[f]$ for $f\in H^{-1}(\Gamma)^3$.
We further identify $[f]$ with its representative $Pf$ and write $[Pf,v]_{T\Gamma}=\langle f,v\rangle_\Gamma$ for $v\in H^1(\Gamma,T\Gamma)$.
When $Pf=f$ in $H^{-1}(\Gamma)^3$, we take $f$ as a representative of $[f]$ instead of $Pf$.
For example, if $\eta\in L^2(\Gamma)$, then
\begin{align*}
  \langle\nabla_\Gamma\eta,v\rangle_\Gamma = -(\eta,\mathrm{div}_\Gamma v+(v\cdot n)H)_{L^2(\Gamma)} = -\bigl(\eta,\mathrm{div}_\Gamma(Pv)\bigr)_{L^2(\Gamma)} = \langle P\nabla_\Gamma\eta,v\rangle_\Gamma
\end{align*}
for all $v\in H^1(\Gamma)^3$ and thus $P\nabla_\Gamma\eta=\nabla_\Gamma\eta$ in $H^{-1}(\Gamma)^3$.
In this case we have
\begin{align} \label{E:TGr_HinT}
  [\nabla_\Gamma\eta,v]_{T\Gamma} = -(\eta,\mathrm{div}_\Gamma v)_{L^2(\Gamma)}, \quad \eta\in L^2(\Gamma),\,v\in H^1(\Gamma,T\Gamma).
\end{align}
For $\eta\in W^{1,\infty}(\Gamma)$ and $f\in H^{-1}(\Gamma,T\Gamma)$ we can define $\eta f\in H^{-1}(\Gamma,T\Gamma)$ by
\begin{align} \label{E:Def_Mul_HinT}
  [\eta f,v]_{T\Gamma} := [f,\eta v]_{T\Gamma}, \quad v\in H^1(\Gamma,T\Gamma)
\end{align}
since $\eta v\in H^1(\Gamma,T\Gamma)$ and $\|\eta v\|_{H^1(\Gamma)}\leq c\|\eta\|_{W^{1,\infty}(\Gamma)}\|v\|_{H^1(\Gamma)}$.
In Section~\ref{S:WSol} we give the characterization of the annihilators in $H^{-1}(\Gamma)^3$ and $H^{-1}(\Gamma,T\Gamma)$ of solenoidal spaces on $\Gamma$.

Since $\Gamma$ is not of class $C^\infty$, the space $C^\infty(\Gamma)$ does not make sense and we cannot consider distributions on $\Gamma$.
To consider the time derivative of functions with values in function spaces on $\Gamma$ we introduce the notion of distributions with values in a Banach space (see~\cite{LiMa72,So01,Te79} for details).
For $T>0$ let $C_c^\infty(0,T)$ be the space of all smooth and compactly supported functions on $(0,T)$.
Also, let $\mathcal{X}$ be a Banach space.
We define $\mathcal{D}'(0,T;\mathcal{X})$ as the space of all continuous linear operators from $C_c^\infty(0,T)$ (equipped with locally convex topology given in~\cite[Definition~6.3, (c)]{Ru91}) into $\mathcal{X}$.
By identifying $f\in L^2(0,T;\mathcal{X})$ with a continuous linear operator
\begin{align*}
  \hat{f}(\varphi) := \int_0^T\varphi(t)f(t)\,dt \in \mathcal{X}, \quad \varphi\in C_c^\infty(0,T)
\end{align*}
we consider $L^2(0,T;\mathcal{X})$ as a subset of $\mathcal{D}'(0,T;\mathcal{X})$.
Let $f\in\mathcal{D}'(0,T;\mathcal{X})$.
We define the time derivative of $f$ by $\partial_tf(\varphi):=-f(\partial_t\varphi)$ for $\varphi\in C_c^\infty(0,T)$.
By definition, $\partial_tf$ belongs to $\mathcal{D}'(0,T;\mathcal{X})$.
Moreover, if $f\in L^2(0,T;\mathcal{X})$ then
\begin{align} \label{E:Dt_Dist_L2}
  \partial_tf(\varphi) = -f(\partial_t\varphi) = -\int_0^T\partial_t\varphi(t)f(t)\,dt \in \mathcal{X}, \quad \varphi\in C_c^\infty(0,T).
\end{align}
For $f\in L^2(0,T;\mathcal{X})$ if there exists $\xi\in L^2(0,T;\mathcal{X})$ such that
\begin{align*}
  \partial_tf(\varphi) = \xi(\varphi), \quad\text{i.e.}\quad -\int_0^T\partial_t\varphi(t)f(t)\,dt = \int_0^T\varphi(t)\xi(t)\,dt \quad(\text{in $\mathcal{X}$})
\end{align*}
for all $\varphi\in C_c^\infty(0,T)$, then we write $\partial_tf=\xi\in L^2(0,T;\mathcal{X})$ and define
\begin{align*}
  H^1(0,T;\mathcal{X}) := \{f\in L^2(0,T;\mathcal{X}) \mid \partial_tf \in L^2(0,T;\mathcal{X})\}.
\end{align*}
When $q\in L^2(0,T;L^2(\Gamma))$, we can consider the time derivative of $\nabla_\Gamma q$ as an element of $\mathcal{D}'(0,T;H^{-1}(\Gamma,T\Gamma))$.
Let us show that the time derivative commutes with the tangential gradient in an appropriate sense.

\begin{lemma} \label{L:Dt_TGr_Com}
  Let $q\in L^2(0,T;L^2(\Gamma))$.
  Then
  \begin{align*}
    \nabla_\Gamma[\partial_tq(\varphi)] = [\partial_t(\nabla_\Gamma q)](\varphi) \quad\text{in}\quad H^{-1}(\Gamma,T\Gamma)
  \end{align*}
  for all $\varphi\in C_c^\infty(0,T)$.
\end{lemma}

\begin{proof}
  For all $v\in H^1(\Gamma,T\Gamma)$ we have
  \begin{align*}
    [\nabla_\Gamma[\partial_tq(\varphi)],v]_{T\Gamma} &= (q(\partial_t\varphi),\mathrm{div}_\Gamma v)_{L^2(\Gamma)} = \int_0^T\partial_t\varphi(t)(q(t),\mathrm{div}_\Gamma v)_{L^2(\Gamma)}\,dt \\
    &= -\int_0^T\partial_t\varphi(t)[\nabla_\Gamma q(t),v]_{T\Gamma}\,dt = \Bigl[\,[\partial_t(\nabla_\Gamma q)](\varphi),v\Bigr]_{T\Gamma}
  \end{align*}
  by \eqref{E:TGr_Hin} and \eqref{E:Dt_Dist_L2}.
  Hence the claim is valid.
\end{proof}

Let $q\in L^2(0,T;L^2(\Gamma))$.
Based on Lemma~\ref{L:Dt_TGr_Com}, we consider the tangential gradient of $\partial_tq\in \mathcal{D}'(0,T;L^2(\Gamma))$ as an element of $\mathcal{D}'(0,T;H^{-1}(\Gamma,T\Gamma))$ given by
\begin{align} \label{E:Def_TGrDt_Hin}
  [\nabla_\Gamma(\partial_tq)](\varphi) := \nabla_\Gamma[(\partial_tq)(\varphi)] = [\partial_t(\nabla_\Gamma q)](\varphi)\in H^{-1}(\Gamma,T\Gamma)
\end{align}
for $\varphi\in C_c^\infty(0,T)$.
We use this relation in construction of an associated pressure in the limit equations (see Lemma~\ref{L:LW_Pres}).

\subsection{Curved thin domain} \label{SS:Pre_Dom}
From now on, we assume that the closed surface $\Gamma$ is of class $C^5$ (except for Section~\ref{S:WSol}).
Let $g_0,g_1\in C^4(\Gamma)$ such that
\begin{align} \label{E:Width_Bound}
  g(y) := g_1(y)-g_0(y) \geq c \quad\text{for all}\quad y\in\Gamma
\end{align}
with a constant $c>0$.
For $\varepsilon\in(0,1)$ we define a curved thin domain $\Omega_\varepsilon$ in $\mathbb{R}^3$ as
\begin{align} \label{E:Def_CTD}
  \Omega_\varepsilon := \{y+rn(y) \mid y\in\Gamma,\,\varepsilon g_0(y) < r < \varepsilon g_1(y)\}.
\end{align}
Let $N$ be the tubular neighborhood of $\Gamma$ of radius $\delta>0$ given in Section~\ref{SS:Pre_Surf}.
Since $g_0$ and $g_1$ are bounded on $\Gamma$, there exists $\tilde{\varepsilon}\in(0,1)$ such that $\varepsilon|g_i|<\delta$ on $\Gamma$, $i=0,1$, i.e. $\overline{\Omega}_\varepsilon\subset N$ for all $\varepsilon\in(0,\tilde{\varepsilon})$.
Replacing $g_i$, $i=0,1$ by $\tilde{\varepsilon} g_i$ we may assume $\tilde{\varepsilon}=1$.

Let $\Gamma_\varepsilon$ be the boundary of $\Omega_\varepsilon$.
It is the union of the inner and outer boundaries $\Gamma_\varepsilon^0$ and $\Gamma_\varepsilon^1$ given by $\Gamma_\varepsilon^i:=\{y+\varepsilon g_i(y)n(y) \mid y\in\Gamma\}$ for $i=0,1$.
Note that $\Gamma_\varepsilon$ is of class $C^4$ by the $C^5$-regularity of $\Gamma$ and $g_0,g_1\in C^4(\Gamma)$.
We use this fact in derivation of a uniform a priori estimate for the vector Laplacian (see Section~\ref{SS:St_Apr}).

Let us give surface quantities on $\Gamma_\varepsilon$.
We define vector fields $\tau_\varepsilon^i$ and $n_\varepsilon^i$ on $\Gamma$ as
\begin{align}
  \tau_\varepsilon^i(y) &:= \{I_3-\varepsilon g_i(y)W(y)\}^{-1}\nabla_\Gamma g_i(y), \label{E:Def_NB_Aux}\\
  n_\varepsilon^i(y) &:= (-1)^{i+1}\frac{n(y)-\varepsilon\tau_\varepsilon^i(y)}{\sqrt{1+\varepsilon^2|\tau_\varepsilon^i(y)|^2}} \label{E:Def_NB}
\end{align}
for $y\in\Gamma$ and $i=0,1$.
Note that $\tau_\varepsilon^i$ is tangential on $\Gamma$ by \eqref{E:P_TGr}, \eqref{E:WReso_P}, and $Pa\cdot n=0$ on $\Gamma$ for $a\in\mathbb{R}^3$.
Also, $\tau_\varepsilon^i$ and $n_\varepsilon^i$ are bounded on $\Gamma$ uniformly in $\varepsilon$ along with their first and second order tangential derivatives.

\begin{lemma} \label{L:NB_Aux}
  There exists a constant $c>0$ independent of $\varepsilon$ such that
  \begin{gather}
    |\tau_\varepsilon^i(y)| \leq c, \quad |\underline{D}_k\tau_\varepsilon^i(y)| \leq c, \quad |\underline{D}_l\underline{D}_k\tau_\varepsilon^i(y)| \leq c, \label{E:Tau_Bound} \\
    |\tau_\varepsilon^i(y)-\nabla_\Gamma g_i(y)| \leq c\varepsilon, \quad |\nabla_\Gamma\tau_\varepsilon^i(y)-\nabla_\Gamma^2g_i(y)| \leq c\varepsilon \label{E:Tau_Diff}
  \end{gather}
  for all $y\in\Gamma$, $i=0,1$, and $k,l=1,2,3$.
  We also have
  \begin{gather}
    |n_\varepsilon^i(y)| = 1, \quad |\underline{D}_kn_\varepsilon^i(y)| \leq c, \quad |\underline{D}_l\underline{D}_kn_\varepsilon^i(y)| \leq c, \label{E:N_Bound} \\
    |n_\varepsilon^0(y)+n_\varepsilon^1(y)| \leq c\varepsilon, \quad |\nabla_\Gamma n_\varepsilon^0(y)+\nabla_\Gamma n_\varepsilon^1(y)| \leq c\varepsilon \label{E:N_Diff}
  \end{gather}
  for all $y\in\Gamma$, $i=0,1$, and $k,l=1,2,3$.
\end{lemma}

Let $n_\varepsilon$ be the unit outward normal vector field of $\Gamma_\varepsilon$.
For $i=0,1$ the direction of $n_\varepsilon$ on $\Gamma_\varepsilon^i$ is the same as that of $(-1)^{i+1}\bar{n}$ since the signed distance function $d$ from $\Gamma$ increases in the direction of $n$.

\begin{lemma} \label{L:Nor_Bo}
  The unit outward normal vector field $n_\varepsilon$ of $\Gamma_\varepsilon$ is given by
  \begin{align} \label{E:Nor_Bo}
    n_\varepsilon(x) = \bar{n}_\varepsilon^i(x), \quad x\in\Gamma_\varepsilon^i,\,i=0,1.
  \end{align}
  Here $\bar{n}_\varepsilon^i=n_\varepsilon^i\circ\pi$ is the constant extension of the vector field $n_\varepsilon^i$ given by \eqref{E:Def_NB}.
\end{lemma}

The proofs of Lemmas~\ref{L:NB_Aux} and~\ref{L:Nor_Bo} are given in Appendix~\ref{S:Ap_DG}.

As in Section~\ref{SS:Pre_Surf} we set $P_\varepsilon:=I_3-n_\varepsilon\otimes n_\varepsilon$ and $Q_\varepsilon:=n_\varepsilon\otimes n_\varepsilon$ on $\Gamma_\varepsilon$ and define the tangential gradient and the tangential derivatives of $\varphi\in C^1(\Gamma_\varepsilon)$ by
\begin{align*}
  \nabla_{\Gamma_\varepsilon}\varphi := P_\varepsilon\nabla\tilde{\varphi}, \quad \underline{D}_i^\varepsilon\varphi := \sum_{j=1}^3[P_\varepsilon]_{ij}\partial_j\tilde{\varphi} \quad\text{on}\quad \Gamma_\varepsilon,\, i=1,2,3,
\end{align*}
where $\tilde{\varphi}$ is any $C^1$-extension of $\varphi$ to an open neighborhood of $\Gamma_\varepsilon$ with $\tilde{\varphi}|_{\Gamma_\varepsilon}=\varphi$.
For $u\in C^1(\Gamma_\varepsilon)^3$ we define its tangential gradient matrix and surface divergence as
\begin{align*}
  \nabla_{\Gamma_\varepsilon}u :=
  \begin{pmatrix}
    \underline{D}_1^\varepsilon u_1 & \underline{D}_1^\varepsilon u_2 & \underline{D}_1^\varepsilon u_3 \\
    \underline{D}_2^\varepsilon u_1 & \underline{D}_2^\varepsilon u_2 & \underline{D}_2^\varepsilon u_3 \\
    \underline{D}_3^\varepsilon u_1 & \underline{D}_3^\varepsilon u_2 & \underline{D}_3^\varepsilon u_3
  \end{pmatrix}, \quad
  \mathrm{div}_{\Gamma_\varepsilon}u := \mathrm{tr}[\nabla_{\Gamma_\varepsilon}u] = \sum_{i=1}^3\underline{D}_i^\varepsilon u_i \quad\text{on}\quad \Gamma_\varepsilon.
\end{align*}
Also, for $u\in C^1(\Gamma_\varepsilon)^3$ and $\varphi\in C(\Gamma_\varepsilon)^3$ we write
\begin{align*}
  (\varphi\cdot\nabla_{\Gamma_\varepsilon})u := (\varphi\cdot\nabla_{\Gamma_\varepsilon}u_1,\varphi\cdot\nabla_{\Gamma_\varepsilon}u_2,\varphi\cdot\nabla_{\Gamma_\varepsilon}u_3) = (\nabla_{\Gamma_\varepsilon}u)^T\varphi \quad\text{on}\quad \Gamma_\varepsilon.
\end{align*}
The Weingarten map $W_\varepsilon$ and (twice) the mean curvature $H_\varepsilon$ of $\Gamma_\varepsilon$ are given by
\begin{align*}
  W_\varepsilon := -\nabla_{\Gamma_\varepsilon}n_\varepsilon, \quad H_\varepsilon := \mathrm{tr}[W_\varepsilon] = -\mathrm{div}_{\Gamma_\varepsilon}n_\varepsilon \quad\text{on}\quad \Gamma_\varepsilon.
\end{align*}
Note that, as in the case of $\Gamma$, the matrices $P_\varepsilon$, $Q_\varepsilon$, and $W_\varepsilon$ are symmetric and
\begin{align} \label{E:PW_Bo}
  \nabla_{\Gamma_\varepsilon}u = P_\varepsilon\nabla\tilde{u}, \quad P_\varepsilon W_\varepsilon = W_\varepsilon P_\varepsilon = W_\varepsilon \quad\text{on}\quad \Gamma_\varepsilon
\end{align}
for $u\in C^1(\Gamma_\varepsilon)^3$, where $\tilde{u}$ is any $C^1$-extension of $u$ to an open neighborhood of $\Gamma_\varepsilon$ with $\tilde{u}|_{\Gamma_\varepsilon}=u$.
We also define the weak tangential derivatives of functions on $\Gamma_\varepsilon$ and the Sobolev spaces $W^{m,p}(\Gamma_\varepsilon)$ for $m=1,2$ and $p\in[1,\infty)$ as in Section~\ref{SS:Pre_Surf}.

By the expression \eqref{E:Def_NB} of the unit outward normal $n_\varepsilon$ to $\Gamma_\varepsilon$, we can compare the surface quantities on $\Gamma_\varepsilon$ with those on $\Gamma$.

\begin{lemma} \label{L:Comp_Nor}
  There exists a constant $c>0$ independent of $\varepsilon$ such that
  \begin{gather}
    \left|n_\varepsilon(x)-(-1)^{i+1}\left\{\bar{n}(x)-\varepsilon\overline{\nabla_\Gamma g_i}(x)\right\}\right| \leq c\varepsilon^2, \label{E:Comp_N} \\
    \left|P_\varepsilon(x)-\overline{P}(x)\right| \leq c\varepsilon, \quad \left|Q_\varepsilon(x)-\overline{Q}(x)\right| \leq c\varepsilon, \label{E:Comp_P} \\
    \left|W_\varepsilon(x)-(-1)^{i+1}\overline{W}(x)\right| \leq c\varepsilon, \quad \left|H_\varepsilon(x)-(-1)^{i+1}\overline{H}(x)\right| \leq c\varepsilon, \label{E:Comp_W} \\
    \left|\underline{D}_j^\varepsilon W_\varepsilon(x)-(-1)^{i+1}\overline{\underline{D}_jW}(x)\right| \leq c\varepsilon \label{E:Comp_DW}
  \end{gather}
  for all $x\in\Gamma_\varepsilon^i$, $i=0,1$, and $j=1,2,3$.
\end{lemma}

From Lemma~\ref{L:Comp_Nor} it immediately follows that $W_\varepsilon$, $H_\varepsilon$, and $\underline{D}_j^\varepsilon W_\varepsilon$, $j=1,2,3$ are uniformly bounded in $\varepsilon$ on $\Gamma_\varepsilon$ (note that $|P_\varepsilon|=2$ and $|Q_\varepsilon|=1$ on $\Gamma_\varepsilon$).
Moreover, we can compare the surface quantities on the inner and outer boundaries.

\begin{lemma} \label{L:Diff_SQ_IO}
  There exists a constant $c>0$ independent of $\varepsilon$ such that
  \begin{align}
    |F_\varepsilon(y+\varepsilon g_1(y)n(y))-F_\varepsilon(y+\varepsilon g_0(y)n(y))| &\leq c\varepsilon, \label{E:Diff_PQ_IO} \\
    |G_\varepsilon(y+\varepsilon g_1(y)n(y))+G_\varepsilon(y+\varepsilon g_0(y)n(y))| &\leq c\varepsilon \label{E:Diff_WH_IO}
  \end{align}
  for all $y\in\Gamma$, where $F_\varepsilon=P_\varepsilon,Q_\varepsilon$ and $G_\varepsilon=W_\varepsilon,H_\varepsilon,\underline{D}_j^\varepsilon W_\varepsilon$ with $j=1,2,3$.
\end{lemma}

The proofs of Lemmas~\ref{L:Comp_Nor} and~\ref{L:Diff_SQ_IO} are given in Appendix~\ref{S:Ap_DG}.

Next we give transformation formulas of integrals over $\Omega_\varepsilon$ and $\Gamma_\varepsilon$.
For functions $\varphi$ on $\Omega_\varepsilon$ and $\eta$ on $\Gamma_\varepsilon^i$, $i=0,1$ we use the notations
\begin{alignat}{2}
  \varphi^\sharp(y,r) &:= \varphi(y+rn(y)), &\quad &y\in\Gamma,\,r\in(\varepsilon g_0(y),\varepsilon g_1(y)), \label{E:Pull_Dom} \\
  \eta_i^\sharp(y) &:= \eta(y+\varepsilon g_i(y)n(y)), &\quad &y\in\Gamma. \label{E:Pull_Bo}
\end{alignat}
Let $J=J(y,r)$ be a function given by
\begin{align} \label{E:Def_Jac}
  J(y,r) := \mathrm{det}[I_3-rW(y)] = \{1-r\kappa_1(y)\}\{1-r\kappa_2(y)\}
\end{align}
for $y\in\Gamma$ and $r\in(-\delta,\delta)$.
By \eqref{E:Curv_Bound} and $\kappa_1,\kappa_2\in C^3(\Gamma)$ we have
\begin{align} \label{E:Jac_Bound}
  c^{-1} \leq J(y,r) \leq c, \quad |\nabla_\Gamma J(y,r)| \leq c, \quad \left|\frac{\partial J}{\partial r}(y,r)\right| \leq c
\end{align}
for all $y\in\Gamma$ and $r\in(-\delta,\delta)$ (here $\nabla_\Gamma J$ stands for the tangential gradient of $J$ with respect to $y\in\Gamma$).
Also, we easily observe that
\begin{align} \label{E:Jac_Diff}
  |J(y,r)-1| \leq c\varepsilon \quad\text{for all}\quad y\in\Gamma,\,r\in[\varepsilon g_0(y),\varepsilon g_1(y)].
\end{align}
The function $J$ is the Jacobian appearing in the change of variables formula
\begin{align} \label{E:CoV_Dom}
  \int_{\Omega_\varepsilon}\varphi(x)\,dx = \int_\Gamma\int_{\varepsilon g_0(y)}^{\varepsilon g_1(y)}\varphi(y+rn(y))J(y,r)\,dr\,d\mathcal{H}^2(y)
\end{align}
for a function $\varphi$ on $\Omega_\varepsilon$ (see e.g.~\cite[Section~14.6]{GiTr01}).
The formula \eqref{E:CoV_Dom} can be seen as a co-area formula.
From \eqref{E:Jac_Bound} and \eqref{E:CoV_Dom} it immediately follows that
\begin{align} \label{E:CoV_Equiv}
  c^{-1}\|\varphi\|_{L^p(\Omega_\varepsilon)}^p \leq \int_\Gamma\int_{\varepsilon g_0(y)}^{\varepsilon g_1(y)}|\varphi^\sharp(y,r)|^p\,dr\,d\mathcal{H}^2(y) \leq c\|\varphi\|_{L^p(\Omega_\varepsilon)}^p
\end{align}
for all $p\in[1,\infty)$ and $\varphi\in L^p(\Omega_\varepsilon)$, where we used the notation \eqref{E:Pull_Dom}.
We frequently use this inequality in the sequel.

\begin{lemma} \label{L:Con_Lp_W1p}
  Let $\eta$ be a function on $\Gamma$ and $\bar{\eta}=\eta\circ\pi$ its constant extension in the normal direction of $\Gamma$.
  Then $\eta\in L^p(\Gamma)$, $p\in[1,\infty)$ if and only if $\bar{\eta}\in L^p(\Omega_\varepsilon)$.
  Moreover, there exists a constant $c>0$ independent of $\varepsilon$ and $\eta$ such that
  \begin{align} \label{E:Con_Lp}
    c^{-1}\varepsilon^{1/p}\|\eta\|_{L^p(\Gamma)} \leq \|\bar{\eta}\|_{L^p(\Omega_\varepsilon)} \leq c\varepsilon^{1/p}\|\eta\|_{L^p(\Gamma)}.
  \end{align}
  Also, $\eta\in W^{1,p}(\Gamma)$ if and only if $\bar{\eta}\in W^{1,p}(\Omega_\varepsilon)$ and we have
  \begin{align} \label{E:Con_Lp_Grad}
    c^{-1}\varepsilon^{1/p}\|\nabla_\Gamma\eta\|_{L^p(\Gamma)} \leq \|\nabla\bar{\eta}\|_{L^p(\Omega_\varepsilon)} \leq c\varepsilon^{1/p}\|\nabla_\Gamma\eta\|_{L^p(\Gamma)}
  \end{align}
  and therefore
  \begin{align} \label{E:Con_W1p}
    c^{-1}\varepsilon^{1/p}\|\eta\|_{W^{1,p}(\Gamma)} \leq \|\bar{\eta}\|_{W^{1,p}(\Omega_\varepsilon)} \leq c\varepsilon^{1/p}\|\eta\|_{W^{1,p}(\Gamma)}.
  \end{align}
\end{lemma}

\begin{proof}
  The change of variables formula \eqref{E:CoV_Dom} implies that
  \begin{align*}
    \|\bar{\eta}\|_{L^p(\Omega_\varepsilon)}^p = \int_\Gamma|\eta(y)|^p\left(\int_{\varepsilon g_0(y)}^{\varepsilon g_1(y)}J(y,r)\,dr\right)d\mathcal{H}^2(y).
  \end{align*}
  Hence the inequality \eqref{E:Con_Lp} follows from \eqref{E:Width_Bound} and \eqref{E:Jac_Bound}.
  Similarly, we get \eqref{E:Con_Lp_Grad} by \eqref{E:ConDer_Bound}, \eqref{E:Width_Bound}, \eqref{E:Jac_Bound}, and \eqref{E:CoV_Dom}.
\end{proof}

\begin{lemma} \label{L:Con_W2p}
  For $p\in[1,\infty)$ let $\eta\in W^{2,p}(\Gamma)$.
  Then $\bar{\eta}=\eta\circ\pi\in W^{2,p}(\Omega_\varepsilon)$ and
  \begin{align} \label{E:Con_W2p}
    \|\bar{\eta}\|_{W^{2,p}(\Omega_\varepsilon)} \leq c\varepsilon^{1/p}\|\eta\|_{W^{2,p}(\Gamma)}
  \end{align}
  with a constant $c>0$ independent of $\varepsilon$ and $\eta$.
\end{lemma}

\begin{proof}
  From \eqref{E:Con_Hess} and \eqref{E:Con_Lp} it follows that
  \begin{align*}
    \|\nabla^2\bar{\eta}\|_{L^p(\Omega_\varepsilon)} \leq c\left(\left\|\overline{\nabla_\Gamma\eta}\right\|_{L^p(\Omega_\varepsilon)}+\left\|\overline{\nabla_\Gamma^2\eta}\right\|_{L^p(\Omega_\varepsilon)}\right) \leq c\varepsilon^{1/p}\|\eta\|_{W^{2,p}(\Gamma)}.
  \end{align*}
  Combining this inequality with \eqref{E:Con_W1p} we obtain \eqref{E:Con_W2p}.
\end{proof}

We also give a change of variables formula for integrals over $\Gamma_\varepsilon$.

\begin{lemma} \label{L:CoV_Surf}
  For $\varphi\in L^1(\Gamma_\varepsilon^i)$, $i=0,1$ let $\varphi_i^\sharp$ be given by \eqref{E:Pull_Bo}.
  Then
  \begin{align} \label{E:CoV_Surf}
    \int_{\Gamma_\varepsilon^i}\varphi(x)\,d\mathcal{H}^2(x) = \int_\Gamma \varphi_i^\sharp(y)J(y,\varepsilon g_i(y))\sqrt{1+\varepsilon^2|\tau_\varepsilon^i(y)|^2}\,d\mathcal{H}^2(y),
  \end{align}
  where $\tau_\varepsilon^i$ is given by \eqref{E:Def_NB_Aux}.
  Moreover, if $\varphi\in L^p(\Gamma_\varepsilon^i)$, $p\in[1,\infty)$ then $\varphi_i^\sharp\in L^p(\Gamma)$ and there exists a constant $c>0$ independent of $\varepsilon$ such that
  \begin{align} \label{E:Lp_CoV_Surf}
    c^{-1}\|\varphi\|_{L^p(\Gamma_\varepsilon^i)} \leq \|\varphi_i^\sharp\|_{L^p(\Gamma)} \leq c\|\varphi\|_{L^p(\Gamma_\varepsilon^i)}.
  \end{align}
\end{lemma}

\begin{proof}
  In Lemma~\ref{L:CoV_Para} we show a change of variables formula
  \begin{align*}
    \int_{\Gamma_h}\varphi(x)\,d\mathcal{H}^2(x) = \int_\Gamma\varphi_h^\sharp(y)J(y,h(y))\sqrt{1+|\tau_h(y)|^2}\,d\mathcal{H}^2(y)
  \end{align*}
  for an integrable function $\varphi$ on a parametrized surface $\Gamma_h:=\{y+h(y)n(y)\mid y\in\Gamma\}$, where $h\in C^1(\Gamma)$ satisfies $|h|<\delta$ on $\Gamma$ and
  \begin{align*}
    \varphi_h^\sharp(y) := \varphi(y+h(y)n(y)), \quad \tau_h(y) := \{I_3-h(y)W(y)\}^{-1}\nabla_\Gamma h(y), \quad y\in\Gamma.
  \end{align*}
  Setting $h=\varepsilon g_i$, $i=0,1$ in the above formula we obtain \eqref{E:CoV_Surf}.
  Also, \eqref{E:Lp_CoV_Surf} follows from the formula \eqref{E:CoV_Surf} and the inequalities \eqref{E:Tau_Bound} and \eqref{E:Jac_Bound}.
\end{proof}

\section{Fundamental tools for analysis} \label{S:Tool}

\subsection{Sobolev inequalities} \label{SS:Tool_Sob}
Let us give Sobolev inequalities on $\Gamma$ and $\Omega_\varepsilon$.
We use the notations and inequalities given in Lemma~\ref{L:Metric}.
First we prove Ladyzhenskaya's inequality on $\Gamma$.

\begin{lemma} \label{L:La_Surf}
  There exists a constant $c>0$ such that
  \begin{align} \label{E:La_Surf}
    \|\eta\|_{L^4(\Gamma)} \leq c\|\eta\|_{L^2(\Gamma)}^{1/2}\|\nabla_\Gamma\eta\|_{L^2(\Gamma)}^{1/2}
  \end{align}
  for all $\eta\in H^1(\Gamma)$.
\end{lemma}

\begin{proof}
  Since $\Gamma$ is compact, by a standard localization argument with a partition of unity on $\Gamma$ we may assume that there exist an open set $U$ in $\mathbb{R}^2$, a local parametrization $\mu\colon U\to\Gamma$ of $\Gamma$, and a compact subset $\mathcal{K}$ of $U$ such that $\eta$ is supported in $\mu(\mathcal{K})$.
  Then $\tilde{\eta}:=\eta\circ\mu$ is supported in $\mathcal{K}$ and belongs to $H^1(U)$ by Lemma~\ref{L:Metric}.
  Hence
  \begin{align*}
    \|\tilde{\eta}\|_{L^4(U)} \leq \sqrt{2}\|\tilde{\eta}\|_{L^2(U)}^{1/2}\|\nabla_s\tilde{\eta}\|_{L^2(U)}^{1/2}
  \end{align*}
  by Ladyzhenskaya's inequality on $\mathbb{R}^2$ (see \cite[Chapter~1, Section~1.1, Lemma~1]{La69}), where $\nabla_s\tilde{\eta}$ is the gradient of $\tilde{\eta}$ in $s\in\mathbb{R}^2$.
  Applying \eqref{E:Lp_Loc} and \eqref{E:W1p_Loc} to this inequality we obtain \eqref{E:La_Surf}.
\end{proof}

Next we provide Poincar\'{e} type inequalities on $\Omega_\varepsilon$.
For $\varphi\in C^1(\Omega_\varepsilon)$ and $x\in\Omega_\varepsilon$ we define the derivative of $\varphi$ in the normal direction of $\Gamma$ by
\begin{align} \label{E:Def_NorDer}
  \partial_n\varphi(x) := (\bar{n}(x)\cdot\nabla)\varphi(x) = \frac{d}{dr}\bigl(\varphi(y+rn(y))\bigr)\Big|_{r=d(x)} \quad (y=\pi(x)\in\Gamma).
\end{align}
Note that for the constant extension $\bar{\eta}=\eta\circ\pi$ of $\eta\in C^1(\Gamma)$ we have
\begin{align} \label{E:NorDer_Con}
  \partial_n\bar{\eta}(x) = (\bar{n}(x)\cdot\nabla)\bar{\eta}(x) = 0, \quad x\in \Omega_\varepsilon.
\end{align}

\begin{lemma} \label{L:Poincare}
  There exists a constant $c>0$ independent of $\varepsilon$ such that
  \begin{align}
    \|\varphi\|_{L^p(\Omega_\varepsilon)} &\leq c\left(\varepsilon^{1/p}\|\varphi\|_{L^p(\Gamma_\varepsilon^i)}+\varepsilon\|\partial_n\varphi\|_{L^p(\Omega_\varepsilon)}\right), \quad i=0,1, \label{E:Poin_Dom} \\
    \|\varphi\|_{L^p(\Gamma_\varepsilon^i)} &\leq c\left(\varepsilon^{-1/p}\|\varphi\|_{L^p(\Omega_\varepsilon)}+\varepsilon^{1-1/p}\|\partial_n\varphi\|_{L^p(\Omega_\varepsilon)}\right), \quad i=0,1 \label{E:Poin_Bo}
  \end{align}
  for all $\varphi\in W^{1,p}(\Omega_\varepsilon)$ with $p\in[1,\infty)$.
\end{lemma}

\begin{proof}
  We show \eqref{E:Poin_Dom} and \eqref{E:Poin_Bo} for $i=0$.
  The proof for $i=1$ is the same.
  We use the notations \eqref{E:Pull_Dom} and \eqref{E:Pull_Bo}.
  For $y\in\Gamma$ and $r\in(\varepsilon g_0(y),\varepsilon g_1(y))$ we have
  \begin{align} \label{Pf_P:FTC}
    \varphi^\sharp(y,r) = \varphi_0^\sharp(y)+\int_{\varepsilon g_0(y)}^r(\partial_n\varphi)^\sharp(y,\tilde{r})\,d\tilde{r}
  \end{align}
  since $\partial\varphi^\sharp/\partial r=(\partial_n\varphi)^\sharp$ by \eqref{E:Def_NorDer}.
  From \eqref{Pf_P:FTC} and H\"{o}lder's inequality it follows that
  \begin{align*}
    |\varphi^\sharp(y,r)| &\leq |\varphi_0^\sharp(y)|+\int_{\varepsilon g_0(y)}^{\varepsilon g_1(y)}|(\partial_n\varphi)^\sharp(y,\tilde{r})|\,d\tilde{r} \\
    &\leq |\varphi_0^\sharp(y)|+c\varepsilon^{1-1/p}\left(\int_{\varepsilon g_0(y)}^{\varepsilon g_1(y)}|(\partial_n\varphi)^\sharp(y,\tilde{r})|^p\,d\tilde{r}\right)^{1/p}.
  \end{align*}
  Noting that the right-hand side is independent of $r$, we integrate the $p$-th power of both sides of this inequality with respect to $r$ to get
  \begin{align} \label{Pf_P:Int_r}
    \int_{\varepsilon g_0(y)}^{\varepsilon g_1(y)}|\varphi^\sharp(y,r)|^p\,dr \leq c\left(\varepsilon|\varphi_0^\sharp(y)|^p+\varepsilon^p\int_{\varepsilon g_0(y)}^{\varepsilon g_1(y)}|(\partial_n\varphi)^\sharp(y,\tilde{r})|^p\,d\tilde{r}\right).
  \end{align}
  Hence the inequalities \eqref{E:CoV_Equiv} and \eqref{Pf_P:Int_r} imply that
  \begin{align*}
    \|\varphi\|_{L^p(\Omega_\varepsilon)}^p \leq c\int_\Gamma\int_{\varepsilon g_0(y)}^{\varepsilon g_1(y)}|\varphi^\sharp(y,r)|^p\,dr\,d\mathcal{H}^2(y) \leq c\left(\varepsilon\|\varphi_0^\sharp\|_{L^p(\Gamma)}^p+\varepsilon^p\|\partial_n\varphi\|_{L^p(\Omega_\varepsilon)}^p\right).
  \end{align*}
  Applying \eqref{E:Lp_CoV_Surf} to the first term on the right-hand side we obtain \eqref{E:Poin_Dom}.

  Next let us prove \eqref{E:Poin_Bo}.
  From \eqref{Pf_P:FTC} we deduce that
  \begin{align*}
    |\varphi_0^\sharp(y)|^p \leq c\left(\varepsilon^{-1}\int_{\varepsilon g_0(y)}^{\varepsilon g_1(y)}|\varphi^\sharp(y,r)|^p\,dr+\varepsilon^{p-1}\int_{\varepsilon g_0(y)}^{\varepsilon g_1(y)}|(\partial_n\varphi)^\sharp(y,\tilde{r})|^p\,d\tilde{r}\right)
  \end{align*}
  as in the proof of \eqref{Pf_P:Int_r}.
  This inequality and \eqref{E:CoV_Equiv} imply that
  \begin{align*}
    \|\varphi_0^\sharp\|_{L^p(\Gamma)} \leq c\left(\varepsilon^{-1/p}\|\varphi\|_{L^p(\Omega_\varepsilon)}+\varepsilon^{1-1/p}\|\partial_n\varphi\|_{L^p(\Omega_\varepsilon)}\right).
  \end{align*}
  We apply \eqref{E:Lp_CoV_Surf} to the left-hand side of the this inequality to get \eqref{E:Poin_Bo}.
\end{proof}

We also show Agmon's inequality on $\Omega_\varepsilon$, which gives an estimate for the $L^\infty(\Omega_\varepsilon)$-norm of a function in $H^2(\Omega_\varepsilon)$ with explicit dependence on $\varepsilon$ of a bound.

\begin{lemma} \label{L:Agmon}
  There exists a constant $c>0$ independent of $\varepsilon$ such that
  \begin{multline} \label{E:Agmon}
    \|\varphi\|_{L^\infty(\Omega_\varepsilon)} \leq c\varepsilon^{-1/2}\|\varphi\|_{L^2(\Omega_\varepsilon)}^{1/4}\|\varphi\|_{H^2(\Omega_\varepsilon)}^{1/2} \\
    \times\left(\|\varphi\|_{L^2(\Omega_\varepsilon)}+\varepsilon\|\partial_n\varphi\|_{L^2(\Omega_\varepsilon)}+\varepsilon^2\|\partial_n^2\varphi\|_{L^2(\Omega_\varepsilon)}\right)^{1/4}
  \end{multline}
  for all $\varphi\in H^2(\Omega_\varepsilon)$.
\end{lemma}

\begin{proof}
  We use the anisotropic Agmon inequality (see \cite[Proposition~2.2]{TeZi96})
  \begin{align} \label{Pf_A:Ans_Agm}
    \|\Phi\|_{L^\infty(V)} \leq c\|\Phi\|_{L^2(V)}^{1/4}\prod_{i=1}^3\left(\|\Phi\|_{L^2(V)}+\|\partial_i\Phi\|_{L^2(V)}+\|\partial_i^2\Phi\|_{L^2(V)}\right)^{1/4}
  \end{align}
  for $\Phi\in H^2(V)$ with $V=(0,1)^3$.
  To this end, we localize $\varphi$ by using a partition of unity on $\Gamma$ and transform it into a function on $V$.
  Since $\Gamma$ is compact, we can take a finite number of open sets $U_k$ in $\mathbb{R}^2$ and local parametrizations $\mu_k\colon U_k\to\Gamma$, $k=1,\dots,k_0$ such that $\{\mu_k(U_k)\}_{k=1}^{k_0}$ is an open covering of $\Gamma$.
  Then setting
  \begin{align*}
    \zeta_k(s) := \mu_k(s')+\varepsilon\{(1-s_3)g_0(\mu_k(s'))+s_3g_1(\mu_k(s'))\}n(\mu_k(s'))
  \end{align*}
  for $s=(s',s_3)\in V_k:=U_k\times(0,1)$ we see that $\{\zeta_k(V_k)\}_{k=1}^{k_0}$ is an open covering of $\Omega_\varepsilon$.
  Let $\{\eta_k\}_{k=1}^{k_0}$ be a partition of unity on $\Gamma$ subordinate to $\{\mu_k(U_k)\}_{k=1}^{k_0}$, $\bar{\eta}_k=\eta_k\circ\pi$ the constant extension of $\eta_k$, and $\varphi_k:=\bar{\eta}_k\varphi$ for $k=1,\dots,k_0$.
  Then since $\{\bar{\eta}_k\}_{k=1}^{k_0}$ is a partition of unity on $\Omega_\varepsilon$ subordinate to $\{\zeta_k(V_k)\}_{k=1}^{k_0}$ and $\partial_n^l\varphi_k=\bar{\eta}_k\partial_n^l\varphi$ in $\Omega_\varepsilon$ for $l=1,2$ by \eqref{E:NorDer_Con}, it is sufficient for \eqref{E:Agmon} to show that
  \begin{multline} \label{Pf_A:Agm_Loc}
    \|\varphi_k\|_{L^\infty(\zeta_k(V_k))} \leq c\varepsilon^{-1/2}\|\varphi_k\|_{L^2(\zeta_k(V_k))}^{1/4}\|\varphi_k\|_{H^2(\zeta_k(V_k))}^{1/2} \\
    \times\left(\|\varphi_k\|_{L^2(\zeta_k(V_k))}+\varepsilon\|\partial_n\varphi_k\|_{L^2(\zeta_k(V_k))}+\varepsilon^2\|\partial_n^2\varphi_k\|_{L^2(\zeta_k(V_k))}\right)^{1/4}
  \end{multline}
  for all $k=1,\dots,k_0$.
  Let us prove \eqref{Pf_A:Agm_Loc}.
  Hereafter we fix and suppress $k$.
  Taking $U\subset\mathbb{R}^2$ small and scaling it, we may assume $U=(0,1)^2$ and $V=(0,1)^3$.
  The local parametrization $\zeta\colon V\to\Omega_\varepsilon$ of $\Omega_\varepsilon$ is given by
  \begin{align} \label{Pf_A:Def_Z}
    \zeta(s) := \mu(s')+\varepsilon\{(1-s_3)\bar{g}_0(\mu(s'))+s_3\bar{g}_1(\mu(s'))\}\bar{n}(\mu(s'))
  \end{align}
  for $s=(s',s_3)\in V$, where $\bar{g}_i=g_i\circ\pi$ and $\bar{n}=n\circ\pi$.
  We differentiate $\zeta(s)$ and apply \eqref{E:ConDer_Surf} with $y=\mu(s')\in\Gamma$ and $-\nabla_\Gamma n=W=W^T$ on $\Gamma$ to get
  \begin{align} \label{Pf_A:Deri_Zeta}
    \begin{aligned}
      \partial_{s_i}\zeta(s) &= \bigl\{I_3-h_\varepsilon(s)W(\mu(s'))\bigr\}\partial_{s_i}\mu(s')+\eta_\varepsilon^i(s)n(\mu(s')), \quad i=1,2, \\
      \partial_{s_3}\zeta(s) &= \varepsilon g(\mu(s'))n(\mu(s'))
    \end{aligned}
  \end{align}
  for $s=(s',s_3)\in V$, where
  \begin{align} \label{Pf_A:DZ_Aux}
    \begin{aligned}
      h_\varepsilon(s) &:= \varepsilon\{(1-s_3)g_0(\mu(s'))+s_3g_1(\mu(s'))\}, \\
      \eta_\varepsilon^i(s) &:= \varepsilon\partial_{s_i}\mu(s')\cdot\{(1-s_3)\nabla_\Gamma g_0(\mu(s'))+s_3\nabla_\Gamma g_1(\mu(s'))\}.
    \end{aligned}
  \end{align}
  Using these formulas we prove in Appendix~\ref{S:Ap_DG} that
  \begin{align} \label{Pf_A:Det_Zeta}
    \det\nabla_s\zeta(s) = \varepsilon g(\mu(s'))J(\mu(s'),h_\varepsilon(s))\sqrt{\det\theta(s')}, \quad s=(s',s_3)\in V,
  \end{align}
  where $\nabla_s\zeta$ is the gradient matrix of $\zeta$ and $\theta$ is the Riemannian metric of $\Gamma$ given by \eqref{E:Def_Met}.
  Now let $\Phi:=\varphi\circ\zeta$ on $V$.
  Then $\Phi$ is supported in $\mathcal{K}\times(0,1)$ with some compact subset $\mathcal{K}$ of $U$ since $\varphi$ is localized by the constant extension of a cut-off function on $\Gamma$.
  By \eqref{E:Width_Bound}, \eqref{E:Jac_Bound}, and \eqref{E:Metric} we have
  \begin{align} \label{Pf_A:DetZ_Bound}
    \det\nabla_s\zeta(s) \geq c\varepsilon, \quad s\in \mathcal{K}\times(0,1)
  \end{align}
  with a constant $c>0$ independent of $\varepsilon$.
  Moreover, by \eqref{Pf_A:Deri_Zeta}, \eqref{Pf_A:DZ_Aux}, \eqref{E:Mu_Bound}, and the boundedness of $n$ and $g$ on $\Gamma$ along their first and second order derivatives,
  \begin{align} \label{Pf_A:GrZ_Bound}
    |\partial_{s_i}\zeta(s)| \leq c, \quad |\partial_{s_i}\partial_{s_j}\zeta(s)| \leq c, \quad s\in\mathcal{K}\times(0,1),\,i,j=1,2,3.
  \end{align}
  Noting that $\Phi=\varphi\circ\zeta$ is supported in $\mathcal{K}\times(0,1)$, we observe by
  \begin{align} \label{Pf_A:Change}
    \int_{\zeta(V)}\varphi(x)\,dx = \int_V\Phi(s)\det\nabla_s\zeta(s)\,ds,
  \end{align}
  the inequalities \eqref{Pf_A:DetZ_Bound} and \eqref{Pf_A:GrZ_Bound}, and $\varphi\in H^2(\Omega_\varepsilon)$ that $\Phi\in H^2(V)$ and
  \begin{align} \label{Pf_A:Trans_Linf_L2}
    \|\Phi\|_{L^\infty(V)} = \|\varphi\|_{L^\infty(\zeta(V))}, \quad \|\Phi\|_{L^2(V)} \leq c\varepsilon^{-1/2}\|\varphi\|_{L^2(\zeta(V))}.
  \end{align}
  Also, we differentiate $\Phi(s)=\varphi(\zeta(s))$ and use \eqref{Pf_A:Deri_Zeta} and $\partial_n\varphi=(\bar{n}\cdot\nabla)\varphi$ to get
  \begin{align*}
    \partial_{s_i}\Phi(s) &= \partial_{s_i}\zeta(s)\cdot\nabla\varphi(\zeta(s)), \\
    \partial_{s_i}^2\Phi(s) &= \partial_{s_i}^2\zeta(s)\cdot\nabla\varphi(\zeta(s))+\partial_{s_i}\zeta(s)\cdot\nabla^2\varphi(\zeta(s))\partial_{s_i}\zeta(s), \\
    \partial_{s_3}\Phi(s) &= \varepsilon g(\mu(s'))\partial_n\varphi(\zeta(s)), \quad \partial_{s_3}^2\Phi(s) = \varepsilon^2g(\mu(s'))^2\partial_n^2\varphi(\zeta(s))
  \end{align*}
  for $s=(s',s_3)\in V$ and $i=1,2$.
  Hence, by \eqref{Pf_A:GrZ_Bound} and the boundedness of $g$ on $\Gamma$,
  \begin{gather*}
    |\partial_{s_i}\Phi(s)| \leq c|\nabla\varphi(\zeta(s))|, \quad |\partial_{s_i}^2\Phi(s)| \leq c(|\nabla\varphi(\zeta(s))|+|\nabla^2\varphi(\zeta(s))|), \\
    |\partial_{s_3}^k\Phi(s)| \leq c\varepsilon^k|\partial_n^k\varphi(\zeta(s))|
  \end{gather*}
  for $s\in\mathcal{K}\times(0,1)$ and $i,k=1,2$.
  Since $\Phi=\varphi\circ\zeta$ is supported in $\mathcal{K}\times(0,1)$, we deduce from the above inequalities, \eqref{Pf_A:DetZ_Bound}, and \eqref{Pf_A:Change} that
  \begin{align} \label{Pf_A:Trans_Hk}
    \|\partial_{s_i}^k\Phi\|_{L^2(V)} \leq c\varepsilon^{-1/2}\|\varphi\|_{H^k(\zeta(V))}, \quad \|\partial_{s_3}^k\Phi\|_{L^2(V)} \leq c\varepsilon^{k-1/2}\|\partial_n^k\varphi\|_{L^2(\zeta(V))}
  \end{align}
  for $i,k=1,2$.
  Applying the anisotropic Agmon inequality \eqref{Pf_A:Ans_Agm} to $\Phi\in H^2(V)$ and using \eqref{Pf_A:Trans_Linf_L2} and \eqref{Pf_A:Trans_Hk} we obtain \eqref{Pf_A:Agm_Loc}.
\end{proof}

\subsection{Consequences of the boundary conditions} \label{SS:Tool_Slip}
In this subsection we derive several properties from the boundary conditions
\begin{align}
  u\cdot n_\varepsilon &= 0, \label{E:Bo_Imp} \\
  2\nu P_\varepsilon D(u)n_\varepsilon+\gamma_\varepsilon u &= 0, \label{E:Bo_Slip}
\end{align}
where $D(u):=(\nabla u)_S=\{\nabla u+(\nabla u)^T\}/2$ is the strain rate tensor.
First we consider vector fields satisfying the impermeable boundary condition \eqref{E:Bo_Imp}.

\begin{lemma} \label{L:Exp_Bo}
  For $i=0,1$ let $u\in C(\Gamma_\varepsilon^i)^3$ satisfy \eqref{E:Bo_Imp} on $\Gamma_\varepsilon^i$.
  Then
  \begin{align} \label{E:Exp_Bo}
    u\cdot\bar{n} = \varepsilon u\cdot\bar{\tau}_\varepsilon^i, \quad |u\cdot\bar{n}| \leq c\varepsilon|u| \quad\text{on}\quad \Gamma_\varepsilon^i,
  \end{align}
  where $\tau_\varepsilon^i$ is given by \eqref{E:Def_NB_Aux} and $c>0$ is a constant independent of $\varepsilon$ and $u$.
\end{lemma}

\begin{proof}
  The first equality of \eqref{E:Exp_Bo} is an immediate consequence of \eqref{E:Def_NB}, \eqref{E:Nor_Bo}, and \eqref{E:Bo_Imp} on $\Gamma_\varepsilon^i$.
  This equality and the first inequality of \eqref{E:Tau_Bound} implies the second inequality of \eqref{E:Exp_Bo}.
\end{proof}

As a consequence of Lemma~\ref{L:Exp_Bo} we derive Poincar\'{e}'s inequalities for the normal component (with respect to $\Gamma$) of a vector field on $\Omega_\varepsilon$.

\begin{lemma} \label{L:Poin_Nor}
  Let $p\in[1,\infty)$.
  There exists $c>0$ independent of $\varepsilon$ such that
  \begin{align} \label{E:Poin_Nor}
    \|u\cdot\bar{n}\|_{L^p(\Omega_\varepsilon)} \leq c\varepsilon\|u\|_{W^{1,p}(\Omega_\varepsilon)}
  \end{align}
  for all $u\in W^{1,p}(\Omega_\varepsilon)^3$ satisfying \eqref{E:Bo_Imp} on $\Gamma_\varepsilon^0$ or on $\Gamma_\varepsilon^1$.
  We also have
  \begin{align} \label{E:Poin_Dnor}
    \left\|\overline{P}\nabla(u\cdot\bar{n})\right\|_{L^p(\Omega_\varepsilon)} \leq c\varepsilon\|u\|_{W^{2,p}(\Omega_\varepsilon)}
  \end{align}
  for all $u\in W^{2,p}(\Omega_\varepsilon)^3$ satisfying \eqref{E:Bo_Imp} on $\Gamma_\varepsilon^0$ or on $\Gamma_\varepsilon^1$.
\end{lemma}

\begin{proof}
  Let $u\in W^{1,p}(\Omega_\varepsilon)^3$.
  We may assume that $u$ satisfies \eqref{E:Bo_Imp} on $\Gamma_\varepsilon^0$ without loss of generality.
  By \eqref{E:NorDer_Con} and \eqref{E:Poin_Dom} with $i=0$,
  \begin{align} \label{Pf_PN:Lp_Dom}
    \|u\cdot\bar{n}\|_{L^p(\Omega_\varepsilon)} \leq c\left(\varepsilon^{1/p}\|u\cdot\bar{n}\|_{L^p(\Gamma_\varepsilon^0)}+\varepsilon\|\partial_nu\|_{L^p(\Omega_\varepsilon)}\right).
  \end{align}
  Moreover, we apply the second inequality of \eqref{E:Exp_Bo} and then use \eqref{E:Poin_Bo} with $i=0$ to the first term on the right-hand side of \eqref{Pf_PN:Lp_Dom} to get
  \begin{align} \label{Pf_PN:Lp_Bo}
    \|u\cdot\bar{n}\|_{L^p(\Gamma_\varepsilon^0)} \leq c\varepsilon\|u\|_{L^p(\Gamma_\varepsilon^0)} \leq c\varepsilon^{1-1/p}\|u\|_{W^{1,p}(\Omega_\varepsilon)}.
  \end{align}
  Combining \eqref{Pf_PN:Lp_Dom} and \eqref{Pf_PN:Lp_Bo} we obtain \eqref{E:Poin_Nor}.

  Next suppose that $u\in W^{2,p}(\Omega_\varepsilon)^3$ satisfies \eqref{E:Bo_Imp} on $\Gamma_\varepsilon^0$.
  Noting that
  \begin{align*}
    \left|\partial_n\Bigl[\overline{P}\nabla(u\cdot\bar{n})\Bigr]\right| \leq c(|\nabla u|+|\nabla^2u|) \quad\text{in}\quad \Omega_\varepsilon
  \end{align*}
  by \eqref{E:NorG_Bound} and \eqref{E:NorDer_Con}, we apply \eqref{E:Poin_Dom} with $i=0$ to get
  \begin{align} \label{Pf_PN:Grad_Lp}
    \left\|\overline{P}\nabla(u\cdot\bar{n})\right\|_{L^p(\Omega_\varepsilon)} \leq c\left(\varepsilon^{1/p}\left\|\overline{P}\nabla(u\cdot\bar{n})\right\|_{L^p(\Gamma_\varepsilon^0)}+\varepsilon\|u\|_{W^{2,p}(\Omega_\varepsilon)}\right).
  \end{align}
  Since the tangential gradient on $\Gamma_\varepsilon$ depends only on the values of a function on $\Gamma_\varepsilon$, we see by \eqref{E:PW_Bo} and the first equality of \eqref{E:Exp_Bo} that
  \begin{align*}
    \overline{P}\nabla(u\cdot\bar{n}) &= \nabla_{\Gamma_\varepsilon}(u\cdot\bar{n})+\Bigl(\overline{P}-P_\varepsilon\Bigr)\nabla(u\cdot\bar{n}) = \varepsilon\nabla_{\Gamma_\varepsilon}(u\cdot\bar{\tau}_\varepsilon^0)+\Bigl(\overline{P}-P_\varepsilon\Bigr)\nabla(u\cdot\bar{n}) \\
    &= \varepsilon P_\varepsilon\nabla(u\cdot\bar{\tau}_\varepsilon^0)+\Bigl(\overline{P}-P_\varepsilon\Bigr)\nabla(u\cdot\bar{n})
  \end{align*}
  on $\Gamma_\varepsilon^0$.
  By this formula, \eqref{E:ConDer_Bound}, \eqref{E:NorG_Bound}, \eqref{E:Tau_Bound}, and \eqref{E:Comp_P},
  \begin{align*}
    \left|\overline{P}\nabla(u\cdot\bar{n})\right| \leq c\varepsilon(|u|+|\nabla u|) \quad\text{on}\quad \Gamma_\varepsilon^0.
  \end{align*}
  From this inequality and \eqref{E:Poin_Bo} it follows that
  \begin{align*}
    \left\|\overline{P}\nabla(u\cdot\bar{n})\right\|_{L^p(\Gamma_\varepsilon^0)} \leq c\varepsilon\left(\|u\|_{L^p(\Gamma_\varepsilon^0)}+\|\nabla u\|_{L^p(\Gamma_\varepsilon^0)}\right) \leq c\varepsilon^{1-1/p}\|u\|_{W^{2,p}(\Omega_\varepsilon)}.
  \end{align*}
  Applying this inequality to the right-hand side of \eqref{Pf_PN:Grad_Lp} we obtain \eqref{E:Poin_Dnor}.
\end{proof}

The next lemma gives an expression for the normal component of the directional derivatives on $\Gamma_\varepsilon$ for vector fields satisfying the impermeable boundary condition in terms of the Weingarten map of $\Gamma_\varepsilon$.

\begin{lemma} \label{L:Imp_Der}
  Let $i=0,1$.
  For $u_1,u_2\in C^1(\overline{\Omega}_\varepsilon)^3$ satisfying \eqref{E:Bo_Imp} on $\Gamma_\varepsilon^i$ we have
  \begin{align} \label{E:Imp_Der}
    (u_1\cdot\nabla)u_2\cdot n_\varepsilon = W_\varepsilon u_1\cdot u_2 = u_1\cdot W_\varepsilon u_2 \quad\text{on}\quad \Gamma_\varepsilon^i.
  \end{align}
\end{lemma}

\begin{proof}
  Since $u_1$ satisfies \eqref{E:Bo_Imp} on $\Gamma_\varepsilon^i$, it is tangential on $\Gamma_\varepsilon^i$ and thus
  \begin{align*}
    (u_1\cdot\nabla)u_2\cdot n_\varepsilon = (u_1\cdot\nabla_{\Gamma_\varepsilon})u_2\cdot n_\varepsilon = u_1\cdot\nabla_{\Gamma_\varepsilon}(u_2\cdot n_\varepsilon)-u_2\cdot(u_1\cdot\nabla_{\Gamma_\varepsilon})n_\varepsilon
  \end{align*}
  on $\Gamma_\varepsilon^i$.
  The first term on the right-hand side vanishes by $u_2\cdot n_\varepsilon=0$ on $\Gamma_\varepsilon^i$ (note that the tangential gradient on $\Gamma_\varepsilon^i$ depends only on the values of a function on $\Gamma_\varepsilon^i$).
  Also, by $-\nabla_{\Gamma_\varepsilon}n_\varepsilon=W_\varepsilon=W_\varepsilon^T$ on $\Gamma_\varepsilon$ we have
  \begin{align*}
    (u_1\cdot\nabla_{\Gamma_\varepsilon})n_\varepsilon = -W_\varepsilon^Tu_1 = -W_\varepsilon u_1 \quad\text{on}\quad \Gamma_\varepsilon^i.
  \end{align*}
  Combining the above two equalities we obtain \eqref{E:Imp_Der}.
\end{proof}

Note that $(u_1\cdot\nabla)u_2\cdot n_\varepsilon$ can be expressed without using the derivatives of $u_1$ and $u_2$ by \eqref{E:Imp_Der}.
We use this fact in the analysis of boundary integrals, see Lemma~\ref{L:Korn_Grad}.

Next we give formulas for vector fields satisfying the slip boundary conditions \eqref{E:Bo_Imp}--\eqref{E:Bo_Slip}.

\begin{lemma} \label{L:NSl}
  Let $i=0,1$.
  If $u\in C^2(\overline{\Omega}_\varepsilon)^3$ satisfies \eqref{E:Bo_Imp}--\eqref{E:Bo_Slip} on $\Gamma_\varepsilon^i$, then
  \begin{gather}
    P_\varepsilon(n_\varepsilon\cdot\nabla)u = -W_\varepsilon u-\frac{\gamma_\varepsilon}{\nu}u \quad\text{on}\quad \Gamma_\varepsilon^i, \label{E:NSl_ND} \\
    n_\varepsilon\times\mathrm{curl}\,u = -n_\varepsilon\times\left\{n_\varepsilon\times\left(2W_\varepsilon u+\frac{\gamma_\varepsilon}{\nu}u\right)\right\} \quad\text{on}\quad \Gamma_\varepsilon^i. \label{E:NSl_Curl}
  \end{gather}
\end{lemma}

\begin{proof}
  Taking the tangential gradient of $u\cdot n_\varepsilon=0$ on $\Gamma_\varepsilon^i$ we have
  \begin{align} \label{Pf_PSl:Gu_N}
    (\nabla_{\Gamma_\varepsilon}u)n_\varepsilon = -(\nabla_{\Gamma_\varepsilon}n_\varepsilon)u = W_\varepsilon u \quad\text{on}\quad \Gamma_\varepsilon^i.
  \end{align}
  From this equality, \eqref{E:Bo_Slip}, and
  \begin{align*}
    2P_\varepsilon D(u)n_\varepsilon = P_\varepsilon\{(\nabla u)n_\varepsilon+(\nabla u)^Tn_\varepsilon\} = (\nabla_{\Gamma_\varepsilon}u)n_\varepsilon+P_\varepsilon(n_\varepsilon\cdot\nabla)u
  \end{align*}
  by \eqref{E:PW_Bo} we deduce that
  \begin{align*}
    P_\varepsilon(n_\varepsilon\cdot\nabla)u = -(\nabla_{\Gamma_\varepsilon}u)n_\varepsilon-\frac{\gamma_\varepsilon}{\nu}u = -W_\varepsilon u-\frac{\gamma_\varepsilon}{\nu}u \quad\text{on}\quad \Gamma_\varepsilon^i.
  \end{align*}
  Hence \eqref{E:NSl_ND} holds.
  To prove \eqref{E:NSl_Curl} we observe that the vector field $n_\varepsilon\times\mathrm{curl}\,u$ is tangential on $\Gamma_\varepsilon^i$.
  By this fact, \eqref{E:PW_Bo}, \eqref{E:NSl_ND}, and \eqref{Pf_PSl:Gu_N} we have
  \begin{align*}
    n_\varepsilon\times\mathrm{curl}\,u &= P_\varepsilon(n_\varepsilon\times\mathrm{curl}\,u) = P_\varepsilon\{(\nabla u)n_\varepsilon-(\nabla u)^Tn_\varepsilon\} \\
    &= (\nabla_{\Gamma_\varepsilon}u)n_\varepsilon-P_\varepsilon(n_\varepsilon\cdot\nabla)u = 2W_\varepsilon u+\frac{\gamma_\varepsilon}{\nu}u
  \end{align*}
  on $\Gamma_\varepsilon^i$.
  Noting that $n_\varepsilon\cdot u=0$, $n_\varepsilon\cdot W_\varepsilon u=0$, and $|n_\varepsilon|^2=1$ on $\Gamma_\varepsilon^i$, we conclude by
  \begin{align*}
    a\times(a\times b) = (a\cdot b)a-|a|^2b, \quad a,b\in\mathbb{R}^3
  \end{align*}
  with $a=n_\varepsilon$ and $b=2W_\varepsilon u+\nu^{-1}\gamma_\varepsilon u$ and the above equality that \eqref{E:NSl_Curl} is valid.
\end{proof}

Let us derive an estimate for the $L^p(\Omega_\varepsilon)$-norm of the tangential component (with respect to $\Gamma$) of the stress vector $D(u)\bar{n}$.
It is used in the study of a singular limit problem for \eqref{E:NS_Eq}--\eqref{E:NS_In} as $\varepsilon$ tends to zero.

\begin{lemma} \label{L:Poin_Str}
  Let $p\in[1,\infty)$.
  Suppose that the inequalities \eqref{E:Fric_Upper} are valid.
  Then there exists a constant $c>0$ independent of $\varepsilon$ such that
  \begin{align} \label{E:Poin_Str}
    \left\|\overline{P}D(u)\bar{n}\right\|_{L^p(\Omega_\varepsilon)} \leq c\varepsilon\|u\|_{W^{2,p}(\Omega_\varepsilon)}
  \end{align}
  for all $u\in W^{2,p}(\Omega_\varepsilon)^3$ satisfying \eqref{E:Bo_Slip} on $\Gamma_\varepsilon^0$ or on $\Gamma_\varepsilon^1$.
\end{lemma}

\begin{proof}
  We proceed as in the proof of Lemma~\ref{L:Poin_Nor}.
  Let $u\in W^{2,p}(\Omega_\varepsilon)^3$ satisfy \eqref{E:Bo_Slip} on $\Gamma_\varepsilon^i$ for $i=0$ or $i=1$.
  By \eqref{E:NorDer_Con} and the boundedness of $n$ and $P$ on $\Gamma$,
  \begin{align*}
    \left|\partial_n\Bigl[\overline{P}D(u)\bar{n}\Bigr]\right| \leq c|\nabla^2u| \quad\text{in}\quad \Omega_\varepsilon.
  \end{align*}
  Hence we use \eqref{E:Poin_Dom} to get
  \begin{align} \label{Pf_PSt:L2_Dom}
    \left\|\overline{P}D(u)\bar{n}\right\|_{L^p(\Omega_\varepsilon)} \leq c\left(\varepsilon^{1/p}\left\|\overline{P}D(u)\bar{n}\right\|_{L^p(\Gamma_\varepsilon^i)}+\varepsilon\|u\|_{W^{2,p}(\Omega_\varepsilon)}\right).
  \end{align}
  Moreover, since $u$ satisfies \eqref{E:Bo_Slip} on $\Gamma_\varepsilon^i$, we have
  \begin{align*}
    \overline{P}D(u)\bar{n} &= (-1)^{i+1}P_\varepsilon D(u)n_\varepsilon+P_\varepsilon D(u)\{\bar{n}-(-1)^{i+1}n_\varepsilon\}+\Bigl(\overline{P}-P_\varepsilon\Bigr)D(u)\bar{n} \\
    &= (-1)^i\frac{\gamma_\varepsilon}{2\nu}u+P_\varepsilon D(u)\{\bar{n}-(-1)^{i+1}n_\varepsilon\}+\Bigl(\overline{P}-P_\varepsilon\Bigr)D(u)\bar{n}
  \end{align*}
  on $\Gamma_\varepsilon^i$.
  Applying \eqref{E:Fric_Upper}, \eqref{E:Comp_N}, and \eqref{E:Comp_P} to the last line of this equality and noting that $P_\varepsilon$ is bounded uniformly in $\varepsilon$ we deduce that
  \begin{align*}
    \left|\overline{P}D(u)\bar{n}\right| \leq c\varepsilon(|u|+|\nabla u|) \quad\text{on}\quad \Gamma_\varepsilon^i.
  \end{align*}
  By this inequality and \eqref{E:Poin_Bo} we get
  \begin{align*}
    \left\|\overline{P}D(u)\bar{n}\right\|_{L^p(\Gamma_\varepsilon^i)} \leq c\varepsilon\left(\|u\|_{L^p(\Gamma_\varepsilon^i)}+\|\nabla u\|_{L^p(\Gamma_\varepsilon^i)}\right) \leq c\varepsilon^{1-1/p}\|u\|_{W^{2,p}(\Omega_\varepsilon)}.
  \end{align*}
  We apply this inequality to \eqref{Pf_PSt:L2_Dom} to obtain \eqref{E:Poin_Str}.
\end{proof}

Finally, we compare the tangential component (with respect to $\Gamma$) of the normal derivative of a vector field $u$ on $\Omega_\varepsilon$ with $-\overline{W}u$.

\begin{lemma} \label{L:PDnU_WU}
  Let $p\in[1,\infty)$.
  Suppose that the inequalities \eqref{E:Fric_Upper} are valid.
  Then there exists a constant $c>0$ independent of $\varepsilon$ such that
  \begin{align} \label{E:PDnU_WU}
    \left\|\overline{P}\partial_nu+\overline{W}u\right\|_{L^p(\Omega_\varepsilon)} \leq c\varepsilon\|u\|_{W^{2,p}(\Omega_\varepsilon)}
  \end{align}
  for all $u\in W^{2,p}(\Omega_\varepsilon)^3$ satisfying the slip boundary conditions \eqref{E:Bo_Imp}--\eqref{E:Bo_Slip} on $\Gamma_\varepsilon^0$ or on $\Gamma_\varepsilon^1$.
  Here $\partial_nu$ is the normal derivative of $u$ given by \eqref{E:Def_NorDer}.
\end{lemma}

\begin{proof}
  For $i=0$ or $i=1$ let $u\in W^{2,p}(\Omega_\varepsilon)^3$ satisfy \eqref{E:Bo_Imp}--\eqref{E:Bo_Slip} on $\Gamma_\varepsilon^i$.
  By \eqref{E:NorDer_Con} and the boundedness of $n$, $P$, and $W$ on $\Gamma$ we have
  \begin{align*}
    \left|\partial_n\Bigl[\overline{P}\partial_n u+\overline{W}u\Bigr]\right| \leq c(|\nabla u|+|\nabla^2u|) \quad\text{in}\quad \Omega_\varepsilon.
  \end{align*}
  We apply \eqref{E:Poin_Dom} and the above inequality to get
  \begin{align} \label{Pf_PuWu:L2_Dom}
    \left\|\overline{P}\partial_nu+\overline{W}u\right\|_{L^p(\Omega_\varepsilon)} \leq c\left(\varepsilon^{1/p}\left\|\overline{P}\partial_nu+\overline{W}u\right\|_{L^p(\Gamma_\varepsilon^i)}+\varepsilon\|u\|_{W^{2,p}(\Omega_\varepsilon)}\right).
  \end{align}
  Moreover, since $\partial_nu=(\bar{n}\cdot\nabla)u$ and $u$ satisfies \eqref{E:Bo_Imp}--\eqref{E:Bo_Slip} on $\Gamma_\varepsilon^i$,
  \begin{align*}
    \overline{P}\partial_nu &= (-1)^{i+1}P_\varepsilon(n_\varepsilon\cdot\nabla)u+P_\varepsilon[\{\bar{n}-(-1)^{i+1}n_\varepsilon\}\cdot\nabla]u+\Bigl(\overline{P}-P_\varepsilon\Bigr)(\bar{n}\cdot\nabla)u \\
    &= (-1)^{i+1}\left(-W_\varepsilon u-\frac{\gamma_\varepsilon}{\nu}u\right)+P_\varepsilon[\{\bar{n}-(-1)^{i+1}n_\varepsilon\}\cdot\nabla]u+\Bigl(\overline{P}-P_\varepsilon\Bigr)(\bar{n}\cdot\nabla)u
  \end{align*}
  on $\Gamma_\varepsilon^i$ by \eqref{E:NSl_ND}.
  Hence by \eqref{E:Fric_Upper} and \eqref{E:Comp_N}--\eqref{E:Comp_W} we get
  \begin{align*}
    \left|\overline{P}\partial_nu+\overline{W}u\right| \leq \left|\overline{W}u-(-1)^{i+1}W_\varepsilon u\right|+c\varepsilon(|u|+|\nabla u|) \leq c\varepsilon(|u|+|\nabla u|)
  \end{align*}
  on $\Gamma_\varepsilon^i$, which together with \eqref{E:Poin_Bo} implies that
  \begin{align*}
    \left\|\overline{P}\partial_nu+\overline{W}u\right\|_{L^p(\Gamma_\varepsilon^i)} \leq c\varepsilon\left(\|u\|_{L^p(\Gamma_\varepsilon^i)}+\|\nabla u\|_{L^p(\Gamma_\varepsilon^i)}\right) \leq c\varepsilon^{1-1/p}\|u\|_{W^{2,p}(\Omega_\varepsilon)}.
  \end{align*}
  Applying this inequality to \eqref{Pf_PuWu:L2_Dom} we obtain \eqref{E:PDnU_WU}.
\end{proof}

\subsection{Impermeable extension of surface vector fields} \label{SS:Tool_TE}
In the analysis of integrals over $\Omega_\varepsilon$ involving a vector field on $\Gamma$ it is convenient to consider its extension to $\Omega_\varepsilon$ satisfying the impermeable boundary condition on $\Gamma_\varepsilon$.
Let $\tau_\varepsilon^0$ and $\tau_\varepsilon^1$ be the vector fields on $\Gamma$ given by \eqref{E:Def_NB_Aux}.
We define a vector field $\Psi_\varepsilon$ on $N$ by
\begin{align} \label{E:Def_ExAux}
  \Psi_\varepsilon(x) := \frac{1}{\bar{g}(x)}\bigl\{\bigl(d(x)-\varepsilon\bar{g}_0(x)\bigr)\bar{\tau}_\varepsilon^1(x)+\bigl(\varepsilon\bar{g}_1(x)-d(x)\bigr)\bar{\tau}_\varepsilon^0(x)\bigr\}, \quad x\in N.
\end{align}
By definition, $\Psi_\varepsilon=\varepsilon\bar{\tau}_\varepsilon^i$ on $\Gamma_\varepsilon^i$, $i=0,1$.
Let us give several estimates for $\Psi_\varepsilon$.

\begin{lemma} \label{L:ExAux_Bound}
  There exists a constant $c>0$ independent of $\varepsilon$ such that
  \begin{align} \label{E:ExAux_Bound}
    |\Psi_\varepsilon| \leq c\varepsilon, \quad |\nabla\Psi_\varepsilon| \leq c, \quad |\nabla^2\Psi_\varepsilon| \leq c \quad\text{in}\quad \Omega_\varepsilon.
  \end{align}
  Moreover, we have
  \begin{align} \label{E:ExAux_TNDer}
    \left|\overline{P}\nabla\Psi_\varepsilon\right| \leq c\varepsilon, \quad \left|\partial_n\Psi_\varepsilon-\frac{1}{\bar{g}}\overline{\nabla_\Gamma g}\right| \leq c\varepsilon \quad\text{in}\quad \Omega_\varepsilon.
  \end{align}
\end{lemma}

\begin{proof}
  Applying \eqref{E:Tau_Bound} and
  \begin{align} \label{Pf_EAB:Width}
    0 \leq d(x)-\varepsilon\bar{g}_0(x) \leq \varepsilon\bar{g}(x), \quad 0 \leq \varepsilon\bar{g}_1(x)-d(x) \leq \varepsilon\bar{g}(x), \quad x\in \Omega_\varepsilon
  \end{align}
  to \eqref{E:Def_ExAux} we get the first inequality of \eqref{E:ExAux_Bound}.
  Also, by $\nabla d=\bar{n}$ in $N$ we have
  \begin{align} \label{Pf_EAB:Grad}
    \nabla\Psi_\varepsilon = \frac{1}{\bar{g}}\{\bar{n}\otimes(\bar{\tau}_\varepsilon^1-\bar{\tau}_\varepsilon^0)+F_\varepsilon\} \quad\text{in}\quad N,
  \end{align}
  where $F_\varepsilon$ is a $3\times3$ matrix-valued function on $N$ given by
  \begin{align*}
    F_\varepsilon := -\nabla\bar{g}\otimes\Psi_\varepsilon+\varepsilon(\nabla\bar{g}_1\otimes\bar{\tau}_\varepsilon^0-\nabla\bar{g}_0\otimes\bar{\tau}_\varepsilon^1)+(d-\varepsilon\bar{g}_0)\nabla\bar{\tau}_\varepsilon^1+(\varepsilon\bar{g}_1-d)\nabla\bar{\tau}_\varepsilon^0.
  \end{align*}
  By \eqref{E:Width_Bound}, \eqref{E:Tau_Bound}, and $|n|=1$ on $\Gamma$ we see that the first term on the right-hand side of \eqref{Pf_EAB:Grad} is bounded on $N$ uniformly in $\varepsilon$.
  Moreover, from \eqref{E:ConDer_Bound}, \eqref{E:Width_Bound}, \eqref{E:Tau_Bound}, the first inequality of \eqref{E:ExAux_Bound}, \eqref{Pf_EAB:Width}, and $g_0,g_1\in C^4(\Gamma)$ we deduce that
  \begin{align} \label{Pf_EAB:Est_F}
    |F_\varepsilon| \leq c\varepsilon \quad\text{in}\quad \Omega_\varepsilon.
  \end{align}
  Hence the second inequality of \eqref{E:ExAux_Bound} follows.
  Similarly, differentiating both sides of \eqref{Pf_EAB:Grad} and using \eqref{E:ConDer_Bound}, \eqref{E:Con_Hess}, \eqref{E:Tau_Bound}, the first and second inequalities of \eqref{E:ExAux_Bound}, and $g_0,g_1\in C^4(\Gamma)$ we can derive the last inequality of \eqref{E:ExAux_Bound}.

  Let us prove \eqref{E:ExAux_TNDer}.
  First note that
  \begin{align*}
    P[n\otimes(\tau_\varepsilon^1-\tau_\varepsilon^0)] &= (Pn)\otimes(\tau_\varepsilon^1-\tau_\varepsilon^0) = 0, \\
    [(\tau_\varepsilon^1-\tau_\varepsilon^0)\otimes n]n &= |n|^2(\tau_\varepsilon^1-\tau_\varepsilon^0) = \tau_\varepsilon^1-\tau_\varepsilon^0
  \end{align*}
  on $\Gamma$.
  These equalities and \eqref{E:Def_NorDer} imply that
  \begin{align*}
    \overline{P}\nabla\Psi_\varepsilon = \frac{1}{\bar{g}}\overline{P}F_\varepsilon, \quad \partial_n\Psi_\varepsilon = (\nabla\Psi_\varepsilon)^T\bar{n} = \frac{1}{\bar{g}}(\bar{\tau}_\varepsilon^1-\bar{\tau}_\varepsilon^0)+F_\varepsilon^T\bar{n} \quad\text{in}\quad N.
  \end{align*}
  Hence we see by \eqref{E:Width_Bound}, \eqref{Pf_EAB:Est_F}, and $|P|=2$ on $\Gamma$ that
  \begin{align*}
    \left|\overline{P}\nabla\Psi_\varepsilon\right| \leq c\left|\overline{P}F_\varepsilon\right| \leq c|F_\varepsilon| \leq c\varepsilon \quad\text{in}\quad \Omega_\varepsilon.
  \end{align*}
  Also, applying \eqref{E:Width_Bound}, \eqref{E:Tau_Diff}, and \eqref{Pf_EAB:Est_F} to the equality for $\partial_n\Psi_\varepsilon$ we obtain
  \begin{align*}
    \left|\partial_n\Psi_\varepsilon-\frac{1}{\bar{g}}\overline{\nabla_\Gamma g}\right| \leq \frac{1}{\bar{g}}\sum_{i=0,1}\left|\bar{\tau}_\varepsilon^i-\overline{\nabla_\Gamma g_i}\right|+|F_\varepsilon| \leq c\varepsilon \quad\text{in}\quad \Omega_\varepsilon.
  \end{align*}
  Thus the inequalities \eqref{E:ExAux_TNDer} are valid.
\end{proof}

For a tangential vector field $v$ on $\Gamma$ (i.e. $v\cdot n=0$ on $\Gamma$) we define
\begin{align}  \label{E:Def_ExTan}
  E_\varepsilon v(x) := \bar{v}(x)+\{\bar{v}(x)\cdot\Psi_\varepsilon(x)\}\bar{n}(x), \quad x\in N,
\end{align}
where $\bar{v}$ and $\bar{n}$ are the constant extensions of $v$ and $n$.
By the definition of $\Psi_\varepsilon$ we easily see that $E_\varepsilon v$ satisfies the impermeable boundary condition on $\Gamma_\varepsilon$.

\begin{lemma} \label{L:ExTan_Imp}
  For all $v\in C(\Gamma,T\Gamma)$ we have $E_\varepsilon v\cdot n_\varepsilon=0$ on $\Gamma_\varepsilon$.
\end{lemma}

\begin{proof}
  For $i=0,1$ we observe by \eqref{E:Def_NB}, \eqref{E:Nor_Bo}, and $v\cdot n=0$ on $\Gamma$ that
  \begin{align*}
    \bar{v}\cdot n_\varepsilon = (-1)^i\frac{\varepsilon\bar{v}\cdot\bar{\tau}_\varepsilon^i}{\sqrt{1+\varepsilon^2|\bar{\tau}_\varepsilon^i|^2}}, \quad \bar{n}\cdot n_\varepsilon = \frac{(-1)^{i+1}}{\sqrt{1+\varepsilon^2|\bar{\tau}_\varepsilon^i|^2}} \quad\text{on}\quad \Gamma_\varepsilon^i.
  \end{align*}
  From these equalities and $\Psi_\varepsilon=\varepsilon\bar{\tau}_\varepsilon^i$ on $\Gamma_\varepsilon^i$ by \eqref{E:Def_ExAux} we get $E_\varepsilon v\cdot n_\varepsilon=0$ on $\Gamma_\varepsilon^i$.
\end{proof}

Also, it is easy to show that $E_\varepsilon v\in W^{m,p}(\Omega_\varepsilon)$ for $v\in W^{m,p}(\Gamma,T\Gamma)$.

\begin{lemma} \label{L:ExTan_Wmp}
  There exists a constant $c>0$ independent of $\varepsilon$ such that
  \begin{align} \label{E:ExTan_Wmp}
    \|E_\varepsilon v\|_{W^{m,p}(\Omega_\varepsilon)} \leq c\varepsilon^{1/2}\|v\|_{W^{m,p}(\Gamma)}
  \end{align}
  for all $v\in W^{m,p}(\Gamma,T\Gamma)$ with $p\in[1,\infty)$ and $m=0,1,2$.
\end{lemma}

\begin{proof}
  By \eqref{E:ConDer_Bound}, \eqref{E:Con_Hess}, \eqref{E:NorG_Bound}, and \eqref{E:ExAux_Bound} we have
  \begin{gather*}
    |E_\varepsilon v| \leq c|\bar{v}|, \quad |\nabla E_\varepsilon v| \leq c\left(|\bar{v}|+\left|\overline{\nabla_\Gamma v}\right|\right), \quad |\nabla^2 E_\varepsilon v| \leq c\left(|\bar{v}|+\left|\overline{\nabla_\Gamma v}\right|+\left|\overline{\nabla_\Gamma^2 v}\right|\right)
  \end{gather*}
  in $\Omega_\varepsilon$.
  These inequalities and \eqref{E:Con_Lp} imply \eqref{E:ExTan_Wmp}.
\end{proof}

If $\widetilde{\Omega}_\varepsilon$ is a flat thin domain of the form
\begin{align*}
  \widetilde{\Omega}_\varepsilon = \{x=(x',x_3)\in\mathbb{R}^3 \mid x'\in\omega,\,\varepsilon\tilde{g}_0(x')<x_3<\varepsilon\tilde{g}_1(x')\},
\end{align*}
where $\omega$ is a domain in $\mathbb{R}^2$ and $\tilde{g}_0$ and $\tilde{g}_1$ are functions on $\omega$, then we have
\begin{align*}
  \mathrm{div}(E_\varepsilon v)(x) = \frac{1}{\tilde{g}(x')}\mathrm{div}(\tilde{g}v)(x'), \quad x=(x',x_3)\in\widetilde{\Omega}_\varepsilon \quad (\tilde{g}:=\tilde{g}_1-\tilde{g}_0)
\end{align*}
for $v\colon\omega\to\mathbb{R}^2$ (see~\cite[Lemma~4.24]{HoSe10} and~\cite[Remark~3.1]{IfRaSe07}).
This is not the case for the curved thin domain $\Omega_\varepsilon$ given by \eqref{E:Intro_CTD} because the principal curvatures of the surface $\Gamma$ do not vanish in general.
However, we can show that the difference between $\mathrm{div}(E_\varepsilon v)$ and $g^{-1}\mathrm{div}_\Gamma(gv)$ is of order $\varepsilon$ in $\Omega_\varepsilon$.

\begin{lemma} \label{L:ExTan_Div}
  There exists a constant $c>0$ independent of $\varepsilon$ such that
  \begin{align} \label{E:ExTan_Grad}
    \left|\nabla E_\varepsilon v-\left\{\overline{\nabla_\Gamma v}+\frac{1}{\bar{g}}\Bigl(\bar{v}\cdot\overline{\nabla_\Gamma g}\Bigr)\overline{Q}\right\}\right| \leq c\varepsilon\left(|\bar{v}|+\left|\overline{\nabla_\Gamma v}\right|\right) \quad\text{in}\quad \Omega_\varepsilon
  \end{align}
  for all $v\in C^1(\Gamma,T\Gamma)$.
  Moreover, we have
  \begin{align} \label{E:ExTan_Div}
    \left|\mathrm{div}(E_\varepsilon v)-\frac{1}{\bar{g}}\overline{\mathrm{div}_\Gamma(gv)}\right| \leq c\varepsilon\left(|\bar{v}|+\left|\overline{\nabla_\Gamma v}\right|\right) \quad\text{in}\quad \Omega_\varepsilon.
  \end{align}
\end{lemma}

\begin{proof}
  Since
  \begin{align*}
    \nabla E_\varepsilon v = \nabla\bar{v}+[(\nabla\bar{v})\Psi_\varepsilon+(\nabla\Psi_\varepsilon)\bar{v}]\otimes\bar{n}+(\bar{v}\cdot\Psi_\varepsilon)\nabla\bar{n} \quad\text{in}\quad N
  \end{align*}
  by \eqref{E:Def_ExTan} and $(v\cdot\nabla_\Gamma g)Q=[(v\cdot\nabla_\Gamma g)n]\otimes n=[(n\otimes\nabla_\Gamma g)v]\otimes n$ on $\Gamma$,
  \begin{multline} \label{Pf_ETD:Diff}
    \left|\nabla E_\varepsilon v-\left\{\overline{\nabla_\Gamma v}+\frac{1}{\bar{g}}\Bigl(\bar{v}\cdot\overline{\nabla_\Gamma g}\Bigr)\overline{Q}\right\}\right| \\
    \leq \left|\nabla\bar{v}-\overline{\nabla_\Gamma v}\right|+|\{(\nabla\bar{v})\Psi_\varepsilon\}\otimes\bar{n}|+|(\bar{v}\cdot\Psi_\varepsilon)\nabla\bar{n}| \\
    +\left|\left[\left(\nabla\Psi_\varepsilon-\frac{1}{\bar{g}}\bar{n}\otimes\overline{\nabla_\Gamma g}\right)\bar{v}\right]\otimes\bar{n}\right| \quad\text{in}\quad N.
  \end{multline}
  By $\nabla\Psi_\varepsilon=\overline{P}\nabla\Psi_\varepsilon+\overline{Q}\nabla\Psi_\varepsilon=\overline{P}\nabla\Psi_\varepsilon+\bar{n}\otimes\partial_n\Psi_\varepsilon$ in $N$ and \eqref{E:ExAux_TNDer} we get
  \begin{align*}
    \left|\nabla\Psi_\varepsilon-\frac{1}{\bar{g}}\bar{n}\otimes\overline{\nabla_\Gamma g}\right| \leq \left|\overline{P}\nabla\Psi_\varepsilon\right|+\left|\bar{n}\otimes\left(\partial_n\Psi_\varepsilon-\frac{1}{\bar{g}}\overline{\nabla_\Gamma g}\right)\right| \leq c\varepsilon \quad\text{in}\quad \Omega_\varepsilon.
  \end{align*}
  Applying this inequality, \eqref{E:ConDer_Bound}, \eqref{E:ConDer_Diff}, \eqref{E:NorG_Bound}, \eqref{E:ExAux_Bound}, and $|d|\leq c\varepsilon$ in $\Omega_\varepsilon$ to the right-hand side of \eqref{Pf_ETD:Diff} we obtain \eqref{E:ExTan_Grad}.
  Also, since $\mathrm{tr}[Q]=n\cdot n=1$ on $\Gamma$,
  \begin{align*}
    \mathrm{tr}\left[\nabla_\Gamma v+\frac{1}{g}(v\cdot\nabla_\Gamma g)Q\right] = \mathrm{div}_\Gamma v+\frac{1}{g}(v\cdot\nabla_\Gamma g) = \frac{1}{g}\mathrm{div}_\Gamma(gv) \quad\text{on}\quad \Gamma.
  \end{align*}
  Hence the inequality \eqref{E:ExTan_Div} follows from \eqref{E:ExTan_Grad}.
\end{proof}

As a consequence of Lemma~\ref{L:ExTan_Div} we have the $L^p$-estimate for $\mathrm{div}(E_\varepsilon v)$ on $\Omega_\varepsilon$.

\begin{lemma} \label{L:Lp_ET_Div}
  Let $p\in[1,\infty)$.
  There exists $c>0$ independent of $\varepsilon$ such that
  \begin{align} \label{E:Lp_ET_Div}
    \|\mathrm{div}(E_\varepsilon v)\|_{L^p(\Omega_\varepsilon)} \leq c\varepsilon^{1/p}\left(\|\mathrm{div}_\Gamma(gv)\|_{L^p(\Gamma)}+\varepsilon\|v\|_{W^{1,p}(\Gamma)}\right)
  \end{align}
  for all $v\in W^{1,p}(\Gamma,T\Gamma)$.
  In particular, if $v$ satisfies $\mathrm{div}_\Gamma(gv)=0$ on $\Gamma$, then
  \begin{align} \label{E:Lp_ETD_Sol}
    \|\mathrm{div}(E_\varepsilon v)\|_{L^p(\Omega_\varepsilon)} \leq c\varepsilon^{1+1/p}\|v\|_{W^{1,p}(\Gamma)}.
  \end{align}
\end{lemma}

\begin{proof}
  By \eqref{E:Width_Bound} and \eqref{E:ExTan_Div} we have
  \begin{align*}
    |\mathrm{div}(E_\varepsilon v)| \leq c\left\{\left|\overline{\mathrm{div}_\Gamma(gv)}\right|+\varepsilon\left(|\bar{v}|+\left|\overline{\nabla_\Gamma v}\right|\right)\right\} \quad\text{in}\quad \Omega_\varepsilon.
  \end{align*}
  The inequality \eqref{E:Lp_ET_Div} follows from this inequality and \eqref{E:Con_Lp}.
\end{proof}

\section{Korn inequalities on a thin domain and a closed surface} \label{S:Korn}
In this section we provide the uniform Korn inequalities on $\Omega_\varepsilon$, which play an important role in the study of the Stokes operator on $\Omega_\varepsilon$.
We also show the Korn inequality on $\Gamma$ used in the study of a singular limit problem for \eqref{E:NS_Eq}--\eqref{E:NS_In}.

\subsection{Uniform Korn inequalities on a thin domain} \label{SS:Korn_Dom}
For the proof of the global existence of a strong solution to \eqref{E:NS_Eq}--\eqref{E:NS_In}, the uniform coerciveness of the bilinear form corresponding to the Stokes problem in $\Omega_\varepsilon$ with slip boundary conditions is essential.
To establish it in Section~\ref{SS:St_Def}, we prove the uniform Korn inequalities on $\Omega_\varepsilon$.
First we give an $L^2$-estimate for the gradient matrix of a vector field on $\Omega_\varepsilon$.

\begin{lemma} \label{L:Korn_Grad}
  There exists a constant $c_{K,1}>0$ independent of $\varepsilon$ such that
  \begin{align} \label{E:Korn_Grad}
    \|\nabla u\|_{L^2(\Omega_\varepsilon)}^2 \leq 4\|D(u)\|_{L^2(\Omega_\varepsilon)}^2+c_{K,1}\|u\|_{L^2(\Omega_\varepsilon)}^2
  \end{align}
  for all $\varepsilon\in(0,1)$ and $u\in H^1(\Omega_\varepsilon)^3$ satisfying \eqref{E:Bo_Imp} on $\Gamma_\varepsilon$.
\end{lemma}

Let us prove an auxiliary density result.

\begin{lemma} \label{L:H1N0_Dense}
  Let $u\in H^1(\Omega_\varepsilon)^3$ satisfy \eqref{E:Bo_Imp} on $\Gamma_\varepsilon$.
  Then there exists a sequence $\{u_k\}_{k=1}^\infty$ in $C^2(\overline{\Omega}_\varepsilon)^3$ such that $u_k$ satisfies \eqref{E:Bo_Imp} on $\Gamma_\varepsilon$ for each $k\in\mathbb{N}$ and
  \begin{align*}
    \lim_{k\to\infty}\|u-u_k\|_{H^1(\Omega_\varepsilon)} = 0.
  \end{align*}
\end{lemma}

\begin{proof}
  We follow the idea of the proof of~\cite[Theorem~IV.4.7]{BoFa13}, but here it is not necessary to localize a vector field on $\Omega_\varepsilon$.
  For $x\in N$ we define
  \begin{align*}
    \tilde{n}(x) := \frac{1}{\varepsilon \bar{g}(x)}\bigl\{\bigl(d(x)-\varepsilon\bar{g}_0(x)\bigr)\bar{n}_\varepsilon^1(x)+\bigl(\varepsilon\bar{g}_1(x)-d(x)\bigr)\bar{n}_\varepsilon^0(x)\bigr\},
  \end{align*}
  where $n_\varepsilon^0$ and $n_\varepsilon^1$ are given by \eqref{E:Def_NB} and $\bar{\eta}=\eta\circ\pi$ denotes the constant extension of a function $\eta$ on $\Gamma$.
  Then $\tilde{n}\in C^2(N)$ by the regularity of $\Gamma$, $g_0$, and $g_1$.
  Moreover, $\tilde{n}=n_\varepsilon$ on $\Gamma_\varepsilon$ by Lemma~\ref{L:Nor_Bo}.
  Hence if $u\in H^1(\Omega_\varepsilon)^3$ satisfies \eqref{E:Bo_Imp} on $\Gamma_\varepsilon$, then we have $u\cdot\tilde{n}\in H_0^1(\Omega_\varepsilon)$ and $w:=u-(u\cdot\tilde{n})\tilde{n}\in H^1(\Omega_\varepsilon)^3$.
  Since $\Gamma_\varepsilon$ is of class $C^4$, there exist sequences $\{\varphi_k\}_{k=1}^\infty$ in $C_c^\infty(\Omega_\varepsilon)$ and $\{w_k\}_{k=1}^\infty$ in $C^\infty(\overline{\Omega}_\varepsilon)^3$ such that
  \begin{align*}
    \lim_{k\to\infty}\|u\cdot\tilde{n}-\varphi_k\|_{H^1(\Omega_\varepsilon)} = \lim_{k\to\infty}\|w-w_k\|_{H^1(\Omega_\varepsilon)} = 0.
  \end{align*}
  Here $C_c^\infty(\Omega_\varepsilon)$ is the space of all smooth and compactly supported functions on $\Omega_\varepsilon$.
  Therefore, setting $u_k:=\varphi_k\tilde{n}+w_k-(w_k\cdot\tilde{n})\tilde{n}\in C^2(\overline{\Omega}_\varepsilon)$ we see that
  \begin{align*}
    u_k\cdot n_\varepsilon = u_k\cdot\tilde{n} = \varphi_k = 0 \quad\text{on}\quad \Gamma_\varepsilon
  \end{align*}
  for each $k\in\mathbb{N}$ and (note that $u=(u\cdot\tilde{n})\tilde{n}+w$ and $w\cdot\tilde{n}=0$ in $\Omega_\varepsilon$)
  \begin{align*}
    \|u-u_k\|_{H^1(\Omega_\varepsilon)} &= \|(u\cdot\tilde{n}-\varphi_k)\tilde{n}+(w-w_k)-\{(w-w_k)\cdot\tilde{n}\}\tilde{n}\|_{H^1(\Omega_\varepsilon)} \\
    &\leq c_\varepsilon\left(\|u\cdot\tilde{n}-\varphi_k\|_{H^1(\Omega_\varepsilon)}+\|w-w_k\|_{H^1(\Omega_\varepsilon)}\right) \to 0
  \end{align*}
  as $k\to\infty$ (here $c_\varepsilon>0$ may depend on $\varepsilon$ but is independent of $k$).
\end{proof}

\begin{proof}[Proof of Lemma~\ref{L:Korn_Grad}]
  By Lemma~\ref{L:H1N0_Dense} and a density argument it suffices to show \eqref{E:Korn_Grad} for all $u\in C^2(\overline{\Omega}_\varepsilon)^3$ satisfying \eqref{E:Bo_Imp} on $\Gamma_\varepsilon$.
  For such $u$ we can apply integration by parts twice and \eqref{E:Bo_Imp} on $\Gamma_\varepsilon$ to get
  \begin{align*}
    \int_{\Omega_\varepsilon}\nabla u:(\nabla u)^T\,dx &= \int_{\Omega_\varepsilon}(\mathrm{div}\,u)^2\,dx+\int_{\Gamma_\varepsilon}\{(u\cdot\nabla)u\cdot n_\varepsilon-(u\cdot n_\varepsilon)\mathrm{div}\,u\}\,d\mathcal{H}^2 \\
    &\geq \int_{\Gamma_\varepsilon}(u\cdot\nabla)u\cdot n_\varepsilon\,d\mathcal{H}^2.
  \end{align*}
  From this inequality and $|\nabla u|^2=2|D(u)|^2-\nabla u:(\nabla u)^T$ in $\Omega_\varepsilon$ we deduce that
  \begin{align} \label{Pf_KG:First_Est}
    \|\nabla u\|_{L^2(\Omega_\varepsilon)}^2 \leq 2\|D(u)\|_{L^2(\Omega_\varepsilon)}^2-\int_{\Gamma_\varepsilon}(u\cdot\nabla)u\cdot n_\varepsilon\,d\mathcal{H}^2.
  \end{align}
  Since $u$ satisfies \eqref{E:Bo_Imp} on $\Gamma_\varepsilon$, we can apply \eqref{E:Imp_Der} to the last term to get
  \begin{align} \label{Pf_KG:Bound_Int}
    \int_{\Gamma_\varepsilon}(u\cdot\nabla)u\cdot n_\varepsilon\,d\mathcal{H}^2 = \int_{\Gamma_\varepsilon}u\cdot W_\varepsilon u\,d\mathcal{H}^2 = \sum_{i=0,1}\int_{\Gamma_\varepsilon^i}u\cdot W_\varepsilon u\,d\mathcal{H}^2.
  \end{align}
  To estimate the right-hand side we set
  \begin{align} \label{Pf_KG:Def_Aux}
    \begin{aligned}
      F_i(y) &:= \sqrt{1+\varepsilon^2|\tau_\varepsilon^i(y)|^2}\,W_{\varepsilon,i}^\sharp(y), \quad i=0,1, \\
      F(y,r) &:= \frac{1}{\varepsilon g(y)}\bigl\{\bigl(r-\varepsilon g_0(y)\bigr)F_1(y)-\bigl(\varepsilon g_1(y)-r\bigr)F_0(y)\bigr\}, \\
      \varphi(y,r) &:= u^\sharp(y,r)\cdot F(y,r)u^\sharp(y,r)J(y,r)
    \end{aligned}
  \end{align}
  for $y\in\Gamma$ and $r\in[\varepsilon g_0(y),\varepsilon g_1(y)]$ with $\tau_\varepsilon^i$, $i=0,1$ and $J$ given by \eqref{E:Def_NB_Aux} and \eqref{E:Def_Jac}.
  Here and in what follows we use the notations \eqref{E:Pull_Dom} and \eqref{E:Pull_Bo} and sometimes suppress the arguments $y$ and $r$.
  By \eqref{Pf_KG:Def_Aux} we see that
  \begin{align*}
    [u\cdot W_\varepsilon u]_i^\sharp(y)\sqrt{1+\varepsilon^2|\tau_\varepsilon^i(y)|^2}J(y,\varepsilon g_i(y)) = (-1)^{i+1}\varphi(y,\varepsilon g_i(y)), \quad y\in\Gamma,\,i=0,1.
  \end{align*}
  From this relation and \eqref{E:CoV_Surf} we deduce that
  \begin{align} \label{Pf_KG:Bound_FTC}
    \begin{aligned}
      \sum_{i=0,1}\int_{\Gamma_\varepsilon^i}[u\cdot W_\varepsilon u](x)\,d\mathcal{H}^2(x) &= \int_\Gamma\{\varphi(y,\varepsilon g_1(y))-\varphi(y,\varepsilon g_0(y))\}\,d\mathcal{H}^2(y) \\
      &= \int_\Gamma\int_{\varepsilon g_0(y)}^{\varepsilon g_1(y)}\frac{\partial\varphi}{\partial r}(y,r)\,dr\,d\mathcal{H}^2(y).
    \end{aligned}
  \end{align}
  To estimate the integrand on the last line we use \eqref{E:Jac_Bound} to get
  \begin{align} \label{PF_KN:Dphi_1}
    \left|\frac{\partial\varphi}{\partial r}\right| \leq c\left\{\left(|F|+\left|\frac{\partial F}{\partial r}\right|\right)|u^\sharp|^2+|F||u^\sharp||(\nabla u)^\sharp|\right\}.
  \end{align}
  By \eqref{E:Tau_Bound} and the uniform boundedness in $\varepsilon$ of $W_\varepsilon$ on $\Gamma_\varepsilon$ we observe that $F_0$ and $F_1$ are bounded on $\Gamma$ uniformly in $\varepsilon$.
  Thus we have
  \begin{align} \label{Pf_KG:Est_F}
    |F(y,r)| \leq \frac{c}{\varepsilon g(y)}\bigl\{\bigl(r-\varepsilon g_0(y)\bigr)+\bigl(\varepsilon g_1(y)-r\bigr)\} = c
  \end{align}
  for $y\in\Gamma$ and $r\in[\varepsilon g_0(y),\varepsilon g_1(y)]$.
  Also, by $\partial F/\partial r=(\varepsilon g)^{-1}(F_1+F_0)$ and \eqref{Pf_KG:Def_Aux},
  \begin{align} \label{Pf_KG:DF_1}
    \left|\frac{\partial F}{\partial r}\right| \leq c\varepsilon^{-1}\left(|W_{\varepsilon,1}^\sharp+W_{\varepsilon,0}^\sharp|+\sum_{i=0,1}\left(\sqrt{1+\varepsilon^2|\tau_\varepsilon^i|^2}-1\right)|W_{\varepsilon,i}^\sharp|\right).
  \end{align}
  By the mean value theorem for the function $\sqrt{1+s}$, $s\geq0$ and \eqref{E:Tau_Bound} we have
  \begin{align} \label{Pf_KG:Sqrt}
    (0 \leq)\, \sqrt{1+\varepsilon^2|\tau_\varepsilon^i(y)|^2}-1 \leq \frac{\varepsilon^2}{2}|\tau_\varepsilon^i(y)|^2 \leq c\varepsilon^2, \quad y\in\Gamma.
  \end{align}
  We apply this inequality, \eqref{E:Diff_WH_IO} with $G_\varepsilon=W_\varepsilon$, and the uniform boundedness in $\varepsilon$ of $W_\varepsilon$ to the right-hand side of \eqref{Pf_KG:DF_1} to obtain
  \begin{align} \label{Pf_KG:Est_DF}
    \left|\frac{\partial F}{\partial r}(y,r)\right| \leq c \quad\text{for all}\quad y\in\Gamma,\,r\in[\varepsilon g_0(y),\varepsilon g_1(y)].
  \end{align}
  From \eqref{PF_KN:Dphi_1}, \eqref{Pf_KG:Est_F}, and \eqref{Pf_KG:Est_DF} we deduce that
  \begin{align} \label{Pf_KG:Est_Dphi}
    \left|\frac{\partial\varphi}{\partial r}(y,r)\right| \leq c\Bigl(|u^\sharp(y,r)|^2+\Bigl[|u^\sharp||(\nabla u)^\sharp|\Bigr](y,r)\Bigr)
  \end{align}
  for all $y\in\Gamma$ and $r\in[\varepsilon g_0(y),\varepsilon g_1(y)]$, where $c>0$ is a constant independent of $\varepsilon$.
  Applying \eqref{Pf_KG:Bound_FTC} and \eqref{Pf_KG:Est_Dphi} to \eqref{Pf_KG:Bound_Int} and using \eqref{E:CoV_Equiv} and Hold\"{o}r's inequality we get
  \begin{align} \label{Pf_KG:Est_Bint}
    \begin{aligned}
      \left|\int_{\Gamma_\varepsilon}(u\cdot\nabla)u\cdot n_\varepsilon\,d\mathcal{H}^2\right| &\leq c\int_\Gamma\int_{\varepsilon g_0}^{\varepsilon g_1}\Bigl(|u^\sharp|^2+|u^\sharp||(\nabla u)^\sharp|\Bigr)\,dr\,d\mathcal{H}^2 \\
      &\leq c\left(\|u\|_{L^2(\Omega_\varepsilon)}^2+\|u\|_{L^2(\Omega_\varepsilon)}\|\nabla u\|_{L^2(\Omega_\varepsilon)}\right).
    \end{aligned}
  \end{align}
  By \eqref{Pf_KG:First_Est}, \eqref{Pf_KG:Est_Bint}, and Young's inequality we obtain
  \begin{align*}
    \|\nabla u\|_{L^2(\Omega_\varepsilon)}^2 &\leq 2\|D(u)\|_{L^2(\Omega_\varepsilon)}^2+c\left(\|u\|_{L^2(\Omega_\varepsilon)}^2+\|u\|_{L^2(\Omega_\varepsilon)}\|\nabla u\|_{L^2(\Omega_\varepsilon)}\right) \\
    &\leq 2\|D(u)\|_{L^2(\Omega_\varepsilon)}^2+c\|u\|_{L^2(\Omega_\varepsilon)}^2+\frac{1}{2}\|\nabla u\|_{L^2(\Omega_\varepsilon)}^2.
  \end{align*}
  Hence \eqref{E:Korn_Grad} follows.
\end{proof}

Next we show a uniform $L^2$-estimate for a vector field on $\Omega_\varepsilon$ by the $L^2$-norms of the gradient matrix and the strain rate tensor on $\Omega_\varepsilon$.
Recall that for a function $\eta$ on $\Gamma$ we denote by $\bar{\eta}=\eta\circ\pi$ its constant extension in the normal direction of $\Gamma$.

\begin{lemma} \label{L:Korn_U}
  For given $\alpha>0$ and $\beta\in[0,1)$ there exist constants
  \begin{align*}
    \varepsilon_K = \varepsilon_K(\alpha,\beta)\in(0,1), \quad c_{K,2} = c_{K,2}(\alpha,\beta) > 0
  \end{align*}
  independent of $\varepsilon$ such that
  \begin{align} \label{E:Korn_U}
    \|u\|_{L^2(\Omega_\varepsilon)}^2 \leq \alpha\|\nabla u\|_{L^2(\Omega_\varepsilon)}^2+c_{K,2}\|D(u)\|_{L^2(\Omega_\varepsilon)}^2
  \end{align}
  for all $\varepsilon\in(0,\varepsilon_K)$ and $u\in H^1(\Omega_\varepsilon)^3$ satisfying \eqref{E:Bo_Imp} on $\Gamma_\varepsilon$ and
  \begin{align} \label{E:Korn_Orth}
    \left|(u,\bar{v})_{L^2(\Omega_\varepsilon)}\right| \leq \beta\|u\|_{L^2(\Omega_\varepsilon)}\|\bar{v}\|_{L^2(\Omega_\varepsilon)} \quad\text{for all}\quad v\in\mathcal{K}_g(\Gamma).
  \end{align}
  Here $\mathcal{K}_g(\Gamma)$ is the function space on $\Gamma$ given by \eqref{E:Def_Kil}.
\end{lemma}

To prove Lemma~\ref{L:Korn_U} we transform integrals over $\Omega_\varepsilon$ into those over the domain $\Omega_1$ with fixed width (note that we assume $\overline{\Omega}_1\subset N$ by scaling $g_0$ and $g_1$).
Define a mapping $\Phi_\varepsilon\colon\Omega_1\to\Omega_\varepsilon$ by
\begin{align} \label{E:Def_Korn_Aux}
  \Phi_\varepsilon(X) := \pi(X)+\varepsilon d(X)\bar{n}(X), \quad X\in\Omega_1.
\end{align}
We easily see that $\Phi_\varepsilon$ is bijective and its inverse $\Phi_\varepsilon^{-1}\colon\Omega_\varepsilon\to\Omega_1$ is given by
\begin{align*}
  \Phi_\varepsilon^{-1}(x) := \pi(x)+\varepsilon^{-1}d(x)\bar{n}(x), \quad x\in\Omega_\varepsilon.
\end{align*}
Also, by \eqref{E:Form_W}, \eqref{E:Pi_Der}, and \eqref{E:Nor_Grad} we have
\begin{align*}
  \nabla\Phi_\varepsilon = \Bigl(I_3-d\overline{W}\Bigr)^{-1}\Bigl(I_3-\varepsilon d\overline{W}\Bigr)\overline{P}+\varepsilon\overline{Q} \quad\text{on}\quad \Omega_1.
\end{align*}
Hence taking an orthonormal basis of $\mathbb{R}^3$ that consists of $n$ and the eigenvectors of $W$ corresponding to the eigenvalues $\kappa_1$ and $\kappa_2$ we get
\begin{align*}
  \det\nabla\Phi_\varepsilon(X) = \varepsilon J(\pi(X),d(X))^{-1}J(\pi(X),\varepsilon d(X)), \quad X\in\Omega_1
\end{align*}
and the change of variables formula
\begin{align} \label{E:CoV_Omega_1}
  \int_{\Omega_\varepsilon}\varphi(x)\,dx = \varepsilon\int_{\Omega_1}\varphi(\Phi_\varepsilon(X))J(\pi(X),d(X))^{-1}J(\pi(X),\varepsilon d(X))\,dX
\end{align}
for a function $\varphi$ on $\Omega_\varepsilon$.
In particular, by \eqref{E:Jac_Bound} we have
\begin{align} \label{E:L2_Omega_1}
  c\varepsilon^{-1} \|\varphi\|_{L^2(\Omega_\varepsilon)}^2 \leq \|\varphi\circ\Phi_\varepsilon\|_{L^2(\Omega_1)}^2 \leq c'\varepsilon^{-1}\|\varphi\|_{L^2(\Omega_\varepsilon)}^2
\end{align}
for $\varphi\in L^2(\Omega_\varepsilon)$, where $c$ and $c'$ are positive constants independent of $\varepsilon$.
By \eqref{E:L2_Omega_1} and direct calculations of matrices we can show the following auxiliary inequalities for the proof of Lemma~\ref{L:Korn_U} (see Appendix~\ref{S:Ap_VM} for the proof).

\begin{lemma} \label{L:Korn_Aux}
  For $u\in H^1(\Omega_\varepsilon)^3$ let $U:=u\circ\Phi_\varepsilon$ on $\Omega_1$.
  Then $U\in H^1(\Omega_1)^3$ and
  \begin{align} \label{E:KAux_Grad}
    \varepsilon^{-1}\|\nabla u\|_{L^2(\Omega_\varepsilon)}^2 \geq c\left(\left\|\overline{P}\nabla U\right\|_{L^2(\Omega_1)}^2+\varepsilon^{-2}\|\partial_nU\|_{L^2(\Omega_1)}^2\right)
  \end{align}
  with a constant $c>0$ independent of $\varepsilon$ and $u$, where $\partial_nU=(\bar{n}\cdot\nabla)U$ is the normal derivative (with respect to $\Gamma$) of $U$.
  We also have
  \begin{align} \label{E:KAux_Du}
    \varepsilon^{-1}\|D(u)\|_{L^2(\Omega_\varepsilon)}^2 \geq c\left(\left\|\overline{P}F_\varepsilon(U)_S\overline{P}\right\|_{L^2(\Omega_1)}^2+\varepsilon^{-2}\|\partial_n(U\cdot\bar{n})\|_{L^2(\Omega_1)}^2\right),
  \end{align}
  where $F_\varepsilon(U)_S=\{F_\varepsilon(U)+F_\varepsilon(U)^T\}/2$ is the symmetric part of the matrix
  \begin{align} \label{E:KAux_Matrix}
    F_\varepsilon(U) := \Bigl(I_3-\varepsilon d\overline{W}\Bigr)^{-1}\Bigl(I_3-d\overline{W}\Bigr)\nabla U \quad\text{on}\quad \Omega_1.
  \end{align}
\end{lemma}

\begin{proof}[Proof of Lemma~\ref{L:Korn_U}]
  Following the idea of the proof of \cite[Lemma~4.14]{HoSe10} we prove \eqref{E:Korn_U} by contradiction.
  Assume to the contrary that there exist a sequence $\{\varepsilon_k\}_{k=1}^\infty$ of positive numbers with $\lim_{k\to\infty}\varepsilon_k=0$ and vector fields $u_k\in H^1(\Omega_{\varepsilon_k})^3$ satisfying \eqref{E:Bo_Imp} on $\Gamma_{\varepsilon_k}$, \eqref{E:Korn_Orth}, and
  \begin{align} \label{PF_KU:Ineq_Contra}
    \|u_k\|_{L^2(\Omega_{\varepsilon_k})}^2 > \alpha\|\nabla u_k\|_{L^2(\Omega_{\varepsilon_k})}^2+k\|D(u_k)\|_{L^2(\Omega_{\varepsilon_k})}^2, \quad k\in\mathbb{N}.
  \end{align}
  For each $k\in\mathbb{N}$ let $U_k:=u_k\circ\Phi_{\varepsilon_k}\in H^1(\Omega_1)^3$ and $F_{\varepsilon_k}(U_k)$ be the matrix given by \eqref{E:KAux_Matrix}.
  Dividing both sides of \eqref{PF_KU:Ineq_Contra} by $\varepsilon_k$ and using \eqref{E:L2_Omega_1}--\eqref{E:KAux_Du} we get
  \begin{multline*}
    \|U_k\|_{L^2(\Omega_1)}^2 > c\alpha\left(\left\|\overline{P}\nabla U_k\right\|_{L^2(\Omega_1)}^2+\varepsilon_k^{-2}\|\partial_nU_k\|_{L^2(\Omega_1)}^2\right) \\
    +ck\left(\left\|\overline{P}F_{\varepsilon_k}(U_k)_S\overline{P}\right\|_{L^2(\Omega_1)}^2+\varepsilon_k^{-2}\|\partial_n(U_k\cdot\bar{n})\|_{L^2(\Omega_1)}^2\right).
  \end{multline*}
  Since $\|U_k\|_{L^2(\Omega_1)}>0$, we may assume
  \begin{align} \label{Pf_KU:L2_Uk}
    \|U_k\|_{L^2(\Omega_1)} = 1, \quad k\in\mathbb{N}
  \end{align}
  by replacing $U_k$ with $U_k/\|U_k\|_{L^2(\Omega_1)}$.
  Then
  \begin{align}
    \left\|\overline{P}\nabla U_k\right\|_{L^2(\Omega_1)}^2+\varepsilon_k^{-2}\|\partial_nU_k\|_{L^2(\Omega_1)}^2 &< c\alpha^{-1}, \label{Pf_KU:L2_Grad_Uk} \\
    \left\|\overline{P}F_{\varepsilon_k}(U_k)_S\overline{P}\right\|_{L^2(\Omega_1)}^2+\varepsilon_k^{-2}\|\partial_n(U_k\cdot\bar{n})\|_{L^2(\Omega_1)}^2 &< ck^{-1} \label{Pf_KU:L2_DUk}
  \end{align}
  and $\{U_k\}_{k=1}^\infty$ is bounded in $H^1(\Omega_1)^3$ by \eqref{Pf_KU:L2_Uk}, \eqref{Pf_KU:L2_Grad_Uk}, and
  \begin{align*}
    |\nabla U_k|^2 = \left|\overline{P}\nabla U_k\right|^2+\left|\overline{Q}\nabla U_k\right|^2, \quad \left|\overline{Q}\nabla U_k\right| = |\bar{n}\otimes\partial_nU_k| = |\partial_nU_k| \quad\text{in}\quad \Omega_1.
  \end{align*}
  Hence there exists $U\in H^1(\Omega_1)^3$ such that (up to a subsequence) $\{U_k\}_{k=1}^\infty$ converges to $U$ weakly in $H^1(\Omega_1)^3$.
  By the compact embedding $H^1(\Omega_1)\hookrightarrow L^2(\Omega_1)$ and \eqref{Pf_KU:L2_Uk} it also converges to $U$ strongly in $L^2(\Omega_1)^3$ and
  \begin{align} \label{Pf_KU:L2_Limit}
    \|U\|_{L^2(\Omega_1)} = \lim_{k\to\infty}\|U_k\|_{L^2(\Omega_1)} = 1.
  \end{align}
  Our goal is to show $U=0$ in $\Omega_1$, which contradicts with \eqref{Pf_KU:L2_Limit}.
  Since
  \begin{align} \label{Pf_KU:Lim_NDer}
    \lim_{k\to\infty}\|\partial_nU_k\|_{L^2(\Omega_1)} = 0
  \end{align}
  by \eqref{Pf_KU:L2_Grad_Uk} and $\{U_k\}_{k=1}^\infty$ converges to $U$ weakly in $H^1(\Omega_1)^3$, it follows that $\partial_nU=0$ in $\Omega_1$, i.e. $U$ is independent of the normal direction of $\Gamma$.
  Hence setting
  \begin{align*}
    v(y) := U(y+g_0(y)n(y)), \quad y\in\Gamma
  \end{align*}
  we can consider $U$ as the constant extension of $v$, i.e. $U=\bar{v}$ in $\Omega_1$.
  Moreover, from \eqref{E:Poin_Bo} with $\varepsilon=1$ and $\partial_n\bar{v}=0$ in $\Omega_1$ we deduce that
  \begin{align*}
    \|U_k-\bar{v}\|_{L^2(\Gamma_1)} \leq c\left(\|U_k-\bar{v}\|_{L^2(\Omega_1)}+\|\partial_nU_k\|_{L^2(\Omega_1)}\right), \quad k\in\mathbb{N}.
  \end{align*}
  Thus, by the strong convergence of $\{U_k\}_{k=1}^\infty$ to $\bar{v}=U$ in $L^2(\Omega_1)^3$ and \eqref{Pf_KU:Lim_NDer},
  \begin{align} \label{Pf_KU:Lim_Bo}
    \lim_{k\to\infty}\|U_k-\bar{v}\|_{L^2(\Gamma_1)} = 0.
  \end{align}
  Let us show $v\in\mathcal{K}_g(\Gamma)$.
  By $\bar{v}=U\in H^1(\Omega_\varepsilon)^3$ and Lemma~\ref{L:Con_Lp_W1p} we have $v\in H^1(\Gamma)^3$.
  For each $k\in\mathbb{N}$, since $u_k$ satisfies \eqref{E:Bo_Imp} on $\Gamma_{\varepsilon_k}$ we can use \eqref{E:Exp_Bo} to get
  \begin{align*}
    |u_k\cdot\bar{n}|\leq c\varepsilon_k|u_k| \quad\text{on}\quad \Gamma_{\varepsilon_k}, \quad\text{i.e.}\quad |U_k\cdot\bar{n}|\leq c\varepsilon_k|U_k| \quad\text{on}\quad \Gamma_1.
  \end{align*}
  By this inequality, \eqref{E:Poin_Bo} with $\varepsilon=1$, and the boundedness of $\{U_k\}_{k=1}^\infty$ in $H^1(\Omega_1)^3$,
  \begin{align} \label{Pf_KU:Lim_NC}
    \|U_k\cdot\bar{n}\|_{L^2(\Gamma_1)} \leq c\varepsilon_k\|U_k\|_{L^2(\Gamma_1)} \leq c\varepsilon_k\|U_k\|_{H^1(\Omega_1)} \to 0 \quad\text{as}\quad k\to\infty.
  \end{align}
  Combining this with \eqref{Pf_KU:Lim_Bo} we get $\bar{v}\cdot\bar{n}=0$ on $\Gamma_1$ and thus $v\cdot n=0$ on $\Gamma$.
  Next we show that $D_\Gamma(v)=P(\nabla_\Gamma v)_SP$ vanishes on $\Gamma$.
  Since $F_{\varepsilon_k}(U_k)$ is given by \eqref{E:KAux_Matrix}, $\{U_k\}_{k=1}^\infty$ converges to $U=\bar{v}$ weakly in $H^1(\Omega_1)^3$, and $(I_3-\varepsilon_kd\overline{W})^{-1}$ converges to $I_3$ uniformly on $\Omega_1$ as $k\to\infty$ by \eqref{E:Wein_Diff} and $|d|\leq c$ in $\Omega_1$,
  \begin{align*}
    \lim_{k\to\infty}F_{\varepsilon_k}(U_k) = \Bigl(I_3-d\overline{W}\Bigr)\nabla\bar{v} = \overline{\nabla_\Gamma v} \quad\text{weakly in}\quad L^2(\Omega_1)^{3\times3}.
  \end{align*}
  Here the last equality follows from \eqref{E:ConDer_Dom}.
  Therefore,
  \begin{align*}
    \lim_{k\to\infty}\overline{P}F_{\varepsilon_k}(U_k)_S\overline{P} = \overline{P}\Bigl(\overline{\nabla_\Gamma v}\Bigr)_S\overline{P} = \overline{D_\Gamma(v)} \quad\text{weakly in}\quad L^2(\Omega_1)^{3\times3}.
  \end{align*}
  Moreover, the inequality \eqref{Pf_KU:L2_DUk} yields
  \begin{align*}
    \lim_{k\to\infty}\left\|\overline{P}F_{\varepsilon_k}(U_k)_S\overline{P}\right\|_{L^2(\Omega_1)} = 0.
  \end{align*}
  Thus $\overline{D_\Gamma(v)}=0$ in $\Omega_1$, i.e. $D_\Gamma(v)=0$ on $\Gamma$.
  To prove $v\in\mathcal{K}_g(\Gamma)$ it remains to verify $v\cdot\nabla_\Gamma g=0$ on $\Gamma$.
  In what follows, we use the notations \eqref{E:Pull_Dom} and \eqref{E:Pull_Bo} (with $\varepsilon=1$).
  For each $k\in\mathbb{N}$, since $u_k$ satisfies \eqref{E:Bo_Imp} on $\Gamma_{\varepsilon_k}$,
  \begin{align*}
    u_k\cdot\bar{\tau}_{\varepsilon_k}^i = \varepsilon_k^{-1}u_k\cdot\bar{n} \quad\text{on}\quad \Gamma_{\varepsilon_k}^i,\,i=0,1
  \end{align*}
  by \eqref{E:Exp_Bo}.
  This equality yields $U_k\cdot\bar{\tau}_{\varepsilon_k}^i=\varepsilon_k^{-1}U_k\cdot\bar{n}$ on $\Gamma_1^i$, $i=0,1$, or equivalently,
  \begin{align*}
    U_{k,i}^\sharp(y)\cdot\tau_{\varepsilon_k}^i(y) = \varepsilon_k^{-1}U_{k,i}^\sharp(y)\cdot n(y), \quad y\in\Gamma,\,i=0,1.
  \end{align*}
  Hence
  \begin{align} \label{Pf_KU:L2_Diff}
    \|U_{k,1}^\sharp\cdot\tau_{\varepsilon_k}^1-U_{k,0}^\sharp\cdot\tau_{\varepsilon_k}^0\|_{L^2(\Gamma)} \leq \varepsilon_k^{-1}\|U_{k,1}^\sharp\cdot n-U_{k,0}^\sharp\cdot n\|_{L^2(\Gamma)}.
  \end{align}
  Moreover, since $\bar{n}^\sharp(y,r)=n(y)$ for $y\in\Gamma$ and $r\in(g_0(y),g_1(y))$,
  \begin{align*}
    (U_{k,1}^\sharp\cdot n)(y)-(U_{k,0}^\sharp\cdot n)(y) &= \int_{g_0(y)}^{g_1(y)}\frac{\partial}{\partial r}\Bigl((U_k\cdot\bar{n})^\sharp(y,r)\Bigr)\,dr \\
    &= \int_{g_0(y)}^{g_1(y)}[\partial_n(U_k\cdot\bar{n})]^\sharp(y,r)\,dr.
  \end{align*}
  By this equality, H\"{o}lder's inequality, \eqref{E:CoV_Equiv}, and \eqref{Pf_KU:L2_DUk},
  \begin{align*}
    \|U_{k,1}^\sharp\cdot n-U_{k,0}^\sharp\cdot n\|_{L^2(\Gamma)}^2 &= \int_\Gamma\left(\int_{g_0(y)}^{g_1(y)}[\partial_n(U_k\cdot\bar{n})]^\sharp(y,r)\,dr\right)^2d\mathcal{H}^2(y) \\
    &\leq c\|\partial_n(U_k\cdot\bar{n})\|_{L^2(\Omega_1)}^2 \leq c\varepsilon_k^2k^{-1}.
  \end{align*}
  Applying this inequality to the right-hand side of \eqref{Pf_KU:L2_Diff} we get
  \begin{align} \label{Pf_KU:Lim_Diff}
    \|U_{k,1}^\sharp\cdot\tau_{\varepsilon_k}^1-U_{k,0}^\sharp\cdot\tau_{\varepsilon_k}^0\|_{L^2(\Gamma)} \leq ck^{-1/2} \to 0 \quad\text{as}\quad k\to\infty.
  \end{align}
  Also, by \eqref{E:Tau_Bound}, \eqref{E:Tau_Diff}, and \eqref{E:Lp_CoV_Surf},
  \begin{align*}
    \|U_{k,i}^\sharp\cdot\tau_{\varepsilon_k}^i-v\cdot\nabla_\Gamma g_i\|_{L^2(\Gamma)} &\leq \|(U_{k,i}^\sharp-v)\cdot\tau_{\varepsilon_k}^i\|_{L^2(\Gamma)}+\|v\cdot(\tau_{\varepsilon_k}^i-\nabla_\Gamma g_i)\|_{L^2(\Gamma)} \\
    &\leq c\left(\|U_{k,i}^\sharp-v\|_{L^2(\Gamma)}+\varepsilon_k\|v\|_{L^2(\Gamma)}\right) \\
    &\leq c\left(\|U_k-\bar{v}\|_{L^2(\Gamma_1^i)}+\varepsilon_k\|v\|_{L^2(\Gamma)}\right).
  \end{align*}
  Since the right-hand side tends to zero as $k\to\infty$ by \eqref{Pf_KU:Lim_Bo},
  \begin{align*}
    \lim_{k\to\infty}\|U_{k,i}^\sharp\cdot\tau_{\varepsilon_k}^i-v\cdot\nabla_\Gamma g_i\|_{L^2(\Gamma)} = 0, \quad i=0,1.
  \end{align*}
  This equality and \eqref{Pf_KU:Lim_Diff} imply
  \begin{align*}
     \|v\cdot\nabla_\Gamma g\|_{L^2(\Gamma)} = \|v\cdot\nabla_\Gamma g_1-v\cdot\nabla_\Gamma g_0\|_{L^2(\Gamma)} = 0.
   \end{align*}
  Hence $v\cdot\nabla_\Gamma g=0$ on $\Gamma$ and we obtain $v\in\mathcal{K}_g(\Gamma)$.

  Now we recall that $u_k\in H^1(\Omega_{\varepsilon_k})^3$ satisfies \eqref{E:Korn_Orth}.
  Then since $v\in\mathcal{K}_g(\Gamma)$,
  \begin{align} \label{Pf_KU:Fin_Ineq}
    |(u_k,\bar{v})_{L^2(\Omega_{\varepsilon_k})}| \leq \beta\|u_k\|_{L^2(\Omega_{\varepsilon_k})}\|\bar{v}\|_{L^2(\Omega_{\varepsilon_k})}, \quad k\in\mathbb{N}
  \end{align}
  with $\beta\in[0,1)$.
  We express this inequality in terms of $U_k$ and send $k\to\infty$.
  Let
  \begin{align} \label{Pf_Ku:Phik}
    \varphi_k(X):=J(\pi(X),d(X))^{-1}J(\pi(X),\varepsilon_kd(X)), \quad X\in \Omega_1.
  \end{align}
  Then by \eqref{E:CoV_Omega_1} and $U_k=u_k\circ\Phi_{\varepsilon_k}$ on $\Omega_1$ we have
  \begin{gather*}
    (u_k,\bar{v})_{L^2(\Omega_{\varepsilon_k})} = \varepsilon_k\int_{\Omega_1}U_k\cdot(\bar{v}\circ\Phi_{\varepsilon_k})\varphi_k\,dX.
  \end{gather*}
  Here $\bar{v}\circ\Phi_{\varepsilon_k}=\bar{v}$ in $\Omega_1$ since $\pi\circ\Phi_{\varepsilon_k}=\pi$ in $\Omega_1$ by \eqref{E:Def_Korn_Aux}.
  Moreover, $\varphi_k$ converges to $J(\pi(\cdot),d(\cdot))^{-1}$ uniformly on $\Omega_1$ as $k\to\infty$ by \eqref{E:Jac_Bound} and \eqref{E:Jac_Diff}.
  From these facts and the strong convergence of $\{U_k\}_{k=1}^\infty$ to $U=\bar{v}$ in $L^2(\Omega_1)^3$ we deduce that
  \begin{align*}
    \lim_{k\to\infty}\varepsilon_k^{-1}(u_k,\bar{v})_{L^2(\Omega_{\varepsilon_k})} = \lim_{k\to\infty}\int_{\Omega_1}(U_k\cdot\bar{v})\varphi_k\,dX = \int_{\Omega_1}|\bar{v}|^2J(\pi(\cdot),d(\cdot))^{-1}\,dX.
  \end{align*}
  By \eqref{E:CoV_Dom} the last term is of the form
  \begin{align*}
    \int_\Gamma\int_{g_0(y)}^{g_1(y)}|v(y)|^2J(y,r)^{-1}J(y,r)\,dr\,d\mathcal{H}^2(y) = \int_\Gamma g(y)|v(y)|^2\,d\mathcal{H}^2(y).
  \end{align*}
  Therefore,
  \begin{align} \label{PF_KU:Fin_Lim_1}
    \lim_{k\to\infty}\varepsilon_k^{-1}(u_k,\bar{v})_{L^2(\Omega_{\varepsilon_k})} = \|g^{1/2}v\|_{L^2(\Gamma)}^2.
  \end{align}
  By the same arguments we have
  \begin{align} \label{Pf_KU:Fin_Lim_2}
    \lim_{k\to\infty}\varepsilon_k^{-1}\|u_k\|_{L^2(\Omega_{\varepsilon_k})}^2 = \lim_{k\to\infty}\varepsilon_k^{-1}\|\bar{v}\|_{L^2(\Omega_{\varepsilon_k})}^2 = \|g^{1/2}v\|_{L^2(\Gamma)}^2.
  \end{align}
  We divide both sides of \eqref{Pf_KU:Fin_Ineq} by $\varepsilon_k$, send $k\to\infty$, and use \eqref{PF_KU:Fin_Lim_1}--\eqref{Pf_KU:Fin_Lim_2} to get
  \begin{align*}
    \|g^{1/2}v\|_{L^2(\Gamma)}^2 \leq \beta\|g^{1/2}v\|_{L^2(\Gamma)}^2.
  \end{align*}
  By this inequality, $\beta<1$, and \eqref{E:Width_Bound} we obtain $v=0$ on $\Gamma$ and thus $U=\bar{v}=0$ in $\Omega_1$, which contradicts with \eqref{Pf_KU:L2_Limit}.
  Hence \eqref{E:Korn_U} is valid.
\end{proof}

Combining Lemmas~\ref{L:Korn_Grad} and~\ref{L:Korn_U} we obtain a uniform $H^1$-estimate for a vector field on $\Omega_\varepsilon$ by the $L^2$-norm of the strain rate tensor on $\Omega_\varepsilon$.

\begin{lemma} \label{L:Korn_H1}
  For $\beta\in[0,1)$ there exist $\varepsilon_{K,\beta}\in(0,1)$ and $c_{K,\beta}>0$ such that
  \begin{align} \label{E:Korn_H1}
    \|u\|_{H^1(\Omega_\varepsilon)}^2 \leq c_{K,\beta}\|D(u)\|_{L^2(\Omega_\varepsilon)}^2
  \end{align}
  for all $\varepsilon\in(0,\varepsilon_{K,\beta})$ and $u\in H^1(\Omega_\varepsilon)^3$ satisfying \eqref{E:Bo_Imp} on $\Gamma_\varepsilon$ and \eqref{E:Korn_Orth}.
\end{lemma}

\begin{proof}
  Let $c_{K,1}>0$ be the constant given in Lemma~\ref{L:Korn_Grad} and $\varepsilon_K\in(0,1)$ and $c_{K,2}>0$ the constants given in Lemma~\ref{L:Korn_U} with $\alpha:=1/2c_{K,1}$.
  For $\varepsilon\in(0,\varepsilon_K)$ let $u\in H^1(\Omega_\varepsilon)^3$ satisfy \eqref{E:Bo_Imp} on $\Gamma_\varepsilon$ and \eqref{E:Korn_Orth}.
  By \eqref{E:Korn_Grad} and \eqref{E:Korn_U} we have
  \begin{align*}
    \|\nabla u\|_{L^2(\Omega_\varepsilon)}^2 \leq (4+c_{K,1}c_{K,2})\|D(u)\|_{L^2(\Omega_\varepsilon)}^2+c_{K,1}\alpha\|\nabla u\|_{L^2(\Omega_\varepsilon)}^2.
  \end{align*}
  Since $\alpha=1/2c_{K,1}$, the above inequality implies that
  \begin{align} \label{Pf_KH:K_Grad}
    \|\nabla u\|_{L^2(\Omega_\varepsilon)}^2 \leq c_{\beta,1}\|D(u)\|_{L^2(\Omega_\varepsilon)}^2, \quad c_{\beta,1} := 2(4+c_{K,1}c_{K,2}).
  \end{align}
  From this inequality and \eqref{E:Korn_U} we further deduce that
  \begin{align} \label{Pf_KH:K_U}
    \|u\|_{L^2(\Omega_\varepsilon)}^2 \leq c_{\beta,2}\|D(u)\|_{L^2(\Omega_\varepsilon)}^2, \quad c_{\beta,2} := 2(2c_{K,1}^{-1}+c_{K,2}).
  \end{align}
  By \eqref{Pf_KH:K_Grad} and \eqref{Pf_KH:K_U} we get \eqref{E:Korn_H1} with $\varepsilon_{K,\beta}:=\varepsilon_K$ and $c_{K,\beta}:=c_{\beta,1}+c_{\beta,2}$.
\end{proof}

Let $\mathcal{R}_g$ be the set of all infinitesimal rigid displacements of $\mathbb{R}^3$ whose restrictions on $\Gamma$ are tangential and orthogonal to $\nabla_\Gamma g$ (see \eqref{E:Def_Rg}).
We show that the uniform Korn inequality holds with $\mathcal{K}_g(\Gamma)$ in the condition \eqref{E:Korn_Orth} replaced by $\mathcal{R}_g$ if $\mathcal{K}_g(\Gamma)$ agrees with $\mathcal{R}_g|_\Gamma:=\{w|_\Gamma \mid w\in\mathcal{R}_g\}$ (see also Remark~\ref{R:Killing}).

\begin{lemma} \label{L:Korn_Rg}
  Suppose that $\mathcal{R}_g|_\Gamma=\mathcal{K}_g(\Gamma)$.
  Then for $\beta\in[0,1)$ there exist $\varepsilon_{K,\beta}\in(0,1)$ and $c_{K,\beta}>0$ such that the inequality \eqref{E:Korn_H1} holds for all $\varepsilon\in(0,\varepsilon_{K,\beta})$ and $u\in H^1(\Omega_\varepsilon)^3$ satisfying \eqref{E:Bo_Imp} on $\Gamma_\varepsilon$ and
  \begin{align} \label{E:KoRg_Orth}
    \left|(u,w)_{L^2(\Omega_\varepsilon)}\right| \leq \beta\|u\|_{L^2(\Omega_\varepsilon)}\|w\|_{L^2(\Omega_\varepsilon)} \quad\text{for all}\quad w\in\mathcal{R}_g.
  \end{align}
\end{lemma}

Note that the vector field $w\in\mathcal{R}_g$ in \eqref{E:KoRg_Orth} has an explicit form $w(x)=a\times x+b$ for $x\in\mathbb{R}^3$, which is essential for the proof of Lemma~\ref{L:Korn_Rg}.

\begin{proof}
  The proof is the same as that of Lemma~\ref{L:Korn_H1} if we show that the statement of Lemma~\ref{L:Korn_U} is still valid under the condition \eqref{E:KoRg_Orth} instead of \eqref{E:Korn_Orth}.
  Assume to the contrary that there exist a sequence $\{\varepsilon_k\}_{k=1}^\infty$ of positive numbers convergent to zero and vector fields $u_k\in H^1(\Omega_{\varepsilon_k})^3$, $k\in\mathbb{N}$ satisfying \eqref{E:Bo_Imp} on $\Gamma_{\varepsilon_k}$, \eqref{PF_KU:Ineq_Contra}, and \eqref{E:KoRg_Orth}.
  Let $\Phi_{\varepsilon_k}$ be the bijection from $\Omega_1$ onto $\Omega_{\varepsilon_k}$ given by \eqref{E:Def_Korn_Aux} and $U_k:=u_k\circ\Phi_{\varepsilon_k}\in H^1(\Omega_1)^3$.
  Then, after replacing $U_k$ with $U_k/\|U_k\|_{L^2(\Omega_1)}$, we can show as in the proof of Lemma~\ref{L:Korn_U} that $\{U_k\}_{k=1}^\infty$ converges (up to a subsequence) strongly in $L^2(\Omega_1)^3$ to the constant extension $\bar{v}$ of some $v\in\mathcal{K}_g(\Gamma)$ and
  \begin{align} \label{Pf_KRg:Lim_L2}
    \|\bar{v}\|_{L^2(\Omega_1)} = \lim_{k\to\infty}\|U_k\|_{L^2(\Omega_1)} =1.
  \end{align}
  Now we can take $w\in\mathcal{R}_g$ such that $w|_\Gamma=v$ on $\Gamma$ by the assumption $\mathcal{R}_g|_\Gamma=\mathcal{K}_g(\Gamma)$.
  Then since $u_k$ satisfies \eqref{E:KoRg_Orth} and $w\in\mathcal{R}_g$,
  \begin{align} \label{Pf_KRg:Orth}
    \left|(u_k,w)_{L^2(\Omega_{\varepsilon_k})}\right| \leq \beta\|u_k\|_{L^2(\Omega_{\varepsilon_k})}\|w\|_{L^2(\Omega_{\varepsilon_k})}, \quad k\in\mathbb{N}.
  \end{align}
  We apply the change of variables formula \eqref{E:CoV_Omega_1} to get
  \begin{align*}
    (u_k,w)_{L^2(\Omega_{\varepsilon_k})} = \varepsilon_k\int_{\Omega_1}U_k\cdot(w\circ\Phi_{\varepsilon_k})\varphi_k\,dX.
  \end{align*}
  Here the function $\varphi_k$ given by \eqref{Pf_Ku:Phik} converges to $J(\pi(\cdot),d(\cdot))^{-1}$ uniformly on $\Omega_1$ as $k\to\infty$ by \eqref{E:Jac_Bound} and \eqref{E:Jac_Diff}.
  Also, since $w\in\mathcal{R}_g$ is of the form $w(x)=a\times x+b$ with $a,b\in\mathbb{R}^3$, the mapping $\Phi_{\varepsilon_k}$ is given by \eqref{E:Def_Korn_Aux}, and $|d|\leq c$ and $|\bar{n}|=1$ in $\Omega_1$,
  \begin{align*}
    |w\circ\Phi_{\varepsilon_k}-w\circ\pi| = \varepsilon_k|d(a\times\bar{n})| \leq c\varepsilon_k \to 0 \quad\text{as}\quad k\to\infty
  \end{align*}
  uniformly on $\Omega_1$.
  By these facts, the strong convergence of $\{U_k\}_{k=1}^\infty$ to $\bar{v}$ in $L^2(\Omega_1)^3$, \eqref{E:CoV_Dom} with $\varepsilon=1$, and $\pi(y+rn(y))=y$ for $y\in\Gamma$ and $r\in(g_0(y),g_1(y))$ we get
  \begin{align*}
    \lim_{k\to\infty}\varepsilon_k^{-1}(u_k,w)_{L^2(\Omega_{\varepsilon_k})} &= \int_{\Omega_1}\bar{v}\cdot(w\circ\pi)J(\pi(\cdot),d(\cdot))^{-1}\,dX \\
    &=\int_\Gamma g(y)v(y)\cdot w(y)\,d\mathcal{H}^2(y).
  \end{align*}
  Moreover, since $w|_\Gamma=v$ on $\Gamma$, the above equality reads
  \begin{align*}
    \lim_{k\to\infty}\varepsilon_k^{-1}(u_k,w)_{L^2(\Omega_{\varepsilon_k})} = \|g^{1/2}v\|_{L^2(\Gamma)}^2.
  \end{align*}
  In the same way we can show that
  \begin{align*}
    \lim_{k\to\infty}\varepsilon_k^{-1}\|u_k\|_{L^2(\Omega_{\varepsilon_k})}^2 = \lim_{k\to\infty}\varepsilon_k^{-1}\|w\|_{L^2(\Omega_{\varepsilon_k})}^2 = \|g^{1/2}v\|_{L^2(\Gamma)}^2.
  \end{align*}
  Thus, as in the last part of the proof of Lemma~\ref{L:Korn_U}, we can derive $v=0$ on $\Gamma$ by dividing both sides of \eqref{Pf_KRg:Orth} by $\varepsilon_k$, sending $k\to\infty$, and using the above equalities, $\beta<1$, and \eqref{E:Width_Bound}.
  This implies $\bar{v}=0$ on $\Omega_1$, which contradicts with \eqref{Pf_KRg:Lim_L2}.
  Hence the statement of Lemma~\ref{L:Korn_U} holds under the condition \eqref{E:KoRg_Orth} instead of \eqref{E:Korn_Orth}.
\end{proof}

\begin{remark} \label{R:Korn_H1}
  The result of Lemma~\ref{L:Korn_H1} was first proved in~\cite[Theorem~2.2]{LeMu11} under the condition \eqref{E:Korn_Orth}.
  Here we gave another proof of it.
  Also, Lemma~\ref{L:Korn_Rg} improves the result of~\cite[Theorem~2.3]{LeMu11} in the case $\mathcal{R}_g|_\Gamma=\mathcal{K}_g(\Gamma)$.
  This is due to the fact that the condition \eqref{E:KoRg_Orth} is imposed for the restrictions on $\Gamma$ of all $w\in\mathcal{R}_g$ in~\cite[Theorem~2.3]{LeMu11}, while we test the vector fields $w$ themselves in Lemma~\ref{L:Korn_Rg}.
\end{remark}

\subsection{Korn inequality on a closed surface} \label{SS:Korn_Surf}
We derive the Korn inequality on $\Gamma$ that is essential for the study of a singular limit problem for \eqref{E:NS_Eq}--\eqref{E:NS_In} in Section~\ref{S:SL}.

\begin{lemma} \label{L:Korn_STG}
  There exists a constant $c>0$ such that
  \begin{align} \label{E:Korn_STG}
    \|\nabla_\Gamma v\|_{L^2(\Gamma)}^2 \leq c\left(\|D_\Gamma(v)\|_{L^2(\Gamma)}^2+\|v\|_{L^2(\Gamma)}^2\right)
  \end{align}
  for all $v\in H^1(\Gamma,T\Gamma)$.
  Here $D_\Gamma(v)$ is the surface strain rate tensor defined as
  \begin{align} \label{E:Strain_Surf}
    D_\Gamma(v) := P(\nabla_\Gamma v)_SP, \quad (\nabla_\Gamma v)_S = \frac{\nabla_\Gamma v+(\nabla_\Gamma v)^T}{2}.
  \end{align}
\end{lemma}

\begin{proof}
  For $\varepsilon\in(0,\delta)$ let $N_\varepsilon:=\{x\in\mathbb{R}^3 \mid -\varepsilon<d(x)<\varepsilon\}$.
  Then $\overline{N}_\varepsilon\subset N$ and we can show as in the proof of Lemma~\ref{L:Korn_Grad} that
  \begin{align} \label{Pf_KST:Korn_TNB}
    \|\nabla u\|_{L^2(N_\varepsilon)}^2 \leq 4\|D(u)\|_{L^2(N_\varepsilon)}^2+c\|u\|_{L^2(N_\varepsilon)}^2
  \end{align}
  for all $u\in H^1(N_\varepsilon)^3$ satisfying $u\cdot n_{N_\varepsilon}=0$ on $\partial N_\varepsilon$, where $c>0$ is a constant independent of $\varepsilon$ and $n_{N_\varepsilon}$ is the unit outward normal to $\partial N_\varepsilon$.
  Let $\bar{v}=v\circ\pi$ be the constant extension of $v\in H^1(\Gamma,T\Gamma)$.
  Then $\bar{v}\in H^1(N_\varepsilon)^3$ by Lemma~\ref{L:Con_Lp_W1p}.
  Moreover, $\bar{v}\cdot n_{N_\varepsilon}=0$ on $\partial N_\varepsilon$ since $v\cdot n=0$ on $\Gamma$ and $n_{N_\varepsilon}(x)=\pm\bar{n}(x)$ for $x\in\partial N_\varepsilon$ with $d(x)=\pm\varepsilon$ (double-sign corresponds).
  Hence we can apply \eqref{Pf_KST:Korn_TNB} to $\bar{v}$ to get
  \begin{align} \label{Pf_KST:Korn_Con}
    \|\nabla\bar{v}\|_{L^2(N_\varepsilon)}^2 \leq 4\|D(\bar{v})\|_{L^2(N_\varepsilon)}^2+c\|\bar{v}\|_{L^2(N_\varepsilon)}^2.
  \end{align}
  Let us derive \eqref{E:Korn_STG} from \eqref{Pf_KST:Korn_Con}.
  By \eqref{E:Con_Lp} and \eqref{E:Con_Lp_Grad} with $\Omega_\varepsilon$ replaced by $N_\varepsilon$,
  \begin{align} \label{Pf_KST:Con_1}
    \|\nabla\bar{v}\|_{L^2(N_\varepsilon)}^2 \geq c\varepsilon\|\nabla_\Gamma v\|_{L^2(\Gamma)}^2, \quad \|\bar{v}\|_{L^2(N_\varepsilon)}^2 \leq c\varepsilon\|v\|_{L^2(\Gamma)}^2.
  \end{align}
  Next we deal with $D(\bar{v})$.
  By \eqref{E:ConDer_Diff} and $|d|\leq\varepsilon$ in $N_\varepsilon$ we have
  \begin{align*}
    \left|D(\bar{v})-\Bigl(\overline{\nabla_\Gamma v}\Bigr)_S\right| \leq \left|\nabla\bar{v}-\overline{\nabla_\Gamma v}\right| \leq c\varepsilon\left|\overline{\nabla_\Gamma v}\right| \quad\text{in}\quad N_\varepsilon.
  \end{align*}
  Also, since $v\cdot n=0$, $|n|=1$, and $W$ is bounded on $\Gamma$, we use \eqref{E:Grad_W} to get
  \begin{align*}
    |(\nabla_\Gamma v)_S| = |D_\Gamma(v)+[(Wv)\otimes n]_S| \leq |D_\Gamma(v)|+c|v| \quad\text{on}\quad\Gamma.
  \end{align*}
  From the above two inequalities we deduce that
  \begin{align*}
    |D(\bar{v})| \leq \left|\Bigl(\overline{\nabla_\Gamma v}\Bigr)_S\right|+\left|D(\bar{v})-\Bigl(\overline{\nabla_\Gamma v}\Bigr)_S\right| \leq \left|\overline{D_\Gamma(v)}\right|+c\left(|\bar{v}|+\varepsilon\left|\overline{\nabla_\Gamma v}\right|\right)
  \end{align*}
  in $N_\varepsilon$.
  This inequality combined with \eqref{E:Con_Lp} implies
  \begin{align} \label{Pf_KST:Con_2}
    \|D(\bar{v})\|_{L^2(N_\varepsilon)}^2 \leq c\varepsilon\left(\|D_\Gamma(v)\|_{L^2(\Gamma)}^2+\|v\|_{L^2(\Gamma)}^2+\varepsilon^2\|\nabla_\Gamma v\|_{L^2(\Gamma)}^2\right).
  \end{align}
  Now we apply \eqref{Pf_KST:Con_1} and \eqref{Pf_KST:Con_2} to \eqref{Pf_KST:Korn_Con} and then divide both sides by $\varepsilon$ to get
  \begin{align*}
    \|\nabla_\Gamma v\|_{L^2(\Gamma)}^2 \leq c\left(\|D_\Gamma(v)\|_{L^2(\Gamma)}^2+\|v\|_{L^2(\Gamma)}^2+\varepsilon^2\|\nabla_\Gamma v\|_{L^2(\Gamma)}^2\right)
  \end{align*}
  with a constant $c>0$ independent of $\varepsilon$.
  Hence we obtain \eqref{E:Korn_STG} by sending $\varepsilon\to0$ in this inequality.
\end{proof}

\section{Stokes operator with slip boundary conditions} \label{S:St}
In this section we study the Helmholtz--Leray projection on $\Omega_\varepsilon$ and the Stokes operator on $\Omega_\varepsilon$ associated with slip boundary conditions.
The results in Section~\ref{SS:St_HL} and the other subsections are used in the study of a singular limit problem and the proof of the global existence of a strong solution, respectively.

\subsection{Helmholtz--Leray projection on a thin domain} \label{SS:St_HL}
Let $L_\sigma^2(\Omega_\varepsilon)$ be the norm closure of $C_{c,\sigma}^\infty(\Omega_\varepsilon):=\{u\in C_c^\infty(\Omega_\varepsilon)^3\mid\text{$\mathrm{div}\,u=0$ in $\Omega_\varepsilon$}\}$ in $L^2(\Omega_\varepsilon)^3$.
It is well known (see e.g.~\cite{BoFa13,Ga11,Te79}) that $L_\sigma^2(\Omega_\varepsilon)$ is characterized by
\begin{align*}
  L_\sigma^2(\Omega_\varepsilon) = \{u \in L^2(\Omega_\varepsilon)^3 \mid \text{$\mathrm{div}\,u=0$ in $\Omega_\varepsilon$, $u\cdot n_\varepsilon=0$ on $\Gamma_\varepsilon$}\}
\end{align*}
and the Helmholtz--Leray decomposition $L^2(\Omega_\varepsilon)^3=L_\sigma^2(\Omega_\varepsilon)\oplus G^2(\Omega_\varepsilon)$ holds with
\begin{align*}
  G^2(\Omega_\varepsilon) = L_\sigma^2(\Omega_\varepsilon)^\perp = \{\nabla p\in L^2(\Omega_\varepsilon)^3 \mid p\in H^1(\Omega_\varepsilon)\}.
\end{align*}
By $\mathbb{L}_\varepsilon$ we denote the Helmholtz--Leray projection from $L^2(\Omega_\varepsilon)^3$ onto $L_\sigma^2(\Omega_\varepsilon)$.
Here we use the nonstandard notation $\mathbb{L}_\varepsilon$ to avoid confusion of the Helmholtz--Leray projection with the orthogonal projection $\mathbb{P}_\varepsilon$ onto the space $\mathcal{H}_\varepsilon$ defined by \eqref{E:Def_Heps}.
For $u\in L^2(\Omega_\varepsilon)^3$ its solenoidal part is given by $\mathbb{L}_\varepsilon u=u-\nabla\varphi$, where $\varphi\in H^1(\Omega_\varepsilon)$ is a weak solution to the Neumann problem of Poisson's equation
\begin{align*}
  \Delta\varphi = \mathrm{div}\,u \quad\text{in}\quad \Omega_\varepsilon, \quad \frac{\partial\varphi}{\partial n_\varepsilon} = u\cdot n_\varepsilon \quad\text{on}\quad \Gamma_\varepsilon.
\end{align*}
Moreover, by the elliptic regularity theorem (see~\cite{Ev10,GiTr01}) we have $\mathbb{L}_\varepsilon u\in H^1(\Omega_\varepsilon)^3$ when $u\in H^1(\Omega_\varepsilon)^3$.
Our aim is to give a uniform estimate for the $H^1(\Omega_\varepsilon)$-norm of the difference $u-\mathbb{L}_\varepsilon u$ for $u\in H^1(\Omega_\varepsilon)^3$ satisfying $u\cdot n_\varepsilon=0$ on $\Gamma_\varepsilon$.

\begin{lemma} \label{L:HP_Dom}
  There exist constants $\varepsilon_\sigma\in(0,1)$ and $c>0$ such that
  \begin{align} \label{E:HP_Dom}
    \|u-\mathbb{L}_\varepsilon u\|_{H^1(\Omega_\varepsilon)} \leq c\|\mathrm{div}\,u\|_{L^2(\Omega_\varepsilon)}
  \end{align}
  for all $\varepsilon\in(0,\varepsilon_\sigma)$ and $u\in H^1(\Omega_\varepsilon)^3$ satisfying $u\cdot n_\varepsilon=0$ on $\Gamma_\varepsilon$.
\end{lemma}

To prove Lemma~\ref{L:HP_Dom} we first derive the uniform Poincar\'{e} inequality on $\Omega_\varepsilon$.

\begin{lemma} \label{L:Uni_Poin_Dom}
  There exist constants $\varepsilon_\sigma\in(0,1)$ and $c>0$ such that
  \begin{align} \label{E:Uni_Poin_Dom}
    \|\varphi\|_{L^2(\Omega_\varepsilon)} \leq c\|\nabla\varphi\|_{L^2(\Omega_\varepsilon)}
  \end{align}
  for all $\varepsilon\in(0,\varepsilon_\sigma)$ and $\varphi\in H^1(\Omega_\varepsilon)$ satisfying $\int_{\Omega_\varepsilon}\varphi\,dx=0$.
\end{lemma}

\begin{proof}
  As in the proof of Lemma~\ref{L:Korn_Aux} we prove \eqref{E:Uni_Poin_Dom} by contradiction.
  Assume to the contrary that there exist a sequence $\{\varepsilon_k\}_{k=1}^\infty$ of positive numbers convergent to zero and $\varphi_k\in H^1(\Omega_{\varepsilon_k})$ such that
  \begin{align} \label{Pf_UPD:Contra}
    \|\varphi_k\|_{L^2(\Omega_{\varepsilon_k})}^2 > k\|\nabla\varphi_k\|_{L^2(\Omega_{\varepsilon_k})}^2, \quad \int_{\Omega_{\varepsilon_k}}\varphi_k\,dx = 0, \quad k\in\mathbb{N}.
  \end{align}
  For each $k\in\mathbb{N}$ let $\Phi_{\varepsilon_k}$ be the bijection from $\Omega_1$ onto $\Omega_{\varepsilon_k}$ given by \eqref{E:Def_Korn_Aux} and define $\xi_k:=\varphi_k\circ\Phi_{\varepsilon_k}\in H^1(\Omega_1)$.
  We divide both sides of the first inequality of \eqref{Pf_UPD:Contra} by $\varepsilon_k$ and apply \eqref{E:L2_Omega_1} and \eqref{E:KAux_Grad} to $u=(\varphi_k,0,0)$ and $U=(\xi_k,0,0)$ to get
  \begin{align*}
    \|\xi_k\|_{L^2(\Omega_1)}^2 > ck\left(\left\|\overline{P}\nabla\xi_k\right\|_{L^2(\Omega_1)}^2+\varepsilon_k^{-2}\|\partial_n\xi_k\|_{L^2(\Omega_1)}^2\right),
  \end{align*}
  where $\partial_n\xi_k$ is the normal derivative of $\xi_k$ given by \eqref{E:Def_NorDer}.
  Since $\|\xi_k\|_{L^2(\Omega_1)}>0$, we may assume that $\|\xi_k\|_{L^2(\Omega_1)}=1$ by replacing $\xi_k$ with $\xi_k/\|\xi_k\|_{L^2(\Omega_1)}$ and thus
  \begin{align} \label{Pf_UPD:Nabla_k}
    \left\|\overline{P}\nabla\xi_k\right\|_{L^2(\Omega_1)}^2 < ck^{-1}, \quad \|\partial_n\xi_k\|_{L^2(\Omega_1)}^2 < c\varepsilon_k^2k^{-1}.
  \end{align}
  By these inequalities, $\|\xi_k\|_{L^2(\Omega_1)}=1$, and
  \begin{align*}
    |\nabla\xi_k|^2 = \left|\overline{P}\nabla\xi_k\right|^2+\left|\overline{Q}\nabla\xi_k\right|^2, \quad \left|\overline{Q}\nabla\xi_k\right| = |\bar{n}\otimes\partial_n\xi_k| = |\partial_n\xi_k| \quad\text{in}\quad \Omega_1,
  \end{align*}
  the sequence $\{\xi_k\}_{k=1}^\infty$ is bounded in $H^1(\Omega_1)$.
  Hence $\{\xi_k\}_{k=1}^\infty$ converges (up to a subsequence) to some $\xi\in H^1(\Omega_1)$ weakly in $H^1(\Omega_1)$.
  By the compact embedding $H^1(\Omega_1)\hookrightarrow L^2(\Omega_1)$ it also converges to $\xi$ strongly in $L^2(\Omega_1)$ and
  \begin{align} \label{Pf_UPD:Limit_L2}
    \|\xi\|_{L^2(\Omega_1)} = \lim_{k\to\infty}\|\xi_k\|_{L^2(\Omega_1)} = 1.
  \end{align}
  Moreover, the weak convergence of $\{\xi_k\}_{k=1}^\infty$ to $\xi$ in $H^1(\Omega_1)$ and \eqref{Pf_UPD:Nabla_k} imply that
  \begin{align} \label{Pf_UPD:Xi_Grad}
    \overline{P}\nabla\xi = 0, \quad \partial_n\xi = 0 \quad\text{in}\quad \Omega_1.
  \end{align}
  Let $\eta(y):=\xi(y+g_0(y)n(y))$ for $y\in\Gamma$.
  Then $\xi=\bar{\eta}$ is the constant extension of $\eta$ by the second equality of \eqref{Pf_UPD:Xi_Grad} and thus $\eta\in H^1(\Gamma)$ by $\xi\in H^1(\Omega_1)$ and Lemma~\ref{L:Con_Lp_W1p}.
  Also, by \eqref{E:P_TGr}, \eqref{E:WReso_P}, \eqref{E:ConDer_Dom}, and the first equality of \eqref{Pf_UPD:Xi_Grad} we see that
  \begin{align*}
    0 = \overline{P}\nabla\xi = \overline{P}\Bigl(I_3-d\overline{W}\Bigr)^{-1}\overline{\nabla_\Gamma\eta} = \Bigl(I_3-d\overline{W}\Bigr)^{-1}\overline{\nabla_\Gamma\eta} \quad\text{in}\quad \Omega_1,
  \end{align*}
  which yields $\overline{\nabla_\Gamma\eta}=0$ in $\Omega_1$, i.e. $\nabla_\Gamma\eta=0$ on $\Gamma$.
  Hence setting
  \begin{align*}
    \hat{\eta} := \eta-\frac{1}{|\Gamma|}\int_\Gamma\eta\,d\mathcal{H}^2 \in H^1(\Gamma) \quad\left(\int_\Gamma\hat\eta\,d\mathcal{H}^2=0\right),
  \end{align*}
  where $|\Gamma|$ is the area of $\Gamma$, and applying Poincar\'{e}'s inequality \eqref{E:Poin_Surf_Lp} to $\hat{\eta}$ we get
  \begin{align*}
    \|\hat{\eta}\|_{L^2(\Gamma)} \leq c\|\nabla_\Gamma\hat{\eta}\|_{L^2(\Gamma)} = c\|\nabla_\Gamma\eta\|_{L^2(\Gamma)} = 0.
  \end{align*}
  This shows that $\hat{\eta}=0$ on $\Gamma$, i.e. $\eta=|\Gamma|^{-1}\int_\Gamma\eta\,d\mathcal{H}^2$ is constant.
  Now we return to the second equality of \eqref{Pf_UPD:Contra}.
  By the change of variables formula \eqref{E:CoV_Omega_1} it reads
  \begin{align*}
    \int_{\Omega_1}\xi_k(X)J(\pi(X),d(X))^{-1}J(\pi(X),\varepsilon_kd(X))\,dX = 0.
  \end{align*}
  Let $k\to\infty$ in this equality.
  Then since $J(\pi(\cdot),\varepsilon_kd(\cdot))$ converges to one uniformly on $\Omega_1$ as $k\to\infty$ by \eqref{E:Jac_Diff} and $\{\xi_k\}_{k=1}^\infty$ converges to $\xi=\bar{\eta}$ strongly in $L^2(\Omega_1)$,
  \begin{align*}
    \int_{\Omega_1}\eta(\pi(X))J(\pi(X),d(X))^{-1}\,dX = 0.
  \end{align*}
  Noting that $\eta$ is constant on $\Gamma$, we apply \eqref{E:CoV_Dom} to this equality to get
  \begin{align*}
    \eta\int_\Gamma\int_{g_0(y)}^{g_1(y)} J(y,r)^{-1}J(y,r)\,dr\,d\mathcal{H}^2(y) = \eta\int_\Gamma g(y)\,d\mathcal{H}^2(y) = 0.
  \end{align*}
  This equality and \eqref{E:Width_Bound} imply $\eta=0$ and thus $\xi=\bar{\eta}=0$ in $\Omega_1$, which contradicts with \eqref{Pf_UPD:Limit_L2}.
  Therefore, the uniform inequality \eqref{E:Uni_Poin_Dom} is valid.
\end{proof}

Next we consider the Neumann problem of Poisson's equation
\begin{align} \label{E:Poisson_Dom}
  -\Delta\varphi = \xi \quad\text{in}\quad \Omega_\varepsilon, \quad \frac{\partial\varphi}{\partial n_\varepsilon} = 0 \quad\text{on}\quad \Gamma_\varepsilon
\end{align}
for $\xi\in H^{-1}(\Omega_\varepsilon)$ satisfying $\langle\xi,1\rangle_{\Omega_\varepsilon}=0$, where $\langle\cdot,\cdot\rangle_{\Omega_\varepsilon}$ denotes the duality product between $H^{-1}(\Omega_\varepsilon)$ and $H^1(\Omega_\varepsilon)$.
By the Lax--Milgram theory there exists a unique weak solution $\varphi\in H^1(\Omega_\varepsilon)$ satisfying
\begin{align} \label{E:PD_Weak}
  (\nabla\varphi,\nabla\zeta)_{L^2(\Omega_\varepsilon)} = \langle\xi,\zeta\rangle_{\Omega_\varepsilon} \quad\text{for all}\quad \zeta\in H^1(\Omega_\varepsilon),\quad \int_{\Omega_\varepsilon}\varphi\,dx = 0.
\end{align}
Moreover, by the elliptic regularity theorem, if $\xi\in L^2(\Omega_\varepsilon)$ then $\varphi\in H^2(\Omega_\varepsilon)$ and
\begin{align} \label{E:PD_H2_Ep}
  \|\varphi\|_{H^2(\Omega_\varepsilon)} \leq c_\varepsilon\|\xi\|_{L^2(\Omega_\varepsilon)}
\end{align}
with a constant $c_\varepsilon>0$ depending on $\varepsilon$.
In this case, the equation \eqref{E:Poisson_Dom} is satisfied in the strong sense.
Let us show that $c_\varepsilon$ in \eqref{E:PD_H2_Ep} can be taken independently of $\varepsilon$.

\begin{lemma} \label{L:PD_H2}
  Let $\varepsilon_\sigma$ be the constant given in Lemma~\ref{L:Uni_Poin_Dom}.
  For $\varepsilon\in(0,\varepsilon_\sigma)$ suppose that $\xi\in L^2(\Omega_\varepsilon)$ satisfies $\int_{\Omega_\varepsilon}\xi\,dx=0$ and let $\varphi\in H^2(\Omega_\varepsilon)$ be a unique solution to \eqref{E:Poisson_Dom}.
  Then there exists a constant $c>0$ independent of $\varepsilon$, $\xi$, and $\varphi$ such that
  \begin{align} \label{E:PD_H2}
    \|\varphi\|_{H^2(\Omega_\varepsilon)} \leq c\|\xi\|_{L^2(\Omega_\varepsilon)}.
  \end{align}
\end{lemma}

\begin{proof}
  Setting $\zeta=\varphi$ in \eqref{E:PD_Weak} and using \eqref{E:Uni_Poin_Dom} we immediately get
  \begin{align*}
    \|\varphi\|_{H^1(\Omega_\varepsilon)} \leq c\|\nabla\varphi\|_{L^2(\Omega_\varepsilon)} \leq c\|\xi\|_{H^{-1}(\Omega_\varepsilon)} \leq c\|\xi\|_{L^2(\Omega_\varepsilon)}.
  \end{align*}
  Hence it is sufficient to show that
  \begin{align} \label{Pf_PDH:Goal}
    \|\nabla^2\varphi\|_{L^2(\Omega_\varepsilon)} \leq c\left(\|\Delta\varphi\|_{L^2(\Omega_\varepsilon)}+\|\varphi\|_{H^1(\Omega_\varepsilon)}\right) = c\left(\|\xi\|_{L^2(\Omega_\varepsilon)}+\|\varphi\|_{H^1(\Omega_\varepsilon)}\right)
  \end{align}
  with some constant $c>0$ independent of $\varepsilon$ (note that $-\Delta\varphi=\xi$ a.e. in $\Omega_\varepsilon$).

  Since $C_c^\infty(\Omega_\varepsilon)$ is dense in $L^2(\Omega_\varepsilon)$, we can take a sequence $\{\xi_k\}_{k=1}^\infty$ of functions in $C_c^\infty(\Omega_\varepsilon)$ that converges strongly to $\xi$ in $L^2(\Omega_\varepsilon)$.
  For each $k\in\mathbb{N}$ let $\varphi_k\in H^1(\Omega_\varepsilon)$ be a unique weak solution to \eqref{E:Poisson_Dom} with source term
  \begin{align*}
    \tilde{\xi}_k(x) := \xi_k(x)-\frac{1}{|\Omega_\varepsilon|}\int_{\Omega_\varepsilon}\xi_k(y)\,dy, \quad x\in\Omega_\varepsilon \quad (\langle\tilde{\xi}_k,1\rangle_{\Omega_\varepsilon}=(\tilde{\xi}_k,1)_{L^2(\Omega_\varepsilon)}=0).
  \end{align*}
  Here $|\Omega_\varepsilon|$ is the volume of $\Omega_\varepsilon$.
  Since $\tilde{\xi}_k\in C^\infty(\overline{\Omega}_\varepsilon)$ and $\Gamma_\varepsilon$ is of class $C^4$, the elliptic regularity theorem yields $\varphi_k\in H^3(\Omega_\varepsilon)$.
  Also,
  \begin{align*}
    \lim_{k\to\infty}\int_{\Omega_\varepsilon}\xi_k\,dx = \lim_{k\to\infty}(\xi_k,1)_{L^2(\Omega_\varepsilon)} = (\xi,1)_{L^2(\Omega_\varepsilon)} = \int_{\Omega_\varepsilon}\xi\,dx = 0
  \end{align*}
  by the strong convergence of $\{\xi_k\}_{k=1}^\infty$ to $\xi$ in $L^2(\Omega_\varepsilon)$ and thus
  \begin{align} \label{Pf_PDH:Xi_Conv}
    \|\xi-\tilde{\xi}_k\|_{L^2(\Omega_\varepsilon)} \leq \|\xi-\xi_k\|_{L^2(\Omega_\varepsilon)}+\frac{1}{|\Omega_\varepsilon|^{1/2}}\left|\int_{\Omega_\varepsilon}\xi_k\,dx\right| \to 0 \quad\text{as}\quad k\to\infty.
  \end{align}
  Since $\varphi-\varphi_k$ is a unique solution to \eqref{E:Poisson_Dom} for the source term $\xi-\tilde{\xi}_k$,
  \begin{align*}
    \|\varphi-\varphi_k\|_{H^2(\Omega_\varepsilon)} \leq c_\varepsilon\|\xi-\tilde{\xi}_k\|_{L^2(\Omega_\varepsilon)} \to 0 \quad\text{as}\quad k\to\infty
  \end{align*}
  by \eqref{E:PD_H2_Ep} and \eqref{Pf_PDH:Xi_Conv} (note that the constant $c_\varepsilon$ does not depend on $k$).
  Hence we can derive \eqref{Pf_PDH:Goal} by showing the same inequality for $\varphi_k$ and sending $k\to\infty$.

  From now on, we fix and suppress the subscript $k$.
  Hence we suppose that $\varphi$ is in $H^3(\Omega_\varepsilon)$ and satisfies \eqref{E:Poisson_Dom} in the strong sense.
  In particular, we have
  \begin{align} \label{Pf_PDH:Bo}
    \nabla\varphi\cdot n_\varepsilon = \frac{\partial\varphi}{\partial n_\varepsilon} = 0 \quad\text{on}\quad \Gamma_\varepsilon.
  \end{align}
  By the regularity of $\varphi$ we can carry out integration by parts twice to get
  \begin{align} \label{Pf_PDH:IbP}
    \begin{aligned}
      \|\nabla^2\varphi\|_{L^2(\Omega_\varepsilon)}^2 &= \|\Delta\varphi\|_{L^2(\Omega_\varepsilon)}^2+\int_{\Gamma_\varepsilon}\{(\nabla\varphi\cdot\nabla)\nabla\varphi\cdot n_\varepsilon-(\nabla\varphi\cdot n_\varepsilon)\Delta\varphi\}\,d\mathcal{H}^2 \\
      &= \|\Delta\varphi\|_{L^2(\Omega_\varepsilon)}^2+\int_{\Gamma_\varepsilon}(\nabla\varphi\cdot\nabla)\nabla\varphi\cdot n_\varepsilon\,d\mathcal{H}^2.
    \end{aligned}
  \end{align}
  Here we used \eqref{Pf_PDH:Bo} in the second equality.
  Moreover, based on \eqref{Pf_PDH:Bo} we can show as in the proof of Lemma~\ref{L:Korn_Grad} (see \eqref{Pf_KG:Est_Bint}) that
  \begin{align*}
    \left|\int_{\Gamma_\varepsilon}(\nabla\varphi\cdot\nabla)\nabla\varphi\cdot n_\varepsilon\,d\mathcal{H}^2\right| \leq c\left(\|\nabla\varphi\|_{L^2(\Omega_\varepsilon)}^2+\|\nabla\varphi\|_{L^2(\Omega_\varepsilon)}\|\nabla^2\varphi\|_{L^2(\Omega_\varepsilon)}\right).
  \end{align*}
  Applying this inequality to \eqref{Pf_PDH:IbP} and using Young's inequality we obtain
  \begin{align*}
    \|\nabla^2\varphi\|_{L^2(\Omega_\varepsilon)}^2 \leq \|\Delta\varphi\|_{L^2(\Omega_\varepsilon)}^2+c\|\nabla\varphi\|_{L^2(\Omega_\varepsilon)}^2+\frac{1}{2}\|\nabla^2\varphi\|_{L^2(\Omega_\varepsilon)},
  \end{align*}
  which yields \eqref{Pf_PDH:Goal}.
  Hence the inequality \eqref{E:PD_H2} follows.
\end{proof}

Now let us derive the uniform estimate \eqref{E:HP_Dom} for the difference $u-\mathbb{L}_\varepsilon u$.

\begin{proof}[Proof of Lemma~\ref{L:HP_Dom}]
  Let $\varepsilon_\sigma$ be the constant given in Lemma~\ref{L:Uni_Poin_Dom} and $\varepsilon\in(0,\varepsilon_\sigma)$.
  For $u\in H^1(\Omega_\varepsilon)^3$ let $\xi:=-\mathrm{div}\,u\in L^2(\Omega_\varepsilon)$.
  If $u\cdot n_\varepsilon=0$ on $\Gamma_\varepsilon$, then
  \begin{align*}
    \langle\xi,1\rangle_{\Omega_\varepsilon} = \int_{\Omega_\varepsilon}\xi\,dx = -\int_{\Gamma_\varepsilon}u\cdot n_\varepsilon\,d\mathcal{H}^2 = 0
  \end{align*}
  by the divergence theorem.
  Since $\mathbb{L}_\varepsilon u=u-\nabla\varphi$ and $\varphi\in H^1(\Omega_\varepsilon)$ is a unique weak solution to \eqref{E:Poisson_Dom} with $\xi=-\mathrm{div}\,u$, Lemma~\ref{L:PD_H2} yields
  \begin{align*}
    \|u-\mathbb{L}_\varepsilon u\|_{H^1(\Omega_\varepsilon)} = \|\nabla\varphi\|_{H^1(\Omega_\varepsilon)} \leq c\|\xi\|_{L^2(\Omega_\varepsilon)} = c\|\mathrm{div}\,u\|_{L^2(\Omega_\varepsilon)}
  \end{align*}
  with a constant $c>0$ independent of $\varepsilon$.
  Hence \eqref{E:HP_Dom} is valid.
\end{proof}

\subsection{Definition and basic properties of the Stokes operator} \label{SS:St_Def}
Let us define the Stokes operator on $\Omega_\varepsilon$ associated with slip boundary conditions and give its basic properties.
For $u_1\in H^2(\Omega_\varepsilon)^3$ and $u_2\in H^1(\Omega_\varepsilon)^3$ we have
\begin{multline} \label{E:IbP_St}
  \int_{\Omega_\varepsilon}\{\Delta u_1+\nabla(\mathrm{div}\,u_1)\}\cdot u_2\,dx \\
  = -2\int_{\Omega_\varepsilon}D(u_1):D(u_2)\,dx+2\int_{\Gamma_\varepsilon}[D(u_1)n_\varepsilon]\cdot u_2\,d\mathcal{H}^2
\end{multline}
by integration by parts.
If $u_1$ and $u_2$ satisfy $\mathrm{div}\,u_1=0$ in $\Omega_\varepsilon$ and
\begin{align} \label{E:IbP_StBo}
  u_1\cdot n_\varepsilon = 0, \quad 2\nu P_\varepsilon D(u_1)n_\varepsilon+\gamma_\varepsilon u_1 = 0, \quad u_2\cdot n_\varepsilon = 0 \quad\text{on}\quad \Gamma_\varepsilon,
\end{align}
then from the above identity it follows that
\begin{align*}
  \nu\int_{\Omega_\varepsilon}\Delta u_1\cdot u_2\,dx = -2\nu\int_{\Omega_\varepsilon}D(u_1):D(u_2)\,dx-\sum_{i=0,1}\gamma_\varepsilon^i\int_{\Gamma_\varepsilon^i}u_1\cdot u_2\,d\mathcal{H}^2.
\end{align*}
Based on this observation we define a bilinear form
\begin{align} \label{E:Def_Bili_Dom}
  a_\varepsilon(u_1,u_2) := 2\nu\bigl(D(u_1),D(u_2)\bigr)_{L^2(\Omega_\varepsilon)}+\sum_{i=0,1}\gamma_\varepsilon^i(u_1,u_2)_{L^2(\Gamma_\varepsilon^i)}
\end{align}
for $u_1,u_2\in H^1(\Omega_\varepsilon)^3$.
By definition, $a_\varepsilon$ is symmetric on $H^1(\Omega_\varepsilon)^3$.

\begin{lemma} \label{L:Bili_Core}
  Under Assumptions~\ref{Assump_1} and~\ref{Assump_2}, let $\mathcal{V}_\varepsilon$ be the subspace of $H^1(\Omega_\varepsilon)^3$ given by \eqref{E:Def_Heps}.
  There exist constants $\varepsilon_0\in(0,1)$ and $c>0$ such that
  \begin{align} \label{E:Bili_Core}
    c^{-1}\|u\|_{H^1(\Omega_\varepsilon)}^2 \leq a_\varepsilon(u,u) \leq c\|u\|_{H^1(\Omega_\varepsilon)}^2
  \end{align}
  for all $\varepsilon\in(0,\varepsilon_0)$ and $u\in\mathcal{V}_\varepsilon$.
\end{lemma}

\begin{proof}
  Let $u\in \mathcal{V}_\varepsilon$.
  By \eqref{E:Fric_Upper} in Assumption~\ref{Assump_1} and \eqref{E:Poin_Bo} we have
  \begin{align*}
    \gamma_\varepsilon^i\|u\|_{L^2(\Gamma_\varepsilon^i)}^2 \leq c\varepsilon\left(\varepsilon^{-1}\|u\|_{L^2(\Omega_\varepsilon)}^2+\varepsilon\|\partial_nu\|_{L^2(\Omega_\varepsilon)}^2\right) \leq c\|u\|_{H^1(\Omega_\varepsilon)}^2
  \end{align*}
  for $i=0,1$.
  Combining this with
  \begin{align*}
    \|D(u)\|_{L^2(\Omega_\varepsilon)}^2 = \frac{1}{2}\left(\|\nabla u\|_{L^2(\Omega_\varepsilon)}^2+\int_{\Omega_\varepsilon}\nabla u:(\nabla u)^T\,dx\right) \leq \|\nabla u\|_{L^2(\Omega_\varepsilon)}^2
  \end{align*}
  by H\"{o}lder's inequality we get the right-hand inequality of \eqref{E:Bili_Core}.

  Let us prove the left-hand inequality of \eqref{E:Bili_Core}.
  First we assume that the condition (A1) of Assumption~\ref{Assump_2} is satisfied.
  Without loss of generality, we may assume that $\gamma_\varepsilon^0\geq c\varepsilon$ for all $\varepsilon\in(0,1)$.
  For $u\in\mathcal{V}_\varepsilon$ we use \eqref{E:Poin_Dom} with $i=0$ to get
  \begin{align*}
    \|u\|_{L^2(\Omega_\varepsilon)}^2 \leq c\left(\varepsilon\|u\|_{L^2(\Gamma_\varepsilon^0)}^2+\varepsilon^2\|\nabla u\|_{L^2(\Omega_\varepsilon)}^2\right).
  \end{align*}
  Moreover, by $\gamma_\varepsilon^0\geq c\varepsilon$ and \eqref{E:Korn_Grad} (note that $u$ satisfies \eqref{E:Bo_Imp} on $\Gamma_\varepsilon$),
  \begin{align*}
    \|u\|_{L^2(\Omega_\varepsilon)}^2 &\leq c\left(\gamma_\varepsilon^0\|u\|_{L^2(\Gamma_\varepsilon^0)}^2+\varepsilon^2\|D(u)\|_{L^2(\Omega_\varepsilon)}^2+\varepsilon^2\|u\|_{L^2(\Omega_\varepsilon)}^2\right) \\
    &\leq c_1a_\varepsilon(u,u)+c_2\varepsilon^2\|u\|_{L^2(\Omega_\varepsilon)}^2
  \end{align*}
  with positive constants $c_1$ and $c_2$ independent of $\varepsilon$.
  We set $\varepsilon_1:=1/\sqrt{2c_2}$ and take $\varepsilon\in(0,\varepsilon_1)$ in the above inequality to get
  \begin{align*}
    \|u\|_{L^2(\Omega_\varepsilon)}^2 \leq 2c_1a_\varepsilon(u,u).
  \end{align*}
  From this inequality and \eqref{E:Korn_Grad} we also deduce that
  \begin{align*}
    \|\nabla u\|_{L^2(\Omega_\varepsilon)}^2 \leq ca_\varepsilon(u,u).
  \end{align*}
  These two inequalities imply the left-hand inequality of \eqref{E:Bili_Core}.

  Next we suppose that the condition (A2) or (A3) of Assumption~\ref{Assump_2} is satisfied.
  Then $u\in\mathcal{V}_\varepsilon$ satisfies \eqref{E:Bo_Imp} on $\Gamma_\varepsilon$ and \eqref{E:Korn_Orth} (resp. \eqref{E:KoRg_Orth}) with $\beta=0$ under the condition (A2) (resp. (A3)).
  Hence by Lemmas~\ref{L:Korn_H1} and~\ref{L:Korn_Rg} there exist $\varepsilon_{K,0}\in(0,1)$ and $c_{K,0}>0$ such that
  \begin{align*}
    \|u\|_{H^1(\Omega_\varepsilon)}^2 \leq c_{K,0}\|D(u)\|_{L^2(\Omega_\varepsilon)}^2 \leq c_{K,0}a_\varepsilon(u,u)
  \end{align*}
  for all $\varepsilon\in(0,\varepsilon_{K,0})$ and $u\in\mathcal{V}_\varepsilon$, i.e. the left-hand inequality of \eqref{E:Bili_Core} holds.

  Therefore, we conclude that the lemma is valid with $\varepsilon_0:=\min\{\varepsilon_1,\varepsilon_{K,0}\}$.
\end{proof}

Under Assumptions~\ref{Assump_1} and~\ref{Assump_2}, let $\varepsilon_0$ be the constant given in Lemma~\ref{L:Bili_Core}.
For $\varepsilon\in(0,\varepsilon_0)$ the bilinear form $a_\varepsilon$ is continuous, coercive, and symmetric on the closed subspace $\mathcal{V}_\varepsilon$ of $H^1(\Omega_\varepsilon)^3$ by Lemma~\ref{L:Bili_Core}.
Hence by the Lax--Milgram theorem there exists a bounded linear operator $A_\varepsilon$ from $\mathcal{V}_\varepsilon$ into its dual space $\mathcal{V}_\varepsilon'$ such that
\begin{align*}
  {}_{\mathcal{V}_\varepsilon'}\langle A_\varepsilon u_1,u_2\rangle_{\mathcal{V}_\varepsilon} = a_\varepsilon(u_1,u_2), \quad u_1,u_2\in \mathcal{V}_\varepsilon,
\end{align*}
where ${}_{\mathcal{V}_\varepsilon'}\langle\cdot,\cdot\rangle_{\mathcal{V}_\varepsilon}$ is the duality product between $\mathcal{V}_\varepsilon'$ and $\mathcal{V}_\varepsilon$ (see e.g.~\cite{BoFa13}).
Let $\mathcal{H}_\varepsilon$ be the closed subspace of $L^2(\Omega_\varepsilon)^3$ given by \eqref{E:Def_Heps} and $\mathbb{P}_\varepsilon$ the orthogonal projection from $L^2(\Omega_\varepsilon)^3$ onto $\mathcal{H}_\varepsilon$.
We consider $A_\varepsilon$ as an unbounded operator on $\mathcal{H}_\varepsilon$ with its domain $D(A_\varepsilon)=\{u\in \mathcal{V}_\varepsilon \mid A_\varepsilon u\in \mathcal{H}_\varepsilon\}$.
Then the Lax--Milgram theory implies that $A_\varepsilon$ is a positive self-adjoint operator on $\mathcal{H}_\varepsilon$ and thus its square root $A_\varepsilon^{1/2}$ is well-defined on $D(A_\varepsilon^{1/2})=\mathcal{V}_\varepsilon$.
Moreover, by a regularity result on the Stokes problem with slip boundary conditions (see ~\cite{Be04,SoSc73}) we see that
\begin{gather*}
  D(A_\varepsilon) = \{u\in \mathcal{V}_\varepsilon\cap H^2(\Omega_\varepsilon)^3 \mid \text{$2\nu P_\varepsilon D(u)n_\varepsilon+\gamma_\varepsilon u=0$ on $\Gamma_\varepsilon$}\}
\end{gather*}
and $A_\varepsilon u=-\nu\mathbb{P}_\varepsilon\Delta u$ for $u\in D(A_\varepsilon)$.
Hereafter we call $A_\varepsilon$ the Stokes operator on $\mathcal{H}_\varepsilon$ associated with slip boundary conditions.
Note that we have
\begin{align} \label{E:L2in_Ahalf}
  (A_\varepsilon u_1,u_2)_{L^2(\Omega_\varepsilon)} = (A_\varepsilon^{1/2}u_1,A_\varepsilon^{1/2} u_2)_{L^2(\Omega_\varepsilon)}
\end{align}
for all $u_1\in D(A_\varepsilon)$ and $u_2\in \mathcal{V}_\varepsilon$ as well as
\begin{align} \label{E:L2nor_Ahalf}
  \|A_\varepsilon^{1/2}u\|_{L^2(\Omega_\varepsilon)}^2 = a_\varepsilon(u,u) = 2\nu\|D(u)\|_{L^2(\Omega_\varepsilon)}^2+\gamma_\varepsilon^0\|u\|_{L^2(\Gamma_\varepsilon^0)}^2+\gamma_\varepsilon^1\|u\|_{L^2(\Gamma_\varepsilon^1)}^2
\end{align}
for all $u\in \mathcal{V}_\varepsilon$.

\begin{lemma} \label{L:Stokes_H1}
  Under Assumptions~\ref{Assump_1} and~\ref{Assump_2}, there exists $c>0$ such that
  \begin{align} \label{E:Stokes_H1}
    c^{-1}\|u\|_{H^1(\Omega_\varepsilon)} \leq \|A_\varepsilon^{1/2}u\|_{L^2(\Omega_\varepsilon)} \leq c\|u\|_{H^1(\Omega_\varepsilon)}
  \end{align}
  for all $\varepsilon\in(0,\varepsilon_0)$ and $u\in \mathcal{V}_\varepsilon$.
  Moreover, if $u\in D(A_\varepsilon)$, then we have
  \begin{align} \label{E:Stokes_Po}
    \|A_\varepsilon^{1/2}u\|_{L^2(\Omega_\varepsilon)} \leq c\|A_\varepsilon u\|_{L^2(\Omega_\varepsilon)}.
  \end{align}
\end{lemma}

\begin{proof}
  The inequality \eqref{E:Stokes_H1} is an immediate consequence of \eqref{E:Bili_Core} and \eqref{E:L2nor_Ahalf}.
  To prove \eqref{E:Stokes_Po} for $u\in D(A_\varepsilon)$ we see by \eqref{E:L2in_Ahalf} and H\"{o}lder's inequality that
  \begin{align*}
    \|A_\varepsilon^{1/2}u\|_{L^2(\Omega_\varepsilon)}^2 = (u,A_\varepsilon u)_{L^2(\Omega_\varepsilon)} \leq \|u\|_{L^2(\Omega_\varepsilon)}\|A_\varepsilon u\|_{L^2(\Omega_\varepsilon)}.
  \end{align*}
  Applying $\|u\|_{L^2(\Omega_\varepsilon)}\leq \|u\|_{H^1(\Omega_\varepsilon)}$ and \eqref{E:Stokes_H1} to the right-hand side we get \eqref{E:Stokes_Po}.
\end{proof}

\subsection{Uniform regularity estimates for the Stokes operator} \label{SS:St_URE}
In this subsection we estimate the difference between the Stokes and Laplace operators and show the uniform equivalence of the norms $\|A_\varepsilon u\|_{L^2(\Omega_\varepsilon)}$ and $\|u\|_{H^2(\Omega_\varepsilon)}$.

First we give an integration by parts formula for the curl of a vector field on $\Omega_\varepsilon$.
Let $n_\varepsilon^0$ and $n_\varepsilon^1$ be the vector fields on $\Gamma$ given by \eqref{E:Def_NB} and
\begin{align*}
  W_\varepsilon^i(x) := -\{I_3-\bar{n}_\varepsilon^i(x)\otimes\bar{n}_\varepsilon^i(x)\}\nabla\bar{n}_\varepsilon^i(x), \quad x\in N,\,i=0,1.
\end{align*}
Here $\bar{n}_\varepsilon^i=n_\varepsilon^i\circ\pi$, $i=0,1$ is the constant extension of $n_\varepsilon^i$.
For $x\in N$ we set
\begin{align*}
  \tilde{n}_1(x) &:= \frac{1}{\varepsilon\bar{g}(x)}\bigl\{\bigl(d(x)-\varepsilon\bar{g}_0(x)\bigr)\bar{n}_\varepsilon^1(x)-\bigl(\varepsilon\bar{g}_1(x)-d(x)\bigr)\bar{n}_\varepsilon^0(x)\bigr\}, \\
  \tilde{n}_2(x) &:= \frac{1}{\varepsilon\bar{g}(x)}\left\{\bigl(d(x)-\varepsilon\bar{g}_0(x)\bigr)\frac{\gamma_\varepsilon^1}{\nu}\bar{n}_\varepsilon^1(x)+\bigl(\varepsilon\bar{g}_1(x)-d(x)\bigr)\frac{\gamma_\varepsilon^0}{\nu}\bar{n}_\varepsilon^0(x)\right\}, \\
  \widetilde{W}(x) &:= \frac{1}{\varepsilon\bar{g}(x)}\bigl\{\bigl(d(x)-\varepsilon\bar{g}_0(x)\bigr)W_\varepsilon^1(x)-\bigl(\varepsilon\bar{g}_1(x)-d(x)\bigr)W_\varepsilon^0(x)\bigr\}.
\end{align*}
From these definitions and Lemma~\ref{L:Nor_Bo} it follows that
\begin{align} \label{E:Tilde_Bo}
  \tilde{n}_1 = (-1)^{i+1}n_\varepsilon, \quad \tilde{n}_2 = \frac{\gamma_\varepsilon}{\nu}n_\varepsilon, \quad \widetilde{W} = (-1)^{i+1}W_\varepsilon \quad\text{on}\quad \Gamma_\varepsilon^i,\, i=0,1.
\end{align}
For a vector field $u\colon\Omega_\varepsilon\to\mathbb{R}^3$ we define $G(u)\colon\Omega_\varepsilon\to\mathbb{R}^3$ by
\begin{align} \label{E:Def_Gu}
  G(u) := G_1(u)+G_2(u), \quad G_1(u) := 2\tilde{n}_1\times \widetilde{W}u, \quad G_2(u) := \tilde{n}_2\times u.
\end{align}

\begin{lemma} \label{L:G_Bound}
  Suppose that the inequalities \eqref{E:Fric_Upper} are valid.
  Then
  \begin{align} \label{E:G_Bound}
    |G(u)| \leq c|u|, \quad |\nabla G(u)| \leq c(|u|+|\nabla u|) \quad\text{in}\quad \Omega_\varepsilon
  \end{align}
  with a constant $c>0$ independent of $\varepsilon$ and $u$.
\end{lemma}

\begin{proof}
  By \eqref{E:ConDer_Bound}, \eqref{E:Con_Hess}, \eqref{E:N_Bound}, and $g_0,g_1\in C^4(\Gamma)$ we see that
  \begin{align} \label{Pf_GB:Est_N_G}
    |\partial_x^\alpha\bar{n}_\varepsilon^i(x)| \leq c, \quad |\partial_x^\alpha\bar{g}_i(x)| \leq c, \quad x \in N,\,i=0,1,\,|\alpha|=0,1,2,
  \end{align}
  where $\partial_x^\alpha=\partial_1^{\alpha_1}\partial_2^{\alpha_2}\partial_3^{\alpha_3}$ for $\alpha=(\alpha_1,\alpha_2,\alpha_3)\in\mathbb{Z}^3$ with $\alpha_j\geq0$, $j=1,2,3$ and $c>0$ is a constant independent of $\varepsilon$.
  From \eqref{Pf_GB:Est_N_G} it also follows that
  \begin{align} \label{Pf_GB:Est_W}
    |W_\varepsilon^i(x)| \leq c, \quad |\partial_kW_\varepsilon^i(x)| \leq c, \quad x \in N, \,i=0,1,\,k=1,2,3.
  \end{align}
  By \eqref{E:Fric_Upper}, \eqref{Pf_EAB:Width}, \eqref{Pf_GB:Est_N_G}, and \eqref{Pf_GB:Est_W} we have
  \begin{align} \label{Pf_GB:Zero_Bound}
    |\tilde{n}_1| \leq c, \quad |\tilde{n}_2| \leq c\varepsilon, \quad \left|\widetilde{W}\right| \leq c \quad\text{in}\quad \Omega_\varepsilon.
  \end{align}
  Thus the first inequality of \eqref{E:G_Bound} is valid.
  To prove the second inequality of \eqref{E:G_Bound} we need to estimate the first order derivatives of $\tilde{n}_1$, $\tilde{n}_2$, and $\widetilde{W}$.
  As in the proof of Lemma~\ref{L:ExAux_Bound} we observe by direct calculations with $\nabla d=\bar{n}$ in $N$ and the inequalities \eqref{E:Width_Bound}, \eqref{Pf_EAB:Width}, \eqref{Pf_GB:Est_N_G}, and \eqref{Pf_GB:Est_W} that
  \begin{gather*}
    \nabla\tilde{n}_1 = \frac{1}{\varepsilon\bar{g}}\bar{n}\otimes(\bar{n}_\varepsilon^0+\bar{n}_\varepsilon^1)+f_1, \quad \nabla\tilde{n}_2 = \frac{1}{\varepsilon\bar{g}}\bar{n}\otimes\left(\frac{\gamma_\varepsilon^1}{\nu}\bar{n}_\varepsilon^1-\frac{\gamma_\varepsilon^0}{\nu}\bar{n}_\varepsilon^0\right)+f_2, \\
    \partial_k\widetilde{W} = \frac{1}{\varepsilon\bar{g}}\bar{n}_k(W_\varepsilon^0+W_\varepsilon^1)+F_k, \quad k=1,2,3
  \end{gather*}
  in $\Omega_\varepsilon$, where $f_1$, $f_2$, and $F_k$ are bounded on $\Omega_\varepsilon$ uniformly in $\varepsilon$.
  We apply \eqref{E:N_Diff} to $\nabla\tilde{n}_1$, \eqref{E:Fric_Upper} and \eqref{E:N_Bound} to $\nabla\tilde{n}_2$, and use \eqref{E:Width_Bound} to show that
  \begin{align} \label{Pf_GB:N_First_Bound}
    |\nabla\tilde{n}_1| \leq \frac{1}{\varepsilon\bar{g}}|\bar{n}_\varepsilon^0+\bar{n}_\varepsilon^1|+|f_1| \leq c, \quad |\nabla\tilde{n}_2| \leq \frac{c(\gamma_\varepsilon^0+\gamma_\varepsilon^1)}{\varepsilon\bar{g}}+|f_2| \leq c \quad\text{in}\quad \Omega_\varepsilon.
  \end{align}
  To estimate the first order derivatives of $\widetilde{W}$ we see by \eqref{E:ConDer_Dom} that
  \begin{multline*}
    W_\varepsilon^0+W_\varepsilon^1 = \{(\bar{n}_\varepsilon^0+\bar{n}_\varepsilon^1)\otimes\bar{n}_\varepsilon^0-\bar{n}_\varepsilon^1\otimes(\bar{n}_\varepsilon^0+\bar{n}_\varepsilon^1)\}\Bigl(I_3-d\overline{W}\Bigr)^{-1}\overline{\nabla_\Gamma n_\varepsilon^0} \\
    -(I_3-\bar{n}_\varepsilon^1\otimes\bar{n}_\varepsilon^1)\Bigl(I_3-d\overline{W}\Bigr)^{-1}\left(\overline{\nabla_\Gamma n_\varepsilon^0}+\overline{\nabla_\Gamma n_\varepsilon^1}\right)
  \end{multline*}
  in $N$.
  Hence $|W_\varepsilon^0+W_\varepsilon^1|\leq c\varepsilon$ in $N$ by \eqref{E:Wein_Bound}, \eqref{E:N_Bound}, and \eqref{E:N_Diff} and we get
  \begin{align} \label{Pf_GB:W_First_Bound}
    \left|\partial_k\widetilde{W}\right| \leq \frac{1}{\varepsilon\bar{g}}|W_\varepsilon^0+W_\varepsilon^1|+|F_k| \leq c \quad\text{in}\quad \Omega_\varepsilon.
  \end{align}
  Applying \eqref{Pf_GB:Zero_Bound}--\eqref{Pf_GB:W_First_Bound} to $\nabla G(u)$ we obtain the second inequality of \eqref{E:G_Bound}.
\end{proof}

\begin{lemma} \label{L:IbP_Curl}
  We have the integration by parts formula
  \begin{multline} \label{E:IbP_Curl}
    \int_{\Omega_\varepsilon}\mathrm{curl}\,\mathrm{curl}\,u\cdot\Phi\,dx \\
    = -\int_{\Omega_\varepsilon}\mathrm{curl}\,G(u)\cdot\Phi\,dx+\int_{\Omega_\varepsilon}\{\mathrm{curl}\,u+G(u)\}\cdot\mathrm{curl}\,\Phi\,dx
  \end{multline}
  for all $u\in H^2(\Omega_\varepsilon)^3$ satisfying the slip boundary conditions \eqref{E:Bo_Imp}--\eqref{E:Bo_Slip} on $\Gamma_\varepsilon$ and $\Phi\in L^2(\Omega_\varepsilon)^3$ with $\mathrm{curl}\,\Phi\in L^2(\Omega_\varepsilon)^3$, where $G(u)$ is given by \eqref{E:Def_Gu}.
\end{lemma}

\begin{proof}
  By standard cut-off, dilatation, and mollification arguments, we can show as in the proof of~\cite[Chapter~1, Theorem~1.1]{Te79} that for $\Phi\in L^2(\Omega_\varepsilon)^3$ with $\mathrm{curl}\,\Phi\in L^2(\Omega_\varepsilon)^3$ there exists a sequence $\{\Phi_k\}_{k=1}^\infty$ in $C^\infty(\overline{\Omega}_\varepsilon)^3$ such that
  \begin{align*}
    \lim_{k\to\infty}\|\Phi-\Phi_k\|_{L^2(\Omega_\varepsilon)} = \lim_{k\to\infty}\|\mathrm{curl}\,\Phi-\mathrm{curl}\,\Phi_k\|_{L^2(\Omega_\varepsilon)} = 0.
  \end{align*}
  Thus, by a density argument, it is sufficient to prove \eqref{E:IbP_Curl} for all $\Phi\in C^\infty(\overline{\Omega}_\varepsilon)^3$.

  Let $u\in H^2(\Omega_\varepsilon)^3$ satisfy \eqref{E:Bo_Imp}--\eqref{E:Bo_Slip} on $\Gamma_\varepsilon$ and $\Phi\in C^\infty(\overline{\Omega}_\varepsilon)^3$.
  Then
  \begin{align} \label{Pf_IC:IBP}
    \int_{\Omega_\varepsilon}\mathrm{curl}\,\mathrm{curl}\,u\cdot\Phi\,dx = \int_{\Gamma_\varepsilon}(n_\varepsilon\times\mathrm{curl}\,u)\cdot\Phi\,d\mathcal{H}^2+\int_{\Omega_\varepsilon}\mathrm{curl}\,u\cdot\mathrm{curl}\,\Phi\,dx
  \end{align}
  by integration by parts.
  Since $u$ satisfies \eqref{E:Bo_Imp}--\eqref{E:Bo_Slip} on $\Gamma_\varepsilon$,
  \begin{align*}
    n_\varepsilon\times\mathrm{curl}\,u &= -n_\varepsilon\times\left\{n_\varepsilon\times\left(2W_\varepsilon u+\frac{\gamma_\varepsilon}{\nu}u\right)\right\} \\
    &= -n_\varepsilon\times\left(2\tilde{n}_1\times\widetilde{W}u+\tilde{n}_2\times u\right) = -n_\varepsilon\times G(u)
  \end{align*}
  on $\Gamma_\varepsilon$ by \eqref{E:NSl_Curl}, \eqref{E:Tilde_Bo}, and \eqref{E:Def_Gu}.
  Hence integration by parts yields
  \begin{align*}
    \int_{\Gamma_\varepsilon}(n_\varepsilon\times\mathrm{curl}\,u)\cdot\Phi\,d\mathcal{H}^2 &= -\int_{\Gamma_\varepsilon}\{n_\varepsilon\times G(u)\}\cdot\Phi\,d\mathcal{H}^2 \\
    &= \int_{\Omega_\varepsilon}\{G(u)\cdot\mathrm{curl}\,\Phi-\mathrm{curl}\,G(u)\cdot\Phi\}\,dx.
  \end{align*}
  Substituting this for \eqref{Pf_IC:IBP} we obtain \eqref{E:IbP_Curl}.
\end{proof}

Using \eqref{E:IbP_Curl} we derive an $L^2$-estimate for the difference between the Stokes and Laplace operators as in~\cite[Theorem~2.1]{Ho08}.
Under Assumptions~\ref{Assump_1} and~\ref{Assump_2}, let $\varepsilon_0$ be the constant given in Lemma~\ref{L:Bili_Core} and $A_\varepsilon$ the Stokes operator on $\mathcal{H}_\varepsilon$.

\begin{lemma} \label{L:Comp_Sto_Lap}
  Under Assumptions~\ref{Assump_1} and~\ref{Assump_2}, we have
  \begin{align} \label{E:Comp_Sto_Lap}
    \|A_\varepsilon u+\nu\Delta u\|_{L^2(\Omega_\varepsilon)} \leq c\|u\|_{H^1(\Omega_\varepsilon)}
  \end{align}
  for all $\varepsilon\in(0,\varepsilon_0)$ and $u\in D(A_\varepsilon)$ with a constant $c>0$ independent of $\varepsilon$ and $u$.
\end{lemma}

Note that the $L^2(\Omega_\varepsilon)$-norm of $A_\varepsilon u+\nu\Delta u$ is bounded by the $H^1(\Omega_\varepsilon)$-norm of $u$ in \eqref{E:Comp_Sto_Lap}, not by its $H^2(\Omega_\varepsilon)$-norm.

\begin{proof}
  Let $\mathbb{L}_\varepsilon$ be the Helmholtz--Leray projection onto $L_\sigma^2(\Omega_\varepsilon)$ given in Section~\ref{SS:St_HL}.
  We first show that there exists a constant $c>0$ such that
  \begin{align} \label{Pf_CSL:Diff_HL}
    \|\nu\Delta u-\nu\mathbb{L}_\varepsilon\Delta u\|_{L^2(\Omega_\varepsilon)} \leq c\|u\|_{H^1(\Omega_\varepsilon)}
  \end{align}
  for all $\varepsilon\in(0,\varepsilon_0)$ and $u\in D(A_\varepsilon)$.
  By the Helmholtz--Leray decomposition
  \begin{align*}
    \nu\Delta u = \nu\mathbb{L}_\varepsilon\Delta u+\nabla q \quad\text{in}\quad L^2(\Omega_\varepsilon)^3 \quad\text{with}\quad (\nu\mathbb{L}_\varepsilon\Delta u,\nabla q)_{L^2(\Omega_\varepsilon)} = 0
  \end{align*}
  for $u\in D(A_\varepsilon)$ and $\Delta u=-\mathrm{curl}\,\mathrm{curl}\,u$ in $\Omega_\varepsilon$ (note that $\mathrm{div}\,u=0$ in $\Omega_\varepsilon$) we have
  \begin{align*}
    \|\nu\Delta u-\nu\mathbb{L}_\varepsilon\Delta u\|_{L^2(\Omega_\varepsilon)}^2 &= (\nu\Delta u,\nabla q)_{L^2(\Omega_\varepsilon)}-(\nu\mathbb{L}_\varepsilon\Delta u,\nabla q)_{L^2(\Omega_\varepsilon)} \\
    &= -\nu(\mathrm{curl}\,\mathrm{curl}\,u,\nabla q)_{L^2(\Omega_\varepsilon)}.
  \end{align*}
  Noting that $\mathrm{curl}\,\nabla q=0$ in $\Omega_\varepsilon$, we apply \eqref{E:IbP_Curl} with $\Phi=\nabla q$ to the last term to get
  \begin{align*}
    -\nu(\mathrm{curl}\,\mathrm{curl}\,u,\nabla q)_{L^2(\Omega_\varepsilon)} = \nu(\mathrm{curl}\,G(u),\nabla q)_{L^2(\Omega_\varepsilon)} \leq c\|\nabla G(u)\|_{L^2(\Omega_\varepsilon)}\|\nabla q\|_{L^2(\Omega_\varepsilon)}
  \end{align*}
  with $G(u)$ given by \eqref{E:Def_Gu}.
  Since the inequalities \eqref{E:Fric_Upper} hold by Assumption~\ref{Assump_1}, we can use \eqref{E:G_Bound} to the right-hand side of this inequality.
  Hence
  \begin{align*}
    \|\nu\Delta u-\nu\mathbb{L}_\varepsilon\Delta u\|_{L^2(\Omega_\varepsilon)}^2 &\leq c\|\nabla G(u)\|_{L^2(\Omega_\varepsilon)}\|\nabla q\|_{L^2(\Omega_\varepsilon)} \\
    &\leq c\|u\|_{H^1(\Omega_\varepsilon)}\|\nu\Delta u-\nu\mathbb{L}_\varepsilon\Delta u\|_{L^2(\Omega_\varepsilon)}
  \end{align*}
  and the inequality \eqref{Pf_CSL:Diff_HL} holds (note that $\nabla q=\nu\Delta u-\nu\mathbb{L}_\varepsilon\Delta u$).
  When the condition (A1) or (A2) of Assumption~\ref{Assump_2} is satisfied, the inequality \eqref{E:Comp_Sto_Lap} is an immediate consequence of \eqref{Pf_CSL:Diff_HL} since $A_\varepsilon u=-\nu\mathbb{P}_\varepsilon\Delta u$ and $\mathbb{L}_\varepsilon$ agrees with the orthogonal projection $\mathbb{P}_\varepsilon$ onto $\mathcal{H}_\varepsilon=L_\sigma^2(\Omega_\varepsilon)$.

  Next we suppose that the condition (A3) of Assumption~\ref{Assump_2} is satisfied.
  In this case $\mathbb{P}_\varepsilon$ is the orthogonal projection onto $\mathcal{H}_\varepsilon=L_\sigma^2(\Omega_\varepsilon)\cap\mathcal{R}_g^\perp$, where $\mathcal{R}_g$ is the space of infinitesimal rigid displacements of $\mathbb{R}^3$ given by \eqref{E:Def_Rg}.
  Let $u\in D(A_\varepsilon)$ and $w\in\mathcal{R}_g$.
  Since $\mathbb{L}_\varepsilon$ is the orthogonal projection onto $L_\sigma^2(\Omega_\varepsilon)$ and $w\in L_\sigma^2(\Omega_\varepsilon)$ by the assumption $\mathcal{R}_g=\mathcal{R}_0\cap\mathcal{R}_1$ and Lemma~\ref{L:IR_Sole}, $(\mathbb{L}_\varepsilon\Delta u,w)_{L^2(\Omega_\varepsilon)} = (\Delta u,w)_{L^2(\Omega_\varepsilon)}$.
  Moreover, under the conditions $\mathcal{R}_g=\mathcal{R}_0\cap\mathcal{R}_1$ and $\gamma_\varepsilon^0=\gamma_\varepsilon^1=0$, the vector fields $u\in D(A_\varepsilon)$ and $w\in\mathcal{R}_g$ satisfy (note that $w$ is of the form $w(x)=a\times x+b$)
  \begin{gather*}
    \mathrm{div}\,u = 0, \quad D(w) = 0 \quad\text{in}\quad \Omega_\varepsilon, \\
    u\cdot n_\varepsilon = 0, \quad P_\varepsilon D(u)n_\varepsilon = 0, \quad w\cdot n_\varepsilon = 0 \quad\text{on}\quad \Gamma_\varepsilon.
  \end{gather*}
  By these equalities and the integration by parts formula \eqref{E:IbP_St} we have
  \begin{align*}
    (\Delta u,w)_{L^2(\Omega_\varepsilon)} = -2\bigl(D(u),D(w)\bigr)_{L^2(\Omega_\varepsilon)}+2(D(u)n_\varepsilon,w)_{L^2(\Gamma_\varepsilon)} = 0.
  \end{align*}
  Hence $(\mathbb{L}_\varepsilon\Delta u,w)_{L^2(\Omega_\varepsilon)}=0$ for all $w\in\mathcal{R}_g$, i.e. $\mathbb{L}_\varepsilon\Delta u\in\mathcal{R}_g^\perp$.
  Now we observe by the Helmholtz--Leray decomposition
  \begin{align*}
    \Delta u = \mathbb{L}_\varepsilon\Delta u+\nabla\tilde{q}, \quad \mathbb{L}_\varepsilon\Delta u \in L_\sigma^2(\Omega_\varepsilon)\cap\mathcal{R}_g^\perp = \mathcal{H}_\varepsilon, \quad \nabla\tilde{q}\in G^2(\Omega_\varepsilon) \subset \mathcal{H}_\varepsilon^\perp
  \end{align*}
  that $\mathbb{P}_\varepsilon\Delta u=\mathbb{P}_\varepsilon\mathbb{L}_\varepsilon\Delta u=\mathbb{L}_\varepsilon\Delta u$, since $\mathbb{P}_\varepsilon$ is the orthogonal projection onto $\mathcal{H}_\varepsilon$.
  Therefore, $A_\varepsilon u=-\nu\mathbb{L}_\varepsilon\Delta u$ and the inequality \eqref{E:Comp_Sto_Lap} follows from \eqref{Pf_CSL:Diff_HL}.
\end{proof}

Next we show that for $u\in D(A_\varepsilon)$ the norm $\|A_\varepsilon u\|_{L^2(\Omega_\varepsilon)}$ is bounded from above and below by the canonical $H^2(\Omega_\varepsilon)$-norm of $u$ with constants independent of $\varepsilon$.

\begin{lemma} \label{L:Lap_Apri}
  Suppose that the inequalities \eqref{E:Fric_Upper} are valid for all $\varepsilon\in(0,1)$.
  Then there exists a constant $c>0$ independent of $\varepsilon$ such that
  \begin{align} \label{E:Lap_Apri}
    \|u\|_{H^2(\Omega_\varepsilon)} \leq c\left(\|\Delta u\|_{L^2(\Omega_\varepsilon)}+\|u\|_{H^1(\Omega_\varepsilon)}\right)
  \end{align}
  for all $\varepsilon\in(0,1)$ and $u\in H^2(\Omega_\varepsilon)^3$ satisfying \eqref{E:Bo_Imp} and \eqref{E:Bo_Slip} on $\Gamma_\varepsilon$.
\end{lemma}

The proof of Lemma~\ref{L:Lap_Apri} is similar to that of Lemma~\ref{L:PD_H2}, but we need to carry out calculations a lot and use some formulas for the Riemannian connection on $\Gamma_\varepsilon$.
We prove Lemma~\ref{L:Lap_Apri} in the next subsection.

Using Lemmas~\ref{L:Comp_Sto_Lap} and~\ref{L:Lap_Apri} we show the uniform norm equivalence for $A_\varepsilon u$.

\begin{lemma} \label{L:Stokes_H2}
  Under Assumptions~\ref{Assump_1} and~\ref{Assump_2}, we have
  \begin{align} \label{E:Stokes_H2}
    c^{-1}\|u\|_{H^2(\Omega_\varepsilon)} \leq \|A_\varepsilon u\|_{L^2(\Omega_\varepsilon)} \leq c\|u\|_{H^2(\Omega_\varepsilon)}
  \end{align}
  for all $\varepsilon\in(0,\varepsilon_0)$ and $u\in D(A_\varepsilon)$ with a constant $c>0$ independent of $\varepsilon$ and $u$.
\end{lemma}

\begin{proof}
  By \eqref{E:Comp_Sto_Lap} and \eqref{E:Lap_Apri} we have
  \begin{align*}
    \|u\|_{H^2(\Omega_\varepsilon)} &\leq c\left(\|\Delta u\|_{L^2(\Omega_\varepsilon)}+\|u\|_{H^1(\Omega_\varepsilon)}\right) \\
    &\leq c\left(\|A_\varepsilon u\|_{L^2(\Omega_\varepsilon)}+\|A_\varepsilon u+\nu\Delta u\|_{L^2(\Omega_\varepsilon)}+\|u\|_{H^1(\Omega_\varepsilon)}\right) \\
    &\leq c\left(\|A_\varepsilon u\|_{L^2(\Omega_\varepsilon)}+\|u\|_{H^1(\Omega_\varepsilon)}\right).
  \end{align*}
  Applying \eqref{E:Stokes_H1} and \eqref{E:Stokes_Po} to the second term on the last line we obtain the left-hand inequality of \eqref{E:Stokes_H2}.
  Also, by \eqref{E:Comp_Sto_Lap} and $\|u\|_{H^1(\Omega_\varepsilon)}\leq\|u\|_{H^2(\Omega_\varepsilon)}$,
  \begin{align*}
    \|A_\varepsilon u\|_{L^2(\Omega_\varepsilon)} \leq \|A_\varepsilon u+\nu\Delta u\|_{L^2(\Omega_\varepsilon)}+\|\nu\Delta u\|_{L^2(\Omega_\varepsilon)} \leq c\|u\|_{H^2(\Omega_\varepsilon)}.
  \end{align*}
  Hence the right-hand inequality of \eqref{E:Stokes_H2} holds.
\end{proof}

As a consequence of Lemmas~\ref{L:Stokes_H1} and~\ref{L:Stokes_H2}, we obtain an interpolation inequality for a vector field in $D(A_\varepsilon)$.

\begin{lemma} \label{L:St_Inter}
  Under Assumptions~\ref{Assump_1} and~\ref{Assump_2}, we have
  \begin{align} \label{E:St_Inter}
    \|u\|_{H^1(\Omega_\varepsilon)} \leq c\|u\|_{L^2(\Omega_\varepsilon)}^{1/2}\|u\|_{H^2(\Omega_\varepsilon)}^{1/2}
  \end{align}
  for all $\varepsilon\in(0,\varepsilon_0)$ and $u\in D(A_\varepsilon)$ with a constant $c>0$ independent of $\varepsilon$ and $u$.
\end{lemma}

\begin{proof}
  Let $u\in D(A_\varepsilon)$.
  From \eqref{E:L2in_Ahalf} and \eqref{E:Stokes_H1} it follows that
  \begin{align*}
    \|u\|_{H^1(\Omega_\varepsilon)}^2 \leq c\|A_\varepsilon^{1/2}u\|_{L^2(\Omega_\varepsilon)}^2 = c(A_\varepsilon u,u)_{L^2(\Omega_\varepsilon)} \leq c\|A_\varepsilon u\|_{L^2(\Omega_\varepsilon)}\|u\|_{L^2(\Omega_\varepsilon)}.
  \end{align*}
  Applying \eqref{E:Stokes_H2} to the right-hand side of this inequality we get
  \begin{align*}
    \|u\|_{H^1(\Omega_\varepsilon)}^2 \leq c\|u\|_{L^2(\Omega_\varepsilon)}\|u\|_{H^2(\Omega_\varepsilon)}.
  \end{align*}
  Hence \eqref{E:St_Inter} is valid.
\end{proof}

\subsection{Uniform a priori estimate for the vector Laplacian} \label{SS:St_Apr}
The purpose of this subsection is to provide the proof of Lemma~\ref{L:Lap_Apri}.
First we give an approximation result of a vector field in $H^2(\Omega_\varepsilon)^3$ satisfying the slip boundary conditions.
To this end, we consider the second order PDEs
\begin{align} \label{E:EE_LE}
  \begin{cases}
    -\nu\{\Delta u+\nabla(\mathrm{div}\,u)\}+u = f &\text{in}\quad \Omega_\varepsilon, \\
    u\cdot n_\varepsilon = 0, \quad 2\nu P_\varepsilon D(u)n_\varepsilon+\gamma_\varepsilon u = 0 &\text{on}\quad \Gamma_\varepsilon
  \end{cases}
\end{align}
for a given data $f\colon\Omega_\varepsilon\to\mathbb{R}^3$.
The bilinear form corresponding to \eqref{E:EE_LE} is given by
\begin{align*}
  \tilde{a}_\varepsilon(u_1,u_2) := 2\nu\bigl(D(u_1),D(u_2)\bigr)_{L^2(\Omega_\varepsilon)}+(u_1,u_2)_{L^2(\Omega_\varepsilon)}+\sum_{i=0,1}\gamma_\varepsilon^i(u_1,u_2)_{L^2(\Gamma_\varepsilon^i)}
\end{align*}
for $u_1,u_2\in H^1(\Omega_\varepsilon)^3$ by \eqref{E:IbP_St} and \eqref{E:IbP_StBo}.
As in Section~\ref{SS:St_HL} we denote by $\langle\cdot,\cdot\rangle_{\Omega_\varepsilon}$ the duality product between $H^{-1}(\Omega_\varepsilon)$ and $H^1(\Omega_\varepsilon)$.
Also, let
\begin{align*}
  H_{n,0}^1(\Omega_\varepsilon) := \{u \in H^1(\Omega_\varepsilon)^3 \mid \text{$u\cdot n_\varepsilon=0$ on $\Gamma_\varepsilon$}\}.
\end{align*}
Note that $H_{n,0}^1(\Omega_\varepsilon)$ is closed in $H^1(\Omega_\varepsilon)^3$ and thus a Hilbert space.

\begin{lemma} \label{L:EE_Reg}
  For $\varepsilon\in(0,1)$ let $f\in H^{-1}(\Omega_\varepsilon)^3$.
  Suppose that the inequalities \eqref{E:Fric_Upper} are valid.
  Then there exists a unique weak solution $u\in H_{n,0}^1(\Omega_\varepsilon)$ to \eqref{E:EE_LE} in the sense that $\tilde{a}_\varepsilon(u,\Phi)=\langle f,\Phi\rangle_{\Omega_\varepsilon}$ for all $\Phi\in H_{n,0}^1(\Omega_\varepsilon)$.
  Moreover, if $f\in L^2(\Omega_\varepsilon)^3$, then $u\in H^2(\Omega_\varepsilon)^3$ and it satisfies \eqref{E:EE_LE} a.e. in $\Omega_\varepsilon$ and on $\Gamma_\varepsilon$, and there exists a constant $c_\varepsilon>0$ depending on $\varepsilon$ such that
  \begin{align} \label{E:EE_Reg}
    \|u\|_{H^2(\Omega_\varepsilon)} \leq c_\varepsilon\|f\|_{L^2(\Omega_\varepsilon)}.
  \end{align}
  If in addition $f\in H^1(\Omega_\varepsilon)^3$ then $u\in H^3(\Omega_\varepsilon)^3$.
\end{lemma}

\begin{proof}
  Using the inequalities \eqref{E:Fric_Upper}, \eqref{E:Poin_Bo}, and \eqref{E:Korn_Grad} we can show that
  \begin{align*}
    c^{-1}\|u\|_{H^1(\Omega_\varepsilon)}^2 \leq \tilde{a}_\varepsilon(u,u) \leq c\|u\|_{H^1(\Omega_\varepsilon)}^2 \quad\text{for all}\quad u \in H_{n,0}^1(\Omega_\varepsilon)
  \end{align*}
  with a constant $c>0$ independent of $\varepsilon$ as in the proof of Lemma~\ref{L:Bili_Core} (note that $\tilde{a}_\varepsilon(u,u)$ contains the square of the $L^2(\Omega_\varepsilon)$-norm of $u$).
  Hence by the Lax--Milgram theorem we get the existence and uniqueness of a weak solution to \eqref{E:EE_LE}.

  When $f\in L^2(\Omega_\varepsilon)^3$, the $H^2$-regularity of $u$ and \eqref{E:EE_Reg} are proved by a standard localization argument and a method of the difference quotient.
  Here we omit their proofs since they are the same as those of \cite[Theorem~1.2]{Be04} and~\cite[Theorem~2]{SoSc73}, which establish the $H^2$-regularity of a weak solution to the Stokes equations in a bounded domain with slip boundary conditions.

  The $H^3$-regularity of $u$ for $f\in H^1(\Omega_\varepsilon)^3$ is proved by induction and a localization argument as in the case of a general second order elliptic equation.
  For details, see \cite[Section~6.3, Theorem~5]{Ev10}.
  (Note that the $C^4$-regularity of the boundary $\Gamma_\varepsilon$ is required for the $H^3$-regularity of $u$, see Section~\ref{SS:Pre_Dom} and the proofs of \cite[Theorem~1.2]{Be04} and~\cite[Theorem~2]{SoSc73}.)
\end{proof}

Based on Lemma~\ref{L:EE_Reg} we show that a vector field in $H^2(\Omega_\varepsilon)^3$ is approximated by those in $H^3(\Omega_\varepsilon)^3$ under the slip boundary conditions \eqref{E:Bo_Imp}--\eqref{E:Bo_Slip}.

\begin{lemma} \label{L:NSl_Approx}
  For $\varepsilon\in(0,1)$ suppose that the inequalities \eqref{E:Fric_Upper} are valid.
  Also, let $u\in H^2(\Omega_\varepsilon)^3$ satisfy \eqref{E:Bo_Imp}--\eqref{E:Bo_Slip} on $\Gamma_\varepsilon$.
  Then there exists a sequence $\{u_k\}_{k=1}^\infty$ in $H^3(\Omega_\varepsilon)^3$ such that $u_k$ satisfies \eqref{E:Bo_Imp}--\eqref{E:Bo_Slip} on $\Gamma_\varepsilon$ for each $k\in\mathbb{N}$ and
  \begin{align*}
    \lim_{k\to\infty}\|u-u_k\|_{H^2(\Omega_\varepsilon)}=0.
  \end{align*}
\end{lemma}

\begin{proof}
  Let $f:=-\nu\{\Delta u+\nabla(\mathrm{div}\,u)\}+u$ for $u\in H^2(\Omega_\varepsilon)^3$ satisfying \eqref{E:Bo_Imp}--\eqref{E:Bo_Slip} on $\Gamma_\varepsilon$.
  Since $f\in L^2(\Omega_\varepsilon)^3$, there exists a sequence $\{f_k\}_{k=1}^\infty$ in $C_c^\infty(\Omega_\varepsilon)^3$ that converges to $f$ strongly in $L^2(\Omega_\varepsilon)^3$.
  For each $k\in\mathbb{N}$ let $u_k$ be a unique weak solution to \eqref{E:EE_LE} with data $f_k\in C_c^\infty(\Omega_\varepsilon)^3$.
  Then $u_k\in H^3(\Omega_\varepsilon)^3$ and it satisfies \eqref{E:Bo_Imp}--\eqref{E:Bo_Slip} on $\Gamma_\varepsilon$ by Lemma~\ref{L:EE_Reg}.
  Moreover, since $u-u_k$ is a unique weak solution to \eqref{E:EE_LE} with data $f-f_k$, we see that
  \begin{align*}
    \|u-u_k\|_{H^2(\Omega_\varepsilon)} \leq c_\varepsilon\|f-f_k\|_{L^2(\Omega_\varepsilon)} \to 0 \quad\text{as}\quad k\to\infty
  \end{align*}
  by \eqref{E:EE_Reg} and the strong convergence of $\{f_k\}_{k=1}^\infty$ to $f$ in $L^2(\Omega_\varepsilon)^3$ (note that the constant $c_\varepsilon$ does not depend on $k$).
\end{proof}

Now let us prove Lemma~\ref{L:Lap_Apri}.
As in Section~\ref{SS:Pre_Surf} we denote by
\begin{align*}
  H^m(\Gamma_\varepsilon,T\Gamma_\varepsilon) := \{u\in H^m(\Gamma_\varepsilon)^3 \mid \text{$u\cdot n_\varepsilon=0$ on $\Gamma_\varepsilon$}\}, \quad m=0,1,2
\end{align*}
the space of all tangential vector fields on $\Gamma_\varepsilon$ of class $H^m$ (here we write $H^0=L^2$).
For $u\in H^1(\Gamma_\varepsilon,T\Gamma_\varepsilon)$ and $v\in L^2(\Gamma_\varepsilon,T\Gamma_\varepsilon)$ we define the covariant derivative
\begin{align*}
  \overline{\nabla}_v^\varepsilon u := P_\varepsilon(v\cdot\nabla)\tilde{u} = P_\varepsilon(v\cdot\nabla_{\Gamma_\varepsilon})u \quad\text{on}\quad \Gamma_\varepsilon,
\end{align*}
where $\tilde{u}$ is any $H^1$-extension of $u$ to an open neighborhood of $\Gamma_\varepsilon$ with $\tilde{u}|_{\Gamma_\varepsilon}=u$.
We use the formulas for the covariant derivatives given in Appendix~\ref{S:Ap_RC}.

\begin{proof}[Proof of Lemma~\ref{L:Lap_Apri}]
  Let $u\in H^2(\Omega_\varepsilon)^3$ satisfy \eqref{E:Bo_Imp}--\eqref{E:Bo_Slip} on $\Gamma_\varepsilon$.
  Since
  \begin{align*}
    \|u\|_{H^2(\Omega_\varepsilon)}^2 = \|u\|_{H^1(\Omega_\varepsilon)}^2+\|\nabla^2u\|_{L^2(\Omega_\varepsilon)}^2,
  \end{align*}
  it is sufficient for \eqref{E:Lap_Apri} to show that
  \begin{align} \label{Pf_LA:Goal}
    \|\nabla^2u\|_{L^2(\Omega_\varepsilon)}^2 \leq c\left(\|\Delta u\|_{L^2(\Omega_\varepsilon)}^2+\|u\|_{H^1(\Omega_\varepsilon)}^2\right).
  \end{align}
  Moreover, we may assume $u\in H^3(\Omega_\varepsilon)^3$ by Lemma~\ref{L:NSl_Approx} and thus its trace on $\Gamma_\varepsilon$ is in $H^2(\Gamma_\varepsilon,T\Gamma_\varepsilon)$.
  For such $u$ we can carry out integration by parts twice to get
  \begin{align} \label{Pf_LA:IbP}
    \|\nabla^2u\|_{L^2(\Omega_\varepsilon)}^2 = \|\Delta u\|_{L^2(\Omega_\varepsilon)}^2+\int_{\Gamma_\varepsilon}\nabla u:\{(n_\varepsilon\cdot\nabla)\nabla u-n_\varepsilon\otimes\Delta u\}\,d\mathcal{H}^2.
  \end{align}
  Here $(n_\varepsilon\cdot\nabla)\nabla u$ denotes a $3\times3$ matrix whose $(i,j)$-entry is given by
  \begin{align*}
    [(n_\varepsilon\cdot\nabla)\nabla u]_{ij} := n_\varepsilon\cdot\nabla(\partial_iu_j), \quad i,j=1,2,3.
  \end{align*}
  Let us estimate the boundary integral in \eqref{Pf_LA:IbP}.
  Since $u$ satisfies the slip boundary conditions \eqref{E:Bo_Imp}--\eqref{E:Bo_Slip} on $\Gamma_\varepsilon$, we can use \eqref{E:NSl_ND} to get
  \begin{align} \label{Pf_LA:ND_Bo}
    (\nabla u)^Tn_\varepsilon = (n_\varepsilon\cdot\nabla)u = -W_\varepsilon u-\tilde{\gamma}_\varepsilon u+\xi_\varepsilon n_\varepsilon \quad\text{on}\quad \Gamma_\varepsilon
  \end{align}
  with $\tilde{\gamma}_\varepsilon:=\gamma_\varepsilon/\nu$ and $\xi_\varepsilon:= (n_\varepsilon\cdot\nabla)u\cdot n_\varepsilon = \nabla u: Q_\varepsilon$ (note that $u$ and $W_\varepsilon u$ are tangential on $\Gamma_\varepsilon$).
  The first step is to show that
  \begin{align} \label{Pf_LA:Int_Sum}
    \int_{\Gamma_\varepsilon}\nabla u:\{(n_\varepsilon\cdot\nabla)\nabla u-n_\varepsilon\otimes\Delta u\}\,d\mathcal{H}^2 = \sum_{k=1}^4\int_{\Gamma_\varepsilon}\varphi_k\,d\mathcal{H}^2,
  \end{align}
  where
  \begin{align} \label{Pf_LA:Def_Phi}
    \begin{aligned}
      \varphi_1 &:= -2\{\nabla_{\Gamma_\varepsilon}W_\varepsilon\cdot u+(\nabla u)W_\varepsilon+\tilde{\gamma_\varepsilon}\nabla u\}:P_\varepsilon(\nabla u)P_\varepsilon,\\
      \varphi_2 &:= W_\varepsilon\nabla u:(\nabla u)P_\varepsilon \\
      &\qquad\qquad -2(u\cdot\mathrm{div}_{\Gamma_\varepsilon}W_\varepsilon+2\nabla u:W_\varepsilon)(\nabla u:Q_\varepsilon)+H_\varepsilon(\nabla u:Q_\varepsilon)^2, \\
      \varphi_3 &:= -(W_\varepsilon^3u-H_\varepsilon W_\varepsilon^2u)\cdot u, \\
      \varphi_4 &:= -\tilde{\gamma}_\varepsilon(2W_\varepsilon^2u-2H_\varepsilon W_\varepsilon u-\tilde{\gamma}_\varepsilon H_\varepsilon u)\cdot u.
    \end{aligned}
  \end{align}
  In \eqref{Pf_LA:Def_Phi} we used the notation $\nabla_{\Gamma_\varepsilon}W_\varepsilon\cdot u$ for the $3\times 3$ matrix given by
  \begin{align} \label{Pf_LA:DW_u_Bo}
    [\nabla_{\Gamma_\varepsilon}W_\varepsilon\cdot u]_{ij} := \sum_{k=1}^3(\underline{D}_i^\varepsilon[W_\varepsilon]_{jk})u_k, \quad i,j=1,2,3.
  \end{align}
  Using a partition of unity on $\Gamma_\varepsilon$ we may assume that $u|_{\Gamma_\varepsilon}$ is compactly supported in a relatively open subset $U$ of $\Gamma_\varepsilon$ on which we can take a local orthonormal frame $\{\tau_1,\tau_2\}$ (see Appendix~\ref{S:Ap_RC}).
  Since $\{\tau_1,\tau_2,n_\varepsilon\}$ is an orthonormal basis of $\mathbb{R}^3$,
  \begin{multline} \label{Pf_LA:Decom_Intg}
    \nabla u: \{(n_\varepsilon\cdot\nabla)\nabla u-n_\varepsilon\otimes\Delta u\} \\
    = (\nabla u)^T: [\{(n_\varepsilon\cdot\nabla)\nabla u\}^T-\Delta u\otimes n_\varepsilon] = \eta_1+\eta_2+\eta_3
  \end{multline}
  on $U$, where
  \begin{align}
    \eta_i &:= (\nabla u)^T\tau_i\cdot[\{(n_\varepsilon\cdot\nabla)\nabla u\}^T\tau_i-(\Delta u\otimes n_\varepsilon)\tau_i], \quad i=1,2, \label{Pf_LA:Def_etai}\\
    \eta_3 &:= (\nabla u)^Tn_\varepsilon\cdot[\{(n_\varepsilon\cdot\nabla)\nabla u\}^Tn_\varepsilon-(\Delta u\otimes n_\varepsilon)n_\varepsilon]. \label{Pf_LA:Def_eta3}
  \end{align}
  In what follows, we carry out calculations on $U$.
  By \eqref{E:Gauss} and $\tau_i\cdot n_\varepsilon=0$,
  \begin{align} \label{Pf_LA:etai_1}
    \begin{aligned}
      (\nabla u)^T\tau_i &= (\tau_i\cdot\nabla)u = \overline{\nabla}_i^\varepsilon u+(W_\varepsilon u\cdot\tau_i)n_\varepsilon, \\
      (\Delta u\otimes n_\varepsilon)\tau_i &= (\tau_i\cdot n_\varepsilon)\Delta u = 0,
    \end{aligned}
  \end{align}
  where $\overline{\nabla}_i^\varepsilon:=\overline{\nabla}_{\tau_i}^\varepsilon$, $i=1,2$.
  Writing $\tau_i^j$ and $n_\varepsilon^j$, $j=1,2,3$ for the $j$-th component of $\tau_i$ and $n_\varepsilon$, we see that the $j$-th component of $\{(n_\varepsilon\cdot\nabla)\nabla u\}^T\tau_i$ is of the form
  \begin{align*}
    \sum_{k,l=1}^3n_\varepsilon^k(\partial_k\partial_lu_j)\tau_i^l &= \sum_{k=1}^3n_\varepsilon^k(\tau_i\cdot\nabla)(\partial_ku_j) = \sum_{k=1}^3n_\varepsilon^k(\tau_i\cdot\nabla_{\Gamma_\varepsilon})(\partial_ku_j) \\
    &= \sum_{k=1}^3\{(\tau_i\cdot\nabla_{\Gamma_\varepsilon})(n_\varepsilon^k\partial_ku_j)-(\tau_i\cdot\nabla_{\Gamma_\varepsilon}n_\varepsilon^k)\partial_ku_j\} \\
    &= (\tau_i\cdot\nabla_{\Gamma_\varepsilon})\{(n_\varepsilon\cdot\nabla)u_j\}-\{(\tau_i\cdot\nabla_{\Gamma_\varepsilon})n_\varepsilon\cdot\nabla\}u_j.
  \end{align*}
  (Note that $\tau_i$ is tangential on $\Gamma_\varepsilon$ and the tangential derivatives depend only on the values of functions on $\Gamma_\varepsilon$).
  Therefore,
  \begin{align*}
    \{(n_\varepsilon\cdot\nabla)\nabla u\}^T\tau_i = (\tau_i\cdot\nabla_{\Gamma_\varepsilon})\{(n_\varepsilon\cdot\nabla)u\}-\{(\tau_i\cdot\nabla_{\Gamma_\varepsilon})n_\varepsilon\cdot\nabla\}u.
  \end{align*}
  By \eqref{Pf_LA:ND_Bo}, \eqref{E:Gauss}, $-\nabla_{\Gamma_\varepsilon}n_\varepsilon=W_\varepsilon=W_\varepsilon^T$, and
  \begin{align} \label{Pf_LA:Direc_TauN}
    (\tau_i\cdot\nabla_{\Gamma_\varepsilon})n_\varepsilon = (\nabla_{\Gamma_\varepsilon}n_\varepsilon)^T\tau_i = -W_\varepsilon\tau_i,
  \end{align}
  we further observe that
  \begin{multline} \label{Pf_LA:etai_2}
    \{(n_\varepsilon\cdot\nabla)\nabla u\}^T\tau_i =-\overline{\nabla}_i^\varepsilon(W_\varepsilon u)-\tilde{\gamma}_\varepsilon\overline{\nabla}_i^\varepsilon u+\overline{\nabla}_{W_\varepsilon\tau_i}^\varepsilon u-\xi_\varepsilon W_\varepsilon\tau_i \\
    +\{(-\tilde{\gamma}_\varepsilon W_\varepsilon u+\nabla_{\Gamma_\varepsilon}\xi_\varepsilon)\cdot\tau_i\}n_\varepsilon.
  \end{multline}
  Note that the first four terms on the right-hand side of \eqref{Pf_LA:etai_2} are tangential on $\Gamma_\varepsilon$.
  From \eqref{Pf_LA:Def_etai}, \eqref{Pf_LA:etai_1}, and \eqref{Pf_LA:etai_2} we deduce that
  \begin{multline*}
    \eta_i = -\overline{\nabla}_i^\varepsilon u\cdot\left\{\overline{\nabla}_i^\varepsilon(W_\varepsilon u)+\tilde{\gamma}_\varepsilon\overline{\nabla}_i^\varepsilon u-\overline{\nabla}_{W_\varepsilon\tau_i}^\varepsilon u+\xi_\varepsilon W_\varepsilon\tau_i\right\} \\
    +(W_\varepsilon u\cdot\tau_i)\{(-\tilde{\gamma}_\varepsilon W_\varepsilon u+\nabla_{\Gamma_\varepsilon}\xi_\varepsilon)\cdot\tau_i\}, \quad i=1,2.
  \end{multline*}
  Since $W_\varepsilon u$ and $\nabla_{\Gamma_\varepsilon}\xi_\varepsilon$ are tangential on $\Gamma_\varepsilon$ and $\{\tau_1,\tau_2\}$ is an orthonormal basis of the tangent plane of $\Gamma_\varepsilon$, by the above equality and \eqref{E:Wtr_Cov}--\eqref{E:Winn_Cov} we obtain
  \begin{multline} \label{Pf_LA:Sum_etai}
    \eta_1+\eta_2 = -\{\nabla_{\Gamma_\varepsilon}(W_\varepsilon u)+\tilde{\gamma}_\varepsilon\nabla_{\Gamma_\varepsilon}u-W_\varepsilon\nabla_{\Gamma_\varepsilon}u\}:(\nabla_{\Gamma_\varepsilon}u)P_\varepsilon \\
    -\xi_\varepsilon(\nabla_{\Gamma_\varepsilon}u:W_\varepsilon)+W_\varepsilon u\cdot(-\tilde{\gamma}_\varepsilon W_\varepsilon u+\nabla_{\Gamma_\varepsilon}\xi_\varepsilon).
  \end{multline}
  To calculate $\eta_3$ we see that the $j$-th component of $\{(n_\varepsilon\cdot\nabla)\nabla u\}^Tn_\varepsilon$ is of the form
  \begin{align*}
    \sum_{k,l=1}^3n_\varepsilon^k(\partial_k\partial_lu_j)n_\varepsilon^l &= \mathrm{tr}[Q_\varepsilon\nabla^2u_j] = \mathrm{tr}[\nabla^2u_j]-\mathrm{tr}[P_\varepsilon\nabla^2u_j] \\
    &= \Delta u_j-\sum_{i=1,2}P_\varepsilon(\nabla^2u_j)\tau_i\cdot\tau_i-P_\varepsilon(\nabla^2u)n_\varepsilon\cdot n_\varepsilon \\
    &= \Delta u_j-\sum_{i=1,2}(\tau_i\cdot\nabla)\nabla u_j\cdot\tau_i
  \end{align*}
  for $j=1,2,3$ by $P_\varepsilon^T=P_\varepsilon$, $P_\varepsilon\tau_i=\tau_i$, and $P_\varepsilon n_\varepsilon=0$.
  From this equality,
  \begin{align*}
    (\tau_i\cdot\nabla)\nabla u_j\cdot\tau_i = (\tau_i\cdot\nabla_{\Gamma_\varepsilon})\nabla u_j\cdot\tau_i = (\tau_i\cdot\nabla_{\Gamma_\varepsilon})\{(\tau_i\cdot\nabla)u_j\}-\{(\tau_i\cdot\nabla_{\Gamma_\varepsilon})\tau_i\cdot\nabla\}u_j,
  \end{align*}
  and $(\Delta u\otimes n_\varepsilon)n_\varepsilon=(n_\varepsilon\cdot n_\varepsilon)\Delta u=\Delta u$ we deduce that
  \begin{multline} \label{Pf_LA:eta3_1}
    \{(n_\varepsilon\cdot\nabla)\nabla u\}^Tn_\varepsilon-(\Delta u\otimes n_\varepsilon)n_\varepsilon \\
    = -\sum_{i=1,2}[(\tau_i\cdot\nabla_{\Gamma_\varepsilon})\{(\tau_i\cdot\nabla)u\}-\{(\tau_i\cdot\nabla_{\Gamma_\varepsilon})\tau_i\cdot\nabla\}u].
  \end{multline}
  Moreover, by \eqref{Pf_LA:ND_Bo}, \eqref{Pf_LA:Direc_TauN}, and \eqref{E:Gauss} we have
  \begin{align*}
    (\tau_i\cdot\nabla_{\Gamma_\varepsilon})\{(\tau_i\cdot\nabla)u\} &= (\tau_i\cdot\nabla_{\Gamma_\varepsilon})\left\{\overline{\nabla}_i^\varepsilon u+(W_\varepsilon u\cdot\tau_i)n_\varepsilon\right\} \\
    &= \overline{\nabla}_i^\varepsilon\overline{\nabla}_i^\varepsilon u-(W_\varepsilon u\cdot\tau_i)W_\varepsilon\tau_i \\
    &\qquad\qquad +\left\{W_\varepsilon \overline{\nabla}_i^\varepsilon u\cdot\tau_i+\tau_i\cdot\nabla_{\Gamma_\varepsilon}(W_\varepsilon u\cdot\tau_i)\right\}n_\varepsilon
  \end{align*}
  and
  \begin{align*}
    \{(\tau_i\cdot\nabla_{\Gamma_\varepsilon})\tau_i\cdot\nabla\}u &= \left[\left\{\overline{\nabla}_i^\varepsilon\tau_i+(W_\varepsilon\tau_i\cdot\tau_i)n_\varepsilon\right\}\cdot\nabla\right]u \\
    &= \overline{\nabla}_{\overline{\nabla}_i^\varepsilon\tau_i}^\varepsilon u-(W_\varepsilon\tau_i\cdot\tau_i)(W_\varepsilon u+\tilde{\gamma}_\varepsilon u) \\
    &\qquad\qquad +\left(W_\varepsilon u\cdot\overline{\nabla}_i^\varepsilon \tau_i+\xi_\varepsilon W_\varepsilon\tau_i\cdot\tau_i\right)n_\varepsilon.
  \end{align*}
  We substitute these expressions for \eqref{Pf_LA:eta3_1} and use
  \begin{align*}
    \sum_{i=1,2}(W_\varepsilon u\cdot\tau_i)W_\varepsilon\tau_i = \sum_{i=1,2}W_\varepsilon(\tau_i\otimes\tau_i)W_\varepsilon u = W_\varepsilon P_\varepsilon W_\varepsilon u = W_\varepsilon^2 u
  \end{align*}
  by $P_\varepsilon=\sum_{i=1,2}\tau_i\otimes\tau_i$ and $P_\varepsilon W_\varepsilon=W_\varepsilon$,
  \begin{align*}
    \sum_{i=1,2}\left\{\tau_i\cdot\nabla_{\Gamma_\varepsilon}(W_\varepsilon u\cdot\tau_i)- W_\varepsilon u\cdot\overline{\nabla}_i^\varepsilon\tau_i\right\} = \sum_{i=1,2}\overline{\nabla}_i^\varepsilon(W_\varepsilon u)\cdot\tau_i = \mathrm{div}_{\Gamma_\varepsilon}(W_\varepsilon u)
  \end{align*}
  by \eqref{E:RiCo_Met} and \eqref{E:Sdiv_Cov}, and the formulas \eqref{E:MC_Local} and \eqref{E:Wtr_Cov} to deduce that
  \begin{multline} \label{Pf_LA:eta3_2}
    \{(n_\varepsilon\cdot\nabla)\nabla u\}^Tn_\varepsilon-(\Delta u\otimes n_\varepsilon)n_\varepsilon \\
    = -\sum_{i=1,2}\left(\overline{\nabla}_i^\varepsilon\overline{\nabla}_i^\varepsilon u-\overline{\nabla}_{\overline{\nabla}_i^\varepsilon\tau_i}^\varepsilon u\right)+W_\varepsilon^2u-H_\varepsilon W_\varepsilon u-\tilde{\gamma}_\varepsilon H_\varepsilon u \\
    -\{\nabla_{\Gamma_\varepsilon}u:W_\varepsilon+\mathrm{div}_{\Gamma_\varepsilon}(W_\varepsilon u)-\xi_\varepsilon H_\varepsilon\}n_\varepsilon.
  \end{multline}
  Hence by \eqref{Pf_LA:ND_Bo}, \eqref{Pf_LA:Def_eta3}, and \eqref{Pf_LA:eta3_2} we get
  \begin{multline} \label{Pf_LA:eta3_Inn}
    \eta_3 = \sum_{i=1,2}\left(\overline{\nabla}_i^\varepsilon\overline{\nabla}_i^\varepsilon u-\overline{\nabla}_{\overline{\nabla}_i^\varepsilon\tau_i}^\varepsilon u\right)\cdot(W_\varepsilon u+\tilde{\gamma}_\varepsilon u) \\
    -(W_\varepsilon^2u-H_\varepsilon W_\varepsilon u-\tilde{\gamma}_\varepsilon H_\varepsilon u)\cdot(W_\varepsilon u+\tilde{\gamma}_\varepsilon u) \\
    -\xi_\varepsilon\{\nabla_{\Gamma_\varepsilon}u:W_\varepsilon+\mathrm{div}_{\Gamma_\varepsilon}(W_\varepsilon u)-\xi_\varepsilon H_\varepsilon\}.
  \end{multline}
  Now we observe by \eqref{E:PW_Bo} and direct calculations that
  \begin{align*}
    \nabla_{\Gamma_\varepsilon}u:(\nabla_{\Gamma_\varepsilon}u)P_\varepsilon = P_\varepsilon(\nabla u):P_\varepsilon(\nabla u)P_\varepsilon = \nabla u:P_\varepsilon^TP_\varepsilon(\nabla u)P_\varepsilon.
  \end{align*}
  Since $P_\varepsilon^T=P_\varepsilon^2=P_\varepsilon$, the above equality implies that
  \begin{align} \label{Pf_LA:Mat_Inn_1}
    \nabla_{\Gamma_\varepsilon}u:(\nabla_{\Gamma_\varepsilon}u)P_\varepsilon = \nabla u:P_\varepsilon(\nabla u)P_\varepsilon.
  \end{align}
  By the same calculations with \eqref{E:PW_Bo} and $W_\varepsilon^T=W_\varepsilon$ we have
  \begin{align} \label{Pf_LA:Mat_Inn_2}
    \begin{aligned}
      \nabla_{\Gamma_\varepsilon}(W_\varepsilon u):(\nabla_{\Gamma_\varepsilon}u)P_\varepsilon &= \{\nabla_{\Gamma_\varepsilon}W_\varepsilon\cdot u+(\nabla_{\Gamma_\varepsilon}u)W_\varepsilon\}:(\nabla_{\Gamma_\varepsilon}u)P_\varepsilon \\
      &= \{\nabla_{\Gamma_\varepsilon}W_\varepsilon\cdot u+(\nabla u)W_\varepsilon\}:P_\varepsilon(\nabla u)P_\varepsilon,
    \end{aligned}
  \end{align}
  where the matrix $\nabla_{\Gamma_\varepsilon}W_\varepsilon\cdot u$ is given by \eqref{Pf_LA:DW_u_Bo}, and
  \begin{align} \label{Pf_LA:Mat_Inn_3}
    W_\varepsilon(\nabla_{\Gamma_\varepsilon}u):(\nabla_{\Gamma_\varepsilon}u)P_\varepsilon = W_\varepsilon(\nabla u):(\nabla u)P_\varepsilon, \quad \nabla_{\Gamma_\varepsilon}u:W_\varepsilon = \nabla u:W_\varepsilon.
  \end{align}
  Also, it is easy to see that
  \begin{align} \label{Pf_LA:Div_Xi}
    \begin{aligned}
      W_\varepsilon u\cdot\nabla_{\Gamma_\varepsilon}\xi_\varepsilon &= \mathrm{div}_{\Gamma_\varepsilon}(\xi_\varepsilon W_\varepsilon u)-\xi_\varepsilon\mathrm{div}_{\Gamma_\varepsilon}(W_\varepsilon u), \\
      \mathrm{div}_{\Gamma_\varepsilon}(W_\varepsilon u) &= u\cdot\mathrm{div}_{\Gamma_\varepsilon}W_\varepsilon+\nabla_{\Gamma_\varepsilon}u: W_\varepsilon =  u\cdot\mathrm{div}_{\Gamma_\varepsilon}W_\varepsilon+\nabla u: W_\varepsilon.
    \end{aligned}
  \end{align}
  Hence we deduce from \eqref{Pf_LA:Decom_Intg}, \eqref{Pf_LA:Sum_etai}, \eqref{Pf_LA:eta3_Inn}--\eqref{Pf_LA:Div_Xi}, and $W_\varepsilon^T=W_\varepsilon$ that
  \begin{multline*}
    \int_{\Gamma_\varepsilon}\nabla u:\{(n_\varepsilon\cdot\nabla)\nabla u-n_\varepsilon\otimes\Delta u\}\,d\mathcal{H}^2 \\
    = \sum_{i=1,2}\int_{\Gamma_\varepsilon}\left(\overline{\nabla}_i^\varepsilon\overline{\nabla}_i^\varepsilon u-\overline{\nabla}_{\overline{\nabla}_i^\varepsilon\tau_i}^\varepsilon u\right)\cdot(W_\varepsilon u+\tilde{\gamma}_\varepsilon u)\,d\mathcal{H}^2\\
    +\int_{\Gamma_\varepsilon}\left(\frac{1}{2}\varphi_1+\sum_{k=2}^4\varphi_k\right)d\mathcal{H}^2+\int_{\Gamma_\varepsilon}\mathrm{div}_{\Gamma_\varepsilon}(\xi_\varepsilon W_\varepsilon u)\,d\mathcal{H}^2,
  \end{multline*}
  where $\varphi_1$, $\dots$, $\varphi_4$ are given by \eqref{Pf_LA:Def_Phi}.
  Here the last term vanishes by the Stokes theorem since $\xi_\varepsilon W_\varepsilon u$ is tangential on $\Gamma_\varepsilon$.
  Moreover, applying \eqref{E:IbP_Cov} to the first term on the right-hand side and then use \eqref{E:InnP_Cov}, \eqref{Pf_LA:Mat_Inn_1}, and \eqref{Pf_LA:Mat_Inn_2} we obtain
  \begin{align*}
    \sum_{i=1,2}\int_{\Gamma_\varepsilon}\left(\overline{\nabla}_i^\varepsilon\overline{\nabla}_i^\varepsilon u-\overline{\nabla}_{\overline{\nabla}_i^\varepsilon\tau_i}^\varepsilon u\right)\cdot(W_\varepsilon u+\tilde{\gamma}_\varepsilon u)\,d\mathcal{H}^2 = \frac{1}{2}\int_{\Gamma_\varepsilon}\varphi_1\,d\mathcal{H}^2.
  \end{align*}
  Hence the equality \eqref{Pf_LA:Int_Sum} follows.

  The second step is to show that
  \begin{align}
    \left|\int_{\Gamma_\varepsilon}\varphi_k\,d\mathcal{H}^2\right| &\leq c\left(\|u\|_{H^1(\Omega_\varepsilon)}^2+\|u\|_{H^1(\Omega_\varepsilon)}\|\nabla^2u\|_{L^2(\Omega_\varepsilon)}\right), \quad k=1,2, \label{Pf_LA:Est_Phi12}\\
    \left|\int_{\Gamma_\varepsilon}\varphi_k\,d\mathcal{H}^2\right| &\leq c\|u\|_{H^1(\Omega_\varepsilon)}^2, \quad k=3,4 \label{Pf_LA:Est_Phi34}
  \end{align}
  with a constant $c>0$ independent of $\varepsilon$.
  The estimate \eqref{Pf_LA:Est_Phi34} for $k=4$ is an easy consequence of \eqref{E:Fric_Upper}, \eqref{E:Poin_Bo}, and the uniform boundedness in $\varepsilon$ of $W_\varepsilon$ and $H_\varepsilon$ on $\Gamma_\varepsilon$:
  \begin{align*}
    \left|\int_{\Gamma_\varepsilon}\varphi_4\,d\mathcal{H}^2\right| \leq c\varepsilon\|u\|_{L^2(\Gamma_\varepsilon)}^2 \leq c\varepsilon(\varepsilon^{-1}\|u\|_{L^2(\Omega_\varepsilon)}^2+\varepsilon\|\partial_nu\|_{L^2(\Omega_\varepsilon)}^2) \leq c\|u\|_{H^1(\Omega_\varepsilon)}^2.
  \end{align*}
  Let us prove the estimate \eqref{Pf_LA:Est_Phi12} for $k=1$.
  We proceed as in the proof of Lemma~\ref{L:Korn_Grad}.
  In what follows, we use the notations \eqref{E:Pull_Dom} and \eqref{E:Pull_Bo} and sometimes suppress the arguments $y$ and $r$.
  For $y\in\Gamma$, $r\in[\varepsilon g_0(y),\varepsilon g_1(y)]$, and $j,k,l=1,2,3$ we set
  \begin{align*}
    F(y,r) &:= \frac{1}{\varepsilon g(y)}\bigl\{\bigl(r-\varepsilon g_0(y)\bigr)W_{\varepsilon,1}^\sharp(y)-\bigl(\varepsilon g_1(y)-r\bigr)W_{\varepsilon,0}^\sharp(y)\bigr\}, \\
    G_{jk}^l(y,r) &:= \frac{1}{\varepsilon g(y)}\bigl\{\bigl(r-\varepsilon g_0(y)\bigr)\bigl(\underline{D}_j^\varepsilon[W_\varepsilon]_{kl}\bigr)_1^\sharp(y)-\bigl(\varepsilon g_1(y)-r\bigr)\bigl(\underline{D}_j^\varepsilon[W_\varepsilon]_{kl}\bigr)_0^\sharp(y)\bigr\}, \\
    \tilde{\gamma}(y,r) &:= \frac{1}{\varepsilon g(y)}\bigl\{\bigl(r-\varepsilon g_0(y)\bigr)\tilde{\gamma}_\varepsilon^1-\bigl(\varepsilon g_1(y)-r\bigr)\tilde{\gamma}_\varepsilon^0\bigr\},
  \end{align*}
  where $\tilde{\gamma}_\varepsilon^i:=\gamma_\varepsilon^i/\nu$, $i=0,1$.
  Then we have
  \begin{multline} \label{Pf_LA:Aux_Eq_1}
    [\nabla_{\Gamma_\varepsilon}W_\varepsilon\cdot u+(\nabla u)W_\varepsilon+\tilde{\gamma}_\varepsilon\nabla u]_i^\sharp(y) \\
    = (-1)^{i+1}[G\cdot u^\sharp+(\nabla u)^\sharp F+\tilde{\gamma}(\nabla u)^\sharp](y,\varepsilon g_i(y)), \quad y\in\Gamma,\,i=0,1,
  \end{multline}
  where $G\cdot u^\sharp$ denotes a $3\times 3$ matrix whose $(j,k)$-entry is given by
  \begin{align*}
    [G\cdot u^\sharp]_{jk} := \sum_{l=1}^3G_{jk}^lu_l^\sharp, \quad j,k=1,2,3.
  \end{align*}
  Moreover, by \eqref{E:Fric_Upper}, \eqref{E:Diff_WH_IO} for $W_\varepsilon$ and $\underline{D}_j^\varepsilon W_\varepsilon$ with $j=1,2,3$,
  \begin{align} \label{Pf_LA:Est_Dist}
    |r-\varepsilon g_i(y)| \leq \varepsilon g(y) \leq c\varepsilon, \quad y\in\Gamma,\,r\in[\varepsilon g_0(y),\varepsilon g_1(y)],\,i=0,1,
  \end{align}
  and the uniform boundedness in $\varepsilon$ of $W_\varepsilon$ and $\underline{D}_j^\varepsilon W_\varepsilon$ on $\Gamma_\varepsilon$ we have
  \begin{align} \label{Pf_LA:Aux_Est_1}
    |\eta(y,r)|+\left|\frac{\partial \eta}{\partial r}(y,r)\right| \leq c, \quad \eta = F,G_{jk}^l,\tilde{\gamma}, \quad y\in\Gamma,\,r\in[\varepsilon g_0(y),\varepsilon g_1(y)]
  \end{align}
  with a constant $c>0$ independent of $\varepsilon$.
  We also define
  \begin{align*}
    R(y,r) := \frac{1}{\varepsilon g(y)}\bigl\{\bigl(r-\varepsilon g_0(y)\bigr)P_{\varepsilon,1}^\sharp(y)+\bigl(\varepsilon g_1(y)-r\bigr)P_{\varepsilon,0}^\sharp(y)\bigr\}
  \end{align*}
  for $y\in\Gamma$ and $r\in[\varepsilon g_0(y),\varepsilon g_1(y)]$, and
  \begin{align*}
    S_i(y) &:= \sqrt{1+\varepsilon^2|\tau_\varepsilon^i(y)|^2}\,P_{\varepsilon,i}^\sharp(y), \quad i=0,1, \\
    S(y,r) &:= \frac{1}{\varepsilon g(y)}\bigl\{\bigl(r-\varepsilon g_0(y)\bigr)S_1(y)+\bigl(\varepsilon g_1(y)-r\bigr)S_0(y)\bigr\},
  \end{align*}
  where $\tau_\varepsilon^0$ and $\tau_\varepsilon^1$ are given by \eqref{E:Def_NB_Aux}.
  Then
  \begin{align} \label{Pf_LA:Aux_Eq_2}
    \sqrt{1+\varepsilon^2|\tau_\varepsilon^i(y)|^2}\,[P_\varepsilon(\nabla u)P_\varepsilon]_i^\sharp(y) = [R(\nabla u)^\sharp S](y,\varepsilon g_i(y)), \quad y\in\Gamma,\,i=0,1.
  \end{align}
  Moreover, from \eqref{E:Diff_PQ_IO} for $P_\varepsilon$, \eqref{Pf_KG:Sqrt}, and \eqref{Pf_LA:Est_Dist} we deduce that
  \begin{align} \label{Pf_LA:Aux_Est_2}
    |\eta(y,r)|+\left|\frac{\partial\eta}{\partial r}(y,r)\right| \leq c, \quad \eta = R,S, \quad y\in\Gamma,\, r\in[\varepsilon g_0(y),\varepsilon g_1(y)].
  \end{align}
  Now we define a function $\Phi_1=\Phi_1(y,r)$ for $y\in\Gamma$ and $r\in[\varepsilon g_0(y),\varepsilon g_1(y)]$ by
  \begin{align*}
    \Phi_1(y,r) := -2[\{G\cdot u^\sharp+(\nabla u)^\sharp F+\tilde{\gamma}(\nabla u)^\sharp\}:R(\nabla u)S](y,r)J(y,r),
  \end{align*}
  where $J$ is given by \eqref{E:Def_Jac}.
  Then by \eqref{E:CoV_Surf}, \eqref{Pf_LA:Aux_Eq_1}, and \eqref{Pf_LA:Aux_Eq_2} we have
  \begin{align*}
    \int_{\Gamma_\varepsilon}\varphi_1(x)\,d\mathcal{H}^2(x) &= \sum_{i=0,1}\int_{\Gamma_\varepsilon^i}\varphi_1(x)\,d\mathcal{H}^2(x) \\
    &= \int_\Gamma\{\Phi_1(y,\varepsilon g_1(y))-\Phi_1(y,\varepsilon g_0(y))\}\,d\mathcal{H}^2(y) \\
    &= \int_\Gamma\int_{\varepsilon g_0(y)}^{\varepsilon g_1(y)}\frac{\partial\Phi_1}{\partial r}(y,r)\,dr\,d\mathcal{H}^2(y).
  \end{align*}
  Furthermore, the inequalities \eqref{E:Jac_Bound}, \eqref{Pf_LA:Aux_Est_1}, and \eqref{Pf_LA:Aux_Est_2} imply that
  \begin{align*}
    \left|\frac{\partial\Phi_1}{\partial r}\right| \leq c\{|u^\sharp|^2+|(\nabla u)^\sharp|^2+(|u^\sharp|+|(\nabla u)^\sharp|)|(\nabla^2u)^\sharp|\}
  \end{align*}
  with some constant $c>0$ independent of $\varepsilon$ (here we also used Young's inequality).
  From the above relations, \eqref{E:CoV_Equiv}, and H\"{o}lder's inequality it follows that
  \begin{align*}
    \left|\int_{\Gamma_\varepsilon}\varphi_1\,d\mathcal{H}^2\right| &\leq c\int_\Gamma\int_{\varepsilon g_0}^{\varepsilon g_1}\{|u^\sharp|^2+|(\nabla u)^\sharp|^2+(|u^\sharp|+|(\nabla u)^\sharp|)|(\nabla^2u)^\sharp|\}\,dr\,d\mathcal{H}^2 \\
    &\leq c\left(\|u\|_{H^1(\Omega_\varepsilon)}^2+\|u\|_{H^1(\Omega_\varepsilon)}\|\nabla^2u\|_{L^2(\Omega_\varepsilon)}\right).
  \end{align*}
  (Note that $\|u\|_{L^2(\Omega_\varepsilon)}\leq\|u\|_{H^1(\Omega_\varepsilon)}$.)
  Thus the inequality \eqref{Pf_LA:Est_Phi12} for $k=1$ holds.
  By the same arguments we can show \eqref{Pf_LA:Est_Phi12} for $k=2$ and \eqref{Pf_LA:Est_Phi34} for $k=3$.

  Finally, from \eqref{Pf_LA:Int_Sum}, \eqref{Pf_LA:Est_Phi12}, and \eqref{Pf_LA:Est_Phi34} we deduce that
  \begin{align*}
    \left|\int_{\Gamma_\varepsilon}\nabla u:\{(n_\varepsilon\cdot\nabla)\nabla u-n_\varepsilon\otimes\Delta u\}\,d\mathcal{H}^2\right| &\leq c\left(\|u\|_{H^1(\Omega_\varepsilon)}^2+\|u\|_{H^1(\Omega_\varepsilon)}\|\nabla^2u\|_{L^2(\Omega_\varepsilon)}\right).
  \end{align*}
  We apply this inequality to \eqref{Pf_LA:IbP} and then use Young's inequality to obtain
  \begin{align*}
    \|\nabla^2u\|_{L^2(\Omega_\varepsilon)}^2 &\leq \|\Delta u\|_{L^2(\Omega_\varepsilon)}^2+c\left(\|u\|_{H^1(\Omega_\varepsilon)}^2+\|u\|_{H^1(\Omega_\varepsilon)}\|\nabla^2u\|_{L^2(\Omega_\varepsilon)}\right) \\
    &\leq \|\Delta u\|_{L^2(\Omega_\varepsilon)}^2+\frac{1}{2}\|\nabla^2u\|_{L^2(\Omega_\varepsilon)}^2+c\|u\|_{H^1(\Omega_\varepsilon)}^2,
  \end{align*}
  which yields \eqref{Pf_LA:Goal}.
  Hence the inequality \eqref{E:Lap_Apri} is valid.
\end{proof}

\section{Average operators in the thin direction} \label{S:Ave}

\subsection{Definition and basic inequalities of the average operators} \label{SS:Ave_Def}
In this section we investigate average operators in the thin direction and establish several inequalities related to them, which are useful in the analysis of the Navier--Stokes equations \eqref{E:NS_Eq}--\eqref{E:NS_In}.
Throughout this section we assume $\varepsilon\in(0,1)$ and denote by $\bar{\eta}=\eta\circ\pi$ the constant extension of a function $\eta$ on $\Gamma$ in the normal direction of $\Gamma$.

\begin{definition} \label{D:Average}
  We define the average operator $M$ as
  \begin{align} \label{E:Def_Ave}
    M\varphi(y) := \frac{1}{\varepsilon g(y)}\int_{\varepsilon g_0(y)}^{\varepsilon g_1(y)}\varphi(y+rn(y))\,dr, \quad y\in\Gamma
  \end{align}
  for a function $\varphi$ on $\Omega_\varepsilon$.
  The operator $M$ is also applied to a vector field $u\colon\Omega_\varepsilon\to\mathbb{R}^3$ and we define the averaged tangential component $M_\tau u$ of $u$ by
  \begin{align} \label{E:Def_Tan_Ave}
    M_\tau u(y) := P(y)Mu(y) = \frac{1}{\varepsilon g(y)}\int_{\varepsilon g_0(y)}^{\varepsilon g_1(y)}P(y)u(y+rn(y))\,dr, \quad y\in\Gamma.
  \end{align}
\end{definition}

For the sake of simplicity, we denote the tangential and normal components (with respect to the surface $\Gamma$) of a vector field $u\colon\Omega_\varepsilon\to\mathbb{R}^3$ by
\begin{align} \label{E:Def_U_TN}
  u_\tau(x) := \overline{P}(x)u(x), \quad u_n(x) := \{u(x)\cdot\bar{n}(x)\}\bar{n}(x), \quad x\in\Omega_\varepsilon
\end{align}
so that $u=u_\tau+u_n$ and $u_\tau\cdot u_n=0$ (note that $u_n$ is a vector field).
Also, we use the notations \eqref{E:Pull_Dom} and \eqref{E:Pull_Bo} and sometimes suppress the arguments of functions.
For example, we write
\begin{align*}
  M\varphi = \frac{1}{\varepsilon g}\int_{\varepsilon g_0}^{\varepsilon g_1}\varphi^\sharp\,dr, \quad M_\tau u = \frac{1}{\varepsilon g}\int_{\varepsilon g_0}^{\varepsilon g_1}u_\tau^\sharp\,dr.
\end{align*}
Let us derive basic inequalities for $M$ and $M_\tau$.
The following results for $M$ are also valid for $M_\tau$ since $M_\tau u=Mu_\tau$ on $\Gamma$ for $u\colon\Omega_\varepsilon\to\mathbb{R}^3$ by \eqref{E:Def_Tan_Ave} and \eqref{E:Def_U_TN}.

\begin{lemma} \label{L:Ave_Lp}
  Let $p\in[1,\infty)$.
  There exists $c>0$ independent of $\varepsilon$ such that
  \begin{align}
    \|M\varphi\|_{L^p(\Gamma)} &\leq c\varepsilon^{-1/p}\|\varphi\|_{L^p(\Omega_\varepsilon)}, \label{E:Ave_Lp_Surf} \\
    \left\|\overline{M\varphi}\right\|_{L^p(\Omega_\varepsilon)} &\leq c\|\varphi\|_{L^p(\Omega_\varepsilon)} \label{E:Ave_Lp_Dom}
  \end{align}
  for all $\varphi\in L^p(\Omega_\varepsilon)$.
\end{lemma}

\begin{proof}
  By H\"{o}lder's inequality and \eqref{E:Width_Bound},
  \begin{align*}
    |M\varphi(y)|^p = \left|\frac{1}{\varepsilon g(y)}\int_{\varepsilon g_0(y)}^{\varepsilon g_1(y)}\varphi^\sharp(y,r)\,dr\right|^p \leq c\varepsilon^{-1}\int_{\varepsilon g_0(y)}^{\varepsilon g_1(y)}|\varphi^\sharp(y,r)|^p\,dr
  \end{align*}
  for all $y\in\Gamma$.
  Integrating both sides of the above inequality over $\Gamma$ and using \eqref{E:CoV_Equiv} we obtain \eqref{E:Ave_Lp_Surf}.
  The inequality \eqref{E:Ave_Lp_Dom} follows from \eqref{E:Con_Lp} and \eqref{E:Ave_Lp_Surf}.
\end{proof}

\begin{lemma} \label{L:Ave_Diff}
  Let $p\in[1,\infty)$.
  There exists $c>0$ independent of $\varepsilon$ such that
  \begin{align}
    \left\|\varphi-\overline{M\varphi}\right\|_{L^p(\Omega_\varepsilon)} &\leq c\varepsilon\|\partial_n\varphi\|_{L^p(\Omega_\varepsilon)}, \label{E:Ave_Diff_Dom}\\
    \left\|\varphi-\overline{M\varphi}\right\|_{L^p(\Gamma_\varepsilon^i)} &\leq c\varepsilon^{1-1/p}\|\partial_n\varphi\|_{L^p(\Omega_\varepsilon)}, \quad i=0,1 \label{E:Ave_Diff_Bo}
  \end{align}
  for all $\varphi\in W^{1,p}(\Omega_\varepsilon)$, where $\partial_n\varphi$ is the normal derivative of $\varphi$ given by \eqref{E:Def_NorDer}.
\end{lemma}

\begin{proof}
  For $y\in\Gamma$ and $r\in(\varepsilon g_0(y),\varepsilon g_1(y))$ we have
  \begin{align} \label{Pf_ADiff:Diff}
    \varphi^\sharp(y,r)-M\varphi(y) &= \frac{1}{\varepsilon g(y)}\int_{\varepsilon g_0(y)}^{\varepsilon g_1(y)}\{\varphi^\sharp(y,r)-\varphi^\sharp(y,r_1)\}\,dr_1.
  \end{align}
  Since $\partial\varphi^\sharp/\partial r=(\partial_n\varphi)^\sharp$ by \eqref{E:Pull_Dom} and \eqref{E:Def_NorDer},
  \begin{align*}
    \begin{aligned}
      |\varphi^\sharp(y,r)-\varphi^\sharp(y,r_1)| &= \left|\int_{r_1}^r\frac{\partial}{\partial r_2}\bigl(\varphi^\sharp(y,r_2)\bigr)\,dr_2\right| \leq \int_{\varepsilon g_0(y)}^{\varepsilon g_1(y)}|(\partial_n\varphi)^\sharp(y,r_2)|\,dr_2.
    \end{aligned}
  \end{align*}
  Noting that the right-hand side is independent of $r_1$, we apply this inequality to the right-hand side of \eqref{Pf_ADiff:Diff} and then use H\"{o}lder's inequality to get
  \begin{align*}
    |\varphi^\sharp(y,r)-M\varphi(y)| &\leq \int_{\varepsilon g_0(y)}^{\varepsilon g_1(y)}|(\partial_n\varphi)^\sharp(y,r_2)|\,dr_2 \\
    &\leq c\varepsilon^{1-1/p}\left(\int_{\varepsilon g_0(y)}^{\varepsilon g_1(y)}|(\partial_n\varphi)^\sharp(y,r_2)|^p\,dr_2\right)^{1/p}
  \end{align*}
  for all $y\in\Gamma$ and $r\in(\varepsilon g_0(y),\varepsilon g_1(y))$.
  Since the last term is independent of $r$, this inequality and \eqref{E:CoV_Equiv} imply that
  \begin{align*}
    \left\|\varphi-\overline{M\varphi}\right\|_{L^p(\Omega_\varepsilon)}^p &\leq c\int_\Gamma\int_{\varepsilon g_0(y)}^{\varepsilon g_1(y)}|\varphi^\sharp(y,r)-M\varphi(y)|^p\,dr\,d\mathcal{H}^2(y) \\
    &\leq c\int_\Gamma\varepsilon g(y)\left(c\varepsilon^{p-1}\int_{\varepsilon g_0(y)}^{\varepsilon g_1(y)}|(\partial_n\varphi)^\sharp(y,r_2)|^p\,dr_2\right)d\mathcal{H}^2(y) \\
    &\leq c\varepsilon^p\|\partial_n\varphi\|_{L^p(\Omega_\varepsilon)}^p.
  \end{align*}
  Hence the inequality \eqref{E:Ave_Diff_Dom} is valid.
  Applying \eqref{E:Poin_Bo} to $\varphi-\overline{M\varphi}$ and using \eqref{E:NorDer_Con} and \eqref{E:Ave_Diff_Dom} we also get \eqref{E:Ave_Diff_Bo}.
\end{proof}

\begin{lemma} \label{L:Ave_N_Lp}
  Let $p\in[1,\infty)$.
  There exists $c>0$ independent of $\varepsilon$ such that
  \begin{align} \label{E:Ave_N_Lp}
    \|Mu\cdot n\|_{L^p(\Gamma)} \leq c\varepsilon^{1-1/p}\|u\|_{W^{1,p}(\Omega_\varepsilon)}
  \end{align}
  for all $u\in W^{1,p}(\Omega_\varepsilon)^3$ satisfying \eqref{E:Bo_Imp} on $\Gamma_\varepsilon^0$ or on $\Gamma_\varepsilon^1$.
\end{lemma}

\begin{proof}
  Applying \eqref{E:Ave_Lp_Surf} to $Mu\cdot n=M(u\cdot\bar{n})$ and using \eqref{E:Poin_Nor} we get \eqref{E:Ave_N_Lp}.
\end{proof}

Unlike the case of a flat thin domain (see~\cite{Ho10,HoSe10,IfRaSe07}), the constant extension of the average operator on $L^2(\Omega_\varepsilon)$ is not symmetric since the Jacobian $J\not\equiv1$ appears in the change of variables formula \eqref{E:CoV_Dom}.
However, its skew-symmetric part is small of order $\varepsilon$.

\begin{lemma} \label{L:Ave_Inner}
  There exists a constant $c>0$ independent of $\varepsilon$ such that
  \begin{align} \label{E:Ave_Inner}
    \left|\Bigl(\overline{M\varphi},\xi\Bigr)_{L^2(\Omega_\varepsilon)}-\Bigl(\varphi,\overline{M\xi}\Bigr)_{L^2(\Omega_\varepsilon)}\right| \leq c\varepsilon\|\varphi\|_{L^2(\Omega_\varepsilon)}\|\xi\|_{L^2(\Omega_\varepsilon)}
  \end{align}
  for all $\varphi,\xi\in L^2(\Omega_\varepsilon)$.
\end{lemma}

\begin{proof}
  By \eqref{E:CoV_Dom} and \eqref{E:Def_Ave} we have
  \begin{align*}
    \Bigl(\overline{M\varphi},\xi\Bigr)_{L^2(\Omega_\varepsilon)} &= \int_\Gamma M\varphi\left(\int_{\varepsilon g_0}^{\varepsilon g_1}\xi^\sharp J\,dr\right)d\mathcal{H}^2 \\
    &= \varepsilon(M\varphi,gM\xi)_{L^2(\Gamma)}+\int_\Gamma M\varphi\left(\int_{\varepsilon g_0}^{\varepsilon g_1}\xi^\sharp(J-1)\,dr\right)d\mathcal{H}^2
  \end{align*}
  and
  \begin{align*}
    \Bigl(\varphi,\overline{M\xi}\Bigr)_{L^2(\Omega_\varepsilon)} = \varepsilon(gM\varphi,M\xi)_{L^2(\Gamma)}+\int_\Gamma\left(\int_{\varepsilon g_0}^{\varepsilon g_1}\varphi^\sharp(J-1)\,dr\right)M\xi\,d\mathcal{H}^2.
  \end{align*}
  Since $(M\varphi,gM\xi)_{L^2(\Gamma)}=(gM\varphi,M\xi)_{L^2(\Gamma)}$, we see by \eqref{E:Jac_Diff} that
  \begin{multline} \label{Pf_AI:Est}
    \left|\Bigl(\overline{M\varphi},\xi\Bigr)_{L^2(\Omega_\varepsilon)}-\Bigl(\varphi,\overline{M\xi}\Bigr)_{L^2(\Omega_\varepsilon)}\right| \\
    \leq c\varepsilon\left\{\int_\Gamma|M\varphi|\left(\int_{\varepsilon g_0}^{\varepsilon g_1}|\xi^\sharp|\,dr\right)d\mathcal{H}^2+\int_\Gamma\left(\int_{\varepsilon g_0}^{\varepsilon g_1}|\varphi^\sharp|\,dr\right)|M\xi|\,d\mathcal{H}^2\right\}.
  \end{multline}
  Moreover, applying H\"{o}lder's inequality twice and using \eqref{E:CoV_Equiv} and \eqref{E:Ave_Lp_Surf} we get
  \begin{align*}
    \int_\Gamma|M\varphi|\left(\int_{\varepsilon g_0}^{\varepsilon g_1}|\xi^\sharp|\,dr\right)d\mathcal{H}^2 &\leq \|M\varphi\|_{L^2(\Gamma)}\left\{\int_\Gamma\varepsilon g\left(\int_{\varepsilon g_0}^{\varepsilon g_1}|\xi^\sharp|^2\,dr\right)d\mathcal{H}^2\right\}^{1/2} \\
    &\leq c\|\varphi\|_{L^2(\Omega_\varepsilon)}\|\xi\|_{L^2(\Omega_\varepsilon)}
  \end{align*}
  and a similar inequality for the last term of \eqref{Pf_AI:Est}.
  Hence \eqref{E:Ave_Inner} follows.
\end{proof}

Now let us consider the time derivative of the average operator.

\begin{lemma} \label{L:Ave_Dt}
  Let $\varphi\in H^1(0,T;L^2(\Omega_\varepsilon))$, $T>0$.
  Then
  \begin{align*}
    \partial_tM\varphi = M(\partial_t\varphi) \in L^2(0,T;L^2(\Gamma))
  \end{align*}
  and there exists a constant $c>0$ independent of $\varepsilon$ and $\varphi$ such that
  \begin{align} \label{E:Ave_Dt}
    \|\partial_tM\varphi\|_{L^2(0,T;L^2(\Gamma))} \leq c\varepsilon^{-1/2}\|\partial_t\varphi\|_{L^2(0,T;L^2(\Omega_\varepsilon))}.
  \end{align}
\end{lemma}

\begin{proof}
  First note that $M(\partial_t\varphi)\in L^2(0,T;L^2(\Gamma))$ by $\partial_t\varphi\in L^2(0,T;L^2(\Omega_\varepsilon))$ and \eqref{E:Ave_Lp_Surf}.
  The relation $\partial_tM\varphi=M(\partial_t\varphi)$ is formally trivial since $g_0$, $g_1$, and the surface quantities on $\Gamma$ are independent of time.
  To verify it rigorously we prove
  \begin{align*}
    \int_0^T\partial_t\xi(t)M\varphi(t)\,dt = -\int_0^T\xi(t)[M(\partial_t\varphi)](t)\,dt \quad\text{in}\quad L^2(\Gamma)
  \end{align*}
  for all $\xi\in C_c^\infty(0,T)$.
  Since $L^2(\Gamma)$ is a Hilbert space, this is equivalent to
  \begin{align} \label{Pf_DtA:Goal}
    \int_0^T\partial_t\xi(t)(M\varphi(t),\eta)_{L^2(\Gamma)}\,dt = -\int_0^T\xi(t)([M(\partial_t\varphi)](t),\eta)_{L^2(\Gamma)}\,dt
  \end{align}
  for all $\xi\in C_c^\infty(0,T)$ and $\eta\in L^2(\Gamma)$.
  We define a function $\tilde{\eta}$ on $\Omega_\varepsilon$ by
  \begin{align*}
    \tilde{\eta}(x) := \frac{\eta(\pi(x))}{\varepsilon g(\pi(x))J(\pi(x),d(x))}, \quad x\in\Omega_\varepsilon.
  \end{align*}
  Then $\tilde{\eta}\in L^2(\Omega_\varepsilon)$ by \eqref{E:Width_Bound}, \eqref{E:Jac_Bound}, and \eqref{E:Con_Lp}.
  Also, by \eqref{E:CoV_Dom} and \eqref{E:Def_Ave} we have
  \begin{align} \label{Pf_DtA:CoV}
    (\varphi(t),\tilde{\eta})_{L^2(\Omega_\varepsilon)} = \int_\Gamma\left(\frac{1}{\varepsilon g}\int_{\varepsilon g_0}^{\varepsilon g_1}\varphi^\sharp(t)\,dr\right)\eta\,d\mathcal{H}^2 = (M\varphi(t),\eta)_{L^2(\Gamma)}
  \end{align}
  for a.a. $t\in(0,T)$.
  Hence
  \begin{align} \label{Pf_DtA:Ave}
    \int_0^T\partial_t\xi(t)(M\varphi(t),\eta)_{L^2(\Gamma)}\,dt = \int_0^T\partial_t\xi(t)(\varphi(t),\tilde{\eta})_{L^2(\Omega_\varepsilon)}\,dt.
  \end{align}
  Moreover, since $\varphi\in H^1(0,T;L^2(\Omega_\varepsilon))$ and $\tilde{\eta}\in L^2(\Omega_\varepsilon)$ is independent of time,
  \begin{align*}
    \int_0^T\partial_t\xi(t)(\varphi(t),\tilde{\eta})_{L^2(\Omega_\varepsilon)}\,dt = -\int_0^T\xi(t)(\partial_t\varphi(t),\tilde{\eta})_{L^2(\Omega_\varepsilon)}\,dt.
  \end{align*}
  Applying this equality to the right-hand side of \eqref{Pf_DtA:Ave} and using \eqref{Pf_DtA:CoV} with $\varphi(t)$ replaced by $\partial_t\varphi(t)$ we obtain \eqref{Pf_DtA:Goal}.
  Hence the relation $\partial_tM\varphi=M(\partial_t\varphi)$ is valid and the inequality \eqref{E:Ave_Dt} follows from \eqref{E:Ave_Lp_Surf}.
\end{proof}

\subsection{Tangential derivatives of the average operators} \label{SS:Ave_TD}
Let us give several formulas and inequalities for the tangential derivatives of the average operators.

\begin{lemma} \label{L:Ave_Der}
  For $\varphi\in C^1(\Omega_\varepsilon)$ we have
  \begin{align} \label{E:Ave_Der}
    \nabla_\Gamma M\varphi = M(B\nabla\varphi)+M\bigl((\partial_n\varphi)\psi_\varepsilon\bigr) \quad\text{on}\quad \Gamma,
  \end{align}
  where the matrix-valued function $B$ and the vector field $\psi_\varepsilon$ are given by
  \begin{align} \label{E:Ave_Der_Aux}
    \begin{aligned}
      B(x) &:= \left\{I_3-d(x)\overline{W}(x)\right\}\overline{P}(x), \\
      \psi_\varepsilon(x) &:= \frac{1}{\bar{g}(x)}\left\{\bigl(d(x)-\varepsilon\bar{g}_0(x)\bigr)\overline{\nabla_\Gamma g_1}(x)+\bigl(\varepsilon\bar{g}_1(x)-d(x)\bigr)\overline{\nabla_\Gamma g_0}(x)\right\}
    \end{aligned}
  \end{align}
  for $x\in N$.
\end{lemma}

\begin{proof}
  The constant extension of $M\varphi$ is given by
  \begin{align*}
    \overline{M\varphi}(x) = \frac{1}{\varepsilon\bar{g}(x)}\int_{\varepsilon\bar{g}_0(x)}^{\varepsilon\bar{g}_1(x)}\varphi(\pi(x)+r\bar{n}(x))\,dr, \quad x \in N.
  \end{align*}
  We differentiate both sides of this equality with respect to $x\in N$ and set $x=y\in\Gamma$.
  Then by \eqref{E:ConDer_Surf}, \eqref{E:Form_W}, \eqref{E:Pi_Der}, and \eqref{E:Nor_Grad} with $d(y)=0$ we get
  \begin{align} \label{Pf_ADer:Der}
    \nabla_\Gamma M\varphi(y) &= \frac{I(y)}{\varepsilon g(y)}+\frac{1}{\varepsilon g(y)}\int_{\varepsilon g_0(y)}^{\varepsilon g_1(y)}\{I_3-rW(y)\}P(y)(\nabla\varphi)^\sharp(y,r)\,dr
  \end{align}
  for $y\in\Gamma$.
  Here we use the notations \eqref{E:Pull_Dom} and \eqref{E:Pull_Bo} and set
  \begin{align*}
    I(y) := -\frac{\nabla_\Gamma g(y)}{g(y)}\int_{\varepsilon g_0(y)}^{\varepsilon g_1(y)}\varphi^\sharp(y,r)\,dr+\varepsilon\varphi_1^\sharp(y)\nabla_\Gamma g_1(y)-\varepsilon\varphi_0^\sharp(y)\nabla_\Gamma g_0(y).
  \end{align*}
  To the right-hand side we apply
  \begin{align*}
    \varepsilon\varphi_1^\sharp(y)\nabla_\Gamma g_1(y)-\varepsilon\varphi_0^\sharp(y)\nabla_\Gamma g_0(y) &= \Bigl[(\varphi\psi_\varepsilon)^\sharp(y,r)\Bigr]_{r=\varepsilon g_0(y)}^{\varepsilon g_1(y)} \\
    &= \int_{\varepsilon g_0(y)}^{\varepsilon g_1(y)}\frac{\partial}{\partial r}\bigl((\varphi\psi_\varepsilon)^\sharp(y,r)\bigr)\,dr
  \end{align*}
  and the equalities $\partial\varphi^\sharp/\partial r=(\partial_n\varphi)^\sharp$ and $\partial\psi_\varepsilon^\sharp/\partial r=\nabla_\Gamma g/g$ by
  \begin{align*}
    \psi_\varepsilon^\sharp(y,r) = \frac{1}{g(y)}\bigl\{\bigl(r-\varepsilon g_0(y)\bigr)\nabla_\Gamma g_1(y)+\bigl(\varepsilon g_1(y)-r\bigr)\nabla_\Gamma g_0(y)\bigr\}.
  \end{align*}
  Then we have
  \begin{align*}
    I(y) = \int_{\varepsilon g_0(y)}^{\varepsilon g_1(y)}\bigl((\partial_n\varphi)\psi_\varepsilon\bigr)^\sharp(y,r)\,dr = \varepsilon g(y)\bigl[M\bigl((\partial_n\varphi)\psi_\varepsilon\bigr)\bigr](y), \quad y\in\Gamma.
  \end{align*}
  Applying this and $\{I_3-rW(y)\}P(y)=B^\sharp(y,r)$ to the right-hand side of \eqref{Pf_ADer:Der} we obtain \eqref{E:Ave_Der}.
\end{proof}

\begin{remark} \label{R:Ave_Der}
  There exists a constant $c>0$ independent of $\varepsilon$ such that
  \begin{align} \label{E:ADA_Bound}
    |B| \leq c, \quad |\psi_\varepsilon| \leq c\varepsilon \quad\text{in}\quad \Omega_\varepsilon
  \end{align}
  by \eqref{Pf_EAB:Width} and \eqref{E:Ave_Der_Aux}.
  From \eqref{E:ConDer_Bound}, \eqref{E:Width_Bound}, and $\nabla d=\bar{n}$ in $N$ it also follows that
  \begin{align} \label{E:ADA_Grad_Bound}
    |\partial_kB| \leq c, \quad |\nabla\psi_\varepsilon| \leq c, \quad \left|\nabla\psi_\varepsilon-\frac{1}{\bar{g}}\bar{n}\otimes\overline{\nabla_\Gamma g}\right| \leq c\varepsilon \quad\text{in}\quad \Omega_\varepsilon,\,k=1,2,3.
  \end{align}
\end{remark}

\begin{lemma} \label{L:Ave_Wmp}
  Let $m=1,2$ and $p\in[1,\infty)$.
  For $\varphi\in W^{m,p}(\Omega_\varepsilon)$ we have
  \begin{align}
    \|M\varphi\|_{W^{m,p}(\Gamma)} &\leq c\varepsilon^{-1/p}\|\varphi\|_{W^{m,p}(\Omega_\varepsilon)}, \label{E:Ave_Wmp_Surf} \\
    \left\|\overline{M\varphi}\right\|_{W^{m,p}(\Omega_\varepsilon)} &\leq c\|\varphi\|_{W^{m,p}(\Omega_\varepsilon)} \label{E:Ave_Wmp_Dom}
  \end{align}
  with some constant $c>0$ independent of $\varepsilon$ and $\varphi$.
\end{lemma}

\begin{proof}
  Let $\varphi\in W^{1,p}(\Omega_\varepsilon)$.
  From \eqref{E:Ave_Lp_Surf} and \eqref{E:Ave_Der} it follows that
  \begin{align*}
    \|\nabla_\Gamma M\varphi\|_{L^p(\Gamma)} &\leq c\left(\|M(B\nabla\varphi)\|_{L^p(\Gamma)}+\left\|M\bigl((\partial_n\varphi)\psi_\varepsilon\bigr)\right\|_{L^p(\Gamma)}\right) \\
    &\leq c\varepsilon^{-1/p}\left(\|B\nabla\varphi\|_{L^p(\Omega_\varepsilon)}+\|(\partial_n\varphi)\psi_\varepsilon\|_{L^p(\Omega_\varepsilon)}\right).
  \end{align*}
  Here $B$ and $\psi_\varepsilon$ are bounded on $\Omega_\varepsilon$ uniformly in $\varepsilon$ (see Remark~\ref{R:Ave_Der}).
  Hence
  \begin{align} \label{Pf_AW:Lp_Grad}
    \|\nabla_\Gamma M\varphi\|_{L^p(\Gamma)} \leq c\varepsilon^{-1/p}\|\nabla\varphi\|_{L^p(\Omega_\varepsilon)} \leq c\varepsilon^{-1/p}\|\varphi\|_{W^{1,p}(\Omega_\varepsilon)}.
  \end{align}
  Combining \eqref{Pf_AW:Lp_Grad} with \eqref{E:Ave_Lp_Surf} we obtain \eqref{E:Ave_Wmp_Surf} with $m=1$.
  When $\varphi\in W^{2,p}(\Omega_\varepsilon)$, we apply \eqref{Pf_AW:Lp_Grad} with $\varphi$ replaced by $B\nabla\varphi$ and $(\partial_n\varphi)\psi_\varepsilon$.
  Then by \eqref{E:ADA_Bound}--\eqref{E:ADA_Grad_Bound},
  \begin{align*}
    \|\nabla_\Gamma M(B\nabla\varphi)\|_{L^p(\Gamma)}+\left\|\nabla_\Gamma M\bigl((\partial_n\varphi)\psi_\varepsilon\bigr)\right\|_{L^p(\Gamma)} \leq c\varepsilon^{-1/p}\|\varphi\|_{W^{2,p}(\Omega_\varepsilon)}.
  \end{align*}
  Therefore, applying $\nabla_\Gamma$ to \eqref{E:Ave_Der} and using the above inequality we get
  \begin{align*}
    \|\nabla_\Gamma^2 M\varphi\|_{L^p(\Gamma)} \leq c\varepsilon^{-1/p}\|\varphi\|_{W^{2,p}(\Omega_\varepsilon)}
  \end{align*}
  and \eqref{E:Ave_Wmp_Surf} with $m=2$ follows from this inequality, \eqref{E:Ave_Lp_Surf}, and \eqref{Pf_AW:Lp_Grad}.
  The inequality \eqref{E:Ave_Wmp_Dom} is an immediate consequence of \eqref{E:Con_W1p}, \eqref{E:Con_W2p}, and \eqref{E:Ave_Wmp_Surf}.
\end{proof}

\begin{lemma} \label{L:Ave_Der_Diff}
  There exists a constant $c>0$ independent of $\varepsilon$ such that
  \begin{align}
    \left\|\overline{P}\nabla\varphi-\overline{\nabla_\Gamma M\varphi}\right\|_{L^p(\Omega_\varepsilon)} &\leq c\varepsilon\|\varphi\|_{W^{2,p}(\Omega_\varepsilon)} \label{E:ADD_Dom} \\
    \left\|\overline{P}\nabla\varphi-\overline{\nabla_\Gamma M\varphi}\right\|_{L^p(\Gamma_\varepsilon^i)} &\leq c\varepsilon^{1-1/p}\|\varphi\|_{W^{2,p}(\Omega_\varepsilon)}, \quad i=0,1 \label{E:ADD_Bo}
  \end{align}
  for all $\varphi\in W^{2,p}(\Omega_\varepsilon)$ with $p\in[1,\infty)$.
\end{lemma}

\begin{proof}
  By \eqref{E:Ave_Der} and \eqref{E:Ave_Der_Aux} we have $\overline{P}\nabla\varphi-\overline{\nabla_\Gamma M\varphi}=u+\bar{v}$ in $\Omega_\varepsilon$, where
  \begin{align*}
    u(x) &:= \overline{P}(x)\nabla\varphi(x)-\left[M\Bigr(\overline{P}\nabla\varphi\Bigl)\right](\pi(x)), \quad x\in\Omega_\varepsilon, \\
    v(y) &:= \left[M\Bigl(d\overline{W}\nabla\varphi\Bigr)\right](y)-\bigl[M\bigl((\partial_n\varphi)\psi_\varepsilon\bigr)\bigr](y), \quad y\in\Gamma.
  \end{align*}
  We apply \eqref{E:Ave_Diff_Dom} to $u$ and use \eqref{E:NorDer_Con} to get
  \begin{align} \label{Pf_ADD:Est_Diff}
    \|u\|_{L^p(\Omega_\varepsilon)} \leq c\varepsilon\left\|\partial_n\Bigl(\overline{P}\nabla\varphi\Bigr)\right\|_{L^p(\Omega_\varepsilon)} \leq c\varepsilon\|\varphi\|_{W^{2,p}(\Omega_\varepsilon)}.
  \end{align}
  Also, from \eqref{E:Ave_Lp_Dom} and
  \begin{align*}
    \left|d\overline{W}\nabla\varphi\right| \leq c\varepsilon|\nabla\varphi|, \quad |(\partial_n\varphi)\psi_\varepsilon| \leq c\varepsilon|\nabla\varphi| \quad\text{in}\quad \Omega_\varepsilon
  \end{align*}
  by \eqref{E:ADA_Bound} and $|d|\leq c\varepsilon$ in $\Omega_\varepsilon$ it follows that
  \begin{align} \label{Pf_ADD:Est_Res}
    \|\bar{v}\|_{L^p(\Omega_\varepsilon)} \leq c\left(\left\|d\overline{W}\nabla\varphi\right\|_{L^p(\Omega_\varepsilon)}+\|(\partial_n\varphi)\psi_\varepsilon\|_{L^p(\Omega_\varepsilon)}\right) \leq c\varepsilon\|\nabla\varphi\|_{L^p(\Omega_\varepsilon)}.
  \end{align}
  Combining \eqref{Pf_ADD:Est_Diff} and \eqref{Pf_ADD:Est_Res} we obtain
  \begin{align*}
    \left\|\overline{P}\nabla\varphi-\overline{\nabla_\Gamma M\varphi}\right\|_{L^p(\Omega_\varepsilon)} \leq \|u\|_{L^p(\Omega_\varepsilon)}+\|\bar{v}\|_{L^p(\Omega_\varepsilon)} \leq c\varepsilon\|\varphi\|_{W^{2,p}(\Omega_\varepsilon)}.
  \end{align*}
  Hence \eqref{E:ADD_Dom} holds.
  Also, \eqref{E:ADD_Bo} follows from \eqref{E:NorDer_Con}, \eqref{E:Poin_Bo}, and \eqref{E:ADD_Dom}.
\end{proof}

\begin{lemma} \label{L:Ave_N_W1p}
  There exists a constant $c>0$ independent of $\varepsilon$ such that
  \begin{align} \label{E:Ave_N_W1p}
    \|Mu\cdot n\|_{W^{1,p}(\Gamma)} \leq c\varepsilon^{1-1/p}\|u\|_{W^{2,p}(\Omega_\varepsilon)}
  \end{align}
  for all $u\in W^{2,p}(\Omega_\varepsilon)^3$, $p\in[1,\infty)$ satisfying \eqref{E:Bo_Imp} on $\Gamma_\varepsilon^0$  or on $\Gamma_\varepsilon^1$.
\end{lemma}

\begin{proof}
  By \eqref{E:Ave_Der} we have
  \begin{align*}
    \nabla_\Gamma(Mu\cdot n) = \nabla_\Gamma M(u\cdot\bar{n}) = M\bigl(B\nabla(u\cdot\bar{n})\bigr)+M(\partial_n(u\cdot\bar{n})\psi_\varepsilon) \quad\text{on}\quad \Gamma.
  \end{align*}
  To the right-hand side we apply \eqref{E:Ave_Lp_Surf} and
  \begin{align*}
    |B\nabla(u\cdot\bar{n})| \leq c\left|\overline{P}\nabla(u\cdot\bar{n})\right|, \quad |\partial_n(u\cdot\bar{n})\psi_\varepsilon| \leq c\varepsilon|\nabla u| \quad\text{in}\quad \Omega_\varepsilon
  \end{align*}
  by \eqref{E:Wein_Bound}, \eqref{E:NorDer_Con}, and \eqref{E:ADA_Bound} to deduce that
  \begin{align*}
    \|\nabla_\Gamma(Mu\cdot n)\|_{L^p(\Gamma)} &\leq c\left(\left\|M\bigl(B\nabla(u\cdot\bar{n})\bigr)\right\|_{L^p(\Gamma)}+\|M(\partial_n(u\cdot\bar{n})\psi_\varepsilon)\|_{L^p(\Gamma)}\right) \\
    &\leq c\varepsilon^{-1/p}\left(\|B\nabla(u\cdot\bar{n})\|_{L^p(\Omega_\varepsilon)}+\|\partial_n(u\cdot\bar{n})\psi_\varepsilon\|_{L^p(\Omega_\varepsilon)}\right) \\
    &\leq c\left(\varepsilon^{-1/p}\left\|\overline{P}\nabla(u\cdot\bar{n})\right\|_{L^p(\Omega_\varepsilon)}+\varepsilon^{1-1/p}\|\nabla u\|_{L^p(\Omega_\varepsilon)}\right).
  \end{align*}
  Applying \eqref{E:Poin_Dnor} to the first term on the last line we obtain
  \begin{align*}
    \|\nabla_\Gamma(Mu\cdot n)\|_{L^p(\Gamma)} \leq c\varepsilon^{1-1/p}\|u\|_{W^{2,p}(\Omega_\varepsilon)}.
  \end{align*}
  The inequality \eqref{E:Ave_N_W1p} follows from this inequality and \eqref{E:Ave_N_Lp}.
\end{proof}

Next we establish estimates for the weighted surface divergence of the average of a solenoidal vector field on $\Omega_\varepsilon$.
They are useful for the proof of the global existence of a strong solution as well as the study of a singular limit problem for \eqref{E:NS_Eq}--\eqref{E:NS_In}.

\begin{lemma} \label{L:Ave_Div}
  Let $p\in[1,\infty)$.
  There exists $c>0$ independent of $\varepsilon$ such that
  \begin{align} \label{E:Ave_Div_Lp}
    \|\mathrm{div}_\Gamma(gMu)\|_{L^p(\Gamma)} \leq c\varepsilon^{1-1/p}\|u\|_{W^{1,p}(\Omega_\varepsilon)}
  \end{align}
  for all $u\in W^{1,p}(\Omega_\varepsilon)^3$ satisfying $\mathrm{div}\,u=0$ in $\Omega_\varepsilon$ and \eqref{E:Bo_Imp} on $\Gamma_\varepsilon$.
  If in addition $u\in W^{2,p}(\Omega_\varepsilon)^3$, then we have
  \begin{align} \label{E:Ave_Div_W1p}
    \|\mathrm{div}_\Gamma(gMu)\|_{W^{1,p}(\Gamma)} \leq c\varepsilon^{1-1/p}\|u\|_{W^{2,p}(\Omega_\varepsilon)}.
  \end{align}
\end{lemma}

\begin{proof}
  We use the notations \eqref{E:Pull_Dom} and \eqref{E:Pull_Bo}.
  Let $u\in W^{1,p}(\Omega_\varepsilon)^3$ satisfy $\mathrm{div}\,u=0$ in $\Omega_\varepsilon$ and \eqref{E:Bo_Imp} on $\Gamma_\varepsilon$.
  We define functions on $\Gamma$ by
  \begin{align} \label{Pf_ADiv:Phi_Aux}
    \begin{gathered}
      \eta_1 := \sum_{i=0,1}(-1)^i(u_i^\sharp-Mu)\cdot\tau_\varepsilon^i, \quad \eta_2 := \sum_{i=0,1}(-1)^iMu\cdot(\tau_\varepsilon^i-\nabla_\Gamma g_i), \\
      \eta_3 := -gM\left(d\,\mathrm{tr}\Bigl[\overline{W}\nabla u\Bigr]\right), \quad \eta_4 := gM(\partial_nu\cdot\psi_\varepsilon),
    \end{gathered}
  \end{align}
  where $\tau_\varepsilon^0$ and $\tau_\varepsilon^1$ are given by \eqref{E:Def_NB_Aux}.
  First we show that
  \begin{align} \label{Pf_ADiv:Sum}
    \mathrm{div}_\Gamma(gMu) = \eta_1+\eta_2+\eta_3+\eta_4 \quad\text{on}\quad \Gamma.
  \end{align}
  By \eqref{E:Ave_Der} and \eqref{E:Ave_Der_Aux} we have
  \begin{align*}
    g\,\mathrm{div}_\Gamma(Mu) &= g\,\mathrm{tr}[\nabla_\Gamma Mu] = gM(\mathrm{tr}[B\nabla u])+gM(\mathrm{tr}[\psi_\varepsilon\otimes\partial_n u]) \\
    &= gM\left(\mathrm{tr}\Bigl[\overline{P}\nabla u\Bigr]\right)+\eta_3+\eta_4
  \end{align*}
  on $\Gamma$.
  Moreover, since $\mathrm{div}\,u=0$ and $(\bar{n}\otimes\bar{n})\nabla u=\bar{n}\otimes\partial_nu$ in $\Omega_\varepsilon$,
  \begin{align*}
    \mathrm{tr}\Bigl[\overline{P}\nabla u\Bigr] = \mathrm{div}\,u-\mathrm{tr}[(\bar{n}\otimes\bar{n})\nabla u] = -\bar{n}\cdot\partial_nu \quad\text{in}\quad \Omega_\varepsilon.
  \end{align*}
  By these equalities and $\mathrm{div}_\Gamma(gMu)=\nabla_\Gamma g\cdot Mu+g\,\mathrm{div}_\Gamma(Mu)$ on $\Gamma$ we get
  \begin{align} \label{Pf_ADiv:For1}
    \mathrm{div}_\Gamma(gMu) = \nabla_\Gamma g\cdot Mu-gM(\partial_nu\cdot\bar{n})+\eta_3+\eta_4 \quad\text{on}\quad \Gamma.
  \end{align}
  Let us calculate the second term on the right-hand side.
  Since
  \begin{align*}
    (\partial_nu\cdot \bar{n})^\sharp(y,r) = \frac{\partial}{\partial r}\Bigl((u\cdot\bar{n})^\sharp(y,r)\Bigr), \quad y\in\Gamma,\,r\in(\varepsilon g_0(y),\varepsilon g_1(y))
  \end{align*}
  by $\partial_n\bar{n}=0$ in $\Omega_\varepsilon$, we have
  \begin{align*}
    g(y)[M(\partial_nu\cdot\bar{n})](y) &= \frac{1}{\varepsilon}\int_{\varepsilon g_0(y)}^{\varepsilon g_1(y)}\frac{\partial}{\partial r}\Bigl((u\cdot\bar{n})^\sharp(y,r)\Bigr)\,dr \\
    &= \frac{1}{\varepsilon}\{(u\cdot\bar{n})^\sharp(y,\varepsilon g_1(y))-(u\cdot\bar{n})^\sharp(y,\varepsilon g_0(y))\}
  \end{align*}
  for $y\in\Gamma$.
  Moreover, since $u$ satisfies \eqref{E:Bo_Imp} on $\Gamma_\varepsilon$,
  \begin{align*}
    (u\cdot\bar{n})^\sharp(y,\varepsilon g_i(y)) = \varepsilon(u\cdot\bar{\tau}_\varepsilon^i)^\sharp(y,\varepsilon g_i(y)) = \varepsilon(u_i^\sharp\cdot\tau_\varepsilon^i)(y), \quad y\in\Gamma
  \end{align*}
  by \eqref{E:Exp_Bo}.
  Hence
  \begin{align*}
    gM(\partial_nu\cdot\bar{n}) = u_1^\sharp\cdot\tau_\varepsilon^1-u_0^\sharp\cdot\tau_\varepsilon^0 = \nabla_\Gamma g\cdot Mu-\eta_1-\eta_2 \quad\text{on}\quad \Gamma.
  \end{align*}
  Combining this with \eqref{Pf_ADiv:For1} we obtain \eqref{Pf_ADiv:Sum}.
  Let us estimate the $L^p(\Gamma)$-norms of the functions $\eta_1,\dots,\eta_4$.
  By \eqref{E:Tau_Bound}, \eqref{E:Lp_CoV_Surf}, and \eqref{E:Ave_Diff_Bo},
  \begin{align} \label{Pf_ADiv:Lp_Phi1}
    \|\eta_1\|_{L^p(\Gamma)} \leq c\sum_{i=0,1}\left\|u-\overline{Mu}\right\|_{L^p(\Gamma_\varepsilon^i)} \leq c\varepsilon^{1-1/p}\|u\|_{W^{1,p}(\Omega_\varepsilon)}.
  \end{align}
  The first inequality of \eqref{E:Tau_Diff} and \eqref{E:Ave_Lp_Surf} imply that
  \begin{align} \label{Pf_ADiv:Lp_Phi2}
    \|\eta_2\|_{L^p(\Gamma)} \leq c\varepsilon\|Mu\|_{L^p(\Gamma)} \leq c\varepsilon^{1-1/p}\|u\|_{L^p(\Omega_\varepsilon)}.
  \end{align}
  To $\eta_3$ and $\eta_4$ we apply \eqref{E:Ave_Lp_Surf} and
  \begin{align*}
    \left|d\,\mathrm{tr}\Bigl[\overline{W}\nabla u\Bigr]\right|\leq c\varepsilon|\nabla u|, \quad |\partial_nu\cdot\psi_\varepsilon|\leq c\varepsilon|\nabla u| \quad\text{in}\quad \Omega_\varepsilon
  \end{align*}
  by \eqref{E:ADA_Bound}, $|d|\leq c\varepsilon$ in $\Omega_\varepsilon$, and the boundedness of $W$ on $\Gamma$ to get
  \begin{align} \label{Pf_ADiv:Lp_Phi34}
    \begin{aligned}
      \|\eta_3\|_{L^p(\Gamma)} &\leq c\varepsilon^{-1/p}\left\|d\,\mathrm{tr}\Bigl[\overline{W}\nabla u\Bigr]\right\|_{L^p(\Omega_\varepsilon)} \leq c\varepsilon^{1-1/p}\|\nabla u\|_{L^p(\Omega_\varepsilon)}, \\
      \|\eta_4\|_{L^p(\Gamma)} &\leq c\varepsilon^{-1/p}\|\partial_nu\cdot\psi_\varepsilon\|_{L^p(\Omega_\varepsilon)} \leq c\varepsilon^{1-1/p}\|\nabla u\|_{L^p(\Omega_\varepsilon)}.
    \end{aligned}
  \end{align}
  Applying \eqref{Pf_ADiv:Lp_Phi1}, \eqref{Pf_ADiv:Lp_Phi2}, and \eqref{Pf_ADiv:Lp_Phi34} to \eqref{Pf_ADiv:Sum} we obtain \eqref{E:Ave_Div_Lp}.

  Now we assume $u\in W^{2,p}(\Omega_\varepsilon)^3$ and estimate the $L^p(\Gamma)$-norms of the tangential gradients of $\eta_1,\dots,\eta_4$.
  By \eqref{E:Pull_Bo} and $\nabla_\Gamma\bigl(y+\varepsilon g_i(y)n(y)\bigr)=P(y)+\varepsilon G_i(y)$, where
  \begin{align*}
    G_i(y) := \nabla_\Gamma g_i(y)\otimes n(y)-g_i(y)W(y), \quad y\in\Gamma,\,i=0,1,
  \end{align*}
  we have $\nabla_\Gamma u_i^\sharp=(P+\varepsilon G_i)(\nabla u)_i^\sharp$ on $\Gamma$ for $i=0,1$.
  Hence
  \begin{align*}
    \nabla_\Gamma\eta_1 &= \sum_{i=0,1}(-1)^i\{(\nabla_\Gamma u_i^\sharp-\nabla_\Gamma Mu)\tau_\varepsilon^i+(\nabla_\Gamma\tau_\varepsilon^i)(u_i^\sharp-Mu)\} \\
    &= \sum_{i=0,1}(-1)^i\Bigl[\bigl\{\bigl(P(\nabla u)_i^\sharp-\nabla_\Gamma Mu\bigr)+\varepsilon G_i(\nabla u)_i^\sharp\bigr\}\tau_\varepsilon^i+(\nabla_\Gamma\tau_\varepsilon^i)(u_i^\sharp-Mu)\Bigr]
  \end{align*}
  on $\Gamma$.
  Since $G_0$ and $G_1$ are bounded on $\Gamma$, we see by \eqref{E:Tau_Bound} that
  \begin{align*}
    |\nabla_\Gamma\eta_1| \leq c\sum_{i=0,1}(|P(\nabla u)_i^\sharp-\nabla_\Gamma Mu|+\varepsilon|(\nabla u)_i^\sharp|+|u_i^\sharp-Mu|) \quad\text{on}\quad \Gamma.
  \end{align*}
  From this inequality and \eqref{E:Lp_CoV_Surf} we deduce that
  \begin{align*}
    &\|\nabla_\Gamma\eta_1\|_{L^p(\Gamma)} \\
    &\qquad \leq c\sum_{i=0,1}\left(\|P(\nabla u)_i^\sharp-\nabla_\Gamma Mu\|_{L^p(\Gamma)}+\varepsilon\|(\nabla u)_i^\sharp\|_{L^p(\Gamma)}+\|u_i^\sharp-Mu\|_{L^p(\Gamma)}\right) \\
    &\qquad \leq c\sum_{i=0,1}\left(\left\|\overline{P}\nabla u-\overline{\nabla_\Gamma Mu}\right\|_{L^p(\Gamma_\varepsilon^i)}+\varepsilon\|\nabla u\|_{L^p(\Gamma_\varepsilon^i)}+\left\|u-\overline{Mu}\right\|_{L^p(\Gamma_\varepsilon^i)}\right).
  \end{align*}
  To the right-hand side we apply \eqref{E:Poin_Bo}, \eqref{E:Ave_Diff_Bo}, and \eqref{E:ADD_Bo} to obtain
  \begin{align} \label{Pf_ADiv:Lp_Grad_Phi1}
    \|\nabla_\Gamma\eta_1\|_{L^p(\Gamma)} \leq c\varepsilon^{1-1/p}\|u\|_{W^{2,p}(\Omega_\varepsilon)}.
  \end{align}
  Next we estimate the $L^p(\Gamma)$-norm of $\nabla_\Gamma\eta_2$.
  From \eqref{E:Tau_Diff} and
  \begin{align*}
    \nabla_\Gamma\eta_2 = \sum_{i=0,1}(-1)^i\{(\nabla_\Gamma Mu)(\tau_\varepsilon^i-\nabla_\Gamma g_i)+(\nabla_\Gamma\tau_\varepsilon^i-\nabla_\Gamma^2 g_i)Mu\} \quad\text{on}\quad \Gamma
  \end{align*}
  it follows that $|\nabla_\Gamma\eta_2|\leq c\varepsilon(|Mu|+|\nabla_\Gamma Mu|)$ on $\Gamma$.
  Hence by \eqref{E:Ave_Wmp_Surf} with $m=1$,
  \begin{align} \label{Pf_ADiv:Lp_Grad_Phi2}
    \|\nabla_\Gamma\eta_2\|_{L^p(\Gamma)} \leq c\varepsilon\|Mu\|_{W^{1,p}(\Gamma)} \leq c\varepsilon^{1-1/p}\|u\|_{W^{1,p}(\Omega_\varepsilon)}.
  \end{align}
  For the tangential gradient of $\eta_3=-gM(d\phi)$ with $\phi:=\mathrm{tr}[\overline{W}\nabla u]$ on $\Omega_\varepsilon$ we have
  \begin{align*}
    \nabla_\Gamma\eta_3 = -M(d\phi)\nabla_\Gamma g-g\bigl\{M(\phi B\nabla d)+M(dB\nabla\phi)+M\bigl(\partial_n(d\phi)\psi_\varepsilon\bigr)\bigr\} \quad\text{on}\quad \Gamma
  \end{align*}
  by \eqref{E:Ave_Der}.
  Here the second term on the right-hand side vanishes since $B\nabla d=0$ in $\Omega_\varepsilon$ by $\nabla d=\bar{n}$ in $\Omega_\varepsilon$ and $Pn=0$ on $\Gamma$.
  These facts and \eqref{E:Ave_Lp_Surf} imply that
  \begin{align*}
    \|\nabla_\Gamma\eta_3\|_{L^p(\Gamma)} &\leq c\left(\|M(d\phi)\|_{L^p(\Gamma)}+\|M(dB\nabla\phi)\|_{L^p(\Gamma)}+\left\|M\bigl(\partial_n(d\phi)\psi_\varepsilon\bigr)\right\|_{L^p(\Gamma)}\right) \\
    &\leq c\varepsilon^{-1/p}\left(\|d\phi\|_{L^p(\Omega_\varepsilon)}+\|dB\nabla\phi\|_{L^p(\Omega_\varepsilon)}+\|\partial_n(d\phi)\psi_\varepsilon\|_{L^p(\Omega_\varepsilon)}\right).
  \end{align*}
  To the second line we apply
  \begin{align*}
    |d\phi| \leq c\varepsilon|\nabla u|, \quad |dB\nabla\phi| \leq c\varepsilon(|\nabla u|+|\nabla^2u|), \quad |\partial_n(d\phi)\psi_\varepsilon| \leq c\varepsilon(|\nabla u|+|\nabla^2u|)
  \end{align*}
  in $\Omega_\varepsilon$ by \eqref{E:ADA_Bound}, $|d|\leq c\varepsilon$ in $\Omega_\varepsilon$, and $\partial_nd=\bar{n}\cdot\nabla d=|\bar{n}|^2=1$ in $N$ to obtain
  \begin{align} \label{Pf_ADiv:Lp_Grad_Phi3}
    \|\nabla_\Gamma\eta_3\|_{L^p(\Gamma)} \leq c\varepsilon^{1-1/p}\|u\|_{W^{2,p}(\Omega_\varepsilon)}.
  \end{align}
  Let us estimate the $L^p(\Gamma)$-norm of $\nabla_\Gamma\eta_4$.
  Setting $\xi:=\partial_nu\cdot\psi_\varepsilon$ on $\Omega_\varepsilon$ we have
  \begin{align*}
    \nabla_\Gamma\eta_4=(M\xi)\nabla_\Gamma g+g\bigl\{M(B\nabla\xi)+M\bigl((\partial_n\xi)\psi_\varepsilon\bigr)\bigr\} \quad\text{on}\quad \Gamma
  \end{align*}
  by \eqref{E:Ave_Der}.
  From this equality and \eqref{E:Ave_Lp_Surf} we deduce that
  \begin{align} \label{Pf_ADiv:GP4_Aux_1}
    \begin{aligned}
      \|\nabla_\Gamma\eta_4\|_{L^p(\Gamma)} &\leq c\left(\|M\xi\|_{L^p(\Gamma)}+\|M(B\nabla\xi)\|_{L^p(\Gamma)}+\left\|M\bigl((\partial_n\xi)\psi_\varepsilon\bigr)\right\|_{L^p(\Gamma)}\right) \\
      &\leq c\varepsilon^{-1/p}\left(\|\xi\|_{L^p(\Omega_\varepsilon)}+\|B\nabla\xi\|_{L^p(\Omega_\varepsilon)}+\|(\partial_n\xi)\psi_\varepsilon\|_{L^p(\Omega_\varepsilon)}\right).
    \end{aligned}
  \end{align}
  We apply \eqref{E:ADA_Bound} and \eqref{E:ADA_Grad_Bound} to $\xi=\partial_nu\cdot\psi_\varepsilon$ and $(\partial_n\xi)\psi_\varepsilon$ to get
  \begin{align} \label{Pf_ADiv:GP4_Aux_2}
    \begin{aligned}
      \|\xi\|_{L^p(\Omega_\varepsilon)} &\leq c\varepsilon\|\nabla u\|_{L^p(\Omega_\varepsilon)}, \\
      \|(\partial_n\xi)\psi_\varepsilon\|_{L^p(\Omega_\varepsilon)} &\leq c\varepsilon\left(\|\nabla u\|_{L^p(\Omega_\varepsilon)}+\|\nabla^2u\|_{L^p(\Omega_\varepsilon)}\right).
      \end{aligned}
  \end{align}
  Moreover, by \eqref{E:ADA_Grad_Bound} and $P(n\otimes\nabla_\Gamma g)=(Pn)\otimes\nabla_\Gamma g=0$ on $\Gamma$ we see that
  \begin{align*}
    \left|\overline{P}\nabla\psi_\varepsilon\right| = \left|\overline{P}\left(\nabla\psi_\varepsilon-\frac{1}{\bar{g}}\bar{n}\otimes\overline{\nabla_\Gamma g}\right)\right| \leq c\varepsilon \quad\text{in}\quad \Omega_\varepsilon.
  \end{align*}
  Using this inequality and \eqref{E:ADA_Bound} to $B\nabla\xi=B\{\nabla(\partial_nu)\}\psi_\varepsilon+B(\nabla\psi_\varepsilon)\partial_nu$ we get
  \begin{align} \label{Pf_ADiv:GP4_Aux_3}
    \|B\nabla\xi\|_{L^p(\Omega_\varepsilon)} \leq c\varepsilon\left(\|\nabla u\|_{L^p(\Omega_\varepsilon)}+\|\nabla^2u\|_{L^p(\Omega_\varepsilon)}\right).
  \end{align}
  From \eqref{Pf_ADiv:GP4_Aux_1}--\eqref{Pf_ADiv:GP4_Aux_3} it follows that
  \begin{align} \label{Pf_ADiv:Lp_Grad_Phi4}
    \|\nabla_\Gamma\eta_4\|_{L^p(\Omega_\varepsilon)}\leq c\varepsilon^{1-1/p}\|u\|_{W^{2,p}(\Omega_\varepsilon)}.
  \end{align}
  Finally, from \eqref{Pf_ADiv:Sum}, \eqref{Pf_ADiv:Lp_Grad_Phi1}--\eqref{Pf_ADiv:Lp_Grad_Phi3}, and \eqref{Pf_ADiv:Lp_Grad_Phi4} we deduce that
  \begin{align*}
    \left\|\nabla_\Gamma\bigl(\mathrm{div}_\Gamma(gMu)\bigr)\right\|_{L^p(\Gamma)} \leq \sum_{j=1}^4\|\nabla_\Gamma\eta_j\|_{L^p(\Gamma)} \leq c\varepsilon^{1-1/p}\|u\|_{W^{2,p}(\Omega_\varepsilon)}
  \end{align*}
  and conclude that \eqref{E:Ave_Div_W1p} is valid.
\end{proof}

\begin{lemma} \label{L:ADiv_Tan}
  There exists a constant $c>0$ independent of $\varepsilon$ such that
  \begin{align} \label{E:ADiv_Tan_Lp}
    \|\mathrm{div}_\Gamma(gM_\tau u)\|_{L^p(\Gamma)} \leq c\varepsilon^{1-1/p}\|u\|_{W^{1,p}(\Omega_\varepsilon)}
  \end{align}
  for all $u\in W^{1,p}(\Omega_\varepsilon)^3$, $p\in[1,\infty)$ satisfying $\mathrm{div}\,u=0$ in $\Omega_\varepsilon$ and \eqref{E:Bo_Imp} on $\Gamma_\varepsilon$.
  If in addition $u\in W^{2,p}(\Omega_\varepsilon)^3$, then we have
  \begin{align} \label{E:ADiv_Tan_W1p}
    \|\mathrm{div}_\Gamma(gM_\tau u)\|_{W^{1,p}(\Gamma)} \leq c\varepsilon^{1-1/p}\|u\|_{W^{2,p}(\Omega_\varepsilon)}.
  \end{align}
\end{lemma}

\begin{proof}
  By $M_\tau u=Mu-(Mu\cdot n)n$ on $\Gamma$, \eqref{E:P_TGr}, and \eqref{E:Def_WHK},
  \begin{align*}
    \mathrm{div}_\Gamma(gM_\tau u) = \mathrm{div}_\Gamma(gMu)+g(Mu\cdot n)H \quad\text{on}\quad \Gamma.
  \end{align*}
  Hence $|\mathrm{div}_\Gamma(gM_\tau u)| \leq c(|\mathrm{div}_\Gamma(gMu)|+|Mu\cdot n|)$ and
  \begin{align*}
    \left|\nabla_\Gamma\bigl(\mathrm{div}_\Gamma(gM_\tau u)\bigr)\right| \leq c\left(\left|\nabla_\Gamma\bigl(\mathrm{div}_\Gamma(gMu)\bigr)\right|+|Mu\cdot n|+|\nabla_\Gamma(Mu\cdot n)|\right)
  \end{align*}
  on $\Gamma$.
  These inequalities, \eqref{E:Ave_N_Lp}, and \eqref{E:Ave_N_W1p}--\eqref{E:Ave_Div_W1p} imply \eqref{E:ADiv_Tan_Lp} and \eqref{E:ADiv_Tan_W1p}.
\end{proof}

By Lemmas~\ref{L:Ave_Der_Diff} and~\ref{L:Ave_Div} we obtain a relation for the normal derivative (with respect to $\Gamma$) of a vector field on $\Omega_\varepsilon$ and its averaged tangential component.

\begin{lemma} \label{L:DnU_N_Ave}
  Let $p\in[1,\infty)$.
  There exists $c>0$ independent of $\varepsilon$ such that
  \begin{align} \label{E:DnU_N_Ave}
    \left\|\partial_nu\cdot\bar{n}-\frac{1}{\bar{g}}\overline{M_\tau u}\cdot\overline{\nabla_\Gamma g}\right\|_{L^p(\Omega_\varepsilon)} \leq c\varepsilon\|u\|_{W^{2,p}(\Omega_\varepsilon)}
  \end{align}
  for all $u\in W^{2,p}(\Omega_\varepsilon)^3$ satisfying $\mathrm{div}\,u=0$ in $\Omega_\varepsilon$ and \eqref{E:Bo_Imp} on $\Gamma_\varepsilon$.
  Here $\partial_nu$ is the normal derivative of $u$ given by \eqref{E:Def_NorDer}.
\end{lemma}

\begin{proof}
  By $Q=n\otimes n$, $I_3=P+Q$ on $\Gamma$ and $\partial_nu=(\bar{n}\cdot\nabla)u=(\nabla u)^T\bar{n}$ in $\Omega_\varepsilon$,
  \begin{align*}
    \partial_nu\cdot\bar{n} = \mathrm{tr}[\bar{n}\otimes\partial_nu] = \mathrm{tr}\Bigl[\overline{Q}\nabla u\Bigr] = \mathrm{div}\,u-\mathrm{tr}\Bigl[\overline{P}\nabla u\Bigr] \quad\text{in}\quad \Omega_\varepsilon.
  \end{align*}
  Also, since $\nabla_\Gamma g$ is tangential on $\Gamma$,
  \begin{align*}
    \frac{1}{g}M_\tau u\cdot\nabla_\Gamma g &= \frac{1}{g}Mu\cdot\nabla_\Gamma g = \frac{1}{g}\mathrm{div}_\Gamma(gMu)-\mathrm{div}_\Gamma(Mu) \\
    &= \frac{1}{g}\mathrm{div}_\Gamma(gMu)-\mathrm{tr}[\nabla_\Gamma Mu]
  \end{align*}
  on $\Gamma$.
  From these equalities, $\mathrm{div}\,u=0$ in $\Omega_\varepsilon$, and \eqref{E:Width_Bound} we deduce that
  \begin{align*}
    \left\|\partial_nu\cdot\bar{n}-\frac{1}{\bar{g}}\overline{M_\tau u}\cdot\overline{\nabla_\Gamma g}\right\|_{L^p(\Omega_\varepsilon)} \leq \left\|\overline{P}\nabla u-\overline{\nabla_\Gamma Mu}\right\|_{L^p(\Omega_\varepsilon)}+c\left\|\overline{\mathrm{div}_\Gamma(gMu)}\right\|_{L^p(\Omega_\varepsilon)}.
  \end{align*}
  Applying \eqref{E:Con_Lp}, \eqref{E:ADD_Dom}, and \eqref{E:Ave_Div_Lp} to the right-hand side we obtain \eqref{E:DnU_N_Ave}.
\end{proof}

When $u\in L^2(\Omega_\varepsilon)^3$ we can consider $\mathrm{div}_\Gamma(gM_\tau u)$ as an element of $H^{-1}(\Gamma)$.
In particular, if $u\in L_\sigma^2(\Omega_\varepsilon)$ then we have its $H^{-1}(\Gamma)$-estimate similar to \eqref{E:ADiv_Tan_Lp}.

\begin{lemma} \label{L:ADiv_Tan_Hin}
  There exists a constant $c>0$ independent of $\varepsilon$ such that
  \begin{align} \label{E:ADiv_Tan_Hin}
    \|\mathrm{div}_\Gamma(gM_\tau u)\|_{H^{-1}(\Gamma)} \leq c\varepsilon^{1/2}\|u\|_{L^2(\Omega_\varepsilon)}
  \end{align}
  for all $u\in L_\sigma^2(\Omega_\varepsilon)$.
\end{lemma}

\begin{proof}
  We use the notation \eqref{E:Pull_Dom} and suppress the arguments of functions.
  Let $\eta$ be an arbitrary function in $H^1(\Gamma)$.
  By \eqref{E:Sdiv_Hin} we have
  \begin{align*}
    \langle\mathrm{div}_\Gamma(gM_\tau u),\eta\rangle_\Gamma = -\int_\Gamma gM_\tau u\cdot\nabla_\Gamma\eta\,d\mathcal{H}^2-\int_\Gamma g(M_\tau u\cdot n)\eta H\,d\mathcal{H}^2.
  \end{align*}
  The second term on the right-hand side vanishes by $M_\tau u\cdot n=0$ on $\Gamma$.
  To estimate the first term, we observe by $M_\tau u\cdot\nabla_\Gamma\eta=Mu\cdot\nabla_\Gamma\eta$ on $\Gamma$ and \eqref{E:Def_Ave} that
  \begin{align*}
    \int_\Gamma gM_\tau u\cdot\nabla_\Gamma\eta\,d\mathcal{H}^2 = \varepsilon^{-1}\int_\Gamma\int_{\varepsilon g_0}^{\varepsilon g_1}u^\sharp\cdot\nabla_\Gamma\eta\,dr\,d\mathcal{H}^2.
  \end{align*}
  From this equality and the change of variables formula \eqref{E:CoV_Dom} it follows that
  \begin{align*}
    \left|\varepsilon^{-1}\int_{\Omega_\varepsilon}u\cdot\nabla\bar{\eta}\,dx-\int_\Gamma gM_\tau u\cdot\nabla_\Gamma\eta\,d\mathcal{H}^2\right| \leq \varepsilon^{-1}(J_1+J_2),
  \end{align*}
  where $\bar{\eta}=\eta\circ\pi$ is the constant extension of $\eta$ and
  \begin{align*}
    J_1 := \left|\int_{\Omega_\varepsilon}u\cdot\Bigl(\nabla\bar{\eta}-\overline{\nabla_\Gamma\eta}\Bigr)dx\right|, \quad J_2 := \left|\int_\Gamma\int_{\varepsilon g_0}^{\varepsilon g_1}(u^\sharp\cdot\nabla_\Gamma\eta)(J-1)\,dr\,d\mathcal{H}^2\right|.
  \end{align*}
  To $J_1$ we apply \eqref{E:ConDer_Diff} with $|d|\leq c\varepsilon$ in $\Omega_\varepsilon$, H\"{o}lder's inequality, and \eqref{E:Con_Lp} to get
  \begin{align*}
    J_1 \leq c\varepsilon\|u\|_{L^2(\Omega_\varepsilon)}\left\|\overline{\nabla_\Gamma\eta}\right\|_{L^2(\Omega_\varepsilon)} \leq c\varepsilon^{3/2}\|u\|_{L^2(\Omega_\varepsilon)}\|\nabla_\Gamma\eta\|_{L^2(\Gamma)}.
  \end{align*}
  We also have the same inequality for $J_2$ by \eqref{E:Jac_Diff}, \eqref{E:CoV_Equiv}, and \eqref{E:Con_Lp}.
  Hence
  \begin{multline*}
    \left|\varepsilon^{-1}\int_{\Omega_\varepsilon}u\cdot\nabla\bar{\eta}\,dx-\int_\Gamma gM_\tau u\cdot\nabla_\Gamma\eta\,d\mathcal{H}^2\right| \\
    \leq \varepsilon^{-1}(J_1+J_2) \leq c\varepsilon^{1/2}\|u\|_{L^2(\Omega_\varepsilon)}\|\nabla_\Gamma\eta\|_{L^2(\Gamma)}.
  \end{multline*}
  Here $\int_{\Omega_\varepsilon}u\cdot\nabla\bar{\eta}\,dx=0$ by $u\in L_\sigma^2(\Omega_\varepsilon)$ and $\nabla\bar{\eta}\in L_\sigma^2(\Omega_\varepsilon)^\perp$.
  Therefore,
  \begin{align*}
    |\langle\mathrm{div}_\Gamma(gM_\tau u),\eta\rangle_\Gamma| = \left|\int_\Gamma gM_\tau u\cdot\nabla_\Gamma\eta\,d\mathcal{H}^2\right| \leq c\varepsilon^{1/2}\|u\|_{L^2(\Omega_\varepsilon)}\|\nabla_\Gamma\eta\|_{L^2(\Gamma)}.
  \end{align*}
  Since this inequality holds for all $\eta\in H^1(\Gamma)$, the inequality \eqref{E:ADiv_Tan_Hin} is valid.
\end{proof}

\subsection{Decomposition of vector fields into the average and residual parts} \label{SS:Ave_Ex}
In the study of the Navier--Stokes equations in a thin domain it is convenient to decompose a three-dimensional vector field into an almost two-dimensional vector field and its residual term and analyze them separately.
Moreover, we can derive a good $L^\infty$-estimate for the residual term if it satisfies the impermeable boundary condition.
To give such a decomposition we use the impermeable extension operator $E_\varepsilon$ given by \eqref{E:Def_ExTan} in Section~\ref{SS:Tool_TE}.

\begin{definition} \label{D:Def_ExAve}
  For a vector field $u$ on $\Omega_\varepsilon$ we define the average part of $u$ by
  \begin{align} \label{E:Def_ExAve}
    u^a(x) := E_\varepsilon M_\tau u(x) = \overline{M_\tau u}(x)+\left\{\overline{M_\tau u}(x)\cdot\Psi_\varepsilon(x)\right\}\bar{n}(x), \quad x\in N,
  \end{align}
  where $\Psi_\varepsilon$ is the vector field given by \eqref{E:Def_ExAux} and $M_\tau u$ is the averaged tangential component of $u$ given by \eqref{E:Def_Tan_Ave}.
  We also call $u^r:=u-u^a$ the residual part of $u$.
\end{definition}

By Lemmas~\ref{L:ExTan_Wmp}, \ref{L:Ave_Lp}, and \ref{L:Ave_Wmp} we observe that if $u\in H^m(\Omega_\varepsilon)^3$, $m=0,1,2$, then $u^a$ and $u^r$ belong to the same space (here we write $H^0=L^2$).

\begin{lemma} \label{L:Hm_UaUr}
  For $u\in H^m(\Omega_\varepsilon)^3$, $m=0,1,2$ we have $u^a, u^r\in H^m(\Omega_\varepsilon)^3$ and
  \begin{align} \label{E:Hm_UaUr}
    \|u^a\|_{H^m(\Omega_\varepsilon)} \leq c\|u\|_{H^m(\Omega_\varepsilon)}, \quad \|u^r\|_{H^m(\Omega_\varepsilon)} \leq c\|u\|_{H^m(\Omega_\varepsilon)}
  \end{align}
  with a constant $c>0$ independent of $\varepsilon$ and $u$.
\end{lemma}

Since the average part $u^a$ can be seen as almost two-dimensional, we expect to have a good $L^2$-estimate for the product of $u^a$ and a function on $\Omega_\varepsilon$.
Indeed, we can apply the following product estimate for functions on $\Gamma$ and $\Omega_\varepsilon$.

\begin{lemma} \label{L:Prod}
  There exists a constant $c>0$ independent of $\varepsilon$ such that
  \begin{align} \label{E:Prod_Surf}
    \|\bar{\eta}\varphi\|_{L^2(\Omega_\varepsilon)} &\leq c\|\eta\|_{L^2(\Gamma)}^{1/2}\|\eta\|_{H^1(\Gamma)}^{1/2}\|\varphi\|_{L^2(\Omega_\varepsilon)}^{1/2}\|\varphi\|_{H^1(\Omega_\varepsilon)}^{1/2}
  \end{align}
  for all $\eta\in H^1(\Gamma)$ and $\varphi\in H^1(\Omega_\varepsilon)$, where $\bar{\eta}=\eta\circ\pi$ is the constant extension of $\eta$.
\end{lemma}

\begin{proof}
  Throughout the proof we use the notation \eqref{E:Pull_Dom} and suppress the arguments of functions.
  By \eqref{E:CoV_Equiv} and \eqref{E:Def_Ave} we have
  \begin{align*}
    \|\bar{\eta}\varphi\|_{L^2(\Omega_\varepsilon)}^2 \leq c\int_\Gamma|\eta|^2\left(\int_{\varepsilon g_0}^{\varepsilon g_1}|\varphi^\sharp|^2\,dr\right)d\mathcal{H}^2 = c\varepsilon\int_\Gamma g|\eta|^2M(|\varphi|^2)\,d\mathcal{H}^2.
  \end{align*}
  Noting that $g$ is bounded on $\Gamma$, we apply H\"{o}lder's inequality to the last term to get
  \begin{align} \label{Pf_Pr:Est_L2}
    \|\bar{\eta}\varphi\|_{L^2(\Omega_\varepsilon)}^2 \leq c\varepsilon\|\eta\|_{L^4(\Gamma)}^2\|M(|\varphi|^2)\|_{L^2(\Gamma)}.
  \end{align}
  The $L^4(\Gamma)$-norm of $\eta$ is estimated by Ladyzhenskaya's inequality \eqref{E:La_Surf}.
  To estimate the last term let us show $M(|\varphi|^2)\in W^{1,1}(\Gamma)$.
  By \eqref{E:Width_Bound} and \eqref{E:CoV_Equiv},
  \begin{align} \label{Pf_Pr:L1_MPhi2}
    \|M(|\varphi|^2)\|_{L^1(\Gamma)} = \int_\Gamma\frac{1}{\varepsilon g}\left(\int_{\varepsilon g_0}^{\varepsilon g_1}|\varphi^\sharp|^2\,dr\right)d\mathcal{H}^2 \leq c\varepsilon^{-1}\|\varphi\|_{L^2(\Omega_\varepsilon)}^2.
  \end{align}
  Also, by \eqref{E:Ave_Der}, \eqref{E:ADA_Bound}, and $\nabla(|\varphi|^2)=2\varphi\nabla\varphi$ in $\Omega_\varepsilon$ we get
  \begin{align*}
    |\nabla_\Gamma M(|\varphi|^2)| \leq |M\bigl(B\nabla(|\varphi|^2)\bigr)|+|M\bigl(\partial_n(|\varphi|^2)\psi_\varepsilon\bigr)| \leq cM(|\varphi\nabla\varphi|) \quad\text{on}\quad \Gamma.
  \end{align*}
  Hence from \eqref{E:Ave_Lp_Surf} and H\"{o}lder's inequality we deduce that
  \begin{align} \label{Pf_Pr:L1_Grad}
    \begin{aligned}
      \|\nabla_\Gamma M(|\varphi|^2)\|_{L^1(\Gamma)} &\leq c\|M(|\varphi\nabla\varphi|)\|_{L^1(\Gamma)} \leq c\varepsilon^{-1}\|\varphi\nabla\varphi\|_{L^1(\Omega_\varepsilon)} \\
      &\leq c\varepsilon^{-1}\|\varphi\|_{L^2(\Omega_\varepsilon)}\|\nabla\varphi\|_{L^2(\Omega_\varepsilon)}.
    \end{aligned}
  \end{align}
  Now we observe that the Sobolev embedding $W^{1,1}(\Gamma)\hookrightarrow L^2(\Gamma)$ is valid since $\Gamma\subset\mathbb{R}^3$ is a two-dimensional compact surface without boundary (see e.g.~\cite[Theorem~2.20]{Au98}).
  By this fact, \eqref{Pf_Pr:L1_MPhi2}, \eqref{Pf_Pr:L1_Grad}, and $\|\varphi\|_{L^2(\Omega_\varepsilon)}\leq\|\varphi\|_{H^1(\Omega_\varepsilon)}$ we obtain
  \begin{align*}
    \|M(|\varphi|^2)\|_{L^2(\Gamma)} \leq c\|M(|\varphi|^2)\|_{W^{1,1}(\Gamma)} \leq c\varepsilon^{-1}\|\varphi\|_{L^2(\Omega_\varepsilon)}\|\varphi\|_{H^1(\Omega_\varepsilon)}.
  \end{align*}
  Finally, we apply the above inequality and \eqref{E:La_Surf} to \eqref{Pf_Pr:Est_L2} to get
  \begin{align*}
    \|\bar{\eta}\varphi\|_{L^2(\Omega_\varepsilon)}^2 \leq c\|\eta\|_{L^2(\Gamma)}\|\eta\|_{H^1(\Gamma)}\|\varphi\|_{L^2(\Omega_\varepsilon)}\|\varphi\|_{H^1(\Omega_\varepsilon)},
  \end{align*}
  which shows \eqref{E:Prod_Surf}.
\end{proof}

\begin{lemma} \label{L:Prod_Ua}
  For $\varphi\in H^1(\Omega_\varepsilon)$, $u\in H^1(\Omega_\varepsilon)^3$, and $u^a$ given by \eqref{E:Def_ExAve},
  \begin{align} \label{E:Prod_Ua}
    \bigl\|\,|u^a|\,\varphi\bigr\|_{L^2(\Omega_\varepsilon)} &\leq c\varepsilon^{-1/2}\|\varphi\|_{L^2(\Omega_\varepsilon)}^{1/2}\|\varphi\|_{H^1(\Omega_\varepsilon)}^{1/2}\|u\|_{L^2(\Omega_\varepsilon)}^{1/2}\|u\|_{H^1(\Omega_\varepsilon)}^{1/2}
  \end{align}
  with a constant $c>0$ independent of $\varepsilon$, $\varphi$, and $u$.
  If in addition $u\in H^2(\Omega_\varepsilon)^3$, then
  \begin{align} \label{E:Prod_Grad_Ua}
    \bigl\|\,|\nabla u^a|\,\varphi\bigr\|_{L^2(\Omega_\varepsilon)} &\leq c\varepsilon^{-1/2}\|\varphi\|_{L^2(\Omega_\varepsilon)}^{1/2}\|\varphi\|_{H^1(\Omega_\varepsilon)}^{1/2}\|u\|_{H^1(\Omega_\varepsilon)}^{1/2}\|u\|_{H^2(\Omega_\varepsilon)}^{1/2}.
  \end{align}
\end{lemma}

\begin{proof}
  By \eqref{E:ExAux_Bound} and $|M_\tau u|=|PMu|\leq|Mu|$ on $\Gamma$,
  \begin{align*}
    |u^a| \leq (1+|\Psi_\varepsilon|)\left|\overline{M_\tau u}\right| \leq c\left|\overline{Mu}\right| \quad\text{in}\quad \Omega_\varepsilon.
  \end{align*}
  We apply this inequality and \eqref{E:Prod_Surf} with $\eta=|Mu|$ to the $L^2$-norm of $|u^a|\,\varphi$ on $\Omega_\varepsilon$ and then use \eqref{E:Ave_Lp_Surf} and \eqref{E:Ave_Wmp_Surf} to obtain \eqref{E:Prod_Ua}.

  Let us prove \eqref{E:Prod_Grad_Ua}.
  We differentiate both sides of \eqref{E:Def_ExAve} to get
  \begin{align*}
    \nabla u^a = \nabla\Bigl(\overline{M_\tau u}\Bigr)+\left[\left\{\nabla\Bigl(\overline{M_\tau u}\Bigr)\right\}\Psi_\varepsilon+(\nabla\Psi_\varepsilon)\overline{M_\tau u}\right]\otimes\bar{n}+\Bigl(\overline{M_\tau u}\cdot\Psi_\varepsilon\Bigr)\nabla\bar{n}
  \end{align*}
  in $\Omega_\varepsilon$.
  Hence by \eqref{E:ConDer_Bound}, \eqref{E:NorG_Bound}, \eqref{E:ExAux_Bound}, $M_\tau u=PMu$ on $\Gamma$, and $P\in C^4(\Gamma)^{3\times 3}$,
  \begin{align*}
    |\nabla u^a| \leq c\left(\left|\overline{M_\tau u}\right|+\left|\overline{\nabla_\Gamma M_\tau u}\right|\right) \leq c\left(\left|\overline{Mu}\right|+\left|\overline{\nabla_\Gamma Mu}\right|\right) \quad\text{in}\quad \Omega_\varepsilon.
  \end{align*}
  Applying this inequality and \eqref{E:Prod_Surf} with $\eta=|Mu|$ and $\eta=|\nabla_\Gamma Mu|$ to the $L^2$-norm of $|\nabla u^a|\,\varphi$ on $\Omega_\varepsilon$ and then using \eqref{E:Ave_Lp_Surf} and \eqref{E:Ave_Wmp_Surf} we obtain \eqref{E:Prod_Grad_Ua}.
\end{proof}

Next we derive Poincar\'{e} type inequalities for the residual part $u^r=u-u^a$.
Since $u^a\cdot n_\varepsilon=0$ on $\Gamma_\varepsilon$ by Lemma~\ref{L:ExTan_Imp}, $u^r\cdot n_\varepsilon=0$ on $\Gamma_\varepsilon$ if $u$ itself satisfies the same boundary condition.
This is essential for derivation of Poincar\'{e} type inequalities.

\begin{lemma} \label{L:Po_Ur}
  Let $u\in H^1(\Omega_\varepsilon)^3$ satisfy \eqref{E:Bo_Imp} on $\Gamma_\varepsilon$.
  Then
  \begin{align} \label{E:Po_Ur}
    \|u^r\|_{L^2(\Omega_\varepsilon)} \leq c\varepsilon\|u^r\|_{H^1(\Omega_\varepsilon)}
  \end{align}
  for $u^r=u-u^a$, where $c>0$ is a constant independent of $\varepsilon$ and $u$.
\end{lemma}

\begin{proof}
  We use the notation \eqref{E:Def_U_TN} for the tangential and normal components (with respect to $\Gamma$) of a vector field on $\Omega_\varepsilon$.
  By \eqref{E:NorDer_Con}, \eqref{E:Def_ExAve}, and $u_\tau^r=\overline{P}u^r$ in $\Omega_\varepsilon$,
  \begin{align*}
    u_\tau^r = u_\tau-\overline{M_\tau u} = u_\tau-\overline{Mu_\tau}, \quad \partial_nu_\tau = \partial_nu_\tau^r = \overline{P}\partial_nu^r \quad\text{in}\quad \Omega_\varepsilon.
  \end{align*}
  Thus we apply \eqref{E:Ave_Diff_Dom} and $|Pa|\leq|a|$ on $\Gamma$ for $a\in\mathbb{R}^3$ to get
  \begin{align} \label{Pf_PUr:Est_Tau}
    \|u_\tau^r\|_{L^2(\Omega_\varepsilon)} = \left\|u_\tau-\overline{Mu_\tau}\right\|_{L^2(\Omega_\varepsilon)} \leq c\varepsilon\|\partial_nu_\tau\|_{L^2(\Omega_\varepsilon)} \leq c\varepsilon\|\partial_nu^r\|_{L^2(\Omega_\varepsilon)}.
  \end{align}
  Also, $u^r=u-u^a$ satisfies \eqref{E:Bo_Imp} on $\Gamma_\varepsilon$ by the assumption on $u$, since $u^a$ satisfies \eqref{E:Bo_Imp} on $\Gamma_\varepsilon$ by Lemma~\ref{L:ExTan_Imp} and \eqref{E:Def_ExAve}.
  Hence we can apply \eqref{E:Poin_Nor} to $u^r$ to get
  \begin{align*}
    \|u_n^r\|_{L^2(\Omega_\varepsilon)} = \|u^r\cdot\bar{n}\|_{L^2(\Omega_\varepsilon)} \leq c\varepsilon\|u^r\|_{H^1(\Omega_\varepsilon)}.
  \end{align*}
  By this inequality, \eqref{Pf_PUr:Est_Tau}, and
  \begin{align*}
    \|u^r\|_{L^2(\Omega_\varepsilon)}^2 = \|u_\tau^r\|_{L^2(\Omega_\varepsilon)}^2+\|u_n^r\|_{L^2(\Omega_\varepsilon)}^2, \quad \|\partial_nu^r\|_{L^2(\Omega_\varepsilon)}\leq c\|u^r\|_{H^1(\Omega_\varepsilon)}
  \end{align*}
  we conclude that \eqref{E:Po_Ur} is valid.
\end{proof}

\begin{lemma} \label{L:Po_Grad_Ur}
  Suppose that the inequalities \eqref{E:Fric_Upper} hold and $u\in H^2(\Omega_\varepsilon)^3$ satisfies $\mathrm{div}\,u=0$ in $\Omega_\varepsilon$ and the slip boundary conditions \eqref{E:Bo_Imp}--\eqref{E:Bo_Slip} on $\Gamma_\varepsilon$.
  Then
  \begin{align} \label{E:Po_Grad_Ur}
    \|\nabla u^r\|_{L^2(\Omega_\varepsilon)} \leq c\left(\varepsilon\|u\|_{H^2(\Omega_\varepsilon)}+\|u\|_{L^2(\Omega_\varepsilon)}\right)
  \end{align}
  for $u^r=u-u^a$, where $c>0$ is a constant independent of $\varepsilon$ and $u$.
\end{lemma}

\begin{proof}
  As in the proof of Lemma~\ref{L:Po_Ur} we use the notation \eqref{E:Def_U_TN}.
  Based on the equalities $u^r=u_\tau^r+u_n^r$ and $I_3=\overline{P}+\overline{Q}$ we split the gradient matrix of $u^r$ into
  \begin{align} \label{Pf_PGUr:Sum}
    \nabla u^r=\overline{P}\nabla u_\tau^r+\overline{Q}\nabla u_\tau^r+\overline{P}\nabla u_n^r+\overline{Q}\nabla u_n^r \quad\text{in}\quad \Omega_\varepsilon
  \end{align}
  and estimate the $L^2(\Omega_\varepsilon)$-norm of each term on the right-hand side.

  First we consider $\overline{P}\nabla u_\tau^r$.
  By $u_\tau^r=u_\tau-\overline{Mu_\tau}$ in $\Omega_\varepsilon$, \eqref{E:P_TGr}, \eqref{E:WReso_P}, and \eqref{E:ConDer_Dom},
  \begin{align*}
    \overline{P}\nabla u_\tau^r = \Bigl(\overline{P}\nabla u_\tau-\overline{\nabla_\Gamma Mu_\tau}\Bigr)-\left\{I_3-\Bigl(I_3-d\overline{W}\Bigr)^{-1}\right\}\overline{\nabla_\Gamma Mu_\tau} \quad\text{in}\quad \Omega_\varepsilon.
  \end{align*}
  Hence we apply \eqref{E:Wein_Diff} with $|d|\leq c\varepsilon$ in $\Omega_\varepsilon$, \eqref{E:Con_Lp}, \eqref{E:Ave_Wmp_Surf}, and \eqref{E:ADD_Dom} to get
  \begin{align} \label{Pf_PGUr:L2_PT}
    \begin{aligned}
      \left\|\overline{P}\nabla u_\tau^r\right\|_{L^2(\Omega_\varepsilon)} &\leq \left\|\overline{P}\nabla u_\tau-\overline{\nabla_\Gamma Mu_\tau}\right\|_{L^2(\Omega_\varepsilon)}+c\varepsilon\left\|\overline{\nabla_\Gamma Mu_\tau}\right\|_{L^2(\Omega_\varepsilon)} \\
      &\leq c\varepsilon\|u_\tau\|_{H^2(\Omega_\varepsilon)} \leq c\varepsilon\|u\|_{H^2(\Omega_\varepsilon)}.
    \end{aligned}
  \end{align}
  Here the last inequality follows from $u_\tau=\overline{P}u$ in $\Omega_\varepsilon$ and $P\in C^4(\Gamma)^{3\times3}$.

  Next we deal with $\overline{Q}\nabla u_\tau^r=\bar{n}\otimes\partial_nu_\tau^r$.
  By $u_\tau^r=\overline{P}u-\overline{M_\tau u}$ in $\Omega_\varepsilon$ and \eqref{E:NorDer_Con},
  \begin{gather*}
    \partial_nu_\tau^r = \overline{P}\partial_nu, \quad \left|\overline{Q}\nabla u_\tau^r\right| = |\bar{n}\otimes\partial_nu_\tau^r| = |\partial_nu_\tau^r| = \left|\overline{P}\partial_nu\right| \quad\text{in}\quad \Omega_\varepsilon.
  \end{gather*}
  Since we assume that the inequalities \eqref{E:Fric_Upper} are valid and $u$ satisfies \eqref{E:Bo_Imp}--\eqref{E:Bo_Slip} on $\Gamma_\varepsilon$, we can apply \eqref{E:PDnU_WU} to $u$.
  Hence we obtain
  \begin{align} \label{Pf_PGUr:L2_QT}
    \begin{aligned}
      \left\|\overline{Q}\nabla u_\tau^r\right\|_{L^2(\Omega_\varepsilon)} = \left\|\overline{P}\partial_nu\right\|_{L^2(\Omega_\varepsilon)} &\leq \left\|\overline{P}\partial_nu+\overline{W}u\right\|_{L^2(\Omega_\varepsilon)}+\left\|\overline{W}u\right\|_{L^2(\Omega_\varepsilon)} \\
      &\leq c\left(\varepsilon\|u\|_{H^2(\Omega_\varepsilon)}+\|u\|_{L^2(\Omega_\varepsilon)}\right).
    \end{aligned}
  \end{align}
  Let us estimate the $L^2(\Omega_\varepsilon)$-norm of $\overline{P}\nabla u_n^r$.
  Since $u_n^r=(u^r\cdot\bar{n})\bar{n}$ in $\Omega_\varepsilon$ we have
  \begin{align*}
    \overline{P}\nabla u_n^r= \Bigl[\overline{P}\nabla(u^r\cdot\bar{n})\Bigr]\otimes\bar{n}-(u^r\cdot\bar{n})\overline{P}\nabla\bar{n} \quad\text{in}\quad \Omega_\varepsilon.
  \end{align*}
  By this formula, \eqref{E:NorG_Bound}, and $|n|=1$ and $|P|=2$ on $\Gamma$,
  \begin{align*}
    \left\|\overline{P}\nabla u_n^r\right\|_{L^2(\Omega_\varepsilon)} \leq c\left(\left\|\overline{P}\nabla(u^r\cdot\bar{n})\right\|_{L^2(\Omega_\varepsilon)}+\|u^r\|_{L^2(\Omega_\varepsilon)}\right).
  \end{align*}
  Here $u^r$ satisfies \eqref{E:Bo_Imp} on $\Gamma_\varepsilon$ by the assumption on $u$ since $u^a$ satisfies \eqref{E:Bo_Imp} on $\Gamma_\varepsilon$ by Lemma~\ref{L:ExTan_Imp} and \eqref{E:Def_ExAve}.
  Thus we apply \eqref{E:Poin_Dnor} and \eqref{E:Po_Ur} to the above inequality and then use \eqref{E:Hm_UaUr} to get
  \begin{align} \label{Pf_PGUr:L2_PN}
    \left\|\overline{P}\nabla u_n^r\right\|_{L^2(\Omega_\varepsilon)} \leq c\varepsilon\|u^r\|_{H^2(\Omega_\varepsilon)} \leq c\varepsilon\|u\|_{H^2(\Omega_\varepsilon)}.
  \end{align}
  Now let us consider $\overline{Q}\nabla u_n^r=\bar{n}\otimes\partial_nu_n^r$.
  Since
  \begin{align*}
    u_n^r = (u^r\cdot\bar{n})\bar{n} = \Bigl(u\cdot\bar{n}-\overline{M_\tau u}\cdot\Psi_\varepsilon\Bigr)\bar{n}, \quad \partial_nu_n^r = \Bigl(\partial_nu\cdot\bar{n}-\overline{M_\tau u}\cdot\partial_n\Psi_\varepsilon\Bigr)\bar{n}
  \end{align*}
  in $\Omega_\varepsilon$ by \eqref{E:NorDer_Con} and \eqref{E:Def_ExAve}, we have
  \begin{align*}
    \left|\overline{Q}\nabla u_n^r\right| &= |\partial_nu_n^r| = \left|\partial_nu\cdot\bar{n}-\overline{M_\tau u}\cdot\partial_n\Psi_\varepsilon\right| \\
    &\leq \left|\partial_nu\cdot\bar{n}-\frac{1}{\bar{g}}\overline{M_\tau u}\cdot\overline{\nabla_\Gamma g}\right|+\left|\overline{M_\tau u}\right|\left|\partial_n\Psi_\varepsilon-\frac{1}{\bar{g}}\overline{\nabla_\Gamma g}\right|
  \end{align*}
  in $\Omega_\varepsilon$.
  By this inequality, $|M_\tau u|\leq |Mu|$ on $\Gamma$, \eqref{E:ExAux_TNDer}, \eqref{E:Ave_Lp_Dom}, and \eqref{E:DnU_N_Ave},
  \begin{align} \label{Pf_PGUr:L2_QN}
    \begin{aligned}
      \left\|\overline{Q}\nabla u_n^r\right\|_{L^2(\Omega_\varepsilon)} &\leq \left\|\partial_nu\cdot\bar{n}-\frac{1}{\bar{g}}\overline{M_\tau u}\cdot\overline{\nabla_\Gamma g}\right\|_{L^2(\Omega_\varepsilon)}+c\varepsilon\left\|\overline{Mu}\right\|_{L^2(\Omega_\varepsilon)} \\
      &\leq c\varepsilon\|u\|_{H^2(\Omega_\varepsilon)}.
    \end{aligned}
  \end{align}
  Here we used the conditions on $u$ except for \eqref{E:Bo_Slip} to apply \eqref{E:DnU_N_Ave}.
  Finally, applying \eqref{Pf_PGUr:L2_PT}--\eqref{Pf_PGUr:L2_QN} to \eqref{Pf_PGUr:Sum} we get \eqref{E:Po_Grad_Ur}.
\end{proof}

As a consequence of Lemmas~\ref{L:Po_Ur} and~\ref{L:Po_Grad_Ur}, we obtain a good $L^\infty$-estimate for the residual part $u^r$ on $\Omega_\varepsilon$.

\begin{lemma} \label{L:Linf_Ur}
  Suppose that the inequalities \eqref{E:Fric_Upper} hold and $u\in H^2(\Omega_\varepsilon)^3$ satisfies $\mathrm{div}\,u=0$ in $\Omega_\varepsilon$ and \eqref{E:Bo_Imp}--\eqref{E:Bo_Slip} on $\Gamma_\varepsilon$.
  Then
  \begin{align} \label{E:Linf_Ur}
    \|u^r\|_{L^\infty(\Omega_\varepsilon)} \leq c\left(\varepsilon^{1/2}\|u\|_{H^2(\Omega_\varepsilon)}+\|u\|_{L^2(\Omega_\varepsilon)}^{1/2}\|u\|_{H^2(\Omega_\varepsilon)}^{1/2}\right)
  \end{align}
  for $u^r=u-u^a$, where $c>0$ is a constant independent of $\varepsilon$ and $u$.
\end{lemma}

\begin{proof}
  Since $u^r\in H^2(\Omega_\varepsilon)^3$ by Lemma~\ref{L:Hm_UaUr}, we have
  \begin{multline*}
    \|u^r\|_{L^\infty(\Omega_\varepsilon)} \leq c\varepsilon^{-1/2}\|u^r\|_{L^2(\Omega_\varepsilon)}^{1/4}\|u^r\|_{H^2(\Omega_\varepsilon)}^{1/2} \\
    \times\left(\|u^r\|_{L^2(\Omega_\varepsilon)}+\varepsilon\|\partial_nu^r\|_{L^2(\Omega_\varepsilon)}+\varepsilon^2\|\partial_n^2u^r\|_{L^2(\Omega_\varepsilon)}\right)^{1/4}
  \end{multline*}
  by Agmon's inequality \eqref{E:Agmon}.
  To the right-hand side we apply
  \begin{align*}
    \|u^r\|_{L^2(\Omega_\varepsilon)} &\leq c\varepsilon\left(\|u^r\|_{L^2(\Omega_\varepsilon)}+\|\nabla u^r\|_{L^2(\Omega_\varepsilon)}\right) \leq c\varepsilon\left(\varepsilon\|u\|_{H^2(\Omega_\varepsilon)}+\|u\|_{L^2(\Omega_\varepsilon)}\right), \\
    \|\partial_nu^r\|_{L^2(\Omega_\varepsilon)} &\leq c\|\nabla u^r\|_{L^2(\Omega_\varepsilon)} \leq c\left(\varepsilon\|u\|_{H^2(\Omega_\varepsilon)}+\|u\|_{L^2(\Omega_\varepsilon)}\right), \\
    \|\partial_n^2u^r\|_{L^2(\Omega_\varepsilon)} &\leq c\|u^r\|_{H^2(\Omega_\varepsilon)} \leq c\|u\|_{H^2(\Omega_\varepsilon)}
  \end{align*}
  by \eqref{E:Hm_UaUr}, \eqref{E:Po_Ur}, and \eqref{E:Po_Grad_Ur} to get
  \begin{align*}
    \|u^r\|_{L^\infty(\Omega_\varepsilon)} \leq c\left(\varepsilon\|u\|_{H^2(\Omega_\varepsilon)}+\|u\|_{L^2(\Omega_\varepsilon)}\right)^{1/2}\|u\|_{H^2(\Omega_\varepsilon)}^{1/2}.
  \end{align*}
  Using $(a+b)^{1/2}\leq a^{1/2}+b^{1/2}$ for $a,b\geq0$ to this inequality we obtain \eqref{E:Linf_Ur}.
\end{proof}

Finally, let us estimate the $L^2(\Omega_\varepsilon)$-norm of $u\otimes u$ and $(u\cdot\nabla)u$ by using the product estimate for the average part and the $L^\infty$-estimate for the residual part.

\begin{lemma} \label{L:Est_UU}
  Suppose that the inequalities \eqref{E:Fric_Upper} hold.
  Then there exists a constant $c>0$ independent of $\varepsilon$ such that
  \begin{multline} \label{E:Est_UU}
    \|u\otimes u\|_{L^2(\Omega_\varepsilon)} \leq c\left(\varepsilon^{-1/2}\|u\|_{L^2(\Omega_\varepsilon)}\|u\|_{H^1(\Omega_\varepsilon)}\right. \\
    \left.+\varepsilon^{1/2}\|u\|_{L^2(\Omega_\varepsilon)}\|u\|_{H^2(\Omega_\varepsilon)}+\|u\|_{L^2(\Omega_\varepsilon)}^{3/2}\|u\|_{H^2(\Omega_\varepsilon)}^{1/2}\right)
  \end{multline}
  and
  \begin{align} \label{E:Est_UGU}
    \|(u\cdot\nabla)u\|_{L^2(\Omega_\varepsilon)} \leq c\left(\varepsilon^{-1/2}\|u\|_{L^2(\Omega_\varepsilon)}+\varepsilon^{1/2}\|u\|_{H^1(\Omega_\varepsilon)}\right)\|u\|_{H^2(\Omega_\varepsilon)}
  \end{align}
  for all  $u\in H^2(\Omega_\varepsilon)^3$ satisfying $\mathrm{div}\,u=0$ in $\Omega_\varepsilon$ and \eqref{E:Bo_Imp}--\eqref{E:Bo_Slip} on $\Gamma_\varepsilon$.
\end{lemma}

\begin{proof}
  Let $u^a$ be the average part of $u$ given by \eqref{E:Def_ExAve} and $u^r=u-u^a$ the residual part.
  Since $u\in H^2(\Omega_\varepsilon)^3$ satisfies $\mathrm{div}\,u=0$ in $\Omega_\varepsilon$ and \eqref{E:Bo_Imp}--\eqref{E:Bo_Slip} on $\Gamma_\varepsilon$,
  \begin{align*}
    \|u^a\otimes u\|_{L^2(\Omega_\varepsilon)} &\leq c\varepsilon^{-1/2}\|u\|_{L^2(\Omega_\varepsilon)}\|u\|_{H^1(\Omega_\varepsilon)}, \\
    \|u^r\otimes u\|_{L^2(\Omega_\varepsilon)} &\leq \|u^r\|_{L^\infty(\Omega_\varepsilon)}\|u\|_{L^2(\Omega_\varepsilon)} \\
    &\leq c\left(\varepsilon^{1/2}\|u\|_{H^2(\Omega_\varepsilon)}+\|u\|_{L^2(\Omega_\varepsilon)}^{1/2}\|u\|_{H^2(\Omega_\varepsilon)}^{1/2}\right)\|u\|_{L^2(\Omega_\varepsilon)}
  \end{align*}
  by \eqref{E:Prod_Ua} and \eqref{E:Linf_Ur}.
  Applying these inequalities to the right-hand side of
  \begin{align*}
    \|u\otimes u\|_{L^2(\Omega_\varepsilon)} \leq \|u^a\otimes u\|_{L^2(\Omega_\varepsilon)}+\|u^r\otimes u\|_{L^2(\Omega_\varepsilon)}
  \end{align*}
  we obtain \eqref{E:Est_UU}.
  To prove \eqref{E:Est_UGU} we observe by \eqref{E:St_Inter} and \eqref{E:Prod_Ua} that
  \begin{align*}
    \|(u^a\cdot\nabla)u\|_{L^2(\Omega_\varepsilon)} &\leq c\varepsilon^{-1/2}\|u\|_{L^2(\Omega_\varepsilon)}^{1/2}\|u\|_{H^1(\Omega_\varepsilon)}\|u\|_{H^2(\Omega_\varepsilon)}^{1/2} \\
    &\leq c\varepsilon^{-1/2}\|u\|_{L^2(\Omega_\varepsilon)}\|u\|_{H^2(\Omega_\varepsilon)}.
  \end{align*}
  Also, by \eqref{E:St_Inter} and \eqref{E:Linf_Ur} we get
  \begin{align*}
    \|(u^r\cdot\nabla)u\|_{L^2(\Omega_\varepsilon)} &\leq c\|u^r\|_{L^\infty(\Omega_\varepsilon)}\|\nabla u\|_{L^2(\Omega_\varepsilon)} \\
    &\leq c\left(\varepsilon^{1/2}\|u\|_{H^2(\Omega_\varepsilon)}+\|u\|_{L^2(\Omega_\varepsilon)}^{1/2}\|u\|_{H^2(\Omega_\varepsilon)}^{1/2}\right)\|u\|_{H^1(\Omega_\varepsilon)} \\
    &= c\left(\|u\|_{L^2(\Omega_\varepsilon)}^{1/2}\|u\|_{H^1(\Omega_\varepsilon)}\|u\|_{H^2(\Omega_\varepsilon)}^{1/2}+\varepsilon^{1/2}\|u\|_{H^1(\Omega_\varepsilon)}\|u\|_{H^2(\Omega_\varepsilon)}\right) \\
    &\leq c\left(\|u\|_{L^2(\Omega_\varepsilon)}+\varepsilon^{1/2}\|u\|_{H^1(\Omega_\varepsilon)}\right)\|u\|_{H^2(\Omega_\varepsilon)}.
  \end{align*}
  We use these two estimates to the right-hand side of
  \begin{align*}
    \|(u\cdot\nabla)u\|_{L^2(\Omega_\varepsilon)} \leq \|(u^a\cdot\nabla)u\|_{L^2(\Omega_\varepsilon)}+\|(u^r\cdot\nabla)u\|_{L^2(\Omega_\varepsilon)}
  \end{align*}
  and note that $1\leq \varepsilon^{-1/2}$ by $\varepsilon<1$ to obtain \eqref{E:Est_UGU}.
\end{proof}

\subsection{Average of bilinear and trilinear forms} \label{SS:Ave_BT}
We consider approximation of bilinear and trilinear forms for functions on $\Omega_\varepsilon$ by those for functions on $\Gamma$ and the average operators.
The results in this subsection are fundamental for the study of a singular limit problem for \eqref{E:NS_Eq}--\eqref{E:NS_In}.
As in the previous subsections we denote by $\bar{\eta}=\eta\circ\pi$ the constant extension of a function $\eta$ on $\Gamma$, use the notations \eqref{E:Pull_Dom} and \eqref{E:Pull_Bo}, and suppress the arguments of functions.

First we consider the $L^2$-inner products on $\Omega_\varepsilon$ and $\Gamma_\varepsilon^i$, $i=0,1$.

\begin{lemma} \label{L:Ave_BiL2_Dom}
  There exists a constant $c>0$ independent of $\varepsilon$ such that
  \begin{align} \label{E:Ave_BiL2_Dom}
    \left|\int_{\Omega_\varepsilon}\varphi\bar{\eta}\,dx-\varepsilon\int_\Gamma g(M\varphi)\eta\,d\mathcal{H}^2\right| \leq c\varepsilon^{3/2}\|\varphi\|_{L^2(\Omega_\varepsilon)}\|\eta\|_{L^2(\Gamma)}
  \end{align}
  for all $\varphi\in L^2(\Omega_\varepsilon)$ and $\eta\in L^2(\Gamma)$.
\end{lemma}

\begin{proof}
  By the formula \eqref{E:CoV_Dom} and the definition \eqref{E:Def_Ave} of $M$,
  \begin{align*}
    \int_{\Omega_\varepsilon}\varphi\bar{\eta}\,dx-\varepsilon\int_\Gamma g(M\varphi)\eta\,d\mathcal{H}^2 = \int_\Gamma\int_{\varepsilon g_0}^{\varepsilon g_1}\varphi^\sharp\eta(J-1)\,dr\,d\mathcal{H}^2.
  \end{align*}
  We use \eqref{E:Jac_Diff}, \eqref{E:CoV_Equiv}, H\"{o}lder's inequality, and \eqref{E:Con_Lp} to the right-hand side to get
  \begin{align*}
    \left|\int_\Gamma\int_{\varepsilon g_0}^{\varepsilon g_1}\varphi^\sharp\eta(J-1)\,dr\,d\mathcal{H}^2\right| \leq c\varepsilon\|\varphi\|_{L^2(\Omega_\varepsilon)}\|\bar{\eta}\|_{L^2(\Omega_\varepsilon)} \leq c\varepsilon^{3/2}\|\varphi\|_{L^2(\Omega_\varepsilon)}\|\eta\|_{L^2(\Gamma)}.
  \end{align*}
  Hence we obtain \eqref{E:Ave_BiL2_Dom}.
\end{proof}

\begin{lemma} \label{L:Ave_BiL2_Bo}
  There exists a constant $c>0$ independent of $\varepsilon$ such that
  \begin{align} \label{E:Ave_BiL2_Bo}
    \left|\int_{\Gamma_\varepsilon^i}\varphi\bar{\eta}\,d\mathcal{H}^2-\int_\Gamma(M\varphi)\eta\,d\mathcal{H}^2\right| \leq \varepsilon^{1/2}\|\varphi\|_{H^1(\Omega_\varepsilon)}\|\eta\|_{L^2(\Gamma)}, \quad i=0,1
  \end{align}
  for all $\varphi\in H^1(\Omega_\varepsilon)$ and $\eta\in L^2(\Gamma)$.
\end{lemma}

\begin{proof}
  Using the change of variables formula \eqref{E:CoV_Surf} we have
  \begin{align} \label{Pf_ABLB:Split}
    \int_{\Gamma_\varepsilon^i}\varphi\bar{\eta}\,d\mathcal{H}^2-\int_\Gamma(M\varphi)\eta\,d\mathcal{H}^2 = K_1+K_2,
  \end{align}
  where we use the notations \eqref{E:Pull_Bo} and $J_i^\sharp(y):=J(y,\varepsilon g_i(y))$ for $y\in\Gamma$ to define
  \begin{align*}
    K_1 := \int_\Gamma\varphi_i^\sharp\eta\left(J_i^\sharp\sqrt{1+\varepsilon^2|\tau_\varepsilon^i|^2}-1\right)d\mathcal{H}^2, \quad K_2 := \int_\Gamma(\varphi_i^\sharp-M\varphi)\eta\,d\mathcal{H}^2.
  \end{align*}
  By \eqref{E:Tau_Bound}, \eqref{E:Jac_Diff}, and \eqref{Pf_KG:Sqrt} we observe that
  \begin{align*}
    \left|J_i^\sharp\sqrt{1+\varepsilon^2|\tau_\varepsilon^i|^2}-1\right| \leq |J_i^\sharp-1|\sqrt{1+\varepsilon^2|\tau_\varepsilon^i|^2}+\left(\sqrt{1+\varepsilon^2|\tau_\varepsilon^i|^2}-1\right) \leq c\varepsilon
  \end{align*}
  on $\Gamma$.
  From this inequality, \eqref{E:Lp_CoV_Surf}, and \eqref{E:Poin_Bo} it follows that
  \begin{align*}
    |K_1| \leq c\varepsilon\|\varphi_i^\sharp\|_{L^2(\Gamma)}\|\eta\|_{L^2(\Gamma)} \leq c\varepsilon\|\varphi\|_{L^2(\Gamma_\varepsilon^i)}\|\eta\|_{L^2(\Gamma)} \leq c\varepsilon^{1/2}\|\varphi\|_{H^1(\Omega_\varepsilon)}\|\eta\|_{L^2(\Gamma)}.
  \end{align*}
  Also, by \eqref{E:Lp_CoV_Surf} and \eqref{E:Ave_Diff_Bo},
  \begin{align*}
    |K_2| &\leq \|\varphi_i^\sharp-M\varphi\|_{L^2(\Gamma)}\|\eta\|_{L^2(\Gamma)} \leq c\left\|\varphi-\overline{M\varphi}\right\|_{L^2(\Gamma_\varepsilon^i)}\|\eta\|_{L^2(\Gamma)} \\
    &\leq c\varepsilon^{1/2}\|\varphi\|_{H^1(\Omega_\varepsilon)}\|\eta\|_{L^2(\Gamma)}.
  \end{align*}
  Applying these two estimates to \eqref{Pf_ABLB:Split} we obtain \eqref{E:Ave_BiL2_Bo}.
\end{proof}

Next we deal with bilinear forms including the strain rate tensor
\begin{align*}
  D(u) = (\nabla u)_S = \frac{\nabla u+(\nabla u)^T}{2}
\end{align*}
for a vector field $u$ on $\Omega_\varepsilon$.

\begin{lemma} \label{L:Ave_BiH1_TT}
  There exists a constant $c>0$ independent of $\varepsilon$ such that
  \begin{align} \label{E:Ave_BiH1_TT}
    \left|\int_{\Omega_\varepsilon}D(u):\overline{A}\,dx-\varepsilon\int_\Gamma gD_\Gamma(M_\tau u):A\,d\mathcal{H}^2\right| \leq c\varepsilon^{3/2}\|u\|_{H^1(\Omega_\varepsilon)}\|A\|_{L^2(\Gamma)}
  \end{align}
  for all $u\in H^1(\Omega_\varepsilon)^3$ satisfying \eqref{E:Bo_Imp} on $\Gamma_\varepsilon$ and $A\in L^2(\Gamma)^{3\times3}$ satisfying
  \begin{align} \label{E:A_Cond}
     PA = AP = A \quad\text{on}\quad \Gamma.
   \end{align}
  Here $D_\Gamma(M_\tau u)$ is the surface strain rate tensor given by \eqref{E:Strain_Surf}.
\end{lemma}

\begin{proof}
  By \eqref{E:A_Cond} and $P^T=P$ on $\Gamma$,
  \begin{align*}
    PM\bigl(D(u)\bigr)P:A = M\bigl(D(u)\bigr):A \quad\text{on}\quad \Gamma.
  \end{align*}
  From this equality, \eqref{E:Ave_BiL2_Dom}, and $\|D(u)\|_{L^2(\Omega_\varepsilon)}\leq c\|u\|_{H^1(\Omega_\varepsilon)}$ it follows that
  \begin{align} \label{Pf_ABTT:First}
    \left|\int_{\Omega_\varepsilon}D(u):\overline{A}\,dx-\varepsilon\int_\Gamma gPM\bigl(D(u)\bigr)P:A\,d\mathcal{H}^2\right| \leq c\varepsilon^{3/2}\|u\|_{H^1(\Omega_\varepsilon)}\|A\|_{L^2(\Gamma)}.
  \end{align}
  By \eqref{E:Ave_Der}, \eqref{E:Ave_Der_Aux}, \eqref{E:ADA_Bound}, and $|d|\leq c\varepsilon$ in $\Omega_\varepsilon$ we have
  \begin{align} \label{Pf_ABTT:Diff_AD}
    |\nabla_\Gamma Mu-PM(\nabla u)| &\leq \left|M\Bigl(d\overline{W}\nabla u\Bigr)\right|+|M(\psi_\varepsilon\otimes \partial_nu)| \leq c\varepsilon M(|\nabla u|)
  \end{align}
  on $\Gamma$.
  Noting that $P\nabla_\Gamma Mu=\nabla_\Gamma Mu$ and $|P|=2$ on $\Gamma$, we see by \eqref{Pf_ABTT:Diff_AD} that
  \begin{align*}
    \left|D_\Gamma(Mu)-PM\bigl(D(u)\bigr)P\right| \leq |\{\nabla_\Gamma Mu-PM(\nabla u)\}P| \leq c\varepsilon M(|\nabla u|) \quad\text{on}\quad \Gamma.
  \end{align*}
  From this inequality, the boundedness of $g$ on $\Gamma$, and \eqref{E:Ave_Lp_Surf} we deduce that
  \begin{multline} \label{Pf_ABTT:Second}
    \left|\int_\Gamma gD_\Gamma(Mu):A\,d\mathcal{H}^2-\int_\Gamma gPM\bigl(D(u)\bigr)P:A\,d\mathcal{H}^2\right| \\
    \leq c\varepsilon\|M(|\nabla u|)\|_{L^2(\Gamma)}\|A\|_{L^2(\Gamma)} \leq c\varepsilon^{1/2}\|u\|_{H^1(\Omega_\varepsilon)}\|A\|_{L^2(\Gamma)}.
  \end{multline}
  Moreover, by $Mu=(Mu\cdot n)n+M_\tau u$ and $-\nabla_\Gamma n=W$ on $\Gamma$ we get
  \begin{align*}
    \nabla_\Gamma Mu = \nabla_\Gamma(Mu\cdot n)\otimes n-(Mu\cdot n)W+\nabla_\Gamma M_\tau u \quad\text{on}\quad \Gamma.
  \end{align*}
  Since $(a\otimes n)P=a\otimes(P^Tn)=a\otimes(Pn)=0$ on $\Gamma$ for any $a\in\mathbb{R}^3$,
  \begin{align*}
    P(\nabla_\Gamma Mu)P-P(\nabla_\Gamma M_\tau u)P = -(Mu\cdot n)PWP = -(Mu\cdot n)W \quad\text{on}\quad \Gamma
  \end{align*}
  by \eqref{E:Form_W}.
  From this equality and the boundedness of $W$ on $\Gamma$ it follows that
  \begin{align*}
    |D_\Gamma(Mu)-D_\Gamma(M_\tau u)| \leq |P(\nabla_\Gamma Mu)P-P(\nabla_\Gamma M_\tau u)P| \leq c|Mu\cdot n| \quad\text{on}\quad \Gamma.
  \end{align*}
  By this inequality and \eqref{E:Ave_N_Lp} (note that $u$ satisfies \eqref{E:Bo_Imp} on $\Gamma_\varepsilon$) we get
  \begin{align*}
    \left|\int_\Gamma gD_\Gamma(Mu):A\,d\mathcal{H}^2-\int_\Gamma gD_\Gamma(M_\tau u):A\,d\mathcal{H}^2\right| &\leq c\|Mu\cdot n\|_{L^2(\Gamma)}\|A\|_{L^2(\Gamma)} \\
    &\leq c\varepsilon^{1/2}\|u\|_{H^1(\Omega_\varepsilon)}\|A\|_{L^2(\Gamma)}.
  \end{align*}
  Combining this inequality, \eqref{Pf_ABTT:First}, and \eqref{Pf_ABTT:Second} we obtain \eqref{E:Ave_BiH1_TT}.
\end{proof}

\begin{lemma} \label{L:Ave_BiH1_NN}
  There exists a constant $c>0$ independent of $\varepsilon$ such that
  \begin{align} \label{E:Ave_BiH1_NN}
    \left|\int_{\Omega_\varepsilon}\Bigl(D(u):\overline{Q}\Bigr)\bar{\eta}\,dx-\varepsilon\int_\Gamma(M_\tau u\cdot\nabla_\Gamma g)\eta\,d\mathcal{H}^2\right| \leq c\varepsilon^{3/2}\|u\|_{H^1(\Omega_\varepsilon)}\|\eta\|_{L^2(\Gamma)}
  \end{align}
  for all $u\in H^1(\Omega_\varepsilon)^3$ satisfying \eqref{E:Bo_Imp} on $\Gamma_\varepsilon$ and $\eta\in L^2(\Gamma)$.
\end{lemma}

\begin{proof}
  By the change of variables formula \eqref{E:CoV_Dom}, the inequalities \eqref{E:Jac_Diff}, \eqref{E:CoV_Equiv}, and \eqref{E:Con_Lp}, and $|Q|=1$ on $\Gamma$ we have
  \begin{multline} \label{Pf_ABNN:First}
    \left|\int_{\Omega_\varepsilon}\Bigl(D(u):\overline{Q}\Bigr)\bar{\eta}\,dx-\int_\Gamma\left(\int_{\varepsilon g_0}^{\varepsilon g_1}D(u)^\sharp:Q\,dr\right)\eta\,d\mathcal{H}^2\right| \\
    \leq c\varepsilon^{3/2}\|u\|_{H^1(\Omega_\varepsilon)}\|\eta\|_{L^2(\Gamma)}.
  \end{multline}
  Let us compute the integral of $D(u)^\sharp(y,r):Q(y)$ in $r$.
  Since $Q=n\otimes n$ is symmetric,
  \begin{align*}
    D(u)(y+rn(y)): Q(y) &= \nabla u(y+rn(y)):n(y)\otimes n(y) \\
    &= [(n(y)\cdot\nabla)u](y+rn(y))\cdot n(y) \\
    &= \frac{\partial}{\partial r}\Bigl(u(y+rn(y))\Bigr)\cdot n(y)
  \end{align*}
  for all $y\in\Gamma$ and $r\in(\varepsilon g_0(y),\varepsilon g_1(y))$, i.e. $D(u)^\sharp:Q=(\partial u^\sharp/\partial r)\cdot n$.
  Hence
  \begin{align} \label{Pf_ABNN:Int_R}
    \int_{\varepsilon g_0}^{\varepsilon g_1}D(u)^\sharp:Q\,dr = \left(\int_{\varepsilon g_0}^{\varepsilon g_1}\frac{\partial u^\sharp}{\partial r}\,dr\right)\cdot n = u_1^\sharp\cdot n-u_0^\sharp\cdot n \quad\text{on}\quad \Gamma.
  \end{align}
  Here and hereafter we use the notation \eqref{E:Pull_Bo}.
  Since $u$ satisfies \eqref{E:Bo_Imp} on $\Gamma_\varepsilon$,
  \begin{align*}
    (-1)^{i+1}u_i^\sharp\cdot n = u_i^\sharp\cdot\{(-1)^{i+1}(n-\varepsilon\nabla_\Gamma g_i)-n_{\varepsilon,i}^\sharp\}+\varepsilon(-1)^{i+1}u_i^\sharp\cdot\nabla_\Gamma g_i
  \end{align*}
  on $\Gamma$ for $i=0,1$.
  From this equality and $g=g_1-g_0$ on $\Gamma$ it follows that
  \begin{align*}
    u_1^\sharp\cdot n-u_0^\sharp\cdot n &= \sum_{i=0,1}(-1)^{i+1}u_i^\sharp\cdot n \\
    &= \sum_{i=0,1}u_i^\sharp\cdot\{(-1)^{i+1}(n-\varepsilon\nabla_\Gamma g_i)-n_{\varepsilon,i}^\sharp\} \\
    &\qquad +\varepsilon\sum_{i=0,1}(-1)^{i+1}(u_i^\sharp-Mu)\cdot\nabla_\Gamma g_i+\varepsilon Mu\cdot\nabla_\Gamma g
  \end{align*}
  on $\Gamma$.
  Applying \eqref{E:Comp_N} to the second line we get
  \begin{align} \label{Pf_ABNN:Diff}
    |(u_1^\sharp\cdot n-u_0^\sharp\cdot n)-\varepsilon Mu\cdot\nabla_\Gamma g| \leq c\varepsilon\sum_{i=0,1}(\varepsilon|u_i^\sharp|+|u_i^\sharp-Mu|) \quad\text{on}\quad \Gamma.
  \end{align}
  Combining \eqref{Pf_ABNN:Int_R} and \eqref{Pf_ABNN:Diff} and using H\"{o}lder's inequality we see that
  \begin{multline*}
    \left|\int_\Gamma\left(\int_{\varepsilon g_0}^{\varepsilon g_1}D(u)^\sharp:Q\,dr\right)\eta\,d\mathcal{H}^2-\varepsilon\int_\Gamma(Mu\cdot\nabla_\Gamma g)\eta\,d\mathcal{H}^2\right| \\
    \leq c\varepsilon\sum_{i=0,1}\left(\varepsilon\|u_i^\sharp\|_{L^2(\Gamma)}+\|u_i^\sharp-Mu\|_{L^2(\Gamma)}\right)\|\eta\|_{L^2(\Gamma)}.
  \end{multline*}
  Moreover, by \eqref{E:Lp_CoV_Surf}, \eqref{E:Poin_Bo}, and \eqref{E:Ave_Diff_Bo},
  \begin{align*}
    \varepsilon\|u_i^\sharp\|_{L^2(\Gamma)}+\|u_i^\sharp-Mu\|_{L^2(\Gamma)} &\leq c\left(\varepsilon\|u\|_{L^2(\Gamma_\varepsilon^i)}+\left\|u-\overline{Mu}\right\|_{L^2(\Gamma_\varepsilon^i)}\right) \\
    &\leq c\varepsilon^{1/2}\|u\|_{H^1(\Omega_\varepsilon)}.
  \end{align*}
  From the above two estimates we deduce that
  \begin{multline} \label{Pf_ABNN:Second}
    \left|\int_\Gamma\left(\int_{\varepsilon g_0}^{\varepsilon g_1}D(u)^\sharp:Q\,dr\right)\eta\,d\mathcal{H}^2-\varepsilon\int_\Gamma(Mu\cdot\nabla_\Gamma g)\eta\,d\mathcal{H}^2\right| \\
    \leq c\varepsilon^{3/2}\|u\|_{H^1(\Omega_\varepsilon)}\|\eta\|_{L^2(\Gamma)}.
  \end{multline}
  Since $\nabla_\Gamma g$ is tangential on $\Gamma$, we have $Mu\cdot\nabla_\Gamma g=M_\tau u\cdot\nabla_\Gamma g$ on $\Gamma$ and thus the inequality \eqref{E:Ave_BiH1_NN} follows from \eqref{Pf_ABNN:First} and \eqref{Pf_ABNN:Second}.
\end{proof}

\begin{lemma} \label{L:Ave_BiH1_TN}
  Let $u^\varepsilon\in H^2(\Omega_\varepsilon)^3$ and $v\in L^2(\Gamma,T\Gamma)$.
  Suppose that the inequalities \eqref{E:Fric_Upper} are valid and $u$ satisfies \eqref{E:Bo_Slip} on $\Gamma_\varepsilon^0$ or on $\Gamma_\varepsilon^1$.
  Then we have
  \begin{align} \label{E:Ave_BiH1_TN}
    \left|\int_{\Omega_\varepsilon}D(u):\bar{v}\otimes\bar{n}\,dx\right| \leq c\varepsilon^{3/2}\|u\|_{H^2(\Omega_\varepsilon)}\|v\|_{L^2(\Gamma)},
  \end{align}
  where $c>0$ is a constant independent of $\varepsilon$, $u$, and $v$.
\end{lemma}

\begin{proof}
  Since $v$ is tangential on $\Gamma$,
  \begin{align*}
    D(u):\bar{v}\otimes\bar{n} = \mathrm{tr}[D(u)^T(\bar{v}\otimes\bar{n})] = (D(u)^T\bar{v})\cdot\bar{n} = \bar{v}\cdot(D(u)\bar{n}) = \bar{v}\cdot\overline{P}D(u)\bar{n}
  \end{align*}
  in $\Omega_\varepsilon$.
  Hence by \eqref{E:Con_Lp} and \eqref{E:Poin_Str} we see that
  \begin{align*}
    \left|\int_{\Omega_\varepsilon}D(u):\bar{v}\otimes\bar{n}\,dx\right| \leq c\left\|\overline{P}D(u)\bar{n}\right\|_{L^2(\Omega_\varepsilon)}\|\bar{v}\|_{L^2(\Omega_\varepsilon)} \leq c\varepsilon^{3/2}\|u\|_{H^2(\Omega_\varepsilon)}\|v\|_{L^2(\Gamma)}.
  \end{align*}
  Here we used the inequalities \eqref{E:Fric_Upper} and the condition on $u$ to apply \eqref{E:Poin_Str}.
\end{proof}

Now let us derive estimates for trilinear forms.
The main tools for the estimates are the product estimate \eqref{E:Prod_Surf} for functions on $\Gamma$ and $\Omega_\varepsilon$ and the $L^\infty$-estimate \eqref{E:Linf_Ur} for the residual part of a vector field on $\Omega_\varepsilon$.

\begin{lemma} \label{L:Ave_TrT}
  Let $u_1\in H^2(\Omega_\varepsilon)^3$, $u_2\in H^1(\Omega_\varepsilon)^3$, and $A\in L^2(\Gamma)^{3\times3}$.
  Suppose that the inequalities \eqref{E:Fric_Upper} are valid, $u_1$ satisfies $\mathrm{div}\,u_1=0$ in $\Omega_\varepsilon$ and \eqref{E:Bo_Imp}--\eqref{E:Bo_Slip} on $\Gamma_\varepsilon$, and $A$ satisfies \eqref{E:A_Cond} on $\Gamma$.
  Then
  \begin{multline} \label{E:Ave_TrT}
    \left|\int_{\Omega_\varepsilon}u_1\otimes u_2:\overline{A}\,dx-\varepsilon\int_\Gamma g(M_\tau u_1)\otimes(M_\tau u_2):A\,d\mathcal{H}^2\right| \\
    \leq cR_\varepsilon(u_1,u_2)\|A\|_{L^2(\Gamma)},
  \end{multline}
  where $c>0$ is a constant independent of $\varepsilon$, $u_1$, $u_2$, and $A$ and
  \begin{multline} \label{E:Ave_TrT_Re}
    R_\varepsilon(u_1,u_2) := \varepsilon\|u_1\|_{H^1(\Omega_\varepsilon)}\|u_2\|_{H^1(\Omega_\varepsilon)} \\
    +\left(\varepsilon\|u_1\|_{H^2(\Omega_\varepsilon)}+\varepsilon^{1/2}\|u_1\|_{L^2(\Omega_\varepsilon)}^{1/2}\|u_1\|_{H^2(\Omega_\varepsilon)}^{1/2}\right)\|u_2\|_{L^2(\Omega_\varepsilon)}.
  \end{multline}
\end{lemma}

\begin{proof}
  We use the notations \eqref{E:Def_U_TN} for the tangential and normal components (with respect to $\Gamma$) of a vector field on $\Omega_\varepsilon$.
  By \eqref{E:A_Cond} and $P^T=P$ on $\Gamma$,
  \begin{align*}
    u_1\otimes u_2:\overline{A} = \overline{P}(u_1\otimes u_2)\overline{P}:\overline{A} = u_{1,\tau}\otimes u_{2,\tau}:\overline{A} \quad\text{in}\quad \Omega_\varepsilon.
  \end{align*}
  Using this equality we decompose the difference
  \begin{align} \label{Pf_TrT:Split}
    \int_{\Omega_\varepsilon}u_1\otimes u_2:\overline{A}\,dx-\varepsilon\int_\Gamma g(M_\tau u_1)\otimes(M_\tau u_2):A\,d\mathcal{H}^2 = J_1+J_2
  \end{align}
  into
  \begin{align*}
    J_1 &:= \int_{\Omega_\varepsilon}u_{1,\tau}\otimes u_{2,\tau}:\overline{A}\,dx-\int_{\Omega_\varepsilon}\Bigl(\overline{M_\tau u_1}\Bigr)\otimes u_{2,\tau}:\overline{A}\,dx, \\
    J_2 &:= \int_{\Omega_\varepsilon}\Bigl(\overline{M_\tau u_1}\Bigr)\otimes u_{2,\tau}:\overline{A}\,dx-\varepsilon\int_\Gamma g(M_\tau u_1)\otimes(M_\tau u_2):A\,d\mathcal{H}^2.
  \end{align*}
  Let $u_1^a$ be the average part of $u_1$ given by \eqref{E:Def_ExAve} and $u_1^r:=u_1-u_1^a$.
  Since
  \begin{align*}
    u_{1,\tau}-\overline{M_\tau u_1} = \overline{P}u_1-\overline{P}u_1^a = \overline{P}u_1^r, \quad u_{2,\tau} = \overline{P}u_2 \quad\text{in}\quad \Omega_\varepsilon
  \end{align*}
  and $|Pa|\leq|a|$ on $\Gamma$ for $a\in\mathbb{R}^3$,
  \begin{align*}
    |J_1| = \left|\int_{\Omega_\varepsilon}\Bigl(\overline{P}u_1^r\Bigr)\otimes u_{2,\tau}:\overline{A}\,dx\right| \leq c\|u_1^r\|_{L^\infty(\Omega_\varepsilon)}\|u_2\|_{L^2(\Omega_\varepsilon)}\left\|\overline{A}\right\|_{L^2(\Omega_\varepsilon)}.
  \end{align*}
  We apply \eqref{E:Con_Lp} and \eqref{E:Linf_Ur} to the right-hand side to obtain
  \begin{align} \label{Pf_TrT:I1}
    |J_1| \leq c\left(\varepsilon\|u_1\|_{H^2(\Omega_\varepsilon)}+\varepsilon^{1/2}\|u_1\|_{L^2(\Omega_\varepsilon)}^{1/2}\|u_1\|_{H^2(\Omega_\varepsilon)}^{1/2}\right)\|u_2\|_{L^2(\Omega_\varepsilon)}\|A\|_{L^2(\Gamma)}.
  \end{align}
  Here we used the inequalities \eqref{E:Fric_Upper} and the conditions on $u_1$ to apply \eqref{E:Linf_Ur}.

  Let us estimate $J_2$.
  By $M_\tau u_2=Mu_{2,\tau}$ on $\Gamma$, \eqref{E:CoV_Dom}, and \eqref{E:Def_Ave},
  \begin{align*}
    J_2 = \int_\Gamma (M_\tau u_1)\otimes\left(\int_{\varepsilon g_0}^{\varepsilon g_1}u_{2,\tau}^\sharp(J-1)\,dr\right):A\,d\mathcal{H}^2.
  \end{align*}
  To the right-hand side we apply \eqref{E:Jac_Diff}, \eqref{E:CoV_Equiv}, and
  \begin{align*}
    |M_\tau u_1| = |PMu_1| \leq |Mu_1| \quad\text{on}\quad \Gamma, \quad |u_{2,\tau}| = \left|\overline{P}u_2\right| \leq |u_2| \quad\text{in}\quad \Omega_\varepsilon
  \end{align*}
  to deduce that
  \begin{align*}
    |J_2| \leq c\varepsilon\int_{\Omega_\varepsilon}\left|\overline{M_\tau u_1}\right||u_{2,\tau}|\left|\overline{A}\right|\,dx \leq c\varepsilon\left\|\,\left|\overline{Mu_1}\right|\,|u_2|\,\right\|_{L^2(\Omega_\varepsilon)}\left\|\overline{A}\right\|_{L^2(\Omega_\varepsilon)}.
  \end{align*}
  Moreover, from \eqref{E:Ave_Lp_Surf}, \eqref{E:Ave_Wmp_Surf}, and \eqref{E:Prod_Surf} it follows that
  \begin{align*}
    \left\|\,\left|\overline{Mu_1}\right|\,|u_2|\,\right\|_{L^2(\Omega_\varepsilon)} &\leq c\|Mu_1\|_{L^2(\Gamma)}^{1/2}\|Mu_1\|_{H^1(\Gamma)}^{1/2}\|u_2\|_{L^2(\Omega_\varepsilon)}^{1/2}\|u_2\|_{H^1(\Omega_\varepsilon)}^{1/2} \\
    &\leq c\varepsilon^{-1/2}\|u_1\|_{L^2(\Omega_\varepsilon)}^{1/2}\|u_1\|_{H^1(\Omega_\varepsilon)}^{1/2}\|u_2\|_{L^2(\Omega_\varepsilon)}^{1/2}\|u_2\|_{H^1(\Omega_\varepsilon)}^{1/2}.
  \end{align*}
  We apply this inequality and \eqref{E:Con_Lp} to the above estimate for $J_2$ to get
  \begin{align} \label{Pf_TrT:I2}
    \begin{aligned}
      |J_2| &\leq c\varepsilon\|u_1\|_{L^2(\Omega_\varepsilon)}^{1/2}\|u_1\|_{H^1(\Omega_\varepsilon)}^{1/2}\|u_2\|_{L^2(\Omega_\varepsilon)}^{1/2}\|u_2\|_{H^1(\Omega_\varepsilon)}^{1/2}\|A\|_{L^2(\Gamma)}. \\
      &\leq c\varepsilon\|u_1\|_{H^1(\Omega_\varepsilon)}\|u_2\|_{H^1(\Omega_\varepsilon)}\|A\|_{L^2(\Gamma)}.
    \end{aligned}
  \end{align}
  By \eqref{Pf_TrT:Split}--\eqref{Pf_TrT:I2} we obtain \eqref{E:Ave_TrT} with $R_\varepsilon(u_1,u_2)$ given by \eqref{E:Ave_TrT_Re}.
\end{proof}

\begin{lemma} \label{L:Ave_TrN}
  Let $u_1,u_2\in H^1(\Omega_\varepsilon)^3$ and $v\in H^1(\Gamma)^3$.
  Suppose that $u_2$ satisfies \eqref{E:Bo_Imp} on $\Gamma_\varepsilon^0$ or on $\Gamma_\varepsilon^1$.
  Then there exists a constant $c>0$ independent of $\varepsilon$, $u_1$, $u_2$, and $v$ such that
  \begin{align} \label{E:Ave_TrN}
    \left|\int_{\Omega_\varepsilon}u_1\otimes u_2:\bar{v}\otimes\bar{n}\,dx\right| \leq c\varepsilon\|u_1\|_{H^1(\Omega_\varepsilon)}\|u_2\|_{H^1(\Omega_\varepsilon)}\|v\|_{H^1(\Gamma)}.
  \end{align}
\end{lemma}

\begin{proof}
  By $u_1\otimes u_2:\bar{v}\otimes\bar{n}=(u_1\cdot\bar{v})(u_2\cdot\bar{n})$ in $\Omega_\varepsilon$, \eqref{E:Poin_Nor}, and \eqref{E:Prod_Surf} we have
  \begin{align*}
    \left|\int_{\Omega_\varepsilon}u_1\otimes u_2:\bar{v}\otimes\bar{n}\,dx\right| &\leq \|u_1\cdot\bar{v}\|_{L^2(\Omega_\varepsilon)}\|u_2\cdot\bar{n}\|_{L^2(\Omega_\varepsilon)} \\
    &\leq c\varepsilon\|u_1\|_{L^2(\Omega_\varepsilon)}^{1/2}\|u_1\|_{H^1(\Omega_\varepsilon)}^{1/2}\|v\|_{L^2(\Gamma)}^{1/2}\|v\|_{H^1(\Gamma)}^{1/2}\|u_2\|_{H^1(\Omega_\varepsilon)} \\
    &\leq c\varepsilon\|u_1\|_{H^1(\Omega_\varepsilon)}\|u_2\|_{H^1(\Omega_\varepsilon)}\|v\|_{H^1(\Gamma)}.
  \end{align*}
  Here we used the condition on $u_2$ to apply \eqref{E:Poin_Nor} to $\|u_2\cdot\bar{n}\|_{L^2(\Omega_\varepsilon)}$.
\end{proof}

\section{Estimate for the trilinear term} \label{S:Tri}
The purpose of this section is to give an estimate for the trilinear term, i.e. the $L^2$-inner product of the inertial and viscous terms, which is essential for the proof of the global existence of a strong solution to \eqref{E:NS_Eq}--\eqref{E:NS_In}.
Throughout this section we impose Assumptions~\ref{Assump_1} and~\ref{Assump_2} and fix the constant $\varepsilon_0$ given in Lemma~\ref{L:Bili_Core}.
For $\varepsilon\in(0,\varepsilon_0)$ let $\mathcal{H}_\varepsilon$ be the subspace of $L^2(\Omega_\varepsilon)^3$ given by \eqref{E:Def_Heps} and $A_\varepsilon$ the Stokes operator on $\mathcal{H}_\varepsilon$ defined in Section~\ref{SS:St_Def}.

\begin{lemma} \label{L:Tri_Est}
  For any $\alpha>0$ there exist $c_\alpha^1,c_\alpha^2>0$ independent of $\varepsilon$ such that
  \begin{multline} \label{E:Tri_Est}
    \left|\bigl((u\cdot\nabla)u,A_\varepsilon u\bigr)_{L^2(\Omega_\varepsilon)}\right| \leq \left(\alpha+c_\alpha^1\varepsilon^{1/2}\|u\|_{H^1(\Omega_\varepsilon)}\right)\|u\|_{H^2(\Omega_\varepsilon)}^2 \\
    +c_\alpha^2\left(\|u\|_{L^2(\Omega_\varepsilon)}^2\|u\|_{H^1(\Omega_\varepsilon)}^4+\varepsilon^{-1}\|u\|_{L^2(\Omega_\varepsilon)}^2\|u\|_{H^1(\Omega_\varepsilon)}^2\right)
  \end{multline}
  for all $\varepsilon\in(0,\varepsilon_0)$ and $u\in D(A_\varepsilon)$.
  (In fact, $c_\alpha^1$ does not depend on $\alpha$.)
\end{lemma}

The main tools for the proof of Lemma~\ref{L:Tri_Est} are the estimates given in Sections~\ref{SS:St_URE} and~\ref{SS:Ave_Ex}.
We also use the following inequality for the tangential component (with respect to $\Gamma$) of the curl of the average part $u^a$ (for the proof, see Appendix~\ref{S:Ap_VM}).

\begin{lemma} \label{L:Tan_Curl_Ua}
  For $u\in C^1(\Omega_\varepsilon)^3$ let $u^a:=E_\varepsilon M_\tau u$ be given by \eqref{E:Def_ExAve}.
  Then
  \begin{align} \label{E:Tan_Curl_Ua}
    \left|\overline{P}\,\mathrm{curl}\,u^a\right| \leq c\left(\left|\overline{Mu}\right|+\varepsilon\left|\overline{\nabla_\Gamma Mu}\right|\right) \quad\text{in}\quad \Omega_\varepsilon,
  \end{align}
  where $c>0$ is a constant independent of $\varepsilon$ and $u$.
\end{lemma}

\begin{proof}[Proof of Lemma~\ref{L:Tri_Est}]
  The proof is basically the same as that of \cite[Proposition~6.1]{Ho10}, but we require further calculations.
  For $u\in D(A_\varepsilon)$ let $u^a$ be the average part of $u$ given by \eqref{E:Def_ExAve}, $u^r:=u-u^a$ the residual part, and $\omega:=\mathrm{curl}\,u$.
  Since
  \begin{align*}
    (u\cdot\nabla)u = \omega\times u+\frac{1}{2}\nabla(|u|^2), \quad A_\varepsilon u \in \mathcal{H}_\varepsilon \subset L_\sigma^2(\Omega_\varepsilon), \quad \nabla(|u|^2) \in L_\sigma^2(\Omega_\varepsilon)^\perp,
  \end{align*}
  we have $(\nabla(|u|^2),A_\varepsilon u)_{L^2(\Omega_\varepsilon)}=0$ and thus
  \begin{align*}
    \bigl((u\cdot\nabla)u,A_\varepsilon u\bigr)_{L^2(\Omega_\varepsilon)} = (\omega\times u,A_\varepsilon u)_{L^2(\Omega_\varepsilon)} = J_1+J_2+J_3,
  \end{align*}
  where
  \begin{gather*}
    J_1 := (\omega\times u^r,A_\varepsilon u)_{L^2(\Omega_\varepsilon)},\\
    J_2 := (\omega\times u^a,A_\varepsilon u+\nu\Delta u)_{L^2(\Omega_\varepsilon)}, \quad J_3 := (\omega\times u^a,-\nu\Delta u)_{L^2(\Omega_\varepsilon)}.
  \end{gather*}
  Let us estimate $J_1$, $J_2$, and $J_3$ separately.
  By \eqref{E:Stokes_H2} and \eqref{E:Linf_Ur},
  \begin{align*}
    |J_1| &\leq \|u^r\|_{L^\infty(\Omega_\varepsilon)}\|\omega\|_{L^2(\Omega_\varepsilon)}\|A_\varepsilon u\|_{L^2(\Omega_\varepsilon)} \\
    &\leq c\left(\varepsilon^{1/2}\|u\|_{H^2(\Omega_\varepsilon)}+\|u\|_{L^2(\Omega_\varepsilon)}^{1/2}\|u\|_{H^2(\Omega_\varepsilon)}^{1/2}\right)\|u\|_{H^1(\Omega_\varepsilon)}\|u\|_{H^2(\Omega_\varepsilon)} \\
    &= c\varepsilon^{1/2}\|u\|_{H^1(\Omega_\varepsilon)}\|u\|_{H^2(\Omega_\varepsilon)}^2+c\|u\|_{L^2(\Omega_\varepsilon)}^{1/2}\|u\|_{H^1(\Omega_\varepsilon)}\|u\|_{H^2(\Omega_\varepsilon)}^{3/2}.
  \end{align*}
  To the last term we apply Young's inequality $ab\leq \alpha a^{4/3}+c_\alpha b^4$ to get
  \begin{align} \label{Pf_TE:Est_I1}
    |J_1| \leq \left(\alpha+c\varepsilon^{1/2}\|u\|_{H^1(\Omega_\varepsilon)}\right)\|u\|_{H^2(\Omega_\varepsilon)}^2+c_\alpha\|u\|_{L^2(\Omega_\varepsilon)}^2\|u\|_{H^1(\Omega_\varepsilon)}^4.
  \end{align}
  (Note that the constant $c$ in the above inequality does not depend on $\alpha$.)

  Next we deal with $J_2$.
  By \eqref{E:Prod_Ua} we have
  \begin{align} \label{Pf_TE:L2_Phi}
    \begin{aligned}
      \|\omega\times u^a\|_{L^2(\Omega_\varepsilon)} &\leq c\varepsilon^{-1/2}\|\omega\|_{L^2(\Omega_\varepsilon)}^{1/2}\|\omega\|_{H^1(\Omega_\varepsilon)}^{1/2}\|u\|_{L^2(\Omega_\varepsilon)}^{1/2}\|u\|_{H^1(\Omega_\varepsilon)}^{1/2} \\
      &\leq c\varepsilon^{-1/2}\|u\|_{L^2(\Omega_\varepsilon)}^{1/2}\|u\|_{H^1(\Omega_\varepsilon)}\|u\|_{H^2(\Omega_\varepsilon)}^{1/2}.
    \end{aligned}
  \end{align}
  From this inequality, \eqref{E:Comp_Sto_Lap}, and \eqref{E:St_Inter} it follows that
  \begin{align*}
    |J_2| &\leq \|\omega\times u^a\|_{L^2(\Omega_\varepsilon)}\|A_\varepsilon u+\nu\Delta u\|_{L^2(\Omega_\varepsilon)} \\
    &\leq c\varepsilon^{-1/2}\|u\|_{L^2(\Omega_\varepsilon)}^{1/2}\|u\|_{H^1(\Omega_\varepsilon)}^2\|u\|_{H^2(\Omega_\varepsilon)}^{1/2} \\
    &\leq c\varepsilon^{-1/2}\|u\|_{L^2(\Omega_\varepsilon)}\|u\|_{H^1(\Omega_\varepsilon)}\|u\|_{H^2(\Omega_\varepsilon)}.
  \end{align*}
   Applying Young's inequality $ab\leq \alpha a^2+c_\alpha b^2$ to the last line we further get
  \begin{align} \label{Pf_TE:Est_I2}
      |J_2| \leq \alpha\|u\|_{H^2(\Omega_\varepsilon)}^2+c_\alpha\varepsilon^{-1}\|u\|_{L^2(\Omega_\varepsilon)}^2\|u\|_{H^1(\Omega_\varepsilon)}^2.
  \end{align}
  The estimate for $J_3$ is more complicated.
  Let $\Phi:=\omega\times u^a$.
  Since $\omega\in H^1(\Omega_\varepsilon)^3$ and $u^a\in H^2(\Omega_\varepsilon)^3$, we have $\Phi\in H^1(\Omega_\varepsilon)^3$ by the Sobolev embeddings $H^1(\Omega_\varepsilon)\hookrightarrow L^4(\Omega_\varepsilon)$ and $H^2(\Omega_\varepsilon)\hookrightarrow L^\infty(\Omega_\varepsilon)$ (see~\cite{AdFo03}).
  Also, since $-\Delta u=\mathrm{curl}\,\omega$ by $\mathrm{div}\,u=0$ in $\Omega_\varepsilon$,
  \begin{align*}
    J_3 &= -\nu(\Delta u,\Phi)_{L^2(\Omega_\varepsilon)} = \nu(\mathrm{curl}\,\omega,\Phi)_{L^2(\Omega_\varepsilon)} \\
    &= -\nu(\mathrm{curl}\,G(u),\Phi)_{L^2(\Omega_\varepsilon)}+\nu(\omega+G(u),\mathrm{curl}\,\Phi)_{L^2(\Omega_\varepsilon)} = J_3^1+J_3^2+J_3^3
  \end{align*}
  by \eqref{E:IbP_Curl}.
  Here $G(u)$ is given by \eqref{E:Def_Gu} and
  \begin{gather*}
    J_3^1 := -\nu(\mathrm{curl}\,G(u),\Phi)_{L^2(\Omega_\varepsilon)}, \\
    J_3^2 := \nu(G(u),\mathrm{curl}\,\Phi)_{L^2(\Omega_\varepsilon)}, \quad J_3^3 := \nu(\omega,\mathrm{curl}\,\Phi)_{L^2(\Omega_\varepsilon)}.
  \end{gather*}
  Noting that $\Phi=\omega\times u^a$, we apply \eqref{E:G_Bound} and \eqref{Pf_TE:L2_Phi} to $J_3^1$ to deduce that
  \begin{align*}
    |J_3^1| \leq c\|\nabla G(u)\|_{L^2(\Omega_\varepsilon)}\|\Phi\|_{L^2(\Omega_\varepsilon)} \leq c\varepsilon^{-1/2}\|u\|_{L^2(\Omega_\varepsilon)}^{1/2}\|u\|_{H^1(\Omega_\varepsilon)}^2\|u\|_{H^2(\Omega_\varepsilon)}^{1/2}.
  \end{align*}
  Then using \eqref{E:St_Inter} and Young's inequality $ab\leq \alpha a^2+c_\alpha b^2$ we get
  \begin{align} \label{Pf_TE:Est_J1}
    \begin{aligned}
    |J_3^1| &\leq c\varepsilon^{-1/2}\|u\|_{L^2(\Omega_\varepsilon)}\|u\|_{H^1(\Omega_\varepsilon)}\|u\|_{H^2(\Omega_\varepsilon)} \\
    &\leq \alpha\|u\|_{H^2(\Omega_\varepsilon)}^2+c_\alpha\varepsilon^{-1}\|u\|_{L^2(\Omega_\varepsilon)}^2\|u\|_{H^1(\Omega_\varepsilon)}^2.
    \end{aligned}
  \end{align}
  Let us estimate $J_3^2$.
  The curl of $\Phi=\omega\times u^a$ is bounded by
  \begin{align*}
    |\mathrm{curl}\,\Phi| \leq c(|\nabla\omega||u^a|+|\omega||\nabla u^a|) \leq c(|u^a||\nabla^2u|+|\nabla u^a||\nabla u|) \quad\text{in}\quad \Omega_\varepsilon.
  \end{align*}
  By this inequality, \eqref{E:G_Bound}, and H\"{o}lder's inequality we get
  \begin{align*}
    |J_3^2| &\leq c\int_{\Omega_\varepsilon}|u|(|u^a||\nabla^2u|+|\nabla u^a||\nabla u|)\,dx \\
    &\leq c\left(\|\,|u^a|\,|u|\,\|_{L^2(\Omega_\varepsilon)}\|\nabla^2u\|_{L^2(\Omega_\varepsilon)}+\|\,|\nabla u^a|\,|u|\,\|_{L^2(\Omega_\varepsilon)}\|\nabla u\|_{L^2(\Omega_\varepsilon)}\right).
  \end{align*}
  To the last line we apply \eqref{E:Prod_Ua} and \eqref{E:Prod_Grad_Ua} to obtain
  \begin{align*}
    |J_3^2| &\leq c\varepsilon^{-1/2}\left(\|u\|_{L^2(\Omega_\varepsilon)}\|u\|_{H^1(\Omega_\varepsilon)}\|u\|_{H^2(\Omega_\varepsilon)}+\|u\|_{L^2(\Omega_\varepsilon)}^{1/2}\|u\|_{H^1(\Omega_\varepsilon)}^2\|u\|_{H^2(\Omega_\varepsilon)}^{1/2}\right) \\
    &\leq c\varepsilon^{-1/2}\|u\|_{L^2(\Omega_\varepsilon)}\|u\|_{H^1(\Omega_\varepsilon)}\|u\|_{H^2(\Omega_\varepsilon)},
  \end{align*}
  where the second inequality follows from \eqref{E:St_Inter}.
  Hence Young's inequality yields
  \begin{align} \label{Pf_TE:Est_J2}
    |J_3^2| \leq \alpha\|u\|_{H^2(\Omega_\varepsilon)}^2+c_\alpha\varepsilon^{-1}\|u\|_{L^2(\Omega_\varepsilon)}^2\|u\|_{H^1(\Omega_\varepsilon)}^2.
  \end{align}
  To estimate $J_3^3=\nu(\omega,\mathrm{curl}\,\Phi)_{L^2(\Omega_\varepsilon)}$ we observe by $\Phi=\omega\times u^a$ that
  \begin{align*}
    \mathrm{curl}\,\Phi = (u^a\cdot\nabla)\omega-(\omega\cdot\nabla)u^a+(\mathrm{div}\,u^a)\omega-(\mathrm{div}\,\omega)u^a \quad\text{in}\quad \Omega_\varepsilon.
  \end{align*}
  Moreover, since $u^a$ satisfies $u^a\cdot n_\varepsilon=0$ on $\Gamma_\varepsilon$ by Lemma~\ref{L:ExTan_Imp} and \eqref{E:Def_ExAve}, we have
  \begin{align*}
    \int_{\Omega_\varepsilon}\omega\cdot(u^a\cdot\nabla)\omega\,dx = -\frac{1}{2}\int_{\Omega_\varepsilon}(\mathrm{div}\,u^a)|\omega|^2\,dx
  \end{align*}
  by integration by parts.
  By these equalities and $\mathrm{div}\,\omega=\mathrm{div}\,\mathrm{curl}\,u=0$ in $\Omega_\varepsilon$ we get
  \begin{align} \label{Pf_TE:J3_Exp}
    \begin{aligned}
      J_3^3 & = \nu(\omega,(u^a\cdot\nabla)\omega-(\omega\cdot\nabla)u^a+(\mathrm{div}\,u^a)\omega)_{L^2(\Omega_\varepsilon)} \\
      & = \frac{\nu}{2}(\mathrm{div}\,u^a,|\omega|^2)_{L^2(\Omega_\varepsilon)}-\nu(\omega,(\omega\cdot\nabla)u^a)_{L^2(\Omega_\varepsilon)}.
    \end{aligned}
  \end{align}
  Since $u^a=E_\varepsilon M_\tau u$ is given by \eqref{E:Def_ExAve}, we see that
  \begin{multline*}
    (\mathrm{div}\,u^a,|\omega|^2)_{L^2(\Omega_\varepsilon)} = \int_{\Omega_\varepsilon}\frac{1}{\bar{g}}\left(\overline{\mathrm{div}_\Gamma(gM_\tau u)}\right)|\omega|^2\,dx \\
    +\int_{\Omega_\varepsilon}\left(\mathrm{div}(E_\varepsilon M_\tau u)-\frac{1}{\bar{g}}\overline{\mathrm{div}_\Gamma(gM_\tau u)}\right)|\omega|^2\,dx.
  \end{multline*}
  Applying \eqref{E:Width_Bound}, \eqref{E:ExTan_Div}, and H\"{o}lder's inequality to the right-hand side we have
  \begin{align*}
    |(\mathrm{div}\,u^a,|\omega|^2)_{L^2(\Omega_\varepsilon)}| \leq c(K_1+\varepsilon K_2+\varepsilon K_3)\|\omega\|_{L^2(\Omega_\varepsilon)},
  \end{align*}
  where
  \begin{gather*}
    K_1 := \left\|\,\left|\overline{\mathrm{div}_\Gamma(gM_\tau u)}\right|\,|\omega|\,\right\|_{L^2(\Omega_\varepsilon)}, \\
    K_2 := \left\|\,\left|\overline{M_\tau u}\right|\,|\omega|\,\right\|_{L^2(\Omega_\varepsilon)}, \quad K_3 := \left\|\,\left|\overline{\nabla_\Gamma M_\tau u}\right|\,|\omega|\,\right\|_{L^2(\Omega_\varepsilon)}.
  \end{gather*}
  To $K_1$ we apply \eqref{E:Prod_Surf} and then use \eqref{E:ADiv_Tan_Lp} and \eqref{E:ADiv_Tan_W1p} to get
  \begin{align*}
    K_1 &\leq c\|\mathrm{div}_\Gamma(gM_\tau u)\|_{L^2(\Gamma)}^{1/2}\|\mathrm{div}_\Gamma(gM_\tau u)\|_{H^1(\Gamma)}^{1/2}\|\omega\|_{L^2(\Omega_\varepsilon)}^{1/2}\|\omega\|_{H^1(\Omega_\varepsilon)}^{1/2} \\
    &\leq c\varepsilon^{1/2}\|u\|_{H^1(\Omega_\varepsilon)}\|u\|_{H^2(\Omega_\varepsilon)} \leq c\varepsilon^{1/2}\|u\|_{H^2(\Omega_\varepsilon)}^2.
  \end{align*}
  Also, by $M_\tau u=PMu$ on $\Gamma$, $P\in C^4(\Gamma)^{3\times3}$, \eqref{E:Ave_Lp_Surf}, and \eqref{E:Ave_Wmp_Surf},
  \begin{align*}
    \|M_\tau u\|_{H^k(\Gamma)} &\leq c\|Mu\|_{H^k(\Gamma)} \leq c\varepsilon^{-1/2}\|u\|_{H^k(\Omega_\varepsilon)}, \\
    \|\nabla_\Gamma M_\tau u\|_{H^k(\Gamma)} &\leq c\|Mu\|_{H^{k+1}(\Gamma)} \leq c\varepsilon^{-1/2}\|u\|_{H^{k+1}(\Omega_\varepsilon)}
  \end{align*}
  for $k=0,1$ (with $H^0=L^2$).
  Using \eqref{E:Prod_Surf} and these inequalities we obtain
  \begin{align*}
    K_2 &\leq c\varepsilon^{-1/2}\|u\|_{L^2(\Omega_\varepsilon)}^{1/2}\|u\|_{H^1(\Omega_\varepsilon)}\|u\|_{H^2(\Omega_\varepsilon)}^{1/2} \leq c\varepsilon^{-1/2}\|u\|_{H^2(\Omega_\varepsilon)}^2, \\
    K_3 &\leq c\varepsilon^{-1/2}\|u\|_{H^1(\Omega_\varepsilon)}\|u\|_{H^2(\Omega_\varepsilon)} \leq c\varepsilon^{-1/2}\|u\|_{H^2(\Omega_\varepsilon)}^2.
  \end{align*}
  From these inequalities and $\|\omega\|_{L^2(\Omega_\varepsilon)}\leq c\|u\|_{H^1(\Omega_\varepsilon)}$ we deduce that
  \begin{align} \label{Pf_TE:DivUa_O}
    \begin{aligned}
      |(\mathrm{div}\,u^a,|\omega|^2)_{L^2(\Omega_\varepsilon)}| &\leq c(K_1+\varepsilon K_2+\varepsilon K_3)\|\omega\|_{L^2(\Omega_\varepsilon)} \\
      &\leq c\varepsilon^{1/2}\|u\|_{H^1(\Omega_\varepsilon)}\|u\|_{H^2(\Omega_\varepsilon)}^2.
    \end{aligned}
  \end{align}
  Let us estimate $(\omega,(\omega\cdot\nabla)u^a)_{L^2(\Omega_\varepsilon)}$.
  By $\omega=\mathrm{curl}\,u^r+\mathrm{curl}\,u^a$ we have
  \begin{align*}
    (\omega,(\omega\cdot\nabla)u^a)_{L^2(\Omega_\varepsilon)} = (\omega,(\mathrm{curl}\,u^r\cdot\nabla)u^a)_{L^2(\Omega_\varepsilon)}+(\omega,(\mathrm{curl}\,u^a\cdot\nabla)u^a)_{L^2(\Omega_\varepsilon)}.
  \end{align*}
  The first term on the right-hand side is bounded by
  \begin{align*}
    |(\omega,(\mathrm{curl}\,u^r\cdot\nabla)u^a)_{L^2(\Omega_\varepsilon)}| \leq c\|\nabla u^r\|_{L^2(\Omega_\varepsilon)}\|\,|\nabla u^a|\,|\omega|\,\|_{L^2(\Omega_\varepsilon)}.
  \end{align*}
  To the right-hand side we apply \eqref{E:Po_Grad_Ur} and
  \begin{align} \label{Pf_TE:GUa_O}
    \begin{aligned}
      \|\,|\nabla u^a|\,|\omega|\,\|_{L^2(\Omega_\varepsilon)} &\leq c\varepsilon^{-1/2}\|\omega\|_{L^2(\Omega_\varepsilon)}^{1/2}\|\omega\|_{H^1(\Omega_\varepsilon)}^{1/2}\|u\|_{H^1(\Omega_\varepsilon)}^{1/2}\|u\|_{H^2(\Omega_\varepsilon)}^{1/2} \\
      &\leq c\varepsilon^{-1/2}\|u\|_{H^1(\Omega_\varepsilon)}\|u\|_{H^2(\Omega_\varepsilon)}
    \end{aligned}
  \end{align}
  by \eqref{E:Prod_Grad_Ua}.
  Then we get
  \begin{multline} \label{Pf_TE:OUrUa}
    |(\omega,(\mathrm{curl}\,u^r\cdot\nabla)u^a)_{L^2(\Omega_\varepsilon)}| \\
    \leq c\left(\varepsilon^{1/2}\|u\|_{H^1(\Omega_\varepsilon)}\|u\|_{H^2(\Omega_\varepsilon)}^2+\varepsilon^{-1/2}\|u\|_{L^2(\Omega_\varepsilon)}\|u\|_{H^1(\Omega_\varepsilon)}\|u\|_{H^2(\Omega_\varepsilon)}\right).
  \end{multline}
  Also, we decompose $(\omega,(\mathrm{curl}\,u^a\cdot\nabla)u^a)_{L^2(\Omega_\varepsilon)}$ into the sum of
  \begin{align*}
    L_1 := \left(\omega,\bigl((\overline{P}\,\mathrm{curl}\,u^a)\cdot\nabla\bigr)u^a\right)_{L^2(\Omega_\varepsilon)}, \quad L_2 := (\omega,(\mathrm{curl}\,u^a\cdot\bar{n})\partial_nu^a)_{L^2(\Omega_\varepsilon)}.
  \end{align*}
  To $L_1$ we apply \eqref{E:Tan_Curl_Ua} and H\"{o}lder's inequality to get
  \begin{align*}
    |L_1| &\leq c\int_{\Omega_\varepsilon}|\omega|\left(\left|\overline{Mu}\right|+\varepsilon\left|\overline{\nabla_\Gamma Mu}\right|\right)|\nabla u^a|\,dx \\
    &\leq c\|\,|\nabla u^a|\,|\omega|\,\|_{L^2(\Omega_\varepsilon)}\left(\left\|\overline{Mu}\right\|_{L^2(\Omega_\varepsilon)}+\varepsilon\left\|\overline{\nabla_\Gamma Mu}\right\|_{L^2(\Omega_\varepsilon)}\right).
  \end{align*}
  Hence from \eqref{E:Con_Lp}, \eqref{E:Ave_Lp_Surf}, \eqref{E:Ave_Wmp_Surf}, \eqref{Pf_TE:GUa_O}, and $\|u\|_{H^1(\Omega_\varepsilon)}\leq\|u\|_{H^2(\Omega_\varepsilon)}$ it follows that
  \begin{align*}
    |L_1| &\leq c\varepsilon^{-1/2}\|u\|_{H^1(\Omega_\varepsilon)}\|u\|_{H^2(\Omega_\varepsilon)}\left(\|u\|_{L^2(\Omega_\varepsilon)}+\varepsilon\|u\|_{H^1(\Omega_\varepsilon)}\right) \\
    &\leq c\left(\varepsilon^{1/2}\|u\|_{H^1(\Omega_\varepsilon)}\|u\|_{H^2(\Omega_\varepsilon)}^2+\varepsilon^{-1/2}\|u\|_{L^2(\Omega_\varepsilon)}\|u\|_{H^1(\Omega_\varepsilon)}\|u\|_{H^2(\Omega_\varepsilon)}\right).
  \end{align*}
  To estimate $L_2$ we see by the definition \eqref{E:Def_ExAve} of $u^a$, \eqref{E:NorDer_Con}, and \eqref{E:ExAux_Bound} that
  \begin{align*}
    |\partial_nu^a| = \left|\overline{M_\tau u}\cdot\partial_n\Psi_\varepsilon\right| \leq c\left|\overline{M_\tau u}\right| = c\left|\overline{PMu}\right| \leq c\left|\overline{Mu}\right| \quad\text{in}\quad \Omega_\varepsilon.
  \end{align*}
  By this inequality, $|\mathrm{curl}\,u^a\cdot\bar{n}|\leq c|\nabla u^a|$ in $\Omega_\varepsilon$, \eqref{E:Ave_Lp_Dom}, and \eqref{Pf_TE:GUa_O},
  \begin{align*}
    |L_2| &\leq c\int_{\Omega_\varepsilon}|\omega||\nabla u^a|\left|\overline{Mu}\right|\,dx \leq c\|\,|\nabla u^a|\,|\omega|\,\|_{L^2(\Omega_\varepsilon)}\left\|\overline{Mu}\right\|_{L^2(\Omega_\varepsilon)} \\
    &\leq c\varepsilon^{-1/2}\|u\|_{L^2(\Omega_\varepsilon)}\|u\|_{H^1(\Omega_\varepsilon)}\|u\|_{H^2(\Omega_\varepsilon)}.
  \end{align*}
  Applying the above estimates to $(\omega,(\mathrm{curl}\,u^a\cdot\nabla)u^a)_{L^2(\Omega_\varepsilon)}=L_1+L_2$ we obtain
  \begin{multline*}
    |(\omega,(\mathrm{curl}\,u^a\cdot\nabla)u^a)_{L^2(\Omega_\varepsilon)}| \\
    \leq c\left(\varepsilon^{1/2}\|u\|_{H^1(\Omega_\varepsilon)}\|u\|_{H^2(\Omega_\varepsilon)}^2+\varepsilon^{-1/2}\|u\|_{L^2(\Omega_\varepsilon)}\|u\|_{H^1(\Omega_\varepsilon)}\|u\|_{H^2(\Omega_\varepsilon)}\right).
  \end{multline*}
  From this inequality and \eqref{Pf_TE:OUrUa} we deduce that
  \begin{align*}
    &|(\omega,(\omega\cdot\nabla)u^a)_{L^2(\Omega_\varepsilon)}| \\
    &\qquad \leq |(\omega,(\mathrm{curl}\,u^r\cdot\nabla)u^a)_{L^2(\Omega_\varepsilon)}|+|(\omega,(\mathrm{curl}\,u^a\cdot\nabla)u^a)_{L^2(\Omega_\varepsilon)}| \\
    &\qquad \leq c\left(\varepsilon^{1/2}\|u\|_{H^1(\Omega_\varepsilon)}\|u\|_{H^2(\Omega_\varepsilon)}^2+\varepsilon^{-1/2}\|u\|_{L^2(\Omega_\varepsilon)}\|u\|_{H^1(\Omega_\varepsilon)}\|u\|_{H^2(\Omega_\varepsilon)}\right).
  \end{align*}
  Using Young's inequality $ab\leq \alpha a^2+c_\alpha b^2$ to the last term we further get
  \begin{multline*}
    |(\omega,(\omega\cdot\nabla)u^a)_{L^2(\Omega_\varepsilon)}| \\
    \leq \left(\alpha+c\varepsilon^{1/2}\|u\|_{H^1(\Omega_\varepsilon)}\right)\|u\|_{H^2(\Omega_\varepsilon)}^2+c_\alpha\varepsilon^{-1}\|u\|_{L^2(\Omega_\varepsilon)}^2\|u\|_{H^1(\Omega_\varepsilon)}^2.
  \end{multline*}
  We apply this inequality and \eqref{Pf_TE:DivUa_O} to \eqref{Pf_TE:J3_Exp} to show that
  \begin{align} \label{Pf_TE:Est_J3}
    \begin{aligned}
      |J_3^3| &\leq c\left(|(\mathrm{div}\,u^a,|\omega|^2)_{L^2(\Omega_\varepsilon)}|+|(\omega,(\omega\cdot\nabla)u^a)_{L^2(\Omega_\varepsilon)}|\right) \\
      &\leq c\left(\alpha+\varepsilon^{1/2}\|u\|_{H^1(\Omega_\varepsilon)}\right)\|u\|_{H^2(\Omega_\varepsilon)}^2+c_\alpha\varepsilon^{-1}\|u\|_{L^2(\Omega_\varepsilon)}^2\|u\|_{H^1(\Omega_\varepsilon)}^2.
    \end{aligned}
  \end{align}
  Since $J_3=J_3^1+J_3^2+J_3^3$, we see by \eqref{Pf_TE:Est_J1}, \eqref{Pf_TE:Est_J2}, and \eqref{Pf_TE:Est_J3} that
  \begin{align*}
    |J_3| \leq c\left(\alpha+\varepsilon^{1/2}\|u\|_{H^1(\Omega_\varepsilon)}\right)\|u\|_{H^2(\Omega_\varepsilon)}^2+c_\alpha\varepsilon^{-1}\|u\|_{L^2(\Omega_\varepsilon)}^2\|u\|_{H^1(\Omega_\varepsilon)}^2
  \end{align*}
  and this inequality combined with \eqref{Pf_TE:Est_I1} and \eqref{Pf_TE:Est_I2} yields
  \begin{align*}
    \left|\bigl((u\cdot\nabla)u,A_\varepsilon u\bigr)_{L^2(\Omega_\varepsilon)}\right| &\leq |J_1|+|J_2|+|J_3| \\
    &\leq \left(c_1\alpha+c_2\varepsilon^{1/2}\|u\|_{H^1(\Omega_\varepsilon)}\right)\|u\|_{H^2(\Omega_\varepsilon)}^2 \\
    &\qquad +c_\alpha\left(\|u\|_{L^2(\Omega_\varepsilon)}^2\|u\|_{H^1(\Omega_\varepsilon)}^4+\varepsilon^{-1}\|u\|_{L^2(\Omega_\varepsilon)}^2\|u\|_{H^1(\Omega_\varepsilon)}^2\right)
  \end{align*}
  with positive constants $c_1$, $c_2$, and $c_\alpha$ independent of $\varepsilon$.
  Replacing $c_1\alpha$ by $\alpha$ in the above inequality we obtain \eqref{E:Tri_Est}.
\end{proof}

Finally, we fix $\alpha$ and write \eqref{E:Tri_Est} in terms of the Stokes operator $A_\varepsilon$.

\begin{corollary} \label{C:Tri_Est_A}
  There exist $d_1,d_2>0$ independent of $\varepsilon$ such that
  \begin{multline} \label{E:Tri_Est_A}
    \left|\bigl((u\cdot\nabla)u,A_\varepsilon u\bigr)_{L^2(\Omega_\varepsilon)}\right| \leq \left(\frac{1}{4}+d_1\varepsilon^{1/2}\|A_\varepsilon^{1/2}u\|_{L^2(\Omega_\varepsilon)}\right)\|A_\varepsilon u\|_{L^2(\Omega_\varepsilon)}^2 \\
    +d_2\left(\|u\|_{L^2(\Omega_\varepsilon)}^2\|A_\varepsilon^{1/2}u\|_{L^2(\Omega_\varepsilon)}^4+\varepsilon^{-1}\|u\|_{L^2(\Omega_\varepsilon)}^2\|A_\varepsilon^{1/2}u\|_{L^2(\Omega_\varepsilon)}^2\right)
  \end{multline}
  for all $\varepsilon\in(0,\varepsilon_0)$ and $u\in D(A_\varepsilon)$.
\end{corollary}

\begin{proof}
  Applying \eqref{E:Stokes_H1} and \eqref{E:Stokes_H2} to the right-hand side of \eqref{E:Tri_Est} we get
  \begin{multline*}
    \left|\bigl((u\cdot\nabla)u,A_\varepsilon u\bigr)_{L^2(\Omega_\varepsilon)}\right| \leq \left(c\alpha+d_\alpha^1\varepsilon^{1/2}\|A_\varepsilon^{1/2}u\|_{L^2(\Omega_\varepsilon)}\right)\|A_\varepsilon u\|_{L^2(\Omega_\varepsilon)}^2 \\
    +d_\alpha^2\left(\|u\|_{L^2(\Omega_\varepsilon)}^2\|A_\varepsilon^{1/2}u\|_{L^2(\Omega_\varepsilon)}^4+\varepsilon^{-1}\|u\|_{L^2(\Omega_\varepsilon)}^2\|A_\varepsilon^{1/2}u\|_{L^2(\Omega_\varepsilon)}^2\right)
  \end{multline*}
  with positive constants $c$, $d_\alpha^1$, and $d_\alpha^2$ independent of $\varepsilon$.
  We take $\alpha=1/4c$ in the above inequality to obtain \eqref{E:Tri_Est_A}.
\end{proof}

\section{Global existence and uniform estimates of a strong solution} \label{S:GE}
Based on the results in the previous sections we prove Theorems~\ref{T:GE} and~\ref{T:UE}.
As in Section~\ref{S:Tri} we impose Assumptions~\ref{Assump_1} and~\ref{Assump_2} and fix the constant $\varepsilon_0$ given in Lemma~\ref{L:Bili_Core}.
For $\varepsilon\in(0,\varepsilon_0)$ let $\mathcal{H}_\varepsilon$ and $\mathcal{V}_\varepsilon$ be the function spaces given by \eqref{E:Def_Heps} and $A_\varepsilon$ the Stokes operator on $\mathcal{H}_\varepsilon$.
We also write $\bar{\eta}=\eta\circ\pi$ for the constant extension of a function of $\eta$ on $\Gamma$ in the normal direction of $\Gamma$.

First we recall the well-known result on the local-in-time existence of a strong solution to the Navier--Stokes equations (see e.g.~\cite{BoFa13,CoFo88,So01,Te79}).

\begin{theorem} \label{T:LE}
  Let $u_0^\varepsilon\in \mathcal{V}_\varepsilon$, $f^\varepsilon\in L^\infty(0,\infty;L^2(\Omega_\varepsilon)^3)$ and suppose that $f^\varepsilon(t)\in\mathcal{R}_g^\perp$ for a.a. $t\in(0,\infty)$ when the condition (A3) of Assumption~\ref{Assump_2} is satisfied.
  Then there exists $T_0\in(0,\infty)$ depending on $\Omega_\varepsilon$, $\nu$, $u_0^\varepsilon$, and $f^\varepsilon$ such that the problem \eqref{E:NS_Eq}--\eqref{E:NS_In} admits a strong solution $u^\varepsilon$ on $[0,T_0)$ satisfying
  \begin{align*}
    u^\varepsilon \in C([0,T];\mathcal{V}_\varepsilon)\cap L^2(0,T;D(A_\varepsilon))\cap H^1(0,T;\mathcal{H}_\varepsilon) \quad\text{for all}\quad T\in(0,T_0).
  \end{align*}
  If $u^\varepsilon$ is maximally defined on the time interval $[0,T_{\max})$ and $T_{\max}$ is finite, then
  \begin{align*}
    \lim_{t\to T_{\max}^-}\|A_\varepsilon^{1/2}u^\varepsilon(t)\|_{L^2(\Omega_\varepsilon)} = \infty.
  \end{align*}
\end{theorem}

Note that the assumption $f^\varepsilon(t)\in\mathcal{R}_g^\perp$ for a.a. $t\in(0,\infty)$ is required to recover the original problem \eqref{E:NS_Eq}--\eqref{E:NS_In} properly from its abstract form (see Remark~\ref{R:Recov_NS}).

To establish the global-in-time existence of a strong solution $u^\varepsilon$ we show that the $L^2(\Omega_\varepsilon)$-norm of $A_\varepsilon^{1/2}u^\varepsilon(t)$ is bounded uniformly in $t$.
We argue by a standard energy method and use the uniform Gronwall inequality (see~\cite[Lemma~D.3]{SeYo02}).

\begin{lemma}[Uniform Gronwall inequality] \label{L:Uni_Gronwall}
  Let $z$, $\xi$, and $\zeta$ be nonnegative functions in $L_{loc}^1([0,T);\mathbb{R})$, $T\in(0,\infty]$.
  Suppose that $z\in C(0,T;\mathbb{R})$ and
  \begin{align*}
    \frac{dz}{dt}(t) \leq \xi(t)z(t)+\zeta(t) \quad\text{for a.a.}\quad t\in(0,T).
  \end{align*}
  Then $z\in L_{loc}^\infty(0,T;\mathbb{R})$ and
  \begin{align*}
    z(t_2) \leq \left(\frac{1}{t_2-t_1}\int_{t_1}^{t_2}z(s)\,ds+\int_{t_1}^{t_2}\zeta(s)\,ds\right)\exp\left(\int_{t_1}^{t_2}\xi(s)\,ds\right)
  \end{align*}
  for all $t_1,t_2\in(0,T)$ with $t_1<t_2$.
\end{lemma}

We also use an estimate for the duality product between a vector field on $\Omega_\varepsilon$ and the constant extension of a tangential vector field on $\Gamma$.

\begin{lemma} \label{L:Con_Dual}
  There exists a constant $c>0$ independent of $\varepsilon$ such that
  \begin{align} \label{E:Con_Dual}
    \left|(\bar{v},u)_{L^2(\Omega_\varepsilon)}\right| \leq c\varepsilon^{1/2}\|v\|_{H^{-1}(\Gamma,T\Gamma)}\|u\|_{H^1(\Omega_\varepsilon)}
  \end{align}
  for all $v\in L^2(\Gamma,T\Gamma)$ and $u\in H^1(\Omega_\varepsilon)^3$.
\end{lemma}

\begin{proof}
  We use the notation \eqref{E:Pull_Dom} and define
  \begin{align*}
    \eta(y) := \int_{\varepsilon g_0(y)}^{\varepsilon g_1(y)}u^\sharp(y,r)J(y,r)\,dr, \quad y\in\Gamma.
  \end{align*}
  In what follows, we suppress the arguments of functions.
  Let us show $\eta\in H^1(\Gamma)^3$.
  By \eqref{E:Jac_Bound}, H\"{o}lder's inequality, and \eqref{E:CoV_Equiv},
  \begin{align} \label{Pf_CDu:Eta_L2}
    \|\eta\|_{L^2(\Gamma)}^2 \leq \int_\Gamma\varepsilon g\left(\int_{\varepsilon g_0}^{\varepsilon g_1}|u^\sharp|^2\,dr\right)d\mathcal{H}^2 \leq c\varepsilon\|u\|_{L^2(\Omega_\varepsilon)}^2.
  \end{align}
  Also, by the same calculations as in the proof of Lemma~\ref{L:Ave_Der} we have
  \begin{align*}
    \nabla_\Gamma\eta = \int_{\varepsilon g_0}^{\varepsilon g_1}\left\{\frac{\partial}{\partial r}\Bigl(J\psi_\varepsilon^\sharp\otimes u^\sharp \Bigr)+J(B\nabla u)^\sharp+\nabla_\Gamma J\otimes u^\sharp\right\}\,dr \quad\text{on}\quad \Gamma,
  \end{align*}
  where $B$ and $\psi_\varepsilon$ are given by \eqref{E:Ave_Der_Aux}.
  By this equality, \eqref{E:Jac_Bound}, \eqref{E:ADA_Bound}, and \eqref{E:ADA_Grad_Bound},
  \begin{align*}
    |\nabla_\Gamma\eta| \leq c\int_{\varepsilon g_0}^{\varepsilon g_1}(|u^\sharp|+|(\nabla u)^\sharp|)\,dr \quad\text{on}\quad \Gamma.
  \end{align*}
  Hence H\"{o}lder's inequality and \eqref{E:CoV_Equiv} imply that
  \begin{align} \label{Pf_Cdu:Eta_H1}
    \|\nabla_\Gamma\eta\|_{L^2(\Gamma)}^2 \leq c\int_\Gamma\varepsilon g\left(\int_{\varepsilon g_0}^{\varepsilon g_1}(|u^\sharp|^2+|(\nabla u)^\sharp|^2)\,dr\right)d\mathcal{H}^2 \leq c\varepsilon\|u\|_{H^1(\Omega_\varepsilon)}^2.
  \end{align}
  From \eqref{Pf_CDu:Eta_L2} and \eqref{Pf_Cdu:Eta_H1} we deduce that $\eta\in H^1(\Gamma)^3$ and
  \begin{align} \label{Pf_Cdu:PEta}
    \|P\eta\|_{H^1(\Gamma)} \leq c\|\eta\|_{H^1(\Gamma)} \leq c\varepsilon^{1/2}\|u\|_{H^1(\Omega_\varepsilon)}.
  \end{align}
  Now we observe by \eqref{E:CoV_Dom} and $v\in L^2(\Gamma,T\Gamma)$ that
  \begin{align*}
    (\bar{v},u)_{L^2(\Omega_\varepsilon)} = \int_\Gamma v\cdot\left(\int_{\varepsilon g_0}^{\varepsilon g_1}u^\sharp J\,dr\right)d\mathcal{H}^2 = (v,\eta)_{L^2(\Gamma)} = (v,P\eta)_{L^2(\Gamma)}.
  \end{align*}
  Moreover, since $(v,P\eta)_{L^2(\Gamma)} = [v,P\eta]_{T\Gamma}$ by $P\eta\in H^1(\Gamma,T\Gamma)$,
  \begin{align*}
    \left|(\bar{v},u)_{L^2(\Omega_\varepsilon)}\right| = \bigl|[v,P\eta]_{T\Gamma}\bigr| \leq \|v\|_{H^{-1}(\Gamma,T\Gamma)}\|P\eta\|_{H^1(\Gamma)}.
  \end{align*}
  Applying \eqref{Pf_Cdu:PEta} to this inequality we obtain \eqref{E:Con_Dual}.
\end{proof}

Now we are ready to establish the global-in-time existence of a strong solution to the Navier--Stokes equations \eqref{E:NS_Eq}--\eqref{E:NS_In}.

\begin{proof}[Proof of Theorem~\ref{T:GE}]
  We follow the idea of the proofs of~\cite[Theorem~7.4]{Ho10} and~\cite[Theorem~3.1]{HoSe10}.
  In what follows, we write $c$ for a general positive constant independent of $\varepsilon$, $c_0$, and $T_{\max}$ and use the notation \eqref{E:Def_U_TN} for the tangential and normal components (with respect to $\Gamma$) of a vector field on $\Omega_\varepsilon$.

  Under Assumptions~\ref{Assump_1} and~\ref{Assump_2}, let $\varepsilon_0$ be the constant given in Lemma~\ref{L:Bili_Core}.
  For $\varepsilon\in(0,\varepsilon_0)$ let $u_0^\varepsilon\in \mathcal{V}_\varepsilon$ and $f\in L^\infty(0,\infty;L^2(\Omega_\varepsilon)^3)$ satisfy \eqref{E:GE_Data} with
  \begin{align} \label{Pf_GE:C0}
    c_0 := \min\left\{1,\frac{d_3^2}{4},\frac{d_3^2}{4d_4}\right\}, \quad d_3 := \frac{1}{4d_1},
  \end{align}
  where $d_1$ is the positive constant given in Corollary~\ref{C:Tri_Est_A} and $d_4$ is a positive constant given later.
  Noting that $M_\tau u_0^\varepsilon=Mu_{0,\tau}^\varepsilon$ on $\Gamma$ and
  \begin{align*}
    u_0^\varepsilon = u_{0,n}^\varepsilon+\Bigl(u_{0,\tau}^\varepsilon-\overline{Mu_{0,\tau}^\varepsilon}\Bigr)+\overline{M_\tau u_0^\varepsilon}, \quad |u_{0,n}^\varepsilon| = |u_0^\varepsilon\cdot\bar{n}| \quad\text{in}\quad \Omega_\varepsilon,
  \end{align*}
  we apply \eqref{E:Con_Lp}, \eqref{E:Poin_Nor}, and \eqref{E:Ave_Diff_Dom} to the right-hand side of
  \begin{align*}
    \|u_0^\varepsilon\|_{L^2(\Omega_\varepsilon)}^2 &\leq c\left(\|u_0^\varepsilon\cdot\bar{n}\|_{L^2(\Omega_\varepsilon)}^2+\left\|u_{0,\tau}^\varepsilon-\overline{Mu_{0,\tau}^\varepsilon}\right\|_{L^2(\Omega_\varepsilon)}^2+\left\|\overline{M_\tau u_0^\varepsilon}\right\|_{L^2(\Omega_\varepsilon)}^2\right)
  \end{align*}
  and then use $\|u_{0,\tau}^\varepsilon\|_{H^1(\Omega_\varepsilon)}\leq c\|u_0^\varepsilon\|_{H^1(\Omega_\varepsilon)}$ and \eqref{E:Stokes_H1} to get
  \begin{align} \label{Pf_GE:Ini_Data}
    \|u_0^\varepsilon\|_{L^2(\Omega_\varepsilon)}^2 \leq c\left(\varepsilon^2\|A_\varepsilon^{1/2}u_0^\varepsilon\|_{L^2(\Omega_\varepsilon)}^2+\varepsilon\|M_\tau u_0^\varepsilon\|_{L^2(\Gamma)}^2\right).
  \end{align}
  Hence from \eqref{E:GE_Data} it follows that
  \begin{align} \label{Pf_GE:Ini_L2}
    \|u_0^\varepsilon\|_{L^2(\Omega_\varepsilon)}^2 \leq cc_0.
  \end{align}
  Let $u^\varepsilon$ be a strong solution to the Navier--Stokes equations \eqref{E:NS_Eq}--\eqref{E:NS_In} defined on the maximal time interval $[0,T_{\max})$.
  First we derive estimates for
  \begin{align*}
    \|u^\varepsilon(t)\|_{L^2(\Omega_\varepsilon)}^2, \quad \int_t^{\min\{t+1,T_{\max}\}}\|A_\varepsilon^{1/2}u^\varepsilon(s)\|_{L^2(\Omega_\varepsilon)}^2\,ds, \quad t\in[0,T_{\max})
  \end{align*}
  with explicit dependence of constants on $\varepsilon$.
  Taking the $L^2(\Omega_\varepsilon)$-inner product of
  \begin{align} \label{Pf_GE:Ab_NS}
    \partial_tu^\varepsilon+A_\varepsilon u^\varepsilon+\mathbb{P}_\varepsilon(u^\varepsilon\cdot\nabla)u^\varepsilon = \mathbb{P}_\varepsilon f^\varepsilon \quad\text{on}\quad (0,T_{\max})
  \end{align}
  with $u^\varepsilon$ and using
  \begin{align*}
    (\mathbb{P}_\varepsilon(u^\varepsilon\cdot\nabla)u^\varepsilon,u^\varepsilon)_{L^2(\Omega_\varepsilon)} &= \bigl((u^\varepsilon\cdot\nabla)u^\varepsilon,u^\varepsilon\bigr)_{L^2(\Omega_\varepsilon)} = 0
  \end{align*}
  by integration by parts, $\mathrm{div}\,u^\varepsilon=0$ in $\Omega_\varepsilon$, and $u^\varepsilon\cdot n_\varepsilon=0$ on $\Gamma_\varepsilon$ we get
  \begin{align} \label{Pf_GE:L2_Inner}
    \frac{1}{2}\frac{d}{dt}\|u^\varepsilon\|_{L^2(\Omega_\varepsilon)}^2+\|A_\varepsilon^{1/2}u^\varepsilon\|_{L^2(\Omega_\varepsilon)}^2 = (\mathbb{P}_\varepsilon f^\varepsilon,u^\varepsilon)_{L^2(\Omega_\varepsilon)} \quad\text{on}\quad (0,T_{\max}).
  \end{align}
  We split the right-hand side into
  \begin{align*}
    J_1 := \Bigl(\mathbb{P}_\varepsilon f^\varepsilon,u^\varepsilon-\overline{M_\tau u^\varepsilon}\Bigr)_{L^2(\Omega_\varepsilon)}, \quad J_2 := \Bigl(\mathbb{P}_\varepsilon f^\varepsilon,\overline{M_\tau u^\varepsilon}\Bigr)_{L^2(\Omega_\varepsilon)}
  \end{align*}
  and estimate them separately.
  Since
  \begin{align*}
    u^\varepsilon-\overline{M_\tau u^\varepsilon} = u_n^\varepsilon+\Bigl(u_\tau^\varepsilon-\overline{Mu_\tau^\varepsilon}\Bigr), \quad |u_n^\varepsilon| = |u^\varepsilon\cdot\bar{n}| \quad\text{in}\quad \Omega_\varepsilon,
  \end{align*}
  we observe by \eqref{E:Poin_Nor}, \eqref{E:Ave_Diff_Dom}, and $\|u_\tau^\varepsilon\|_{H^1(\Omega_\varepsilon)}\leq c\|u^\varepsilon\|_{H^1(\Omega_\varepsilon)}$ that
  \begin{align*}
    \left\|u^\varepsilon-\overline{M_\tau u^\varepsilon}\right\|_{L^2(\Omega_\varepsilon)} \leq \|u^\varepsilon\cdot\bar{n}\|_{L^2(\Omega_\varepsilon)}+\left\|u_\tau^\varepsilon-\overline{Mu_\tau^\varepsilon}\right\|_{L^2(\Omega_\varepsilon)} \leq c\varepsilon\|u^\varepsilon\|_{H^1(\Omega_\varepsilon)}.
  \end{align*}
  This inequality and \eqref{E:Stokes_H1} imply
  \begin{align*}
    |J_1| \leq \|\mathbb{P}_\varepsilon f^\varepsilon\|_{L^2(\Omega_\varepsilon)}\left\|u^\varepsilon-\overline{M_\tau u^\varepsilon}\right\|_{L^2(\Omega_\varepsilon)} \leq c\varepsilon\|\mathbb{P}_\varepsilon f^\varepsilon\|_{L^2(\Omega_\varepsilon)}\|A_\varepsilon^{1/2}u^\varepsilon\|_{L^2(\Omega_\varepsilon)}.
  \end{align*}
  To $J_2$ we use \eqref{E:Ave_Inner} and \eqref{E:Con_Dual} to get
  \begin{align*}
    |J_2| &\leq \left|\Bigl(\overline{M_\tau\mathbb{P}_\varepsilon f^\varepsilon},u^\varepsilon\Bigr)_{L^2(\Omega_\varepsilon)}\right|+\left|\Bigl(\mathbb{P}_\varepsilon f^\varepsilon,\overline{M_\tau u^\varepsilon}\Bigr)_{L^2(\Omega_\varepsilon)}-\Bigl(\overline{M_\tau\mathbb{P}_\varepsilon f^\varepsilon},u^\varepsilon\Bigr)_{L^2(\Omega_\varepsilon)}\right| \\
    &\leq c\left(\varepsilon^{1/2}\|M_\tau\mathbb{P}_\varepsilon f^\varepsilon\|_{H^{-1}(\Gamma,T\Gamma)}\|u^\varepsilon\|_{H^1(\Omega_\varepsilon)}+\varepsilon\|\mathbb{P}_\varepsilon f^\varepsilon\|_{L^2(\Omega_\varepsilon)}\|u^\varepsilon\|_{L^2(\Omega_\varepsilon)}\right).
  \end{align*}
  We apply these estimates to $(\mathbb{P}_\varepsilon f^\varepsilon,u^\varepsilon)_{L^2(\Omega_\varepsilon)}=J_1+J_2$ and use
  \begin{align} \label{Pf_GE:L2_Po}
    \|u^\varepsilon\|_{L^2(\Omega_\varepsilon)} \leq \|u^\varepsilon\|_{H^1(\Omega_\varepsilon)} \leq c\|A_\varepsilon^{1/2}u^\varepsilon\|_{L^2(\Omega_\varepsilon)}
  \end{align}
  by \eqref{E:Stokes_H1} and Young's inequality to obtain
  \begin{multline*}
    |(\mathbb{P}_\varepsilon f^\varepsilon,u^\varepsilon)_{L^2(\Omega_\varepsilon)}| \\
    \leq \frac{1}{2}\|A_\varepsilon^{1/2}u^\varepsilon\|_{L^2(\Omega_\varepsilon)}^2+c\left(\varepsilon^2\|\mathbb{P}_\varepsilon f^\varepsilon\|_{L^2(\Omega_\varepsilon)}^2+\varepsilon\|M_\tau\mathbb{P}_\varepsilon f^\varepsilon\|_{H^{-1}(\Gamma,T\Gamma)}^2\right).
  \end{multline*}
  From this inequality and \eqref{Pf_GE:L2_Inner} we deduce that
  \begin{multline} \label{Pf_GE:L2_InEst1}
    \frac{d}{dt}\|u^\varepsilon\|_{L^2(\Omega_\varepsilon)}^2+\|A_\varepsilon^{1/2}u^\varepsilon\|_{L^2(\Omega_\varepsilon)}^2 \\
    \leq c\left(\varepsilon^2\|\mathbb{P}_\varepsilon f^\varepsilon\|_{L^2(\Omega_\varepsilon)}^2+\varepsilon\|M_\tau\mathbb{P}_\varepsilon f^\varepsilon\|_{H^{-1}(\Gamma,T\Gamma)}^2\right) \quad\text{on}\quad (0,T_{\max}).
  \end{multline}
  By \eqref{Pf_GE:L2_Po} we further get
  \begin{multline*}
    \frac{d}{dt}\|u^\varepsilon\|_{L^2(\Omega_\varepsilon)}^2+\frac{1}{a_1}\|u^\varepsilon\|_{L^2(\Omega_\varepsilon)}^2 \\
    \leq c\left(\varepsilon^2\|\mathbb{P}_\varepsilon f^\varepsilon\|_{L^2(\Omega_\varepsilon)}^2+\varepsilon\|M_\tau\mathbb{P}_\varepsilon f^\varepsilon\|_{H^{-1}(\Gamma,T\Gamma)}^2\right) \quad\text{on}\quad (0,T_{\max})
  \end{multline*}
  with a constant $a_1>0$ independent of $\varepsilon$, $c_0$, and $T_{\max}$.
  For each $t\in(0,T_{\max})$ we multiply both sides of this inequality at $s\in(0,t)$ by $e^{(s-t)/a_1}$ and integrate them over $(0,t)$.
  Then we have
  \begin{multline} \label{Pf_GE:LinfL2}
    \|u^\varepsilon(t)\|_{L^2(\Omega_\varepsilon)}^2 \leq e^{-t/a_1}\|u_0^\varepsilon\|_{L^2(\Omega_\varepsilon)}^2 \\
    +ca_1(1-e^{-t/a_1})\left(\varepsilon^2\|\mathbb{P}_\varepsilon f^\varepsilon\|_{L^\infty(0,\infty;L^2(\Omega_\varepsilon))}^2+\varepsilon\|M_\tau\mathbb{P}_\varepsilon f^\varepsilon\|_{L^\infty(0,\infty;H^{-1}(\Gamma,T\Gamma))}^2\right).
  \end{multline}
  Also, integrating \eqref{Pf_GE:L2_InEst1} over $(t,t_\ast)$ with $t_\ast:=\min\{t+1,T_{\max}\}$ we deduce that
  \begin{multline} \label{Pf_GE:L2H1}
    \int_t^{t_\ast}\|A_\varepsilon^{1/2}u^\varepsilon(s)\|_{L^2(\Omega_\varepsilon)}^2\,ds \leq \|u_0^\varepsilon\|_{L^2(\Omega_\varepsilon)}^2 \\
    +c\left(\varepsilon^2\|\mathbb{P}_\varepsilon f^\varepsilon\|_{L^\infty(0,\infty;L^2(\Omega_\varepsilon))}^2+\varepsilon\|M_\tau\mathbb{P}_\varepsilon f^\varepsilon\|_{L^\infty(0,\infty;H^{-1}(\Gamma,T\Gamma))}^2\right).
  \end{multline}
  Hence we apply \eqref{E:GE_Data} and \eqref{Pf_GE:Ini_L2} to the right-hand sides of \eqref{Pf_GE:LinfL2} and \eqref{Pf_GE:L2H1} to get
  \begin{align} \label{Pf_GE:L2_Ener}
    \|u^\varepsilon(t)\|_{L^2(\Omega_\varepsilon)}^2+\int_t^{t_\ast}\|A_\varepsilon^{1/2}u^\varepsilon(s)\|_{L^2(\Omega_\varepsilon)}^2\,ds \leq cc_0 \quad\text{for all}\quad t\in[0,T_{\max}).
  \end{align}
  Next we show that $\|A_\varepsilon^{1/2}u^\varepsilon(t)\|_{L^2(\Omega_\varepsilon)}$ is uniformly bounded in $t\in[0,T_{\max})$ (note that it is continuous on $[0,T_{\max})$ by $u^\varepsilon\in C([0,T_{\max});\mathcal{V}_\varepsilon)$).
  Our goal is to prove
  \begin{align} \label{Pf_GE:Goal}
    \varepsilon^{1/2}\|A_\varepsilon^{1/2}u^\varepsilon(t)\|_{L^2(\Omega_\varepsilon)} < d_3 = \frac{1}{4d_1} \quad\text{for all}\quad t\in[0,T_{\max}).
  \end{align}
  If \eqref{Pf_GE:Goal} is valid, then Theorem~\ref{T:LE} implies that $T_{\max}=\infty$, i.e. the strong solution $u^\varepsilon$ exists on the whole time interval $[0,\infty)$.
  First note that \eqref{Pf_GE:Goal} is valid at $t=0$ by \eqref{E:GE_Data} and \eqref{Pf_GE:C0}.
  Let us prove \eqref{Pf_GE:Goal} for all $t\in(0,T_{\max})$ by contradiction.
  Assume to the contrary that there exists $T\in(0,T_{\max})$ such that
  \begin{align}
    \varepsilon^{1/2}\|A_\varepsilon^{1/2}u^\varepsilon(t)\|_{L^2(\Omega_\varepsilon)} &< d_3 \quad\text{for all}\quad t\in[0,T), \label{Pf_GE:Cont_Ine} \\
    \varepsilon^{1/2}\|A_\varepsilon^{1/2}u^\varepsilon(T)\|_{L^2(\Omega_\varepsilon)} &= d_3. \label{Pf_GE:Cont_Eq}
  \end{align}
  We consider \eqref{Pf_GE:Ab_NS} on $(0,T]$ and take its $L^2(\Omega_\varepsilon)$-inner product with $A_\varepsilon u^\varepsilon$ to get
  \begin{multline} \label{PF_GE:H1_Inner}
    \frac{1}{2}\frac{d}{dt}\|A_\varepsilon^{1/2}u^\varepsilon\|_{L^2(\Omega_\varepsilon)}^2+\|A_\varepsilon u^\varepsilon\|_{L^2(\Omega_\varepsilon)}^2 \\
    \leq \left|\bigl((u^\varepsilon\cdot\nabla)u^\varepsilon,A_\varepsilon u^\varepsilon\bigr)_{L^2(\Omega_\varepsilon)}\right|+|(\mathbb{P}_\varepsilon f^\varepsilon,A_\varepsilon u^\varepsilon)_{L^2(\Omega_\varepsilon)}| \quad\text{on}\quad (0,T].
  \end{multline}
  By \eqref{E:Tri_Est_A} and \eqref{Pf_GE:Cont_Ine}--\eqref{Pf_GE:Cont_Eq} with $d_3=1/4d_1$ we have
  \begin{multline*}
    \left|\bigl((u^\varepsilon\cdot\nabla)u^\varepsilon,A_\varepsilon u^\varepsilon\bigr)_{L^2(\Omega_\varepsilon)}\right| \leq \frac{1}{2}\|A_\varepsilon u^\varepsilon\|_{L^2(\Omega_\varepsilon)}^2 \\
    +d_2\left(\|u^\varepsilon\|_{L^2(\Omega_\varepsilon)}^2\|A_\varepsilon^{1/2}u^\varepsilon\|_{L^2(\Omega_\varepsilon)}^4+\varepsilon^{-1}\|u^\varepsilon\|_{L^2(\Omega_\varepsilon)}^2\|A_\varepsilon^{1/2}u^\varepsilon\|_{L^2(\Omega_\varepsilon)}^2\right)
  \end{multline*}
  on $(0,T]$.
  Also, Young's inequality yields
  \begin{align*}
    |(\mathbb{P}_\varepsilon f^\varepsilon,A_\varepsilon u^\varepsilon)_{L^2(\Omega_\varepsilon)}| \leq \frac{1}{4}\|A_\varepsilon u^\varepsilon\|_{L^2(\Omega_\varepsilon)}^2+\|\mathbb{P}_\varepsilon f^\varepsilon\|_{L^2(\Omega_\varepsilon)}^2.
  \end{align*}
  Applying these inequalities to the right-hand side of \eqref{PF_GE:H1_Inner} we obtain
  \begin{align} \label{Pf_GE:H1_InEst1}
    \frac{d}{dt}\|A_\varepsilon^{1/2}u^\varepsilon\|_{L^2(\Omega_\varepsilon)}^2+\frac{1}{2}\|A_\varepsilon u^\varepsilon\|_{L^2(\Omega_\varepsilon)}^2 \leq \xi\|A_\varepsilon^{1/2}u^\varepsilon\|_{L^2(\Omega_\varepsilon)}^2+\zeta \quad\text{on}\quad (0,T],
  \end{align}
  where the functions $\xi$ and $\zeta$ are given by
  \begin{align} \label{Pf_GE:Def_Xi}
    \begin{aligned}
      \xi(t) &:= 2d_2\|u^\varepsilon(t)\|_{L^2(\Omega_\varepsilon)}^2\|A_\varepsilon^{1/2}u^\varepsilon(t)\|_{L^2(\Omega_\varepsilon)}^2, \\
     \zeta(t) &:= 2\left(d_2\varepsilon^{-1}\|u^\varepsilon(t)\|_{L^2(\Omega_\varepsilon)}^2\|A_\varepsilon^{1/2}u^\varepsilon(t)\|_{L^2(\Omega_\varepsilon)}^2+\|\mathbb{P}_\varepsilon f^\varepsilon(t)\|_{L^2(\Omega_\varepsilon)}^2\right)
    \end{aligned}
  \end{align}
  for $t\in(0,T]$.
  By \eqref{Pf_GE:L2_Ener}, \eqref{Pf_GE:Cont_Ine}, and \eqref{Pf_GE:Cont_Eq} we see that
  \begin{align*}
    \xi \leq cc_0\varepsilon^{-1}, \quad \zeta \leq c\left(c_0\varepsilon^{-1}\|A_\varepsilon^{1/2}u^\varepsilon\|_{L^2(\Omega_\varepsilon)}^2+\|\mathbb{P}_\varepsilon f^\varepsilon\|_{L^2(\Omega_\varepsilon)}^2\right) \quad\text{on}\quad (0,T].
  \end{align*}
  Applying these inequalities to \eqref{Pf_GE:H1_InEst1} we have
  \begin{multline} \label{Pf_GE:H1_InEst2}
    \frac{d}{dt}\|A_\varepsilon^{1/2}u^\varepsilon\|_{L^2(\Omega_\varepsilon)}^2+\frac{1}{2}\|A_\varepsilon u^\varepsilon\|_{L^2(\Omega_\varepsilon)}^2 \\
    \leq c\left(c_0\varepsilon^{-1}\|A_\varepsilon^{1/2}u^\varepsilon\|_{L^2(\Omega_\varepsilon)}^2+\|\mathbb{P}_\varepsilon f^\varepsilon\|_{L^2(\Omega_\varepsilon)}^2\right) \quad\text{on}\quad (0,T].
  \end{multline}
  From \eqref{E:Stokes_Po} and \eqref{Pf_GE:H1_InEst2} we further deduce that
  \begin{align*}
    \frac{d}{dt}\|A_\varepsilon^{1/2}u^\varepsilon\|_{L^2(\Omega_\varepsilon)}^2+\frac{1}{a_2}\|A_\varepsilon^{1/2}u^\varepsilon\|_{L^2(\Omega_\varepsilon)}^2 \leq c\left(c_0\varepsilon^{-1}\|A_\varepsilon^{1/2}u^\varepsilon\|_{L^2(\Omega_\varepsilon)}^2+\|\mathbb{P}_\varepsilon f^\varepsilon\|_{L^2(\Omega_\varepsilon)}^2\right)
  \end{align*}
  on $(0,T]$ with a constant $a_2>0$ independent of $\varepsilon$, $c_0$, and $T_{\max}$.
  For $t\in(0,T]$ we multiply both sides of the above inequality at $s\in(0,t)$ by $e^{(s-t)/a_2}$ and integrate them over $(0,t)$ to get
  \begin{multline} \label{Pf_GE:LinfH1_Int}
    \|A_\varepsilon^{1/2}u^\varepsilon(t)\|_{L^2(\Omega_\varepsilon)}^2 \\
    \leq e^{-t/a_2}\|A_\varepsilon^{1/2}u_0^\varepsilon\|_{L^2(\Omega_\varepsilon)}^2+cc_0\varepsilon^{-1}\int_0^te^{(s-t)/a_2}\|A_\varepsilon^{1/2}u^\varepsilon(s)\|_{L^2(\Omega_\varepsilon)}^2\,ds \\
    +ca_2(1-e^{-t/a_2})\|\mathbb{P}_\varepsilon f^\varepsilon\|_{L^\infty(0,\infty;L^2(\Omega_\varepsilon))}^2.
  \end{multline}
  When $t\leq T_\ast:=\min\{1,T\}$ we apply \eqref{E:GE_Data}, \eqref{Pf_GE:L2_Ener}, and $c_0\leq 1$ to \eqref{Pf_GE:LinfH1_Int} to obtain
  \begin{align} \label{Pf_GE:LinfH1_1}
    \|A_\varepsilon^{1/2}u^\varepsilon(t)\|_{L^2(\Omega_\varepsilon)}^2 \leq cc_0(1+c_0)\varepsilon^{-1} \leq cc_0\varepsilon^{-1} \quad\text{for all}\quad t\in(0,T_\ast].
  \end{align}
  Now we suppose that $T\geq 1$ and estimate $\|A_\varepsilon^{1/2}u^\varepsilon(t)\|_{L^2(\Omega_\varepsilon)}$ for $t\in[1,T]$.
  Since
  \begin{align*}
    \frac{d}{dt}\|A_\varepsilon^{1/2}u^\varepsilon\|_{L^2(\Omega_\varepsilon)}^2 \leq \xi\|A_\varepsilon^{1/2}u^\varepsilon\|_{L^2(\Omega_\varepsilon)}^2+\zeta \quad\text{on}\quad (0,T]
  \end{align*}
  by \eqref{Pf_GE:H1_InEst1}, we can use Lemma~\ref{L:Uni_Gronwall} with $z(t)=\|A_\varepsilon^{1/2}u^\varepsilon(t)\|_{L^2(\Omega_\varepsilon)}^2$ to obtain
  \begin{multline} \label{Pf_GE:Uni_Gron}
    \|A_\varepsilon^{1/2}u^\varepsilon(t)\|_{L^2(\Omega_\varepsilon)}^2 \\
    \leq \left(\int_{t-1}^t\|A_\varepsilon^{1/2}u^\varepsilon(s)\|_{L^2(\Omega_\varepsilon)}^2\,ds+\int_{t-1}^t\zeta(s)\,ds\right)\exp\left(\int_{t-1}^t\xi(s)\,ds\right)
  \end{multline}
  for all $t\in[1,T]$.
  Moreover, the functions $\xi$ and $\zeta$ given by \eqref{Pf_GE:Def_Xi} satisfy
  \begin{align*}
    \int_{t-1}^t\xi(s)\,ds &\leq cc_0\int_{t-1}^t\|A_\varepsilon^{1/2}u^\varepsilon(s)\|_{L^2(\Omega_\varepsilon)}^2\,ds \leq c, \\
    \int_{t-1}^t\zeta(s)\,ds &\leq c\left(c_0\varepsilon^{-1}\int_{t-1}^t\|A_\varepsilon^{1/2}u^\varepsilon(s)\|_{L^2(\Omega_\varepsilon)}^2\,ds+\|\mathbb{P}_\varepsilon f^\varepsilon\|_{L^\infty(0,\infty;L^2(\Omega_\varepsilon))}^2\right) \\
    &\leq cc_0\varepsilon^{-1}
  \end{align*}
  by \eqref{E:GE_Data}, \eqref{Pf_GE:L2_Ener}, and $c_0\leq1$.
  Using these inequalities and \eqref{Pf_GE:L2_Ener} to \eqref{Pf_GE:Uni_Gron} we have
  \begin{align} \label{Pf_GE:LinfH1_2}
    \|A_\varepsilon^{1/2}u^\varepsilon(t)\|_{L^2(\Omega_\varepsilon)}^2 \leq cc_0\varepsilon^{-1} \quad\text{for all}\quad t\in[1,T].
  \end{align}
  Now we combine \eqref{Pf_GE:LinfH1_1} and \eqref{Pf_GE:LinfH1_2} to observe that
  \begin{align*}
    \|A_\varepsilon^{1/2}u^\varepsilon(t)\|_{L^2(\Omega_\varepsilon)}^2 \leq d_4c_0\varepsilon^{-1} \quad\text{for all}\quad t\in(0,T]
  \end{align*}
  with a constant $d_4>0$ independent of $\varepsilon$, $c_0$, and $T$.
  Hence if we define the constant $c_0$ by \eqref{Pf_GE:C0}, then by setting $t=T$ in the above inequality we get
  \begin{align*}
    \|A_\varepsilon^{1/2}u^\varepsilon(T)\|_{L^2(\Omega_\varepsilon)}^2 \leq \frac{d_3^2\varepsilon^{-1}}{4}, \quad\text{i.e.}\quad \varepsilon^{1/2}\|A_\varepsilon^{1/2}u^\varepsilon(T)\|_{L^2(\Omega_\varepsilon)} \leq \frac{d_3}{2} < d_3,
  \end{align*}
  which contradicts with \eqref{Pf_GE:Cont_Eq}.
  Therefore, the inequality \eqref{Pf_GE:Goal} is valid for all $t\in[0,T_{\max})$ and we conclude by Theorem~\ref{T:LE} that $T_{\max}=\infty$, i.e. the strong solution $u^\varepsilon$ to \eqref{E:NS_Eq}--\eqref{E:NS_In} exists on the whole time interval $[0,\infty)$.
\end{proof}

Using the inequalities given in the proof of Theorem~\ref{T:GE}, we can also show the uniform estimates \eqref{E:UE_L2} and \eqref{E:UE_H1} for a strong solution.

\begin{proof}[Proof of Theorem~\ref{T:UE}]
  Let $\varepsilon_0,c_0\in(0,1)$ be the constants given in Theorem~\ref{T:GE}.
  Since $\alpha$ and $\beta$ are positive we can take $\varepsilon_1\in(0,\varepsilon_0)$ such that $c_1\varepsilon^\alpha\leq c_0$ and $c_2\varepsilon^\beta\leq c_0$ for all $\varepsilon\in(0,\varepsilon_1)$.
  Hence for $\varepsilon\in(0,\varepsilon_1)$ if $u_0^\varepsilon$ and $f^\varepsilon$ satisfy \eqref{E:UE_Data} then the inequality \eqref{E:GE_Data} holds and Theorem~\ref{T:GE} gives the existence of a global strong solution $u^\varepsilon$ to \eqref{E:NS_Eq}--\eqref{E:NS_In}.

  Let us derive the estimates \eqref{E:UE_L2} and \eqref{E:UE_H1} for the strong solution $u^\varepsilon$.
  Hereafter we denote by $c$ a general positive constant independent of $\varepsilon$.
  First note that
  \begin{align} \label{Pf_UE:Ini}
    \|u_0^\varepsilon\|_{L^2(\Omega_\varepsilon)}^2 \leq c(\varepsilon^{1+\alpha}+\varepsilon^\beta)
  \end{align}
  by \eqref{E:UE_Data} and \eqref{Pf_GE:Ini_Data}.
  We apply this inequality and \eqref{E:UE_Data} to \eqref{Pf_GE:LinfL2} to get
  \begin{align} \label{Pf_UE:U_LinfL2}
    \|u^\varepsilon(t)\|_{L^2(\Omega_\varepsilon)}^2 \leq c(\varepsilon^{1+\alpha}+\varepsilon^\beta) \quad\text{for all}\quad t\geq 0.
  \end{align}
  Also, integrating \eqref{Pf_GE:L2_InEst1} over $[0,t]$ and using \eqref{E:UE_Data} and \eqref{Pf_UE:Ini} we have
  \begin{align} \label{Pf_UE:U_L2H1}
    \int_0^t\|A_\varepsilon^{1/2}u^\varepsilon(s)\|_{L^2(\Omega_\varepsilon)}^2\,ds \leq c(\varepsilon^{1+\alpha}+\varepsilon^\beta)(1+t) \quad\text{for all}\quad t\geq 0.
  \end{align}
  Combining \eqref{Pf_UE:U_LinfL2} and \eqref{Pf_UE:U_L2H1} with \eqref{E:Stokes_H1} we obtain \eqref{E:UE_L2}.

  Next let us prove \eqref{E:UE_H1}.
  We use \eqref{E:UE_Data} and \eqref{Pf_UE:Ini} to \eqref{Pf_GE:L2H1} to deduce that
  \begin{align} \label{Pf_UE:L2H1_Short}
    \int_t^{t+1}\|A_\varepsilon^{1/2}u^\varepsilon(s)\|_{L^2(\Omega_\varepsilon)}^2\,ds \leq c(\varepsilon^{1+\alpha}+\varepsilon^\beta) \quad\text{for all}\quad t\geq 0.
  \end{align}
  Note that $t_\ast=\min\{t+1,T_{\max}\}=t+1$ in \eqref{Pf_GE:L2H1} by $T_{\max}=\infty$.
  Since \eqref{Pf_GE:Goal} and \eqref{Pf_UE:U_LinfL2} are valid for all $t\geq 0$, we can derive \eqref{Pf_GE:LinfH1_Int} for all $t\geq 0$ as in the proof of Theorem~\ref{T:GE}.
  When $t\in[0,1]$, we apply \eqref{E:UE_Data} and \eqref{Pf_UE:U_L2H1} to \eqref{Pf_GE:LinfH1_Int} to obtain
  \begin{align} \label{Pf_UE:U_LinfH1_1}
    \|A_\varepsilon^{1/2}u^\varepsilon(t)\|_{L^2(\Omega_\varepsilon)}^2 \leq c(\varepsilon^{-1+\alpha}+\varepsilon^{-1+\beta}) \quad\text{for all}\quad t\in[0,1].
  \end{align}
  Let $t\geq1$.
  In \eqref{Pf_GE:Uni_Gron} the functions $\xi$ and $\zeta$ given by \eqref{Pf_GE:Def_Xi} satisfy
  \begin{align*}
    \int_{t-1}^t\xi(s)\,ds &\leq c\int_{t-1}^t\|A_\varepsilon^{1/2}u^\varepsilon(s)\|_{L^2(\Omega_\varepsilon)}^2\,ds \leq c, \\
    \int_{t-1}^t\zeta(s)\,ds &\leq c\left(\varepsilon^{-1}\int_{t-1}^t\|A_\varepsilon^{1/2}u^\varepsilon(s)\|_{L^2(\Omega_\varepsilon)}^2\,ds+\|\mathbb{P}_\varepsilon f^\varepsilon\|_{L^\infty(0,\infty;L^2(\Omega_\varepsilon))}^2\right) \\
    &\leq c(\varepsilon^{-1+\alpha}+\varepsilon^{-1+\beta})
  \end{align*}
  by \eqref{E:UE_Data}, \eqref{Pf_GE:L2_Ener}, and \eqref{Pf_UE:L2H1_Short}.
  Applying these estimates and \eqref{Pf_UE:L2H1_Short} to \eqref{Pf_GE:Uni_Gron} we get
  \begin{align} \label{Pf_UE:U_LinfH1_2}
    \|A_\varepsilon^{1/2}u^\varepsilon(t)\|_{L^2(\Omega_\varepsilon)}^2 \leq c(\varepsilon^{-1+\alpha}+\varepsilon^{-1+\beta}) \quad\text{for all}\quad t\geq 1.
  \end{align}
  By \eqref{E:Stokes_H1}, \eqref{Pf_UE:U_LinfH1_1}, and \eqref{Pf_UE:U_LinfH1_2} we obtain the first inequality of \eqref{E:UE_H1}.
  To derive the second one we see that \eqref{Pf_GE:H1_InEst2} holds on $[0,\infty)$ since \eqref{Pf_GE:L2_Ener} and \eqref{Pf_GE:Goal} are valid on $[0,\infty)$.
  Hence we integrate \eqref{Pf_GE:H1_InEst2} over $[0,t]$ and use \eqref{E:UE_Data} and \eqref{Pf_UE:U_L2H1} to get
  \begin{align*}
    \int_0^t\|A_\varepsilon u^\varepsilon(s)\|_{L^2(\Omega_\varepsilon)}^2\,ds \leq c(\varepsilon^{-1+\alpha}+\varepsilon^{-1+\beta})(1+t) \quad\text{for all}\quad t\geq 0.
  \end{align*}
  This inequality combined with \eqref{E:Stokes_H2} yields the second inequality of \eqref{E:UE_H1}.
\end{proof}

For the study of a singular limit problem in Section~\ref{S:SL} we need estimates for the tensor product and the time derivative of a strong solution with $\beta=1$.
Let us derive them by using the inequalities given in Lemma~\ref{L:Est_UU}.

\begin{theorem} \label{T:Est_Ue}
  Under the assumptions of Theorem~\ref{T:GE}, let $c_1,c_2>0$, $\alpha\in(0,1]$, $\beta=1$, and $\varepsilon_1$ be the constant given in Theorem~\ref{T:UE}.
  For $\varepsilon\in(0,\varepsilon_1)$ let $u_0^\varepsilon\in \mathcal{V}_\varepsilon$ and $f^\varepsilon\in L^2(0,\infty;L^2(\Omega_\varepsilon)^3)$ satisfy \eqref{E:UE_Data}.
  When the condition (A3) of Assumption~\ref{Assump_2} is satisfied, suppose that $f^\varepsilon(t)\in\mathcal{R}_g^\perp$ for a.a. $t\in(0,\infty)$.
  Then there exists a global strong solution $u^\varepsilon$ to \eqref{E:NS_Eq}--\eqref{E:NS_In} such that
  \begin{align} \label{E:Est_Ue}
    \begin{gathered}
      \|u^\varepsilon(t)\|_{L^2(\Omega_\varepsilon)}^2 \leq c\varepsilon, \quad \int_0^t\|u^\varepsilon(s)\|_{H^1(\Omega_\varepsilon)}^2\,ds \leq c\varepsilon(1+t), \\
      \|u^\varepsilon(t)\|_{H^1(\Omega_\varepsilon)}^2 \leq c\varepsilon^{-1+\alpha}, \quad \int_0^t\|u^\varepsilon(s)\|_{H^2(\Omega_\varepsilon)}^2\,ds \leq c\varepsilon^{-1+\alpha}(1+t)
    \end{gathered}
  \end{align}
  for all $t\geq0$, and
  \begin{align}
    \int_0^t\|[u^\varepsilon\otimes u^\varepsilon](s)\|_{L^2(\Omega_\varepsilon)}^2\,ds &\leq c\varepsilon(1+t), \label{E:Est_UeUe} \\
    \int_0^t\|\partial_tu^\varepsilon(s)\|_{L^2(\Omega_\varepsilon)}^2\,ds &\leq c\varepsilon^{-1+\alpha}(1+t) \label{E:Est_DtUe}
  \end{align}
  for all $t\geq0$, where $c>0$ is a constant independent of $\varepsilon$, $u^\varepsilon$, and $t$.
\end{theorem}

\begin{proof}
  A global strong solution $u^\varepsilon$ exists by Theorem~\ref{T:UE}.
  Also, the estimates \eqref{E:Est_Ue} follow from \eqref{E:UE_L2}--\eqref{E:UE_H1} since $\beta=1$ and $\varepsilon\leq\varepsilon^\alpha$ by $\alpha\leq 1$ and $\varepsilon<1$.

  Let us derive \eqref{E:Est_UeUe}.
  Hereafter we suppress the argument $s$ of integrands.
  Since $u^\varepsilon\in L_{loc}^2([0,\infty);D(A_\varepsilon))$ satisfies the conditions of Lemma~\ref{L:Est_UU},
  \begin{multline} \label{Pf_EsSt:UeUe}
    \int_0^t\|u^\varepsilon\otimes u^\varepsilon\|_{L^2(\Omega_\varepsilon)}^2\,ds \leq c\left(\varepsilon^{-1}\int_0^t\|u^\varepsilon\|_{L^2(\Omega_\varepsilon)}^2\|u\|_{H^1(\Omega_\varepsilon)}^2\,ds\right.\\
    \left.+\varepsilon\int_0^t\|u^\varepsilon\|_{L^2(\Omega_\varepsilon)}^2\|u^\varepsilon\|_{H^2(\Omega_\varepsilon)}^2\,ds+\int_0^t\|u^\varepsilon\|_{L^2(\Omega_\varepsilon)}^3\|u^\varepsilon\|_{H^2(\Omega_\varepsilon)}\,ds\right)
  \end{multline}
  for all $t\geq 0$ by \eqref{E:Est_UU}.
  From \eqref{E:Est_Ue} we deduce that
  \begin{align}
    \int_0^t\|u^\varepsilon\|_{L^2(\Omega_\varepsilon)}^2\|u^\varepsilon\|_{H^1(\Omega_\varepsilon)}^2\,ds &\leq c\varepsilon^2(1+t), \label{Pf_EsSt:L2H1} \\
    \int_0^t\|u^\varepsilon\|_{L^2(\Omega_\varepsilon)}^2\|u^\varepsilon\|_{H^2(\Omega_\varepsilon)}^2\,ds &\leq c\varepsilon^\alpha(1+t). \label{Pf_EsSt:L2H2}
  \end{align}
  Also, H\"{o}lder's inequality and \eqref{E:Est_Ue} imply that
  \begin{align} \label{Pf_EsSt:L2H2_31}
    \begin{aligned}
      \int_0^t\|u^\varepsilon\|^3_{L^2(\Omega_\varepsilon)}\|u^\varepsilon\|_{H^2(\Omega_\varepsilon)}\,ds &\leq c\varepsilon^{3/2}t^{1/2}\left(\int_0^t\|u^\varepsilon\|_{H^2(\Omega_\varepsilon)}^2\,ds\right)^{1/2} \\
      &\leq c\varepsilon^{1+\alpha/2}(1+t).
    \end{aligned}
  \end{align}
  Applying \eqref{Pf_EsSt:L2H1}--\eqref{Pf_EsSt:L2H2_31} to \eqref{Pf_EsSt:UeUe} and noting that $\varepsilon^\alpha,\varepsilon^{\alpha/2}\leq 1$ we obtain \eqref{E:Est_UeUe}.

  Next we prove \eqref{E:Est_DtUe}.
  We take the $L^2(\Omega_\varepsilon)$-inner product of $\partial_tu^\varepsilon$ with
  \begin{align*}
    \partial_tu^\varepsilon+A_\varepsilon u^\varepsilon+\mathbb{P}_\varepsilon[(u^\varepsilon\cdot\nabla)u^\varepsilon] = \mathbb{P}_\varepsilon f^\varepsilon \quad\text{on}\quad (0,\infty)
  \end{align*}
  and then integrate it over $(0,t)$ with $t>0$ and use Young's inequality to get
  \begin{multline} \label{Pf_EsSt:DtUe}
    \int_0^t\|\partial_tu^\varepsilon\|_{L^2(\Omega_\varepsilon)}^2\,ds+\|A_\varepsilon^{1/2}u^\varepsilon(t)\|_{L^2(\Omega_\varepsilon)}^2 \\
    \leq \|A_\varepsilon^{1/2}u_0^\varepsilon\|_{L^2(\Omega_\varepsilon)}^2+c\int_0^t\left(\|(u^\varepsilon\cdot\nabla)u^\varepsilon\|_{L^2(\Omega_\varepsilon)}^2+\|\mathbb{P}_\varepsilon f^\varepsilon\|_{L^2(\Omega_\varepsilon)}^2\right)ds.
  \end{multline}
  Let us estimate the integral of $(u^\varepsilon\cdot\nabla)u^\varepsilon$.
  By \eqref{E:Est_UGU} we have
  \begin{align*}
    \int_0^t\|(u^\varepsilon\cdot\nabla)u^\varepsilon\|_{L^2(\Omega_\varepsilon)}^2\,ds \leq c\int_0^t\left(\varepsilon^{-1}\|u^\varepsilon\|_{L^2(\Omega_\varepsilon)}^2+\varepsilon\|u^\varepsilon\|_{H^1(\Omega_\varepsilon)}^2\right)\|u^\varepsilon\|_{H^2(\Omega_\varepsilon)}^2\,ds.
  \end{align*}
  To the right-hand side we apply \eqref{Pf_EsSt:L2H2} and
  \begin{align} \label{Pf_EsSt:H1H2}
    \int_0^r\|u^\varepsilon\|_{H^1(\Omega_\varepsilon)}^2\|u^\varepsilon\|_{H^2(\Omega_\varepsilon)}^2\,ds \leq c\varepsilon^{-2+2\alpha}(1+t)
  \end{align}
  by \eqref{E:Est_Ue}, and then use $\varepsilon^\alpha\leq 1$ by $\varepsilon<1$ to deduce that
  \begin{align*}
    \int_0^t\|(u^\varepsilon\cdot\nabla)u^\varepsilon\|_{L^2(\Omega_\varepsilon)}^2\,ds \leq c\varepsilon^{-1+\alpha}(1+t).
  \end{align*}
  Applying this inequality and \eqref{E:UE_Data} to \eqref{Pf_EsSt:DtUe} we obtain \eqref{E:Est_DtUe}.
\end{proof}

\section{Weighted solenoidal spaces on a closed surface} \label{S:WSol}
In the study of a singular limit problem for the Navier--Stokes equations \eqref{E:NS_Eq}--\eqref{E:NS_In} we deal with a weighted solenoidal space of the form
\begin{align*}
  H_{g\sigma}^1(\Gamma,T\Gamma) = \{v\in H^1(\Gamma,T\Gamma) \mid \text{$\mathrm{div}_\Gamma(gv) = 0$ on $\Gamma$}\}.
\end{align*}
The purpose of this section is to give several properties of weighted solenoidal spaces on a closed surface.
Throughout this section we assume that $\Gamma$ is a two-dimensional closed, connected, and oriented surface in $\mathbb{R}^3$ of class $C^2$.
We use the notations for the surface quantities on $\Gamma$ given in Section~\ref{SS:Pre_Surf}.

\subsection{Ne\v{c}as inequality on a closed surface} \label{SS:WS_Nec}
Let $q\in L^2(\Gamma)$.
We consider $q$ and its weak tangential gradient as elements in $H^{-1}(\Gamma)$ and $H^{-1}(\Gamma,T\Gamma)$ given by \eqref{E:L2_Hin} and \eqref{E:TGr_HinT}.
Then we immediately get
\begin{align*}
  \|q\|_{H^{-1}(\Gamma)}+\|\nabla_\Gamma q\|_{H^{-1}(\Gamma,T\Gamma)} \leq c\|q\|_{L^2(\Gamma)}.
\end{align*}
For bounded Lipschitz domains in $\mathbb{R}^m$, $m\in\mathbb{N}$ the inverse inequality is also valid and known as the Ne\v{c}as inequality (see~\cite[Theorem~IV.1.1]{BoFa13} and~\cite[Chapter~3, Lemma~7.1]{Ne12}).
Let us show the Ne\v{c}as inequality on the closed surface $\Gamma$.

\begin{lemma} \label{L:Necas}
  There exists a constant $c>0$ such that
  \begin{align} \label{E:Necas}
    \|q\|_{L^2(\Gamma)} \leq c\left(\|q\|_{H^{-1}(\Gamma)}+\|\nabla_\Gamma q\|_{H^{-1}(\Gamma,T\Gamma)}\right)
  \end{align}
  for all $q\in L^2(\Gamma)$.
\end{lemma}

In the proof of Lemma~\ref{L:Necas} we use the Na\v{c}as inequality on the whole space
\begin{align} \label{E:Necas_Whole}
  \|\tilde{q}\|_{L^2(\mathbb{R}^2)} \leq c\left(\|\tilde{q}\|_{H^{-1}(\mathbb{R}^2)}+\|\nabla_s\tilde{q}\|_{H^{-1}(\mathbb{R}^2)}\right)
\end{align}
for $\tilde{q}\in L^2(\mathbb{R}^2)$, where $H^{-1}(\mathbb{R}^2)$ is the dual space of $H^1(\mathbb{R}^2)$ (via the $L^2(\mathbb{R}^2)$-inner product).
Also, $\tilde{q}\in H^{-1}(\mathbb{R}^2)$ and $\nabla_s\tilde{q}\in H^{-1}(\mathbb{R}^2)^2$ are given by
\begin{align*}
  \langle\tilde{q},\xi\rangle_{\mathbb{R}^2} := (\tilde{q},\xi)_{L^2(\mathbb{R}^2)}, \quad \langle\nabla_s\tilde{q},\varphi\rangle_{\mathbb{R}^2} := -(\tilde{q},\partial_{s_1}\varphi_1+\partial_{s_2}\varphi_2)_{L^2(\mathbb{R}^2)}
\end{align*}
for $\xi\in H^1(\mathbb{R}^2)$ and $\varphi=(\varphi_1,\varphi_2)\in H^1(\mathbb{R}^2)^2$, where $\langle\cdot,\cdot\rangle_{\mathbb{R}^2}$ is the duality product between $H^{-1}(\mathbb{R}^2)$ and $H^1(\mathbb{R}^2)$.
The inequality \eqref{E:Necas_Whole} follows from the characterization of the $L^2$-Sobolev spaces on $\mathbb{R}^2$ by the Fourier transform.
For details, see the proofs of~\cite[Proposition~IV.1.2]{BoFa13} and~\cite[Chapter~3, Lemma~7.1]{Ne12}.

\begin{proof}
  First we note that it is sufficient to show \eqref{E:Necas} when $q$ is compactly supported in a relatively open subset of $\Gamma$ on which we can take a local coordinate system.
  To see this, let $q\in L^2(\Gamma)$ and $\eta\in C^2(\Gamma)$.
  For $\xi\in H^1(\Gamma)$ we have
  \begin{align*}
    |\langle \eta q,\xi \rangle_\Gamma| = |\langle q,\eta\xi \rangle_\Gamma| \leq \|q\|_{H^{-1}(\Gamma)}\|\eta\xi\|_{H^1(\Gamma)} \leq c\|\eta\|_{W^{1,\infty}(\Gamma)}\|q\|_{H^{-1}(\Gamma)}\|\xi\|_{H^1(\Gamma)},
  \end{align*}
  where $c>0$ is a constant independent of $q$, $\eta$, and $\xi$.
  Also,
  \begin{align*}
    [\nabla_\Gamma(\eta q),v]_{T\Gamma} &= -(\eta q,\mathrm{div}_\Gamma v)_{L^2(\Gamma)} = -\bigl(q,\mathrm{div}_\Gamma(\eta v)\bigr)_{L^2(\Gamma)}+(q,\nabla_\Gamma\eta\cdot v)_{L^2(\Gamma)} \\
    &= [\nabla_\Gamma q,\eta v]_{T\Gamma}+\langle q,\nabla_\Gamma\eta\cdot v\rangle_\Gamma
  \end{align*}
  for all $v\in H^1(\Gamma,T\Gamma)$ by \eqref{E:TGr_HinT} (note that $\eta v\in H^1(\Gamma,T\Gamma)$) and thus
  \begin{align*}
    |[\nabla_\Gamma(\eta q),v]_{T\Gamma}| &\leq \|\nabla_\Gamma q\|_{H^{-1}(\Gamma,T\Gamma)}\|\eta v\|_{H^1(\Gamma)}+\|q\|_{H^{-1}(\Gamma)}\|\nabla_\Gamma\eta\cdot v\|_{H^1(\Gamma)} \\
    &\leq c\|\eta\|_{W^{2,\infty}(\Gamma)}\left(\|q\|_{H^{-1}(\Gamma)}+\|\nabla_\Gamma q\|_{H^{-1}(\Gamma,T\Gamma)}\right)\|v\|_{H^1(\Gamma)}.
  \end{align*}
  From the above inequalities it follows that
  \begin{align*}
    \|\eta q\|_{H^{-1}(\Gamma)} &\leq c\|\eta\|_{W^{1,\infty}(\Gamma)}\|q\|_{H^{-1}(\Gamma)}, \\
    \|\nabla_\Gamma(\eta q)\|_{H^{-1}(\Gamma,T\Gamma)} &\leq c\|\eta\|_{W^{2,\infty}(\Gamma)}\left(\|q\|_{H^{-1}(\Gamma)}+\|\nabla_\Gamma q\|_{H^{-1}(\Gamma,T\Gamma)}\right).
  \end{align*}
  Hence if we localize $q$ by a partition of unity on $\Gamma$ consisting of functions in $C^2(\Gamma)$ (such functions exist by the $C^2$-regularity of $\Gamma$) and prove \eqref{E:Necas} for each localized function, then we can get \eqref{E:Necas} for $q$ by the above inequalities.

  Now we assume that $q\in L^2(\Gamma)$ is supported in $\mu(\mathcal{K})$, where $\mu\colon U\to\mathbb{R}^2$ is a $C^2$ local parametrization of $\Gamma$ with an open set $U$ in $\mathbb{R}^2$ and $\mathcal{K}$ is a compact subset of $U$.
  Let $\theta=(\theta_{ij})_{i,j}$ be the Riemannian metric of $\Gamma$ given by \eqref{E:Def_Met} and $\theta^{-1}=(\theta^{ij})_{i,j}$ its inverse.
  We set $\tilde{q}:=q\circ\mu$ on $U$ and extend it to $\mathbb{R}^2$ by setting zero outside of $U$.
  Then $\tilde{q}\in L^2(\mathbb{R}^2)$ and it is supported in $\mathcal{K}$ by the assumption on $q$ and Lemma~\ref{L:Metric}.
  Hence we observe by \eqref{E:Necas_Whole} and \eqref{E:Lp_Loc} that
  \begin{align} \label{Pf_Nec:Whole}
    \|q\|_{L^2(\Gamma)} \leq c\|\tilde{q}\|_{L^2(\mathbb{R}^2)} \leq c\left(\|\tilde{q}\|_{H^{-1}(\mathbb{R}^2)}+\|\nabla_s\tilde{q}\|_{H^{-1}(\mathbb{R}^2)}\right).
  \end{align}
  To estimate the right-hand side of \eqref{Pf_Nec:Whole} by that of \eqref{E:Necas}, we first consider $\langle \tilde{q},\xi\rangle_{\mathbb{R}^2} = (\tilde{q},\xi)_{L^2(\mathbb{R}^2)}$ for $\xi\in H^1(\mathbb{R}^2)$.
  We may assume that $\xi$ is supported in $\mathcal{K}$ since $\tilde{q}$ is so.
  Noting that $\det\theta>0$ on $\mathcal{K}$ by \eqref{E:Metric}, we define a function $\eta$ on $\mu(\mathcal{K})\subset\Gamma$ by
  \begin{align} \label{Pf_Nec:Def_Eta}
    \eta(\mu(s)) := \frac{\xi(s)}{\sqrt{\det\theta(s)}}, \quad s\in \mathcal{K}
  \end{align}
  and extend it to $\Gamma$ by setting zero outside of $\mu(\mathcal{K})$.
  Then we have
  \begin{align} \label{Pf_Nec:Dual_L2}
    \langle\tilde{q},\xi\rangle_{\mathbb{R}^2} = (\tilde{q},\xi)_{L^2(\mathbb{R}^2)} = \int_{\mathcal{K}}\tilde{q}\left(\frac{\xi}{\sqrt{\det\theta}}\right)\sqrt{\det\theta}\,ds = \int_{\mu(\mathcal{K})}q\eta\,d\mathcal{H}^2.
  \end{align}
  Let us show $\eta\in H^1(\Gamma)$.
  Since $\eta$ is supported in $\mu(\mathcal{K})$, we see by \eqref{E:Metric} that
  \begin{align} \label{Pf_Nec:Eta}
    \|\eta\|_{L^2(\Gamma)}^2 = \int_{\mathcal{K}}|\eta\circ\mu|^2\sqrt{\det\theta}\,ds \leq c\|\xi\|_{L^2(\mathcal{K})}^2 = c\|\xi\|_{L^2(\mathbb{R}^2)}^2.
  \end{align}
  Also, differentiating \eqref{Pf_Nec:Def_Eta} with respect to $s_i$ and using \eqref{E:Metric} we get
  \begin{align*}
    |\partial_{s_i}(\eta\circ\mu)(s)| \leq c(|\xi(s)|+|\partial_{s_i}\xi(s)|), \quad s\in\mathcal{K},\,i=1,2.
  \end{align*}
  By $(\nabla_\Gamma\eta)\circ\mu=\sum_{i,j=1}^2\theta^{ij}\partial_{s_i}(\eta\circ\mu)\partial_{s_j}\mu$ in $\mathcal{K}$, \eqref{E:Metric}, and the above inequality,
  \begin{align*}
    |(\nabla_\Gamma\eta)\circ\mu|^2 = \sum_{i,j=1}^2\theta^{ij}\partial_{s_i}(\eta\circ\mu)\partial_{s_j}(\eta\circ\mu) \leq c(|\xi|^2+|\nabla_s\xi|^2) \quad\text{in}\quad \mathcal{K}.
  \end{align*}
  Here $\nabla_s$ is the gradient operator in $s\in\mathbb{R}^2$.
  Noting that $\eta$ and $\xi$ are supported in $\mu(\mathcal{K})$ and $\mathcal{K}$, respectively, we use this inequality and \eqref{E:Metric} to obtain
  \begin{align*}
    \|\nabla_\Gamma\eta\|_{L^2(\Gamma)}^2 = \int_{\mathcal{K}}|(\nabla_\Gamma\eta)\circ\mu|^2\sqrt{\det\theta}\,ds \leq c\|\xi\|_{H^1(\mathcal{K})}^2 = c\|\xi\|_{H^1(\mathbb{R}^2)}^2.
  \end{align*}
  By this inequality and \eqref{Pf_Nec:Eta} we have $\eta\in H^1(\Gamma)$ and $\|\eta\|_{H^1(\Gamma)}\leq c\|\xi\|_{H^1(\mathbb{R}^2)}$.
  Hence
  \begin{align*}
    \langle\tilde{q},\xi\rangle_{\mathbb{R}^2} = \int_{\mu(\mathcal{K})}q\eta\,d\mathcal{H}^2 = (q,\eta)_{L^2(\Gamma)} = \langle q,\eta\rangle_\Gamma
  \end{align*}
  by \eqref{Pf_Nec:Dual_L2} and the above inequality implies that
  \begin{align*}
    \bigl|\langle\tilde{q},\xi\rangle_{\mathbb{R}^2}\bigr| = \bigl|\langle q,\eta\rangle_\Gamma\bigr| \leq \|q\|_{H^{-1}(\Gamma)}\|\eta\|_{H^1(\Gamma)} \leq c\|q\|_{H^{-1}(\Gamma)}\|\xi\|_{H^1(\mathbb{R}^2)}
  \end{align*}
  for all $\xi\in H^1(\mathbb{R}^2)$.
  Therefore,
  \begin{align} \label{Pf_Nec:TilQ}
    \|\tilde{q}\|_{H^{-1}(\mathbb{R}^2)} \leq c\|q\|_{H^{-1}(\Gamma)}.
  \end{align}
  Next we consider the duality product $\langle\nabla_s\tilde{q},\varphi\rangle_{\mathbb{R}^2}=-(\tilde{q},\partial_{s_1}\varphi_1+\partial_{s_2}\varphi_2)_{L^2(\mathbb{R}^2)}$ for $\varphi=(\varphi_1,\varphi_2)\in H^1(\mathbb{R}^2)^2$.
  We may assume again that $\varphi$ is supported in $\mathcal{K}$ since $\tilde{q}$ is so.
  Noting that the surface divergence of $X\circ\mu=\sum_{i=1}^2X_i\partial_{s_i}\mu$ on $U$ is given by
  \begin{align*}
    \mathrm{div}_\Gamma X(\mu(s)) = \frac{1}{\sqrt{\det\theta(s)}}\sum_{i=1}^2\frac{\partial}{\partial s_i}\Bigl(X_i(s)\sqrt{\det\theta(s)}\Bigr), \quad s\in U,
  \end{align*}
  we define a tangential vector field $\Phi$ on $\mu(\mathcal{K})$ by
  \begin{align} \label{Pf_Nec:Def_Phi}
    \Phi(\mu(s)) := \frac{1}{\sqrt{\det\theta(s)}}\sum_{i=1}^2\varphi_i(s)\partial_{s_i}\mu(s), \quad s\in\mathcal{K}
  \end{align}
  and extend it to $\Gamma$ by setting zero outside of $\mu(\mathcal{K})$.
  Then since
  \begin{align*}
    \mathrm{div}_\Gamma\Phi(\mu(s)) = \frac{\partial_{s_1}\varphi_1(s)+\partial_{s_2}\varphi_2(s)}{\sqrt{\det\theta(s)}}, \quad s\in\mathcal{K}
  \end{align*}
  and $\tilde{q}=q\circ\mu$ is supported in $\mathcal{K}$, we have
  \begin{align} \label{Pf_Nec:Dual_H1}
    \begin{aligned}
      \langle\nabla_s\tilde{q},\varphi\rangle_{\mathbb{R}^2} &= -(\tilde{q},\partial_{s_1}\varphi_1+\partial_{s_2}\varphi_2)_{L^2(\mathbb{R}^2)} \\
      &= -\int_{\mathcal{K}}\tilde{q}\left(\frac{\partial_{s_1}\varphi_1+\partial_{s_2}\varphi_2}{\sqrt{\det\theta}}\right)\sqrt{\det\theta}\,ds = -\int_{\mu(\mathcal{K})}q\,\mathrm{div}_\Gamma\Phi\,d\mathcal{H}^2.
    \end{aligned}
  \end{align}
  Let us estimate the $H^1(\Gamma)$-norm of $\Phi$.
  By \eqref{Pf_Nec:Def_Phi} and \eqref{E:Metric} we have
  \begin{align*}
    |\Phi\circ\mu|^2 = \frac{1}{\det\theta}\sum_{i,j=1}^2\theta_{ij}\varphi_i\varphi_j \leq c|\varphi|^2 \quad\text{in}\quad \mathcal{K}.
  \end{align*}
  Since $\Phi$ and $\varphi$ are supported in $\mu(\mathcal{K})$ and $\mathcal{K}$, the above inequality and \eqref{E:Metric} imply
  \begin{align} \label{Pf_Nec:Phi}
    \|\Phi\|_{L^2(\Gamma)}^2 = \int_{\mathcal{K}}|\Phi\circ\mu|^2\sqrt{\det\theta}\,ds \leq c\|\varphi\|_{L^2(\mathcal{K})}^2 = c\|\varphi\|_{L^2(\mathbb{R}^2)}^2.
  \end{align}
  To estimate the tangential derivatives of $\Phi$, let $\{e_1,e_2,e_3\}$ be the standard basis of $\mathbb{R}^3$ and $\Phi_k:=\Phi\cdot e_k$ for $k=1,2,3$.
  Then $\Phi_k$ is supported in $\mu(\mathcal{K})$ and
  \begin{align*}
    \Phi_k(\mu(s)) = \frac{1}{\sqrt{\det\theta(s)}}\sum_{j=1}^2\varphi_j(s)\partial_{s_j}\mu(s)\cdot e_k, \quad s\in\mathcal{K}.
  \end{align*}
  We differentiate both sides with respect to $s_i$ and use \eqref{E:Mu_Bound} and \eqref{E:Metric} to get
  \begin{align*}
    \left|\partial_{s_i}(\Phi_k\circ\mu)(s)\right| \leq c\left(|\varphi(s)|+|\nabla_s\varphi(s)|\right), \quad s\in\mathcal{K},\,i=1,2.
  \end{align*}
  From this inequality and \eqref{E:Metric} we deduce that
  \begin{align*}
    |(\nabla_\Gamma\Phi_k)\circ\mu|^2 = \sum_{i,j=1}^2\theta^{ij}\partial_{s_i}(\Phi_k\circ\mu)\partial_{s_j}(\Phi_k\circ\mu) \leq c(|\varphi|^2+|\nabla_s\varphi|^2) \quad\text{in}\quad \mathcal{K}
  \end{align*}
  for $k=1,2,3$ and thus $|(\nabla_\Gamma\Phi)\circ\mu|^2\leq c(|\varphi|^2+|\nabla_s\varphi|^2)$ in $\mathcal{K}$.
  Noting that $\Phi$ and $\varphi$ are supported in $\mu(\mathcal{K})$ and $\mathcal{K}$, we use this inequality and \eqref{E:Metric} to get
  \begin{align*}
    \|\nabla_\Gamma\Phi\|_{L^2(\Gamma)}^2 = \int_{\mathcal{K}}|(\nabla_\Gamma\Phi)\circ\mu|^2\sqrt{\det\theta}\,ds \leq c\|\varphi\|_{H^1(\mathcal{K})}^2 = c\|\varphi\|_{H^1(\mathbb{R}^2)}^2.
  \end{align*}
  This inequality and \eqref{Pf_Nec:Phi} show that $\Phi\in H^1(\Gamma,T\Gamma)$ and $\|\Phi\|_{H^1(\Gamma)}\leq c\|\varphi\|_{H^1(\mathbb{R}^2)}$.
  Now we return to \eqref{Pf_Nec:Dual_H1} and use \eqref{E:TGr_HinT} to get
  \begin{align*}
    \langle\nabla_s\tilde{q},\varphi\rangle_{\mathbb{R}^2} = -\int_{\mu(\mathcal{K})}q\,\mathrm{div}_\Gamma\Phi\,d\mathcal{H}^2 = -(q,\mathrm{div}_\Gamma\Phi)_{L^2(\Gamma)} = [\nabla_\Gamma q,\Phi]_{T\Gamma}.
  \end{align*}
  Hence by $\|\Phi\|_{H^1(\Gamma)}\leq c\|\varphi\|_{H^1(\mathbb{R}^2)}$ we obtain
  \begin{align*}
    \bigl|\langle\nabla_s\tilde{q},\varphi\rangle_{\mathbb{R}^2}\bigr| = \bigl|[\nabla_\Gamma q,\Phi]_{T\Gamma}\bigr| &\leq \|\nabla_\Gamma q\|_{H^{-1}(\Gamma,T\Gamma)}\|\Phi\|_{H^1(\Gamma)} \\
    &\leq c\|\nabla_\Gamma q\|_{H^{-1}(\Gamma,T\Gamma)}\|\varphi\|_{H^1(\mathbb{R}^2)}
  \end{align*}
  for all $\varphi\in H^1(\mathbb{R}^2)^2$, which implies that
  \begin{align*}
    \|\nabla_s\tilde{q}\|_{H^{-1}(\mathbb{R}^2)} \leq c\|\nabla_\Gamma q\|_{H^{-1}(\Gamma,T\Gamma)}.
  \end{align*}
  Finally, we apply this inequality and \eqref{Pf_Nec:TilQ} to \eqref{Pf_Nec:Whole} to get \eqref{E:Necas}.
\end{proof}

Based on \eqref{E:Necas} we prove Poincar\'{e}'s inequality for $q\in L^2(\Gamma)$.
First we show that $\nabla_\Gamma q$ vanishes in $H^{-1}(\Gamma,T\Gamma)$ if and only if $q$ is constant on $\Gamma$.

\begin{lemma} \label{L:TGr_HinT_Con}
  Let $q\in L^2(\Gamma)$.
  Then
  \begin{align*}
    \nabla_\Gamma q = 0 \quad\text{in}\quad H^{-1}(\Gamma,T\Gamma)
  \end{align*}
  if and only if $q$ is constant on $\Gamma$.
\end{lemma}

\begin{proof}
  Suppose that $q$ is constant on $\Gamma$.
  Then for all $v\in H^1(\Gamma,T\Gamma)$ we have
  \begin{align*}
    [\nabla_\Gamma q,v]_{T\Gamma} = -q\int_\Gamma\mathrm{div}_\Gamma v\,d\mathcal{H}^2 = 0
  \end{align*}
  by \eqref{E:TGr_HinT} and the Stokes theorem.
  Hence $\nabla_\Gamma q=0$ in $H^{-1}(\Gamma,T\Gamma)$.

  Conversely, assume that $\nabla_\Gamma q=0$ in $H^{-1}(\Gamma,T\Gamma)$.
  We first prove $q\in H^1(\Gamma)$.
  For $\eta\in C^1(\Gamma)$ and $i=1,2,3$ we set $v:=\eta Pe_i$ on $\Gamma$, where $\{e_1,e_2,e_3\}$ is the standard basis of $\mathbb{R}^3$.
  Then $v\in H^1(\Gamma,T\Gamma)$ since $P$ is of class $C^1$ on $\Gamma$.
  Moreover,
  \begin{align*}
    \mathrm{div}_\Gamma v = \nabla_\Gamma \eta\cdot Pe_i+\eta(\mathrm{div}_\Gamma P\cdot e_i) = \underline{D}_i\eta+\eta Hn_i \quad\text{on}\quad \Gamma
  \end{align*}
  by $P^T=P$ on $\Gamma$, \eqref{E:P_TGr}, and \eqref{E:Form_W}.
  From this equality we deduce that
  \begin{align*}
    0 = [\nabla_\Gamma q,v]_{T\Gamma} = -(q,\mathrm{div}_\Gamma v)_{L^2(\Gamma)} = -(q,\underline{D}_i\eta+\eta Hn_i)_{L^2(\Gamma)}
  \end{align*}
  for all $\eta\in C^1(\Gamma)$.
  Hence by the definition of the weak tangential derivative in $L^2(\Gamma)$ (see \eqref{E:Def_WTD}) we get $\underline{D}_iq=0$ in $L^2(\Gamma)$ for $i=1,2,3$, which shows that $q\in H^1(\Gamma)$ and $\nabla_\Gamma q=0$ in $L^2(\Gamma)^3$.
  Now let $|\Gamma|$ be the area of $\Gamma$ and
  \begin{align*}
    \hat{q} := q-\frac{1}{|\Gamma|}\int_\Gamma q\,d\mathcal{H}^2 \in H^1(\Gamma).
  \end{align*}
  Then since $\int_\Gamma\hat{q}\,d\mathcal{H}^2=0$ we can apply Poincar\'{e}'s inequality \eqref{E:Poin_Surf_Lp} to $\hat{q}$ to get
  \begin{align*}
    \|\hat{q}\|_{L^2(\Gamma)} \leq c\|\nabla_\Gamma\hat{q}\|_{L^2(\Gamma)} = c\|\nabla_\Gamma q\|_{L^2(\Gamma)} = 0.
  \end{align*}
  Hence $\hat{q}=0$ on $\Gamma$ and $q=|\Gamma|^{-1}\int_\Gamma q\,d\mathcal{H}^2$ is constant on $\Gamma$.
\end{proof}

Next we estimate $q\in L^2(\Gamma)$ in $H^{-1}(\Gamma)$ by its tangential gradient.

\begin{lemma} \label{L:Poin_HinT}
  There exists a constant $c>0$ such that
  \begin{align} \label{E:Poin_HinT}
    \|q\|_{H^{-1}(\Gamma)} \leq c\left(\left|\int_\Gamma q\,d\mathcal{H}^2\right|+\|\nabla_\Gamma q\|_{H^{-1}(\Gamma,T\Gamma)}\right)
  \end{align}
  for all $q\in L^2(\Gamma)$.
\end{lemma}

\begin{proof}
  Assume to the contrary that there exists $q_k\in L^2(\Gamma)$ such that
  \begin{align*}
    \|q_k\|_{H^{-1}(\Gamma)} > k\left(\left|\int_\Gamma q_k\,d\mathcal{H}^2\right|+\|\nabla_\Gamma q_k\|_{H^{-1}(\Gamma,T\Gamma)}\right)
  \end{align*}
  for each $k\in\mathbb{N}$.
  Replacing $q_k$ with $q_k/\|q_k\|_{H^{-1}(\Gamma)}$ we may assume that
  \begin{align} \label{Pf_PHin:Contra}
    \|q_k\|_{H^{-1}(\Gamma)} = 1, \quad \left|\int_\Gamma q_k\,d\mathcal{H}^2\right|+\|\nabla_\Gamma q_k\|_{H^{-1}(\Gamma,T\Gamma)} < k^{-1}.
  \end{align}
  From the second inequality it follows that
  \begin{align} \label{Pf_PHin:Conv}
    \lim_{k\to\infty}\int_\Gamma q_k\,d\mathcal{H}^2 = \lim_{k\to\infty}(q_k,1)_{L^2(\Gamma)} = 0, \quad \lim_{k\to\infty}\|\nabla_\Gamma q_k\|_{H^{-1}(\Gamma,T\Gamma)} = 0.
  \end{align}
  Since $\{q_k\}_{k=1}^\infty$ is bounded in $L^2(\Gamma)$ by \eqref{E:Necas} and \eqref{Pf_PHin:Contra} and the embedding $L^2(\Gamma)\hookrightarrow H^{-1}(\Gamma)$ is compact, (up to a subsequence) $\{q_k\}_{k=1}^\infty$ converges to some $q\in L^2(\Gamma)$ weakly in $L^2(\Gamma)$ and strongly in $H^{-1}(\Gamma)$.
  Then by the first equality of \eqref{Pf_PHin:Contra},
  \begin{align} \label{Pf_PHin:Lim_Hin}
    \|q\|_{H^{-1}(\Gamma)} = \lim_{k\to\infty}\|q_k\|_{H^{-1}(\Gamma)} = 1.
  \end{align}
  By \eqref{E:TGr_HinT} and the weak convergence of $\{q_k\}_{k=1}^\infty$ to $q$ in $L^2(\Gamma)$ we see that $\{\nabla_\Gamma q_k\}_{k=1}^\infty$ converges to $\nabla_\Gamma q$ weakly in $H^{-1}(\Gamma,T\Gamma)$.
  Hence by \eqref{Pf_PHin:Conv} we get
  \begin{align*}
    \int_\Gamma q\,d\mathcal{H}^2 = (q,1)_{L^2(\Gamma)} = 0, \quad \|\nabla_\Gamma q\|_{H^{-1}(\Gamma,T\Gamma)} = 0.
  \end{align*}
  By these equalities and Lemma~\ref{L:TGr_HinT_Con} we find that $q=0$ on $\Gamma$ and thus $\|q\|_{H^{-1}(\Gamma)}=0$, which contradicts with \eqref{Pf_PHin:Lim_Hin}.
  Hence \eqref{E:Poin_HinT} is valid.
\end{proof}

Combining \eqref{E:Necas} and \eqref{E:Poin_HinT} we obtain Poincar\'{e}'s inequality for $q\in L^2(\Gamma)$.

\begin{lemma} \label{L:Poin_L2_HinT}
  There exists a constant $c>0$ such that
  \begin{align} \label{E:Poin_L2_HinT}
    \|q\|_{L^2(\Gamma)} \leq c\left(\left|\int_\Gamma q\,d\mathcal{H}^2\right|+\|\nabla_\Gamma q\|_{H^{-1}(\Gamma,T\Gamma)}\right)
  \end{align}
  for all $q\in L^2(\Gamma)$.
\end{lemma}

\subsection{Characterization of the annihilator of a weighted solenoidal space} \label{SS:WS_An}
Let $g\in C^1(\Gamma)$ satisfy $g\geq c$ on $\Gamma$ with some constant $c>0$.
We define a weighted solenoidal space of tangential vector fields
\begin{align*}
  H_{g\sigma}^1(\Gamma,T\Gamma) := \{v\in H^1(\Gamma,T\Gamma) \mid \text{$\mathrm{div}_\Gamma(gv)=0$ on $\Gamma$}\}.
\end{align*}
If $q\in L^2(\Gamma)$, then \eqref{E:TGr_HinT} and \eqref{E:Def_Mul_HinT} imply that
\begin{align*}
  [g\nabla_\Gamma q, v]_{T\Gamma} = -\bigl(q,\mathrm{div}_\Gamma(gv)\bigr)_{L^2(\Gamma)} = 0 \quad\text{for all}\quad v\in H_{g\sigma}^1(\Gamma,T\Gamma).
\end{align*}
Let us prove the converse of this statement for an element of $H^{-1}(\Gamma,T\Gamma)$, which is a weighted version of de Rham's theorem.

\begin{theorem} \label{T:DeRham_T}
  Suppose that $f\in H^{-1}(\Gamma,T\Gamma)$ satisfies
  \begin{align*}
    [f,v]_{T\Gamma} = 0 \quad\text{for all}\quad v \in H_{g\sigma}^1(\Gamma,T\Gamma).
  \end{align*}
  Then there exists a unique $q\in L^2(\Gamma)$ such that
  \begin{align*}
    f = g\nabla_\Gamma q \quad\text{in}\quad H^{-1}(\Gamma,T\Gamma), \quad \int_\Gamma q\,d\mathcal{H}^2 = 0.
  \end{align*}
  Moreover, there exists a constant $c>0$ independent of $f$ such that
  \begin{align} \label{E:DeRham_T_Ineq}
    \|q\|_{L^2(\Gamma)} \leq c\|f\|_{H^{-1}(\Gamma,T\Gamma)}.
  \end{align}
\end{theorem}

We give auxiliary lemmas for Theorem~\ref{T:DeRham_T}.

\begin{lemma} \label{L:TGr_G_HinT}
  There exists a constant $c>0$ such that
  \begin{align} \label{E:TGr_G_HinT}
    c^{-1}\|\nabla_\Gamma q\|_{H^{-1}(\Gamma,T\Gamma)} \leq \|g\nabla_\Gamma q\|_{H^{-1}(\Gamma,T\Gamma)} \leq c\|\nabla_\Gamma q\|_{H^{-1}(\Gamma,T\Gamma)}
  \end{align}
  for all $q\in L^2(\Gamma)$.
\end{lemma}

\begin{proof}
  Since $g\in C^1(\Gamma)$ is bounded from below by a positive constant,
  \begin{align*}
    \left|[\nabla_\Gamma q,v]_{T\Gamma}\right| = \left|[g\nabla_\Gamma q,g^{-1}v]_{T\Gamma}\right| &\leq \|g\nabla_\Gamma q\|_{H^{-1}(\Gamma,T\Gamma)}\|g^{-1}v\|_{H^1(\Gamma)} \\
    &\leq c\|g\nabla_\Gamma q\|_{H^{-1}(\Gamma,T\Gamma)}\|v\|_{H^1(\Gamma)}
  \end{align*}
  for all $v\in H^1(\Gamma,T\Gamma)$.
  Hence the left-hand inequality of \eqref{E:TGr_G_HinT} holds.
  Similarly, we can show the right-hand inequality of \eqref{E:TGr_G_HinT}.
\end{proof}

\begin{lemma} \label{L:TGr_Closed}
  The subspace
  \begin{align} \label{E:Sub_TGr}
    \mathcal{X} := \{g\nabla_\Gamma q\in H^{-1}(\Gamma,T\Gamma) \mid q\in L^2(\Gamma)\}
  \end{align}
  is closed in $H^{-1}(\Gamma,T\Gamma)$.
\end{lemma}

\begin{proof}
  Let $\{q_k\}_{k=1}^\infty$ be a sequence in $L^2(\Gamma)$ such that $\{g\nabla_\Gamma q_k\}_{k=1}^\infty$ converges to some $f$ strongly in $H^{-1}(\Gamma,T\Gamma)$.
  For $k\in\mathbb{N}$ we subtract the average of $q_k$ over $\Gamma$ from $q_k$ to assume $\int_\Gamma q_k\,d\mathcal{H}^2=0$ without changing $g\nabla_\Gamma q_k$ (see Lemma~\ref{L:TGr_HinT_Con}).
  Then
  \begin{align*}
    \|q_k-q_l\|_{L^2(\Gamma)} \leq c\|g\nabla_\Gamma q_k-g\nabla_\Gamma q_l\|_{H^{-1}(\Gamma,T\Gamma)} \to 0 \quad\text{as}\quad k,l\to\infty
  \end{align*}
  by \eqref{E:Poin_L2_HinT} and \eqref{E:TGr_G_HinT}.
  Hence $\{q_k\}_{k=1}^\infty$ is a Cauchy sequence in $L^2(\Gamma)$ and converges to some $q$ strongly in $L^2(\Gamma)$.
  Moreover, by \eqref{E:TGr_HinT} and \eqref{E:Def_Mul_HinT} we easily get
  \begin{align*}
    \|g\nabla_\Gamma q-g\nabla_\Gamma q_k\|_{H^{-1}(\Gamma,T\Gamma)} \leq c\|q-q_k\|_{L^2(\Gamma)} \to 0 \quad\text{as}\quad k\to\infty.
  \end{align*}
  Since $\{g\nabla_\Gamma q_k\}_{k=1}^\infty$ converges to $f$ strongly in $H^{-1}(\Gamma,T\Gamma)$, we conclude by the above convergence that $f=g\nabla_\Gamma q\in\mathcal{X}$.
  Therefore, $\mathcal{X}$ is closed in $H^{-1}(\Gamma,T\Gamma)$.
\end{proof}

To prove Theorem~\ref{T:DeRham_T} we use basic results of functional analysis.
Let $\mathcal{B}$ be a Banach space.
For a subset $\mathcal{X}$ of $\mathcal{B}$ we define the annihilator of $\mathcal{X}$ by
\begin{align*}
    \mathcal{X}^\perp:=\{f\in\mathcal{B}' \mid \text{${}_{\mathcal{B}'}\langle f,v\rangle_{\mathcal{B}}=0$ for all $v\in\mathcal{X}$}\},
\end{align*}
where $\mathcal{B}'$ is the dual of $\mathcal{B}$ and ${}_{\mathcal{B}'}\langle\cdot,\cdot\rangle_{\mathcal{B}}$ is the duality product between $\mathcal{B}'$ and $\mathcal{B}$.

\begin{lemma} \label{L:Re_FA_1}
  For subsets $\mathcal{X}$ and $\mathcal{Y}$ of $\mathcal{B}$, if $\mathcal{X}\subset\mathcal{Y}$ in $\mathcal{B}$ then $\mathcal{Y}^\perp\subset\mathcal{X}^\perp$ in $\mathcal{B}'$.
\end{lemma}

\begin{lemma} \label{L:Re_FA_2}
  If $\mathcal{B}$ is reflexive and $\mathcal{X}$ is a closed subspace of $\mathcal{B}$, then $(\mathcal{X}^\perp)^\perp=\mathcal{X}$.
\end{lemma}

Lemma~\ref{L:Re_FA_1} is an immediate consequence of the definition of the annihilator.
Also, Lemma~\ref{L:Re_FA_2} follows from the Hahn--Banach theorem, see e.g.~\cite[Theorem~4.7]{Ru91}.

\begin{proof}[Proof of Theorem~\ref{T:DeRham_T}]
  Since $H^{-1}(\Gamma,T\Gamma)$ is the dual of the Hilbert space $H^1(\Gamma,T\Gamma)$, it is also a Hilbert space and thus reflexive.
  Let $\mathcal{X}$ be the subspace of $H^{-1}(\Gamma,T\Gamma)$ given by \eqref{E:Sub_TGr} and $v\in\mathcal{X}^\perp\subset H^1(\Gamma,T\Gamma)$.
  Then
  \begin{align*}
    0 = [g\nabla_\Gamma q,v]_{T\Gamma} = -\bigl(q,\mathrm{div}_\Gamma(gv)\bigr)_{L^2(\Gamma)}
  \end{align*}
  for all $q\in L^2(\Gamma)$ by $g\nabla_\Gamma q\in\mathcal{X}$, \eqref{E:TGr_HinT}, and \eqref{E:Def_Mul_HinT}.
  This implies $\mathrm{div}_\Gamma(gv)=0$ on $\Gamma$, i.e. $v\in H_{g\sigma}^1(\Gamma,T\Gamma)$.
  Hence $\mathcal{X}^\perp\subset H_{g\sigma}^1(\Gamma,T\Gamma)$ in $H^1(\Gamma,T\Gamma)$ and, by Lemma~\ref{L:Re_FA_1},
  \begin{align*}
    H_{g\sigma}^1(\Gamma,T\Gamma)^\perp = \{f\in H^{-1}(\Gamma,T\Gamma) \mid \text{$[f,v]_{T\Gamma}=0$ for all $v\in H_{g\sigma}^1(\Gamma,T\Gamma)$}\} \subset (\mathcal{X}^\perp)^\perp.
  \end{align*}
  Since $\mathcal{X}$ is closed in $H^{-1}(\Gamma,T\Gamma)$ by Lemma~\ref{L:TGr_Closed}, we have $(\mathcal{X}^\perp)^\perp=\mathcal{X}$ by Lemma~\ref{L:Re_FA_2}.
  Hence $H_{g\sigma}^1(\Gamma,T\Gamma)^\perp\subset\mathcal{X}$, i.e. for $f\in H_{g\sigma}^1(\Gamma,T\Gamma)^\perp$ there exists $q\in L^2(\Gamma)$ such that $f=g\nabla_\Gamma q$ in $H^{-1}(\Gamma,T\Gamma)$.
  Moreover, subtracting the average of $q$ over $\Gamma$ from $q$ we may assume $\int_\Gamma q\,d\mathcal{H}^2=0$ without changing $g\nabla_\Gamma q$ (see Lemma~\ref{L:TGr_HinT_Con}).
  Therefore, the existence part of the theorem is valid.
  To prove the uniqueness, suppose that $q_1,q_2\in L^2(\Gamma)$ satisfy
  \begin{align*}
    g\nabla_\Gamma q_1 = g\nabla_\Gamma q_2 \quad\text{in}\quad H^{-1}(\Gamma,T\Gamma), \quad \int_\Gamma q_1\,d\mathcal{H}^2 = \int_\Gamma q_2\,d\mathcal{H}^2 = 0.
  \end{align*}
  Then $\nabla_\Gamma(q_1-q_2)=0$ in $H^{-1}(\Gamma,T\Gamma)$ by \eqref{E:TGr_G_HinT} and thus $q_1-q_2$ is constant on $\Gamma$ by Lemma~\ref{L:TGr_HinT_Con}.
  Since $\int_\Gamma(q_1-q_2)\,d\mathcal{H}^2=0$, the constant $q_1-q_2$ is equal to zero, i.e. $q_1=q_2$ on $\Gamma$.
  Hence the uniqueness is also valid.
  Finally, the estimate \eqref{E:DeRham_T_Ineq} follows from \eqref{E:Poin_L2_HinT} with $\int_\Gamma q\,d\mathcal{H}^2=0$ and \eqref{E:TGr_G_HinT}.
\end{proof}

\subsection{Weighted Helmholtz--Leray decomposition of tangential vector fields} \label{SS:WS_HL}
The aim of this subsection is to give the weighted Helmholtz--Leray decomposition in $L^2(\Gamma,T\Gamma)$ and show estimates for its gradient part.
For $v\in L^2(\Gamma)^3$ we consider $\mathrm{div}_\Gamma(gv)$ in $H^{-1}(\Gamma)$ by \eqref{E:Sdiv_Hin} and define a subspace of $L^2(\Gamma,T\Gamma)$ by
\begin{align*}
  L_{g\sigma}^2(\Gamma,T\Gamma) := \{v\in L^2(\Gamma,T\Gamma) \mid \text{$\mathrm{div}_\Gamma(gv)=0$ in $H^{-1}(\Gamma)$}\}.
\end{align*}
By \eqref{E:Sdiv_Hin} and $g\in C^1(\Gamma)$ we easily deduce that
\begin{align*}
  \|\mathrm{div}_\Gamma(gv)\|_{H^{-1}(\Gamma)} \leq c\|v\|_{L^2(\Gamma)} \quad\text{for all}\quad v\in L^2(\Gamma)^3.
\end{align*}
Hence $L_{g\sigma}^2(\Gamma,T\Gamma)$ is closed in $L^2(\Gamma,T\Gamma)$, which is a closed subspace of $L^2(\Gamma)^3$.
Let $L_{g\sigma}^2(\Gamma,T\Gamma)^\perp$ be the orthogonal complement of $L_{g\sigma}^2(\Gamma,T\Gamma)$ in $L^2(\Gamma,T\Gamma)$.

\begin{lemma} \label{L:L2gs_Orth}
  The orthogonal complement of $L_{g\sigma}^2(\Gamma,T\Gamma)$ is of the form
  \begin{align*}
    L_{g\sigma}^2(\Gamma,T\Gamma)^\perp = \{g\nabla_\Gamma q \in L^2(\Gamma,T\Gamma) \mid q\in H^1(\Gamma)\}.
  \end{align*}
\end{lemma}

\begin{proof}
  Let $\mathcal{X}:=\{g\nabla_\Gamma q \in L^2(\Gamma,T\Gamma) \mid q\in H^1(\Gamma)\}$.
  By \eqref{E:Sdiv_Hin} we have
  \begin{align*}
    (v,g\nabla_\Gamma q)_{L^2(\Gamma)} = -\langle\mathrm{div}_\Gamma(gv),q\rangle_\Gamma = 0 \quad\text{for all}\quad v\in L_{g\sigma}^2(\Gamma,T\Gamma),\, q\in H^1(\Gamma).
  \end{align*}
  Thus $\mathcal{X}\subset L_{g\sigma}^2(\Gamma,T\Gamma)^\perp$.
  Conversely, let $f\in L_{g\sigma}^2(\Gamma,T\Gamma)^\perp$.
  Then, as an element of $H^{-1}(\Gamma,T\Gamma)$ (see Section~\ref{SS:Pre_Surf}), $f=Pf$ satisfies
  \begin{align*}
    [f,v]_{T\Gamma} = (f,v)_{L^2(\Gamma)} = 0 \quad\text{for all}\quad v\in H_{g\sigma}^1(\Gamma,T\Gamma)\subset L_{g\sigma}^2(\Gamma,T\Gamma).
  \end{align*}
  Hence $f=g\nabla_\Gamma q$ in $H^{-1}(\Gamma,T\Gamma)$ with some $q\in L^2(\Gamma)$ by Theorem~\ref{T:DeRham_T}.
  Let us prove $q\in H^1(\Gamma)$.
  For $\eta\in C^1(\Gamma)$ and $i=1,2,3$ let $v:=g^{-1}\eta Pe_i$ on $\Gamma$, where $\{e_1,e_2,e_3\}$ is the standard basis of $\mathbb{R}^3$.
  Then $v\in H^1(\Gamma,T\Gamma)$ since $P$ and $g$ are of class $C^1$ and $g \geq c>0$ on $\Gamma$.
  Moreover, by $P^T=P$ on $\Gamma$, \eqref{E:P_TGr}, and \eqref{E:Form_W},
  \begin{align*}
    \mathrm{div}_\Gamma(gv) = \nabla_\Gamma\eta\cdot Pe_i+\eta(\mathrm{div}_\Gamma P\cdot e_i) = \underline{D}_i\eta+\eta Hn_i \quad\text{on}\quad \Gamma.
  \end{align*}
  By this equality, $f=g\nabla_\Gamma q$ in $H^{-1}(\Gamma,T\Gamma)$, and $P^Tf=Pf=f$ on $\Gamma$ we obtain
  \begin{align*}
    -(q,\underline{D}_i\eta+\eta Hn_i)_{L^2(\Gamma)} &= -\bigl(q,\mathrm{div}_\Gamma(gv)\bigr)_{L^2(\Gamma)} = [g\nabla_\Gamma q,v]_{T\Gamma} \\
    &= [f,v]_{T\Gamma} = (f,v)_{L^2(\Gamma)} = (g^{-1}f_i,\eta)_{L^2(\Gamma)}
  \end{align*}
  for all $\eta\in C^1(\Gamma)$, where $f_i$ is the $i$-th component of $f$.
  This implies $\underline{D}_iq=g^{-1}f_i\in L^2(\Gamma)$ for $i=1,2,3$ by the definition of the weak tangential derivative in $L^2(\Gamma)$ (see \eqref{E:Def_WTD}).
  Hence $q\in H^1(\Gamma)$ and $f=g\nabla_\Gamma q\in\mathcal{X}$, which shows $L_{g\sigma}^2(\Gamma,T\Gamma)^\perp\subset\mathcal{X}$.
\end{proof}

For $q_1,q_2\in H^1(\Gamma)$ we have $g\nabla_\Gamma q_1=g\nabla_\Gamma q_2$ on $\Gamma$ if and only if $q_1-q_2$ is constant on $\Gamma$ by \eqref{E:Width_Bound} and Lemma~\ref{L:TGr_HinT_Con}.
By this fact and Lemma~\ref{L:L2gs_Orth} we obtain the weighted Helmholtz--Leray decomposition of tangential vector fields on $\Gamma$.

\begin{theorem} \label{T:HL_L2T}
  For each $v\in L^2(\Gamma,T\Gamma)$ we have the orthogonal decomposition
  \begin{align*}
    v = v_g+g\nabla_\Gamma q \quad\text{in}\quad L^2(\Gamma,T\Gamma), \quad v_g\in L_{g\sigma}^2(\Gamma,T\Gamma), \, g\nabla_\Gamma q \in L_{g\sigma}^2(\Gamma,T\Gamma)^\perp.
  \end{align*}
  Here $q\in H^1(\Gamma)$ is uniquely determined up to a constant.
\end{theorem}

Note that here we derived the weighted Helmholtz--Leray decomposition without introducing the notion of differential forms.
When $g\equiv1$ it was also proved in the recent work~\cite{Re18} without calculus of differential forms, where the solenoidal part is further decomposed into the curl of some function and a harmonic field whose surface divergence and curl vanish.

Next we consider approximation of surface solenoidal vector fields.
In general, for $v\in L_{g\sigma}^2(\Gamma,T\Gamma)$ with $g\equiv1$ and $\eta\in C^1(\Gamma)$, $\mathrm{div}_\Gamma(\eta v)=\nabla_\Gamma\eta\cdot v$ does not vanish in $H^{-1}(\Gamma)$.
Hence standard localization and mollification arguments with a partition of unity on $\Gamma$ do not work on approximation of surface solenoidal vector fields by smooth ones.
Instead, we use a solution to Poisson's equation on $\Gamma$.

\begin{lemma} \label{L:Pois_Surf}
  For each $\eta\in H^{-1}(\Gamma)$ satisfying $\langle\eta,1\rangle_\Gamma=0$ there exists a unique weak solution $q\in H^1(\Gamma)$ to Poisson's equation
  \begin{align} \label{E:Pois_Surf}
    -\Delta_\Gamma q = \eta \quad\text{on}\quad \Gamma, \quad \int_\Gamma q\,d\mathcal{H}^2 = 0
  \end{align}
  in the sense that
  \begin{align} \label{E:Weak_Pois}
    (\nabla_\Gamma q,\nabla_\Gamma\xi)_{L^2(\Gamma)} = \langle\eta,\xi\rangle_\Gamma \quad\text{for all}\quad \xi\in H^1(\Gamma).
  \end{align}
  Moreover, there exists a constant $c>0$ such that
  \begin{align} \label{E:H1_Pois}
    \|q\|_{H^1(\Gamma)} \leq c\|\eta\|_{H^{-1}(\Gamma)}.
  \end{align}
  If in addition $\eta\in L^2(\Gamma)$, then $q\in H^2(\Gamma)$ and
  \begin{align} \label{E:H2_Pois}
    \|q\|_{H^2(\Gamma)} \leq c\|\eta\|_{L^2(\Gamma)}.
  \end{align}
\end{lemma}

The existence and uniqueness of a weak solution to \eqref{E:Pois_Surf} and the estimate \eqref{E:H1_Pois} follow from Poincar\'{e}'s inequality \eqref{E:Poin_Surf_Lp} and the Lax--Milgram theorem.
Also, the $H^2$-regularity of a weak solution and \eqref{E:H2_Pois} are proved by a localization argument and the elliptic regularity theorem.
For details, see \cite[Theorems~3.1 and~3.3]{DzEl13}.

\begin{lemma} \label{L:H1gs_Dense}
  The space $H_{g\sigma}^1(\Gamma,T\Gamma)$ is dense in $L_{g\sigma}^2(\Gamma,T\Gamma)$.
\end{lemma}

\begin{proof}
  Let $v\in L_{g\sigma}^2(\Gamma,T\Gamma)$.
  Since $\Gamma$ is of class $C^2$, we can take a sequence $\{\tilde{v}_k\}_{k=1}^\infty$ in $C^1(\Gamma,T\Gamma)$ that converges to $v$ strongly in $L^2(\Gamma,T\Gamma)$ by Lemma~\ref{L:Wmp_Tan_Appr}.
  For $k\in\mathbb{N}$ it follows from $\mathrm{div}_\Gamma(gv)=0$ in $H^{-1}(\Gamma)$ and \eqref{E:Sdiv_Hin} that
  \begin{align} \label{Pf_HgDe:Hin_Conv}
    \|\mathrm{div}_\Gamma(g\tilde{v}_k)\|_{H^{-1}(\Gamma)} = \|\mathrm{div}_\Gamma[g(\tilde{v}_k-v)]\|_{H^{-1}(\Gamma)} \leq c\|\tilde{v}_k-v\|_{L^2(\Gamma)}.
  \end{align}
  Let $\eta_k:=-\mathrm{div}_\Gamma(g\tilde{v}_k)$.
  Since $\eta_k\in L^2(\Gamma)$, there exists a unique solution $q_k\in H^2(\Gamma)$ to \eqref{E:Pois_Surf} with source term $\eta_k$ by Lemma~\ref{L:Pois_Surf}.
  Moreover, by \eqref{E:H1_Pois} and \eqref{Pf_HgDe:Hin_Conv},
  \begin{align*}
    \|q_k\|_{H^1(\Gamma)} \leq c\|\eta_k\|_{H^{-1}(\Gamma)} = c\|\mathrm{div}_\Gamma(g\tilde{v}_k)\|_{H^{-1}(\Gamma)} \leq c\|v-\tilde{v}_k\|_{L^2(\Gamma)}.
  \end{align*}
  Hence $v_k:=\tilde{v}_k-g^{-1}\nabla_\Gamma q_k\in H_{g\sigma}^1(\Gamma,T\Gamma)$ and
  \begin{align*}
    \|v-v_k\|_{L^2(\Gamma)} \leq \|v-\tilde{v}_k\|_{L^2(\Gamma)}+c\|q_k\|_{H^1(\Gamma)} \leq c\|v-\tilde{v}_k\|_{L^2(\Gamma)} \to 0 \quad\text{as}\quad k\to\infty
  \end{align*}
  by $g\geq c>0$ on $\Gamma$ and the strong convergence of $\{\tilde{v}_k\}_{k=1}^\infty$ to $v$ in $L^2(\Gamma,T\Gamma)$.
\end{proof}

Let $\mathbb{P}_g$ be the orthogonal projection from $L^2(\Gamma,T\Gamma)$ onto $L_{g\sigma}^2(\Gamma,T\Gamma)$.
We call it the weighted Helmholtz--Leray projection.
In the study of a singular limit problem for \eqref{E:NS_Eq}--\eqref{E:NS_In} we need to estimate the difference $v-\mathbb{P}_gv$ for $v\in L^2(\Gamma,T\Gamma)$.

\begin{lemma} \label{L:HLT_Est_L2}
  There exists a constant $c>0$ such that
  \begin{align} \label{E:HLT_Est_L2}
    \|v-\mathbb{P}_gv\|_{L^2(\Gamma)} \leq c\|\mathrm{div}_\Gamma(gv)\|_{H^{-1}(\Gamma)}
  \end{align}
  for all $v\in L^2(\Gamma,T\Gamma)$.
  If in addition $v\in H^1(\Gamma,T\Gamma)$, then $\mathbb{P}_gv\in H_{g\sigma}^1(\Gamma,T\Gamma)$ and
  \begin{align} \label{E:HLT_Est_H1}
    \|v-\mathbb{P}_gv\|_{H^1(\Gamma)} \leq c\|\mathrm{div}_\Gamma(gv)\|_{L^2(\Gamma)}.
  \end{align}
\end{lemma}

\begin{proof}
  Let $v\in L^2(\Gamma,T\Gamma)$ and $\eta:=-\mathrm{div}_\Gamma(gv)\in H^{-1}(\Gamma)$.
  Since
  \begin{align*}
    \langle\eta,1\rangle_\Gamma = (gv,Hn)_{L^2(\Gamma)} = \int_\Gamma g(v\cdot n)H\,d\mathcal{H}^2 = 0
  \end{align*}
  by \eqref{E:Sdiv_Hin} and $v\cdot n=0$ on $\Gamma$, there exists a unique weak solution $q\in H^1(\Gamma)$ to \eqref{E:Pois_Surf} with $\eta$ by Lemma~\ref{L:Pois_Surf}.
  Then $\mathbb{P}_gv=v-g^{-1}\nabla_\Gamma q\in L_{g\sigma}^2(\Gamma,T\Gamma)$ by the uniqueness of the weighted Helmholtz--Leray decomposition.
  Moreover, since
  \begin{align*}
    \|q\|_{H^1(\Gamma)} \leq c\|\eta\|_{H^{-1}(\Gamma)} = c\|\mathrm{div}_\Gamma(gv)\|_{H^{-1}(\Gamma)}
  \end{align*}
  by \eqref{E:H1_Pois} and $g\in C^1(\Gamma)$ is bounded from below by a positive constant,
  \begin{align*}
    \|v-\mathbb{P}_gv\|_{L^2(\Gamma)} = \|g^{-1}\nabla_\Gamma q\|_{L^2(\Gamma)} \leq c\|q\|_{H^1(\Gamma)} \leq c\|\mathrm{div}_\Gamma(gv)\|_{H^{-1}(\Gamma)}.
  \end{align*}
  Hence the inequality \eqref{E:HLT_Est_L2} holds.

  If $v\in H^1(\Gamma,T\Gamma)$, then $\eta=-\mathrm{div}_\Gamma(gv)\in L^2(\Gamma)$ and Lemma~\ref{L:Pois_Surf} implies that $q\in H^2(\Gamma)$ and $\mathbb{P}_gv=v-g^{-1}\nabla_\Gamma q\in H_{g\sigma}^1(\Gamma,T\Gamma)$.
  Also, since
  \begin{align*}
    \|q\|_{H^2(\Gamma)} \leq c\|\eta\|_{L^2(\Gamma)} =c\|\mathrm{div}_\Gamma(gv)\|_{L^2(\Gamma)}
  \end{align*}
  by \eqref{E:H2_Pois} and $g\in C^1(\Gamma)$ is bounded from below by a positive constant,
  \begin{align*}
    \|v-\mathbb{P}_gv\|_{H^1(\Gamma)} = \|g^{-1}\nabla_\Gamma q\|_{H^1(\Gamma)} \leq c\|q\|_{H^2(\Gamma)} \leq c\|\mathrm{div}_\Gamma(gv)\|_{L^2(\Gamma)},
  \end{align*}
  i.e. the inequality \eqref{E:HLT_Est_H1} is valid.
\end{proof}

\begin{lemma} \label{L:HLT_Bound}
  There exists a constant $c>0$ such that
  \begin{align} \label{E:HLT_Bound}
    \|\mathbb{P}_gv\|_{H^k(\Gamma)} \leq c\|v\|_{H^k(\Gamma)}
  \end{align}
  for all $v\in H^k(\Gamma,T\Gamma)$, $k=0,1$ (note that $H^0=L^2$).
\end{lemma}

\begin{proof}
  If $k=0$, then \eqref{E:HLT_Bound} holds with $c=1$ since $\mathbb{P}_g$ is the orthogonal projection from $L^2(\Gamma,T\Gamma)$ onto $L_{g\sigma}^2(\Gamma,T\Gamma)$.
  Moreover, the inequality \eqref{E:HLT_Bound} for $k=1$ follows from \eqref{E:HLT_Est_H1} and $\|\mathrm{div}_\Gamma(gv)\|_{L^2(\Gamma)}\leq c\|v\|_{H^1(\Gamma)}$.
\end{proof}

Next we consider the time derivative of $v-\mathbb{P}_gv$.
We derive an estimate for the time derivative of a weak solution to Poisson's equation \eqref{E:Pois_Surf}.

\begin{lemma} \label{L:Pois_Dt}
  Let $T>0$.
  Suppose that $\eta\in H^1(0,T;H^{-1}(\Gamma))$ satisfies
  \begin{align*}
    \langle\eta(t),1\rangle_\Gamma = 0 \quad\text{for all}\quad t\in[0,T].
  \end{align*}
  For each $t\in[0,T]$ let $q(t)\in H^1(\Gamma)$ be a unique weak solution to \eqref{E:Pois_Surf} with source term $\eta(t)$.
  Then $q\in H^1(0,T;H^1(\Gamma))$ and there exists a constant $c>0$ such that
  \begin{align} \label{E:Pois_Dt}
    \|\partial_tq\|_{L^2(0,T;H^1(\Gamma))} \leq c\|\partial_t\eta\|_{L^2(0,T;H^{-1}(\Gamma))}.
  \end{align}
  Moreover, for a.a. $t\in(0,T)$ the time derivative $\partial_tq(t)\in H^1(\Gamma)$ is a unique weak solution to \eqref{E:Pois_Surf} with source term $\partial_t\eta(t)$.
\end{lemma}

Note that, when $\eta\in H^1(0,T;H^{-1}(\Gamma))$, $\eta(t)\in H^{-1}(\Gamma)$ is well-defined for each $t\in[0,T]$ since $H^1(0,T;H^{-1}(\Gamma))$ is continuously embedded into $C([0,T];H^{-1}(\Gamma))$.

\begin{proof}
  Since $\eta\in L^2(0,T;H^{-1}(\Gamma))$, we have $q\in L^2(0,T;H^1(\Gamma))$ by \eqref{E:H1_Pois}.
  Let us prove $\partial_tq\in L^2(0,T;H^1(\Gamma))$ by means of the difference quotient.
  Fix $\delta\in(0,T/2)$ and $h\in\mathbb{R}\setminus\{0\}$ with $|h|<\delta/2$.
  For $t\in(\delta,T-\delta)$ we define
  \begin{align*}
    D_hq(t) := \frac{q(t+h)-q(t)}{h} \in H^1(\Gamma), \quad D_h\eta(t) := \frac{\eta(t+h)-\eta(t)}{h} \in H^{-1}(\Gamma).
  \end{align*}
  Note that these definitions make sense since $t+h\in(\delta/2,T-\delta/2)$.
  Moreover,
  \begin{align*}
    \int_\Gamma D_hq(t)\,d\mathcal{H}^2 = 0, \quad \langle D_h\eta(t),1\rangle_\Gamma = 0, \quad t\in(\delta,T-\delta)
  \end{align*}
  since $q(t)$ and $\eta(t)$ satisfy the same equalities for all $t\in[0,T]$.
  For each $\xi\in H^1(\Gamma)$ we subtract \eqref{E:Weak_Pois} for $q(t)$ from that for $q(t+h)$ and divide both sides by $h$ to get
  \begin{align} \label{Pf_PoDt:Weak}
    (\nabla_\Gamma D_hq(t),\nabla_\Gamma\xi)_{L^2(\Gamma)} = \langle D_h\eta(t),\xi\rangle_\Gamma.
  \end{align}
  Hence $D_hq(t)$ is a unique weak solution to \eqref{E:Pois_Surf} with $D_h\eta(t)$.
  By \eqref{E:H1_Pois} we have
  \begin{align*}
    \|D_hq(t)\|_{H^1(\Gamma)} \leq c\|D_h\eta(t)\|_{H^{-1}(\Gamma)}, \quad t\in(\delta,T-\delta),
  \end{align*}
  where $c>0$ is a constant independent of $t$, $\delta$, and $h$, and thus
  \begin{align*}
    \|D_hq\|_{L^2(\delta,T-\delta;H^1(\Gamma))} \leq c\|D_h\eta\|_{L^2(\delta,T-\delta;H^{-1}(\Gamma))}.
  \end{align*}
  Moreover, since $\eta\in H^1(0,T;H^{-1}(\Gamma))$,
  \begin{align*}
    \|D_h\eta\|_{L^2(\delta,T-\delta;H^{-1}(\Gamma))} \leq c\|\partial_t\eta\|_{L^2(0,T;H^{-1}(\Gamma))}
  \end{align*}
  with a constant $c>0$ independent of $h$ and $\delta$ (see~\cite[Section~5.8, Theorem~3 (i)]{Ev10}).
  Combining the above two estimates we obtain
  \begin{align*}
    \|D_hq\|_{L^2(\delta,T-\delta;H^1(\Gamma))} \leq c\|\partial_t\eta\|_{L^2(0,T;H^{-1}(\Gamma))}
  \end{align*}
  for all $h\in\mathbb{R}\setminus\{0\}$ with $|h|<\delta/2$.
  Since the right-hand side of this inequality is independent of $h$, it follows that $\partial_tq\in L^2(\delta,T-\delta;H^1(\Gamma))$ and
  \begin{align*}
    \|\partial_tq\|_{L^2(\delta,T-\delta;H^1(\Gamma))} \leq c\|\partial_t\eta\|_{L^2(0,T;H^{-1}(\Gamma))}
  \end{align*}
  for all $\delta\in(0,T/2)$ (see \cite[Section~5.8, Theorem~3 (ii)]{Ev10}).
  In particular, we have $\partial_tq(t)\in H^1(\Gamma)$ for a.a. $t\in(0,T)$.
  Since the right-hand side of the above inequality is independent of $\delta$, the monotone convergence theorem yields
  \begin{align*}
    \|\partial_tq\|_{L^2(0,T;H^1(\Gamma))} = \lim_{\delta\to0}\|\partial_tq\|_{L^2(\delta,T-\delta;H^1(\Gamma))} \leq c\|\partial_t\eta\|_{L^2(0,T;H^{-1}(\Gamma))}.
  \end{align*}
  Hence $\partial_tq\in L^2(0,T;H^1(\Gamma))$ and the inequality \eqref{E:Pois_Dt} is valid.

  Next we show that $\partial_tq(t)$ is a unique weak solution to \eqref{E:Pois_Surf} with source term $\partial_t\eta(t)$ for a.a. $t\in(0,T)$.
  Let $\xi\in H^1(\Gamma)$ and $\varphi\in C_c^\infty(0,T)$.
  Suppose that $\varphi$ is supported in $(\delta,T-\delta)$ with $\delta\in(0,T/2)$.
  We extend $\varphi$ to $\mathbb{R}$ by setting zero outside of $(0,T)$.
  For $h\in\mathbb{R}\setminus\{0\}$, $|h|<\delta/2$ we multiply both sides of \eqref{Pf_PoDt:Weak} by $\varphi(t)$, integrate them over $(\delta,T-\delta)$, and make the change of a variable
  \begin{align*}
    \int_\delta^{T-\delta}\psi(t+h)\varphi(t)\,dt = \int_{\delta+h}^{T-\delta+h}\psi(s)\varphi(s-h)\,ds
  \end{align*}
  for $\psi(t)=(\nabla_\Gamma q(t),\nabla_\Gamma\xi)_{L^2(\Gamma)}$, $\langle\eta(t),\xi\rangle_\Gamma$ to get
  \begin{align*}
    -\int_0^T(\nabla_\Gamma q(t),\nabla_\Gamma\xi)_{L^2(\Gamma)}D_{-h}\varphi(t)\,dt = -\int_0^T\langle\eta(t),\xi\rangle_\Gamma D_{-h}\varphi(t)\,dt,
  \end{align*}
  where $D_{-h}\varphi(t):=\{\varphi(t-h)-\varphi(t)\}/(-h)$ (note that $\varphi$ is supported in $(\delta,T-\delta)$).
  Let $h\to0$ in this equality.
  Then since $D_{-h}\varphi$ converges to $\partial_t\varphi$ uniformly on $(0,T)$,
  \begin{align*}
    -\int_0^T(\nabla_\Gamma q(t),\nabla_\Gamma\xi)_{L^2(\Gamma)}\partial_t\varphi(t)\,dt = -\int_0^T\langle\eta(t),\xi\rangle_\Gamma\partial_t\varphi(t)\,dt
  \end{align*}
  for all $\varphi \in C_c^\infty(0,T)$.
  Hence for all $\xi\in H^1(\Gamma)$ and a.a. $t\in(0,T)$ we have
  \begin{align*}
    ([\nabla_\Gamma(\partial_tq)](t),\nabla_\Gamma\xi)_{L^2(\Gamma)} = \langle\partial_t\eta(t),\xi\rangle_\Gamma
  \end{align*}
  by $q\in H^1(0,T;H^1(\Gamma))$ and $\eta\in H^1(0,T;H^{-1}(\Gamma))$ (note that $\partial_t(\nabla_\Gamma q)=\nabla_\Gamma(\partial_tq)$ a.e. on $\Gamma\times(0,T)$ by the space-time regularity of $q$).
  In the same way we can show
  \begin{align*}
    \int_\Gamma\partial_tq(t)\,d\mathcal{H}^2 = 0 \quad\text{for a.a.}\quad t\in(0,T)
  \end{align*}
  since $q(t)$ satisfies the same equality for all $t\in[0,T]$.
  Hence $\partial_tq(t)$ is a unique weak solution to \eqref{E:Pois_Surf} with source term $\partial_t\eta(t)$ for a.a. $t\in(0,T)$.
\end{proof}

Based on Lemma~\ref{L:Pois_Dt} we give an estimate for the time derivative of $v-\mathbb{P}_gv$.

\begin{lemma} \label{L:HLT_Est_Dt}
  Let $v\in H^1(0,T;L^2(\Gamma,T\Gamma))$, $T>0$.
  Then
  \begin{align*}
    \mathbb{P}_gv\in H^1(0,T;L_{g\sigma}^2(\Gamma,T\Gamma))
  \end{align*}
  and there exists a constant $c>0$ such that
  \begin{align} \label{E:HLT_Est_Dt}
    \|\partial_tv-\partial_t\mathbb{P}_gv\|_{L^2(0,T;L^2(\Gamma))} \leq c\|\mathrm{div}_\Gamma(g\partial_t v)\|_{L^2(0,T;H^{-1}(\Gamma))}.
  \end{align}
\end{lemma}

\begin{proof}
  Let $\eta:=-\mathrm{div}_\Gamma(gv)$.
  Since $v\in H^1(0,T;L^2(\Gamma,T\Gamma))$, we have
  \begin{align*}
    \eta \in H^1(0,T;H^{-1}(\Gamma)), \quad \partial_t\eta = -\mathrm{div}_\Gamma(g\partial_tv)\in L^2(0,T;H^{-1}(\Gamma)).
  \end{align*}
  Here the second relation is due to the fact that $g$ and $P$ are independent of time (note that $P$ appears in the definition of the tangential derivatives).
  For each $t\in[0,T]$ let $q(t)\in H^1(\Gamma)$ be a unique weak solution to \eqref{E:Pois_Surf} with $\eta(t)=-\mathrm{div}_\Gamma[gv(t)]$, which satisfies $\langle\eta(t),1\rangle_\Gamma=0$ as in the proof of Lemma~\ref{L:HLT_Est_L2}.
  Then Lemma~\ref{L:Pois_Dt} implies that $q\in H^1(0,T;H^1(\Gamma))$ and
  \begin{align*}
    \|\partial_tq\|_{L^2(0,T;H^1(\Gamma))} \leq c\|\partial_t\eta\|_{L^2(0,T;H^{-1}(\Gamma))} = c\|\mathrm{div}_\Gamma(g\partial_tv)\|_{L^2(0,T;H^{-1}(\Gamma))}.
  \end{align*}
  Moreover, for a.a. $t\in(0,T)$ the time derivative $\partial_tq(t)\in H^1(\Gamma)$ is a unique weak solution to \eqref{E:Pois_Surf} with source term $\partial_t\eta(t)=-\mathrm{div}_\Gamma[g\partial_tv(t)]$.
  By these facts and
  \begin{align*}
    \mathbb{P}_gv = v-g^{-1}\nabla_\Gamma q, \quad \partial_t\mathbb{P}_gv = \partial_tv-g^{-1}\nabla_\Gamma(\partial_tq) \quad\text{a.e. on}\quad \Gamma\times(0,T)
  \end{align*}
  we observe that $\mathbb{P}_gv\in H^1(0,T;L_{g\sigma}^2(\Gamma,T\Gamma))$ and
  \begin{align*}
    \|\partial_tv-\partial_t\mathbb{P}_gv\|_{L^2(0,T;L^2(\Gamma))} &= \|g^{-1}\nabla_\Gamma(\partial_tq)\|_{L^2(0,T;L^2(\Gamma))} \\
    &\leq c\|\partial_tq\|_{L^2(0,T;H^1(\Gamma))} \\
    &\leq c\|\mathrm{div}_\Gamma(g\partial_tv)\|_{L^2(0,T;H^{-1}(\Gamma))},
  \end{align*}
  where we also used $g\geq c>0$ on $\Gamma$.
  Thus the claim is valid.
\end{proof}

\subsection{Solenoidal spaces of general vector fields} \label{SS:WS_NTS}
In this subsection we briefly investigate solenoidal spaces of general (not necessarily tangential) vector fields on $\Gamma$.
Although the results of this subsection are not used in the sequel, we believe that they are useful for the future study of surface fluid equations including fluid equations on an evolving surface (see e.g.~\cite{JaOlRe17,KoLiGi17,KoLiGi18Er,Mi18}).
For the sake of simplicity, we only consider the case $g\equiv1$ and give a remark on the case of general $g$ at the end of this subsection.

Let $q\in L^2(\Gamma)$.
By \eqref{E:L2_Hin} and \eqref{E:TGr_Hin} we have
\begin{align} \label{E:TGrHn_Hin}
  \langle \nabla_\Gamma q+qHn,v\rangle_\Gamma = -(q,\mathrm{div}_\Gamma v)_{L^2(\Gamma)}
\end{align}
for all $v\in H^1(\Gamma)^3$.
Hence $\langle \nabla_\Gamma q+qHn,v\rangle_\Gamma=0$ for all $v$ in the solenoidal space
\begin{align*}
  H_\sigma^1(\Gamma) := \{v\in H^1(\Gamma)^3 \mid \text{$\mathrm{div}_\Gamma v = 0$ on $\Gamma$}\}.
\end{align*}
Our goal is to prove that each element of the annihilator of $H_\sigma^1(\Gamma)$ is of the form $\nabla_\Gamma q+qHn$.
To this end, we give two properties of a functional of this form.

\begin{lemma} \label{L:TGrHn_Hin_Con}
  Let $q\in L^2(\Gamma)$.
  Then
  \begin{align*}
    \nabla_\Gamma q+qHn = 0 \quad\text{in}\quad H^{-1}(\Gamma)^3
  \end{align*}
  if and only if $q=0$ on $\Gamma$.
\end{lemma}

\begin{proof}
  First note that for all $q\in L^2(\Gamma)$ we have
  \begin{align} \label{Pf_TnHC:Ineq}
    \|\nabla_\Gamma q\|_{H^{-1}(\Gamma,T\Gamma)} \leq \|\nabla_\Gamma q+qHn\|_{H^{-1}(\Gamma)}
  \end{align}
  since $[\nabla_\Gamma q,v]_{T\Gamma}=\langle\nabla_\Gamma q+qHn,v\rangle_\Gamma$ for all $v\in H^1(\Gamma,T\Gamma)$ by \eqref{E:TGr_HinT} and \eqref{E:TGrHn_Hin}.

  Suppose that $\nabla_\Gamma q+qHn=0$ in $H^{-1}(\Gamma)^3$.
  Then $\nabla_\Gamma q=0$ in $H^{-1}(\Gamma,T\Gamma)$ by \eqref{Pf_TnHC:Ineq} and thus $q$ is constant on $\Gamma$ by Lemma~\ref{L:TGr_HinT_Con}.
  To prove $q=0$ we set $v:=\xi n$ in \eqref{E:TGrHn_Hin} with $\xi\in H^1(\Gamma)$ (note that $n\in C^1(\Gamma)^3$) and use \eqref{E:P_TGr} and \eqref{E:Def_WHK} to get
  \begin{align*}
    0 = \langle \nabla_\Gamma q+qHn,\xi n\rangle_\Gamma = -\bigl(q,\mathrm{div}_\Gamma(\xi n)\bigr)_{L^2(\Gamma)} = q\int_\Gamma \xi H\,d\mathcal{H}^2.
  \end{align*}
  Since $H\in C(\Gamma)\subset L^2(\Gamma)$ and $H^1(\Gamma)$ is dense in $L^2(\Gamma)$ (see Lemma~\ref{L:Wmp_Appr}), we observe by the above equality and a density argument that
  \begin{align*}
    q\int_\Gamma H^2\,d\mathcal{H}^2 = 0.
  \end{align*}
  Moreover, for the compact surface $\Gamma$ in $\mathbb{R}^3$ it is known (see (16.32) in~\cite{GiTr01}) that
  \begin{align*}
    \frac{1}{4}\int_\Gamma H^2\,d\mathcal{H}^2 \geq 4\pi.
  \end{align*}
  (Note that in our definition $H$ is not divided by the dimension of $\Gamma$.)
  Hence $q=0$.

  Conversely, if $q=0$ on $\Gamma$, then $\nabla_\Gamma q+qHn=0$ in $H^{-1}(\Gamma)^3$ by \eqref{E:TGrHn_Hin}.
\end{proof}

\begin{lemma} \label{L:Poin_L2_HinN}
  There exists a constant $c>0$ such that
  \begin{align} \label{E:Poin_L2_HinN}
    \|q\|_{L^2(\Gamma)} \leq c\|\nabla_\Gamma q+qHn\|_{H^{-1}(\Gamma)}
  \end{align}
  for all $q\in L^2(\Gamma)$.
\end{lemma}

\begin{proof}
  By the Ne\v{c}as inequality \eqref{E:Necas} and \eqref{Pf_TnHC:Ineq} it is sufficient to show that
  \begin{align} \label{Pf_PLHN:Goal}
    \|q\|_{H^{-1}(\Gamma)} \leq c\|\nabla_\Gamma q+qHn\|_{H^{-1}(\Gamma)}
  \end{align}
  for all $q\in L^2(\Gamma)$.
  Assume to the contrary that there exists $q_k\in L^2(\Gamma)$ such that
  \begin{align} \label{Pf_PLHN:Contra}
    \|q_k\|_{H^{-1}(\Gamma)} > k\|\nabla_\Gamma q_k+q_kHn\|_{H^{-1}(\Gamma)}
  \end{align}
  for each $k\in\mathbb{N}$.
  Replacing $q_k$ with $q_k/\|q_k\|_{H^{-1}(\Gamma)}$ we may assume $\|q_k\|_{H^{-1}(\Gamma)}=1$.
  Then $\{q_k\}_{k=1}^\infty$ is bounded in $L^2(\Gamma)$ by \eqref{E:Necas}, \eqref{Pf_TnHC:Ineq}, and \eqref{Pf_PLHN:Contra}.
  By this fact and the compact embedding $L^2(\Gamma)\hookrightarrow H^{-1}(\Gamma)$, (up to a subsequence) $\{q_k\}_{k=1}^\infty$ converges to some $q\in L^2(\Gamma)$ weakly in $L^2(\Gamma)$ and strongly in $H^{-1}(\Gamma)$, and thus
  \begin{align} \label{Pf_PLHN:Q_Hin}
    \|q\|_{H^{-1}(\Gamma)} = \lim_{k\to\infty}\|q_k\|_{H^{-1}(\Gamma)} = 1.
  \end{align}
  Moreover, the weak convergence of $\{q_k\}_{k=1}^\infty$ to $q$ in $L^2(\Gamma)$ and \eqref{E:TGrHn_Hin} imply that
  \begin{align*}
    \lim_{k\to\infty}(\nabla_\Gamma q_k+q_kHn) = \nabla_\Gamma q+qHn \quad\text{weakly in}\quad H^{-1}(\Gamma)^3.
  \end{align*}
  By this fact and \eqref{Pf_PLHN:Contra} with $\|q_k\|_{H^{-1}(\Gamma)}=1$ we have
  \begin{align*}
    \|\nabla_\Gamma q+qHn\|_{H^{-1}(\Gamma)} \leq \liminf_{k\to\infty}\|\nabla_\Gamma q_k+q_kHn\|_{H^{-1}(\Gamma)} =0.
  \end{align*}
  Hence $\nabla_\Gamma q+qHn=0$ in $H^{-1}(\Gamma)$ and $q=0$ on $\Gamma$ by Lemma~\ref{L:TGrHn_Hin_Con}.
  This yields $\|q\|_{H^{-1}(\Gamma)}=0$, which contradicts with \eqref{Pf_PLHN:Q_Hin}.
  Therefore, \eqref{Pf_PLHN:Goal} is valid.
\end{proof}

Now we prove de Rham's theorem for the annihilator of $H_\sigma^1(\Gamma)$.

\begin{theorem} \label{T:DeRham_Ge}
  Suppose that $f\in H^{-1}(\Gamma)^3$ satisfies
  \begin{align*}
    \langle f,v\rangle_\Gamma = 0 \quad\text{for all}\quad  v\in H_\sigma^1(\Gamma).
  \end{align*}
  Then there exists a unique $q\in L^2(\Gamma)$ such that
  \begin{align*}
    f = \nabla_\Gamma q+qHn \quad\text{in}\quad H^{-1}(\Gamma)^3, \quad \|q\|_{L^2(\Gamma)} \leq c\|f\|_{H^{-1}(\Gamma)}
  \end{align*}
  with a constant $c>0$ independent of $f$.
\end{theorem}

\begin{proof}
  Using \eqref{E:Poin_L2_HinN} we can show as in the proof of Lemma~\ref{L:TGr_Closed} that
  \begin{align*}
    \mathcal{X} := \{\nabla_\Gamma q+qHn\in H^{-1}(\Gamma)^3 \mid q\in L^2(\Gamma)\}
  \end{align*}
  is closed in $H^{-1}(\Gamma)^3$.
  Moreover, $\mathcal{X}^\perp\subset H_\sigma^1(\Gamma)$ in $H^1(\Gamma)^3$ by \eqref{E:TGrHn_Hin}.
  Since the dual $H^{-1}(\Gamma)^3$ of the Hilbert space $H^1(\Gamma)^3$ is reflexive, by Lemmas~\ref{L:Re_FA_1} and~\ref{L:Re_FA_2} we get
  \begin{align*}
    H_\sigma^1(\Gamma)^\perp = \{f\in H^{-1}(\Gamma)^3 \mid \text{$\langle f,v\rangle_\Gamma=0$ for all $v\in H_\sigma^1(\Gamma)$}\} \subset (\mathcal{X}^\perp)^\perp = \mathcal{X}
  \end{align*}
  in $H^{-1}(\Gamma)^3$.
  Hence the existence part of the theorem is valid.
  Also, the uniqueness and the estimate immediately follow from Lemma~\ref{L:Poin_L2_HinN}.
\end{proof}

Next we derive the Helmholtz--Leray decomposition in $L^2(\Gamma)^3$.
Let
\begin{align*}
  L_\sigma^2(\Gamma) := \{v\in L^2(\Gamma)^3 \mid \text{$\mathrm{div}_\Gamma v=0$ in $H^{-1}(\Gamma)$}\}.
\end{align*}
It is closed in $L^2(\Gamma)^3$ since $\|\mathrm{div}_\Gamma v\|_{H^{-1}(\Gamma)}\leq c\|v\|_{L^2(\Gamma)}$ for $v\in L^2(\Gamma)^3$ by \eqref{E:Sdiv_Hin}.

\begin{lemma} \label{L:L2sGe_Orth}
  The orthogonal complement of $L_\sigma^2(\Gamma)$ in $L^2(\Gamma)^3$ is of the form
  \begin{align*}
    L_\sigma^2(\Gamma)^\perp = \{\nabla_\Gamma q+qHn\in L^2(\Gamma)^3 \mid q\in H^1(\Gamma)\}.
  \end{align*}
\end{lemma}

\begin{proof}
  The proof is similar to that of Lemma~\ref{L:L2gs_Orth}.
  By \eqref{E:Sdiv_Hin} we immediately get $\nabla_\Gamma q+qHn\in L_\sigma^2(\Gamma)^\perp$ for $q\in H^1(\Gamma)$.
  Conversely, let $f\in L_\sigma^2(\Gamma)^\perp$.
  Since
  \begin{align*}
    \langle f,v\rangle_\Gamma = (f,v)_{L^2(\Gamma)} = 0 \quad\text{for all}\quad v\in H_\sigma^1(\Gamma)\subset L_\sigma^2(\Gamma),
  \end{align*}
  Theorem~\ref{T:DeRham_Ge} implies that $f=\nabla_\Gamma q+qHn$ in $H^{-1}(\Gamma)^3$ with $q\in L^2(\Gamma)$.
  To prove $q\in H^1(\Gamma)$ let $v:=\eta e_i$ on $\Gamma$ for $\eta\in C^1(\Gamma)$ and $i=1,2,3$, where $\{e_1,e_2,e_3\}$ is the standard basis of $\mathbb{R}^3$.
  Then since $v\in H^1(\Gamma)^3$ and $\mathrm{div}_\Gamma v=\underline{D}_i\eta$ on $\Gamma$,
  \begin{align*}
    -(q,\underline{D}_i\eta)_{L^2(\Gamma)} &= -(q,\mathrm{div}_\Gamma v)_{L^2} = \langle\nabla_\Gamma q+qHn,v\rangle_\Gamma \\
    &= \langle f,v\rangle_\Gamma = (f,v)_{L^2(\Gamma)} = (f_i,\eta)_{L^2(\Gamma)},
  \end{align*}
  where $f_i$ is the $i$-th component of $f$.
  From this equality we deduce that
  \begin{align*}
    -(q,\underline{D}_i\eta+\eta Hn_i)_{L^2(\Gamma)} = (f_i-qHn_i,\eta)_{L^2(\Gamma)} \quad\text{for all}\quad \eta\in C^1(\Gamma).
  \end{align*}
  Hence $\underline{D}_iq=f_i-qHn_i\in L^2(\Gamma)$ by the definition of the weak tangential derivative in $L^2(\Gamma)$ (see \eqref{E:Def_WTD}).
  This shows that $q\in H^1(\Gamma)$ and $f=\nabla_\Gamma q+qHn$ in $L^2(\Gamma)^3$.
\end{proof}

The result of Lemma~\ref{L:L2sGe_Orth} was given in~\cite[Lemma~2.7]{KoLiGi17} and~\cite[Theorem~1.1]{KoLiGi18Er}.
Here we presented another proof of it.
By Lemmas~\ref{L:TGrHn_Hin_Con} and~\ref{L:L2sGe_Orth} we obtain the Helmholtz--Leray decomposition in $L^2(\Gamma)^3$ with uniqueness of the scalar potential.

\begin{theorem} \label{T:HL_L2Ge}
  For each $v\in L^2(\Gamma)^3$ we have the orthogonal decomposition
  \begin{align*}
    v = v_\sigma+\nabla_\Gamma q+qHn \quad\text{in}\quad L^2(\Gamma)^3, \quad v_\sigma\in L_\sigma^2(\Gamma),\,\nabla_\Gamma q+qHn\in L_\sigma^2(\Gamma)^\perp.
  \end{align*}
  Here $q\in H^1(\Gamma)$ is uniquely determined.
\end{theorem}

The Helmholtz--Leray decomposition in Theorem~\ref{T:HL_L2Ge} was already stated in~\cite{KoLiGi18Er} without an explicit formulation (see a remark after~\cite[Theorem~1.1]{KoLiGi18Er}).

\begin{remark} \label{R:HL_L2Ge}
  In general, Theorem~\ref{T:HL_L2Ge} applied to a tangential vector field on $\Gamma$ does not imply the tangential Helmholtz--Leray decomposition (with $g\equiv1$) given in Theorem~\ref{T:HL_L2T}.
  To see this, suppose that $\Gamma$ is strictly convex and thus the mean curvature $H$ of $\Gamma$ does not vanish on the whole surface.
  We take $v\in L^2(\Gamma,T\Gamma)$ such that $\mathrm{div}_\Gamma v\neq 0$ in $H^{-1}(\Gamma)$.
  By Theorem~\ref{T:HL_L2Ge} we get the orthogonal decomposition
  \begin{align*}
    v = v_\sigma+\nabla_\Gamma q+qHn \quad\text{in}\quad L^2(\Gamma)^3, \quad v_\sigma \in L_\sigma^2(\Gamma),\,\nabla_\Gamma q+qHn \in L_\sigma^2(\Gamma)^\perp
  \end{align*}
  with a unique $q\in H^1(\Gamma)$.
  Since $v$ is tangential on $\Gamma$,
  \begin{align*}
    0 = v_\sigma\cdot n+qH, \quad\text{i.e.}\quad v_\sigma\cdot n = -qH \quad\text{on}\quad \Gamma.
  \end{align*}
  Moreover, $q\neq 0$ in $H^1(\Gamma)$ by $\mathrm{div}_\Gamma v\neq 0$ in $H^{-1}(\Gamma)$.
  From this property and the fact that $H$ does not vanish on the whole surface $\Gamma$ by its strict convexity, it follows that $v_\sigma\cdot n=-qH\neq0$ in $L^2(\Gamma)$.
  Hence the solenoidal part $v_\sigma$ of $v$ given by Theorem~\ref{T:HL_L2Ge} is not tangential on $\Gamma$, while the solenoidal part $v_g$ (with $g\equiv1$) of the same $v$ given by Theorem~\ref{T:HL_L2T} must be tangential on $\Gamma$.
\end{remark}

The vector field $\nabla_\Gamma q+qHn$ appears in the interface equations of two-phase flows~\cite{BaGaNu15,BoPr10,NiVoWe12} as well as the Navier--Stokes equations on an evolving surface~\cite{JaOlRe17,KoLiGi17}.
By \eqref{E:P_TGr} and \eqref{E:Form_W} we observe that the surface divergence of $qP$ is of this form:
\begin{align*}
  \mathrm{div}_\Gamma(qP) = P\nabla_\Gamma q+q\,\mathrm{div}_\Gamma P = \nabla_\Gamma q+qHn.
\end{align*}
The tensor $qP$ is a part of the Boussinesq--Scriven surface stress tensor~\cite{Ar89,Bo1913,Sc60}
\begin{align*}
  S_\Gamma = \{q+(\lambda_s-\mu_s)\mathrm{div}_\Gamma v\}P+2\mu_s D_\Gamma(v).
\end{align*}
Here $q$ is the surface tension, $\lambda_s$ the surface dilatational viscosity, $\mu_s$ the surface shear viscosity, $v$ the total velocity of surface flow, and $D_\Gamma(v)$ the surface strain rate tensor given by \eqref{E:Strain_Surf}.

Finally, we give a remark on the case of general $g$.
Let $g\in C^1(\Gamma)$ be bounded from below by a positive constant.
We define weighted solenoidal spaces
\begin{align*}
  L_{g\sigma}^2(\Gamma) &:= \{v\in L^2(\Gamma)^3 \mid \text{$\mathrm{div}_\Gamma(gv)=0$ in $H^{-1}(\Gamma)$}\}, \\
    H_{g\sigma}^1(\Gamma) &:= \{v\in H^1(\Gamma)^3 \mid \text{$\mathrm{div}_\Gamma(gv)=0$ on $\Gamma$}\}.
\end{align*}
By \eqref{E:Def_Mul_Hin}, \eqref{E:TGr_Hin}, and \eqref{E:Sdiv_Hin} we have
\begin{alignat*}{2}
  \langle g(\nabla_\Gamma q+qHn),v\rangle_\Gamma &= -\bigl(q,\mathrm{div}_\Gamma(gv)\bigr)_{L^2(\Gamma)}, &\quad &q\in L^2(\Gamma),\,v\in H^1(\Gamma)^3, \\
  \langle \mathrm{div}_\Gamma(gw),\eta\rangle_\Gamma &= -\bigl(w,g(\nabla_\Gamma \eta+\eta Hn)\bigr)_{L^2(\Gamma)}, &\quad &w\in L^2(\Gamma)^3,\,\eta\in H^1(\Gamma).
\end{alignat*}
Using these formulas and applying Theorem~\ref{T:DeRham_Ge} or Lemma~\ref{L:L2sGe_Orth} to $g^{-1}f$ for $f$ in $H_{g\sigma}^1(\Gamma)^\perp$ or $L_{g\sigma}^2(\Gamma)^\perp$, we can show the following weighted version of the main results in this subsection.

\begin{theorem} \label{T:DeRham_Ge_W}
  Suppose that $f\in H^{-1}(\Gamma)^3$ satisfies
  \begin{align*}
    \langle f,v\rangle_\Gamma = 0 \quad\text{for all}\quad v\in H_{g\sigma}^1(\Gamma).
  \end{align*}
  Then there exists a unique $q\in L^2(\Gamma)$ such that
  \begin{align*}
    f = g(\nabla_\Gamma q+qHn) \quad\text{in}\quad H^{-1}(\Gamma)^3, \quad \|q\|_{L^2(\Gamma)} \leq c\|f\|_{H^{-1}(\Gamma)}
  \end{align*}
  with a constant $c>0$ independent of $f$.
\end{theorem}

\begin{theorem} \label{T:HL_L2Ge_W}
  For each $v\in L^2(\Gamma)^3$ we have the orthogonal decomposition
  \begin{align*}
    v = v_g+g(\nabla_\Gamma q+qHn) \quad\text{in}\quad L^2(\Gamma)^3, \quad v_g\in L_{g\sigma}^2(\Gamma),\,g(\nabla_\Gamma q+qHn)\in L_{g\sigma}^2(\Gamma)^\perp.
  \end{align*}
  Here $q\in H^1(\Gamma)$ is uniquely determined.
\end{theorem}

\section{Singular limit problem as the thickness tends to zero} \label{S:SL}
In this section we study a singular limit problem for the Navier--Stokes equations \eqref{E:NS_Eq}--\eqref{E:NS_In} as the curved thin domain $\Omega_\varepsilon$ degenerates into the closed surface $\Gamma$.
Our goal is to derive limit equations on $\Gamma$ and compare them with \eqref{E:NS_Eq}--\eqref{E:NS_In}.

Throughout this section we impose Assumptions~\ref{Assump_1} and~\ref{Assump_2}, fix the constants $\varepsilon_1$ and $\varepsilon_\sigma$ given in Theorem~\ref{T:UE} and Lemma~\ref{L:HP_Dom}, and suppose that $\varepsilon\in(0,\varepsilon'_1)$ with $\varepsilon'_1:=\min\{\varepsilon_1,\varepsilon_\sigma\}$ and the assumptions of Theorem~\ref{T:Est_Ue} are satisfied.
We also denote by $\bar{\eta}=\eta\circ\pi$ the constant extension of a function $\eta$ on $\Gamma$.
Let $\mathcal{H}_\varepsilon$ and $\mathcal{V}_\varepsilon$ be the function spaces defined by \eqref{E:Def_Heps}, $A_\varepsilon$ the Stokes operator on $\mathcal{H}_\varepsilon$ (see Section~\ref{SS:St_Def}), and $M$ and $M_\tau$ the average operators given in Definition~\ref{D:Average}.

\subsection{Weak formulation for the bulk system} \label{SS:SL_NSW}
Our ansatz is a weak formulation for \eqref{E:NS_Eq}--\eqref{E:NS_In} satisfied by a strong solution.
Let
\begin{align*}
  u^\varepsilon \in C([0,\infty);\mathcal{V}_\varepsilon)\cap L_{loc}^2([0,\infty);D(A_\varepsilon))\cap H_{loc}^1([0,\infty);\mathcal{H}_\varepsilon)
\end{align*}
be the global strong solution to \eqref{E:NS_Eq}--\eqref{E:NS_In} given by Theorem~\ref{T:Est_Ue}.
It satisfies
\begin{align} \label{E:NS_Weak}
  \int_0^T\{(\partial_tu^\varepsilon,\varphi)_{L^2(\Omega_\varepsilon)}+a_\varepsilon(u^\varepsilon,\varphi)+b_\varepsilon(u^\varepsilon,u^\varepsilon,\varphi)\}\,dt = \int_0^T(\mathbb{P}_\varepsilon f^\varepsilon,\varphi)_{L^2(\Omega_\varepsilon)}\,dt
\end{align}
for all $T>0$ and $\varphi\in L^2(0,T;\mathcal{V}_\varepsilon)$, and $u^\varepsilon|_{t=0}=u_0^\varepsilon$ in $\mathcal{V}_\varepsilon$.
Here $a_\varepsilon$ is the bilinear form given by \eqref{E:Def_Bili_Dom}, i.e.
\begin{align*}
  a_\varepsilon(u_1,u_2) = 2\nu\bigl(D(u_1),D(u_2)\bigr)_{L^2(\Omega_\varepsilon)}+\gamma_\varepsilon^0(u_1,u_2)_{L^2(\Gamma_\varepsilon^0)}+\gamma_\varepsilon^1(u_1,u_2)_{L^2(\Gamma_\varepsilon^1)}
\end{align*}
for $u_1,u_2\in H^1(\Omega_\varepsilon)^3$ and $b_\varepsilon$ is a trilinear form defined by
\begin{align} \label{E:Def_Tri_Bo}
  b_\varepsilon(u_1,u_2,u_3) := -(u_1\otimes u_2,\nabla u_3)_{L^2(\Omega_\varepsilon)}
\end{align}
for $u_1,u_2,u_3\in H^1(\Omega_\varepsilon)^3$.
Note that $u_1\otimes u_2\in L^2(\Omega_\varepsilon)^{3\times3}$ for $u_1,u_2\in H^1(\Omega_\varepsilon)^3$ by the Sobolev embedding $H^1(\Omega_\varepsilon)\hookrightarrow L^4(\Omega_\varepsilon)$ (see~\cite{AdFo03}).
Also, if $u\in\mathcal{V}_\varepsilon$ then
\begin{align*}
  \int_{\Omega_\varepsilon}(u_1\cdot\nabla)u_2\cdot u_3\,dx = b(u_1,u_2,u_3)
\end{align*}
by integration by parts, $\mathrm{div}\,u_1=0$ in $\Omega_\varepsilon$, and $u_1\cdot n_\varepsilon=0$ on $\Gamma_\varepsilon$.

Our goal is to derive the limit of the weak formulation \eqref{E:NS_Weak} as well as to show the convergence of the average of the strong solution $u^\varepsilon$ as $\varepsilon\to0$.

\subsection{Average of the weak formulation} \label{SS:SL_Ave}
The first step is to derive a weak formulation satisfied by the averaged tangential component $M_\tau u^\varepsilon$ of the strong solution $u^\varepsilon$, in which we take a test function from the weighted solenoidal space
\begin{align*}
  \mathcal{V}_g := H_{g\sigma}^1(\Gamma,T\Gamma) = \{v\in H^1(\Gamma,T\Gamma) \mid \text{$\mathrm{div}_\Gamma(gv)=0$ on $\Gamma$}\}.
\end{align*}
Since the constant extension of a vector field in $\mathcal{V}_g$ is not in $\mathcal{V}_\varepsilon$, we need to construct an appropriate test function in $\mathcal{V}_\varepsilon$ from a weighted solenoidal vector field on $\Gamma$.
To this end, we use the impermeable extension operator $E_\varepsilon$ given by \eqref{E:Def_ExTan} and the Helmholtz--Leray projection $\mathbb{L}_\varepsilon$ onto $L_\sigma^2(\Omega_\varepsilon)$ given in Section~\ref{SS:St_HL}.

\begin{lemma} \label{L:Const_Test}
  For $\eta\in \mathcal{V}_g$ let $\eta_\varepsilon:=\mathbb{L}_\varepsilon E_\varepsilon\eta$.
  Then $\eta_\varepsilon\in H^1(\Omega_\varepsilon)^3\cap L_\sigma^2(\Omega_\varepsilon)$ and
  \begin{gather}
    \|\eta_\varepsilon-\bar{\eta}\|_{L^2(\Omega_\varepsilon)}+\left\|\nabla\eta_\varepsilon-\overline{F(\eta)}\right\|_{L^2(\Omega_\varepsilon)} \leq c\varepsilon^{3/2}\|\eta\|_{H^1(\Gamma)}, \label{E:Test_Dom} \\
    \|\eta_\varepsilon-\bar{\eta}\|_{L^2(\Gamma_\varepsilon)} \leq c\varepsilon\|\eta\|_{H^1(\Gamma)}, \label{E:Test_Bo}
  \end{gather}
  where $c>0$ is a constant independent of $\varepsilon$ and $\eta$, and
  \begin{align} \label{E:Test_Aux}
    F(\eta) := \nabla_\Gamma\eta+\frac{1}{g}(\eta\cdot\nabla_\Gamma g)Q \quad\text{on}\quad \Gamma.
  \end{align}
\end{lemma}

Note that $\mathcal{V}_\varepsilon$ may be smaller than $H^1(\Omega_\varepsilon)^3\cap L_\sigma^2(\Omega_\varepsilon)$ when the condition (A3) of Assumption~\ref{Assump_2} is satisfied (see \eqref{E:Def_Heps}).
We deal with this point at the first step of derivation of a weak formulation for $M_\tau u^\varepsilon$ (see Lemma~\ref{L:Mu_Weak}).

\begin{proof}
  Since $E_\varepsilon\eta\in H^1(\Omega_\varepsilon)^3$ by $\eta\in H^1(\Gamma)^3$ and Lemma~\ref{L:ExTan_Wmp}, $\eta_\varepsilon=\mathbb{L}_\varepsilon E_\varepsilon\eta$ belongs to $H^1(\Omega_\varepsilon)^3\cap L_\sigma^2(\Omega_\varepsilon)$ (see Section~\ref{SS:St_HL}).
  Let us show \eqref{E:Test_Dom} and \eqref{E:Test_Bo}.
  By the definition \eqref{E:Def_ExTan} of $E_\varepsilon\eta$ and the inequalities \eqref{E:Con_Lp} and \eqref{E:ExAux_Bound} we have
  \begin{align} \label{Pf_CTe:L2}
    \|E_\varepsilon\eta-\bar{\eta}\|_{L^2(\Omega_\varepsilon)} \leq c\varepsilon\|\bar{\eta}\|_{L^2(\Omega_\varepsilon)} \leq c\varepsilon^{3/2}\|\eta\|_{L^2(\Gamma)}.
  \end{align}
  Also, from \eqref{E:Con_Lp} and \eqref{E:ExTan_Grad} it follows that
  \begin{align} \label{Pf_CTe:H1}
    \left\|\nabla E_\varepsilon\eta-\overline{F(\eta)}\right\|_{L^2(\Omega_\varepsilon)} \leq c\varepsilon\left(\|\bar{\eta}\|_{L^2(\Omega_\varepsilon)}+\left\|\overline{\nabla_\Gamma\eta}\right\|_{L^2(\Omega_\varepsilon)}\right) \leq c\varepsilon^{3/2}\|\eta\|_{H^1(\Gamma)}.
  \end{align}
  Since $E_\varepsilon\eta\in H^1(\Omega_\varepsilon)^3$ satisfies $E_\varepsilon\eta\cdot n_\varepsilon=0$ on $\Gamma_\varepsilon$ by Lemma~\ref{L:ExTan_Imp}, we have
  \begin{align*}
    \|\eta_\varepsilon-E_\varepsilon\eta\|_{H^1(\Omega_\varepsilon)} \leq c\|\mathrm{div}(E_\varepsilon\eta)\|_{L^2(\Omega_\varepsilon)}
  \end{align*}
  by \eqref{E:HP_Dom} with $u=E_\varepsilon \eta$ and $\mathbb{L}_\varepsilon u=\eta_\varepsilon$.
  Noting that $\eta\in \mathcal{V}_g$ satisfies $\mathrm{div}_\Gamma(gv)=0$ on $\Gamma$, we further apply \eqref{E:Lp_ETD_Sol} to the right-hand side to deduce that
  \begin{align} \label{Pf_CTe:HL}
    \|\eta_\varepsilon-E_\varepsilon\eta\|_{H^1(\Omega_\varepsilon)} \leq c\varepsilon^{3/2}\|\eta\|_{H^1(\Gamma)}.
  \end{align}
  By \eqref{Pf_CTe:L2}--\eqref{Pf_CTe:HL} we obtain \eqref{E:Test_Dom}.
  To prove \eqref{E:Test_Bo} we use \eqref{E:Poin_Bo} and \eqref{E:Test_Dom} to get
  \begin{align*}
    \|\eta_\varepsilon-\bar{\eta}\|_{L^2(\Gamma_\varepsilon)} &\leq c\left(\varepsilon^{-1/2}\|\eta_\varepsilon-\bar{\eta}\|_{L^2(\Omega_\varepsilon)}+\varepsilon^{1/2}\|\partial_n\eta_\varepsilon-\partial_n\bar{\eta}\|_{L^2(\Omega_\varepsilon)}\right) \\
    &\leq c\left(\varepsilon\|\eta\|_{H^1(\Gamma)}+\varepsilon^{1/2}\|\eta_\varepsilon\|_{H^1(\Omega_\varepsilon)}\right).
  \end{align*}
  Here we used \eqref{E:NorDer_Con} for $\partial_n\bar{\eta}$ in the second inequality.
  Moreover, by \eqref{E:ExTan_Wmp} and \eqref{Pf_CTe:HL},
  \begin{align*}
    \|\eta_\varepsilon\|_{H^1(\Omega_\varepsilon)} \leq \|E_\varepsilon\eta\|_{H^1(\Omega_\varepsilon)}+\|\eta_\varepsilon-E_\varepsilon\eta\|_{H^1(\Omega_\varepsilon)} \leq c\varepsilon^{1/2}\|\eta\|_{H^1(\Gamma)}.
  \end{align*}
  Hence the inequality \eqref{E:Test_Bo} follows.
\end{proof}

Next we approximate the bilinear and trilinear forms $a_\varepsilon$ and $b_\varepsilon$ by bilinear and trilinear forms for tangential vector fields on $\Gamma$.
Let $\gamma^0,\gamma^1\geq0$.
We define
\begin{multline} \label{E:Def_Bili_Surf}
  a_g(v_1,v_2) := 2\nu\left\{\bigl(gD_\Gamma(v_1),D_\Gamma(v_2)\bigr)_{L^2(\Gamma)}+(g^{-1}(v_1\cdot\nabla_\Gamma g),v_2\cdot\nabla_\Gamma g)_{L^2(\Gamma)}\right\} \\
  +(\gamma^0+\gamma^1)(v_1,v_2)_{L^2(\Gamma)}
\end{multline}
for $v_1,v_2\in H^1(\Gamma,T\Gamma)$, where $D_\Gamma(v_1)$ is given by \eqref{E:Strain_Surf}, and
\begin{align} \label{E:Def_Tri_Surf}
  b_g(v_1,v_2,v_3) := -(g(v_1\otimes v_2),\nabla_\Gamma v_3)_{L^2(\Gamma)}
\end{align}
for $v_1,v_2,v_3\in H^1(\Gamma,T\Gamma)$ (note that $v_1\otimes v_2\in L^2(\Gamma)^{3\times3}$ by $H^1(\Gamma)\hookrightarrow L^4(\Gamma)$, see Lemma~\ref{L:La_Surf}).
First we give their basic properties.

\begin{lemma} \label{L:Bi_Surf}
  There exists a constant $c>0$ such that
  \begin{align} \label{E:Bi_Surf}
    \|\nabla_\Gamma v\|_{L^2(\Gamma)}^2 \leq c\left\{a_g(v,v)+\|v\|_{L^2(\Gamma)}^2\right\}
  \end{align}
  for all $v\in H^1(\Gamma,T\Gamma)$.
\end{lemma}

\begin{proof}
  By \eqref{E:Width_Bound} and the Korn inequality \eqref{E:Korn_STG} we have
  \begin{align*}
    \|\nabla_\Gamma v\|_{L^2(\Gamma)}^2 &\leq c\left(\|D_\Gamma(v)\|_{L^2(\Gamma)}^2+\|v\|_{L^2(\Gamma)}^2\right) \\
    &\leq c\left(2\nu\|g^{1/2}D_\Gamma(v)\|_{L^2(\Gamma)}^2+\|v\|_{L^2(\Gamma)}^2\right) \\
    &\leq c\left\{a_g(v,v)+\|v\|_{L^2(\Gamma)}^2\right\}
  \end{align*}
  for all $v\in H^1(\Gamma,T\Gamma)$.
  Hence \eqref{E:Bi_Surf} holds.
\end{proof}

\begin{lemma} \label{L:Tri_Surf}
  There exists a constant $c>0$ such that
  \begin{align} \label{E:Tri_Surf}
    |b_g(v_1,v_2,v_3)| \leq c\|v_1\|_{L^2(\Gamma)}^{1/2}\|v_1\|_{H^1(\Gamma)}^{1/2}\|v_2\|_{L^2(\Gamma)}^{1/2}\|v_2\|_{H^1(\Gamma)}^{1/2}\|v_3\|_{H^1(\Gamma)}
  \end{align}
  for all $v_1,v_2,v_3\in H^1(\Gamma,T\Gamma)$.
  Moreover,
  \begin{align} \label{E:TriS_Vg}
    b_g(v_1,v_2,v_3) = -b_g(v_1,v_3,v_2), \quad b_g(v_1,v_2,v_2) = 0
  \end{align}
  for all $v_1\in \mathcal{V}_g$ and $v_2,v_3\in H^1(\Gamma,T\Gamma)$.
\end{lemma}

\begin{proof}
  The inequality \eqref{E:Tri_Surf} follows from H\"{o}lder's inequality and Ladyzhenskaya's inequality \eqref{E:La_Surf}.
  Let $v_1\in \mathcal{V}_g$ and $v_2,v_3\in H^1(\Gamma,T\Gamma)$.
  For $a\in\mathbb{R}^3$ and $i=1,2,3$ we denote by $a^i$ the $i$-th component of $a$.
  Since
  \begin{align*}
    g(v_1\otimes v_2):\nabla_\Gamma v_3 &= \sum_{i,j=1}^3gv_1^iv_2^j\underline{D}_iv_3^j \\
    &= \sum_{i,j=1}^3\{\underline{D}_i(gv_1^iv_2^jv_3^j)-v_2^jv_3^j\underline{D}_i(gv_1^i)-gv_1^iv_3^j\underline{D}_iv_2^j\} \\
    &= \mathrm{div}_\Gamma[g(v_2\cdot v_3)v_1]-(v_2\cdot v_3)\mathrm{div}_\Gamma(gv_1)-v_1\otimes v_3:\nabla_\Gamma v_2
  \end{align*}
  on $\Gamma$ and $\mathrm{div}_\Gamma(gv_1)=0$ by $v_1\in \mathcal{V}_g$, we have
  \begin{align*}
    b_g(v_1,v_2,v_3) = \int_\Gamma\mathrm{div}_\Gamma[g(v_2\cdot v_3)v_1]\,d\mathcal{H}^2-b_g(v_1,v_3,v_2).
  \end{align*}
  Here the first term on the right-hand side vanishes by the Stokes theorem, since $g(v_2\cdot v_3)v_1$ is tangential on the closed surface $\Gamma$.
  Hence the first equality of \eqref{E:TriS_Vg} follows.
  We also get the second equality by setting $v_2=v_3$ in the first one.
\end{proof}

Let us approximate $a_\varepsilon$ by $a_g$ and $b_\varepsilon$ by $b_g$.

\begin{lemma} \label{L:Appr_Bili}
  Let $u\in H^2(\Omega_\varepsilon)^3$ satisfy the slip boundary conditions \eqref{E:Bo_Imp}--\eqref{E:Bo_Slip} on $\Gamma_\varepsilon$.
  Also, let $\eta\in \mathcal{V}_g$ and $\eta_\varepsilon:=\mathbb{L}_\varepsilon E_\varepsilon\eta\in H^1(\Omega_\varepsilon)^3\cap L_\sigma^2(\Omega_\varepsilon)$.
  Then
  \begin{align} \label{E:Appr_Bili}
    |a_\varepsilon(u,\eta_\varepsilon)-\varepsilon a_g(M_\tau u,\eta)| \leq cR_\varepsilon^a(u)\|\eta\|_{H^1(\Gamma)},
  \end{align}
  where $c>0$ is a constant independent of $\varepsilon$, $u$, and $\eta$, and
  \begin{align} \label{E:ApBi_Re}
    R_\varepsilon^a(u) := \varepsilon^{3/2}\|u\|_{H^2(\Omega_\varepsilon)}+\varepsilon^{1/2}\|u\|_{L^2(\Omega_\varepsilon)}\sum_{i=0,1}|\varepsilon^{-1}\gamma_\varepsilon^i-\gamma^i|.
  \end{align}
\end{lemma}

\begin{proof}
  Let $F(\eta)$ be the matrix given by \eqref{E:Test_Aux} and
  \begin{align*}
    J_1 &:= \bigl(D(u),D(\eta_\varepsilon)\bigr)_{L^2(\Omega_\varepsilon)}-\Bigl(D(u),\overline{F(\eta)}\Bigr)_{L^2(\Omega_\varepsilon)}, \\
    J_2 &:= \Bigl(D(u),\overline{F(\eta)}\Bigr)_{L^2(\Omega_\varepsilon)} \\
    &\qquad\qquad -\varepsilon\left\{\bigl(gD_\Gamma(M_\tau u),D_\Gamma(\eta)\bigr)_{L^2(\Gamma)}+\bigl(M_\tau u\cdot\nabla_\Gamma g,g^{-1}(\eta\cdot\nabla_\Gamma g)\bigr)_{L^2(\Gamma)}\right\}.
  \end{align*}
  We also define
  \begin{gather*}
    K_1 := \sum_{i=0,1}\gamma_\varepsilon^i\{(u,\eta_\varepsilon)_{L^2(\Gamma_\varepsilon^i)}-(u,\bar{\eta})_{L^2(\Gamma_\varepsilon^i)}\}, \\
    K_2 := \sum_{i=0,1}\gamma_\varepsilon^i\{(u,\bar{\eta})_{L^2(\Gamma_\varepsilon^i)}-(M_\tau u,\eta)_{L^2(\Gamma)}\}, \quad K_3 := \sum_{i=0,1}(\gamma_\varepsilon^i-\varepsilon\gamma^i)(M_\tau u,\eta)_{L^2(\Gamma)}
  \end{gather*}
  so that $a_\varepsilon(u,\eta_\varepsilon)-\varepsilon a_g(M_\tau u,\eta) = 2\nu(J_1+J_2)+K_1+K_2+K_3$.
  Let us estimate each term on the right-hand side.
  Since $D(u)$ is symmetric, $D(u):D(\eta_\varepsilon) = D(u):\nabla\eta_\varepsilon$ in $\Omega_\varepsilon$.
  By this equality and \eqref{E:Test_Dom} we have
  \begin{align} \label{Pf_ApBi:I1}
    |J_1| \leq \|D(u)\|_{L^2(\Omega_\varepsilon)}\left\|\nabla\eta_\varepsilon-\overline{F(\eta)}\right\|_{L^2(\Omega_\varepsilon)} \leq c\varepsilon^{3/2}\|u\|_{H^1(\Omega_\varepsilon)}\|\eta\|_{H^1(\Gamma)}.
  \end{align}
  Next we deal with $J_2$.
  Since $\eta\cdot n=0$ on $\Gamma$ we can apply \eqref{E:Grad_W} to get
  \begin{align*}
    F(\eta) = A+v\otimes n+\xi Q, \quad A := P(\nabla_\Gamma\eta)P, \quad v := W\eta, \quad \xi := g^{-1}(\eta\cdot\nabla_\Gamma g)
  \end{align*}
  on $\Gamma$.
  Using this decomposition and $D_\Gamma(M_\tau u):D_\Gamma(\eta)=D_\Gamma(M_\tau u):A$ on $\Gamma$ by the symmetry of $D_\Gamma(M_\tau u)$, we split $J_2=J_2^1+J_2^2+J_2^3$ into
  \begin{gather*}
    J_2^1 := \Bigl(D(u),\overline{A}\Bigr)_{L^2(\Omega_\varepsilon)}-\varepsilon(gD_\Gamma(M_\tau u),A)_{L^2(\Gamma)}, \\
    J_2^2 := \Bigl(D(u),\overline{\xi Q}\Bigr)_{L^2(\Omega_\varepsilon)}-\varepsilon(M_\tau u\cdot\nabla_\Gamma g,\xi)_{L^2(\Gamma)}, \quad J_2^3 := (D(u),\bar{v}\otimes\bar{n})_{L^2(\Omega_\varepsilon)}.
  \end{gather*}
  Since $u$ and $A$ satisfies the conditions of Lemma~\ref{L:Ave_BiH1_TT}, we can use \eqref{E:Ave_BiH1_TT} to get
  \begin{align*}
    |J_2^1| \leq c\varepsilon^{3/2}\|u\|_{H^1(\Omega_\varepsilon)}\|A\|_{L^2(\Gamma)} \leq c\varepsilon^{3/2}\|u\|_{H^1(\Omega_\varepsilon)}\|\eta\|_{H^1(\Gamma)}.
  \end{align*}
  Noting that $u$ satisfies \eqref{E:Bo_Imp} on $\Gamma_\varepsilon$, we deduce from \eqref{E:Ave_BiH1_NN} that
  \begin{align*}
    |J_2^2| \leq c\varepsilon^{3/2}\|u\|_{H^1(\Omega_\varepsilon)}\|\xi\|_{L^2(\Gamma)} \leq c\varepsilon^{3/2}\|u\|_{H^1(\Omega_\varepsilon)}\|\eta\|_{L^2(\Gamma)}.
  \end{align*}
  Also, since $v=W\eta\in H^1(\Gamma,T\Gamma)$ and $u$ satisfies \eqref{E:Bo_Slip} on $\Gamma_\varepsilon$, by \eqref{E:Ave_BiH1_TN} we have
  \begin{align*}
    |J_2^3| \leq c\varepsilon^{3/2}\|u\|_{H^2(\Omega_\varepsilon)}\|v\|_{L^2(\Gamma)} \leq c\varepsilon^{3/2}\|u\|_{H^2(\Omega_\varepsilon)}\|\eta\|_{L^2(\Gamma)}.
  \end{align*}
  From the above three estimates it follows that
  \begin{align} \label{Pf_ApBi:I2}
    |J_2| \leq |J_2^1|+|J_2^2|+|J_2^3| \leq c\varepsilon^{3/2}\|u\|_{H^2(\Omega_\varepsilon)}\|\eta\|_{H^1(\Gamma)}.
  \end{align}
  Let us estimate $K_1$, $K_2$, and $K_3$.
  To $K_1$ we apply \eqref{E:Fric_Upper}, \eqref{E:Poin_Bo}, and \eqref{E:Test_Bo} to get
  \begin{align} \label{Pf_ApBi:J1}
    |K_1| \leq c\varepsilon\|u\|_{L^2(\Gamma_\varepsilon)}\|\eta_\varepsilon-\bar{\eta}\|_{L^2(\Gamma_\varepsilon)} \leq c\varepsilon^{3/2}\|u\|_{H^1(\Omega_\varepsilon)}\|\eta\|_{H^1(\Gamma)}.
  \end{align}
  Also, since $\eta$ is tangential on $\Gamma$, we have $M_\tau u\cdot\eta=Mu\cdot\eta$ on $\Gamma$ and thus
  \begin{align} \label{Pf_ApBi:J2}
    |K_2| \leq c\varepsilon\sum_{i=0,1}\left|(u,\bar{\eta})_{L^2(\Gamma_\varepsilon^i)}-(Mu,\eta)_{L^2(\Gamma)}\right| \leq c\varepsilon^{3/2}\|u\|_{H^1(\Omega_\varepsilon)}\|\eta\|_{L^2(\Gamma)}
  \end{align}
  by \eqref{E:Fric_Upper} and \eqref{E:Ave_BiL2_Bo}.
  To $K_3$ we just use \eqref{E:Ave_Lp_Surf} to obtain
  \begin{align} \label{Pf_ApBi:J3}
    \begin{aligned}
      |K_3| &\leq c\varepsilon^{-1/2}\|u\|_{L^2(\Omega_\varepsilon)}\|\eta\|_{L^2(\Gamma)}\sum_{i=0,1}|\gamma_\varepsilon^i-\varepsilon\gamma^i| \\
      &= c\varepsilon^{1/2}\|u\|_{L^2(\Omega_\varepsilon)}\|\eta\|_{L^2(\Gamma)}\sum_{i=0,1}|\varepsilon^{-1}\gamma_\varepsilon^i-\gamma^i|.
    \end{aligned}
  \end{align}
  Finally, we deduce from \eqref{Pf_ApBi:I1}--\eqref{Pf_ApBi:J3} that
  \begin{align*}
    |a_\varepsilon(u,\eta_\varepsilon)-\varepsilon a_g(M_\tau u,\eta)| &\leq c(|J_1|+|J_2|+|K_1|+|K_2|+|K_3|) \\
    &\leq cR_\varepsilon^a(u)\|\eta\|_{H^1(\Gamma)},
  \end{align*}
  where $R_\varepsilon^a(u)$ is given by \eqref{E:ApBi_Re}.
  Hence \eqref{E:Appr_Bili} is valid.
\end{proof}

\begin{lemma} \label{L:Appr_Tri}
  Let $u_1\in H^2(\Omega_\varepsilon)^3$ satisfy $\mathrm{div}\,u_1=0$ in $\Omega_\varepsilon$ and \eqref{E:Bo_Imp}--\eqref{E:Bo_Slip} on $\Gamma_\varepsilon$.
  Also, let $u_2\in H^1(\Omega_\varepsilon)^3$ satisfy \eqref{E:Bo_Imp} on $\Gamma_\varepsilon^0$ or on $\Gamma_\varepsilon^1$.
  Then
  \begin{align} \label{E:Appr_Tri}
    |b_\varepsilon(u_1,u_2,\eta_\varepsilon)-\varepsilon b_g(M_\tau u_1,M_\tau u_2,\eta)| \leq cR_\varepsilon^b(u_1,u_2)\|\eta\|_{H^1(\Gamma)}
  \end{align}
  for all $\eta\in \mathcal{V}_g$ with $\eta_\varepsilon:=\mathbb{L}_\varepsilon E_\varepsilon\eta\in H^1(\Omega_\varepsilon)^3\cap L_\sigma^2(\Omega_\varepsilon)$, where
  \begin{multline} \label{E:ApTr_Re}
    R_\varepsilon^b(u_1,u_2) := \varepsilon^{3/2}\|u_1\otimes u_2\|_{L^2(\Omega_\varepsilon)}+\varepsilon\|u_1\|_{H^1(\Omega_\varepsilon)}\|u_2\|_{H^1(\Omega_\varepsilon)} \\
    +\left(\varepsilon\|u_1\|_{H^2(\Omega_\varepsilon)}+\varepsilon^{1/2}\|u_1\|_{L^2(\Omega_\varepsilon)}^{1/2}\|u_1\|_{H^2(\Omega_\varepsilon)}^{1/2}\right)\|u_2\|_{L^2(\Omega_\varepsilon)}
  \end{multline}
  and $c>0$ is a constant independent of $\varepsilon$, $u_1$, $u_2$, and $\eta$.
\end{lemma}

\begin{proof}
  Let $F(\eta)$ be the matrix given by \eqref{E:Test_Aux}.
  By \eqref{E:Def_Tri_Bo} and \eqref{E:Test_Dom} we have
  \begin{multline} \label{Pf_ApTr:TAux}
    \left|b_\varepsilon(u_1,u_2,\eta_\varepsilon)+\Bigl(u_1\otimes u_2,\overline{F(\eta)}\Bigr)_{L^2(\Omega_\varepsilon)}\right| \\
    \leq \|u_1\otimes u_2\|_{L^2(\Omega_\varepsilon)}\left\|\nabla\eta_\varepsilon-\overline{F(\eta)}\right\|_{L^2(\Omega_\varepsilon)} \\
    \leq c\varepsilon^{3/2}\|u_1\otimes u_2\|_{L^2(\Omega_\varepsilon)}\|\eta\|_{H^1(\Gamma)}.
  \end{multline}
  Let $A:=P(\nabla_\Gamma\eta)P$ and $v:=W\eta+g^{-1}(\eta\cdot\nabla_\Gamma g)n$ so that $F(\eta) = A+v\otimes n$ on $\Gamma$ by \eqref{E:Grad_W}.
  Since $u_1$ and $A$ satisfies the conditions of Lemma~\ref{L:Ave_TrT},
  \begin{multline} \label{Pf_ApTr:TT}
    \left|\Bigl(u_1\otimes u_2,\overline{A}\Bigr)_{L^2(\Omega_\varepsilon)}-\varepsilon(g(M_\tau u_1)\otimes(M_\tau u_2),A)_{L^2(\Gamma)}\right| \\
    \leq cR_\varepsilon(u_1,u_2)\|A\|_{L^2(\Gamma)} \leq cR_\varepsilon(u_1,u_2)\|\eta\|_{H^1(\Gamma)}
  \end{multline}
  by \eqref{E:Ave_TrT}.
  Here $R_\varepsilon(u_1,u_2)$ is given by \eqref{E:Ave_TrT_Re}.
  Also, since $u_2$ satisfies \eqref{E:Bo_Imp} on $\Gamma_\varepsilon^0$ or on $\Gamma_\varepsilon^1$ and $v\in H^1(\Gamma)^3$, the inequality \eqref{E:Ave_TrN} yields
  \begin{align} \label{Pf_ApTr:TN}
    \begin{aligned}
      |(u_1\otimes u_2,\bar{v}\otimes\bar{n})_{L^2(\Omega_\varepsilon)}| &\leq c\varepsilon\|u_1\|_{H^1(\Omega_\varepsilon)}\|u_2\|_{H^1(\Omega_\varepsilon)}\|v\|_{H^1(\Gamma)} \\
      &\leq c\varepsilon\|u_1\|_{H^1(\Omega_\varepsilon)}\|u_2\|_{H^1(\Omega_\varepsilon)}\|\eta\|_{H^1(\Gamma)}.
    \end{aligned}
  \end{align}
  Noting that $F(\eta)=A+v\otimes n$ on $\Gamma$ we deduce from \eqref{Pf_ApTr:TAux}--\eqref{Pf_ApTr:TN} that
  \begin{align} \label{Pf_ApTr:Goal}
    \left|b_\varepsilon(u_1,u_2,\eta_\varepsilon)+\varepsilon(g(M_\tau u_1)\otimes(M_\tau u_2),A)_{L^2(\Gamma)}\right| \leq cR_\varepsilon^b(u_1,u_2)\|\eta\|_{H^1(\Gamma)}
  \end{align}
  with $R_\varepsilon^b(u_1,u_2)$ given by \eqref{E:ApTr_Re}.
  Now we observe that
  \begin{align*}
    (M_\tau u_1)\otimes(M_\tau u_2):A = (M_\tau u_1)\otimes(M_\tau u_2):\nabla_\Gamma\eta \quad\text{on}\quad \Gamma
  \end{align*}
  by $A=P(\nabla_\Gamma\eta)P$ and $PM_\tau u_i=M_\tau u_i$ on $\Gamma$ for $i=1,2$.
  Hence
  \begin{align*}
    (g(M_\tau u_1)\otimes(M_\tau u_2),A)_{L^2(\Gamma)} = -b_g(M_\tau u_1,M_\tau u_2,\eta)
  \end{align*}
  by \eqref{E:Def_Tri_Surf} and the inequality \eqref{E:Appr_Tri} follows from \eqref{Pf_ApTr:Goal}.
\end{proof}

Now let us derive a weak formulation for $M_\tau u^\varepsilon$ from \eqref{E:NS_Weak}.

\begin{lemma} \label{L:Mu_Weak}
  Under the assumptions of Theorem~\ref{T:Est_Ue}, let $\varepsilon\in(0,\varepsilon'_1)$ and $u^\varepsilon$ be the global strong solution to \eqref{E:NS_Eq}--\eqref{E:NS_In} given by Theorem~\ref{T:Est_Ue}.
  Then
  \begin{align*}
    M_\tau u^\varepsilon \in C([0,\infty);H^1(\Gamma,T\Gamma))\cap H_{loc}^1([0,\infty);L^2(\Gamma,T\Gamma))
  \end{align*}
  and for all $\eta\in L^2(0,T;\mathcal{V}_g)$, $T>0$ we have
  \begin{multline} \label{E:Mu_Weak}
    \int_0^T\{(g\partial_tM_\tau u^\varepsilon,\eta)_{L^2(\Gamma)}+a_g(M_\tau u^\varepsilon,\eta)+b_g(M_\tau u^\varepsilon,M_\tau u^\varepsilon,\eta)\}\,dt \\
    = \int_0^T(gM_\tau\mathbb{P}_\varepsilon f^\varepsilon,\eta)_{L^2(\Gamma)}\,dt+R_\varepsilon^1(\eta).
  \end{multline}
  Here the residual term $R_\varepsilon^1(\eta)$ satisfies
  \begin{align} \label{E:Mu_Weak_Re}
    |R_\varepsilon^1(\eta)| \leq c\left(\varepsilon^{\alpha/4}+\sum_{i=0,1}|\varepsilon^{-1}\gamma_\varepsilon^i-\gamma^i|\right)(1+T)^{1/2}\|\eta\|_{L^2(0,T;H^1(\Gamma))}
  \end{align}
  with a constant $c>0$ independent of $\varepsilon$, $u^\varepsilon$, $\eta$, and $T$.
\end{lemma}

\begin{proof}
  The space-time regularity of $M_\tau u^\varepsilon$ follows from that of $u^\varepsilon$ and Lemmas~\ref{L:Ave_Lp}, \ref{L:Ave_Dt}, and~\ref{L:Ave_Wmp}.
  Let us derive \eqref{E:Mu_Weak}.
  For $\eta\in L^2(0,T;\mathcal{V}_g)$, $T>0$ let
  \begin{align*}
    \eta_\varepsilon := \mathbb{L}_\varepsilon E_\varepsilon\eta \in L^2(0,T;H^1(\Omega_\varepsilon)^3\cap L_\sigma^2(\Omega_\varepsilon)).
  \end{align*}
  In what follows, we suppress the time variable $t\in(0,T)$.
  If the condition (A1) or (A2) of Assumption~\ref{Assump_2} is satisfied, then $\eta_\varepsilon\in\mathcal{V}_\varepsilon$ by \eqref{E:Def_Heps} and we can take $\eta_\varepsilon$ as a test function in the weak formulation \eqref{E:NS_Weak}.
  On the other hand, if the condition (A3) is satisfied, $\mathcal{V}_\varepsilon$ may be smaller than $H^1(\Omega_\varepsilon)^3\cap L_\sigma^2(\Omega_\varepsilon)$ and we cannot substitute $\eta_\varepsilon$ for $\varphi$ in \eqref{E:NS_Weak}.
  In this case, we use the orthogonal projection $\mathbb{P}_\varepsilon$ onto $\mathcal{H}_\varepsilon$ and consider $\mathbb{P}_\varepsilon\eta_\varepsilon\in\mathcal{H}_\varepsilon$.
  Then since $\eta_\varepsilon\in L_\sigma^2(\Omega_\varepsilon)=\mathcal{H}_\varepsilon\oplus\mathcal{R}_g$ with $\mathcal{R}_g$ given by \eqref{E:Def_Rg}, there exists $w_\varepsilon(x)=a_\varepsilon\times x+b_\varepsilon\in\mathcal{R}_g$ such that $\eta_\varepsilon=\mathbb{P}_\varepsilon\eta_\varepsilon+w_\varepsilon$.
  By this relation and $\eta_\varepsilon,w_\varepsilon\in H^1(\Omega_\varepsilon)^3$ we get $\mathbb{P}_\varepsilon\eta_\varepsilon=\eta_\varepsilon-w_\varepsilon\in H^1(\Omega_\varepsilon)^3\cap\mathcal{H}_\varepsilon=\mathcal{V}_\varepsilon$ and thus we can substitute it for $\varphi$ in \eqref{E:NS_Weak}.
  In the resulting equality we have
  \begin{align*}
    (\partial_tu^\varepsilon,w_\varepsilon)_{L^2(\Omega_\varepsilon)} = (\mathbb{P}_\varepsilon f^\varepsilon,w_\varepsilon)_{L^2(\Omega_\varepsilon)} = 0
  \end{align*}
  by $\partial_tu^\varepsilon,\mathbb{P}_\varepsilon f^\varepsilon\in\mathcal{H}_\varepsilon$ and $w_\varepsilon\in\mathcal{H}_\varepsilon^\perp$.
  Moreover, since we assume $\gamma_\varepsilon^0=\gamma_\varepsilon^1=0$ in the condition (A3) and $D(w_\varepsilon)=0$ in $\Omega_\varepsilon$ by $w_\varepsilon(x)=a_\varepsilon\times x+b_\varepsilon$,
  \begin{align*}
    a_\varepsilon(u^\varepsilon,w_\varepsilon) = 2\nu\bigl(D(u^\varepsilon),D(w_\varepsilon)\bigr)_{L^2(\Omega_\varepsilon)} = 0.
  \end{align*}
  We also get $b_\varepsilon(u^\varepsilon,u^\varepsilon,w_\varepsilon)=-(u^\varepsilon,a_\varepsilon\times u^\varepsilon)_{L^2(\Omega_\varepsilon)}=0$ by direct calculations.
  Hence all terms including $w_\varepsilon$ vanish in \eqref{E:NS_Weak} with $\varphi=\mathbb{P}_\varepsilon\eta_\varepsilon=\eta_\varepsilon-w_\varepsilon$.
  Therefore, under any condition of (A1), (A2), and (A3), we deduce from \eqref{E:NS_Weak} that
  \begin{align*}
    \int_0^T\{(\partial_tu^\varepsilon,\eta_\varepsilon)_{L^2(\Omega_\varepsilon)}+a_\varepsilon(u^\varepsilon,\eta_\varepsilon)+b_\varepsilon(u^\varepsilon,u^\varepsilon,\eta_\varepsilon)\}\,dt = \int_0^T(\mathbb{P}_\varepsilon f^\varepsilon,\eta_\varepsilon)_{L^2(\Omega_\varepsilon)}\,dt
  \end{align*}
  with $\eta_\varepsilon=\mathbb{L}_\varepsilon E_\varepsilon\eta$ for $\eta\in L^2(0,T;\mathcal{V}_g)$.
  We divide both sides of this equality by $\varepsilon$ and replace each term of the resulting equality by the corresponding term of \eqref{E:Mu_Weak}.
  Then we get \eqref{E:Mu_Weak} with $R_\varepsilon^1(\eta):=\varepsilon^{-1}(J_1+J_2+J_3+J_4)$, where
  \begin{align*}
    J_1 &:= -\int_0^T(\partial_tu^\varepsilon,\eta_\varepsilon)_{L^2(\Omega_\varepsilon)}\,dt+\varepsilon\int_0^T(g\partial_tM_\tau u^\varepsilon,\eta)_{L^2(\Gamma)}\,dt, \\
    J_2 &:= -\int_0^Ta_\varepsilon(u^\varepsilon,\eta_\varepsilon)\,dt+\varepsilon\int_0^Ta_g(M_\tau u^\varepsilon,\eta)\,dt, \\
    J_3 &:= -\int_0^T b_\varepsilon(u^\varepsilon,u^\varepsilon,\eta_\varepsilon)\,dt+\varepsilon\int_0^Tb_g(M_\tau u^\varepsilon,M_\tau u^\varepsilon,\eta)\,dt, \\
    J_4 &:= -\int_0^T(\mathbb{P}_\varepsilon f^\varepsilon,\eta_\varepsilon)_{L^2(\Omega_\varepsilon)}\,dt+\varepsilon\int_0^T(gM_\tau\mathbb{P}_\varepsilon f^\varepsilon,\eta)_{L^2(\Gamma)}\,dt.
  \end{align*}
  Let us estimate these differences.
  First note that
  \begin{align*}
    (g\partial_tM_\tau u^\varepsilon,\eta)_{L^2(\Gamma)} = (gM_\tau(\partial_tu^\varepsilon),\eta)_{L^2(\Gamma)} = (gM(\partial_tu^\varepsilon),\eta)_{L^2(\Gamma)}
  \end{align*}
  by Lemma~\ref{L:Ave_Dt} and the fact that $\eta$ is tangential on $\Gamma$.
  Thus, by \eqref{E:Ave_BiL2_Dom} and \eqref{E:Test_Bo},
  \begin{align*}
    &|(\partial_tu^\varepsilon,\eta_\varepsilon)_{L^2(\Omega_\varepsilon)}-\varepsilon(g\partial_tM_\tau u^\varepsilon,\eta)_{L^2(\Gamma)}| \\
    &\qquad \leq |(\partial_tu^\varepsilon,\bar{\eta})_{L^2(\Omega_\varepsilon)}-\varepsilon(gM(\partial_tu^\varepsilon),\eta)_{L^2(\Gamma)}|+\|\partial_tu^\varepsilon\|_{L^2(\Omega_\varepsilon)}\|\eta_\varepsilon-\bar{\eta}\|_{L^2(\Omega_\varepsilon)} \\
    &\qquad \leq c\varepsilon^{3/2}\|\partial_tu^\varepsilon\|_{L^2(\Omega_\varepsilon)}\|\eta\|_{L^2(\Gamma)}.
  \end{align*}
  From this inequality, H\"{o}lder's inequality, and \eqref{E:Est_DtUe} it follows that
  \begin{align} \label{Pf_MuW:Dt}
    \begin{aligned}
      |J_1| &\leq c\varepsilon^{3/2}\|\partial_tu^\varepsilon\|_{L^2(0,T;L^2(\Omega_\varepsilon))}\|\eta\|_{L^2(0,T;L^2(\Gamma))} \\
      &\leq c\varepsilon^{1+\alpha/2}(1+T)^{1/2}\|\eta\|_{L^2(0,T;L^2(\Gamma))}.
    \end{aligned}
  \end{align}
  In the same way, we apply \eqref{E:Ave_BiL2_Dom} and \eqref{E:Test_Bo} to $J_4$ and then use \eqref{E:UE_Data} to get
  \begin{align} \label{Pf_MuW:F}
    |J_4| \leq c\varepsilon^{1+\alpha/2}T^{1/2}\|\eta\|_{L^2(0,T;L^2(\Gamma))}.
  \end{align}
  Next we deal with $J_2$.
  By \eqref{E:Appr_Bili} we see that
  \begin{align*}
    |J_2| \leq c\left(\int_0^TR_\varepsilon^a(u^\varepsilon)^2\,dt\right)^{1/2}\|\eta\|_{L^2(0,T;H^1(\Gamma))},
  \end{align*}
  where $R_\varepsilon^a(u^\varepsilon)$ is given by \eqref{E:ApBi_Re}.
  Moreover, by \eqref{E:Est_Ue} we have
  \begin{align*}
    \int_0^TR_\varepsilon^a(u^\varepsilon)^2\,dt &\leq c\left(\varepsilon^3\int_0^T\|u^\varepsilon\|_{H^2(\Omega_\varepsilon)}^2\,dt+\varepsilon\gamma(\varepsilon)^2\int_0^T\|u^\varepsilon\|_{L^2(\Omega_\varepsilon)}^2\,dt\right) \\
    &\leq c\varepsilon^2\{\varepsilon^\alpha+\gamma(\varepsilon)^2\}(1+T)
  \end{align*}
  with $\gamma(\varepsilon):=\sum_{i=0,1}|\varepsilon^{-1}\gamma_\varepsilon^i-\gamma^i|$.
  Therefore,
  \begin{align} \label{Pf_MuW:A}
    |J_2| \leq c\varepsilon\{\varepsilon^{\alpha/2}+\gamma(\varepsilon)\}(1+T)^{1/2}\|\eta\|_{L^2(0,T;H^1(\Gamma))}.
  \end{align}
  Now let us estimate $J_3$.
  By \eqref{E:Appr_Tri} we have
  \begin{align*}
    |J_3| \leq c\left(\int_0^TR_\varepsilon^b(u^\varepsilon,u^\varepsilon)^2\,dt\right)^{1/2}\|\eta\|_{L^2(0,T;H^1(\Gamma))},
  \end{align*}
  where $R_\varepsilon^b(u^\varepsilon,u^\varepsilon)$ is given by \eqref{E:ApTr_Re}.
  To estimate the right-hand side, we see that
  \begin{align*}
    \int_0^T\|u^\varepsilon\|_{H^1(\Omega_\varepsilon)}^4\,dt \leq \left(\sup_{t\in(0,T)}\|u^\varepsilon(t)\|_{H^1(\Omega_\varepsilon)}^2\right)\int_0^T\|u^\varepsilon\|_{H^1(\Omega_\varepsilon)}^2\,dt \leq c\varepsilon^\alpha(1+T)
  \end{align*}
  by \eqref{E:Est_Ue}.
  Using this inequality, \eqref{E:Est_UeUe}, \eqref{Pf_EsSt:L2H2}, and \eqref{Pf_EsSt:L2H2_31} we deduce that
  \begin{align*}
    &\int_0^TR_\varepsilon^b(u^\varepsilon,u^\varepsilon)^2\,dt \\
    &\qquad \leq c\left(\varepsilon^3\int_0^T\|u^\varepsilon\otimes u^\varepsilon\|_{L^2(\Omega_\varepsilon)}^2\,dt+\varepsilon^2\int_0^T\|u^\varepsilon\|_{H^1(\Omega_\varepsilon)}^4\,dt\right. \\
    &\qquad\qquad \left.+\varepsilon^2\int_0^T\|u^\varepsilon\|_{L^2(\Omega_\varepsilon)}^2\|u^\varepsilon\|_{H^2(\Omega_\varepsilon)}^2\,dt+\varepsilon\int_0^T\|u^\varepsilon\|_{L^2(\Omega_\varepsilon)}^3\|u^\varepsilon\|_{H^2(\Omega_\varepsilon)}\,dt\right)\\
    &\qquad \leq c\varepsilon^2(\varepsilon^2+\varepsilon^\alpha+\varepsilon^{\alpha/2})(1+T) \leq c\varepsilon^{2+\alpha/2}(1+T).
  \end{align*}
  Hence we obtain
  \begin{align} \label{Pf_MuW:B}
    |J_3| \leq c\varepsilon^{1+\alpha/4}(1+T)^{1/2}\|\eta\|_{L^2(0,T;H^1(\Gamma))}.
  \end{align}
  Finally, by \eqref{Pf_MuW:Dt}--\eqref{Pf_MuW:B} we observe that
  \begin{align*}
    |R_\varepsilon^1(\eta)| \leq \varepsilon^{-1}\sum_{k=1}^4|J_k| \leq c\{\varepsilon^{\alpha/4}+\varepsilon^{\alpha/2}+\gamma(\varepsilon)\}(1+T)^{1/2}\|\eta\|_{L^2(0,T;H^1(\Gamma))}
  \end{align*}
  and thus \eqref{E:Mu_Weak_Re} holds by $\gamma(\varepsilon)=\sum_{i=0,1}|\varepsilon^{-1}\gamma_\varepsilon^i-\gamma^i|$ and $\varepsilon^{\alpha/2}\leq\varepsilon^{\alpha/4}$.
\end{proof}

\begin{remark} \label{R:Mu_Weak}
  By $u^\varepsilon\in L_{loc}^2([0,\infty);D(A_\varepsilon))$ and Lemma~\ref{L:Ave_Wmp} we also have
  \begin{align*}
    M_\tau u^\varepsilon \in L_{loc}^2([0,\infty);H^2(\Gamma,T\Gamma)).
  \end{align*}
  We do not use this property in the sequel.
\end{remark}

\subsection{Energy estimate for the average of a solution} \label{SS:SL_Ener}
Next we prove the energy estimate for the averaged tangential component $M_\tau u^\varepsilon$.
In the weak formulation \eqref{E:Mu_Weak} for $M_\tau u^\varepsilon$ we cannot take $M_\tau u^\varepsilon$ itself as a test function since it is not in $\mathcal{V}_g$, i.e. $\mathrm{div}_\Gamma(gM_\tau u^\varepsilon)$ does not vanish on $\Gamma$ in general.
To overcome this difficulty we use the weighted Helmholtz--Leray projection $\mathbb{P}_g$ onto $L_{g\sigma}^2(\Gamma,T\Gamma)$ given in Section~\ref{SS:WS_HL} and replace $M_\tau u^\varepsilon$ in \eqref{E:Mu_Weak} with $\mathbb{P}_gM_\tau u^\varepsilon$.

\begin{lemma} \label{L:Rep_WHL_L2}
  Let $u\in L_\sigma^2(\Omega_\varepsilon)$.
  Then $\mathbb{P}_gM_\tau u\in L_{g\sigma}^2(\Gamma,T\Gamma)$ and
  \begin{align} \label{E:Rep_WHL_L2}
    \|M_\tau u-\mathbb{P}_gM_\tau u\|_{L^2(\Gamma)} \leq c\varepsilon^{1/2}\|u\|_{L^2(\Omega_\varepsilon)}
  \end{align}
  with a constant $c>0$ independent of $\varepsilon$ and $u$.
  Moreover, if $u\in H^1(\Omega_\varepsilon)^3\cap L_\sigma^2(\Omega_\varepsilon)$ then $\mathbb{P}_gM_\tau u\in \mathcal{V}_g$ and
  \begin{align} \label{E:Rep_WHL_H1}
    \|M_\tau u-\mathbb{P}_gM_\tau u\|_{H^1(\Gamma)} \leq c\varepsilon^{1/2}\|u\|_{H^1(\Omega_\varepsilon)}.
  \end{align}
\end{lemma}

\begin{proof}
  For $u\in L_\sigma^2(\Omega_\varepsilon)$ we have $M_\tau u\in L^2(\Gamma,T\Gamma)$ by Lemma~\ref{L:Ave_Lp} and thus $\mathbb{P}_gM_\tau u\in L_{g\sigma}^2(\Gamma,T\Gamma)$ by the definition of $\mathbb{P}_g$.
  Moreover, applying \eqref{E:HLT_Est_L2} to $v=M_\tau u$ and then using \eqref{E:ADiv_Tan_Hin} we obtain \eqref{E:Rep_WHL_L2}.

  If $u\in H^1(\Omega_\varepsilon)^3\cap L_\sigma^2(\Omega_\varepsilon)$ then Lemma~\ref{L:Ave_Wmp} implies $M_\tau u\in H^1(\Gamma,T\Gamma)$.
  Hence we get $\mathbb{P}_gM_\tau u\in \mathcal{V}_g$ by Lemma~\ref{L:HLT_Est_L2}.
  Also, the inequality \eqref{E:Rep_WHL_H1} follows from \eqref{E:ADiv_Tan_Lp} and \eqref{E:HLT_Est_H1} with $v=M_\tau u$.
\end{proof}

\begin{lemma} \label{L:Rep_WHL_Dt}
  For $T>0$, let $u\in H^1(0,T;L_\sigma^2(\Omega_\varepsilon))$.
  Then
  \begin{align} \label{E:Rep_WHL_Reg}
    \mathbb{P}_gM_\tau u\in H^1(0,T;L_{g\sigma}^2(\Gamma,T\Gamma))
  \end{align}
  and there exists a constant $c>0$ independent of $\varepsilon$ and $u$ such that
  \begin{align} \label{E:Rep_WHL_Dt}
    \|\partial_tM_\tau u-\partial_t\mathbb{P}_gM_\tau u\|_{L^2(0,T;L^2(\Gamma))} \leq c\varepsilon^{1/2}\|\partial_tu\|_{L^2(0,T;L^2(\Omega_\varepsilon))}.
  \end{align}
\end{lemma}

\begin{proof}
  By the space-time regularity of $u$ and Lemmas~\ref{L:Ave_Lp} and~\ref{L:Ave_Dt} we see that $M_\tau u$ is in $H^1(0,T;L^2(\Gamma,T\Gamma))$.
  Hence we get \eqref{E:Rep_WHL_Reg} by Lemma~\ref{L:HLT_Est_Dt} and
  \begin{align*}
    \|\partial_tM_\tau u-\partial_t\mathbb{P}_gM_\tau u\|_{L^2(0,T;L^2(\Gamma))} \leq c\|\mathrm{div}_\Gamma(g\partial_tM_\tau u)\|_{L^2(0,T;H^{-1}(\Gamma))}
  \end{align*}
  by \eqref{E:HLT_Est_Dt}.
  Noting that $\partial_tM_\tau u=M_\tau(\partial_tu)$ in $L^2(0,T;L^2(\Gamma,T\Gamma))$ by Lemma~\ref{L:Ave_Dt}, we deduce from $\partial_tu\in L^2(0,T;L_\sigma^2(\Omega_\varepsilon))$ and \eqref{E:ADiv_Tan_Hin} that
  \begin{align*}
    \|\mathrm{div}_\Gamma(g\partial_tM_\tau u)\|_{L^2(0,T;H^{-1}(\Gamma))} &= \|\mathrm{div}_\Gamma[gM_\tau(\partial_tu)]\|_{L^2(0,T;H^{-1}(\Gamma))} \\
    &\leq c\varepsilon^{1/2}\|\partial_tu\|_{L^2(0,T;L^2(\Omega_\varepsilon))}.
  \end{align*}
  Combining the above two inequalities we obtain \eqref{E:Rep_WHL_Dt}.
\end{proof}

Using Lemmas~\ref{L:Rep_WHL_L2} and~\ref{L:Rep_WHL_Dt} we derive a weak formulation for $\mathbb{P}_gM_\tau u^\varepsilon$.

\begin{lemma} \label{L:PMu_Weak}
  Let $u^\varepsilon$ be as in Lemma~\ref{L:Mu_Weak}.
  Then
  \begin{align*}
    v^\varepsilon := \mathbb{P}_gM_\tau u^\varepsilon \in C([0,\infty);\mathcal{V}_g)\cap H_{loc}^1([0,\infty),L_{g\sigma}^2(\Gamma,T\Gamma))
  \end{align*}
  and there exists a constant $c>0$ independent of $\varepsilon$ and $u^\varepsilon$ such that
  \begin{align} \label{E:Diff_PMu}
    \begin{aligned}
      \|M_\tau u^\varepsilon(t)-v^\varepsilon(t)\|_{L^2(\Gamma)}^2 &\leq c\varepsilon^2, \\
      \int_0^t\|M_\tau u^\varepsilon(s)-v^\varepsilon(s)\|_{H^1(\Gamma)}^2\,ds &\leq c\varepsilon^2(1+t)
      \end{aligned}
  \end{align}
  for all $t\geq 0$.
  Moreover, for all $\eta\in L^2(0,T;\mathcal{V}_g)$, $T>0$ we have
  \begin{multline} \label{E:PMu_Weak}
    \int_0^T\{(g\partial_tv^\varepsilon,\eta)_{L^2(\Gamma)}+a_g(v^\varepsilon,\eta)+b_g(v^\varepsilon,v^\varepsilon,\eta)\}\,dt \\
    = \int_0^T(gM_\tau\mathbb{P}_\varepsilon f^\varepsilon,\eta)_{L^2(\Gamma)}\,dt+R_\varepsilon^1(\eta)+R_\varepsilon^2(\eta),
  \end{multline}
  where $R_\varepsilon^1(\eta)$ is given in Lemma~\ref{L:Mu_Weak} and $R_\varepsilon^2(\eta)$ satisfies
  \begin{align} \label{E:PMu_Weak_Re}
    |R_\varepsilon^2(\eta)| \leq c\varepsilon^{\alpha/2}(1+T)^{1/2}\|\eta\|_{L^2(0,T;H^1(\Gamma))}
  \end{align}
  with a constant $c>0$ independent of $\varepsilon$, $v^\varepsilon$, $\eta$, and $T$.
\end{lemma}

\begin{proof}
  By Lemmas~\ref{L:Mu_Weak}, \ref{L:Rep_WHL_L2}, and \ref{L:Rep_WHL_Dt} we obtain the space-time regularity of $v^\varepsilon$.
  Also, the inequalities \eqref{E:Diff_PMu} immediately follow from \eqref{E:Est_Ue}, \eqref{E:Rep_WHL_L2}, and \eqref{E:Rep_WHL_H1}.
  For $\eta\in L^2(0,T;\mathcal{V}_g)$ let $R_\varepsilon^2(\eta):=J_1+J_2+J_3$ with
  \begin{align*}
    J_1 &:= -\int_0^T(g\partial_t M_\tau u^\varepsilon,\eta)_{L^2(\Gamma)}\,dt+\int_0^T(g\partial_tv^\varepsilon,\eta)_{L^2(\Gamma)}\,dt, \\
    J_2 &:= -\int_0^Ta_g(M_\tau u^\varepsilon,\eta)\,dt+\int_0^Ta_g(v^\varepsilon,\eta)\,dt, \\
    J_3 &:= -\int_0^Tb_g(M_\tau u^\varepsilon,M_\tau u^\varepsilon,\eta)\,dt+\int_0^Tb_g(v^\varepsilon,v^\varepsilon,\eta)\,dt.
  \end{align*}
  Then we get \eqref{E:PMu_Weak} by \eqref{E:Mu_Weak}.
  Let us estimate $J_1$, $J_2$, and $J_3$.
  For $J_1$,
  \begin{align} \label{Pf_PMuW:I1}
    \begin{aligned}
      |J_1| &\leq c\|\partial_tM_\tau u^\varepsilon-\partial_tv^\varepsilon\|_{L^2(0,T;L^2(\Gamma))}\|\eta\|_{L^2(0,T;L^2(\Gamma))} \\
      &\leq c\varepsilon^{1/2}\|\partial_tu^\varepsilon\|_{L^2(0,T;L^2(\Omega_\varepsilon))}\|\eta\|_{L^2(0,T;L^2(\Gamma))} \\
      &\leq c\varepsilon^{\alpha/2}(1+T)^{1/2}\|\eta\|_{L^2(0,T;L^2(\Gamma))}
    \end{aligned}
  \end{align}
  by \eqref{E:Est_DtUe} and \eqref{E:Rep_WHL_Dt}.
  Also, we use \eqref{E:Width_Bound} and \eqref{E:Diff_PMu} to get
  \begin{align} \label{Pf_PMuW:I2}
    \begin{aligned}
      |J_2| &\leq c\|M_\tau u^\varepsilon-v^\varepsilon\|_{L^2(0,T;H^1(\Gamma))}\|\eta\|_{L^2(0,T;H^1(\Gamma))} \\
      &\leq c\varepsilon(1+T)^{1/2}\|\eta\|_{L^2(0,T;H^1(\Gamma))}.
    \end{aligned}
  \end{align}
  Now let us consider $J_3$.
  Using H\"{o}lder's inequality twice we have
  \begin{multline*}
    |b_g(M_\tau u^\varepsilon,M_\tau u^\varepsilon,\eta)-b_g(v^\varepsilon,v^\varepsilon,\eta)| \\
    \leq \|M_\tau u^\varepsilon-v^\varepsilon\|_{L^4(\Gamma)}\left(\|M_\tau u^\varepsilon\|_{L^4(\Gamma)}+\|v^\varepsilon\|_{L^4(\Gamma)}\right)\|\nabla_\Gamma\eta\|_{L^2(\Gamma)}.
  \end{multline*}
  Moreover, by \eqref{E:La_Surf}, \eqref{E:Ave_Lp_Surf}, \eqref{E:Ave_Wmp_Surf}, \eqref{E:HLT_Bound}, \eqref{E:Rep_WHL_L2}, and \eqref{E:Rep_WHL_H1} we observe that
  \begin{align*}
    \|w\|_{L^4(\Gamma)} &\leq c\varepsilon^{-1/2}\|u^\varepsilon\|_{L^2(\Omega_\varepsilon)}^{1/2}\|u^\varepsilon\|_{H^1(\Omega_\varepsilon)}^{1/2}, \quad w = M_\tau u^\varepsilon,v^\varepsilon, \\
    \|M_\tau u^\varepsilon-v^\varepsilon\|_{L^4(\Gamma)} &\leq c\varepsilon^{1/2}\|u^\varepsilon\|_{L^2(\Omega_\varepsilon)}^{1/2}\|u^\varepsilon\|_{H^1(\Omega_\varepsilon)}^{1/2}.
  \end{align*}
  Hence
  \begin{align*}
    |b_g(M_\tau u^\varepsilon,M_\tau u^\varepsilon,\eta)-b_g(v^\varepsilon,v^\varepsilon,\eta)| \leq c\|u^\varepsilon\|_{L^2(\Omega_\varepsilon)}\|u^\varepsilon\|_{H^1(\Omega_\varepsilon)}\|\eta\|_{H^1(\Gamma)}
  \end{align*}
  and we use \eqref{Pf_EsSt:L2H1} to obtain
  \begin{align} \label{Pf_PMuW:I3}
    \begin{aligned}
      |J_3| &\leq c\left(\int_0^T\|u^\varepsilon\|_{L^2(\Omega_\varepsilon)}^2\|u^\varepsilon\|_{H^1(\Omega_\varepsilon)}^2\,dt\right)^{1/2}\|\eta\|_{L^2(0,T;H^1(\Gamma))} \\
      &\leq c\varepsilon(1+T)^{1/2}\|\eta\|_{L^2(0,T;H^1(\Gamma))}.
    \end{aligned}
  \end{align}
  Applying \eqref{Pf_PMuW:I1}--\eqref{Pf_PMuW:I3} and $\varepsilon\leq\varepsilon^{\alpha/2}$ to $R_\varepsilon^2(\eta)=J_1+J_2+J_3$ we get \eqref{E:PMu_Weak_Re}.
\end{proof}

Based on \eqref{E:PMu_Weak} we prove the energy estimate for $v^\varepsilon=\mathbb{P}_gM_\tau u^\varepsilon$.

\begin{lemma} \label{L:PMu_Energy}
  Let $u^\varepsilon$ be as in Lemma~\ref{L:Mu_Weak} and $v^\varepsilon=\mathbb{P}_gM_\tau u^\varepsilon$.
  Then
  \begin{align} \label{E:PMu_Energy}
    \max_{t\in[0,T]}\|v^\varepsilon(t)\|_{L^2(\Gamma)}^2+\int_0^T\|\nabla_\Gamma v^\varepsilon(t)\|_{L^2(\Gamma)}^2\,dt \leq c_T
  \end{align}
  for all $T>0$, where $c_T>0$ is a constant depending on $T$ but independent of $\varepsilon$.
\end{lemma}

\begin{proof}
  For $t\in[0,T]$ let $1_{[0,t]}\colon\mathbb{R}\to\mathbb{R}$ be the characteristic function of $[0,t]\subset\mathbb{R}$.
  Since $v^\varepsilon\in C([0,\infty);\mathcal{V}_g)$ we can substitute $\eta:=1_{[0,t]}v^\varepsilon$ for \eqref{E:PMu_Weak}.
  Then
  \begin{multline} \label{Pf_PMuE:Weak}
    \int_0^t\{(g\partial_sv^\varepsilon,v^\varepsilon)_{L^2(\Gamma)}+a_g(v^\varepsilon,v^\varepsilon)\}\,ds \\
    = \int_0^t(gM_\tau\mathbb{P}_\varepsilon f^\varepsilon,v^\varepsilon)_{L^2(\Gamma)}\,ds+R_\varepsilon^1(v^\varepsilon)+R_\varepsilon^2(v^\varepsilon)
  \end{multline}
  by \eqref{E:TriS_Vg} and $R_\varepsilon^1(v^\varepsilon)$ and $R_\varepsilon^2(v^\varepsilon)$ satisfy \eqref{E:Mu_Weak_Re} and \eqref{E:PMu_Weak_Re}.
  Let us calculate each term of \eqref{Pf_PMuE:Weak}.
  Since $g$ is independent of time and nonnegative by \eqref{E:Width_Bound},
  \begin{align} \label{Pf_PMuE:Dt}
    \begin{aligned}
      \int_0^t(g\partial_sv^\varepsilon,v^\varepsilon)_{L^2(\Gamma)}\,ds &= \frac{1}{2}\int_0^t\frac{d}{ds}\|g^{1/2}v^\varepsilon\|_{L^2(\Gamma)}^2\,ds \\
      &= \frac{1}{2}\|g^{1/2}v^\varepsilon(t)\|_{L^2(\Gamma)}^2-\frac{1}{2}\|g^{1/2}v^\varepsilon(0)\|_{L^2(\Gamma)}^2.
    \end{aligned}
  \end{align}
  Also, we see by \eqref{E:Bi_Surf} that
  \begin{align} \label{Pf_PMuE:A}
    \int_0^t\|\nabla_\Gamma v^\varepsilon\|_{L^2(\Gamma)}^2\,ds \leq c\int_0^t\left\{a_g(v^\varepsilon,v^\varepsilon)+\|v^\varepsilon\|_{L^2(\Gamma)}^2\right\}\,ds.
  \end{align}
  For the right-hand side of \eqref{Pf_PMuE:Weak}, we consider $M_\tau\mathbb{P}_\varepsilon f^\varepsilon=PM_\tau\mathbb{P}_\varepsilon f^\varepsilon$ as an element of $H^{-1}(\Gamma,T\Gamma)$ (see Section~\ref{SS:Pre_Surf}).
  Then we have
  \begin{align} \label{Pf_PMuE:F}
    \begin{aligned}
      \int_0^t(gM_\tau\mathbb{P}_\varepsilon f^\varepsilon,v^\varepsilon)_{L^2(\Gamma)}\,ds &= \int_0^t[M_\tau\mathbb{P}_\varepsilon f^\varepsilon,gv^\varepsilon]_{T\Gamma}\,ds \\
      &\leq \int_0^t\|M_\tau\mathbb{P}_\varepsilon f^\varepsilon\|_{H^{-1}(\Gamma,T\Gamma)}\|gv^\varepsilon\|_{H^1(\Gamma)}\,ds \\
      &\leq c\int_0^t\|M_\tau\mathbb{P}_\varepsilon f^\varepsilon\|_{H^{-1}(\Gamma,T\Gamma)}\|v^\varepsilon\|_{H^1(\Gamma)}\,ds.
    \end{aligned}
  \end{align}
  To estimate the residual terms, we see that $\varepsilon^{-1}\gamma_\varepsilon^0$ and $\varepsilon^{-1}\gamma_\varepsilon^1$ are bounded by \eqref{E:Fric_Upper}.
  Hence by \eqref{E:Mu_Weak_Re} and \eqref{E:PMu_Weak_Re} (with $T$ replaced by $t$) we obtain
  \begin{align} \label{Pf_PMuE:Re}
    |R_\varepsilon^1(v^\varepsilon)|+|R_\varepsilon^2(v^\varepsilon)| \leq c(1+t)^{1/2}\left(\int_0^t\|v^\varepsilon\|_{H^1(\Gamma)}^2\,ds\right)^{1/2}.
  \end{align}
  Now we deduce from \eqref{Pf_PMuE:Weak}--\eqref{Pf_PMuE:Re} that
  \begin{multline*}
    \|g^{1/2}v^\varepsilon(t)\|_{L^2(\Gamma)}^2+\int_0^t\|\nabla_\Gamma v^\varepsilon\|_{L^2(\Gamma)}^2\,ds \\
    \leq c\left\{\|g^{1/2}v^\varepsilon(0)\|_{L^2(\Gamma)}^2+\int_0^t\left(\|v^\varepsilon\|_{L^2(\Gamma)}^2+\|M_\tau\mathbb{P}_\varepsilon f^\varepsilon\|_{H^{-1}(\Gamma,T\Gamma)}\|v^\varepsilon\|_{H^1(\Gamma)}\right)ds\right\}\\
    +c(1+t)^{1/2}\left(\int_0^t\|v^\varepsilon\|_{H^1(\Gamma)}^2\,ds\right)^{1/2}.
  \end{multline*}
  Noting that $\|v^\varepsilon\|_{H^1(\Gamma)}^2=\|v^\varepsilon\|_{L^2(\Gamma)}^2+\|\nabla_\Gamma v^\varepsilon\|_{L^2(\Gamma)}^2$, we apply Young's inequality to the last two terms of the above inequality to get
  \begin{multline*}
    \|g^{1/2}v^\varepsilon(t)\|_{L^2(\Gamma)}^2+\int_0^t\|\nabla_\Gamma v^\varepsilon\|_{L^2(\Gamma)}^2\,ds \\
    \leq c\left\{\|g^{1/2}v^\varepsilon(0)\|_{L^2(\Gamma)}^2+\int_0^t\left(\|v^\varepsilon\|_{L^2(\Gamma)}^2+\|M_\tau\mathbb{P}_\varepsilon f^\varepsilon\|_{H^{-1}(\Gamma,T\Gamma)}^2\right)ds+1+t\right\}\\
    +\frac{1}{2}\int_0^t\|\nabla_\Gamma v^\varepsilon\|_{L^2(\Gamma)}^2\,ds.
  \end{multline*}
  Then we make the last term absorbed into the left-hand side and use the inequalities \eqref{E:UE_Data} with $\beta=1$, \eqref{E:Width_Bound}, and
  \begin{align*}
    \|g^{1/2}v^\varepsilon(0)\|_{L^2(\Gamma)} &\leq c\|v^\varepsilon(0)\|_{L^2(\Gamma)} \leq c\|M_\tau u^\varepsilon(0)\|_{L^2(\Gamma)} = c\|M_\tau u_0^\varepsilon\|_{L^2(\Gamma)}
  \end{align*}
  by \eqref{E:HLT_Bound} with $k=0$ (note that $v^\varepsilon=\mathbb{P}_gM_\tau u^\varepsilon$) and $u^\varepsilon(0)=u_0^\varepsilon$ in $\mathcal{V}_\varepsilon$ to obtain
  \begin{align} \label{Pf_PMuE:Gron}
    \|v^\varepsilon(t)\|_{L^2(\Gamma)}^2+\int_0^t\|\nabla_\Gamma v^\varepsilon\|_{L^2(\Gamma)}^2\,ds \leq c\left(1+t+\int_0^t\|v^\varepsilon\|_{L^2(\Gamma)}^2\,ds\right)
  \end{align}
  for all $t\in[0,T]$.
  From this inequality we deduce that
  \begin{align*}
    \|v^\varepsilon(t)\|_{L^2(\Gamma)}^2+1 \leq c\left\{1+\int_0^t\left(\|v^\varepsilon\|_{L^2(\Gamma)}^2+1\right)ds\right\}, \quad t\in[0,T]
  \end{align*}
  and thus Gronwall's inequality yields
  \begin{align*}
    \|v^\varepsilon(t)\|_{L^2(\Gamma)}^2+1 \leq ce^{ct} \leq ce^{cT}, \quad t\in[0,T].
  \end{align*}
  Applying this inequality to \eqref{Pf_PMuE:Gron} with $t=T$ we also get
  \begin{align*}
    \int_0^T\|\nabla_\Gamma v^\varepsilon\|_{L^2(\Gamma)}^2\,dt \leq c(1+T+e^{cT}).
  \end{align*}
  Hence we conclude that \eqref{E:PMu_Energy} holds with $c_T:=c(1+T+e^{cT})$, where $c>0$ is a constant independent of $\varepsilon$ and $T$.
\end{proof}

As a consequence of \eqref{E:Diff_PMu} and \eqref{E:PMu_Energy} we get the energy estimate for $M_\tau u^\varepsilon$.

\begin{corollary} \label{C:Mu_Energy}
  Let $u^\varepsilon$ be as in Lemma~\ref{L:Mu_Weak}.
  Then
  \begin{align} \label{E:Mu_Energy}
    \max_{t\in[0,T]}\|M_\tau u^\varepsilon(t)\|_{L^2(\Gamma)}^2+\int_0^T\|\nabla_\Gamma M_\tau u^\varepsilon(t)\|_{L^2(\Gamma)}^2\,dt \leq c_T
  \end{align}
  for all $T>0$, where $c_T>0$ is a constant depending on $T$ but independent of $\varepsilon$.
\end{corollary}

\subsection{Estimate for the time derivative of the average} \label{SS:SL_EDt}
By \eqref{E:Mu_Energy} we observe that (a subsequence of) $M_\tau u^\varepsilon$ converges weakly in an appropriate function space on $\Gamma$.
However, for the convergence of the trilinear term in \eqref{E:Mu_Weak} we also require the strong convergence of $M_\tau u^\varepsilon$.
In this subsection we estimate the time derivative of $M_\tau u^\varepsilon$ to apply the Aubin--Lions lemma for the strong convergence of $M_\tau u^\varepsilon$.
We first construct an appropriate test function.

\begin{lemma} \label{L:Mu_Dt_Test}
  For $w\in H^1(\Gamma,T\Gamma)$ there exist $\eta\in \mathcal{V}_g$, $q\in H^2(\Gamma)$, and a constant $c>0$ independent of $w$ such that $w=g\eta+g\nabla_\Gamma q$ on $\Gamma$ and
  \begin{align} \label{E:Mu_Dt_Test}
    \|\eta\|_{H^1(\Gamma)} \leq c\|w\|_{H^1(\Gamma)}.
  \end{align}
\end{lemma}

\begin{proof}
  Let $w\in H^1(\Gamma,T\Gamma)$ and $\xi:=-\mathrm{div}_\Gamma w\in L^2(\Gamma)$.
  Since $w$ is tangential on $\Gamma$, the integral of $\xi$ over $\Gamma$ vanishes by the Stokes theorem.
  Also,
  \begin{align} \label{Pf_MuDT:Poin}
    \|q\|_{L^2(\Gamma)} \leq c\|\nabla_\Gamma q\|_{L^2(\Gamma)} \leq c\|g^{1/2}\nabla_\Gamma q\|_{L^2(\Gamma)}
  \end{align}
  for all $q\in H^1(\Gamma)$ with $\int_\Gamma q\,d\mathcal{H}^2=0$ by Poincar\'{e}'s inequality \eqref{E:Poin_Surf_Lp} and \eqref{E:Width_Bound}.
  Hence the Lax--Milgram theorem shows that the problem
  \begin{align*}
    -\mathrm{div}_\Gamma(g\nabla_\Gamma q) = \xi \quad\text{on}\quad \Gamma, \quad \int_\Gamma q\,d\mathcal{H}^2 = 0
  \end{align*}
  admits a unique weak solution $q\in H^1(\Gamma)$ in the sense that
  \begin{align} \label{Pf_MuDT:Weak}
    (g\nabla_\Gamma q,\nabla_\Gamma\varphi)_{L^2(\Gamma)} = (\xi,\varphi)_{L^2(\Gamma)} \quad\text{for all}\quad \varphi\in H^1(\Gamma).
  \end{align}
  From this equality with $\varphi= q$ and \eqref{Pf_MuDT:Poin} we deduce that
  \begin{align} \label{Pf_MuDT:Q_H1}
    \|q\|_{H^1(\Gamma)} \leq c\|\xi\|_{L^2(\Gamma)} = c\|\mathrm{div}_\Gamma w\|_{L^2(\Gamma)} \leq c\|w\|_{H^1(\Gamma)}.
  \end{align}
  Moreover, replacing $\varphi$ by $g^{-1}\varphi$ in \eqref{Pf_MuDT:Weak} we get
  \begin{align*}
    (\nabla_\Gamma q,\nabla_\Gamma\varphi)_{L^2(\Gamma)} = (g^{-1}(\xi+\nabla_\Gamma g\cdot\nabla_\Gamma q),\varphi)_{L^2(\Gamma)} \quad\text{for all}\quad \varphi\in H^1(\Gamma),
  \end{align*}
  which combined with \eqref{E:Poin_Surf_Lp} shows that $q$ is a unique weak solution to
  \begin{align*}
    -\Delta_\Gamma\psi = \tilde{\xi} := g^{-1}(\xi+\nabla_\Gamma g\cdot\nabla_\Gamma q) \in L^2(\Gamma), \quad \int_\Gamma\psi\,d\mathcal{H}^2 = 0.
  \end{align*}
  (Note that the integral of $\tilde{\xi}$ over $\Gamma$ vanishes by \eqref{Pf_MuDT:Weak}.)
  Hence by Lemma~\ref{L:Pois_Surf} and the inequalities \eqref{E:Width_Bound} and \eqref{Pf_MuDT:Q_H1} we see that $q\in H^2(\Gamma)$ and
  \begin{align} \label{Pf_MuDT:Q_H2}
    \|q\|_{H^2(\Gamma)} \leq c\|g^{-1}(\xi+\nabla_\Gamma g\cdot\nabla_\Gamma q)\|_{L^2(\Gamma)} \leq c\|w\|_{H^1(\Gamma)}.
  \end{align}
  Now we set $\eta:=g^{-1}w-\nabla_\Gamma q$ on $\Gamma$.
  Then $\eta\in \mathcal{V}_g$ by $q\in H^2(\Gamma)$ and $\mathrm{div}_\Gamma(g\nabla_\Gamma q)=-\xi=\mathrm{div}_\Gamma w$ on $\Gamma$.
  Moreover, from \eqref{E:Width_Bound} and \eqref{Pf_MuDT:Q_H2} it follows that
  \begin{align*}
    \|\eta\|_{H^1(\Gamma)} \leq c\left(\|w\|_{H^1(\Gamma)}+\|\nabla_\Gamma q\|_{H^1(\Gamma)}\right) \leq c\|w\|_{H^1(\Gamma)}.
  \end{align*}
  Hence we obtain $w=g\eta+g\nabla_\Gamma q$ on $\Gamma$ and \eqref{E:Mu_Dt_Test}.
\end{proof}

As in the previous subsection, we estimate the time derivative of $v^\varepsilon$ and then derive an estimate for the time derivative of $M_\tau u^\varepsilon$ by using a difference estimate.

\begin{lemma} \label{L:PMu_Dt}
  Let $u^\varepsilon$ be as in Lemma~\ref{L:Mu_Weak} and $v^\varepsilon=\mathbb{P}_gM_\tau u^\varepsilon$.
  Then
  \begin{align} \label{E:PMu_Dt}
    \|\partial_tv^\varepsilon\|_{L^2(0,T;H^{-1}(\Gamma,T\Gamma))} \leq c_T
  \end{align}
  for all $T>0$, where $c_T>0$ is a constant depending on $T$ but independent of $\varepsilon$.
\end{lemma}

Note that we estimate $\partial_tv^\varepsilon$ in the dual $H^{-1}(\Gamma,T\Gamma)$ of $H^1(\Gamma,T\Gamma)$, not in the dual of the weighted solenoidal space $\mathcal{V}_g$ (see Remark~\ref{R:Mu_Dt} below).

\begin{proof}
  Let $w\in L^2(0,T;H^1(\Gamma,T\Gamma))$.
  By Lemma~\ref{L:Mu_Dt_Test} we have $w=g\eta+g\nabla_\Gamma q$ with $\eta\in L^2(0,T;\mathcal{V}_g)$ and $q\in L^2(0,T;H^2(\Gamma))$.
  Since $\partial_tv^\varepsilon(t)\in L_{g\sigma}^2(\Gamma,T\Gamma)$ and $g\nabla_\Gamma q(t)\in L_{g\sigma}^2(\Gamma,T\Gamma)^\perp$ for a.a. $t\in(0,T)$ by Lemmas~\ref{L:L2gs_Orth} and~\ref{L:PMu_Weak},
  \begin{align*}
    \int_0^T(\partial_tv^\varepsilon,g\nabla_\Gamma q)_{L^2(\Gamma)}\,dt = 0.
  \end{align*}
  By this equality and $g\eta=w-g\nabla_\Gamma q$ we have
  \begin{align*}
    \int_0^T(g\partial_tv^\varepsilon,\eta)_{L^2(\Gamma)}\,dt = \int_0^T(\partial_tv^\varepsilon,g\eta)_{L^2(\Gamma)}\,dt = \int_0^T(\partial_tv^\varepsilon,w)_{L^2(\Gamma)}\,dt.
  \end{align*}
  We substitute $\eta=g^{-1}w-\nabla_\Gamma q$ for \eqref{E:PMu_Weak} and use this equality.
  Then
  \begin{multline} \label{Pf_PMuDt:Eq}
    \int_0^T(\partial_tv^\varepsilon,w)_{L^2(\Gamma)}\,dt = -\int_0^Ta_g(v^\varepsilon,\eta)\,dt-\int_0^Tb_g(v^\varepsilon,v^\varepsilon,\eta)\,dt \\
    +\int_0^T(gM_\tau\mathbb{P}_\varepsilon f^\varepsilon,\eta)_{L^2(\Gamma)}\,dt+R_\varepsilon^1(\eta)+R_\varepsilon^2(\eta),
  \end{multline}
  where $R_\varepsilon^1(\eta)$ and $R_\varepsilon^2(\eta)$ are given in Lemmas~\ref{L:Mu_Weak} and~\ref{L:PMu_Weak}.
  To the first term on the right-hand side we apply \eqref{E:Width_Bound}, \eqref{E:PMu_Energy}, and \eqref{E:Mu_Dt_Test} to get
  \begin{align*}
    \left|\int_0^Ta_g(v^\varepsilon,\eta)\,dt\right| \leq c\|v^\varepsilon\|_{L^2(0,T;H^1(\Gamma))}\|\eta\|_{L^2(0,T;H^1(\Gamma))} \leq c_T\|w\|_{L^2(0,T;H^1(\Gamma))}.
  \end{align*}
  Here and in what follows we denote by $c_T$ a general positive constant depending on $T$ but independent of $\varepsilon$.
  Also, by \eqref{E:Tri_Surf}, \eqref{E:PMu_Energy}, and \eqref{E:Mu_Dt_Test},
  \begin{align*}
    \left|\int_0^Tb_g(v^\varepsilon,v^\varepsilon,\eta)\,dt\right| &\leq c\int_0^T\|v^\varepsilon\|_{L^2(\Gamma)}\|v^\varepsilon\|_{H^1(\Gamma)}\|\eta\|_{H^1(\Gamma)}\,dt \\
    &\leq c\|v^\varepsilon\|_{L^\infty(0,T;L^2(\Gamma))}\|v^\varepsilon\|_{L^2(0,T;H^1(\Gamma))}\|\eta\|_{L^2(0,T;H^1(\Gamma))} \\
    &\leq c_T\|w\|_{L^2(0,T;H^1(\Gamma))}.
  \end{align*}
  For the other terms we proceed as in the proof of Lemma~\ref{L:PMu_Energy} (see \eqref{Pf_PMuE:F}--\eqref{Pf_PMuE:Re}) and use \eqref{E:UE_Data} with $\beta=1$ and \eqref{E:Mu_Dt_Test}.
  Then we get
  \begin{align*}
    \left|\int_0^T(gM_\tau\mathbb{P}_\varepsilon f^\varepsilon,\eta)_{L^2(\Gamma)}\,dt\right| &\leq c\int_0^T\|M_\tau\mathbb{P}_\varepsilon f^\varepsilon\|_{H^{-1}(\Gamma,T\Gamma)}\|\eta\|_{H^1(\Gamma)}\,dt \\
    &\leq cT^{1/2}\|\eta\|_{L^2(0,T;H^1(\Gamma))} \leq cT^{1/2}\|w\|_{L^2(0,T;H^1(\Gamma))}
  \end{align*}
  and
  \begin{align*}
    |R_\varepsilon^1(\eta)|+|R_\varepsilon^2(\eta)| \leq c(1+T)^{1/2}\|\eta\|_{L^2(0,T;H^1(\Gamma))} \leq c(1+T)^{1/2}\|w\|_{L^2(0,T;H^1(\Gamma))}.
  \end{align*}
  Applying these inequalities to the right-hand side of \eqref{Pf_PMuDt:Eq} we obtain
  \begin{align*}
    \left|\int_0^T(\partial_tv^\varepsilon,w)_{L^2(\Gamma)}\,dt\right| \leq c_T\|w\|_{L^2(0,T;H^1(\Gamma))}
  \end{align*}
  for all $w\in L^2(0,T;H^1(\Gamma,T\Gamma))$.
  Hence \eqref{E:PMu_Dt} holds.
\end{proof}

\begin{corollary} \label{C:Mu_Dt}
  Let $u^\varepsilon$ be as in Lemma~\ref{L:Mu_Weak}.
  Then
  \begin{align} \label{E:Mu_Dt}
    \|\partial_tM_\tau u^\varepsilon\|_{L^2(0,T;H^{-1}(\Gamma,T\Gamma))} \leq c_T
  \end{align}
  for all $T>0$, where $c_T>0$ is a constant depending on $T$ but independent of $\varepsilon$.
\end{corollary}

\begin{proof}
  Let $v^\varepsilon=\mathbb{P}_gM_\tau u^\varepsilon$.
  From \eqref{E:Est_DtUe} and \eqref{E:Rep_WHL_Dt} we deduce that
  \begin{align*}
    \|\partial_tM_\tau u^\varepsilon-\partial_tv^\varepsilon\|_{L^2(0,T;H^{-1}(\Gamma,T\Gamma))} &\leq \|\partial_tM_\tau u^\varepsilon-\partial_tv^\varepsilon\|_{L^2(0,T;L^2(\Gamma))} \\
    &\leq \varepsilon^{1/2}\|\partial_tu^\varepsilon\|_{L^2(0,T;L^2(\Omega_\varepsilon))} \\
    &\leq c\varepsilon^{\alpha/2}(1+T)^{1/2}.
  \end{align*}
  By this inequality, \eqref{E:PMu_Dt}, and $\varepsilon^{\alpha/2}\leq 1$ we obtain \eqref{E:Mu_Dt}.
\end{proof}

\begin{remark} \label{R:Mu_Dt}
  In construction of a weak solution to the Navier--Stokes equations, we usually estimate the time derivative of an approximate solution in the dual of a solenoidal space.
  However, in Lemma~\ref{L:PMu_Dt} we estimate $\partial_tv^\varepsilon$ in $H^{-1}(\Gamma,T\Gamma)$, not in the dual $\mathcal{V}'_g$ of $\mathcal{V}_g$.
  This is because we multiply $\partial_tv^\varepsilon$ by $g$ in \eqref{E:PMu_Weak}.
  When $f\in\mathcal{V}'_g$, we cannot define a functional $gf\colon v\mapsto{}_{\mathcal{V}'_g}\langle f,gv \rangle_{\mathcal{V}_g}$ for $v\in \mathcal{V}_g$ since $gv$ is not in $\mathcal{V}_g$ in general (here ${}_{\mathcal{V}'_g}\langle \cdot,\cdot \rangle_{\mathcal{V}_g}$ stands for the duality product between $\mathcal{V}'_g$ and $\mathcal{V}_g$).
  To avoid this issue, we consider $\partial_tv^\varepsilon$ and $\partial_tM_\tau u^\varepsilon$ in $H^{-1}(\Gamma,T\Gamma)$ (see \eqref{E:Def_Mul_HinT}).
\end{remark}

\subsection{Weak convergence of the average and characterization of the limit} \label{SS:SL_WeCh}
The goal of this subsection is to prove Theorem~\ref{T:SL_Weak}.
We proceed as in the case of a bounded domain in $\mathbb{R}^2$ (see e.g.~\cite{BoFa13,CoFo88,So01,Te79}).
First we give the definition of a weak solution to the limit equations \eqref{E:Limit_Eq}--\eqref{E:Limit_Div} based on \eqref{E:Mu_Weak}.

\begin{definition} \label{D:Lim_W_T}
  Let $T>0$, $v_0\in L_{g\sigma}^2(\Gamma,T\Gamma)$, and $f\in L^2(0,T;H^{-1}(\Gamma,T\Gamma))$.
  We say that a vector field
  \begin{align*}
    v \in L^\infty(0,T;L_{g\sigma}^2(\Gamma,T\Gamma))\cap L^2(0,T;\mathcal{V}_g) \quad\text{with}\quad \partial_t v\in L^1(0,T;H^{-1}(\Gamma,T\Gamma))
  \end{align*}
  is a weak solution to the equations \eqref{E:Limit_Eq}--\eqref{E:Limit_Div} on $[0,T)$ if it satisfies
  \begin{align} \label{E:Limit_Weak}
    \int_0^T\{[g\partial_tv,\eta]_{T\Gamma}+a_g(v,\eta)+b_g(v,v,\eta)\}\,dt = \int_0^T[gf,\eta]_{T\Gamma}\,dt
  \end{align}
  for all $\eta\in C_c(0,T;\mathcal{V}_g)$ and $v|_{t=0}=v_0$ in $H^{-1}(\Gamma,T\Gamma)$.
\end{definition}

\begin{definition} \label{D:Lim_W_Inf}
  Let $v_0\in L_{g\sigma}^2(\Gamma,T\Gamma)$ and $f\in L_{loc}^2([0,\infty);H^{-1}(\Gamma,T\Gamma))$.
  We say that $v$ is a weak solution to \eqref{E:Limit_Eq}--\eqref{E:Limit_Div} on $[0,\infty)$ if it is a weak solution to \eqref{E:Limit_Eq}--\eqref{E:Limit_Div} on $[0,T)$ for all $T>0$.
\end{definition}

For $T>0$, a weak solution to \eqref{E:Limit_Eq}--\eqref{E:Limit_Div} on $[0,T)$ is continuous on $[0,T]$ with values in $H^{-1}(\Gamma,T\Gamma)$ and thus the initial condition makes sense.
In fact, it becomes a continuous function with values in $L^2(\Gamma,T\Gamma)$.

\begin{lemma} \label{L:Lim_W_L2}
  Let $T>0$ and $f\in L^2(0,T;H^{-1}(\Gamma,T\Gamma))$.
  Suppose that
  \begin{align*}
    v \in L^\infty(0,T;L_{g\sigma}^2(\Gamma,T\Gamma))\cap L^2(0,T;\mathcal{V}_g)\quad\text{with}\quad \partial_tv\in L^1(0,T;H^{-1}(\Gamma,T\Gamma))
  \end{align*}
  satisfies \eqref{E:Limit_Weak} for all $\eta\in C_c(0,T;\mathcal{V}_g)$.
  Then
  \begin{align*}
    v \in C([0,T];L_{g\sigma}^2(\Gamma,T\Gamma)), \quad \partial_tv\in L^2(0,T;H^{-1}(\Gamma,T\Gamma)),
  \end{align*}
  and \eqref{E:Limit_Weak} is valid for all $\eta\in L^2(0,T;\mathcal{V}_g)$.
\end{lemma}

Note that here the initial condition $v|_{t=0}=v_0$ in $H^{-1}(\Gamma,T\Gamma)$ is not imposed.

\begin{proof}
  We estimate $\partial_tv$ as in the proof of Lemma~\ref{L:PMu_Dt}, where we used $\partial_tv^\varepsilon(t)\in L_{g\sigma}^2(\Gamma,T\Gamma)$ for a.a. $t\in(0,T)$.
  This is not valid for $\partial_tv$, but we have
  \begin{align} \label{Pf_LWL2:Anni}
    [\partial_tv(t),g\nabla_\Gamma q]_{T\Gamma} = 0 \quad\text{for all $q\in H^2(\Gamma)$ and a.a. $t\in(0,T)$}.
  \end{align}
  Indeed, for all $\xi\in C_c^\infty(0,T)$,
  \begin{align*}
    \int_0^T\xi(t)[\partial_tv(t),g\nabla_\Gamma q]_{T\Gamma}\,dt &= -\int_0^T\partial_t\xi(t)[v(t),g\nabla_\Gamma q]_{T\Gamma}\,dt \\
    &= -\int_0^T\partial_t\xi(t)(v(t),g\nabla_\Gamma q)_{L^2(\Gamma)}\,dt = 0
  \end{align*}
  by $v(t)\in L_{g\sigma}^2(\Gamma,T\Gamma)$ for a.a. $t\in(0,T)$ and $g\nabla_\Gamma q\in L_{g\sigma}^2(\Gamma,T\Gamma)^\perp$ (see Lemma~\ref{L:L2gs_Orth}).
  Hence \eqref{Pf_LWL2:Anni} is valid.
  Now let $w\in C_c(0,T;H^1(\Gamma,T\Gamma))$.
  By Lemma~\ref{L:Mu_Dt_Test} we can take $\eta\in C_c(0,T;\mathcal{V}_g)$ and $q\in C_c(0,T;H^2(\Gamma))$ such that $w=g\eta+g\nabla_\Gamma q$.
  Moreover,
  \begin{align*}
    \int_0^T[\partial_tv,w]_{T\Gamma}\,dt = \int_0^T\bigl([\partial_tv,g\eta]_{T\Gamma}+[\partial_tv,g\nabla_\Gamma q]_{T\Gamma}\bigr)\,dt = \int_0^T[g\partial_tv,\eta]_{T\Gamma}\,dt
  \end{align*}
  by \eqref{Pf_LWL2:Anni}.
  We substitute $\eta$ for \eqref{E:Limit_Weak}.
  Then using the above equality,
  \begin{align} \label{Pf_LWL2:B}
    \left|\int_0^Tb_g(v,v,\eta)\,dt\right| \leq c\|v\|_{L^\infty(0,T;L^2(\Gamma))}\|v\|_{L^2(0,T;H^1(\Gamma))}\|\eta\|_{L^2(0,T;H^1(\Gamma))}
  \end{align}
  by \eqref{E:Tri_Surf}, and \eqref{E:Mu_Dt_Test} we calculate as in the proof of Lemma~\ref{L:PMu_Dt} to get
  \begin{align*}
    \left|\int_0^T[\partial_tv,w]_{T\Gamma}\,dt\right| \leq c\|w\|_{L^2(0,T;H^1(\Gamma))} \quad\text{for all}\quad w\in C_c(0,T;H^1(\Gamma,T\Gamma)).
  \end{align*}
  Since $C_c(0,T;H^1(\Gamma,T\Gamma))$ is dense in $L^2(0,T;H^1(\Gamma,T\Gamma))$, this inequality implies
  \begin{align} \label{Pf_LWL2:Dt}
    \partial_tv\in L^2(0,T;H^{-1}(\Gamma,T\Gamma)).
  \end{align}
  Combining this property with $v \in L^2(0,T;\mathcal{V}_g) \subset L^2(0,T;H^1(\Gamma,T\Gamma))$ we apply the interpolation result of Lions--Magenes~\cite[Chapter~1, Theorem~3.1]{LiMa72} (see also~\cite[Chapter~III, Lemma~1.2]{Te79}) to $v$ to obtain $v \in C([0,T];L^2(\Gamma,T\Gamma))$.
  Moreover, since $v\in L^\infty(0,T;L_{g\sigma}^2(\Gamma,T\Gamma))$, the vector field $v(t)$ is in $L_{g\sigma}^2(\Gamma,T\Gamma)$ for a.a. $t\in(0,T)$ and, in particular, for all $t$ in a dense subset of $[0,T]$.
  Hence, by the continuity of $v(t)$ on $[0,T]$ in $L^2(\Gamma,T\Gamma)$ and the fact that $L_{g\sigma}^2(\Gamma,T\Gamma)$ is closed in $L^2(\Gamma,T\Gamma)$, we get $v(t)\in L_{g\sigma}^2(\Gamma,T\Gamma)$ for all $t\in[0,T]$ and thus $v\in C([0,T];L_{g\sigma}^2(\Gamma,T\Gamma))$.

  Finally, since $C_c(0,T;\mathcal{V}_g)$ is dense in $L^2(0,T;\mathcal{V}_g)$ and both sides of \eqref{E:Limit_Weak} are linear and continuous for $\eta\in L^2(0,T;\mathcal{V}_g)$ by \eqref{Pf_LWL2:B} and \eqref{Pf_LWL2:Dt}, the equality \eqref{E:Limit_Weak} is also valid for all $\eta\in L^2(0,T;\mathcal{V}_g)$.
\end{proof}

By Lemma~\ref{L:Lim_W_L2} the initial condition for a weak solution to \eqref{E:Limit_Eq}--\eqref{E:Limit_Div} makes sense in $L^2(\Gamma,T\Gamma)$.
Let us prove the uniqueness of a weak solution.

\begin{lemma} \label{L:Limit_Uni}
  For given $v_0\in L_{g\sigma}^2(\Gamma,T\Gamma)$ and $f\in L^2(0,T;H^{-1}(\Gamma,T\Gamma))$, $T>0$ there exists at most one weak solution to \eqref{E:Limit_Eq}--\eqref{E:Limit_Div} on $[0,T)$.
\end{lemma}

\begin{proof}
  For weak solutions $v_1$ and $v_2$ to \eqref{E:Limit_Eq}--\eqref{E:Limit_Div} let $w:=v_1-v_2$.
  Then
  \begin{align} \label{Pf_LU:Class}
    w\in C([0,T];L_{g\sigma}^2(\Gamma,T\Gamma))\cap L^2(0,T;\mathcal{V}_g), \quad \partial_tw \in L^2(0,T;H^{-1}(\Gamma,T\Gamma))
  \end{align}
  and $w|_{t=0}=0$ in $L^2(\Gamma,T\Gamma)$ by Lemma~\ref{L:Lim_W_L2}.
  For $\eta\in L^2(0,T;\mathcal{V}_g)$ we subtract the weak formulation \eqref{E:Limit_Weak} for $v_2$ from that for $v_1$ to get
  \begin{align} \label{Pf_LU:Weak}
    \int_0^T\{[g\partial_tw,\eta]_{T\Gamma}+a_g(w,\eta)+b_g(w,v_1,\eta)+b_g(v_2,w,\eta)\}\,ds = 0.
  \end{align}
  For $t\in[0,T]$ let $1_{[0,t]}$ be the characteristic function of $[0,t]\subset\mathbb{R}$ and $\eta:=1_{[0,t]}w$.
  Since $\eta\in L^2(0,T;\mathcal{V}_g)$ we can substitute it for \eqref{Pf_LU:Weak}.
  Then we use \eqref{E:Bi_Surf},
  \begin{align} \label{Pf_LU:Int_Dt}
    \begin{aligned}
      \int_0^t[g\partial_sw,w]_{T\Gamma}\,ds &= \frac{1}{2}\int_0^t\frac{d}{ds}\|g^{1/2}w\|_{L^2(\Gamma)}^2\,ds \\
      &= \frac{1}{2}\|g^{1/2}w(t)\|_{L^2(\Gamma)}^2-\frac{1}{2}\|g^{1/2}w(0)\|_{L^2(\Gamma)}^2 \\
      &\geq c\left(\|w(t)\|_{L^2(\Gamma)}^2-\|w(0)\|_{L^2(\Gamma)}^2\right)
    \end{aligned}
  \end{align}
  by \eqref{Pf_LU:Class} and $c^{-1}\leq g\leq c$ on $\Gamma$ with a constant $c>0$ (see Section~\ref{SS:Pre_Dom}), and
  \begin{align*}
    |b_g(w,v_1,w)| = |b_g(w,w,v_1)| \leq c\|w\|_{L^2(\Gamma)}\|w\|_{H^1(\Gamma)}\|v_1\|_{H^1(\Gamma)}
  \end{align*}
  and $b_g(v_2,w,w)=0$ by $v_2,w\in\mathcal{V}_g$ and \eqref{E:Tri_Surf}--\eqref{E:TriS_Vg} to obtain
  \begin{multline*}
    \|w(t)\|_{L^2(\Gamma)}^2+\int_0^t\|\nabla_\Gamma w\|_{L^2(\Gamma)}^2\,ds \\
    \leq c\left\{\|w(0)\|_{L^2(\Gamma)}^2+\int_0^t\left(\|w\|_{L^2(\Gamma)}^2+\|w\|_{L^2(\Gamma)}\|w\|_{H^1(\Gamma)}\|v_1\|_{H^1(\Gamma)}\right)ds\right\}.
  \end{multline*}
  We further apply Young's inequality to the last term to get
  \begin{multline*}
    \|w(t)\|_{L^2(\Gamma)}^2+\int_0^t\|\nabla_\Gamma w\|_{L^2(\Gamma)}^2\,ds \\
    \leq c\left\{\|w(0)\|_{L^2(\Gamma)}^2+\int_0^t\left(1+\|v_1\|_{H^1(\Gamma)}^2\right)\|w\|_{L^2(\Gamma)}^2\,ds\right\}+\frac{1}{2}\int_0^t\|\nabla_\Gamma w\|_{L^2(\Gamma)}^2\,ds.
  \end{multline*}
  Then we make the last term absorbed into the left-hand side and use $w|_{t=0}=0$ in $L^2(\Gamma,T\Gamma)$ to find that (we omit the integral of $\|\nabla_\Gamma w\|_{L^2(\Gamma)}^2$ on the left-hand side)
  \begin{align*}
    \|w(t)\|_{L^2(\Gamma)}^2 \leq c\int_0^t\left(1+\|v_1\|_{H^1(\Gamma)}^2\right)\|w\|_{L^2(\Gamma)}^2\,ds \quad\text{for all}\quad t\in[0,T].
  \end{align*}
  Since $1+\|v_1\|_{H^1(\Gamma)}^2$ is integrable on $(0,T)$, we can apply Gronwall's inequality to this inequality to get $\|w(t)\|_{L^2(\Gamma)}^2=0$ for all $t\in[0,T]$.
  Hence $v_1=v_2$.
\end{proof}

Next we show the existence of an associated pressure after giving two auxiliary results.
Recall that we identity $H^{-1}(\Gamma,T\Gamma)$ with the quotient space
\begin{align*}
  \mathcal{Q} = \{[f]\mid f\in H^{-1}(\Gamma)^3\}, \quad [f] = \{\tilde{f}\in H^{-1}(\Gamma)^3 \mid \text{$Pf=P\tilde{f}$ in $H^{-1}(\Gamma)^3$}\}
\end{align*}
and take $Pf$ (or $f$ when $Pf=f$ in $H^{-1}(\Gamma)^3$) as a representative of the equivalence class $[f]$ to write $[Pf,v]_{T\Gamma}=\langle f,v\rangle_\Gamma$ for $v\in H^1(\Gamma,T\Gamma)$ (see Section~\ref{SS:Pre_Surf}).

\begin{lemma} \label{L:LWP_Aux_1}
  Let $A\in L^2(\Gamma)^{3\times3}$ satisfy $A^T=PA=AP=A$ on $\Gamma$.
  Then
  \begin{align} \label{E:LWP_Aux_1}
    \bigl(A,D_\Gamma(\eta)\bigr)_{L^2(\Gamma)} = -[P\mathrm{div}_\Gamma A,\eta]_{T\Gamma}
  \end{align}
  for all $\eta\in H^1(\Gamma,T\Gamma)$.
\end{lemma}

\begin{proof}
  The assumption on $A$ yields $A:D_\Gamma(\eta)=A:\nabla_\Gamma\eta$ and $A^Tn=0$ on $\Gamma$.
  Using these equalities and \eqref{E:Def_TD_Hin} and noting that $\eta\in H^1(\Gamma,T\Gamma)$ we get
  \begin{align*}
    \bigl(A,D_\Gamma(\eta)\bigr)_{L^2(\Gamma)} &= (A,\nabla_\Gamma\eta)_{L^2(\Gamma)} = \sum_{i,j=1}^3(A_{ij},\underline{D}_i\eta_j)_{L^2(\Gamma)} \\
    &= -\sum_{i,j=1}^3\left\{\langle \underline{D}_iA_{ij}, \eta_j\rangle_\Gamma+(A_{ij}Hn_i,\eta_j)_{L^2(\Gamma)}\right\} \\
    &= -\left\{\langle \mathrm{div}_\Gamma A,\eta\rangle_\Gamma+(HA^Tn,\eta)_{L^2(\Gamma)}\right\} = -[P\mathrm{div}_\Gamma A,\eta]_{T\Gamma}.
  \end{align*}
  Hence \eqref{E:LWP_Aux_1} holds.
\end{proof}

\begin{lemma} \label{L:LWP_Aux_2}
  Let $v\in H^1(\Gamma,T\Gamma)$.
  Then
  \begin{align*}
    \overline{\nabla}_vv := P(v\cdot\nabla_\Gamma)v \in L^{4/3}(\Gamma,T\Gamma)
  \end{align*}
  and we can consider $\overline{\nabla}_vv$ as an element of $H^{-1}(\Gamma,T\Gamma)$ by
  \begin{align*}
    \Bigl[\overline{\nabla}_vv,\eta\Bigr]_{T\Gamma} := \int_\Gamma\overline{\nabla}_vv\cdot\eta\,d\mathcal{H}^2, \quad \eta\in H^1(\Gamma,T\Gamma).
  \end{align*}
  Moreover, if $v\in\mathcal{V}_g$ and $\eta\in H^1(\Gamma,T\Gamma)$, then
  \begin{align} \label{E:LWP_Aux_2}
    b_g(v,v,\eta) = \int_\Gamma g\overline{\nabla}_vv\cdot\eta\,d\mathcal{H}^2 = \Bigl[g\overline{\nabla}_vv,\eta\Bigr]_{T\Gamma}.
  \end{align}
\end{lemma}

Note that $\overline{\nabla}_vv$ is the covariant derivative of $v$ along itself (see Appendix~\ref{S:Ap_RC}).

\begin{proof}
  Let $v\in H^1(\Gamma,T\Gamma)$.
  By H\"{o}lder's inequality and \eqref{E:La_Surf},
  \begin{align*}
    \left\|\overline{\nabla}_vv\right\|_{L^{4/3}(\Gamma)} \leq \|v\|_{L^4(\Gamma)}\|\nabla_\Gamma v\|_{L^2(\Gamma)} \leq c\|v\|_{L^2(\Gamma)}^{1/2}\|v\|_{H^1(\Gamma)}^{3/2} =: c_v.
  \end{align*}
  Hence $\overline{\nabla}_vv\in L^{4/3}(\Gamma,T\Gamma)$.
  We again use H\"{o}lder's inequality and \eqref{E:La_Surf} to get
  \begin{align*}
    \left|\int_\Gamma \overline{\nabla}_vv\cdot\eta\,dx\right| \leq \left\|\overline{\nabla}_vv\right\|_{L^{4/3}(\Gamma)}\|\eta\|_{L^4(\Gamma)} \leq c_v\|\eta\|_{H^1(\Gamma)}
  \end{align*}
  for all $\eta\in H^1(\Gamma,T\Gamma)$, which shows $\overline{\nabla}_vv\in H^{-1}(\Gamma,T\Gamma)$.
  Also,
  \begin{align*}
    b_g(v,v,\eta) = -b_g(v,\eta,v) = \int_\Gamma g(v\otimes\eta):\nabla_\Gamma v\,d\mathcal{H}^2
  \end{align*}
  for $v\in\mathcal{V}_g$ and $\eta\in H^1(\Gamma,T\Gamma)$ by \eqref{E:TriS_Vg}.
  To the last term we apply
  \begin{align*}
    v\otimes\eta:\nabla_\Gamma v = (v\cdot\nabla_\Gamma)v\cdot\eta = [P(v\cdot\nabla_\Gamma)v]\cdot\eta = \overline{\nabla}_vv\cdot\eta \quad\text{on}\quad \Gamma
  \end{align*}
  by $\eta\cdot n=0$ on $\Gamma$ to obtain \eqref{E:LWP_Aux_2}.
\end{proof}

\begin{lemma} \label{L:LW_Pres}
  For a weak solution $v$ to \eqref{E:Limit_Eq}--\eqref{E:Limit_Div} on $[0,T)$, $T>0$ there exists a unique $\hat{q}\in C([0,T];L^2(\Gamma))$ such that $\int_\Gamma\hat{q}(t)\,d\mathcal{H}^2=0$ for all $t\in[0,T]$ and
  \begin{multline} \label{E:LW_Pres}
    g\Bigl(\partial_tv+\overline{\nabla}_vv\Bigr)-2\nu\left\{P\,\mathrm{div}_\Gamma[gD_\Gamma(v)]-\frac{1}{g}(\nabla_\Gamma g\otimes\nabla_\Gamma g)v\right\} \\
    +(\gamma^0+\gamma^1)v+g\nabla_\Gamma q = gf \quad\text{in}\quad \mathcal{D}'(0,T;H^{-1}(\Gamma,T\Gamma))
  \end{multline}
  with $q:=\partial_t\hat{q}\in\mathcal{D}'(0,T;L^2(\Gamma))$ (see Section~\ref{SS:Pre_Surf}).
\end{lemma}

\begin{proof}
  Let $v$ be a weak solution to \eqref{E:Limit_Eq}--\eqref{E:Limit_Div} on $[0,T)$.
  Then
  \begin{align} \label{Pf_LWP:Reg_V}
    v \in H^1(0,T;H^{-1}(\Gamma,T\Gamma))\subset C([0,T];H^{-1}(\Gamma,T\Gamma))
  \end{align}
  by Lemma~\ref{L:Lim_W_L2}.
  Also, since the given data $f$, $B_g(v,v):=g\overline{\nabla}_vv$, and
  \begin{align} \label{Pf_LWP:Ag}
    A_gv := -2\nu\left\{P\mathrm{div}_\Gamma[gD_\Gamma(v)]-\frac{1}{g}(\nabla_\Gamma g\otimes\nabla_\Gamma g)v\right\}+(\gamma^0+\gamma^1)v
  \end{align}
  belong to $L^2(0,T;H^{-1}(\Gamma,T\Gamma))$ by Definition~\ref{D:Lim_W_T} and Lemma~\ref{L:LWP_Aux_2}, we see that
    \begin{gather*}
    \widehat{A}_gv(t) := \int_0^tA_gv(s)\,ds, \quad \widehat{B}_g(v,v)(t) := \int_0^tB_g(v(s),v(s))\,ds, \\
    \hat{f}(t) := \int_0^tf(s)\,ds, \quad t\in[0,T]
  \end{gather*}
  are continuous functions on $[0,T]$ with values in $H^{-1}(\Gamma,T\Gamma)$.
  Therefore,
  \begin{align*}
    F := gv-gv_0+\widehat{A}_gv+\widehat{B}_g(v,v)-g\hat{f} \in C([0,T];H^{-1}(\Gamma,T\Gamma)).
  \end{align*}
  Let us show that $F$ satisfies the condition of Theorem~\ref{T:DeRham_T}.
  We apply \eqref{E:LWP_Aux_1} with $A=gD_\Gamma(v)$, \eqref{E:LWP_Aux_2}, and $(v\cdot\nabla_\Gamma g)\nabla_\Gamma g=(\nabla_\Gamma g\otimes\nabla_\Gamma g)v$ on $\Gamma$ to \eqref{E:Limit_Weak} to get
  \begin{align*}
    \int_0^T\bigl([g\partial_tv,\eta]_{T\Gamma}+[A_gv,\eta]_{T\Gamma}+[B_g(v,v),\eta]_{T\Gamma}\bigr)\,ds = \int_0^T[gf,\eta]_{T\Gamma}\,ds
  \end{align*}
  for all $\eta\in L^2(0,T;\mathcal{V}_g)$ (see Lemma~\ref{L:Lim_W_L2}).
  For $t\in[0,T]$ and $\xi\in\mathcal{V}_g$ we substitute $\eta(s):=1_{[0,t]}(s)\xi$, $s\in[0,T]$ for this equality, where $1_{[0,t]}$ is the characteristic function of $[0,t]\subset\mathbb{R}$.
  Then since $\xi$ is independent of time and
  \begin{align*}
    \int_0^t\partial_sv(s)\,ds = v(t)-v_0 \quad\text{in}\quad H^{-1}(\Gamma,T\Gamma)
  \end{align*}
  by \eqref{Pf_LWP:Reg_V}, we get $[F(t),\xi]_{T\Gamma}=0$ for all $\xi\in\mathcal{V}_g$, i.e. $F(t)$ satisfies the condition of Theorem~\ref{T:DeRham_T}.
  Hence for each $t\in[0,T]$ there exists a unique $\hat{q}(t)\in L^2(\Gamma)$ such that
  \begin{align*}
    F(t) = -g\nabla_\Gamma\hat{q}(t) \quad\text{in}\quad H^{-1}(\Gamma,T\Gamma), \quad \int_\Gamma\hat{q}(t)\,d\mathcal{H}^2 = 0.
  \end{align*}
  Moreover, by \eqref{E:DeRham_T_Ineq} and $F\in C([0,T];H^{-1}(\Gamma,T\Gamma))$ we see that
  \begin{align*}
    \hat{q} \in C([0,T];L^2(\Gamma)) \subset L^2(0,T;L^2(\Gamma))
  \end{align*}
  and thus $q:=\partial_t\hat{q}\in \mathcal{D}'(0,T;L^2(\Gamma))$ is well-defined.
  Now we have
  \begin{align*}
  0 = -\int_0^T\partial_t\varphi(t)\{F(t)+g\nabla_\Gamma\hat{q}(t)\}\,dt = \int_0^T\varphi(t)\{\partial_tF(t)+g\partial_t(\nabla_\Gamma\hat{q})(t)\}\,dt
  \end{align*}
  in $H^{-1}(\Gamma,T\Gamma)$ for all $\varphi\in C_c^\infty(0,T)$, which means that
  \begin{align*}
    \partial_tF(t)+g\partial_t(\nabla_\Gamma\hat{q})(t) = 0 \quad\text{in}\quad \mathcal{D}'(0,T;H^{-1}(\Gamma,T\Gamma)).
  \end{align*}
  This implies \eqref{E:LW_Pres} since $\partial_t(\nabla_\Gamma\hat{q})=\nabla_\Gamma q$ by \eqref{E:Def_TGrDt_Hin},
  \begin{align*}
    \partial_tF = g\partial_tv+A_gv+B_g(v,v)-gf = g\Bigl(\partial_tv+\overline{\nabla}_vv\Bigr)+A_gv-gf,
  \end{align*}
  and $A_gv$ is of the form \eqref{Pf_LWP:Ag}.
\end{proof}

Now let us prove Theorem~\ref{T:SL_Weak}.
First we present an auxiliary result on the weak limit of the averaged tangential component of a vector field in $L_\sigma^2(\Omega_\varepsilon)$.

\begin{lemma} \label{L:WC_Sole}
  For $\varepsilon\in(0,1)$ let $u^\varepsilon\in L_\sigma^2(\Omega_\varepsilon)$.
  Also, let $v\in L^2(\Gamma,T\Gamma)$.
  Suppose that there exist $\varepsilon'\in(0,1)$, $c>0$, and $\alpha>0$ such that
  \begin{align*}
    \|u^\varepsilon\|_{L^2(\Omega_\varepsilon)}^2 \leq c\varepsilon^{-1+\alpha} \quad\text{for all}\quad \varepsilon\in(0,\varepsilon')
  \end{align*}
  and that $M_\tau u^\varepsilon$ converges to $v$ weakly in $L^2(\Gamma,T\Gamma)$ as $\varepsilon\to0$.
  Then $v\in L_{g\sigma}^2(\Gamma,T\Gamma)$.
\end{lemma}

\begin{proof}
  By \eqref{E:Sdiv_Hin} and the weak convergence of $\{M_\tau u^\varepsilon\}_\varepsilon$ to $v$ in $L^2(\Gamma,T\Gamma)$ we see that $\mathrm{div}_\Gamma(gM_\tau u^\varepsilon)$ converges to $\mathrm{div}_\Gamma(gv)$ weakly in $H^{-1}(\Gamma)$ as $\varepsilon\to0$.
  Moreover,
  \begin{align*}
    \|\mathrm{div}_\Gamma(gM_\tau u^\varepsilon)\|_{H^{-1}(\Gamma)} \leq c\varepsilon^{1/2}\|u^\varepsilon\|_{L^2(\Omega_\varepsilon)} \leq c\varepsilon^{\alpha/2}
  \end{align*}
  for all $\varepsilon\in(0,\varepsilon')$ with $\alpha>0$ by \eqref{E:ADiv_Tan_Hin} and the assumption on $u^\varepsilon$.
  Hence
  \begin{align*}
    \|\mathrm{div}_\Gamma(gv)\|_{H^{-1}(\Gamma)} \leq \liminf_{\varepsilon\to0}\|\mathrm{div}_\Gamma(gM_\tau u^\varepsilon)\|_{H^{-1}(\Gamma)} = 0
  \end{align*}
  and $\mathrm{div}_\Gamma(gv)=0$ in $H^{-1}(\Gamma)$, which means that $v\in L_{g\sigma}^2(\Gamma,T\Gamma)$.
\end{proof}

\begin{proof}[Proof of Theorem~\ref{T:SL_Weak}]
  Suppose that the assumptions of Theorem~\ref{T:SL_Weak} are satisfied.
  Let $\varepsilon_1$, $\varepsilon_\sigma$, and $\varepsilon_2$ be the constants given in Theorem~\ref{T:UE}, Lemma~\ref{L:HP_Dom}, and the condition (a) of Theorem~\ref{T:SL_Weak}, and let $\varepsilon_3:=\min\{\varepsilon_1,\varepsilon_\sigma,\varepsilon_2\}\in(0,1)$.
  For each $\varepsilon\in(0,\varepsilon_3)$ the inequalities \eqref{E:UE_Data} with $\beta=1$ are valid since the condition (a) is satisfied and $\{M_\tau u_0^\varepsilon\}_\varepsilon$ and $\{M_\tau\mathbb{P}_\varepsilon f^\varepsilon\}_\varepsilon$ are bounded in $L^2(\Gamma,T\Gamma)$ and $L^\infty(0,\infty;H^{-1}(\Gamma,T\Gamma))$ by the condition (b).
  Hence by Theorem~\ref{T:Est_Ue} there exists a global strong solution $u^\varepsilon$ to \eqref{E:NS_Eq}--\eqref{E:NS_In} satisfying \eqref{E:Est_Ue}--\eqref{E:Est_DtUe}.
  Moreover, by \eqref{E:Ave_N_Lp} and \eqref{E:Est_Ue},
  \begin{align*}
    \sup_{t\in[0,\infty)}\|Mu^\varepsilon(t)\cdot n\|_{L^2(\Gamma)} \leq c\varepsilon^{1/2}\sup_{t\in[0,\infty)}\|u^\varepsilon(t)\|_{H^1(\Omega_\varepsilon)} \leq c\varepsilon^{\alpha/2} \to 0
  \end{align*}
  as $\varepsilon\to0$.
  Thus $\{Mu^\varepsilon\cdot n\}_\varepsilon$ converges to zero strongly in $C([0,\infty);L^2(\Gamma))$.

  Now let us consider the averaged tangential component $M_\tau u^\varepsilon$.
  First note that the weak limit $v_0$ of $\{M_\tau u_0^\varepsilon\}_\varepsilon$ is in $L_{g\sigma}^2(\Gamma,T\Gamma)$ by the condition (a), the inequality \eqref{E:Stokes_H1}, and Lemma~\ref{L:WC_Sole}.
  Since all results in the previous subsections apply to $u^\varepsilon$, for fixed $T>0$ we see by \eqref{E:Mu_Energy} and \eqref{E:Mu_Dt} that
  \begin{itemize}
    \item $\{M_\tau u^\varepsilon\}_\varepsilon$ is bounded in $L^\infty(0,T;L^2(\Gamma,T\Gamma))\cap L^2(0,T;H^1(\Gamma,T\Gamma))$,
    \item $\{\partial_tM_\tau u^\varepsilon\}_\varepsilon$ is bounded in $L^2(0,T;H^{-1}(\Gamma,T\Gamma))$.
  \end{itemize}
  Hence there exist $\varepsilon_k\in(0,\varepsilon_3)$, $k\in\mathbb{N}$ and a vector field
  \begin{align*}
    v \in L^\infty(0,T;L^2(\Gamma,T\Gamma))\cap L^2(0,T;H^1(\Gamma,T\Gamma))
  \end{align*}
  with $\partial_tv \in L^2(0,T;H^{-1}(\Gamma,T\Gamma))$ such that $\lim_{k\to\infty}\varepsilon_k=0$ and
  \begin{align} \label{Pf_SLW:W_Conv}
    \begin{alignedat}{3}
      \lim_{k\to\infty}M_\tau u^{\varepsilon_k} &= v &\quad &\text{weakly-$\star$ in} &\quad &L^\infty(0,T;L^2(\Gamma,T\Gamma)), \\
      \lim_{k\to\infty}M_\tau u^{\varepsilon_k} &= v &\quad &\text{weakly in} &\quad &L^2(0,T;H^1(\Gamma,T\Gamma)), \\
      \lim_{k\to\infty}\partial_tM_\tau u^{\varepsilon_k} &= \partial_tv &\quad &\text{weakly in} &\quad &L^2(0,T;H^{-1}(\Gamma,T\Gamma)).
    \end{alignedat}
  \end{align}
  Moreover, by the Aubin--Lions lemma (see e.g.~\cite[Theorem~II.5.16]{BoFa13}) there exists a subsequence of $\{M_\tau u^{\varepsilon_k}\}_{k=1}^\infty$, which we denote by $\{M_\tau u^{\varepsilon_k}\}_{k=1}^\infty$ again, such that
  \begin{align} \label{Pf_SLW:St_Conv}
    \lim_{k\to\infty}M_\tau u^{\varepsilon_k} = v \quad\text{strongly in}\quad L^2(0,T;L^2(\Gamma,T\Gamma)).
  \end{align}
  Then $v(t)\in L_{g\sigma}^2(\Gamma,T\Gamma)$ for a.a. $t\in (0,T)$ by \eqref{E:Est_Ue} and Lemma~\ref{L:WC_Sole}.
  Hence
  \begin{align*}
    v \in L^\infty(0,T;L_{g\sigma}^2(\Gamma,T\Gamma))\cap L^2(0,T;\mathcal{V}_g).
  \end{align*}
  Let us show that $v$ satisfies \eqref{E:Limit_Weak} for all $\eta\in C_c(0,T;\mathcal{V}_g)$.
  In what follows, we write $c$ for a general positive constant that may depend on $v$ and $\eta$ but is independent of $\varepsilon_k$ and $u^{\varepsilon_k}$.
  We consider the weak formulation \eqref{E:Mu_Weak} for $M_\tau u^{\varepsilon_k}$:
  \begin{multline} \label{Pf_SLW:Weak_Muek}
    \int_0^T\{[g\partial_tM_\tau u^{\varepsilon_k},\eta]_{T\Gamma}+a_g(M_\tau u^{\varepsilon_k},\eta)+b_g(M_\tau u^{\varepsilon_k},M_\tau u^{\varepsilon_k},\eta)\}\,dt \\
    = \int_0^T[gM_\tau\mathbb{P}_{\varepsilon_k}f^{\varepsilon_k},\eta]_{T\Gamma}\,dt+R_{\varepsilon_k}^1(\eta).
  \end{multline}
  Here $\partial_tM_\tau u^{\varepsilon_k}$ and $M_\tau\mathbb{P}_{\varepsilon_k}f^{\varepsilon_k}$ are considered in $H^{-1}(\Gamma,T\Gamma)$ (see Section~\ref{SS:Pre_Surf}).
  Let $k\to\infty$ in \eqref{Pf_SLW:Weak_Muek}.
  Then by the assumption (b), \eqref{E:Def_Mul_HinT} and \eqref{Pf_SLW:W_Conv} we get
  \begin{align} \label{Pf_SLW:Conv_Lin}
    \begin{aligned}
      \lim_{k\to\infty}\int_0^T[g\partial_tM_\tau u^{\varepsilon_k},\eta]_{T\Gamma}\,dt &= \int_0^T[g\partial_tv,\eta]_{T\Gamma}\,dt, \\
      \lim_{k\to\infty}\int_0^Ta_g(M_\tau u^{\varepsilon_k},\eta)\,dt &= \int_0^Ta_g(v,\eta)\,dt, \\
      \lim_{k\to\infty}\int_0^T[gM_\tau\mathbb{P}_{\varepsilon_k}f^{\varepsilon_k},\eta]_{T\Gamma}\,dt &= \int_0^T[gf,\eta]_{T\Gamma}\,dt.
    \end{aligned}
  \end{align}
  Also, by \eqref{E:Mu_Weak_Re}, the assumption (c), and $\alpha>0$,
  \begin{align} \label{Pf_SLW:Conv_Re}
    |R_{\varepsilon_k}^1(\eta)| \leq c\left(\varepsilon_k^{\alpha/4}+\sum_{i=0,1}|\varepsilon_k^{-1}\gamma_{\varepsilon_k}^i-\gamma^i|\right)(1+T)^{1/2}\|\eta\|_{L^2(0,T;H^1(\Gamma))} \to 0
  \end{align}
  as $k\to\infty$.
  To show the convergence of the trilinear term, we set
  \begin{align*}
    J_1^k &:= \int_0^Tb_g(M_\tau u^{\varepsilon_k},M_\tau u^{\varepsilon_k},\eta)\,dt-\int_0^Tb_g(v,M_\tau u^{\varepsilon_k},\eta)\,dt, \\
    J_2^k &:= \int_0^Tb_g(v,M_\tau u^{\varepsilon_k},\eta)\,dt-\int_0^Tb_g(v,v,\eta)\,dt.
  \end{align*}
  Since $\|\eta(t)\|_{H^1(\Gamma)}$ is bounded on $[0,T]$ by $\eta\in C_c(0,T;\mathcal{V}_g)$, we see by \eqref{E:Tri_Surf} that
  \begin{align*}
    |J_1^k| &\leq c\int_0^T\|M_\tau u^{\varepsilon_k}-v\|_{L^2(\Gamma)}^{1/2}\|M_\tau u^{\varepsilon_k}-v\|_{H^1(\Gamma)}^{1/2}\|M_\tau u^{\varepsilon_k}\|_{H^1(\Gamma)}\|\eta\|_{H^1(\Gamma)}\,dt \\
    &\leq c\|M_\tau u^{\varepsilon_k}-v\|_{L^2(0,T;L^2(\Gamma))}^{1/2}\|M_\tau u^{\varepsilon_k}-v\|_{L^2(0,T;H^1(\Gamma))}^{1/2}\|M_\tau u^{\varepsilon_k}\|_{L^2(0,T;H^1(\Gamma))}.
  \end{align*}
  Applying \eqref{E:Mu_Energy} and \eqref{Pf_SLW:St_Conv} to the last line we obtain
  \begin{align} \label{Pf_SLW:I1}
    |J_1^k| \leq c\|M_\tau u^{\varepsilon_k}-v\|_{L^2(0,T;L^2(\Gamma))}^{1/2} \to 0 \quad\text{as}\quad k\to\infty.
  \end{align}
  For $J_2^k$, we consider the linear functional
  \begin{align*}
    \Phi(\xi) := \int_0^Tb_g(v,\xi,\eta)\,dt, \quad \xi \in L^2(0,T;H^1(\Gamma,T\Gamma)).
  \end{align*}
  By \eqref{E:Tri_Surf} and the boundedness of $\|\eta(t)\|_{H^1(\Gamma)}$ on $[0,T]$ we get
  \begin{align*}
    |\Phi(\xi)| \leq c\|\eta\|_{L^\infty(0,T;H^1(\Gamma))}\|v\|_{L^2(0,T;H^1(\Gamma))}\|\xi\|_{L^2(0,T;H^1(\Gamma))}
  \end{align*}
  for all $\xi\in L^2(0,T;H^1(\Gamma,T\Gamma))$.
  Hence $\Phi$ is bounded on $L^2(0,T;H^1(\Gamma,T\Gamma))$ and the weak convergence \eqref{Pf_SLW:W_Conv} in $L^2(0,T;H^1(\Gamma,T\Gamma))$ implies that
  \begin{align*}
    \lim_{k\to\infty}J_2^k = \lim_{k\to\infty}\{\Phi(M_\tau u^{\varepsilon_k})-\Phi(v)\} = 0.
  \end{align*}
  Combining this equality with \eqref{Pf_SLW:I1} we obtain
  \begin{align} \label{Pf_SLW:Conv_Tri}
    \lim_{k\to\infty}\int_0^Tb_g(M_\tau u^{\varepsilon_k},M_\tau u^{\varepsilon_k},\eta)\,dt = \int_0^Tb_g(v,v,\eta)\,dt.
  \end{align}
  We send $k\to\infty$ in \eqref{Pf_SLW:Weak_Muek} and apply \eqref{Pf_SLW:Conv_Lin}, \eqref{Pf_SLW:Conv_Re}, and \eqref{Pf_SLW:Conv_Tri} to show that $v$ satisfies \eqref{E:Limit_Weak} for all $\eta\in C_c(0,T;\mathcal{V}_g)$.
  Moreover, by Lemma~\ref{L:Lim_W_L2} we see that
  \begin{align*}
    v\in C([0,T],L_{g\sigma}^2(\Gamma,T\Gamma)), \quad \partial_tv \in L^2(0,T;H^{-1}(\Gamma,T\Gamma))
  \end{align*}
  and \eqref{E:Limit_Weak} is valid for all $\eta\in L^2(0,T;\mathcal{V}_g)$.

  To show that $v$ is a weak solution to \eqref{E:Limit_Eq}--\eqref{E:Limit_Div} on $[0,T)$ it remains to verify the initial condition.
  Let $\xi\in \mathcal{V}_g$ and $\varphi\in C^\infty([0,T])$ satisfy $\varphi(0)=1$ and $\varphi(T)=0$.
  We substitute $\eta:=\varphi\xi\in L^2(0,T;\mathcal{V}_g)$ for \eqref{E:Limit_Weak} and \eqref{Pf_SLW:Weak_Muek} and carry out integration by parts for $\partial_tv$ and $\partial_tM_\tau u^{\varepsilon_k}$.
  Then we get
  \begin{align} \label{Pf_SLW:Ini}
    (gv(0),\xi)_{L^2(\Gamma)} = J_\infty, \quad (gM_\tau u_0^{\varepsilon_k},\xi)_{L^2(\Gamma)} = J_k
  \end{align}
  by $\varphi(0)=1$ and $\varphi(T)=0$ (also note that $\xi$ is independent of time), where
  \begin{align*}
    J_\infty := -\int_0^T\partial_t\varphi(gv,\xi)_{L^2(\Gamma)}\,dt+\int_0^T\{a_g(v,\eta)+b_g(v,v,\eta)\}\,dt-\int_0^T[gf,\eta]_{T\Gamma}\,dt
  \end{align*}
  and
  \begin{multline*}
    J_k := -\int_0^T\partial_t\varphi(gM_\tau u^{\varepsilon_k},\xi)_{L^2(\Gamma)}\,dt+\int_0^T\{a_g(M_\tau u^{\varepsilon_k},\eta)+b_g(M_\tau u^{\varepsilon_k},M_\tau u^{\varepsilon_k},\eta)\}\,dt \\
    -\int_0^T[gM_\tau\mathbb{P}_{\varepsilon_k}f^{\varepsilon_k},\eta]_{T\Gamma}\,dt-R_{\varepsilon_k}^1(\eta).
  \end{multline*}
  We send $k\to\infty$ in the second equality of \eqref{Pf_SLW:Ini}.
  Then the left-hand side converges to $(gv_0,\xi)_{L^2(\Gamma)}$ by the assumption (b).
  Also, we use \eqref{Pf_SLW:Conv_Lin}, \eqref{Pf_SLW:Conv_Re}, \eqref{Pf_SLW:Conv_Tri}, and
  \begin{align*}
    \lim_{k\to\infty}\int_0^T\partial_t\varphi(gM_\tau u^{\varepsilon_k},\xi)_{L^2(\Gamma)}\,dt = \int_0^T\partial_t\varphi(gv,\xi)_{L^2(\Gamma)}\,dt
  \end{align*}
  by \eqref{Pf_SLW:St_Conv} to find that $\lim_{k\to\infty}J_k=J_\infty$ (note that in the proof of \eqref{Pf_SLW:Conv_Tri} we only used the boundedness of $\|\eta(t)\|_{H^1(\Gamma)}$ on $[0,T]$).
  Hence
  \begin{align*}
    (gv(0),\xi)_{L^2(\Gamma)} = J_\infty = (gv_0,\xi)_{L^2(\Gamma)} \quad\text{for all}\quad \xi\in \mathcal{V}_g.
  \end{align*}
  Since $\mathcal{V}_g$ is dense in $L_{g\sigma}^2(\Gamma,T\Gamma)$ (see Lemma~\ref{L:H1gs_Dense}), the above equality is also valid for all $\xi\in L_{g\sigma}^2(\Gamma,T\Gamma)$.
  Thus, setting $\xi:=v(0)-v_0$ we get
  \begin{align*}
    (g\{v(0)-v_0\},v(0)-v_0)_{L^2(\Gamma)} = \|g^{1/2}\{v(0)-v_0\}\|_{L^2(\Gamma)}^2 = 0,
  \end{align*}
  which combined with \eqref{E:Width_Bound} shows $v|_{t=0}=v_0$ on $\Gamma$.
  Therefore, $v$ is a unique weak solution to \eqref{E:Limit_Eq}--\eqref{E:Limit_Div} on $[0,T)$ (here the uniqueness follows from Lemma~\ref{L:Limit_Uni}).

  Let us prove the convergence of the full sequence
  \begin{align} \label{Pf_SLW:Full_Conv}
    \begin{alignedat}{3}
      \lim_{\varepsilon\to0}M_\tau u^\varepsilon &= v &\quad &\text{weakly in} &\quad L^2(0,T;H^1(\Gamma,T\Gamma)), \\
      \lim_{\varepsilon\to0}\partial_tM_\tau u^\varepsilon &= \partial_tv &\quad &\text{weakly in} &\quad L^2(0,T;H^{-1}(\Gamma,T\Gamma)).
    \end{alignedat}
  \end{align}
  By the boundedness of $\{M_\tau u^\varepsilon\}_\varepsilon$ and $\{\partial_t M_\tau u^\varepsilon\}_\varepsilon$ (see \eqref{E:Mu_Energy} and \eqref{E:Mu_Dt}) and the uniqueness of a weak solution to \eqref{E:Limit_Eq}--\eqref{E:Limit_Div} (see Lemma~\ref{L:Limit_Uni}) we can show as above that for any sequence $\{\varepsilon_l\}_{l=1}^\infty$ of positive numbers convergent to zero there exists its subsequence $\{\varepsilon_k\}_{k=1}^\infty$ such that $\{M_\tau u^{\varepsilon_k}\}_{k=1}^\infty$ converges to $v$ in the sense of \eqref{Pf_SLW:W_Conv} and \eqref{Pf_SLW:St_Conv}.
  This proves \eqref{Pf_SLW:Full_Conv}.

  Since the strong solution $u^\varepsilon$ to \eqref{E:NS_Eq}--\eqref{E:NS_In} exists globally in time for $\varepsilon\in(0,\varepsilon_3)$, by the above arguments we get a unique weak solution
  \begin{align*}
    v_T \in C([0,T];L_{g\sigma}^2(\Gamma,T\Gamma))\cap L^2(0,T;\mathcal{V}_g)\cap H^1(0,T;H^{-1}(\Gamma,T\Gamma))
  \end{align*}
  to \eqref{E:Limit_Eq}--\eqref{E:Limit_Div} on $[0,T)$ satisfying \eqref{Pf_SLW:Full_Conv} for all $T>0$.
  Moreover, if $T<T'$ then $v_T=v_{T'}$ on $[0,T]$ by the uniqueness of a weak solution.
  Hence we can define
  \begin{align*}
    v \in C([0,\infty);L_{g\sigma}^2(\Gamma,T\Gamma))\cap L_{loc}^2([0,\infty);\mathcal{V}_g)\cap H_{loc}^1([0,\infty);H^{-1}(\Gamma,T\Gamma))
  \end{align*}
  by $v:=v_T$ on $[0,T]$ for each $T>0$, which is a unique weak solution to \eqref{E:Limit_Eq}--\eqref{E:Limit_Div} on $[0,\infty)$ and satisfies \eqref{Pf_SLW:Full_Conv} for all $T>0$.
\end{proof}

As a consequence of Theorem~\ref{T:SL_Weak}, we obtain the existence of a weak solution to \eqref{E:Limit_Eq}--\eqref{E:Limit_Div} for the initial velocity $v_0$ and the external force $f$ given by the weak and weak-$\star$ limit of $M_\tau u_0^\varepsilon$ and $M_\tau\mathbb{P}_\varepsilon f^\varepsilon$, respectively.
For general $v_0\in L_{g\sigma}^2(\Gamma,T\Gamma)$ and $f\in L_{loc}^2([0,\infty);H^{-1}(\Gamma,T\Gamma))$ we can construct a global weak solution to \eqref{E:Limit_Eq}--\eqref{E:Limit_Div} by the Galerkin method as in the case of the Navier--Stokes equations in a two-dimensional bounded domain (see e.g.~\cite{BoFa13,CoFo88,Te79}).
Here we just give the outline of construction of a weak solution by the Galerkin method.

\emph{Countable basis of $\mathcal{V}_g$.}
We define a bilinear form
\begin{align*}
  \tilde{a}_g(v_1,v_2) := a_g(v_1,v_2)+(v_1,v_2)_{L^2(\Gamma)}, \quad v_1,v_2\in H^1(\Gamma,T\Gamma).
\end{align*}
By \eqref{E:Bi_Surf} we observe that $\tilde{a}_g$ is bounded and coercive on the closed subspace $\mathcal{V}_g$ of $H^1(\Gamma,T\Gamma)$.
Hence it induces a linear homeomorphism $\widetilde{A}_g$ from $\mathcal{V}_g$ onto the dual space $\mathcal{V}_g'$ by the Lax--Milgram theorem.
We consider $\widetilde{A}_g$ as an unbounded operator on $L_{g\sigma}^2(\Gamma,T\Gamma)$ equipped with inner product
\begin{align*}
  (v_1,v_2)_{L_g^2(\Gamma)} := (g^{1/2}v_1,g^{1/2}v_2)_{L^2(\Gamma)}, \quad v_1,v_2\in L_{g\sigma}^2(\Gamma,T\Gamma),
\end{align*}
which is equivalent to the canonical $L^2(\Gamma)$-inner product by \eqref{E:Width_Bound}.
Then as in the case of the Stokes operator on a bounded domain (see e.g.~\cite[Theorem~IV.5.5]{BoFa13}) we can show that there exists a sequence $\{w_k\}_{k=1}^\infty$ of eigenvectors of $\widetilde{A}_g$ that is an orthonormal basis of $L_{g\sigma}^2(\Gamma,T\Gamma)$ equipped with inner product $(\cdot,\cdot)_{L_g^2(\Gamma)}$ as well as an orthogonal basis of $\mathcal{V}_g$ equipped with inner product $\tilde{a}_g(\cdot,\cdot)$.
In particular,
\begin{align} \label{E:Weight_Orth}
  (gw_i,w_j)_{L^2(\Gamma)} = (w_i,w_j)_{L_g^2(\Gamma)} = \delta_{ij}, \quad i,j\in\mathbb{N},
\end{align}
where $\delta_{ij}$ is the Kronecker delta.

\emph{Approximate problem.}
For $k\in\mathbb{N}$ let $\mathcal{V}_g^k$ be the linear span of $\{w_i\}_{i=1}^k$.
We seek for an approximate solution $v_k(t)=\sum_{i=1}^k\xi_i(t)w_i\in\mathcal{V}_g^k$ that satisfies
\begin{align} \label{E:LimG_Appr}
  (g\partial_tv_k(t),\eta_k)_{L^2(\Gamma)}+a_g(v_k(t),\eta_k)+b_g(v_k(t),v_k(t),\eta_k) = [gf(t),\eta_k]_{T\Gamma}
\end{align}
for all $\eta_k\in \mathcal{V}_g^k$ and $t\in(0,T)$, $T>0$ with initial condition (here we approximate $f(t)$ by a continuous function).
This problem is equivalent to a system of ODEs
\begin{align*}
  \sum_{i=1}^k(gw_i,w_j)_{L^2(\Gamma)}\frac{d\xi_i}{dt}(t) = \mathcal{P}_j(\xi(t))+[gf(t),w_j]_{T\Gamma}, \quad j=1,\dots,k
\end{align*}
with polynomials $\mathcal{P}_1,\dots,\mathcal{P}_k$ of $\xi=(\xi_1,\dots,\xi_k)$.
Applying \eqref{E:Weight_Orth} to the left-hand side we see that this system reduces to $d\xi_j/dt=\mathcal{P}_j(\xi)+[gf,w_j]_{T\Gamma}$, which we can solve locally by the Cauchy--Lipschitz theorem.
By \eqref{E:Weight_Orth} we can also derive the energy estimate for $v_k$ and show its global existence.

\emph{Estimate for the time derivative of the approximate solution.}
As in Lemma~\ref{L:PMu_Dt} we estimate $\partial_tv_k$ in $H^{-1}(\Gamma,T\Gamma)$ (see also Remark~\ref{R:Mu_Dt}).
For $w\in H^1(\Gamma,T\Gamma)$ we use Lemma~\ref{L:Mu_Dt_Test} to get $w=g\eta+g\nabla_\Gamma q$ with $\eta\in \mathcal{V}_g$ satisfying \eqref{E:Mu_Dt_Test} and $q\in H^2(\Gamma)$.
Since $\{w_k\}_{k=1}^\infty$ is an orthogonal basis of $\mathcal{V}_g$ equipped with inner product $\tilde{a}_g(\cdot,\cdot)$,
\begin{align*}
  \eta = \sum_{i=1}^\infty \tilde{a}_g(\eta,\tilde{w_i})\tilde{w}_i \quad\text{in}\quad \mathcal{V}_g, \quad \tilde{w}_i := \frac{w_i}{a_g(w_i,w_i)^{1/2}}.
\end{align*}
Then we set $\eta_k:=\sum_{i=1}^k\tilde{a}_g(\eta,\tilde{w}_k)\tilde{w}_k\in \mathcal{V}_g^k$ to get
\begin{align*}
  \|\eta_k\|_{H^1(\Gamma)}\leq c\|\eta\|_{H^1(\Gamma)}\leq c\|w\|_{H^1(\Gamma)}
\end{align*}
by \eqref{E:Mu_Dt_Test} and the equivalence of $\tilde{a}_g(\cdot,\cdot)$ to $(\cdot,\cdot)_{H^1(\Gamma)}$ on $\mathcal{V}_g$.
Moreover,
\begin{align*}
  (g\partial_tv_k,\eta_k)_{L^2(\Gamma)} &= (g\partial_tv_k,\eta)_{L^2(\Gamma)} = (\partial_tv_k,g\eta)_{L^2(\Gamma)} \\
  &= (\partial_tv_k,w-g\nabla_\Gamma q)_{L^2(\Gamma)} = (\partial_tv_k,w)_{L^2(\Gamma)},
\end{align*}
where the first equality follows from \eqref{E:Weight_Orth} and the last equality is due to $\partial_tv_k\in L_{g\sigma}^2(\Gamma,T\Gamma)$ and $g\nabla_\Gamma q\in L_{g\sigma}^2(\Gamma,T\Gamma)^\perp$ (note that here we take the canonical $L^2(\Gamma)$-inner product).
Hence, substituting $\eta_k$ for \eqref{E:LimG_Appr} and using these relations, we can show as in the proof of Lemma~\ref{L:PMu_Dt} that
\begin{align*}
  \|\partial_tv_k(t)\|_{H^{-1}(\Gamma,T\Gamma)} \leq c\left\{\left(1+\|v_k(t)\|_{L^2(\Gamma)}\right)\|v_k(t)\|_{H^1(\Gamma)}+\|f(t)\|_{H^{-1}(\Gamma,T\Gamma)}\right\}
\end{align*}
for all $t\in(0,T)$.
By this inequality and the energy estimate for $v_k$ we obtain the boundedness of $\{\partial_tv_k\}_{k=1}^\infty$ in $L^2(0,T;H^{-1}(\Gamma,T\Gamma))$ and we can prove the convergence of $\{v_k\}_{k=1}^\infty$ to a weak solution to \eqref{E:Limit_Eq}--\eqref{E:Limit_Div} as in the proof of Theorem~\ref{T:SL_Weak}.

\subsection{Strong convergence of the average and error estimates} \label{SS:SL_ErST}
In this subsection we show the strong convergence of the averaged tangential component of a strong solution to \eqref{E:NS_Eq}--\eqref{E:NS_In} towards a weak solution to \eqref{E:Limit_Eq}--\eqref{E:Limit_Div}.
We also give estimates in $\Omega_\varepsilon$ for the difference between solutions to the bulk and limit equations.

\begin{theorem} \label{T:Diff_Mu_V}
  Under the assumptions of Theorem~\ref{T:Est_Ue}, let $\varepsilon\in(0,\varepsilon'_1)$ and $u^\varepsilon$ be the global strong solution to \eqref{E:NS_Eq}--\eqref{E:NS_In} given by Theorem~\ref{T:Est_Ue}.
  Also, let $v$ be a weak solution to \eqref{E:Limit_Eq}--\eqref{E:Limit_Div} on $[0,\infty)$ with given data $v_0\in L_{g\sigma}^2(\Gamma,T\Gamma)$ and $f\in L^\infty(0,\infty;H^{-1}(\Gamma,T\Gamma))$.
  Then for all $T>0$ we have
  \begin{multline} \label{E:Diff_Mu_V}
    \max_{t\in[0,T]}\|M_\tau u^\varepsilon(t)-v(t)\|_{L^2(\Gamma)}^2+\int_0^T\|\nabla_\Gamma M_\tau u^\varepsilon(t)-\nabla_\Gamma v(t)\|_{L^2(\Gamma)}^2\,dt \\
    \leq c_T\left\{\delta(\varepsilon)^2+\|M_\tau u_0^\varepsilon-v_0\|_{L^2(\Gamma)}^2+\|M_\tau\mathbb{P}_\varepsilon f^\varepsilon-f\|_{L^\infty(0,\infty;H^{-1}(\Gamma,T\Gamma))}^2\right\},
  \end{multline}
  where $c_T>0$ is a constant depending only on $T$ and
  \begin{align} \label{E:DMuV_Re}
    \delta(\varepsilon) := \varepsilon^{\alpha/4}+\sum_{i=0,1}|\varepsilon^{-1}\gamma_\varepsilon^i-\gamma^i|.
  \end{align}
\end{theorem}

As in Section~\ref{SS:SL_Ener}, we first compare the auxiliary vector field $v^\varepsilon=\mathbb{P}_gM_\tau u^\varepsilon$ with $v$ and then derive \eqref{E:Diff_Mu_V} by using the estimates for the difference $M_\tau u^\varepsilon-v^\varepsilon$.

\begin{lemma} \label{L:Diff_Ve_V}
  Under the same assumptions as in Theorem~\ref{T:Diff_Mu_V}, we have
  \begin{multline} \label{E:Diff_Ve_V}
    \max_{t\in[0,T]}\|v^\varepsilon(t)-v(t)\|_{L^2(\Gamma)}^2+\int_0^T\|\nabla_\Gamma v^\varepsilon(t)-\nabla_\Gamma v(t)\|_{L^2(\Gamma)}^2\,dt \\
    \leq c_T\left\{\delta(\varepsilon)^2+\|v^\varepsilon(0)-v_0\|_{L^2(\Gamma)}^2+\|M_\tau\mathbb{P}_\varepsilon f^\varepsilon-f\|_{L^\infty(0,\infty;H^{-1}(\Gamma,T\Gamma))}^2\right\}
  \end{multline}
  for all $T>0$, where $v^\varepsilon=\mathbb{P}_gM_\tau u^\varepsilon$ is given in Lemma~\ref{L:PMu_Weak}, $c_T>0$ is a constant depending only on $T$, and $\delta(\varepsilon)$ is given by \eqref{E:DMuV_Re}.
\end{lemma}

\begin{proof}
  For the sake of simplicity, we set
  \begin{align*}
    w^\varepsilon := v^\varepsilon-v, \quad w_0^\varepsilon := v^\varepsilon(0)-v_0 = \mathbb{P}_gM_\tau u_0^\varepsilon-v_0, \quad F^\varepsilon := M_\tau \mathbb{P}_\varepsilon f^\varepsilon-f.
  \end{align*}
  Let $T>0$.
  We subtract both sides of \eqref{E:Limit_Weak} from those of \eqref{E:PMu_Weak} to get
  \begin{multline*}
    \int_0^T\{[g\partial_sw^\varepsilon,\eta]_{T\Gamma}+a_g(w^\varepsilon,\eta)+b_g(w^\varepsilon,v^\varepsilon,\eta)+b_g(v,w^\varepsilon,\eta)\}\,ds \\
    = \int_0^T[gF^\varepsilon,\eta]_{T\Gamma}\,ds+R_\varepsilon^1(\eta)+R_\varepsilon^2(\eta)
  \end{multline*}
  for all $\eta\in L^2(0,T;\mathcal{V}_g)$ (see also Lemma~\ref{L:Lim_W_L2}), where $R_\varepsilon^1(\eta)$ and $R_\varepsilon^2(\eta)$ are given in Lemmas~\ref{L:Mu_Weak} and \ref{L:PMu_Weak}.
  For each $t\in[0,T]$ let $1_{[0,t]}\colon\mathbb{R}\to\mathbb{R}$ be the characteristic function of $[0,t]$.
  We substitute $\eta=1_{[0,t]}w^\varepsilon$ for the above equality and calculate as in the proofs of Lemmas~\ref{L:PMu_Energy} and~\ref{L:Limit_Uni} by using \eqref{E:Bi_Surf}--\eqref{E:TriS_Vg}, \eqref{E:Mu_Weak_Re}, \eqref{E:PMu_Weak_Re}, \eqref{Pf_LU:Int_Dt}, and Young's inequality.
  Then we get
  \begin{multline} \label{Pf_DVeV:Gron}
    \|w^\varepsilon(t)\|_{L^2(\Gamma)}^2+\int_0^t\|\nabla_\Gamma w^\varepsilon\|_{L^2(\Gamma)}^2\,ds \\
    \leq c\left\{\|w_0^\varepsilon\|_{L^2(\Gamma)}^2+\int_0^t\left(1+\|v^\varepsilon\|_{H^1(\Gamma)}^2\right)\|w^\varepsilon\|_{L^2(\Gamma)}^2\,ds \right. \\
    \left. +\int_0^t\|F^\varepsilon\|_{H^{-1}(\Gamma,T\Gamma)}^2\,ds+\delta(\varepsilon)^2(1+t)\right\}
  \end{multline}
  for all $t\in[0,T]$.
  Here $\delta(\varepsilon)$ given by \eqref{E:DMuV_Re} comes from \eqref{E:Mu_Weak_Re} and \eqref{E:PMu_Weak_Re} (note that $\varepsilon^{\alpha/2}\leq\varepsilon^{\alpha/4}$).
  This inequality implies that
  \begin{align*}
    \xi(t) \leq c\left\{\xi(0)+\int_0^t\left(\varphi(s)\xi(s)+\|F^\varepsilon(s)\|_{H^{-1}(\Gamma,T\Gamma)}^2\right)ds\right\} \quad\text{for all}\quad t\in[0,T],
  \end{align*}
  where $\xi(t):=\delta(\varepsilon)^2+\|w^\varepsilon(t)\|_{L^2(\Gamma)}^2$ and $\varphi(t):=1+\|v^\varepsilon(t)\|_{H^1(\Gamma)}^2$.
  Hence by Gronwall's inequality we have
  \begin{align*}
    \xi(t) \leq c\left(\xi(0)+\int_0^t\|F^\varepsilon(s)\|_{H^{-1}(\Gamma,T\Gamma)}^2\,ds\right)\exp\left(c\int_0^t\varphi(s)\,ds\right), \quad t\in[0,T].
  \end{align*}
  From this inequality and the estimate \eqref{E:PMu_Energy} for $\|v^\varepsilon\|_{H^1(\Gamma)}^2$ we deduce that
  \begin{align*}
    \|w^\varepsilon(t)\|_{L^2(\Gamma)}^2 \leq c_T\left\{\delta(\varepsilon)^2+\|w_0^\varepsilon\|_{L^2(\Gamma)}^2+\|F^\varepsilon\|_{L^\infty(0,\infty;H^{-1}(\Gamma,T\Gamma))}^2\right\}
  \end{align*}
  for all $t\in[0,T]$, where $c_T>0$ is a constant depending only on $T$.
  Applying this inequality and \eqref{E:PMu_Energy} to \eqref{Pf_DVeV:Gron} we also get the same estimate for the time integral of $\|\nabla_\Gamma w^\varepsilon\|_{L^2(\Gamma)}^2$.
  Therefore, the inequality \eqref{E:Diff_Ve_V} is valid.
\end{proof}

\begin{proof}[Proof of Theorem~\ref{T:Diff_Mu_V}]
  Let $v^\varepsilon=\mathbb{P}_gM_\tau u^\varepsilon$.
  Since $v_0\in L_{g\sigma}^2(\Gamma,T\Gamma)$ and $\mathbb{P}_g$ is the orthogonal projection from $L^2(\Gamma,T\Gamma)$ onto $L_{g\sigma}^2(\Gamma,T\Gamma)$,
  \begin{align*}
    \|v^\varepsilon(0)-v_0\|_{L^2(\Gamma)} = \|\mathbb{P}_g(M_\tau u_0^\varepsilon-v_0)\|_{L^2(\Gamma)} \leq \|M_\tau u_0^\varepsilon-v_0\|_{L^2(\Gamma)}.
  \end{align*}
  By this inequality, \eqref{E:Diff_PMu}, and \eqref{E:Diff_Ve_V} we obtain \eqref{E:Diff_Mu_V} (note that $\varepsilon^2\leq\delta(\varepsilon)^2$).
\end{proof}

Now we see that Theorem~\ref{T:SL_Strong} is an immediate consequence of Theorem~\ref{T:Diff_Mu_V}.

\begin{proof}[Proof of Theorem~\ref{T:SL_Strong}]
  The condition (b') implies the condition (b) of Theorem~\ref{T:SL_Weak}.
  Thus the statements of Theorem~\ref{T:SL_Weak} are valid.
  Also, for each $T>0$ the right-hand side of \eqref{E:Diff_Mu_V} converges to zero as $\varepsilon\to0$ by \eqref{E:DMuV_Re}, $\alpha>0$, and the conditions (b') and (c).
  Hence $M_\tau u^\varepsilon$ converges to $v$ strongly in $C([0,T];L^2(\Gamma,T\Gamma))$ and $L^2(0,T;H^1(\Gamma,T\Gamma))$ as $\varepsilon\to0$.
\end{proof}

Next let us estimate the difference between a strong solution to \eqref{E:NS_Eq}--\eqref{E:NS_In} and a weak solution to \eqref{E:Limit_Eq}--\eqref{E:Limit_Div} in $\Omega_\varepsilon$.
Recall that we denote by $\bar{\eta}=\eta\circ\pi$ the constant extension of a function $\eta$ on $\Gamma$ in the normal direction of $\Gamma$.

\begin{theorem} \label{T:Diff_Ue_Cv}
  Under the same assumptions as in Theorem~\ref{T:Diff_Mu_V}, we have
  \begin{multline} \label{E:Diff_Ue_Cv}
    \max_{t\in[0,T]}\|u^\varepsilon(t)-\bar{v}(t)\|_{L^2(\Omega_\varepsilon)}^2+\int_0^T\left\|\overline{P}\nabla u^\varepsilon(t)-\overline{\nabla_\Gamma v}(t)\right\|_{L^2(\Omega_\varepsilon)}^2\,dt \\
  \leq c_T\varepsilon\left\{\delta(\varepsilon)^2+\|M_\tau u_0^\varepsilon-v_0\|_{L^2(\Gamma)}^2+\|M_\tau\mathbb{P}_\varepsilon f^\varepsilon-f\|_{L^\infty(0,\infty;H^{-1}(\Gamma,T\Gamma))}^2\right\}
  \end{multline}
  for all $T>0$, where $c_T>0$ is a constant depending only on $T$.
  In particular,
  \begin{align*}
    \lim_{\varepsilon\to0}\varepsilon^{-1}\|u^\varepsilon-\bar{v}\|_{C([0,T];L^2(\Omega_\varepsilon))}^2 = \lim_{\varepsilon\to0}\varepsilon^{-1}\left\|\overline{P}\nabla u^\varepsilon-\overline{\nabla_\Gamma v}\right\|_{L^2(0,T;L^2(\Omega_\varepsilon))}^2 = 0
  \end{align*}
  for all $T>0$ provided that $\lim_{\varepsilon\to0}\varepsilon^{-1}\gamma_\varepsilon^i=\gamma^i$ for $i=0,1$ and
  \begin{align*}
    \lim_{\varepsilon\to0}\|M_\tau u_0^\varepsilon-v_0\|_{L^2(\Gamma)}^2 = \lim_{\varepsilon\to0}\|M_\tau\mathbb{P}_\varepsilon f^\varepsilon-f\|_{L^\infty(0,\infty;H^{-1}(\Gamma,T\Gamma))}^2 = 0.
  \end{align*}
\end{theorem}

\begin{proof}
  For the sake of simplicity, we denote by $R_\varepsilon$ the right-hand side of \eqref{E:Diff_Mu_V}.
  Also, we fix $T>0$ and suppress $t\in[0,T]$.
  Let us estimate $u^\varepsilon-\bar{v}$.
  Since
  \begin{align*}
    u^\varepsilon-\bar{v} = \Bigl(u^\varepsilon-\overline{Mu^\varepsilon}\Bigr)+\Bigl(\overline{Mu^\varepsilon}\cdot\bar{n}\Bigr)\bar{n}+\Bigl(\overline{M_\tau u^\varepsilon}-\bar{v}\Bigr) \quad\text{in}\quad \Omega_\varepsilon,
  \end{align*}
  we apply \eqref{E:Con_Lp}, \eqref{E:Ave_Diff_Dom}, and \eqref{E:Ave_N_Lp} to the right-hand side to get
  \begin{align*}
    \|u^\varepsilon-\bar{v}\|_{L^2(\Omega_\varepsilon)}^2 \leq c\varepsilon\left(\varepsilon\|u^\varepsilon\|_{H^1(\Omega_\varepsilon)}^2+\|M_\tau u^\varepsilon-v\|_{L^2(\Gamma)}^2\right).
  \end{align*}
  Hence we see by \eqref{E:Est_Ue}, \eqref{E:Diff_Mu_V}, and \eqref{E:DMuV_Re} that
  \begin{align} \label{Pf_DUeCv:L2}
    \|u^\varepsilon(t)-\bar{v}(t)\|_{L^2(\Omega_\varepsilon)}^2 \leq c\varepsilon(\varepsilon^\alpha+R_\varepsilon) \leq c\varepsilon R_\varepsilon \quad\text{for all}\quad t\in[0,T].
  \end{align}
  Next we consider the second term on the left-hand side of \eqref{E:Diff_Ue_Cv}.
  Let
  \begin{gather*}
    J_1 := \left\|\overline{P}\nabla u^\varepsilon-\overline{\nabla_\Gamma Mu^\varepsilon}\right\|_{L^2(\Omega_\varepsilon)}, \quad J_2 := \left\|\overline{\nabla_\Gamma[(Mu^\varepsilon\cdot n)n]}\right\|_{L^2(\Omega_\varepsilon)}, \\
    J_3 := \left\|\overline{\nabla_\Gamma M_\tau u^\varepsilon}-\overline{\nabla_\Gamma v}\right\|_{L^2(\Omega_\varepsilon)}.
  \end{gather*}
  By \eqref{E:Con_Lp} and \eqref{E:ADD_Dom} we have
  \begin{align*}
    J_1 \leq c\varepsilon\|u^\varepsilon\|_{H^2(\Omega_\varepsilon)}, \quad J_3 \leq c\varepsilon^{1/2}\|\nabla_\Gamma M_\tau u^\varepsilon-\nabla_\Gamma v\|_{L^2(\Gamma)}.
  \end{align*}
  Also, the inequalities \eqref{E:Con_Lp} and \eqref{E:Ave_N_W1p} imply that
  \begin{align*}
    J_2 \leq c\varepsilon^{1/2}\|Mu^\varepsilon\cdot n\|_{H^1(\Gamma)} \leq c\varepsilon\|u^\varepsilon\|_{H^2(\Omega_\varepsilon)}.
  \end{align*}
  From these inequalities we deduce that
  \begin{align*}
    \left\|\overline{P}\nabla u^\varepsilon-\overline{\nabla_\Gamma v}\right\|_{L^2(\Omega_\varepsilon)} &\leq J_1+J_2+J_3 \\
    &\leq c\varepsilon^{1/2}\left(\varepsilon^{1/2}\|u^\varepsilon\|_{H^2(\Omega_\varepsilon)}+\|\nabla_\Gamma M_\tau u^\varepsilon-\nabla_\Gamma v\|_{L^2(\Gamma)}\right).
  \end{align*}
  Thus, by \eqref{E:Est_Ue} and \eqref{E:Diff_Mu_V},
  \begin{align*}
    \int_0^T\left\|\overline{P}\nabla u^\varepsilon-\overline{\nabla_\Gamma v}\right\|_{L^2(\Omega_\varepsilon)}^2\,dt &\leq c\varepsilon\int_0^T\left(\varepsilon\|u^\varepsilon\|_{H^2(\Omega_\varepsilon)}^2+\|\nabla_\Gamma M_\tau u^\varepsilon-\nabla_\Gamma v\|_{L^2(\Gamma)}^2\right)dt \\
    &\leq c\varepsilon\{\varepsilon^\alpha(1+T)+R_\varepsilon\} \leq c\varepsilon(1+T)R_\varepsilon.
  \end{align*}
  Combining this inequality and \eqref{Pf_DUeCv:L2} we obtain \eqref{E:Diff_Ue_Cv}.
\end{proof}

We also compare the normal derivative (with respect to $\Gamma$) of a strong solution to \eqref{E:NS_Eq}--\eqref{E:NS_In} with a weak solution to \eqref{E:Limit_Eq}--\eqref{E:Limit_Div}.

\begin{theorem} \label{T:Diff_DnUe_Cv}
  Under the same assumptions as in Theorem~\ref{T:Diff_Mu_V}, we have
  \begin{multline} \label{E:Diff_DnUe_Cv}
    \int_0^T\left(\left\|\overline{P}\partial_nu^\varepsilon(t)+\overline{Wv}(t)\right\|_{L^2(\Omega_\varepsilon)}^2+\left\|\partial_nu^\varepsilon(t)\cdot\bar{n}-\frac{1}{\bar{g}}\bar{v}(t)\cdot\overline{\nabla_\Gamma g}\right\|_{L^2(\Omega_\varepsilon)}^2\right)dt \\
    \leq c_T\varepsilon\left\{\delta(\varepsilon)^2+\|M_\tau u_0^\varepsilon-v_0\|_{L^2(\Gamma)}^2+\|M_\tau\mathbb{P}_\varepsilon f^\varepsilon-f\|_{L^\infty(0,\infty;H^{-1}(\Gamma,T\Gamma))}^2\right\}
  \end{multline}
  for all $T>0$, where $\partial_nu^\varepsilon=(\bar{n}\cdot\nabla)u^\varepsilon$ is the normal derivative of $u^\varepsilon$ given by \eqref{E:Def_NorDer} and $c_T>0$ is a constant depending only on $T$.
  Hence, setting
  \begin{align*}
    V := -Wv+\frac{1}{g}(v\cdot\nabla_\Gamma g)n \quad\text{on}\quad \Gamma\times(0,\infty)
  \end{align*}
  we have (note that $Wv$ is tangential on $\Gamma$)
  \begin{align*}
    \lim_{\varepsilon\to0}\varepsilon^{-1}\left\|\partial_nu^\varepsilon-\overline{V}\right\|_{L^2(0,T;L^2(\Omega_\varepsilon))}^2 = 0
  \end{align*}
  for all $T>0$ provided that $\lim_{\varepsilon\to0}\varepsilon^{-1}\gamma_\varepsilon^i=\gamma^i$ for $i=0,1$ and
  \begin{align*}
    \lim_{\varepsilon\to0}\|M_\tau u_0^\varepsilon-v_0\|_{L^2(\Gamma)}^2 = \lim_{\varepsilon\to0}\|M_\tau\mathbb{P}_\varepsilon f^\varepsilon-f\|_{L^\infty(0,\infty;H^{-1}(\Gamma,T\Gamma))}^2 = 0.
  \end{align*}
\end{theorem}

\begin{proof}
  We fix $T>0$ and suppress $t\in[0,T]$.
  Let
  \begin{gather*}
    J_1 := \left\|\overline{P}\partial_nu^\varepsilon+\overline{W}u^\varepsilon\right\|_{L^2(\Omega_\varepsilon)}, \quad J_2 := \left\|\overline{WMu^\varepsilon}-\overline{W}u^\varepsilon\right\|_{L^2(\Omega_\varepsilon)}, \\
    J_3 := \left\|\overline{Wv}-\overline{WMu^\varepsilon}\right\|_{L^2(\Omega_\varepsilon)}.
  \end{gather*}
  We apply \eqref{E:PDnU_WU} to $J_1$ and \eqref{E:Ave_Diff_Dom} to $J_2$ to get
  \begin{align*}
    J_1 \leq c\varepsilon\|u^\varepsilon\|_{H^2(\Omega_\varepsilon)}, \quad J_2 \leq c\left\|u^\varepsilon-\overline{Mu^\varepsilon}\right\|_{L^2(\Omega_\varepsilon)} \leq c\varepsilon\|u^\varepsilon\|_{H^1(\Omega_\varepsilon)}.
  \end{align*}
  Also, noting that $WMu^\varepsilon=WM_\tau u^\varepsilon$ on $\Gamma$ by \eqref{E:Form_W}, we use \eqref{E:Con_Lp} to get
  \begin{align*}
    J_3 \leq c\varepsilon^{1/2}\|Wv-WM_\tau u^\varepsilon\|_{L^2(\Gamma)} \leq c\varepsilon^{1/2}\|v-M_\tau u^\varepsilon\|_{L^2(\Gamma)}.
  \end{align*}
  From these inequalities we deduce that
  \begin{align*}
    \left\|\overline{P}\partial_nu^\varepsilon+\overline{Wv}\right\|_{L^2(\Omega_\varepsilon)} &\leq J_1+J_2+J_3 \\
    &\leq c\varepsilon^{1/2}\left(\varepsilon^{1/2}\|u^\varepsilon\|_{H^2(\Omega_\varepsilon)}+\|M_\tau u^\varepsilon-v\|_{L^2(\Gamma)}\right).
  \end{align*}
  We also observe by \eqref{E:Width_Bound}, \eqref{E:Con_Lp}, and \eqref{E:DnU_N_Ave} that
  \begin{align*}
    &\left\|\partial_nu^\varepsilon\cdot\bar{n}-\frac{1}{\bar{g}}\bar{v}\cdot\overline{\nabla_\Gamma g}\right\|_{L^2(\Omega_\varepsilon)} \\
    &\qquad \leq \left\|\partial_nu^\varepsilon\cdot\bar{n}-\frac{1}{\bar{g}}\overline{M_\tau u^\varepsilon}\cdot\overline{\nabla_\Gamma g}\right\|_{L^2(\Omega_\varepsilon)}+\left\|\frac{1}{\bar{g}}\Bigl(\overline{M_\tau u^\varepsilon}-\bar{v}\Bigr)\cdot\overline{\nabla_\Gamma g}\right\|_{L^2(\Omega_\varepsilon)} \\
    &\qquad \leq c\varepsilon^{1/2}\left(\varepsilon^{1/2}\|u^\varepsilon\|_{H^2(\Omega_\varepsilon)}+\|M_\tau u^\varepsilon-v\|_{L^2(\Gamma)}\right).
  \end{align*}
  Hence, as in the proof of Theorem~\ref{T:Diff_Ue_Cv}, we integrate the square of the above inequalities over $(0,T)$ and then use \eqref{E:Est_Ue} and \eqref{E:Diff_Mu_V} to obtain \eqref{E:Diff_DnUe_Cv}.
\end{proof}

\begin{remark} \label{R:Diff_DnUe_Cv}
  In \eqref{E:Diff_DnUe_Cv} the Weingarten map $W$ represents the curvatures of the limit surface $\Gamma$.
  On the other hand, the functions $g_0$ and $g_1$ with $g=g_1-g_0$ are used to define the inner and outer boundaries of the curved thin domain $\Omega_\varepsilon$.
  Therefore, roughly speaking, the tangential component (with respect to $\Gamma$) of the normal derivative $\partial_nu^\varepsilon$ of the bulk velocity depends only on the shape of $\Gamma$, while the geometry of the boundaries of $\Omega_\varepsilon$ affects only the normal component of $\partial_nu^\varepsilon$.
\end{remark}

\begin{appendices}
\section{Notations and basic formulas on vectors and matrices} \label{S:Ap_VM}
We fix notations and give basic formulas on vectors and matrices, and use them to prove Lemmas~\ref{L:Korn_Aux} and~\ref{L:Tan_Curl_Ua}.

For a matrix $A\in\mathbb{R}^{3\times3}$ we denote by $A^T$ and $A_S:=(A+A^T)/2$ the transpose and the symmetric part of $A$, respectively.
We define the tensor product of vectors $a\in\mathbb{R}^l$ and $b\in\mathbb{R}^m$ with $l,m\in\mathbb{N}$ as
  \begin{align*}
    a\otimes b :=
    \begin{pmatrix}
      a_1b_1 & \cdots & a_1b_m \\
      \vdots & & \vdots \\
      a_lb_1 & \cdots & a_lb_m
    \end{pmatrix}, \quad
    a = (a_1,\dots,a_l), \, b = (b_1,\dots,b_m).
  \end{align*}
For vector fields $u=(u_1,u_2,u_3)$ and $\varphi$ on an open set in $\mathbb{R}^3$ we write
\begin{gather*}
  \nabla u :=
  \begin{pmatrix}
    \partial_1u_1 & \partial_1u_2 & \partial_1u_3 \\
    \partial_2u_1 & \partial_2u_2 & \partial_2u_3 \\
    \partial_3u_1 & \partial_3u_2 & \partial_3u_3
  \end{pmatrix}, \quad
  |\nabla^2u|^2 := \sum_{i,j,k=1}^3|\partial_i\partial_ju_k|^2 \quad\left(\partial_i := \frac{\partial}{\partial x_i}\right), \\
  (\varphi\cdot\nabla)u := (\varphi\cdot\nabla u_1,\varphi\cdot\nabla u_2,\varphi\cdot\nabla u_3) = (\nabla u)^T\varphi.
\end{gather*}
We define the inner product of matrices $A,B\in\mathbb{R}^{3\times3}$ and the norm of $A$ as
\begin{align*}
  A: B := \mathrm{tr}[A^TB] = \sum_{i=1}^3AE_i\cdot BE_i, \quad |A| := \sqrt{A:A},
\end{align*}
where $\{E_1,E_2,E_3\}$ is an orthonormal basis of $\mathbb{R}^3$.
Note that $A:B$ does not depend on a choice of $\{E_1,E_2,E_3\}$.
In particular, taking the standard basis of $\mathbb{R}^3$ we get
\begin{align*}
  A:B = \sum_{i,j=1}^3A_{ij}B_{ij} = B:A = A^T:B^T, \quad AB:C = A:CB^T = B:A^TC
\end{align*}
for $A,B,C\in\mathbb{R}^{3\times3}$.
Also, for $a,b\in\mathbb{R}^3$ we have $|a\otimes b|=|a||b|$.

\begin{lemma} \label{L:Norm_Mat}
  Let $n_0,\tau_1,\tau_2\in\mathbb{R}^3$ and $A_1,A_2\in\mathbb{R}^{3\times3}$ satisfy
  \begin{align*}
    |n_0| = 1, \quad n_0\cdot\tau_1 = n_0\cdot\tau_2 = 0, \quad A_1^Tn_0 = 0, \quad A_2n_0 = A_2^Tn_0 = 0.
  \end{align*}
  Then for $B:=A_1+n_0\otimes a$ and $C:=A_2+\tau_1\otimes n_0+n_0\otimes\tau_2+cn_0\otimes n_0$ with $a\in\mathbb{R}^3$ and $c\in\mathbb{R}$ we have
  \begin{align} \label{E:Norm_PA}
    |B|^2 = |A_1|^2+|a|^2, \quad |C|^2 = |A_2|^2+|\tau_1|^2+|\tau_2|^2+|c|^2.
  \end{align}
\end{lemma}

\begin{proof}
  For $B=A_1+n_0\otimes a$ we use $|n_0|=1$ and $A_1^Tn_0=0$ to get
  \begin{align*}
    B^TB = A_1^TA_1+(A_1^Tn_0)\otimes a+a\otimes(A_1^Tn_0)+|n_0|^2a\otimes a = A_1^TA_1+a\otimes a.
  \end{align*}
  By this equality and $\mathrm{tr}[a\otimes a]=|a|^2$ we obtain the first equality of \eqref{E:Norm_PA}.
  Similarly, for $C=A_2+\tau_1\otimes n_0+n_0\otimes\tau_2+cn_0\otimes n_0$ we have
  \begin{multline*}
    C^TC = A_2^TA_2+\tau_2\otimes\tau_2+(|\tau_1|^2+|c|^2)n_0\otimes n_0 \\
    +(A_2^T\tau_1)\otimes n_0+n_0\otimes(A_2^T\tau_1)+c(\tau_2\otimes n_0+n_0\otimes\tau_2)
  \end{multline*}
  by $|n_0|=1$, $n_0\cdot\tau_1=0$, and $A_2^Tn_0=0$.
  From this equality and
  \begin{gather*}
    \mathrm{tr}[\tau_2\otimes\tau_2] = |\tau_2|^2, \quad \mathrm{tr}[n_0\otimes n_0] = |n_0|^2 = 1, \\
    \mathrm{tr}[\tau_2\otimes n_0] = \mathrm{tr}[n_0\otimes\tau_2] = n_0\cdot\tau_2 = 0, \\
    \mathrm{tr}[(A_2^T\tau_1)\otimes n_0] = \mathrm{tr}[n_0\otimes(A_2^T\tau_1)] = (A_2^T\tau_1)\cdot n_0 = \tau_1\cdot(A_2n_0) = 0
  \end{gather*}
  by $A_2n_0=0$ the second equality of \eqref{E:Norm_PA} follows.
\end{proof}

Based on the formulas \eqref{E:Norm_PA} we prove Lemma~\ref{L:Korn_Aux}.

\begin{proof}[Proof of Lemma~\ref{L:Korn_Aux}]
  Let $\Phi_\varepsilon$ be the bijection from $\Omega_1$ onto $\Omega_\varepsilon$ given by \eqref{E:Def_Korn_Aux}.
  Since its inverse is of the form $\Phi_\varepsilon^{-1}(x)=\pi(x)+\varepsilon^{-1}d(x)\bar{n}(x)$ for $x\in\Omega_\varepsilon$,
  \begin{align*}
    \nabla\Phi_\varepsilon^{-1}(x) = \left\{I_3-d(x)\overline{W}(x)\right\}^{-1}\left\{I_3-\varepsilon^{-1}d(x)\overline{W}(x)\right\}\overline{P}(x)+\varepsilon^{-1}\overline{Q}(x)
  \end{align*}
  for $x\in\Omega_\varepsilon$ by \eqref{E:Form_W}, \eqref{E:Pi_Der}, \eqref{E:Nor_Grad}, and $\nabla d(x)=\bar{n}(x)$.
  For $X\in\Omega_1$ let $x=\Phi_\varepsilon(X)$ in this equality.
  Then by $d(\Phi_\varepsilon(X))=\varepsilon d(X)$ and $\pi(\Phi_\varepsilon(X))=\pi(X)$ we get
  \begin{gather*}
    \nabla\Phi_\varepsilon^{-1}(\Phi_\varepsilon(X))=\Lambda_\varepsilon(X)\overline{P}(X)+\varepsilon^{-1}\overline{Q}(X), \\
    \Lambda_\varepsilon(X) := \left\{I_3-\varepsilon d(X)\overline{W}(X)\right\}^{-1}\left\{I_3-d(X)\overline{W}(X)\right\}.
  \end{gather*}
  For $u\in H^1(\Omega_\varepsilon)^3$ let $U:=u\circ\Phi_\varepsilon\colon\Omega_1\to\mathbb{R}^3$.
  From the above formula and
  \begin{align*}
    \nabla u(\Phi_\varepsilon(X)) = \nabla\Phi_\varepsilon^{-1}(\Phi_\varepsilon(X))\nabla U(X), \quad \overline{Q}(X)\nabla U(X) = \bar{n}(X)\otimes\partial_n U(X)
  \end{align*}
  it follows that
  \begin{align} \label{Pf_KA:Grad_U}
    (\nabla u\circ\Phi_\varepsilon)(X) = \Lambda_\varepsilon(X)\overline{P}(X)\nabla U(X)+\varepsilon^{-1}\bar{n}(X)\otimes\partial_nU(X), \quad X\in\Omega_1.
  \end{align}
  Let $n_0:=\bar{n}$, $A_1:=\Lambda_\varepsilon\overline{P}\nabla U$, and $a:=\varepsilon^{-1}\partial_nU$ so that $(\nabla u)\circ\Phi_\varepsilon=A_1+n_0\otimes a$ in $\Omega_1$.
  Then $|n_0|=1$ and $A_1^Tn_0=0$ in $\Omega_1$ since $P^Tn=Pn=0$ on $\Gamma$ and $\Lambda_\varepsilon\overline{P}=\overline{P}\Lambda_\varepsilon$ in $\Omega_1$ by \eqref{E:Form_W} and \eqref{E:WReso_P}.
  Hence we can use the first equality of \eqref{E:Norm_PA} to get
  \begin{align*}
    |(\nabla u)\circ\Phi_\varepsilon|^2 = \left|\Lambda_\varepsilon\overline{P}\nabla U\right|^2+\varepsilon^{-2}|\partial_nU|^2 \geq c\left(\left|\overline{P}\nabla U\right|^2+\varepsilon^{-2}|\partial_nU|^2\right) \quad\text{in}\quad \Omega_1,
  \end{align*}
  where the second inequality follows from \eqref{E:Wein_Bound}.
  By this inequality and \eqref{E:L2_Omega_1},
  \begin{align*}
    \varepsilon^{-1}\|\nabla u\|_{L^2(\Omega_\varepsilon)}^2 \geq c\|(\nabla u)\circ\Phi_\varepsilon\|_{L^2(\Omega_1)}^2 \geq c\left(\left\|\overline{P}\nabla U\right\|_{L^2(\Omega_1)}^2+\varepsilon^{-2}\|\partial_nU\|_{L^2(\Omega_1)}^2\right).
  \end{align*}
  Hence \eqref{E:KAux_Grad} holds.
  We also have $U\in H^1(\Omega_1)^3$ by $u\in H^1(\Omega_\varepsilon)^3$, \eqref{E:L2_Omega_1}, \eqref{E:KAux_Grad}, and
  \begin{align*}
    |\nabla U|^2 = \left|\overline{P}\nabla U\right|^2+\left|\overline{Q}\nabla U\right|^2, \quad \left|\overline{Q}\nabla U\right| = |\bar{n}\otimes\partial_nU| = |\partial_nU| \quad\text{in}\quad \Omega_1.
  \end{align*}
  To prove \eqref{E:KAux_Du} we observe by $I_3=P+Q$ on $\Gamma$ and \eqref{E:NorDer_Con} that
  \begin{align*}
    \overline{P}\nabla U &= \overline{P}(\nabla U)\overline{P}+\overline{P}(\nabla U)\overline{Q} = \overline{P}(\nabla U)\overline{P}+\Bigl[\overline{P}(\nabla U)\bar{n}\Bigr]\otimes\bar{n}, \\
    \partial_nU &= \partial_n\Bigl[\overline{P}U+(U\cdot\bar{n})\bar{n}\Bigr] = \overline{P}\partial_nU+\{\partial_n(U\cdot\bar{n})\}\bar{n}
  \end{align*}
  in $\Omega_1$.
  We apply these equalities to \eqref{Pf_KA:Grad_U} and then use $\Lambda_\varepsilon\overline{P}=\overline{P}\Lambda_\varepsilon$ in $\Omega_1$ by \eqref{E:Form_W} and \eqref{E:WReso_P} to obtain
  \begin{multline*}
    (\nabla u)\circ\Phi_\varepsilon = \overline{P}\Lambda_\varepsilon(\nabla U)\overline{P}+\Bigl[\overline{P}\Lambda_\varepsilon(\nabla U)\bar{n}\Bigr]\otimes\bar{n} \\
    +\varepsilon^{-1}\bar{n}\otimes\Bigl(\overline{P}\partial_nU\Bigr)+\varepsilon^{-1}\{\partial_n(U\cdot\bar{n})\}\bar{n}\otimes\bar{n} \quad\text{in}\quad \Omega_1.
  \end{multline*}
  Hence $D(u)\circ\Phi_\varepsilon=\{(\nabla u)\circ\Phi_\varepsilon+(\nabla u)^T\circ\Phi_\varepsilon\}/2$ is of the form
  \begin{gather*}
    D(u)\circ\Phi_\varepsilon = A_2+\tau\otimes n_0+n_0\otimes\tau+\varepsilon^{-1}\{\partial_n(U\cdot\bar{n})\}n_0\otimes n_0 \quad\text{in}\quad \Omega_1, \\
    n_0 := \bar{n}, \quad A_2 := \overline{P}F_\varepsilon(U)_S\overline{P}, \quad \tau := \frac{1}{2}\overline{P}\{\Lambda_\varepsilon(\nabla U)\bar{n}+\varepsilon^{-1}\partial_nU\}.
  \end{gather*}
  Here $F_\varepsilon(U)_S$ is the symmetric part of the matrix $F_\varepsilon(U):=\Lambda_\varepsilon\nabla U$.
  Since $\tau\cdot n_0=0$ and $A_2n_0=A_2^Tn_0=0$ in $\Omega_1$ by $P^Tn=Pn=0$ on $\Gamma$, we can apply the second equality of \eqref{E:Norm_PA} to $C=D(u)\circ\Phi_\varepsilon$ to obtain
  \begin{align*}
    |D(u)\circ\Phi_\varepsilon|^2 = |A_2|^2+2|\tau|^2+\varepsilon^{-2}|\partial_n(U\cdot\bar{n})|^2 \geq |A_2|^2+\varepsilon^{-2}|\partial_n(U\cdot\bar{n})|^2
  \end{align*}
  in $\Omega_1$.
  From this inequality and \eqref{E:L2_Omega_1} we deduce that
  \begin{align*}
    \varepsilon^{-1}\|D(u)\|_{L^2(\Omega_\varepsilon)}^2 &\geq c\|D(u)\circ\Phi_\varepsilon\|_{L^2(\Omega_1)}^2 \geq c\left(\left\|A_2\right\|_{L^2(\Omega_1)}^2+\varepsilon^{-2}\|\partial_n(U\cdot\bar{n})\|_{L^2(\Omega_1)}^2\right).
  \end{align*}
  This shows \eqref{E:KAux_Du} since $A_2=\overline{P}F_\varepsilon(U)_S\overline{P}$ and $F_\varepsilon(U)$ is of the form \eqref{E:KAux_Matrix}.
\end{proof}

Next we give a formula on the curl of a vector field and show Lemma~\ref{L:Tan_Curl_Ua}.

\begin{lemma} \label{L:Curl_Exp}
  Let $E_1$, $E_2$, and $E_3$ be vector fields on an open subset $U$ of $\mathbb{R}^3$ such that $\{E_1(x),E_2(x),E_3(x)\}$ is an orthonormal basis of $\mathbb{R}^3$ for each $x\in\mathbb{R}^3$ and
  \begin{align*}
    E_1\times E_2 = E_3, \quad E_2\times E_3 = E_1, \quad E_3\times E_1 = E_2 \quad\text{in}\quad U.
  \end{align*}
  Then for $u\in C^1(U)^3$ we have
  \begin{multline} \label{E:Curl_Exp}
    \mathrm{curl}\,u = \{(E_2\cdot\nabla)u\cdot E_3-(E_3\cdot\nabla)u\cdot E_2\}E_1 \\
    +\{(E_3\cdot\nabla)u\cdot E_1-(E_1\cdot\nabla)u\cdot E_3\}E_2 \\
    +\{(E_1\cdot\nabla)u\cdot E_2-(E_2\cdot\nabla)u\cdot E_1\}E_3 \quad\text{in}\quad U.
  \end{multline}
\end{lemma}

\begin{proof}
  By the assumption, $\mathrm{curl}\,u=\sum_{i=1}^3(\mathrm{curl}\,u\cdot E_i)E_i$.
  Since $E_1=E_2\times E_3$,
  \begin{align*}
    \mathrm{curl}\,u\cdot E_1 &= \mathrm{curl}\,u\cdot(E_2\times E_3) = E_2\cdot(E_3\times\mathrm{curl}\,u) \\
    &= E_2\cdot\{(\nabla u)E_3-(\nabla u)^TE_3\} \\
    &= (\nabla u)^TE_2\cdot E_3-(\nabla u)^TE_3\cdot E_2 \\
    &= (E_2\cdot\nabla)u\cdot E_3-(E_3\cdot\nabla)u\cdot E_2.
  \end{align*}
  Calculating $\mathrm{curl}\,u\cdot E_i$, $i=2,3$ in the same way we obtain \eqref{E:Curl_Exp}.
\end{proof}

\begin{proof}[Proof of Lemma~\ref{L:Tan_Curl_Ua}]
  Let $u\in C^1(\Omega_\varepsilon)^3$ and $u^a:=E_\varepsilon M_\tau u$ be given by \eqref{E:Def_ExAve}.
  Since the surface $\Gamma$ is compact, we can take finite relatively open subsets $O_k$ of $\Gamma$ and pairs of tangential vector fields $\{\tau_1^k,\tau_2^k\}$ on $O_k$, $k=1,\dots,k_0$ such that $\Gamma=\bigcup_{k=1}^{k_0}O_k$, the triplet $\{\tau_1^k,\tau_2^k,n\}$ forms an orthonormal basis of $\mathbb{R}^3$ on $O_k$, and
  \begin{align*}
    \tau_1^k\times\tau_2^k = n, \quad \tau_2^k\times n = \tau_1^k, \quad n\times\tau_1^k = \tau_2^k \quad\text{on}\quad O_k
  \end{align*}
  for each $k=1,\dots,k_0$.
  Then since $\Omega_\varepsilon=\bigcup_{k=1}^{k_0}U_k$ with
  \begin{align*}
    U_k:=\{y+rn(y)\mid y\in O_k,\,r\in(\varepsilon g_0(y),\varepsilon g_1(y))\}, \quad k=1,\dots,k_0,
  \end{align*}
  it is sufficient to show \eqref{E:Tan_Curl_Ua} in $U_k$ for each $k=1,\dots,k_0$.
  From now on, we fix and suppress $k$ and carry out calculations in $U$ unless otherwise stated.
  We apply \eqref{E:Curl_Exp} to $u^a$ with $E_1:=\bar{\tau}_1$, $E_2:=\bar{\tau}_2$, and $E_3:=\bar{n}$.
  Then since $P\tau_i=\tau_i$ for $i=1,2$ and $Pn=0$ on $O$, we have
  \begin{align*}
    \overline{P}\,\mathrm{curl}\,u^a = \{(\bar{\tau}_2\cdot\nabla)u^a\cdot\bar{n}-(\bar{n}\cdot\nabla)u^a\cdot\bar{\tau}_2\}\bar{\tau}_1+\{(\bar{n}\cdot\nabla)u^a\cdot\bar{\tau}_1-(\bar{\tau}_1\cdot\nabla)u^a\cdot\bar{n}\}\bar{\tau}_2.
  \end{align*}
  By this equality, $(\bar{n}\cdot\nabla)u^a=\partial_nu^a$, and $|\bar{\tau}_1|=|\bar{\tau}_2|=|\bar{n}|=1$ we get
  \begin{align} \label{Pf_TCUa:Est}
    \left|\overline{P}\,\mathrm{curl}\,u^a\right| \leq c\left(|\partial_nu^a|+|(\bar{\tau}_1\cdot\nabla)u^a\cdot\bar{n}|+|(\bar{\tau}_2\cdot\nabla)u^a\cdot\bar{n}|\right).
  \end{align}
  Let us estimate each term on the right-hand side.
  By \eqref{E:NorDer_Con}, \eqref{E:ExAux_Bound}, and \eqref{E:Def_ExAve},
  \begin{align} \label{Pf_TCUa:DN}
    |\partial_nu^a| = \left|\overline{M_\tau u}\cdot\partial_n\Psi_\varepsilon\right| \leq c\left|\overline{M_\tau u}\right| = c\left|\overline{PMu}\right| \leq c\left|\overline{Mu}\right|.
  \end{align}
  To estimate the other terms we set
  \begin{align*}
    u_\tau^a := \overline{P}u^a = \overline{M_\tau u}, \quad u_n^a := (u^a\cdot\bar{n})\bar{n} = \Bigl(\overline{M_\tau u}\cdot\Psi_\varepsilon\Bigr)\bar{n}
  \end{align*}
  so that $u^a=u_\tau^a+u_n^a$.
  Let $i=1,2$.
  Since $u_\tau^a\cdot\bar{n}=0$, we have
  \begin{align*}
    (\bar{\tau}_i\cdot\nabla)u_\tau^a\cdot\bar{n} = (\bar{\tau}_i\cdot\nabla)(u_\tau^a\cdot\bar{n})-u_\tau^a\cdot(\bar{\tau}_i\cdot\nabla)\bar{n} = -u_\tau^a\cdot(\bar{\tau}_i\cdot\nabla)\bar{n}.
  \end{align*}
  Hence by \eqref{E:NorG_Bound} and $|\bar{\tau}_i|=1$ we get
  \begin{align} \label{Pf_TCUa:Utaua}
    |(\bar{\tau}_i\cdot\nabla)u_\tau^a\cdot\bar{n}| \leq c|u_\tau^a| \leq c\left|\overline{Mu}\right|, \quad i=1,2.
  \end{align}
  Next we deal with $(\bar{\tau}_i\cdot\nabla)u_n^a\cdot\bar{n}$.
  Since $\tau_i=P\tau_i$, $P=P^T$, and $|\tau_i|=1$ on $O$,
  \begin{align*}
    |(\bar{\tau}_i\cdot\nabla)u_n^a| = \left|(\nabla u_n^a)^T\overline{P}\bar{\tau}_i\right| = \left|\Bigl[\overline{P}(\nabla u_n^a)\Bigr]^T\bar{\tau}_i\right| \leq \left|\overline{P}(\nabla u_n^a)\right|.
  \end{align*}
  Moreover, by $u_n^a=(\overline{M_\tau u}\cdot\Psi_\varepsilon)\bar{n}$ we have
  \begin{align*}
    \overline{P}(\nabla u_n^a) = \left[\left\{\overline{P}\nabla\Bigl(\overline{M_\tau u}\Bigr)\right\}\Psi_\varepsilon+\Bigl(\overline{P}\nabla\Psi_\varepsilon\Bigr)\overline{M_\tau u}\right]\otimes\bar{n}+\Bigl(\overline{M_\tau u}\cdot\Psi_\varepsilon\Bigr)\overline{P}\nabla\bar{n}
  \end{align*}
  and thus the inequalities \eqref{E:ConDer_Bound}, \eqref{E:NorG_Bound}, and \eqref{E:ExAux_Bound}--\eqref{E:ExAux_TNDer} imply that
  \begin{align*}
    \left|\overline{P}(\nabla u_n^a)\right| \leq c\varepsilon\left(\left|\overline{M_\tau u}\right|+\left|\overline{\nabla_\Gamma M_\tau u}\right|\right) \leq c\varepsilon\left(\left|\overline{Mu}\right|+\left|\overline{\nabla_\Gamma Mu}\right|\right).
  \end{align*}
  Here the last inequality follows from $M_\tau u=PMu$ on $\Gamma$ and $P\in C^4(\Gamma)^{3\times3}$.
  Hence
  \begin{align} \label{Pf_TCUa:Uan}
    |(\bar{\tau}_i\cdot\nabla)u_n^a\cdot\bar{n}| \leq c\varepsilon\left(\left|\overline{Mu}\right|+\left|\overline{\nabla_\Gamma Mu}\right|\right), \quad i=1,2.
  \end{align}
  Noting that $u^a=u_\tau^a+u_n^a$, we apply \eqref{Pf_TCUa:DN}--\eqref{Pf_TCUa:Uan} to \eqref{Pf_TCUa:Est} to obtain \eqref{E:Tan_Curl_Ua} in $U$.
\end{proof}

\section{Calculations involving differential geometry of surfaces} \label{S:Ap_DG}
The purpose of this appendix is to provide the proofs of the lemmas in Section~\ref{S:Pre} and related results, which involve calculations of the surface quantities on $\Gamma$, $\Gamma_\varepsilon^0$, and $\Gamma_\varepsilon^1$.
We also show the formula \eqref{Pf_A:Det_Zeta} in the proof of Lemma~\ref{L:Agmon}.

Let $\Gamma$ be a two-dimensional closed, connected, and oriented surface in $\mathbb{R}^3$ of class $C^\ell$, $\ell\geq2$.
First we give auxiliary inequalities for the Riemannian metric of $\Gamma$.

\begin{lemma} \label{L:Metric}
  Let $U$ be an open set in $\mathbb{R}^2$, $\mu\colon U\to\Gamma$ a $C^\ell$ local parametrization of $\Gamma$, and $\mathcal{K}$ a compact subset of $U$.
  Then there exists a constant $c>0$ such that
  \begin{align} \label{E:Mu_Bound}
    |\partial_{s_i}\mu(s)| \leq c, \quad |\partial_{s_i}\partial_{s_j}\mu(s)| \leq c \quad\text{for all}\quad s\in\mathcal{K},\,i,j=1,2.
  \end{align}
  We define the Riemannian metric $\theta=(\theta_{ij})_{i,j}$ of $\Gamma$ by
  \begin{align} \label{E:Def_Met}
    \theta(s) := \nabla_s\mu(s)\{\nabla_s\mu(s)\}^T, \quad s\in U, \quad \nabla_s\mu :=
    \begin{pmatrix}
      \partial_{s_1}\mu_1 & \partial_{s_1}\mu_2 & \partial_{s_1}\mu_3 \\
      \partial_{s_2}\mu_1 & \partial_{s_2}\mu_2 & \partial_{s_2}\mu_3
    \end{pmatrix}
  \end{align}
  and denote by $\theta^{-1}=(\theta^{ij})_{i,j}$ the inverse matrix of $\theta$.
  Then
  \begin{align} \label{E:Metric}
      |\theta^k(s)| \leq c, \quad |\partial_{s_i}\theta^k(s)| \leq c, \quad c^{-1} \leq \det\theta(s) \leq c
  \end{align}
  for all $s\in\mathcal{K}$, $i=1,2$, and $k=\pm1$.

  Let $p\in[1,\infty)$.
  If $\eta\in L^p(\Gamma)$ is supported in $\mu(\mathcal{K})$, then $\tilde{\eta}:=\eta\circ\mu\in L^p(U)$ and
  \begin{align} \label{E:Lp_Loc}
    c^{-1}\|\tilde{\eta}\|_{L^p(U)} \leq \|\eta\|_{L^p(\Gamma)} \leq c\|\tilde{\eta}\|_{L^p(U)}.
  \end{align}
  If in addition $\eta\in W^{1,p}(\Gamma)$, then $\tilde{\eta}\in W^{1,p}(U)$ and
  \begin{align} \label{E:W1p_Loc}
    c^{-1}\|\nabla_s\tilde{\eta}\|_{L^p(U)} \leq \|\nabla_\Gamma\eta\|_{L^p(\Gamma)} \leq c\|\nabla_s\tilde{\eta}\|_{L^p(U)}.
  \end{align}
  Here $\nabla_s$ is the gradient operator in $s\in\mathbb{R}^2$.
\end{lemma}

\begin{proof}
  The inequalities \eqref{E:Mu_Bound} follow from the $C^\ell$-regularity of $\mu$ on $U$ and the compactness of $\mathcal{K}$.
  Using them and $\partial_{s_i}\theta^{-1}=-\theta^{-1}(\partial_{s_i}\theta)\theta^{-1}$ in $U$ we obtain the first and second inequalities of \eqref{E:Metric}.
  Also, the third inequality is valid since $\det\theta$ is continuous and strictly positive on $U$ and $\mathcal{K}$ is compact in $U$.

  Let $\eta\in L^p(\Gamma)$, $p\in[1,\infty)$ be supported in $\mu(\mathcal{K})$ and $\tilde{\eta}:=\eta\circ\mu$ on $U$.
  Since
  \begin{align} \label{Pf_Met:Int}
    \int_\Gamma|\eta(y)|^p\,d\mathcal{H}^2(y) = \int_U|\tilde{\eta}(s)|^p\sqrt{\det\theta(s)}\,ds
  \end{align}
  by the definition of an integral over a surface, we have \eqref{E:Lp_Loc} by \eqref{E:Metric}.

  To prove \eqref{E:W1p_Loc} let us show
  \begin{align} \label{Pf_Met:Aux}
    c^{-1}|a|^2 \leq \theta^{-1}(s)a\cdot a \leq c|a|^2, \quad s\in\mathcal{K},\,a\in\mathbb{R}^2.
  \end{align}
  For $s\in U$ and $a=(a_1,a_2)\in\mathbb{R}^2$ we set $X(s,a):=\sum_{i,j=1}^2\theta^{ij}(s)a_i\partial_{s_j}\mu(s)$.
  Since $\partial_{s_1}\mu(s)$ and $\partial_{s_2}\mu(s)$ are linearly independent, $X(s,a)$ vanishes if and only if
  \begin{align*}
    \sum_{i=1,2}\theta^{ij}(s)a_i = \sum_{i=1,2}\theta^{ji}(s)a_i = 0 \quad\text{for}\quad j=1,2, \quad\text{i.e.}\quad \theta^{-1}(s)a = 0,
  \end{align*}
  which is equivalent to $a=0$ (note that $\theta^{-1}$ is symmetric since $\theta$ is so).
  Thus
  \begin{align*}
    |X(s,a)|^2 = \sum_{i,j=1}^2\theta^{ij}(s)a_ia_j = \theta^{-1}(s)a\cdot a, \quad (s,a)\in U\times\mathbb{R}^2
  \end{align*}
  is continuous and does not vanish for $a\neq0$.
  In particular, it is bounded from above and below by positive constants on the compact set $\mathcal{K}\times S^1$, where $S^1$ is the unit circle in $\mathbb{R}^2$.
  Hence \eqref{Pf_Met:Aux} follows.
  Now let $\eta\in W^{1,p}(\Gamma)$ be supported in $\mu(\mathcal{K})$ and $\tilde{\eta}=\eta\circ\mu$ on $U$.
  Since
  \begin{align} \label{Pf_Met:Grad}
    \begin{aligned}
      \nabla_\Gamma\eta(\mu(s)) &= \sum_{i,j=1}^2\theta^{ij}(s)\partial_{s_i}\tilde{\eta}(s)\partial_{s_j}\mu(s), \\
      |\nabla_\Gamma\eta(\mu(s))|^2 &= \sum_{i,j=1}^2\theta^{ij}(s)\partial_{s_i}\tilde{\eta}(s)\partial_{s_j}\tilde{\eta}(s) = \theta^{-1}(s)\nabla_s\tilde{\eta}(s)\cdot\nabla_s\tilde{\eta}(s)
    \end{aligned}
  \end{align}
  for $s\in U$, the inequality \eqref{Pf_Met:Aux} yields
  \begin{align*}
    c^{-1}|\nabla_s\tilde{\eta}(s)| \leq |\nabla_\Gamma\eta(\mu(s))| \leq c|\nabla_s\tilde{\eta}(s)|, \quad s\in\mathcal{K}.
  \end{align*}
  Since $\eta$ is supported in $\mu(\mathcal{K})$, we obtain \eqref{E:W1p_Loc} by this inequality and \eqref{Pf_Met:Int}.
\end{proof}

Hereafter we always write $\theta=(\theta_{ij})_{i,j}$ and $\theta^{-1}=(\theta^{ij})_{i,j}$ for the Riemannian metric of $\Gamma$ given by \eqref{E:Def_Met} and its inverse.
Let us prove the lemmas in Section~\ref{SS:Pre_Surf}.

\begin{proof}[Proof of Lemma~\ref{L:Wein}]
  Since $W$ has the eigenvalues zero, $\kappa_1$, and $\kappa_2$,
  \begin{align*}
    \det[I_3-rW(y)] = \{1-r\kappa_1(y)\}\{1-r\kappa_2(y)\} > 0, \quad y\in\Gamma,\,r\in(-\delta,\delta)
  \end{align*}
  by \eqref{E:Curv_Bound}.
  Hence $I_3-rW(y)$ is invertible.
  Also, the equality \eqref{E:WReso_P} follows from \eqref{E:Form_W}.

  Let us prove \eqref{E:Wein_Bound} and \eqref{E:Wein_Diff}.
  We fix and suppress $y\in\Gamma$.
  Since $W$ is real and symmetric by Lemma~\ref{L:Form_W} and has the eigenvalues $\kappa_1$, $\kappa_2$, and zero with $Wn=0$, we can take an orthonormal basis $\{\tau_1,\tau_2,n\}$ of $\mathbb{R}^3$ such that $W\tau_i=\kappa_i\tau_i$, $i=1,2$.
  Then for $r\in(-\delta,\delta)$, $i=1,2$, and $k=\pm1$ we have
  \begin{align} \label{Pf_We:Inv}
    (I_3-rW)^k\tau_i = (1-r\kappa_i)^k\tau_i, \quad (I_3-rW)^kn = n.
  \end{align}
  Since $\{\tau_1,\tau_2,n\}$ is an orthonormal basis of $\mathbb{R}^3$, these formulas imply that
  \begin{align*}
    (I_3-rW)^ka &= \sum_{i=1,2}(a\cdot\tau_i)(I_3-rW)^k\tau_i+(a\cdot n)(I_3-rW)^kn \\
    &= \sum_{i=1,2}(a\cdot\tau_i)(1-r\kappa_i)^k\tau_i+(a\cdot n)n
  \end{align*}
  for all $a\in\mathbb{R}^3$ and $k=\pm1$.
  Hence
  \begin{align*}
    \bigl|(I_3-rW)^ka\bigr|^2 = \sum_{i=1,2}(a\cdot\tau_i)^2(1-r\kappa_i)^{2k}+(a\cdot n)^2
  \end{align*}
  and \eqref{E:Wein_Bound} follows from \eqref{E:Curv_Bound} and $|a|^2=(a\cdot\tau_1)^2+(a\cdot\tau_2)^2+(a\cdot n)^2$.
  Also,
  \begin{align*}
    \bigl|I_3-(I_3-rW)^{-1}\bigr|^2 = \sum_{i=1,2}|1-(1-r\kappa_i)^{-1}|^2 = \sum_{i=1,2}|r\kappa_i(1-r\kappa_i)^{-1}|^2 \leq c|r|^2
  \end{align*}
  by \eqref{Pf_We:Inv}, $|\tau_1|=|\tau_2|=1$, and \eqref{E:Curv_Bound}.
  Hence \eqref{E:Wein_Diff} is valid.
\end{proof}

\begin{proof}[Proof of Lemma~\ref{L:Pi_Der}]
  Let $x\in N$.
  Since $\pi(x)\in\Gamma$,
  \begin{gather*}
    \pi(x) = x-d(x)n(\pi(x)) = x-d(x)\bar{n}(\pi(x)), \\
    -\nabla\bar{n}(\pi(x)) = -\overline{\nabla_\Gamma n}(\pi(x)) = \overline{W}(\pi(x)) = \overline{W}(x)
  \end{gather*}
  by \eqref{E:Nor_Coord} and \eqref{E:ConDer_Surf}.
  We differentiate both sides of the first equality with respect to $x$ and use the second equality and $\nabla d(x)=\bar{n}(x)=\bar{n}(\pi(x))$ to get
  \begin{align*}
    \nabla\pi(x) = I_3-\bar{n}(x)\otimes\bar{n}(x)-d(x)\nabla\pi(x)\nabla\bar{n}(\pi(x)) = \overline{P}(x)+d(x)\nabla\pi(x)\overline{W}(x)
  \end{align*}
  for $x\in N$ and thus
  \begin{align*}
    \pi(x)\left\{I_3-d(x)\overline{W}(x)\right\} = \overline{P}(x), \quad x\in N.
  \end{align*}
  The equality \eqref{E:Pi_Der} follows from this equality and \eqref{E:WReso_P} since $I_3-d\overline{W}$ is invertible in $N$ by Lemma~\ref{L:Wein}.
  Also, we get \eqref{E:ConDer_Dom} by \eqref{E:P_TGr} and \eqref{E:Pi_Der}.
  The inequalities \eqref{E:ConDer_Bound} and \eqref{E:ConDer_Diff} follow from \eqref{E:Wein_Bound}, \eqref{E:Wein_Diff}, and \eqref{E:ConDer_Dom}.

  Now let $\Gamma$ be of class $C^3$ and $\eta\in C^2(\Gamma)$.
  For $i=1,2,3$ we differentiate both sides of \eqref{E:ConDer_Dom} with respect to $x_i$ to get
  \begin{align} \label{Pf_PD:Second}
    \partial_i\nabla\bar{\eta} = \left\{\partial_i\Bigl(I_3-d\overline{W}\Bigr)^{-1}\right\}\overline{\nabla_\Gamma\eta}+\Bigl(I_3-d\overline{W}\Bigr)^{-1}\partial_i\Bigl(\overline{\nabla_\Gamma\eta}\Bigr) \quad\text{in}\quad N.
  \end{align}
  To estimate the right-hand side we differentiate both sides of
  \begin{align*}
    \left\{I_3-d(x)\overline{W}(x)\right\}^{-1}\left\{I_3-d(x)\overline{W}(x)\right\} = I_3, \quad x\in N
  \end{align*}
  with respect to $x_i$ and use $\nabla d(x)=\bar{n}(x)$ to get
  \begin{align} \label{Pf_PD:Der_Res}
    \partial_i\Bigl(I_3-d\overline{W}\Bigr)^{-1} = \Bigl(I_3-d\overline{W}\Bigr)^{-1}\Bigl(\bar{n}_i\overline{W}+d\partial_i\overline{W}\Bigr)\Bigl(I_3-d\overline{W}\Bigr)^{-1} \quad\text{in}\quad N.
  \end{align}
  The right-hand side of \eqref{Pf_PD:Der_Res} is bounded on $N$ by \eqref{E:Wein_Bound} and \eqref{E:ConDer_Dom} since $W$ is of class $C^1$ on $\Gamma$ by the $C^3$-regularity of $\Gamma$.
  By this fact, \eqref{E:Wein_Bound}, \eqref{E:ConDer_Bound}, and \eqref{Pf_PD:Second},
  \begin{align*}
    |\partial_i\nabla\bar{\eta}| \leq c\left(\left|\overline{\nabla_\Gamma\eta}\right|+\left|\overline{\nabla_\Gamma^2\eta}\right|\right) \quad\text{in}\quad N,\,i=1,2,3,
  \end{align*}
  which shows \eqref{E:Con_Hess}.
  Moreover, by \eqref{E:ConDer_Surf}, \eqref{Pf_PD:Second}, \eqref{Pf_PD:Der_Res}, and $d=0$ on $\Gamma$,
  \begin{align*}
    \partial_i\nabla\bar{\eta} = n_iW\nabla_\Gamma\eta+\underline{D}_i(\nabla_\Gamma\eta) , \quad\text{i.e.}\quad \partial_i\partial_j\bar{\eta} = n_i\sum_{k=1}^3W_{jk}\underline{D}_k\eta+\underline{D}_i\underline{D}_j\eta
  \end{align*}
  on $\Gamma$ for $i,j=1,2,3$, which implies that
  \begin{align*}
    \Delta\bar{\eta} = \sum_{i=1}^3\partial_i^2\bar{\eta} = \sum_{i,k=1}^3(n_iW_{ik}\underline{D}_k\eta+\underline{D}_i^2\eta) = W^Tn\cdot\nabla_\Gamma\eta+\Delta_\Gamma\eta \quad\text{on}\quad \Gamma.
  \end{align*}
  Since $W^Tn=Wn=0$ by Lemma~\ref{L:Form_W}, we obtain $\Delta\bar{\eta}=\Delta_\Gamma\eta$ on $\Gamma$.
\end{proof}

\begin{proof}[Proof of Lemma~\ref{L:Wmp_Appr}]
  Here we only show the density of $C^\ell(\Gamma)$ in $W^{m,p}(\Gamma)$ for $\ell=m=2$ and $p\in[1,\infty)$.
  The assertion in other cases are proved similarly.

  Let $\eta\in W^{2,p}(\Gamma)$.
  Since $\Gamma$ is compact and of class $C^2$, by a localization argument with a partition of unity on $\Gamma$ we may assume that there exist an open set $U$ in $\mathbb{R}^2$, a compact subset $\mathcal{K}$ of $U$, and a $C^2$ local parametrization $\mu\colon U\to\Gamma$ of $\Gamma$ such that $\eta$ is supported in $\mu(\mathcal{K})$.
  Let $\tilde{\eta}:=\eta\circ\mu$ on $U$.
  Note that $\tilde{\eta}$ is supported in the compact subset $\mathcal{K}$ of $U$.
  We show that there exists $c>0$ such that
  \begin{align} \label{Pf_WA:Equ_W2p}
    c^{-1}\|\tilde{\eta}\|_{W^{2,p}(U)} \leq \|\eta\|_{W^{2,p}(\Gamma)} \leq c\|\tilde{\eta}\|_{W^{2,p}(U)}.
  \end{align}
  By \eqref{E:Lp_Loc} and \eqref{E:W1p_Loc} it is sufficient to consider the second order derivatives of $\eta$ and $\tilde{\eta}$.
  Let $\{e_1,e_2,e_3\}$ be the standard basis of $\mathbb{R}^3$ and
  \begin{align*}
    \tilde{\eta}_k(s) := \underline{D}_k\eta(\mu(s)) = \nabla_\Gamma\eta(\mu(s))\cdot e_k, \quad s\in U,\,k=1,2,3.
  \end{align*}
  By the right-hand inequality of \eqref{E:W1p_Loc} with $\eta$ replaced by $\underline{D}_k\eta$ we have
  \begin{align*}
    \|\nabla_\Gamma\underline{D}_k\eta\|_{L^p(\Gamma)} \leq c\|\nabla_s\tilde{\eta}_k\|_{L^p(U)}.
  \end{align*}
  To estimate the right-hand side we observe by \eqref{Pf_Met:Grad} that
  \begin{align*}
    \tilde{\eta}_k(s) = \nabla_\Gamma\eta(\mu(s))\cdot e_k = \sum_{i,j=1}^2\theta^{ij}(s)\partial_{s_i}\tilde{\eta}(s)\partial_{s_j}\mu(s)\cdot e_k, \quad s\in U.
  \end{align*}
  We differentiate both sides with respect to $s$ and use \eqref{E:Mu_Bound} and \eqref{E:Metric} to get
  \begin{align*}
    |\nabla_s\tilde{\eta}_k(s)| \leq c(|\nabla_s\tilde{\eta}(s)|+|\nabla_s^2\tilde{\eta}(s)|), \quad s\in\mathcal{K}.
  \end{align*}
  Since $\tilde{\eta}$ is supported in $\mathcal{K}$, the above inequalities show that
  \begin{align*}
    \|\nabla_\Gamma\underline{D}_k\eta\|_{L^p(\Gamma)} \leq c\|\nabla_s\tilde{\eta}_k\|_{L^p(U)} \leq c\left(\|\nabla_s\tilde{\eta}\|_{L^p(U)}+\|\nabla_s^2\tilde{\eta}\|_{L^p(U)}\right).
  \end{align*}
  for $k=1,2,3$.
  Therefore,
  \begin{align} \label{Pf_WA:Lp_Hess_R}
    \|\nabla_\Gamma^2\eta\|_{L^p(\Gamma)} \leq c\left(\|\nabla_s\tilde{\eta}\|_{L^p(U)}+\|\nabla_s^2\tilde{\eta}\|_{L^p(U)}\right) \leq c\|\tilde{\eta}\|_{W^{2,p}(U)}.
  \end{align}
  Next we take the inner product of the first equality of \eqref{Pf_Met:Grad} with $\partial_{s_k}\mu(s)$ to get
  \begin{align*}
    \nabla_\Gamma\eta(\mu(s))\cdot\partial_{s_k}\mu(s) = \sum_{i,j=1}^2\theta^{ij}(s)\theta_{jk}(s)\partial_{s_i}\tilde{\eta}(s) = \partial_{s_k}\tilde{\eta}(s), \quad s\in U,\,k=1,2,3.
  \end{align*}
  Let $\overline{\nabla_\Gamma\eta}$ be the constant extension of $\nabla_\Gamma\eta$ in the normal direction of $\Gamma$.
  For $s\in U$ and $l=1,2,3$ we observe by \eqref{E:ConDer_Surf} with $y=\mu(s)\in\Gamma$ that
  \begin{align*}
    \partial_{s_l}\Bigl(\nabla_\Gamma\eta(\mu(s))\Bigr) &= \partial_{s_l}\Bigl(\overline{\nabla_\Gamma\eta}(\mu(s))\Bigr) = \left[\left\{\nabla\Bigl(\overline{\nabla_\Gamma\eta}\Bigr)\right\}(\mu(s))\right]^T\partial_{s_l}\mu(s) \\
    &= [\nabla_\Gamma^2\eta(\mu(s))]^T\partial_{s_l}\mu(s).
  \end{align*}
  Hence for $s\in U$ and $k,l=1,2,3$ we have
  \begin{align*}
    \partial_{s_l}\partial_{s_k}\tilde{\eta}(s) &= \partial_{s_l}\Bigl(\nabla_\Gamma\eta(\mu(s))\cdot\partial_{s_k}\mu(s)\Bigr) \\
    &= [\nabla_\Gamma^2\eta(\mu(s))]^T\partial_{s_l}\mu(s)\cdot\partial_{s_k}\mu(s)+\nabla_\Gamma\eta(\mu(s))\cdot\partial_{s_l}\partial_{s_k}\mu(s)
  \end{align*}
  and by applying \eqref{E:Mu_Bound} to the last line we deduce that
  \begin{align*}
    |\nabla_s^2\tilde{\eta}(s)| \leq c(|\nabla_\Gamma\eta(\mu(s))|+|\nabla_\Gamma^2\eta(\mu(s))|), \quad s\in\mathcal{K}.
  \end{align*}
  Noting that $\eta$ is supported in $\mu(\mathcal{K})$, we use this inequality, \eqref{E:Metric}, and \eqref{Pf_Met:Int} to get
  \begin{align} \label{Pf_WA:Lp_Hess_L}
    \|\nabla_s^2\tilde{\eta}\|_{L^p(U)} \leq c\left(\|\nabla_\Gamma\eta\|_{L^p(\Gamma)}+\|\nabla_\Gamma^2\eta\|_{L^p(\Gamma)}\right) \leq c\|\eta\|_{W^{2,p}(\Gamma)}.
  \end{align}
  By \eqref{E:Lp_Loc}, \eqref{E:W1p_Loc}, \eqref{Pf_WA:Lp_Hess_R}, and \eqref{Pf_WA:Lp_Hess_L} we obtain \eqref{Pf_WA:Equ_W2p}.

  Now we have $\tilde{\eta}\in W^{2,p}(U)$ by $\eta\in W^{2,p}(\Gamma)$ and the left-hand inequality of \eqref{Pf_WA:Equ_W2p}.
  Since $\tilde{\eta}$ is compactly supported in $U$, by a standard mollification argument (see e.g. \cite[Lemma~3.16]{AdFo03}) we can take a sequence $\{\tilde{\eta}_k\}_{k=1}^\infty$ in $C_c^\infty(U)$ that converges to $\tilde{\eta}$ strongly in $W^{2,p}(U)$.
  Hence the right-hand inequality of \eqref{Pf_WA:Equ_W2p} yields
  \begin{align*}
    \|\eta-\eta_k\|_{W^{2,p}(\Gamma)} \leq c\|\tilde{\eta}-\tilde{\eta}_k\|_{W^{2,p}(U)} \to 0 \quad\text{as}\quad k\to\infty,
  \end{align*}
  where $\eta_k(\mu(s)):=\tilde{\eta}_k(s)$ for $s\in U$ and $k\in\mathbb{N}$ and we extend $\eta_k$ to $\Gamma$ by setting zero outside of $\mu(U)$.
  Since $\mu$ is of class $C^2$ on $U$ and $\tilde{\eta}_k\in C_c^\infty(U)$, we have $\eta_k\in C^2(\Gamma)$ for each $k\in\mathbb{N}$.
  Hence $C^2(\Gamma)$ is dense in $W^{2,p}(\Gamma)$.
\end{proof}

Next we assume that $\Gamma$ is of class $C^5$ and prove the formulas and inequalities in Section~\ref{SS:Pre_Dom} for the surface quantities on the boundary of the curved thin domain.

\begin{proof}[Proof of Lemma~\ref{L:NB_Aux}]
  First note that, since $W\in C^3(\Gamma)^{3\times3}$ by the $C^5$-regularity of $\Gamma$ and $g_0,g_1\in C^4(\Gamma)$, they are bounded on $\Gamma$ along with their first and second order tangential derivatives.

  Let $\tau_\varepsilon^i$ and $n_\varepsilon^i$, $i=0,1$ be the vector fields on $\Gamma$ given by \eqref{E:Def_NB_Aux} and \eqref{E:Def_NB}.
  Then the first inequalities of \eqref{E:Tau_Bound} and \eqref{E:Tau_Diff} immediately follow from \eqref{E:Wein_Bound} and \eqref{E:Wein_Diff}.
  To show the second inequalities of \eqref{E:Tau_Bound} and \eqref{E:Tau_Diff} we set
  \begin{align} \label{Pf_NBA:Def_R}
    R_\varepsilon^i(y) := \{I_3-\varepsilon g_i(y)W(y)\}^{-1}, \quad y\in\Gamma
  \end{align}
  and apply $\underline{D}_k$, $k=1,2,3$ to both sides of $R_\varepsilon^i(I_3-\varepsilon g_iW)=I_3$ on $\Gamma$ to get
  \begin{align} \label{Pf_NBA:D_Inv}
    \underline{D}_kR_\varepsilon^i = \varepsilon R_\varepsilon^i\{(\underline{D}_kg_i)W+g_i\underline{D}_kW\}R_\varepsilon^i \quad\text{on}\quad \Gamma.
  \end{align}
  Hence by \eqref{E:Wein_Bound} there exists a constant $c>0$ independent of $\varepsilon$ such that
  \begin{align} \label{Pf_NBA:Est_D_Inv}
    |\underline{D}_kR_\varepsilon^i| \leq c\varepsilon \quad\text{on}\quad \Gamma.
  \end{align}
  Applying \eqref{E:Wein_Bound}, \eqref{E:Wein_Diff}, and \eqref{Pf_NBA:Est_D_Inv} to $\underline{D}_k\tau_\varepsilon^i=(\underline{D}_kR_\varepsilon^i)\nabla_\Gamma g_i+R_\varepsilon^i(\underline{D}_k\nabla_\Gamma g)$ we obtain
  \begin{align*}
    |\underline{D}_k\tau_\varepsilon^i| \leq c, \quad |\underline{D}_k\tau_\varepsilon^i-\underline{D}_k\nabla_\Gamma g| &\leq |(\underline{D}_kR_\varepsilon^i)\nabla_\Gamma g_i|+|(R_\varepsilon^i-I_3)(\underline{D}_k\nabla_\Gamma g)| \leq c\varepsilon
  \end{align*}
  on $\Gamma$ for $k=1,2,3$.
  Hence the second inequalities of \eqref{E:Tau_Bound} and \eqref{E:Tau_Diff} are valid.
  We further apply $\underline{D}_l$, $l=1,2,3$ to both sides of \eqref{Pf_NBA:D_Inv} and use \eqref{E:Wein_Bound} and \eqref{Pf_NBA:Est_D_Inv} to obtain $|\underline{D}_l\underline{D}_kR_\varepsilon^i|\leq c\varepsilon$ on $\Gamma$.
  Using this inequality, \eqref{E:Wein_Bound}, and \eqref{Pf_NBA:Est_D_Inv} to
  \begin{multline*}
    \underline{D}_l\underline{D}_k\tau_\varepsilon^i = (\underline{D}_l\underline{D}_kR_\varepsilon^i)\nabla_\Gamma g_i+(\underline{D}_kR_\varepsilon^i)(\underline{D}_l\nabla_\Gamma g_i) \\+(\underline{D}_lR_\varepsilon^i)(\underline{D}_k\nabla_\Gamma g_i)+R_\varepsilon^i(\underline{D}_l\underline{D}_k\nabla_\Gamma g_i)
  \end{multline*}
  we get $|\underline{D}_l\underline{D}_k\tau_\varepsilon^i|\leq c$ on $\Gamma$ for $k,l=1,2,3$.
  This shows the third inequality of \eqref{E:Tau_Bound}.

  Next we show \eqref{E:N_Bound} and \eqref{E:N_Diff}.
  By \eqref{E:Def_NB} we immediately get the first equality of \eqref{E:N_Bound}.
  The other inequalities of \eqref{E:N_Bound} follow from \eqref{E:Tau_Bound}.
  To prove \eqref{E:N_Diff} let
  \begin{align*}
    \varphi_\varepsilon := \frac{1}{\sqrt{1+\varepsilon^2|\tau_\varepsilon^1|^2}}-\frac{1}{\sqrt{1+\varepsilon^2|\tau_\varepsilon^0|^2}}, \quad \tau_\varepsilon := -\frac{\tau_\varepsilon^1}{\sqrt{1+\varepsilon^2|\tau_\varepsilon^1|^2}}+\frac{\tau_\varepsilon^0}{\sqrt{1+\varepsilon^2|\tau_\varepsilon^0|^2}}
  \end{align*}
  so that $n_\varepsilon^0+n_\varepsilon^1=\varphi_\varepsilon n+\varepsilon\tau_\varepsilon$ on $\Gamma$.
  From \eqref{E:Tau_Bound} we deduce that
  \begin{align} \label{Pf_NBA:Est_Tau}
    |\tau_\varepsilon| \leq c, \quad |\nabla_\Gamma\tau_\varepsilon| \leq c \quad\text{on}\quad \Gamma
  \end{align}
  with a constant $c>0$ independent of $\varepsilon$.
  Also,
  \begin{align} \label{Pf_NBA:Est_Phi}
    |\varphi_\varepsilon| \leq \frac{\varepsilon^2}{2}\bigl||\tau_\varepsilon^1|^2-|\tau_\varepsilon^0|^2\bigr| \leq c\varepsilon^2 \quad\text{on}\quad \Gamma
  \end{align}
  by the mean value theorem for $(1+s)^{-1/2}$, $s\geq 0$ and \eqref{E:Tau_Bound}.
  Since
  \begin{align*}
    \nabla_\Gamma\left(\frac{1}{\sqrt{1+\varepsilon^2|\tau_\varepsilon^i|^2}}\right) = -\frac{\varepsilon^2(\nabla_\Gamma\tau_\varepsilon^i)\tau_\varepsilon^i}{(1+\varepsilon^2|\tau_\varepsilon^i|^2)^{3/2}} \quad\text{on}\quad \Gamma,\, i=0,1,
  \end{align*}
  we have $|\nabla_\Gamma\varphi_\varepsilon|\leq c\varepsilon^2$ on $\Gamma$ by \eqref{E:Tau_Bound}.
  Applying this inequality, \eqref{Pf_NBA:Est_Tau}, and \eqref{Pf_NBA:Est_Phi} to $n_\varepsilon^0+n_\varepsilon^1=\varphi_\varepsilon n+\varepsilon\tau_\varepsilon$ and its tangential gradient matrix we obtain \eqref{E:N_Diff}.
\end{proof}

\begin{proof}[Proof of Lemma~\ref{L:Nor_Bo}]
  For $i=0,1$ let $\tau_\varepsilon^i$ and $n_\varepsilon^i$ be the vector fields on $\Gamma$ given by \eqref{E:Def_NB_Aux} and \eqref{E:Def_NB}, and $\bar{n}_\varepsilon^i=n_\varepsilon^i\circ\pi$ the constant extension of $n_\varepsilon^i$.
  Since $|n_\varepsilon^i|=1$ on $\Gamma$ and both $n_\varepsilon$ and $\bar{n}_\varepsilon^i$ have the direction of $(-1)^{i+1}\bar{n}$ on $\Gamma_\varepsilon^i$, it is sufficient for \eqref{E:Nor_Bo} to show that $\bar{n}_\varepsilon^i$ is perpendicular to the tangent plane of $\Gamma_\varepsilon^i$.
  Let $\mu\colon U\to\Gamma$ be a local parametrization of $\Gamma$ with an open set $U$ of $\mathbb{R}^2$ and
  \begin{align*}
    \mu_\varepsilon^i(s) := \mu(s)+\varepsilon g_i(\mu(s))n(\mu(s)), \quad s\in U.
  \end{align*}
  Then $\mu_\varepsilon^i$ is a local parametrization of $\Gamma_\varepsilon^i$ and $\{\partial_{s_1}\mu_\varepsilon^i(s),\partial_{s_2}\mu_\varepsilon^i(s)\}$ is a basis of the tangent plane of $\Gamma_\varepsilon^i$ at $\mu_\varepsilon^i(s)$.
  Hence to show that $\bar{n}_\varepsilon^i$ is perpendicular to the tangent plane of $\Gamma_\varepsilon^i$ it suffices to prove
  \begin{align} \label{Pf_NoBo:Perp}
    \bar{n}_\varepsilon^i(\mu_\varepsilon^i(s))\cdot\partial_{s_k}\mu_\varepsilon^i(s) = 0, \quad s\in U,\,k=1,2.
  \end{align}
  Moreover, $\bar{n}_\varepsilon^i(\mu_\varepsilon^i(s))=n_\varepsilon^i(\mu(s))$ for $s\in U$ by $\pi(\mu_\varepsilon^i(s))=\mu(s)\in\Gamma$.
  By this fact and \eqref{E:Def_NB} the condition \eqref{Pf_NoBo:Perp} reduces to
  \begin{align} \label{Pf_NoBo:Goal}
    n(\mu(s))\cdot\partial_{s_k}\mu_\varepsilon^i(s) = \varepsilon\tau_\varepsilon^i(\mu(s))\cdot\partial_{s_k}\mu_\varepsilon^i(s), \quad s\in U,\,k=1,2.
  \end{align}
  From now on, we write $\eta^\flat(s):=\eta(\mu(s))$, $s\in U$ for a function $\eta$ on $\Gamma$ and suppress the argument $s\in U$.
  Note that, since $\mu(s)\in\Gamma$ for $s\in U$ and $\eta=\bar{\eta}$ on $\Gamma$,
  \begin{align} \label{Pf_NoBo:Sharp}
    \partial_{s_k}\eta^\flat(s) = \partial_{s_k}\bigl(\bar{\eta}(\mu(s))\bigr) = \partial_{s_k}\mu(s)\cdot\nabla\bar{\eta}(\mu(s)) = \partial_{s_k}\mu(s)\cdot\nabla_\Gamma\eta(\mu(s))
  \end{align}
  for $s\in U$ and $k=1,2$ by \eqref{E:ConDer_Surf}.
  Let us prove \eqref{Pf_NoBo:Goal}.
  For $k=1,2$ we differentiate $\mu_\varepsilon^i=\mu+\varepsilon g_i^\flat n^\flat$ with respect to $s_k$.
  Then we have
  \begin{align*}
    \partial_{s_k}\mu_\varepsilon^i = (I_3-\varepsilon g_i^\flat W^\flat)\partial_{s_k}\mu+\varepsilon\{\partial_{s_k}\mu\cdot(\nabla_\Gamma g_i)^\flat\}n^\flat
  \end{align*}
  by \eqref{Pf_NoBo:Sharp} and $-\nabla_\Gamma n=W=W^T$ on $\Gamma$.
  Since $\partial_{s_k}\mu(s)$ is tangent to $\Gamma$ at $\mu(s)$ and $W^Tn=Wn=0$ on $\Gamma$, we deduce from the above equality that
  \begin{align*}
    n^\flat\cdot\partial_{s_k}\mu_\varepsilon^i = \varepsilon\partial_{s_k}\mu\cdot(\nabla_\Gamma g_i)^\flat.
  \end{align*}
  Also, since $\tau_\varepsilon^i=(I_3-\varepsilon g_iW)^{-1}\nabla_\Gamma g_i$ is tangential and $W$ is symmetric on $\Gamma$,
  \begin{align*}
    \varepsilon\tau_\varepsilon^{i,\flat}\cdot\partial_{s_k}\mu_\varepsilon^i = \varepsilon(I_3-\varepsilon g_i^\flat W^\flat)^{-1}(\nabla_\Gamma g_i)^\flat\cdot(I_3-\varepsilon g_i^\flat W^\flat)\partial_{s_k}\mu = \varepsilon(\nabla_\Gamma g_i)^\flat\cdot\partial_{s_k}\mu.
  \end{align*}
  The above two equalities imply \eqref{Pf_NoBo:Goal} and thus the claim is valid.
\end{proof}

\begin{proof}[Proof of Lemma~\ref{L:Comp_Nor}]
  Throughout the proof we write $c$ for a general positive constant independent of $\varepsilon$ and denote by $\bar{\eta}=\eta\circ\pi$ the constant extension of a function $\eta$ on $\Gamma$.
  First note that $\bar{n}$, $\overline{P}$, and $\overline{W}$ are bounded on $N$ independently of $\varepsilon$ along their first order derivatives by the $C^5$-regularity of $\Gamma$ and \eqref{E:ConDer_Bound}.

  For $i=0,1$ let $\tau_\varepsilon^i$ and $n_\varepsilon^i$ be given by \eqref{E:Def_NB_Aux} and \eqref{E:Def_NB}, and
  \begin{align*}
    \varphi_\varepsilon^i(x) := \frac{1}{\sqrt{1+\varepsilon^2|\bar{\tau}_\varepsilon^i(x)|^2}}-1, \quad x\in N.
  \end{align*}
  We apply the mean value theorem for $(1+s)^{-1/2}$, $s\geq0$ to $\varphi_\varepsilon^i$, compute the first and second order derivatives of $\varphi_\varepsilon^i$ directly, and use \eqref{E:ConDer_Bound}, \eqref{E:Con_Hess}, and \eqref{E:Tau_Bound} to get
  \begin{align} \label{Pf_CN:Aux_1}
    |\partial_x^\alpha\varphi_\varepsilon^i(x)| \leq c\varepsilon^2, \quad x\in N,\,|\alpha|=0,1,2,
  \end{align}
  where $\partial_x^\alpha=\partial_1^{\alpha_1}\partial_2^{\alpha_2}\partial_3^{\alpha_3}$ for $\alpha=(\alpha_1,\alpha_2,\alpha_3)\in\mathbb{Z}^3$ with $\alpha_j\geq0$, $j=1,2,3$.
  Since
  \begin{align*}
    n_\varepsilon-(-1)^{i+1}\Bigl(\bar{n}-\varepsilon\overline{\nabla_\Gamma g_i}\Bigr) = (-1)^{i+1}\varphi_\varepsilon^i(\bar{n}-\varepsilon\bar{\tau}_\varepsilon^i)-(-1)^{i+1}\varepsilon\Bigl(\bar{\tau}_\varepsilon^i-\overline{\nabla_\Gamma g_i}\Bigr)
  \end{align*}
  on $\Gamma_\varepsilon^i$ by \eqref{E:Def_NB} and \eqref{E:Nor_Bo}, we obtain \eqref{E:Comp_N} by \eqref{E:Tau_Bound}, \eqref{E:Tau_Diff}, and \eqref{Pf_CN:Aux_1}.
  Also, the inequalities \eqref{E:Comp_P} follow from \eqref{E:Comp_N} and the definitions of $P$, $Q$, $P_\varepsilon$, and $Q_\varepsilon$.

  Next we prove \eqref{E:Comp_W}.
  For $x\in N$ we set
  \begin{align*}
    \Phi_\varepsilon^i(x) := (-1)^{i+1}\left\{\varphi_\varepsilon^i(x)\bar{n}(x)-\frac{\varepsilon\bar{\tau}_\varepsilon^i(x)}{\sqrt{1+\varepsilon^2|\bar{\tau}_\varepsilon^i(x)|^2}}\right\}.
  \end{align*}
  Then we see by \eqref{E:ConDer_Bound}, \eqref{E:Con_Hess}, \eqref{E:Tau_Bound}, and \eqref{Pf_CN:Aux_1} that
  \begin{align} \label{Pf_CN:Aux_2}
    |\partial_x^\alpha\Phi_\varepsilon^i(x)| \leq c\varepsilon, \quad x\in N,\,|\alpha|=0,1,2.
  \end{align}
  Since $\bar{n}_\varepsilon^i(x)=(-1)^{i+1}\bar{n}(x)+\Phi_\varepsilon^i(x)$ for $x\in N$, by \eqref{E:Nor_Grad} we have
  \begin{align*}
    \nabla\bar{n}_\varepsilon^i(x) = (-1)^i\{I_3-d(x)\overline{W}(x)\}^{-1}\overline{W}(x)+\nabla\Phi_\varepsilon^i(x), \quad x\in N.
  \end{align*}
  Moreover, since $\bar{n}_\varepsilon^i$ is an extension of $n_\varepsilon|_{\Gamma_\varepsilon^i}$ to $N$, the Weingarten map of $\Gamma_\varepsilon^i$ is given by $W_\varepsilon=-P_\varepsilon\nabla\bar{n}_\varepsilon^i$ on $\Gamma_\varepsilon^i$.
  Thus the above equality yields
  \begin{align} \label{Pf_CN:Wein}
    W_\varepsilon(x) = P_\varepsilon(x)\left\{(-1)^{i+1}\overline{R}_\varepsilon^i(x)\overline{W}(x)-\nabla\Phi_\varepsilon^i(x)\right\}, \quad x\in\Gamma_\varepsilon^i,
  \end{align}
  where $R_\varepsilon^i$ is given by \eqref{Pf_NBA:Def_R}.
  Noting that $PR_\varepsilon^iW=R_\varepsilon^iPW=R_\varepsilon^iW$ on $\Gamma$ by \eqref{E:Form_W} and \eqref{E:WReso_P}, we deduce from the above equality that
  \begin{align*}
    \left|W_\varepsilon-(-1)^{i+1}\overline{W}\right| &\leq \left|\Bigl(P_\varepsilon-\overline{P}\Bigr)\overline{R}_\varepsilon^i\overline{W}\right|+\left|\Bigl(\overline{R}_\varepsilon^i-I_3\Bigr)\overline{W}\right|+|P_\varepsilon\nabla\Phi_\varepsilon^i| \quad\text{on}\quad \Gamma_\varepsilon^i.
  \end{align*}
  Hence we obtain the first inequality of \eqref{E:Comp_W} by applying \eqref{E:Wein_Bound}, \eqref{E:Wein_Diff}, \eqref{E:Comp_P}, and \eqref{Pf_CN:Aux_2} to the above inequality.
  Also, the second inequality of \eqref{E:Comp_W} follows from the first one since $H=\mathrm{tr}[W]$ and $H_\varepsilon=\mathrm{tr}[W_\varepsilon]$.

  Let us show \eqref{E:Comp_DW}.
  Based on \eqref{Pf_CN:Wein} we define an extension of $W_\varepsilon|_{\Gamma_\varepsilon^i}$ to $N$ by
  \begin{align*}
    \widetilde{W}_\varepsilon^i(x) := \overline{P}_\varepsilon^i(x)\left\{(-1)^{i+1}\overline{R}_\varepsilon^i(x)\overline{W}(x)-\nabla\Phi_\varepsilon^i(x)\right\}, \quad x\in N,
  \end{align*}
  where $P_\varepsilon^i:=I_3-n_\varepsilon^i\otimes n_\varepsilon^i$ on $\Gamma$.
  For $x\in N$ let
  \begin{gather*}
    E_\varepsilon^i(x) := (-1)^{i+1}\left\{\overline{P}_\varepsilon^i(x)-\overline{P}(x)\right\}\overline{R}_\varepsilon^i(x)\overline{W}(x), \\
    F_\varepsilon^i(x) := (-1)^{i+1}\overline{P}(x)\left\{\overline{R}_\varepsilon^i(x)-I_3\right\}\overline{W}(x), \quad G_\varepsilon^i(x) := -\overline{P}_\varepsilon^i(x)\nabla\Phi_\varepsilon^i(x)
  \end{gather*}
  so that $\widetilde{W}_\varepsilon^i=(-1)^{i+1}\overline{W}+E_\varepsilon^i+F_\varepsilon^i+G_\varepsilon^i$ in $N$ by \eqref{E:Form_W}.
  Then by \eqref{E:ConDer_Dom} we get
  \begin{align} \label{Pf_CN:Der_Wein}
    \partial_j\widetilde{W}_\varepsilon^i = \sum_{k=1}^3(-1)^{i+1}\left[\Bigl(I_3-d\overline{W}\Bigr)^{-1}\right]_{jk}\overline{\underline{D}_kW}+\partial_jE_\varepsilon^i+\partial_jF_\varepsilon^i+\partial_jG_\varepsilon^i
  \end{align}
  in $N$ for $j=1,2,3$.
  To estimate the last three terms we see that
  \begin{align*}
    \overline{P}_\varepsilon^i-\overline{P} = (-1)^i(\bar{n}\otimes\Phi_\varepsilon^i+\Phi_\varepsilon^i\otimes\bar{n})-\Phi_\varepsilon^i\otimes\Phi_\varepsilon^i \quad\text{in}\quad N
  \end{align*}
  by $\bar{n}_\varepsilon^i=(-1)^{i+1}\bar{n}+\Phi_\varepsilon^i$ in $N$ and the definitions of $P$ and $P_\varepsilon^i$.
  Hence
  \begin{align*}
    \left|\overline{P}_\varepsilon^i-\overline{P}\right| \leq c\varepsilon, \quad \left|\partial_j\overline{P}_\varepsilon^i-\partial_j\overline{P}\right| \leq c\varepsilon \quad\text{in}\quad N
  \end{align*}
  for $j=1,2,3$ by $|n|=1$ on $\Gamma$, \eqref{E:NorG_Bound}, and \eqref{Pf_CN:Aux_2}.
  These inequalities, \eqref{E:Wein_Bound}, \eqref{E:Wein_Diff}, \eqref{E:ConDer_Bound}, \eqref{Pf_NBA:Est_D_Inv}, and \eqref{Pf_CN:Aux_2} show that
  \begin{align*}
    |\partial_jE_\varepsilon^i| \leq c\varepsilon, \quad |\partial_jF_\varepsilon^i| \leq c\varepsilon, \quad |\partial_jG_\varepsilon^i| \leq c\varepsilon \quad\text{in}\quad N.
  \end{align*}
  Applying \eqref{E:Wein_Diff} and the above inequalities to \eqref{Pf_CN:Der_Wein} we get
  \begin{align} \label{Pf_CN:Diff_DW}
    \left|\partial_j\widetilde{W}_\varepsilon^i(x)-(-1)^{i+1}\overline{\underline{D}_jW}(x)\right| \leq c(|d(x)|+\varepsilon), \quad x\in N,\, j=1,2,3.
  \end{align}
  Now we observe that $\underline{D}_j^\varepsilon W_\varepsilon=\sum_{k=1}^3[P_\varepsilon]_{jk}\partial_k\widetilde{W}_\varepsilon^i$ on $\Gamma_\varepsilon^i$ since $\widetilde{W}_\varepsilon^i$ is an extension of $W_\varepsilon|_{\Gamma_\varepsilon^i}$ to $N$.
  From this fact and $\underline{D}_jW=\sum_{k=1}^3P_{jk}\underline{D}_kW$ on $\Gamma$ by \eqref{E:P_TGr} we have
  \begin{multline*}
    \left|\underline{D}_j^\varepsilon W_\varepsilon-(-1)^{i+1}\overline{\underline{D}_jW}\right| \\
    \leq \sum_{k=1}^3\left(\left|\Bigl[P_\varepsilon-\overline{P}\Bigr]_{jk}\partial_k\widetilde{W}_\varepsilon^i\right|+\left|\overline{P}_{jk}\left\{\partial_k\widetilde{W}_\varepsilon^i-(-1)^{i+1}\overline{\underline{D}_kW}\right\}\right|\right)
  \end{multline*}
  on $\Gamma_\varepsilon^i$ for $j=1,2,3$.
  Applying \eqref{E:Comp_P} and \eqref{Pf_CN:Diff_DW} with $|d|=\varepsilon|\bar{g}_i|\leq c\varepsilon$ on $\Gamma_\varepsilon^i$ to the right-hand side we conclude that \eqref{E:Comp_DW} is valid.
\end{proof}

\begin{proof}[Proof of Lemma~\ref{L:Diff_SQ_IO}]
  Let $\overline{P}=P\circ\pi$ be the constant extension of $P$.
  Since
  \begin{align*}
    \overline{P}(y+\varepsilon g_0(y)n(y)) = \overline{P}(y+\varepsilon g_1(y)n(y)) = P(y), \quad y\in\Gamma,
  \end{align*}
  we observe that
  \begin{align*}
    |P_\varepsilon(y+\varepsilon g_1(y)n(y))-P_\varepsilon(y+\varepsilon g_0(y)n(y))| \leq \sum_{i=0,1}\left|\Bigl[P_\varepsilon-\overline{P}\Bigr](y+\varepsilon g_i(y)n(y))\right|.
  \end{align*}
  To the right-hand side we apply \eqref{E:Comp_P} to get \eqref{E:Diff_PQ_IO} for $F_\varepsilon=P_\varepsilon$.
  Using \eqref{E:Comp_P}--\eqref{E:Comp_DW} we can show the other inequalities in the same way.
\end{proof}

Let us prove the formula \eqref{Pf_A:Det_Zeta} in the proof of Lemma~\ref{L:Agmon}.

\begin{proof}[Proof of \eqref{Pf_A:Det_Zeta}]
  Let $\mu\colon(0,1)^2\to\Gamma$ be a local parametrization of $\Gamma$ and $\zeta$ the mapping given by \eqref{Pf_A:Def_Z}.
  In what follows, we write $\eta^\flat(s'):=\eta(\mu(s'))$, $s'\in(0,1)^2$ for a function $\eta$ on $\Gamma$ and suppress the arguments $s'$ and $s\in(0,1)^3$.
  Since
  \begin{align*}
     \partial_{s_i}\zeta = (I_3-h_\varepsilon W^\flat)\partial_{s_i}\mu+\eta_\varepsilon^in^\flat, \quad \partial_{s_3}\zeta = \varepsilon g^\flat n^\flat \quad\text{on}\quad (0,1)^3,\,i=1,2
  \end{align*}
  by \eqref{Pf_A:Deri_Zeta} with $h_\varepsilon$ and $\eta_\varepsilon^i$ given by \eqref{Pf_A:DZ_Aux}, the gradient matrix of $\zeta$ is of the form
  \begin{align*}
    \nabla_s\zeta =
    \begin{pmatrix}
      \partial_{s_1}\zeta_1 & \partial_{s_1}\zeta_2 & \partial_{s_1}\zeta_3 \\
      \partial_{s_2}\zeta_1 & \partial_{s_2}\zeta_2 & \partial_{s_2}\zeta_3 \\
      \partial_{s_3}\zeta_1 & \partial_{s_3}\zeta_2 & \partial_{s_3}\zeta_3
    \end{pmatrix} =
    \begin{pmatrix}
      \nabla_s\mu(I_3-h_\varepsilon W^\flat)^T+\eta_\varepsilon\otimes n^\flat \\
      \varepsilon g^\flat(n^\flat)^T
    \end{pmatrix}.
  \end{align*}
  Here we consider $n^\flat\in\mathbb{R}^3$ and $\eta_\varepsilon:=(\eta_\varepsilon^1,\eta_\varepsilon^2)\in\mathbb{R}^2$ as column vectors and, by abuse of notation, write $\nabla_s\mu$ for the gradient matrix of $\mu(s')$, $s'\in(0,1)^2$ of the form \eqref{E:Def_Met}.
  Since $\partial_{s_1}\mu$ and $\partial_{s_2}\mu$ are tangent to $\Gamma$ at $\mu(s')$ we have $(\nabla_s\mu)n^\flat=0$.
  Moreover,
  \begin{align*}
    W^\flat n^\flat=0, \quad (\eta_\varepsilon\otimes n^\flat)n^\flat=|n^\flat|^2\eta_\varepsilon=\eta_\varepsilon, \quad (\eta_\varepsilon\otimes n^\flat)(n^\flat\otimes \eta_\varepsilon) = \eta_\varepsilon\otimes\eta_\varepsilon.
  \end{align*}
  From these equalities and the symmetry of the matrix $W^\flat$ it follows that
  \begin{align*}
    \nabla_s\zeta(\nabla_s\zeta)^T =
    \begin{pmatrix}
      \nabla_s\mu(I_3-h_\varepsilon W^\flat)^2(\nabla_s\mu)^T+\eta_\varepsilon\otimes\eta_\varepsilon & \varepsilon g^\flat\eta_\varepsilon \\
      \varepsilon g^\flat\eta_\varepsilon^T & \varepsilon^2(g^\flat)^2
    \end{pmatrix}.
  \end{align*}
  Hence by elementary row operations we have
  \begin{align*}
    \det[\nabla_s\zeta(\nabla_s\zeta)^T] &= \det
    \begin{pmatrix}
      \nabla_s\mu(I_3-h_\varepsilon W^\flat)^2(\nabla_s\mu)^T+\eta_\varepsilon\otimes\eta_\varepsilon & \varepsilon g^\flat\eta_\varepsilon \\
      \varepsilon g^\flat\eta_\varepsilon^T & \varepsilon^2(g^\flat)^2
    \end{pmatrix} \\
    &= \det
    \begin{pmatrix}
      \nabla_s\mu(I_3-h_\varepsilon W^\flat)^2(\nabla_s\mu)^T & 0 \\
      \varepsilon g^\flat\eta_\varepsilon^T & \varepsilon^2 (g^\flat)^2
    \end{pmatrix} \\
    &= \varepsilon^2 (g^\flat)^2\det[\nabla_s\mu(I_3-h_\varepsilon W^\flat)^2(\nabla_s\mu)^T].
  \end{align*}
  Let $\theta=\nabla_s\mu(\nabla_s\mu)^T$ be the Riemannian metric of $\Gamma$.
  If the relation
  \begin{align} \label{Pf_DZ:Det_Pr}
    \det[\nabla_s\mu(I_3-h_\varepsilon W^\flat)^2(\nabla_s\mu)^T] = J(\mu,h_\varepsilon)^2\det\theta
  \end{align}
  is valid, then by the above equality we get
  \begin{align*}
    \det\nabla_s\zeta = \bigl(\det[\nabla_s\zeta(\nabla_s\zeta)^T]\bigr)^{1/2} = \varepsilon g^\flat J(\mu,h_\varepsilon)\sqrt{\det\theta},
  \end{align*}
  i.e. the formula \eqref{Pf_A:Det_Zeta} holds.
  To prove \eqref{Pf_DZ:Det_Pr} we define $3\times 3$ matrices
  \begin{align*}
    A :=
    \begin{pmatrix}
      \nabla_s\mu \\
      (n^\flat)^T
    \end{pmatrix}, \quad
    A_h :=
    \begin{pmatrix}
      \nabla_s\mu(I_3-h_\varepsilon W^\flat) \\
      (n^\flat)^T
    \end{pmatrix}.
  \end{align*}
  Then by $(\nabla_s\mu)n^\flat=0$, $W^\flat n^\flat=0$, and the symmetry of $W^\flat$ we have
  \begin{gather*}
    A_h = A(I_3-h_\varepsilon W^\flat), \\
    AA^T =
    \begin{pmatrix}
      \theta & 0 \\
      0 & 1
    \end{pmatrix}, \quad
    A_hA_h^T =
    \begin{pmatrix}
      \nabla_s\mu(I_3-h^\flat W^\flat)^2(\nabla_s\mu)^T & 0 \\
      0 & 1
    \end{pmatrix}.
  \end{gather*}
  From these equalities and $\det(I_3-h_\varepsilon W^\flat)=J(\mu,h_\varepsilon)$ we deduce that
  \begin{align*}
    \begin{aligned}
      \det[\nabla_s\mu(I_3-h_\varepsilon W^\flat)^2(\nabla_s\mu)^T] &= \det[A_hA_h^T] = \det[A(I_3-h_\varepsilon W^\flat)^2A^T] \\
      &= \det[(I_3-h_\varepsilon W^\flat)^2]\det[AA^T] \\
      &= J(\mu,h_\varepsilon)^2\det\theta.
    \end{aligned}
  \end{align*}
  (Note that $A$ and $I_3-h_\varepsilon W^\flat$ are $3\times3$ matrices.)
  Hence \eqref{Pf_DZ:Det_Pr} is valid.
\end{proof}

Let us derive a change of variables formula for an integral over a parametrized surface used in Lemma~\ref{L:CoV_Surf}.
For $h\in C^1(\Gamma)$ satisfying $|h|<\delta$ on $\Gamma$ we set
\begin{align} \label{E:Def_Para_Surf}
  \Gamma_h := \{y+h(y)n(y)\mid y\in\Gamma\} \subset \mathbb{R}^3.
\end{align}
Note that $\Gamma_h\subset N$ by $|h|<\delta$ on $\Gamma$ (see Section~\ref{SS:Pre_Surf}).
We also define
\begin{align} \label{E:N_Para}
  \tau_h(y) := \{I_3-h(y)W(y)\}^{-1}\nabla_\Gamma h(y), \quad n_h(y) := \frac{n(y)-\tau_h(y)}{\sqrt{1+|\tau_h(y)|^2}}
\end{align}
for $y\in\Gamma$.
Note that $\tau_h$ is tangential on $\Gamma$.
We assume that the orientation of $\Gamma_h$ is the same as that of $\Gamma$.
Then as in the proof of Lemma~\ref{L:Nor_Bo} we can show that the constant extension $\bar{n}_h=n_h\circ\pi$ of $n_h$ gives the unit outward normal vector field of $\Gamma_h$.
For $\varphi\in C^1(\Gamma_h)$ we define the tangential gradient $\nabla_{\Gamma_h}\varphi$ as
\begin{align} \label{E:Para_TGr}
  \nabla_{\Gamma_h}\varphi(x) := \{I_3-\bar{n}_h(x)\otimes\bar{n}_h(x)\}\nabla\tilde{\varphi}(x), \quad x\in\Gamma_h,
\end{align}
where $\tilde{\varphi}$ is an arbitrary extension of $\varphi$ to $N$ satisfying $\tilde{\varphi}|_{\Gamma_h}=\varphi$.

\begin{lemma} \label{L:CoV_Para}
  Suppose that $\Gamma$ is of class $C^2$ and $h\in C^1(\Gamma)$ satisfies $|h|<\delta$ on $\Gamma$.
  Let $\Gamma_h$ be the parametrized surface given by \eqref{E:Def_Para_Surf}.
  For $\varphi\in L^1(\Gamma_h)$ we have
  \begin{align} \label{E:CoV_Para}
    \int_{\Gamma_h}\varphi(x)\,d\mathcal{H}^2(x) = \int_\Gamma\varphi_h^\sharp(y)J(y,h(y))\sqrt{1+|\tau_h(y)|^2}\,d\mathcal{H}^2(y),
  \end{align}
  where $\varphi_h^\sharp(y):=\varphi(y+h(y)n(y))$, $y\in\Gamma$ and $J$ and $\tau_h$ are given by \eqref{E:Def_Jac} and \eqref{E:N_Para}.
\end{lemma}

\begin{proof}
  Since $\Gamma$ is compact, we can take a finite number of open sets $U_k$ in $\mathbb{R}^2$ and local parametrizations $\mu^k\colon U_k\to\Gamma$, $k=1,\dots,k_0$ such that $\{\mu^k(U_k)\}_{k=1}^{k_0}$ is an open covering of $\Gamma$.
  Let $\{\eta^k\}_{k=1}^{k_0}$ be a partition of unity on $\Gamma$ subordinate to $\{\mu^k(U_k)\}_{k=1}^{k_0}$.
  For $k=1,\dots,k_0$ and $s\in U_k$ we define
  \begin{align*}
  \mu_h^k(s) := \mu_k(s)+h(\mu_k(s))n(\mu_k(s)), \quad \eta_h^k(\mu_h^k(s)) := \eta^k(\mu^k(s)).
  \end{align*}
  Then $\mu_h^k\colon U_k\to\Gamma_h$, $k=1,\dots,k_0$ are local parametrizations of $\Gamma_h$, $\{\mu_h^k(U_k)\}_{k=1}^{k_0}$ is an open covering of $\Gamma_h$, and $\{\eta_h^k\}_{k=1}^{k_0}$ is a partition of unity on $\Gamma_h$ subordinate to $\{\mu_h^k(U_k)\}_{k=1}^{k_0}$.
  Moreover, for $k=1,\dots,k_0$ we observe that
  \begin{align*}
    \eta_h^k(\mu_h^k(s))\varphi(\mu_h^k(s)) = \eta^k(\mu^k(s))\varphi_h^\sharp(\mu^k(s)), \quad s\in U_k,
  \end{align*}
  i.e. $(\eta_h^k\varphi)_h^\sharp=\eta^k\varphi_h^\sharp$ on $\mu^k(U_k)\subset \Gamma$ by $\varphi_h^\sharp(y)=\varphi(y+h(y)n(y))$ for $y\in\Gamma$.
  Hence it suffices to prove \eqref{E:CoV_Para} for $\varphi^k:=\eta_h^k\varphi$ instead of $\varphi$.

  From now on, we fix and suppress $k$.
  Let $\mu\colon U\to\Gamma$ be a local parametrization of $\Gamma$ with an open set $U$ in $\mathbb{R}^2$, $\nabla_s\mu$ and $\theta$ the gradient matrix of $\mu$ and the Riemannian metric of $\Gamma$ given by \eqref{E:Def_Met}, $\mu_h(s):=\mu(s)+h(\mu(s))n(\mu(s))$ for $s\in U$, and $\nabla_s\mu_h$ and $\theta_h$ the gradient matrix of $\mu_h$ and the Riemannian metric of $\Gamma_h$ similarly defined as in \eqref{E:Def_Met}.
  By the first part of the proof we may assume that $\varphi\circ\mu_h$ is compactly supported in $U$.
  Then since $\varphi\circ\mu_h=\varphi_h^\sharp\circ\mu$ on $U$ and
  \begin{align*}
    \int_{\Gamma_h}\varphi\,d\mathcal{H}^2 &= \int_U\varphi\circ\mu_h\sqrt{\det\theta_h}\,ds, \\
    \int_\Gamma\varphi_h^\sharp J(\cdot,h)\sqrt{1+|\tau_h|^2}\,d\mathcal{H}^2 &= \int_U(\varphi_h^\sharp\circ\mu)J(\mu,h\circ\mu)\sqrt{(1+|\tau_h\circ\mu|^2)\det\theta}\,ds,
  \end{align*}
  it is sufficient for \eqref{E:CoV_Para} to show that
  \begin{align} \label{Pf_CoVP:Goal}
    \sqrt{\det\theta_h} = J(\mu,h\circ\mu)\sqrt{(1+|\tau_h\circ\mu|^2)\det\theta} \quad\text{on}\quad U.
  \end{align}
  Hereafter we write $\eta^\flat(s):=\eta(\mu(s))$, $s\in U$ for a function $\eta$ on $\Gamma$ and suppress the argument $s\in U$.
  Let $\bar{h}=h\circ\pi$ be the constant extension of $h$.
  First we prove
  \begin{align} \label{Pf_CoVP:First}
    (1-|(\nabla_{\Gamma_h}\bar{h})\circ\mu_h|^2)\det\theta_h = J(\mu,h^\flat)^2\det\theta.
  \end{align}
  We differentiate $\mu_h=\mu+h^\flat n^\flat$ and use \eqref{Pf_NoBo:Sharp} and $-\nabla_\Gamma n=W=W^T$ on $\Gamma$ to get
  \begin{align*}
    \nabla_s\mu_h=\nabla_s\mu(I_3-h^\flat W^\flat)+\nabla_sh^\flat\otimes n^\flat.
  \end{align*}
  Since $\partial_{s_1}\mu$ and $\partial_{s_2}\mu$ are tangent to $\Gamma$ at $\mu(s)$, we have $(\nabla_s\mu)n^\flat=0$.
  Also,
  \begin{align*}
    W^\flat n^\flat = 0, \quad (\nabla_sh^\flat\otimes n^\flat)(n^\flat\otimes\nabla_s h^\flat) = |n^\flat|^2\nabla_sh^\flat\otimes\nabla_sh^\flat = \nabla_sh^\flat\otimes\nabla_sh^\flat.
  \end{align*}
  Noting that $W^\flat$ is symmetric, we deduce from these equalities that
  \begin{align*}
    \theta_h=\nabla_s\mu_h(\nabla_s\mu_h)^T=\nabla_s\mu(I_3-h^\flat W^\flat)^2(\nabla_s\mu)^T+\nabla_sh^\flat\otimes\nabla_sh^\flat
  \end{align*}
  and thus
  \begin{align*}
    \det(\theta_h-\nabla_sh^\flat\otimes\nabla_sh^\flat) = \det[\nabla_s\mu(I_3-h^\flat W^\flat)^2(\nabla_s\mu)^T].
  \end{align*}
  To this equality we apply \eqref{Pf_DZ:Det_Pr} with $h_\varepsilon$ replaced by $h^\flat$ and
  \begin{align*}
    \det(\theta_h-\nabla_sh^\flat\otimes\nabla_sh^\flat) &= \det[I_3-(\theta_h^{-1}\nabla_sh^\flat)\otimes\nabla_sh^\flat]\det\theta_h \\
    &= \{1-(\theta_h^{-1}\nabla_sh^\flat)\cdot\nabla_sh^\flat\}\det\theta_h
  \end{align*}
  by $\det(I_2+a\otimes b)=1+a\cdot b$ for $a,b\in\mathbb{R}^2$, where $\theta_h^{-1}=(\theta_h^{ij})_{i,j}$ is the inverse matrix of $\theta_h$.
  Then we obtain
  \begin{align} \label{Pf_CoVP:Det}
    \{1-(\theta_h^{-1}\nabla_sh^\flat)\cdot\nabla_sh^\flat\}\det\theta_h = J(\mu,h^\flat)^2\det\theta.
  \end{align}
  Now we recall that the tangential gradient of $\bar{h}=h\circ\pi$ on $\Gamma_h$ is expressed as
  \begin{align*}
    \nabla_{\Gamma_h}\bar{h}(\mu_h(s)) = \sum_{i,j=1}^2\theta_h^{ij}(s)\frac{\partial(\bar{h}\circ\mu_h)}{\partial s_i}(s)\partial_{s_j}\mu_h(s), \quad s\in U.
  \end{align*}
  Since $\pi(\mu_h(s))=\mu(s)$ for $s\in U$, we have $\bar{h}(\mu_h(s))=h(\mu(s))=h^\flat(s)$ and thus
  \begin{gather*}
    (\nabla_{\Gamma_h}\bar{h})\circ\mu_h = \sum_{i,j=1}^2\theta_h^{ij}(\partial_{s_i}h^\flat)\partial_{s_j}\mu_h, \\
    \left|(\nabla_{\Gamma_h}\bar{h})\circ\mu_h\right|^2 = \sum_{i,j=1}^2\theta_h^{ij}(\partial_{s_i}h^\flat)(\partial_{s_j}h^\flat) = (\theta_h^{-1}\nabla_sh^\flat)\cdot\nabla_sh^\flat.
  \end{gather*}
  Applying this equality to the left-hand side of \eqref{Pf_CoVP:Det} we obtain \eqref{Pf_CoVP:First}.

  Let $\tau_h$ and $n_h$ be given by \eqref{E:N_Para} and $\alpha:=(1+|\tau_h|^2)^{1/2}$ on $\Gamma$.
  We next show
  \begin{align} \label{Pf_CoVP:Second}
    1-\left|\nabla_{\Gamma_h}\bar{h}(y+h(y)n(y))\right|^2 = \frac{1}{\alpha(y)^2} = \frac{1}{1+|\tau_h(y)|^2}, \quad y\in\Gamma.
  \end{align}
  For $y\in\Gamma$ we see by \eqref{E:ConDer_Dom}, $d(y+h(y)n(y))=h(y)$, and $\pi(y+h(y)n(y))=y$ that
  \begin{align*}
    \nabla\bar{h}(y+h(y)n(y)) = \{I_3-h(y)W(y)\}^{-1}\nabla_\Gamma h(y) = \tau_h(y).
  \end{align*}
  From this equality, \eqref{E:Para_TGr}, $\bar{n}_h(y+h(y)n(y))=n_h(y)$, and
  \begin{align} \label{PF_CoVP:Tau}
    n_h = \alpha^{-1}(n-\tau_h), \quad \tau_h\cdot n = 0, \quad \alpha^2 = 1+|\tau_h|^2 \quad\text{on}\quad \Gamma
  \end{align}
  we deduce that (here we suppress the argument $y\in\Gamma$ of functions on $\Gamma$)
  \begin{align*}
    \nabla_{\Gamma_h}\bar{h}(y+hn) = (I_3-n_h\otimes n_h)\nabla\bar{h}(y+hn)= \alpha^{-2}(|\tau_h|^2n+\tau_h).
  \end{align*}
  Hence we again use the second and third equalities of \eqref{PF_CoVP:Tau} to obtain
  \begin{align*}
    1-\left|\nabla_{\Gamma_h}\bar{h}(y+hn)\right|^2 = 1-\alpha^{-4}(|\tau_h|^4+|\tau_h|^2) = \alpha^{-2}.
  \end{align*}
  This shows \eqref{Pf_CoVP:Second} and we conclude by \eqref{Pf_CoVP:First} and \eqref{Pf_CoVP:Second} that \eqref{Pf_CoVP:Goal} is valid.
\end{proof}

Finally, we give a regularity result for a Killing vector field on a surface.
Recall that for $v\in H^1(\Gamma)^3$ we define the surface strain rate tensor $D_\Gamma(v)$ by \eqref{E:Strain_Surf}.

\begin{lemma} \label{L:Kil_Reg}
  If $\Gamma$ is of class $C^\ell$ with $\ell\geq3$ and $v\in H^1(\Gamma,T\Gamma)$ satisfies $D_\Gamma(v)=0$ on $\Gamma$, then $v$ is of class $C^{\ell-3}$ on $\Gamma$.
  In particular, $v$ is smooth if $\Gamma$ is smooth.
\end{lemma}

To prove Lemma~\ref{L:Kil_Reg} we use the formula (see e.g.~\cite[Lemma~2.2]{Mi18})
\begin{align} \label{E:TD_Exc}
  \underline{D}_i\underline{D}_j\xi-\underline{D}_j\underline{D}_i\xi = [W\nabla_\Gamma\xi]_in_j-[W\nabla_\Gamma\xi]_jn_i \quad\text{on}\quad \Gamma,\,i,j=1,2,3
\end{align}
for $\xi\in C^2(\Gamma)$.
Here we denote by $[W\nabla_\Gamma\xi]_i$ the $i$-th component of $W\nabla_\Gamma\xi$.

\begin{proof}
  Let $v\in H^1(\Gamma,T\Gamma)$ satisfy $D_\Gamma(v)=0$ on $\Gamma$.
  First we prove $\mathrm{div}_\Gamma v=0$ on $\Gamma$.
  We fix and suppress $y\in\Gamma$.
  Let $\tau_1$ and $\tau_2$ be tangential vectors at $y$ such that $\{\tau_1,\tau_2,n\}$ is an orthonormal basis of $\mathbb{R}^3$.
  Then
  \begin{align*}
    \mathrm{div}_\Gamma v &= \mathrm{tr}[\nabla_\Gamma v] = \sum_{i=1,2}(\nabla_\Gamma v)\tau_i\cdot\tau_i+(\nabla_\Gamma v)n\cdot n = \sum_{i=1,2}P(\nabla_\Gamma v)P\tau_i\cdot\tau_i \\
    &= \mathrm{tr}[P(\nabla_\Gamma v)P] = \mathrm{tr}[D_\Gamma(v)] = 0
  \end{align*}
  by $P\nabla_\Gamma v=\nabla_\Gamma v$, $P\tau_i=\tau_i$ for $i=1,2$, $P^Tn=Pn=0$, and $D_\Gamma(v)=0$.

  Next we show that the $i$-th component of $v$ satisfies Poisson's equation
    \begin{align} \label{Pf_KiRe:Poi}
    \Delta_\Gamma v_i = \mathrm{div}_\Gamma(n_iHv) = \nabla_\Gamma(n_iH)\cdot v \quad\text{on}\quad \Gamma
  \end{align}
  for $i=1,2,3$.
  Here the second equality follows from $\mathrm{div}_\Gamma v=0$ on $\Gamma$.
  Noting that $v$ is tangential on $\Gamma$, we apply \eqref{E:Grad_W}, $P^T=P$, and $D_\Gamma(v)=0$ on $\Gamma$ to get
  \begin{align*}
    \nabla_\Gamma v+(\nabla_\Gamma v)^T &= 2D_\Gamma(v)+(Wv)\otimes n+n\otimes(Wv) \\
    &= (Wv)\otimes n+n\otimes(Wv)
  \end{align*}
  on $\Gamma$.
  Thus for $i=1,2,3$ we have
  \begin{align*}
    \nabla_\Gamma v_i = -\underline{D}_iv+n_iWv+[Wv]_in \quad\text{on}\quad \Gamma, \quad \underline{D}_iv=(\underline{D}_iv_1,\underline{D}_iv_2,\underline{D}_iv_3).
  \end{align*}
  Let $\xi\in C^2(\Gamma)$.
  By the above equality, \eqref{E:P_TGr}, and $W^T=W$ on $\Gamma$ we have
  \begin{align} \label{Pf_KiRe:TGr_Vi}
    (\nabla_\Gamma v_i,\nabla_\Gamma\xi)_{L^2(\Gamma)} = -(\underline{D}_iv,\nabla_\Gamma\xi)_{L^2(\Gamma)}+(n_iv,W\nabla_\Gamma\xi)_{L^2(\Gamma)}.
  \end{align}
  Let us compute $J_1:=-(\underline{D}_iv,\nabla_\Gamma\xi)_{L^2(\Gamma)}$.
  By \eqref{E:Def_WTD} with $\eta=v_k$ and $\eta_i=\underline{D}_iv_k$,
  \begin{align*}
    J_1 = -\sum_{k=1}^3(\underline{D}_iv_k,\underline{D}_k\xi)_{L^2(\Gamma)} = \sum_{k=1}^3\{(v_k,\underline{D}_i\underline{D}_k\xi)_{L^2(\Gamma)}+(v_kHn_i,\underline{D}_k\xi)_{L^2(\Gamma)}\}.
  \end{align*}
  To the first term on the right-hand side we apply \eqref{E:TD_Exc}.
  Then we get
  \begin{align*}
    \sum_{k=1}^3(v_k,\underline{D}_i\underline{D}_k\xi)_{L^2(\Gamma)} &= \sum_{k=1}^3(v_k,\underline{D}_k\underline{D}_i\xi+[W\nabla_\Gamma\xi]_in_k-[W\nabla_\Gamma\xi]_kn_i)_{L^2(\Gamma)}.
  \end{align*}
  Moreover, by \eqref{E:Def_WTD}, $\mathrm{div}_\Gamma v=0$, and $v\cdot n=0$ on $\Gamma$,
  \begin{align*}
    \sum_{k=1}^3(v_k,\underline{D}_k\underline{D}_i\xi)_{L^2(\Gamma)} &= -\sum_{k=1}^3(\underline{D}_kv_k+v_kHn_k,\underline{D}_i\xi)_{L^2(\Gamma)} \\
    &= -(\mathrm{div}_\Gamma v+(v\cdot n)H,\underline{D}_i\xi)_{L^2(\Gamma)} = 0, \\
    \sum_{k=1}^3(v_k,[W\nabla_\Gamma\xi]_in_k)_{L^2(\Gamma)} &= (v\cdot n,[W\nabla_\Gamma\xi]_i)_{L^2(\Gamma)} = 0.
  \end{align*}
  By these equalities and $\sum_{k=1}^3(v_k,[W_\Gamma\xi]_kn_i)_{L^2(\Gamma)}=(n_iv,W\nabla_\Gamma\xi)_{L^2(\Gamma)}$ we get
  \begin{align} \label{Pf_KiRe:J1}
    J_1 = -(n_iv,W\nabla_\Gamma\xi)_{L^2(\Gamma)}+\sum_{k=1}^3(v_kHn_i,\underline{D}_k\xi)_{L^2(\Gamma)}.
  \end{align}
  To the last term we again use \eqref{E:Def_WTD} and $v\cdot n=0$ on $\Gamma$.
  Then it follows that
  \begin{align*}
    \sum_{k=1}^3(v_kHn_i,\underline{D}_k\xi)_{L^2(\Gamma)} &= -\sum_{k=1}^3(\underline{D}_k(v_kHn_i)+(v_kHn_i)Hn_k,\xi)_{L^2(\Gamma)} \\
    &= -(\mathrm{div}_\Gamma(n_iHv)+n_iH^2(v\cdot n),\xi)_{L^2(\Gamma)} \\
    &= -(\mathrm{div}_\Gamma(n_iHv),\xi)_{L^2(\Gamma)}.
  \end{align*}
  Combining this equality with \eqref{Pf_KiRe:TGr_Vi} and \eqref{Pf_KiRe:J1} we obtain
  \begin{align*}
    (\nabla_\Gamma v_i,\nabla_\Gamma\xi)_{L^2(\Gamma)} = -(\mathrm{div}_\Gamma(n_iHv),\xi)_{L^2(\Gamma)} \quad\text{for all}\quad \xi \in C^2(\Gamma).
  \end{align*}
  Since $C^2(\Gamma)$ is dense in $H^1(\Gamma)$ by Lemma~\ref{L:Wmp_Appr}, this equality is valid for all $\xi\in H^1(\Gamma)$.
  Hence $v_i$ satisfies \eqref{Pf_KiRe:Poi} in the weak sense for $i=1,2,3$.

  Now we recall that $\Gamma$ is assumed to be of class $C^\ell$ with $\ell\geq3$.
  Let us prove the $C^{\ell-3}$-regularity of $v$ on $\Gamma$.
  By a localization argument with a partition of unity on $\Gamma$ we may assume that $v$ is compactly supported in a relatively open subset $\mu(U)$ of $\Gamma$, where $U$ is an open set in $\mathbb{R}^2$ and $\mu\colon U\to\Gamma$ is a local parametrization of $\Gamma$.
  Let $\theta$ be the Riemannian metric of $\Gamma$ on $U$, $\theta^{-1}=(\theta^{kl})_{k,l}$ its inverse, and
  \begin{align*}
    \tilde{v}_i(s) := v_i(\mu(s)), \quad b_i^j(s) := [\underline{D}_j(n_iH)](\mu(s)), \quad s\in U,\,i,j=1,2,3.
  \end{align*}
  For $i=1,2,3$ the function $\tilde{v}_i$ is compactly supported in $U$ and belongs to $H^1(U)$ by Lemma~\ref{L:Metric}.
  Moreover, by \eqref{Pf_KiRe:Poi} it is a weak solution to the linear elliptic PDE
  \begin{align*}
    \frac{1}{\sqrt{\det\theta}}\sum_{k,l=1}^2\partial_{s_k}\Bigl(\theta^{kl}(\partial_{s_l}\tilde{v}_i)\sqrt{\det\theta}\Bigr) = \sum_{j=1}^3b_i^j\tilde{v}_j \quad\text{on}\quad U.
  \end{align*}
  Noting that $\theta,\theta^{-1}\in C^{\ell-1}(U)^{2\times2}$ and $b_i^j\in C^{\ell-3}(U)$ by the $C^\ell$-regularity of $\Gamma$ (see Section~\ref{SS:Pre_Surf}) and $\tilde{v}_i$ is compactly supported in $U$, we apply the elliptic regularity theorem (see~\cite{Ev10,GiTr01}) and a bootstrap argument to the above equation to get $\tilde{v}_i\in H^{\ell-1}(U)$.
  Hence the Sobolev embedding theorem (see~\cite{AdFo03}) yields $\tilde{v}_i\in C^{\ell-3}(U)$ for $i=1,2,3$, which implies the $C^{\ell-3}$-regularity of $v=(v_1,v_2,v_3)$ on $\Gamma$ (note that $v$ is compactly supported in $\mu(U)$ and $\mu$ is a $C^\ell$ mapping by the regularity of $\Gamma$).
\end{proof}

\section{Riemannian connection on a surface} \label{S:Ap_RC}
In this appendix we introduce the Riemannian (or Levi-Civita) connection on a surface in $\mathbb{R}^3$ and give several formulas for the covariant derivatives.

Let $\Gamma$ be a two-dimensional closed, connected, and oriented surface in $\mathbb{R}^3$.
We assume that $\Gamma$ is of class $C^3$ and use the notations for the surface quantities on $\Gamma$ given in Section~\ref{SS:Pre_Surf}.
For $X\in C^1(\Gamma,T\Gamma)$ and $Y\in C(\Gamma,T\Gamma)$ we define the covariant derivative of $X$ along $Y$ as
\begin{align} \label{E:Def_Covari}
  \overline{\nabla}_YX := P\left\{(Y\cdot\nabla)\widetilde{X}\right\} \quad\text{on}\quad \Gamma.
\end{align}
Here $\widetilde{X}$ is a $C^1$-extension of $X$ to an open neighborhood of $\Gamma$ with $\widetilde{X}|_\Gamma=X$.
Note that $(Y\cdot\nabla)\widetilde{X}=(Y\cdot\nabla_\Gamma)X$ since $Y$ is tangential on $\Gamma$.
Hence the value of $\overline{\nabla}_YX$ does not depend on the choice of an extension of $X$.
The directional derivative $(Y\cdot\nabla_\Gamma)X$ is expressed in terms of $\overline{\nabla}_YX$ and the Weingarten map $W$.

\begin{lemma} \label{L:Gauss}
  For $X\in C^1(\Gamma,T\Gamma)$ and $Y\in C(\Gamma,T\Gamma)$ we have
  \begin{align} \label{E:Gauss}
    (Y\cdot\nabla)\widetilde{X} = (Y\cdot\nabla_\Gamma)X = \overline{\nabla}_YX+(WX\cdot Y)n \quad\text{on}\quad \Gamma,
  \end{align}
  where $\widetilde{X}$ is any $C^1$-extension of $X$ to an open neighborhood of $\Gamma$ satisfying $\widetilde{X}|_\Gamma=X$.
\end{lemma}

\begin{proof}
  Since $X\cdot n=0$ and $-\nabla_\Gamma n=W$ on $\Gamma$,
  \begin{align*}
    (Y\cdot\nabla_\Gamma)X\cdot n &= Y\cdot\nabla_\Gamma(X\cdot n)-X\cdot(Y\cdot\nabla_\Gamma)n \\
    &= X\cdot(-\nabla_\Gamma n)^TY = X\cdot W^TY = WX\cdot Y
  \end{align*}
  on $\Gamma$.
  Combining this with \eqref{E:Def_Covari} we obtain \eqref{E:Gauss}.
\end{proof}

The formula \eqref{E:Gauss} is called the Gauss formula (see e.g.~\cite[Section~4.2]{Ch15} and~\cite[Section~VII.3]{KoNo96}).
Let us prove fundamental properties of the covariant derivative.

\begin{lemma} \label{L:RiCo}
  The following equalities hold on $\Gamma$:
  \begin{itemize}
    \item For $X\in C^1(\Gamma,T\Gamma)$, $Y,Z\in C(\Gamma,T\Gamma)$, and $\eta,\xi\in C(\Gamma)$,
    \begin{align} \label{E:RiCo_AfY}
      \overline{\nabla}_{\eta Y+\xi Z}X = \eta\overline{\nabla}_YX+\xi\overline{\nabla}_ZX.
    \end{align}
    \item For $X\in C^1(\Gamma,T\Gamma)$, $Y\in C(\Gamma,T\Gamma)$, and $\eta\in C^1(\Gamma)$,
    \begin{align} \label{E:RiCo_AfX}
      \overline{\nabla}_Y(\eta X) = (Y\cdot\nabla_\Gamma\eta)X+\eta\overline{\nabla}_YX.
    \end{align}
    \item For $X,Y\in C^1(\Gamma,T\Gamma)$ and $Z\in C(\Gamma,T\Gamma)$,
    \begin{align} \label{E:RiCo_Met}
      Z\cdot\nabla_\Gamma(X\cdot Y) = \overline{\nabla}_ZX\cdot Y+X\cdot\overline{\nabla}_ZY.
    \end{align}
    \item For $X,Y\in C^1(\Gamma,T\Gamma)$ and $\eta\in C^2(\Gamma)$,
    \begin{align} \label{E:RiCo_Tor}
      X\cdot\nabla_\Gamma(Y\cdot\nabla_\Gamma\eta)-Y\cdot\nabla_\Gamma(X\cdot\nabla_\Gamma\eta) = \Bigl(\overline{\nabla}_XY-\overline{\nabla}_YX\Bigr)\cdot\nabla_\Gamma\eta.
    \end{align}
  \end{itemize}
\end{lemma}

\begin{proof}
  The equalities \eqref{E:RiCo_AfY} and \eqref{E:RiCo_AfX} immediately follow from \eqref{E:Def_Covari}.
  Also, we apply \eqref{E:Gauss} and $X\cdot n=Y\cdot n=0$ on $\Gamma$ to the right-hand side of
  \begin{align*}
    Z\cdot\nabla_\Gamma(X\cdot Y) = (Z\cdot\nabla_\Gamma)X\cdot Y+X\cdot(Z\cdot\nabla_\Gamma)Y \quad\text{on}\quad \Gamma
  \end{align*}
  to get \eqref{E:RiCo_Met}.
  Let us prove \eqref{E:RiCo_Tor}.
  The left-hand side of \eqref{E:RiCo_Tor} is of the form
  \begin{gather*}
    \sum_{i,j=1}^3\{X_i\underline{D}_i(Y_j\underline{D}_j\eta)-Y_i\underline{D}_i(X_j\underline{D}_j\eta)\} = J_1+J_2, \\
    J_1 := \sum_{i,j=1}^3(X_i\underline{D}_iY_j-Y_i\underline{D}_iX_j)\underline{D}_j\eta, \quad J_2 := \sum_{i,j=1}^3(X_iY_j-X_jY_i)\underline{D}_i\underline{D}_j\eta.
  \end{gather*}
  By \eqref{E:Gauss} and $\nabla_\Gamma\eta\cdot n=0$ on $\Gamma$ we have
  \begin{align*}
    J_1 = \{(X\cdot\nabla_\Gamma)Y-(Y\cdot\nabla_\Gamma)X\}\cdot\nabla_\Gamma\eta = \Bigl(\overline{\nabla}_XY-\overline{\nabla}_YX\Bigr)\cdot\nabla_\Gamma\eta.
  \end{align*}
  Also, using \eqref{E:TD_Exc} and $X\cdot n=Y\cdot n=0$ on $\Gamma$ we observe that
  \begin{align*}
    J_2 &= \sum_{i,j=1}^3X_iY_j(\underline{D}_i\underline{D}_j\eta-\underline{D}_j\underline{D}_i\eta) = \sum_{i,j=1}^3X_iY_j([W\nabla\eta]_in_j-[W\nabla_\Gamma\eta]_jn_i) \\
    &= (X\cdot W\nabla_\Gamma\eta)(Y\cdot n)-(X\cdot n)(Y\cdot W\nabla_\Gamma\eta) = 0.
  \end{align*}
  Combining the above three equalities we obtain \eqref{E:RiCo_Tor}.
\end{proof}

By Lemma~\ref{L:RiCo} the assignment $\overline{\nabla}\colon(X,Y)\mapsto\overline{\nabla}_YX$ defines the Riemannian (or Levi-Civita) connection on $\Gamma$ (see e.g.~\cite{Ch15,Jo11,Pe06,Ta11_I}).
Note that the formula \eqref{E:RiCo_Tor} stands for the torsion-free condition $[X,Y]=\overline{\nabla}_XY-\overline{\nabla}_YX$, where $[X,Y]=XY-YX$ is the Lie bracket of $X$ and $Y$.

Let $U$ be a relatively open subset of $\Gamma$ and $\tau_1$ and $\tau_2$ be $C^1$ tangential vector fields on $U$ such that $\{\tau_1(y),\tau_2(y)\}$ is an orthonormal basis of the tangent plane of $\Gamma$ at each $y\in U$ (by the regularity of $\Gamma$ such vector fields exist on a sufficiently small $U$).
We call the pair $\{\tau_1,\tau_2\}$ a local orthonormal frame of the tangent plane of $\Gamma$ on $U$, or briefly a local orthonormal frame on $U$.
Note that
\begin{align} \label{E:MC_Local}
  H = \mathrm{tr}[W] = W\tau_1\cdot\tau_1+W\tau_2\cdot\tau_2 \quad\text{on}\quad U
\end{align}
since $\{\tau_1,\tau_2,n\}$ is an orthonormal basis of $\mathbb{R}^3$ and $Wn=0$ on $\Gamma$.
We express several quantities related to the tangential gradient matrix of tangential vector fields on $\Gamma$ in terms of the covariant derivatives and the local orthonormal frame.

\begin{lemma} \label{L:Form_Cov}
  Let $\{\tau_1,\tau_2\}$ be a local orthonormal frame of the tangent plane of $\Gamma$ on a relatively open subset $U$ of $\Gamma$.
  For $X,Y\in C^1(\Gamma,T\Gamma)$ we have
  \begin{align}
    \mathrm{div}_\Gamma X &= \sum_{i=1,2}\overline{\nabla}_iX\cdot\tau_i, \label{E:Sdiv_Cov} \\
    \nabla_\Gamma X:W &= \sum_{i=1,2}\overline{\nabla}_iX\cdot W\tau_i = \sum_{i=1,2}W\overline{\nabla}_iX\cdot\tau_i, \label{E:Wtr_Cov} \\
    \nabla_\Gamma X:\nabla_\Gamma Y &= \sum_{i=1,2}\overline{\nabla}_iX\cdot\overline{\nabla}_iY+WX\cdot WY, \label{E:Inn_Cov} \\
    \nabla_\Gamma X:(\nabla_\Gamma Y)P &= \sum_{i=1,2}\overline{\nabla}_iX\cdot\overline{\nabla}_iY, \label{E:InnP_Cov} \\
    W\nabla_\Gamma X:(\nabla_\Gamma Y)P &= \sum_{i=1,2}\overline{\nabla}_{W\tau_i}X\cdot\overline{\nabla}_iY \label{E:Winn_Cov}
  \end{align}
  on $U$, where $\overline{\nabla}_i:=\overline{\nabla}_{\tau_i}$ for $i=1,2$.
\end{lemma}

\begin{proof}
  We carry out calculations on $U$.
  By \eqref{E:P_TGr} and \eqref{E:Gauss} we have
  \begin{align} \label{Pf_FC:Form}
    \begin{aligned}
      (\nabla_\Gamma X)^T\tau_i &= (\tau_i\cdot\nabla_\Gamma)X = \overline{\nabla}_iX+(WX\cdot\tau_i)n, \quad i=1,2, \\
      (\nabla_\Gamma X)^Tn &= (n\cdot\nabla_\Gamma)X = 0.
    \end{aligned}
  \end{align}
  Since $\{\tau_1,\tau_2,n\}$ forms an orthonormal basis of $\mathbb{R}^3$,
  \begin{align*}
    \mathrm{div}_\Gamma X = \mathrm{tr}[\nabla_\Gamma X] = \sum_{i=1,2}(\nabla_\Gamma X)^T\tau_i\cdot\tau_i+(\nabla_\Gamma X)^Tn\cdot n.
  \end{align*}
  The equality \eqref{E:Sdiv_Cov} follows from this equality and \eqref{Pf_FC:Form}.
  We also obtain \eqref{E:Wtr_Cov} by applying \eqref{Pf_FC:Form}, $W^T=W$, and $Wn=0$ to the right-hand side of
  \begin{align*}
    \nabla_\Gamma X:W = (\nabla_\Gamma X)^T:W^T = \sum_{i=1,2}(\nabla_\Gamma X)^T\tau_i\cdot W^T\tau_i+(\nabla_\Gamma X)^Tn\cdot W^Tn.
  \end{align*}
  Let us prove \eqref{E:Inn_Cov}.
  By \eqref{E:Gauss}, \eqref{Pf_FC:Form}, and $\overline{\nabla}_iX\cdot n=\overline{\nabla}_iY\cdot n=0$,
  \begin{align*}
    \nabla_\Gamma X:\nabla_\Gamma Y &= (\nabla_\Gamma X)^T:(\nabla_\Gamma Y)^T \\
    &= \sum_{i=1,2}(\nabla_\Gamma X)^T\tau_i\cdot(\nabla_\Gamma Y)^T\tau_i+(\nabla_\Gamma X)^Tn\cdot(\nabla_\Gamma Y)^Tn \\
    &= \sum_{i=1,2}\Bigl\{\overline{\nabla}_iX+(WX\cdot\tau_i)n\Bigr\}\cdot\Bigl\{\overline{\nabla}_iY+(WY\cdot\tau_i)n\Bigr\} \\
    &= \sum_{i=1,2}\overline{\nabla}_iX\cdot\overline{\nabla}_iY+\sum_{i=1,2}(WX\cdot\tau_i)(WY\cdot\tau_i).
  \end{align*}
  Here the last term is equal to $WX\cdot WY$ since $WX$ and $WY$ are tangential on $\Gamma$ and $\{\tau_1,\tau_2\}$ is a local orthonormal frame.
  Hence \eqref{E:Inn_Cov} is valid.
  Also, since
  \begin{align*}
  (\nabla_\Gamma X)^TW\tau_i &= (W\tau_i\cdot\nabla_\Gamma)X = \overline{\nabla}_{W\tau_i}X+(WX\cdot W\tau_i)n, \\
  P(\nabla_\Gamma Y)^T\tau_i &= P\Bigl\{\overline{\nabla}_iY+(WY\cdot\tau_i)n\Bigr\} = P\overline{\nabla}_iY = \overline{\nabla}_iY
  \end{align*}
  by \eqref{E:Gauss} and $Pn=0$, we can prove \eqref{E:InnP_Cov} and \eqref{E:Winn_Cov} as above by using the above relations, \eqref{Pf_FC:Form}, $P^T=P$, and $W^T=W$.
\end{proof}

Next we give an integration by parts formula for integrals over $\Gamma$ of the covariant derivatives along vector fields of a local orthonormal frame.

\begin{lemma} \label{L:IbP_Cov}
  Let $\{\tau_1,\tau_2\}$ be a local orthonormal frame of the tangent plane of $\Gamma$ on a relatively open subset $U$ of $\Gamma$ and $\overline{\nabla}_i:=\overline{\nabla}_{\tau_i}$ for $i=1,2$.
  Suppose that $X\in C^2(\Gamma,T\Gamma)$ and $Y\in C^1(\Gamma,T\Gamma)$ are compactly supported in $U$.
  Then we have
  \begin{align} \label{E:IbP_Cov}
    \sum_{i=1,2}\int_\Gamma\Bigl(\overline{\nabla}_i\overline{\nabla}_iX-\overline{\nabla}_{\overline{\nabla}_i\tau_i}X\Bigr)\cdot Y\,d\mathcal{H}^2 = -\sum_{i=1,2}\int_\Gamma\overline{\nabla}_iX\cdot\overline{\nabla}_iY\,d\mathcal{H}^2.
  \end{align}
\end{lemma}

\begin{proof}
  The proof is basically the same as that of~\cite[Proposition~34]{Pe06}.
  Let
  \begin{align*}
    Z := \sum_{i=1,2}\Bigl(\overline{\nabla}_iX\cdot Y\Bigr)\tau_i \quad\text{on}\quad U.
  \end{align*}
  We extend $Z$ to $\Gamma$ by setting zero outside of $U$ to get $Z\in C^1(\Gamma,T\Gamma)$, since $\tau_1$ and $\tau_2$ are of class $C^1$ on $U$ and $X\in C^2(\Gamma,T\Gamma)$ and $Y\in C^1(\Gamma,T\Gamma)$ are compactly supported in $U$.
  Moreover, since $\{\tau_1,\tau_2\}$ is a local orthonormal frame, we see by \eqref{E:RiCo_AfY} that $Z\cdot V=\overline{\nabla}_VX\cdot Y$ on $\Gamma$ for all $V\in C(\Gamma,T\Gamma)$.
  By this fact and \eqref{E:RiCo_Met},
  \begin{align*}
    \overline{\nabla}_iZ\cdot\tau_i = \tau_i\cdot\nabla_\Gamma(Z\cdot\tau_i)-Z\cdot\overline{\nabla}_i\tau_i = \tau_i\cdot\nabla_\Gamma\Bigl(\overline{\nabla}_iX\cdot Y\Bigr)-\overline{\nabla}_{\overline{\nabla}_i\tau_i}X\cdot Y \quad\text{on}\quad U
  \end{align*}
  for $i=1,2$.
  Applying \eqref{E:RiCo_Met} again to the first term on the right-hand side we get
  \begin{align*}
    \overline{\nabla}_iZ\cdot\tau_i = \Bigl(\overline{\nabla}_i\overline{\nabla}_iX-\overline{\nabla}_{\overline{\nabla}_i\tau_i}X\Bigr)\cdot Y+\overline{\nabla}_iX\cdot\overline{\nabla}_iY \quad\text{on}\quad U.
  \end{align*}
  By this equality and \eqref{E:Sdiv_Cov} we see that
  \begin{align} \label{Pf_ICov:Div}
    \mathrm{div}_\Gamma Z = \sum_{i=1,2}\left\{\Bigl(\overline{\nabla}_i\overline{\nabla}_iX-\overline{\nabla}_{\overline{\nabla}_i\tau_i}X\Bigr)\cdot Y+\overline{\nabla}_iX\cdot\overline{\nabla}_iY\right\} \quad\text{on}\quad U.
  \end{align}
  Since $X$, $Y$, and $Z$ are supported in $U$, we may assume that \eqref{Pf_ICov:Div} holds on the whole surface $\Gamma$.
  Hence we obtain \eqref{E:IbP_Cov} by integrating both sides of \eqref{Pf_ICov:Div} over $\Gamma$, since the integral of $\mathrm{div}_\Gamma Z$ over $\Gamma$ vanishes by the Stokes theorem (note that $Z$ is tangential and $\Gamma$ has no boundary).
\end{proof}

The formula \eqref{E:IbP_Cov} gives a relation between two Laplacians acting on tangential vector fields on $\Gamma$.
For $X\in C^2(\Gamma,T\Gamma)$, $Y\in C^1(\Gamma,T\Gamma)$, and $Z\in C(\Gamma,T\Gamma)$ we define the second covariant derivative by
\begin{align*}
  \overline{\nabla}_{Z,Y}^2X := \overline{\nabla}_Z\overline{\nabla}_YX-\overline{\nabla}_{\overline{\nabla}_ZY}X \quad\text{on}\quad \Gamma.
\end{align*}
The trace of the second covariant derivative is called the connection Laplacian:
\begin{align*}
  \mathrm{tr}\overline{\nabla}^2X := \sum_{i=1,2}\overline{\nabla}_{\tau_i,\tau_i}^2X = \sum_{i=1,2}\Bigl(\overline{\nabla}_i\overline{\nabla}_iX-\overline{\nabla}_{\overline{\nabla}_i\tau_i}X\Bigr) \quad\text{on}\quad \Gamma, \quad X\in C^2(\Gamma,T\Gamma).
\end{align*}
Here $\{\tau_1,\tau_2\}$ is a local orthonormal frame and $\overline{\nabla}_i=\overline{\nabla}_{\tau_i}$ for $i=1,2$.
On the other hand, the Bochner Laplacian $\Delta_BX$ of $X\in C^2(\Gamma,T\Gamma)$ is defined as a tangential vector field on $\Gamma$ satisfying
\begin{align} \label{E:Def_Boch}
  \int_\Gamma\Delta_BX\cdot Y\,d\mathcal{H}^2 = -\sum_{i=1,2}\int_\Gamma\overline{\nabla}_iX\cdot\overline{\nabla}_iY\,d\mathcal{H}^2
\end{align}
for all $Y\in C^1(\Gamma,T\Gamma)$, where we localize $X$ and $Y$ by using a partition of unity on $\Gamma$.
(Note that there are other definitions of these Laplacians in which one takes the opposite sign.)
Then by \eqref{E:IbP_Cov} these Laplacians agree, i.e.
\begin{align*}
  \mathrm{tr}\overline{\nabla}^2X = \Delta_BX \quad\text{on}\quad \Gamma, \quad X\in C^2(\Gamma,T\Gamma).
\end{align*}
When the surface $\Gamma$ is embedded in $\mathbb{R}^3$, we can also consider the Laplace--Beltrami operator acting on each component of a vector field $X=(X_1,X_2,X_3)$ on $\Gamma$:
\begin{align*}
  \Delta_\Gamma X = (\Delta_\Gamma X_1,\Delta_\Gamma X_2,\Delta_\Gamma X_3), \quad \Delta_\Gamma X_j = \sum_{i=1}^3\underline{D}_i^2X_j, \quad j=1,2,3.
\end{align*}
Let us give a relation between the two Laplacians $\Delta_B$ and $\Delta_\Gamma$.

\begin{lemma} \label{L:Boch_Lap}
  For $X\in C^2(\Gamma,T\Gamma)$ we have
  \begin{align} \label{E:Boch_Lap}
    \Delta_BX = P\Delta_\Gamma X+W^2X \quad\text{on}\quad \Gamma.
  \end{align}
\end{lemma}

\begin{proof}
  The relation \eqref{E:Boch_Lap} is proved in \cite[Appendix~B]{Mi18}.
  Here we give another proof of it.
  By a localization argument with a partition of unity on $\Gamma$, it is sufficient to show \eqref{E:Boch_Lap} on a relatively open subset $U$ of $\Gamma$ on which we can take a local orthonormal frame $\{\tau_1,\tau_2\}$.
  Let $Y\in C^1(\Gamma,T\Gamma)$ be compactly supported in $U$.
  By \eqref{E:Inn_Cov} and \eqref{E:Def_Boch} we have
  \begin{align} \label{Pf_BL:First}
    \begin{aligned}
      \int_\Gamma\Delta_BX\cdot Y\,d\mathcal{H}^2 &= -\sum_{i=1,2}\int_\Gamma\overline{\nabla}_iX\cdot\overline{\nabla}_iY\,d\mathcal{H}^2 \\
      &= -\int_\Gamma\nabla_\Gamma X:\nabla_\Gamma Y\,d\,\mathcal{H}^2+\int_\Gamma WX\cdot WY\,d\mathcal{H}^2.
    \end{aligned}
  \end{align}
  To the first integral on the last line we use \eqref{E:IbP_TD} to get
  \begin{align*}
    \int_\Gamma\nabla_\Gamma X:\nabla_\Gamma Y\,d\,\mathcal{H}^2 &= \sum_{i,j=1}^3\int_\Gamma(\underline{D}_iX_j)(\underline{D}_iY_j)\,d\mathcal{H}^2 \\
    &= -\sum_{i,j=1}^3\int_\Gamma\{(\underline{D}_i^2X_j)Y_j+(\underline{D}_iX_j)Y_jHn_i\}\,d\mathcal{H}^2 \\
    &= -\int_\Gamma\Delta_\Gamma X\cdot Y\,d\mathcal{H}^2-\int_\Gamma\{(n\cdot\nabla_\Gamma)X\cdot Y\}H\,d\mathcal{H}^2.
  \end{align*}
  From this equality, \eqref{Pf_BL:First}, $(n\cdot\nabla_\Gamma)X=0$, and $W^T=W$ on $\Gamma$ we deduce that
  \begin{align*}
    \int_\Gamma\Delta_BX\cdot Y\,d\mathcal{H}^2 = \int_\Gamma(\Delta_\Gamma X+W^2X)\cdot Y\,d\mathcal{H}^2.
  \end{align*}
  Thus, setting $Y:=Pv$ and noting that $\Delta_BX$ and $W^2X$ are tangential on $\Gamma$ we get
  \begin{align*}
    \int_\Gamma \Delta_BX\cdot v\,d\mathcal{H}^2 = \int_\Gamma(P\Delta_\Gamma X+W^2X)\cdot v\,d\mathcal{H}^2
  \end{align*}
  for all $v\in C^1(\Gamma)^3$ compactly supported in $U$.
  Hence by the fundamental lemma of calculus of variations we conclude that \eqref{E:Boch_Lap} holds on $U$.
\end{proof}

Note that the normal component of $\Delta_\Gamma X$ does not vanish in general even if $X$ is tangential on $\Gamma$.
Indeed, by $X\cdot n=0$ and $-\nabla_\Gamma n=W$ on $\Gamma$,
\begin{align*}
  \Delta_\Gamma X\cdot n = \mathrm{div}_\Gamma(WX)+W:\nabla_\Gamma X = \mathrm{div}_\Gamma W\cdot X+2W:\nabla_\Gamma X \quad\text{on}\quad \Gamma.
\end{align*}
When $\Gamma$ is a flat domain in $\mathbb{R}^2$, we have $W=-\nabla_\Gamma n=0$ by $n=(0,0,1)$.
Hence the normal component of $\Delta_\Gamma X$ vanishes and $\Delta_BX=\Delta_\Gamma X$ reduces to the usual Laplacian on $\mathbb{R}^2$ acting on each component of $X=(X_1,X_2)$.

\begin{remark} \label{R:RC}
  When $\Gamma$ is of class $C^3$, the space $C^2(\Gamma,T\Gamma)$ is dense in $H^m(\Gamma,T\Gamma)$ for $m=0,1,2$ by Lemma~\ref{L:Wmp_Tan_Appr}.
  Hence the formulas given in this appendix are also valid (a.e. on $\Gamma$) if we replace $C^m(\Gamma,T\Gamma)$, $m=0,1,2$ with $H^m(\Gamma,T\Gamma)$.
\end{remark}

\section{Infinitesimal rigid displacements on a closed surface} \label{S:Ap_IR}
We provide several results on infinitesimal rigid displacements of $\mathbb{R}^3$ which have tangential restrictions on a closed surface.

Let $\Gamma$ be a closed, connected, and oriented surface in $\mathbb{R}^3$ of class $C^2$ and $\mathcal{R}$ the function space given by \eqref{E:Def_R}.

\begin{lemma} \label{L:IR_Surf}
  Let $w(x)=a\times x+b\in\mathcal{R}$.
  If $w\not\equiv0$, then $a\neq0$, $a\cdot b=0$, and $\Gamma$ is axially symmetric around the axis parallel to the vector $a$ and passing through the point $b_a:=|a|^{-2}(a\times b)$.
  Conversely, if $\Gamma$ is axially symmetric around the axis parallel to $a\neq0$ and passing through $\tilde{b}\in\mathbb{R}^3$, then $\tilde{w}(x)=a\times(x-\tilde{b})\in\mathcal{R}\setminus\{0\}$.
\end{lemma}

\begin{proof}
  Suppose that $w(x)=a\times x+b\in\mathcal{R}$ does not identically vanish.
  If $a=0$, then $w\cdot n=b\cdot n=0$ on $\Gamma$, which yields $b=0$ since $\Gamma$ is closed.
  Hence $a\neq0$ when $w\not\equiv0$.
  Now we consider the flow map $x(\cdot,t)\colon\mathbb{R}^3\to\mathbb{R}^3$ of $w$ (here $\dot{x}=\partial x/\partial t$)
  \begin{align} \label{Pf_IRS:Flow}
    x(X,0) = X \in \mathbb{R}^3, \quad \dot{x}(X,t) = w(x(X,t),t) = a\times x(X,t)+b, \quad t>0.
  \end{align}
  Since $a\neq0$, we can take an orthonormal basis $\{E_1,E_2,E_3\}$ of $\mathbb{R}^3$ such that
  \begin{align} \label{Pf_IRS:ONB}
    E_3 = |a|^{-1}a, \quad E_1\times E_2 = E_3, \quad E_2\times E_3 = E_1, \quad E_3\times E_1 = E_2.
  \end{align}
  In what follows, we write $x_E^i:=x\cdot E_i$ for $x\in\mathbb{R}^3$ and $i=1,2,3$.
  Then since
  \begin{align*}
    a\times x = |a|E_3\times(x_E^1E_1+x_E^2E_2+x_E^3E_3) = -|a|x_E^2E_1+|a|x_E^1E_2,
  \end{align*}
  the system \eqref{Pf_IRS:Flow} for $x(t)=\sum_{i=1}^3x_E^i(t)E_i$ is equivalent to
  \begin{align*}
    \left\{
    \begin{aligned}
      \dot{x}_E^1(t) &= -|a|x_E^2(t)+b_E^1, & & x_E^1(0) = X_E^1, \\
      \dot{x}_E^2(t) &= |a|x_E^1(t)+b_E^2, & & x_E^2(0) = X_E^2, \\
      \dot{x}_E^3(t) &= b_E^3, & & x_E^3(0) = X_E^3.
    \end{aligned}
    \right.
  \end{align*}
  Here we suppressed the argument $X$.
  We solve this system to get
  \begin{align} \label{Pf_IRS:Sol}
    \left\{
    \begin{aligned}
      x_E^1(t)+|a|^{-1}b_E^2 &= (X_E^1+|a|^{-1}b_E^2)\cos(|a|t)-(X_E^2-|a|^{-1}b_E^1)\sin(|a|t), \\
      x_E^2(t)-|a|^{-1}b_E^1 &= (X_E^1+|a|^{-1}b_E^2)\sin(|a|t)+(X_E^2-|a|^{-1}b_E^1)\cos(|a|t), \\
      x_E^3(t) &= X_E^3+b_E^3t.
    \end{aligned}
    \right.
  \end{align}
  Now we observe by $w\cdot n=0$ on $\Gamma$ that $x(\cdot,t)$ maps $\Gamma$ into itself for all $t>0$.
  This fact and the compactness of $\Gamma$ implies that $x_E^3(t)=X_E^3+b_E^3t$ remains bounded for $X\in\Gamma$, which yields $b_E^3=b\cdot E_3=0$, i.e. $a\cdot b=0$.
  Also, setting
  \begin{align} \label{Pf_IRS:Ba}
    b_a := -|a|^{-1}b_E^2E_1+|a|^{-1}b_E^1E_2 = |a|^{-2}a\times b,
  \end{align}
  where the second equality follows from $b=b_E^1E_1+b_E^2E_2$ and \eqref{Pf_IRS:ONB}, we get
  \begin{align*}
    x(X,t)-b_a &= \{x_E^1(X,t)+|a|^{-1}b_E^2\}E_1+\{x_E^2(X,t)-|a|^{-1}b_E^1\}E_2+x_E^3(X,t)E_3.
  \end{align*}
  Hence by \eqref{Pf_IRS:Sol} with $b_E^3=0$ we obtain
  \begin{align} \label{Pf_IRS:Rot}
    x(X,t)-b_a = P_aR_a(t)P_a^T(X-b_a), \quad X\in\mathbb{R}^3,\,t\geq0,
  \end{align}
  where $P_a:=(E_1 \, E_2 \, E_3)$ is a $3\times 3$ matrix whose $i$-th column is $E_i$ for $i=1,2,3$ and
  \begin{align} \label{Pf_IRS:Def_Ra}
    R_a(t) :=
    \begin{pmatrix}
      \cos(|a|t) & -\sin(|a|t) & 0 \\
      \sin(|a|t) & \cos(|a|t) & 0 \\
      0 & 0 & 1
    \end{pmatrix}.
  \end{align}
  By \eqref{Pf_IRS:Rot} we observe that the flow map $x(\cdot,t)$ of $w$ is given by the rotation through the angle $|a|t$ around the axis parallel to $E_3=|a|^{-1}a$ and passing through $b_a$.
  Since $x(\cdot,t)$ maps $\Gamma$ into itself for all $t>0$ by $w\cdot n=0$ on $\Gamma$, we conclude that $\Gamma$ is axially symmetric around the axis parallel to $a$ and passing through $b_a$.

  Conversely, suppose that $\Gamma$ is axially symmetric around the axis parallel to $a\neq0$ and passing through $\tilde{b}\in\mathbb{R}^3$.
  Let $\{E_1,E_2,E_3\}$ be an orthonormal basis of $\mathbb{R}^3$ satisfying \eqref{Pf_IRS:ONB}.
  Then by the first part of the proof we see that the mapping
  \begin{align*}
    \Phi_t(X) := P_aR_a(t)P_a^T(X-\tilde{b})+\tilde{b}, \quad X\in\mathbb{R}^3
  \end{align*}
  preserves $\Gamma$ for all $t\in\mathbb{R}$, where $P_a=(E_1 \, E_2 \, E_3)$ and $R_a(t)$ is given by \eqref{Pf_IRS:Def_Ra}.
  Hence for each $Y\in\Gamma$ the time derivative of $\Phi_t(Y)$ at $t=0$ gives a tangent vector on $\Gamma$ at $Y$.
  Moreover, setting $Z:=Y-\tilde{b}$ and $Z_E^i:=Z\cdot E_i$, $i=1,2,3$ we have
  \begin{align*}
    \frac{d}{dt}\Phi_t(Y)\Bigl|_{t=0} &= \frac{d}{dt}P_aR_a(t)P_a^TZ\Bigl|_{t=0} = P_a
    \begin{pmatrix}
      0 & -|a| & 0 \\
      |a| & 0 & 0 \\
      0 & 0 & 0
    \end{pmatrix}
    P_a^TZ \\
    &= |a|(-Z_E^2E_1+Z_E^1E_2) = |a|E_3\times(Z_E^1E_1+Z_E^2E_2+Z_E^3E_3) \\
    &= |a|E_3\times Z = a\times(Y-\tilde{b})
  \end{align*}
  by \eqref{Pf_IRS:ONB}, $E_3\times E_3=0$, and $Z=\sum_{i=1}^3Z_E^iE_i$.
  Hence $\tilde{w}(x)=a\times(x-\tilde{b})\in\mathcal{R}\setminus\{0\}$.
\end{proof}

\begin{lemma} \label{L:SR_Eigen}
  Let $w(x)=a\times x+b\in\mathcal{R}$ satisfy $w\not\equiv0$.
  Then
  \begin{align} \label{E:SR_Eigen}
    W(y)w(y) = \lambda(y)w(y), \quad a\times n(y) = -\lambda(y)w(y) \quad\text{for all}\quad y\in\Gamma
  \end{align}
  with some constant $\lambda(y)\in\mathbb{R}$.
  Here $W=-\nabla_\Gamma n$ is the Weingarten map of $\Gamma$.
\end{lemma}

\begin{proof}
  Since $w(x)=a\times x+b\in\mathcal{R}$ and $w\not\equiv0$, Lemma~\ref{L:IR_Surf} implies that $a\neq0$, $a\cdot b=0$, and $\Gamma$ is axially symmetric around the axis parallel to $a$ and passing through $b_a=|a|^{-2}(a\times b)$.
  Also, since $a\times b_a=-b$ by \eqref{Pf_IRS:ONB}, \eqref{Pf_IRS:Ba}, and $a\cdot b=0$, we have $w(x)=a\times(x-b_a)$.
  Hence by a translation along $b_a$ and a rotation of coordinates we may assume that $\Gamma$ is axially symmetric around the $x_3$-axis and $w(x)=\alpha(e_3\times x)$, where $\alpha=|a|>0$ and $e_3=(0,0,1)$.
  Replacing $w$ by $\alpha^{-1}w$ we may further assume $\alpha=1$, i.e. $a=e_3$ and $w(x)=e_3\times x$.
  Under these assumptions, $\Gamma$ is represented as a surface of revolution
  \begin{align} \label{Pf_SRE:SoR}
    \Gamma = \{\mu(s,\vartheta) = (\varphi(s)\cos\vartheta,\varphi(s)\sin\vartheta,\psi(s)) \mid s\in[0,L],\,\vartheta\in[0,2\pi]\}.
  \end{align}
  Here $\gamma(s)=(\varphi(s),0,\psi(s))$ is a $C^2$ curve parametrized by the arc length $s\in[0,L]$, $L>0$ such that $\varphi(s)>0$ for $s\neq0,L$.
  We may further assume that if $\varphi(s)=0$ for $s=0,L$ then $\psi'(s)=0$, otherwise $\Gamma$ is not of class $C^2$ at $(0,0,\psi(s))$.
  By the arc length parametrization of $\gamma$ we have
  \begin{align} \label{Pf_SRE:Arc}
    \{\varphi'(s)\}^2+\{\psi'(s)\}^2 = 1, \quad s\in[0,L].
  \end{align}
  Let $y=\mu(s,\vartheta)\in\Gamma$.
  We suppress the arguments of $\mu$ and its derivatives.
  From
  \begin{align} \label{Pf_SRE:Tau}
    \partial_s\mu =
    \begin{pmatrix}
      \varphi'(s)\cos\vartheta \\ \varphi'(s)\sin\vartheta \\ \psi'(s)
    \end{pmatrix}, \quad
    \partial_\vartheta\mu =
    \begin{pmatrix}
      -\varphi(s)\sin\vartheta \\ \varphi(s)\cos\vartheta \\ 0
    \end{pmatrix}
    = e_3\times \mu = w(y)
  \end{align}
  and \eqref{Pf_SRE:Arc} we deduce that
  \begin{align*}
    \partial_s\mu\times\partial_\vartheta\mu = \varphi(s)
    \begin{pmatrix}
      -\psi'(s)\cos\vartheta \\ -\psi'(s)\sin\vartheta \\ \varphi'(s)
    \end{pmatrix}, \quad
    |\partial_s\mu\times\partial_\vartheta\mu| = \varphi(s).
  \end{align*}
  Suppose that $\varphi(s)>0$.
  Without loss of generality, we may assume that the direction of $\partial_s\mu\times\partial_\vartheta\mu$ is the same as that of the unit outward normal $n(y)$.
  Then
  \begin{align} \label{Pf_SRE:Nor}
    n(y) = n(\mu(s,\vartheta)) = \frac{\partial_s\mu\times\partial_\vartheta\mu}{|\partial_s\mu\times\partial_\vartheta\mu|} =
    \begin{pmatrix}
      -\psi'(s)\cos\vartheta \\ -\psi'(s)\sin\vartheta \\ \varphi'(s)
    \end{pmatrix}.
  \end{align}
  We differentiate $n(\mu(s,\vartheta))$ with respect to $\vartheta$ and use \eqref{Pf_SRE:Tau} to get
  \begin{align*}
    \frac{\partial}{\partial \vartheta}\bigl(n(\mu(s,\vartheta))\bigr) =
    \begin{pmatrix}
      \psi'(s)\sin\vartheta \\ -\psi'(s)\cos\vartheta \\ 0
    \end{pmatrix}
    = -\lambda(y)w(y), \quad \lambda(y) := \frac{\psi'(s)}{\varphi(s)}.
  \end{align*}
  Moreover, by \eqref{E:ConDer_Surf} with $y=\mu(s,\vartheta)\in\Gamma$, $-\nabla_\Gamma n=W=W^T$ on $\Gamma$, and \eqref{Pf_SRE:Tau},
  \begin{align*}
    \frac{\partial}{\partial \vartheta}\bigl(n(\mu(s,\vartheta))\bigr) = [\nabla_\Gamma n(y)]^T\partial_\vartheta\mu = -W(y)w(y).
  \end{align*}
  Hence $W(y)w(y)=\lambda(y)w(y)$.
  We also have $e_3\times n(y)=-\lambda(y)w(y)$ by \eqref{Pf_SRE:Tau} and \eqref{Pf_SRE:Nor}.
  Therefore, \eqref{E:SR_Eigen} is valid when $\varphi(s)>0$ (note that we assume $a=e_3$).

  Now suppose that $\varphi(s)=0$ for $s=0,L$.
  Then $\psi'(s)=0$ by our assumption and thus the tangent plane of $\Gamma$ at $y=(0,0,\psi(s))$ is orthogonal to the $x_3$-axis.
  Hence $n(y)=\pm e_3$ and $e_3\times n(y)=0$.
  Moreover, $w(y)=0$ by \eqref{Pf_SRE:Tau} and $\varphi(s)=0$.
  By these facts we conclude that the equalities \eqref{E:SR_Eigen} are valid for any $\lambda(y)\in\mathbb{R}$.
\end{proof}

Let $\mathcal{K}(\Gamma)$ be the space of Killing vector fields on $\Gamma$ given by \eqref{E:Def_Kil}.
We show that $\mathcal{K}(\Gamma)=\mathcal{R}|_\Gamma=\{w|_\Gamma\mid w\in\mathcal{R}\}$ if $\Gamma$ is axially symmetric.

\begin{lemma} \label{L:IR_Kil}
  Suppose that $\Gamma$ is of class $C^5$ and $\mathcal{R}\neq\{0\}$.
  Then $\mathcal{K}(\Gamma)=\mathcal{R}|_\Gamma$.
\end{lemma}

\begin{proof}
  For $w(x)=a\times x+b$ with $a=(a_1,a_2,a_3)\in\mathbb{R}^3$ and $b\in\mathbb{R}^3$ we have
  \begin{align*}
    \nabla_\Gamma w = P\nabla w = PA \quad\text{on}\quad \Gamma, \quad A :=
    \begin{pmatrix}
      0 & a_3 & -a_2 \\
      -a_3 & 0 & a_1 \\
      a_2 & -a_1 & 0
    \end{pmatrix}
    = -A^T.
  \end{align*}
  Hence $D_\Gamma(w)=P\{\nabla_\Gamma w+(\nabla_\Gamma w)^T\}P/2=0$ on $\Gamma$ by $P^T=P^2=P$.
  In particular, $w\in\mathcal{K}(\Gamma)$ if it is tangential on $\Gamma$, which shows $\mathcal{R}|_\Gamma\subset\mathcal{K}(\Gamma)$.

  Suppose that $\Gamma$ is a sphere in $\mathbb{R}^3$.
  By a translation we may assume that $\Gamma$ is centered at the origin.
  Then $\mathcal{R}|_\Gamma=\{w(y)=a\times y,\,y\in\Gamma\mid a\in\mathbb{R}^3\}$ and thus it is a three-dimensional subspace of $\mathcal{K}(\Gamma)$.
  Since the dimension of $\mathcal{K}(\Gamma)$ is at most three (see e.g.~\cite[Theorem~35]{Pe06}), we have $\mathcal{K}(\Gamma)=\mathcal{R}|_\Gamma$.

  Next suppose that $\Gamma$ is not a sphere.
  Since $\mathcal{R}\neq\{0\}$, it is axially symmetric by Lemma~\ref{L:IR_Surf} and we may assume as in the proof of Lemma~\ref{L:SR_Eigen} that $\Gamma$ is of the form \eqref{Pf_SRE:SoR} with $C^5$ functions $\varphi$ and $\psi$ satisfying \eqref{Pf_SRE:Arc} and $\varphi(s)>0$ for $s\neq0,L$.
  Then the Gaussian curvature of $\Gamma$ is given by (see e.g.~\cite[Section~5.7]{ON06})
  \begin{align} \label{Pf_IRK:Gau}
    K(\mu(s,\vartheta)) = -\frac{\varphi''(s)}{\varphi(s)}, \quad s\in(0,L),\,\vartheta\in[0,2\pi].
  \end{align}
  We use this formula later.
  Differentiating both sides of \eqref{Pf_SRE:Arc} we also have
  \begin{align} \label{Pf_IRK:D_Arc}
    \varphi'(s)\varphi''(s)+\psi'(s)\psi''(s) = 0, \quad s\in(0,L).
  \end{align}
  Let $X\in\mathcal{K}(\Gamma)$ be of the form
  \begin{align*}
    X(\mu(s,\vartheta)) = X^s(s,\vartheta)\partial_s\mu(s,\vartheta)+X^\vartheta(s,\vartheta)\partial_\vartheta\mu(s,\vartheta), \quad s\in[0,L],\,\vartheta\in[0,2\pi].
  \end{align*}
  Note that $X\in C^2(\Gamma,T\Gamma)$ by the $C^5$-regularity of $\Gamma$ and Lemma~\ref{L:Kil_Reg}.
  By $D_\Gamma(X)=0$ on $\Gamma$ we have $(Y\cdot\nabla_\Gamma)X\cdot Z+Y\cdot(Z\cdot\nabla_\Gamma)X=0$ on $\Gamma$ for all $Y,Z\in C(\Gamma,T\Gamma)$.
  We substitute $\partial_s\mu$ and $\partial_\vartheta\mu$ for $Y$ and $Z$ and then use \eqref{Pf_SRE:Arc}, \eqref{Pf_SRE:Tau}, \eqref{Pf_IRK:D_Arc},
  \begin{gather*}
    (\partial_s\mu\cdot\nabla_\Gamma)X = \frac{\partial(X\circ\mu)}{\partial s}= (\partial_sX^s)\partial_s\mu+X^s\partial_s^2\mu+(\partial_sX^\vartheta)\partial_\vartheta\mu+X^\vartheta\partial_s\partial_\vartheta\mu, \\
    \partial_s^2\mu =
    \begin{pmatrix}
      \varphi''(s)\cos\vartheta \\ \varphi''(s)\sin\vartheta \\ \psi''(s)
    \end{pmatrix}, \quad
    \partial_s\partial_\vartheta\mu =
    \begin{pmatrix}
      -\varphi'(s)\sin\vartheta \\ \varphi'(s)\cos\vartheta \\ 0
    \end{pmatrix}, \quad
    \partial_\vartheta^2\mu =
    \begin{pmatrix}
      -\varphi(s)\cos\vartheta \\ -\varphi(s)\sin\vartheta \\ 0
    \end{pmatrix},
  \end{gather*}
  and a similar equality for $(\partial_\vartheta\mu\cdot\nabla_\Gamma)X$ to get
  \begin{align} \label{Pf_IRK:Kil_Eq}
    \partial_sX^s = 0, \quad \partial_\vartheta X^s+\varphi^2\partial_sX^\vartheta = 0, \quad \varphi^2\partial_\vartheta X^\vartheta+\varphi\varphi'X^s = 0.
  \end{align}
  If $X^s\equiv0$ then $X^\vartheta\equiv c$ is constant by the second and third equations of \eqref{Pf_IRK:Kil_Eq} (note that $\varphi>0$ on $(0,L)$ and $X$ is of class $C^2$).
  In this case,
  \begin{align*}
    X(y) = c\partial_\vartheta\mu(s,\vartheta) = c(e_3\times y)\in\mathcal{R}|_\Gamma, \quad y=\mu(s,\vartheta)\in\Gamma,\,e_3=(0,0,1)
  \end{align*}
  by \eqref{Pf_SRE:Tau}.
  Let us show that each $X\in\mathcal{K}(\Gamma)$ is of this form (here the arguments are essentially the same as in~\cite[Section~74]{Ei49}).
  Assume to the contrary that $X^s\not\equiv0$.
  By the first equation of \eqref{Pf_IRK:Kil_Eq}, $X^s=X^s(\vartheta)$ is independent of $s$.
  Since $X^s$ continuous and $X^s\not\equiv0$, it does not vanish on some open interval $I\subset[0,2\pi]$.
  By the second and third equations of \eqref{Pf_IRK:Kil_Eq},
  \begin{align*}
    \partial_sX^\vartheta(s,\vartheta) = -\frac{\partial_\vartheta X^s(\vartheta)}{\{\varphi(s)\}^2}, \quad \partial_\vartheta X^\vartheta(s,\vartheta) = -\frac{\varphi'(s)X^s(\vartheta)}{\varphi(s)}
  \end{align*}
  for $s\in(0,L)$ and $\vartheta\in[0,2\pi]$ (note that $\varphi(s)>0$ for $s\neq0,L$).
  Since $\partial_\vartheta\partial_s X^\vartheta=\partial_s\partial_\vartheta X^\vartheta$ by the $C^2$-regularity of $X$, we deduce from the above equations that
  \begin{align*}
    \frac{\partial_\vartheta^2 X^s(\vartheta)}{X^s(\vartheta)} = \varphi(s)\varphi''(s)-\{\varphi'(s)\}^2, \quad s\in(0,L),\,\vartheta\in I.
  \end{align*}
  Noting that the left-hand side is independent of $s$ and the function $\varphi$ is of class $C^5$, we differentiate both sides of this equality with respect to $s$ to get
  \begin{align*}
    \varphi(s)\varphi'''(s)-\varphi'(s)\varphi''(s) = 0, \quad s\in(0,L).
  \end{align*}
  Now we observe by this equality and \eqref{Pf_IRK:Gau} that
  \begin{align*}
    \frac{\partial}{\partial s}\bigl(K(\mu(s,\vartheta))\bigr) = -\frac{\varphi(s)\varphi'''(s)-\varphi'(s)\varphi''(s)}{\{\varphi(s)\}^2} = 0, \quad s\in(0,L),\,\vartheta\in[0,2\pi].
  \end{align*}
  By this fact and the continuity of $K$ and $\mu$ on $\Gamma$ and $[0,L]\times[0,2\pi]$, the Gaussian curvature $K$ is constant on the whole closed surface $\Gamma$.
  Thus Liebmann's theorem (see e.g.~\cite[Section~6.3, Theorem~3.7]{ON06}) implies that $\Gamma$ is a sphere, which contradicts with our assumption that $\Gamma$ is not a sphere.
  Hence $\mathcal{K}(\Gamma)$ contains only vector fields of the form $w(y)=c(e_3\times y)$ with $c\in\mathbb{R}$.
  This shows $\mathcal{K}(\Gamma)\subset\mathcal{R}|_\Gamma$ and, since $\mathcal{R}|_\Gamma$ is a subspace of $\mathcal{K}(\Gamma)$, we conclude that $\mathcal{K}(\Gamma)=\mathcal{R}|_\Gamma$.
  \end{proof}

Now we assume again that $\Gamma$ is of class $C^2$ and take $g_0,g_1\in C^1(\Gamma)$ satisfying \eqref{E:Width_Bound}.
Let $\mathcal{R}_0$, $\mathcal{R}_1$, and $\mathcal{R}_g$ be the subspaces of $\mathcal{R}$ given by \eqref{E:Def_Rg} and $\Omega_\varepsilon$ the curved thin domain of the form \eqref{E:Def_CTD} with boundary $\Gamma_\varepsilon$.
As in Section~\ref{SS:Pre_Dom} we scale $g_i$ to assume $|g_i|<\delta$ on $\Gamma$ for $i=0,1$, where $\delta$ is the radius of the tubular neighborhood $N$ of $\Gamma$ given in Section~\ref{SS:Pre_Surf}, and thus $\overline{\Omega}_\varepsilon\subset N$ for all $\varepsilon\in(0,1)$.

\begin{lemma} \label{L:IR_CTD}
  For an infinitesimal rigid displacement $w(x)=a\times x+b$ of $\mathbb{R}^3$ with $a,b\in\mathbb{R}^3$ the following conditions are equivalent:
  \begin{itemize}
    \item[(a)] For all $\varepsilon\in(0,1)$ the restriction of $w$ on $\Gamma_\varepsilon$ satisfies $w|_{\Gamma_\varepsilon}\cdot n_\varepsilon=0$ on $\Gamma_\varepsilon$.
    \item[(b)] There exists a sequence $\{\varepsilon_k\}_{k=1}^\infty$ of positive numbers such that
    \begin{align*}
      \lim_{k\to\infty}\varepsilon_k = 0, \quad w|_{\Gamma_{\varepsilon_k}}\cdot n_{\varepsilon_k} = 0 \quad\text{on}\quad \Gamma_{\varepsilon_k} \quad\text{for all}\quad k\in\mathbb{N}.
    \end{align*}
    \item[(c)] The vector field $w$ belongs to $\mathcal{R}_0\cap\mathcal{R}_1$.
  \end{itemize}
\end{lemma}

\begin{proof}
  For $\varepsilon\in(0,1)$ and $i=0,1$ let $\tau_\varepsilon^i$ be given by \eqref{E:Def_NB_Aux}.
  Then
  \begin{align*}
    n_\varepsilon(y+\varepsilon g_i(y)n(y)) = (-1)^{i+1}\frac{n(y)-\varepsilon\tau_\varepsilon^i(y)}{\sqrt{1+\varepsilon^2|\tau_\varepsilon^i(y)|^2}}, \quad y+\varepsilon g_i(y)n(y) \in \Gamma_\varepsilon^i
  \end{align*}
  with $y\in\Gamma$ by Lemma~\ref{L:Nor_Bo}.
  Moreover, for $w(x)=a\times x+b$ we have
  \begin{align*}
    w(y+\varepsilon g_i(y)n(y)) = w(y)+\varepsilon g_i(y)\{a\times n(y)\}, \quad \{a\times n(y)\}\cdot n(y) = 0, \quad y\in\Gamma.
  \end{align*}
  Hence the condition $w|_{\Gamma_\varepsilon^i}\cdot n_\varepsilon=0$ on $\Gamma_\varepsilon^i$ is equivalent to
  \begin{align} \label{Pf_IRC:Reduce}
    w|_\Gamma\cdot n-\varepsilon w|_\Gamma\cdot\tau_\varepsilon^i-\varepsilon^2g_i(a\times n)\cdot \tau_\varepsilon^i = 0 \quad\text{on}\quad\Gamma.
  \end{align}
  Now let us prove the lemma.
  The condition (a) clearly implies (b).
  Next we show that (b) yields (c).
  Suppose that (b) is satisfied.
  Then by \eqref{Pf_IRC:Reduce} we have
  \begin{align*}
    w|_\Gamma\cdot n-\varepsilon_kw|_\Gamma\cdot\tau_{\varepsilon_k}^i-\varepsilon_k^2g_i(a\times n)\cdot \tau_{\varepsilon_k}^i = 0 \quad\text{on}\quad\Gamma
  \end{align*}
  for $k\in\mathbb{N}$ and $i=0,1$.
  Letting $k\to\infty$ in this equality we obtain $w|_\Gamma\cdot n=0$ on $\Gamma$ by \eqref{E:Tau_Bound}.
  Hence $w\in\mathcal{R}$.
  Moreover, from the above equalities we deduce that
  \begin{align*}
    w|_\Gamma\cdot\tau_{\varepsilon_k}^i+\varepsilon_kg_i(a\times n)\cdot\tau_{\varepsilon_k}^i = 0 \quad\text{on}\quad \Gamma.
  \end{align*}
  Since $\{\tau_{\varepsilon_k}^i\}_{k=1}^\infty$ converges to $\nabla_\Gamma g_i$ uniformly on $\Gamma$ by \eqref{E:Tau_Diff}, we send $k\to\infty$ in this equality to get $w|_\Gamma\cdot\nabla_\Gamma g_i=0$ on $\Gamma$ for $i=0,1$.
  Thus $w\in\mathcal{R}_0\cap\mathcal{R}_1$, i.e. (c) is valid.

  Let us show that (c) implies (a).
  If $w\equiv0$ then (a) is trivial.
  Suppose that $w\not\equiv0$ belongs to $\mathcal{R}_0\cap\mathcal{R}_1$.
  Let $\varepsilon\in(0,1)$ and $i=0,1$.
  Since the condition $w|_{\Gamma_\varepsilon^i}\cdot n_\varepsilon=0$ on $\Gamma_\varepsilon^i$ is equivalent to \eqref{Pf_IRC:Reduce} and $w\in\mathcal{R}_0\cap\mathcal{R}_1\subset\mathcal{R}$ satisfies $w|_\Gamma\cdot n=0$ on $\Gamma$, it is sufficient for (a) to show that
  \begin{align} \label{Pf_IRC:Goal}
    w(y)\cdot\tau_\varepsilon^i(y) = 0, \quad \{a\times n(y)\}\cdot\tau_\varepsilon^i(y) = 0 \quad\text{for all}\quad y\in\Gamma.
  \end{align}
  Hereafter we fix and suppress the argument $y$.
  If $w=0$, then $a\times n=0$ by \eqref{E:SR_Eigen} and the equalities \eqref{Pf_IRC:Goal} are valid.
  Suppose $w\neq 0$.
  Then $w$ is the eigenvector of $W$ corresponding to the eigenvalue $\lambda$ by \eqref{E:SR_Eigen}.
  Since $W$ has the eigenvalues $\kappa_1$, $\kappa_2$, and zero with $Wn=0$ and $w\neq n$ by $w\cdot n=0$, we have $\lambda=\kappa_1$ or $\lambda=\kappa_2$.
  Without loss of generality, we may assume $\lambda=\kappa_1$, i.e. $Ww=\kappa_1w$.
  Then since
  \begin{align*}
    (I_3-\varepsilon g_iW)w = (1-\varepsilon g_i\kappa_1)w, \quad 1-\varepsilon g_i\kappa_1 > 0,
  \end{align*}
  and $I_3-\varepsilon g_iW$ is invertible by $|g_i|<\delta$ on $\Gamma$, \eqref{E:Curv_Bound}, and Lemma~\ref{L:Wein},
  \begin{align} \label{Pf_IRC:WRe_W}
    (I_3-\varepsilon g_iW)^{-1}w = (1-\varepsilon g_i\kappa_1)^{-1}w.
  \end{align}
  We use \eqref{E:Def_NB_Aux}, the symmetry of $W$, \eqref{Pf_IRC:WRe_W}, and $w\cdot\nabla_\Gamma g_i=0$ by $w\in\mathcal{R}_i$ to get
  \begin{align} \label{Pf_IRC:W_Tau}
    w\cdot\tau_\varepsilon^i = (I_3-\varepsilon g_iW)^{-1}w\cdot\nabla_\Gamma g_i = (1-\varepsilon g_i\kappa_1)^{-1}(w\cdot\nabla_\Gamma g_i) = 0.
  \end{align}
  Moreover, by \eqref{E:SR_Eigen} with $\lambda=\kappa_1$ and \eqref{Pf_IRC:WRe_W},
  \begin{align*}
    (I_3-\varepsilon g_iW)^{-1}(a\times n) = -\kappa_1(I_3-\varepsilon g_iW)^{-1}w = -\kappa_1(1-\varepsilon g_i\kappa_1)^{-1}w.
  \end{align*}
  Using this equality we can show $(a\times n)\cdot\tau_\varepsilon^i=0$ as in \eqref{Pf_IRC:W_Tau}.
  Thus we get \eqref{Pf_IRC:Goal} and conclude that $w|_{\Gamma_\varepsilon^i}\cdot n_\varepsilon=0$ on $\Gamma_\varepsilon^i$ for all $\varepsilon\in(0,1)$ and $i=0,1$, i.e. (a) holds.
\end{proof}

By Lemmas~\ref{L:IR_Surf} and~\ref{L:IR_CTD} we observe that the nontriviality of $\mathcal{R}_0\cap\mathcal{R}_1$ implies the uniform axial symmetry of $\Omega_\varepsilon$.

\begin{lemma} \label{L:CTD_AS}
  If there exists $w(x)=a\times x+b\in\mathcal{R}_0\cap\mathcal{R}_1$ such that $w\not\equiv0$, then $a\neq 0$, $a\cdot b=0$, and $\Omega_\varepsilon$ is axially symmetric around the axis parallel to $a$ and passing through $b_a=|a|^{-2}(a\times b)$ for all $\varepsilon\in(0,1)$.
\end{lemma}

\begin{proof}
  Let $w(x)=a\times x+b\in\mathcal{R}_0\cap\mathcal{R}_1$.
  Then $w|_{\Gamma_\varepsilon^i}\cdot n_\varepsilon=0$ on $\Gamma_\varepsilon^i$ for each $\varepsilon\in(0,1)$ and $i=0,1$ by Lemma~\ref{L:IR_CTD}.
  Hence if $w\not\equiv0$ then Lemma~\ref{L:IR_Surf} implies that $a\neq0$, $a\cdot b=0$, and both $\Gamma_\varepsilon^0$ and $\Gamma_\varepsilon^1$ are axially symmetric around the axis parallel to $a$ and passing through $b_a$, which yields the same axial symmetry of $\Omega_\varepsilon$.
\end{proof}

Next we see that the triviality of $\mathcal{R}_g$ yields the axial asymmetry of $\Omega_\varepsilon$.

\begin{lemma} \label{L:CTD_Rg}
  If $\mathcal{R}_g=\{0\}$, then there exists $\tilde{\varepsilon}\in(0,1)$ such that $\Omega_\varepsilon$ is not axially symmetric around any axis for all $\varepsilon\in(0,\tilde{\varepsilon})$.
\end{lemma}

\begin{proof}
  We prove the contrapositive statement: if there exists a sequence $\{\varepsilon_k\}_{k=1}^\infty$ convergent to zero such that $\Omega_{\varepsilon_k}$ is (and thus $\Gamma_{\varepsilon_k}^0$ and $\Gamma_{\varepsilon_k}^1$ are) axially symmetric around some axis $l_k$ for each $k\in\mathbb{N}$, then $\mathcal{R}_g\neq\{0\}$.

  Suppose that such a sequence $\{\varepsilon_k\}_{k=1}^\infty$ exists and that for each $k\in\mathbb{N}$ the axis $l_k$ is of the form $l_k=\{sa_k+b_k\mid s\in\mathbb{R}\}$ with $a_k,b_k\in\mathbb{R}^3$, $a_k\neq0$, i.e. $l_k$ is parallel to $a_k$ and passing through $b_k$.
  Replacing $a_k$ with $|a_k|^{-1}a_k$ we may assume $a_k\in S^2$ for all $k\in\mathbb{N}$ without changing $l_k$ (here $S^2$ is the unit sphere in $\mathbb{R}^3$).
  Since $\Omega_\varepsilon$ is contained in the bounded set $N$ for all $\varepsilon\in(0,1)$, there exists an open ball $B_R$ centered at the origin of radius $R>0$ such that $\Omega_{\varepsilon_k}\subset B_R$ for all $k\in\mathbb{N}$.
  Then, by the axial symmetry of $\Omega_{\varepsilon_k}$ around the axis $l_k$, the intersection $l_k\cap B_R$ is not empty: otherwise the ball generated by the rotation of $B_R$ through the angle $\pi$ around $l_k$ does not intersect with $B_R$ and thus $\Omega_{\varepsilon_k}\subset B_R$ is not axially symmetric around $l_k$.
  Hence we may assume $b_k\in l_k\cap B_R$ for all $k\in\mathbb{N}$ by replacing $b_k$ with $b_k-sa_k$ for an appropriate $s\in\mathbb{R}$.
  Now $\{a_k\}_{k=1}^\infty$ and $\{b_k\}_{k=1}^\infty$ are bounded in $\mathbb{R}^3$ and thus converge (up to a subsequence) to some $a\in S^2$ and $b\in\mathbb{R}^3$, respectively.

  Let us prove $w(x):=a\times(x-b)\in\mathcal{R}_g$.
  For $k\in\mathbb{N}$ and $i=0,1$ let $\tau_{\varepsilon_k}^i$ be the vector field on $\Gamma$ given by \eqref{E:Def_NB_Aux} and $w_k(x):=a_k\times(x-b_k)$, $x\in\mathbb{R}^3$.
  Then
  \begin{align} \label{Pf_CRg:Lim_Wk}
    \lim_{k\to\infty}\tau_{\varepsilon_k}^i(y) = \nabla_\Gamma g_i(y), \quad \lim_{k\to\infty}w_k(y) = w(y) \quad\text{for all}\quad y\in\Gamma
  \end{align}
  by \eqref{E:Tau_Diff}, $\lim_{k\to\infty}a_k=a$, and $\lim_{k\to\infty}b_k=b$.
  For each $k\in\mathbb{N}$ and $i=0,1$, since $\Gamma_{\varepsilon_k}^i$ is axially symmetric around the axis $l_k$, Lemma~\ref{L:IR_Surf} implies that $w_k|_{\Gamma_{\varepsilon_k}^i}\cdot n_{\varepsilon_k}=0$ on $\Gamma_{\varepsilon_k}^i$.
  By the proof of Lemma~\ref{L:IR_CTD} (see \eqref{Pf_IRC:Reduce}) this condition is equivalent to
  \begin{align} \label{Pf_CRg:Reduce}
    w_k|_\Gamma\cdot n-\varepsilon_k w_k|_\Gamma\cdot\tau_{\varepsilon_k}^i-\varepsilon_k^2g_i(a_k\times n)\cdot\tau_{\varepsilon_k}^i = 0 \quad\text{on}\quad \Gamma.
  \end{align}
  Letting $k\to\infty$ in \eqref{Pf_CRg:Reduce} we get $w|_\Gamma\cdot n=0$ on $\Gamma$ by \eqref{Pf_CRg:Lim_Wk} and $\lim_{k\to\infty}a_k=a$.
  Thus $w\in\mathcal{R}$.
  Next we subtract \eqref{Pf_CRg:Reduce} for $i=1$ from that for $i=0$ and divide the resulting equality by $\varepsilon_k$.
  Then since $w_k|_\Gamma\cdot n$ does not depend on $i$ we have
  \begin{align*}
    w_k|_\Gamma\cdot(\tau_{\varepsilon_k}^1-\tau_{\varepsilon_k}^0)+\varepsilon_k(a_k\times n)\cdot(g_1\tau_{\varepsilon_k}^1-g_0\tau_{\varepsilon_k}^0) = 0 \quad\text{on}\quad \Gamma.
  \end{align*}
  We send $k\to\infty$ in this equality and use \eqref{Pf_CRg:Lim_Wk} and $\lim_{k\to\infty}a_k=a$ to obtain
  \begin{align*}
    w|_\Gamma\cdot(\nabla_\Gamma g_1-\nabla_\Gamma g_0) = w|_\Gamma\cdot\nabla_\Gamma g = 0 \quad\text{on}\quad \Gamma.
  \end{align*}
  This shows $w\in\mathcal{R}_g$.
  Since $w\not\equiv0$ by $a\in S^2$, we conclude that $\mathcal{R}_g\neq\{0\}$.
\end{proof}

Finally, we show that the restriction on $\Omega_\varepsilon$ of a vector field in $\mathcal{R}_0\cap\mathcal{R}_1$ belongs to the solenoidal space $L_\sigma^2(\Omega_\varepsilon)=\{u\in L^2(\Omega_\varepsilon)^3\mid\text{$\mathrm{div}\,u=0$ in $\Omega_\varepsilon$, $u\cdot n_\varepsilon=0$ on $\Gamma_\varepsilon$}\}$.

\begin{lemma} \label{L:IR_Sole}
  The inclusion $\mathcal{R}_0\cap\mathcal{R}_1\subset L_\sigma^2(\Omega_\varepsilon)$ holds for each $\varepsilon\in(0,1)$.
\end{lemma}

\begin{proof}
  Let $w(x)=a\times x+b\in\mathcal{R}_0\cap\mathcal{R}_1$.
  Then $\mathrm{div}\,w=0$ in $\mathbb{R}^3$ by direct calculations.
  Moreover, $w|_{\Gamma_\varepsilon}\cdot n_\varepsilon=0$ on $\Gamma_\varepsilon$ by Lemma~\ref{L:IR_CTD}.
  Hence $w\in L_\sigma^2(\Omega_\varepsilon)$.
\end{proof}

\end{appendices}

\section*{Acknowledgments}
The author would like to thank Professor Yoshikazu Giga for his valuable comments on this work.
He is also grateful to Mr. Yuuki Shimizu for fruitful discussions about Killing vector fields on a surface.

The work of the author was supported by Grant-in-Aid for JSPS Fellows No. 16J02664 and the Program for Leading Graduate Schools, MEXT, Japan.


\end{document}